\documentclass[10pt,final]{book}
\usepackage[a4paper,hmargin={2.4cm,2.4cm},vmargin={3.1cm,2.9cm},headheight=14pt]{geometry}

\usepackage{combelow}
\usepackage{amsthm}
\usepackage{amssymb}
\usepackage{stmaryrd}
\usepackage{makeidx}
\usepackage[all]{xy}
\usepackage{epsf}
\usepackage{enumerate}
\usepackage{fancyhdr}
\usepackage{amsmath,amscd}
\newtheorem{theorem}{Theorem}[section]
\newtheorem{mtheorem}{Theorem}
\newtheorem*{theoremzero}{Theorem 0}
\newtheorem*{TubeLemma}{The Tube Lemma}
\newtheorem*{theoremtwo}{Theorem 2}
\newtheorem*{theoremthree}{Theorem 3}
\newtheorem*{theoremfour}{Theorem 4}
\newtheorem*{theoremfive}{Theorem 5}
\newtheorem*{theoremone}{Theorem 1}

\newtheorem*{tvL}{Tame Vanishing Lemma}
\newtheorem{proposition}[theorem]{Proposition}
\newtheorem{lemma}[theorem]{Lemma}
\newtheorem{corollary}[theorem]{Corollary}

\newtheorem{definition}[theorem]{Definition}
\newtheorem{definitions}[theorem]{Definitions}
\newtheorem{example}[theorem]{Example}
\newtheorem{remark}[theorem]{Remark}

\newcommand{\rmap}{\longrightarrow}
\newcommand{\plin}{\pi_{\mathrm{lin}}}
\newcommand{\diffto}{\xrightarrow{\raisebox{-0.2 em}[0pt][0pt]{\smash{\ensuremath{\sim}}}}}
\newcommand{\Gr}{\textrm{Gr}_k(M)}

\newcommand{\thesistitle}
{Normal forms in Poisson geometry}

\newcommand{\nativetitle}
{Normaalvormen in de Poisson-meetkunde}

\makeindex

\begin{document}
 \frontmatter
 \thispagestyle{empty}
$\phantom{1}$
 \clearpage

 \thispagestyle{empty}
\begin{center}

\vspace*{0.2\textheight}

\begin{minipage}{\textwidth}
\begin{center}
  \renewcommand{\baselinestretch}{1.2}
  \bfseries\huge\thesistitle
\end{center}
\end{minipage}

\end{center}

 \clearpage


\thispagestyle{empty}

\noindent $\phantom{0}$\\

\noindent $\phantom{0}$\\
$\phantom{0}$\\
$\phantom{0}$\\
$\phantom{0}$\\
$\phantom{0}$\\

\vfill

\noindent Normal forms in Poisson geometry, Ioan M\u{a}rcu\cb{t}

\noindent Ph.D. Thesis Utrecht University, February 2013


\noindent ISBN: 978-90-393-5913-6

\noindent

\clearpage

\thispagestyle{empty}
\begin{center}

\vspace*{3em}

\begin{minipage}{\textwidth}
\begin{center}
  \renewcommand{\baselinestretch}{1.2}
  \bfseries\huge\thesistitle
\end{center}
\end{minipage}

\vspace{5em}

\begin{minipage}{\textwidth}
\begin{center}
  \renewcommand{\baselinestretch}{1.2}
  \bfseries\Large\nativetitle
\end{center}
\end{minipage}

\vspace{\baselineskip} (met een samenvatting in het Nederlands)

\vfill

{\LARGE Proefschrift}

\vspace{3em}

\begin{minipage}{0.9\textwidth}
\begin{center}
ter verkrijging van de graad van doctor aan de Universiteit Utrecht\\
op gezag van de rector magnificus, prof.\,dr.\ G.J.\ van der Zwaan,\\
ingevolge het besluit van het college voor promoties in het openbaar\\
te verdedigen op maandag 11 februari 2013 des middags te 2.30 uur
\end{center}
\end{minipage}
\vspace{\baselineskip}

door \vspace{\baselineskip}

{\Large Ioan Tiberiu M\u{a}rcu\cb{t}} \vspace{\baselineskip}

geboren op 24 maart 1984 te Sibiu, Roemeni\"e \vspace{\baselineskip}

\end{center}

\clearpage

\thispagestyle{empty}

\begin{tabular}{ll}
Promotor:  & Prof.\ dr. M.\ Crainic 
\end{tabular}

\vfill

This thesis was accomplished with financial support from the Netherlands Organisation for Scientific Research (NWO)
under the Vidi Project ``Poisson Topology'' no. 639.032.712.

 \cleardoublepage

\thispagestyle{empty}

\vspace*{4cm}
\begin{center}
 \textit{Pentru Ioana \& Dodu}
\end{center}

\clearpage

\thispagestyle{empty}

\cleardoublepage

\pagestyle{fancy}

\fancyhead{}
\fancyhead[CE]{Contents} 
\fancyhead[CO]{Contents}

\tableofcontents

\clearpage \pagestyle{plain}

\mainmatter

\fancyhead{} \setcounter{tocdepth}{1}
\chapter*{Introduction}
\addcontentsline{toc}{chapter}{Introduction} \pagestyle{fancy} 
\fancyhead[CE]{Introduction} 
\fancyhead[CO]{Introduction}

In this thesis we prove four local normal form theorems in Poisson geometry using several techniques. Theorems 1 and
2 are normal form results around symplectic leaves, which are both proven geometrically. While Theorem 1 is in the
setting of symplectic foliations, Theorem 2 treats arbitrary Poisson manifolds. Theorem 3 is a formal equivalence result
around Poisson submanifolds with an algebraic proof that uses formal power series expansions of diffeomorphisms and
multivector fields. Theorem 4 is a local rigidity result around Poisson submanifolds with an analytic proof involving the
fast convergence method of Nash and Moser. This result implies a strengthening of Theorem 2. As a more surprising
application of Theorem 4, we obtain a description of the smooth deformations of Lie-Poisson spheres (Theorem 5),
which represents the first computation of a Poisson moduli space in dimension greater or equal than three around a
degenerate (i.e.\ non-symplectic) Poisson structure. We also present a new approach
to the existence problem of symplectic realizations (Theorem 0).\\

\noindent\underline{\textbf{Poisson geometry}}\\

Poisson geometry has its origins in the Hamiltonian formulation of classical mechanics, where the phase space of
certain mechanical systems carries a Poisson bracket. In general, a Poisson structure on a manifold $M$ is defined as
Lie bracket $\{\cdot,\cdot\}$ on the space of smooth functions on $M$ that acts by derivations in each entry.
Explicitly, the operation $\{\cdot,\cdot\}$ is required to satisfy:
\[\{f,g\}=-\{g,f\}, \]
\[ \{f,gh\}=\{f,g\}h+\{f,h\}g,\]
\[\{f,\{g,h\}\}+\{g,\{h,f\}\}+\{h,\{f,g\}\}=0.\]
The second equation shows that every function $f\in C^{\infty}(M)$ induces a vector field $H_f$ (the Hamiltonian
vector field of $f$) via $L_{H_f}(g)=\{f,g\}$.

A Poisson structure can be described also by a bivector field $\pi\in\Gamma(\Lambda^2TM)$ satisfying $[\pi,\pi]=0$
for the Schouten bracket. The relation between the bivector and the bracket is given by
\[\{f,g\}=\langle\pi,df\wedge dg\rangle.\]

The beauty and the strength of Poisson geometry lies in the fact that it brings together various fields of differential
geometry:

\vspace{0.1cm}

\textbf{Foliation Theory}. A Poisson manifold carries a natural partition into immersed submanifolds, called leaves.
Two points that can be connected by flow lines of Hamiltonian vector fields belong to the same leaf, and in fact, this
relation defines the partition. If the leaves have the same dimension, one deals with standard foliations and one says
that the Poisson structure is regular.

\textbf{Symplectic Geometry} can be seen as the nondegenerate version of Poisson geometry. Namely, the condition
that a nondegenerate two-form is closed is equivalent to its inverse being a Poisson bivector. The relations between
symplectic and Poisson geometry are actually deeper: the leaves of a Poisson manifold are (canonically) symplectic
manifolds, while the global object associated to a Poisson structure is symplectic (the Weinstein groupoid).


\textbf{Lie Theory}. Lie algebras are the same as linear Poisson structures: a Lie algebra
$(\mathfrak{g},[\cdot,\cdot])$ induces a canonical Poisson structure $\pi_{\mathfrak{g}}$ on $\mathfrak{g}^*$,
whose Poisson bracket extends the Lie bracket $[\cdot,\cdot]$ form linear functions to all smooth functions;
conversely, any Poisson bivector on a vector space whose coefficients are linear functions is of this form. In fact, for
any Poisson manifold $(M,\pi)$, the normal spaces to the leaves are linear Poisson, i.e.\ for $x\in M$, the annihilator of
the tangent space to the leaf through $x$ carries a Lie algebra structure. This is called the isotropy Lie algebra at $x$
and is denoted by
$(\mathfrak{g}_x,[\cdot,\cdot])$.\\

In contrast with symplectic geometry, Poisson geometry is interested even locally, and a large part of the research is
motivated by such local issues. The study of local properties of Poisson manifolds begins with the work of Weinstein
\cite{Wein}. We mention here (see \emph{loc.cit.}): \vspace{0.2cm}

\textbf{The local splitting theorem}. Locally, a Poisson manifold decomposes as the product of a symplectic manifold
(which is a small open in the symplectic leaf) and a Poisson manifold with a fixed point (which is the Poisson structure
induced on a small transversal to the leaf).\vspace{0.1cm}

\textbf{Local existence of symplectic realizations}. Around every point in $(M,\pi)$, there exists an open $U$ and a
symplectic manifold $(\Sigma,\omega)$ with a surjective Poisson submersion $(\Sigma,\omega)\to
(U,\pi)$.\vspace{0.1cm}

\textbf{Formal linearization}. If the isotropy Lie algebra $\mathfrak{g}_x$ at a fixed point $x$ is semisimple, then the
Poisson structure is formally isomorphic around $x$ to the linear Poisson structure $\pi_{\mathfrak{g}_x}$.
\vspace{0.1cm}

Weinstein also conjectures the following theorem, proven by Conn \cite{Conn}. \vspace{0.1cm}

\textbf{Conn's linearization theorem}. If the isotropy Lie algebra $\mathfrak{g}_x$ at a fixed point $x$ is semisimple
and compact, then the Poisson structure is isomorphic around $x$ to the linear Poisson structure
$\pi_{\mathfrak{g}_x}$. \vspace{0.1cm}

Conn's proof is analytic, based on the Nash-Moser method. We rewrite his proof with some simplifications in section
\ref{Section_Conn}. Conn's result is a ``first order normal form theorem'', in the sense that both the hypothesis and
the local model depend only on the first jet of the Poisson bivector at the fixed point.\\

\noindent\underline{\textbf{The main results of this thesis}}\\

We briefly explain our main results. The order of the presentation is different from that of the thesis, but more natural
for this introduction.

\vspace{0.2cm}

\textbf{Theorem 2: The normal form theorem around symplectic leaves}

\vspace{0.2cm}

One of the main results of this thesis is a first order normal form theorem around symplectic leaves (Theorem 2 below),
which generalizes Conn's theorem. Consider a Poisson manifold $(M,\pi)$ with a symplectic leaf $(S,\omega_S)$. The
vector bundle
\[A_S:=T^*M_{|S}\]
carries the structure of a transitive Lie algebroid over $S$, and the pair $(A_S,\omega_S)$ encodes the first order jet of
$\pi$ at $S$. Out of this pair one builds the \emph{local model} of $\pi$ around $S$, which generalizes the linear
Poisson structure from Conn's theorem. This is the Poisson-geometric first order approximation of $\pi$ at $S$, and
was first constructed by Vorobjev \cite{Vorobjev,Vorobjev2}; we describe it in detail in chapter \ref{ChNormalForms},
where we present also new approaches. The conditions of Theorem 2 are expressed in terms of the $1$-connected
integration of $A_S$, called the \emph{Poisson homotopy bundle} of $S$ and denoted by $P$.
\begin{theoremtwo}
If the Poisson homotopy bundle $P$ of $S$ is smooth, compact and $H^2(P)=0$, then $\pi$ is isomorphic around $S$
to its local model.
\end{theoremtwo}

The proof of Theorem 2 is based on the geometric proof of Conn's theorem of Crainic and Fernandes \cite{CrFe-Conn}.
In particular, we follow the same steps: first, we reduce the problem to a cohomological one using the Moser path
method; second, using the Van Est map and vanishing of cohomology of proper groupoids, we show that integrability
implies vanishing of the cohomological obstructions; in the third step we show that integrability is implied by the
existence of some special symplectic realizations; finally, we constructs such realizations using the space of cotangent
paths. Compared to \cite{CrFe-Conn}, we describe explicitly the symplectic structure on transversals in the space of
cotangent paths via a simple formula; this is presented in section \ref{Constructing symplectic realizations from
transversals in the manifold of cotangent paths}.\\

\textbf{Theorem 0: On the existence of symplectic realizations}

\vspace{0.2cm}

The formula for the symplectic structure on transversals in the space of cotangent paths served as an inspiration for a
new approach to the existence problem of symplectic realizations, which we explain in the sequel. Let $(M,\pi)$ be a
Poisson manifold. A ``quadratic'' vector field $\mathcal{V}_{\pi}$ on $T^*M$ with the property that
$\mathcal{V}_{\pi,\xi}$ projects to $\pi^{\sharp}(\xi)$, is called a \emph{Poisson spray}. Such vector fields are easily
constructed. Using the principles of ``contravariant geometry'', in section \ref{Ontheexistence} we prove the
following:
\begin{theoremzero}
Let $(M, \pi)$ be a Poisson manifold and let $\mathcal{V}_{\pi}$ be a Poisson spray with flow $\varphi_t$. There exists
an open $\mathcal{U}\subset T^*M$ around the zero section so that
\[ \omega:= \int_{0}^{1} (\varphi_t)^*\omega_{\textrm{can}} dt\]
is a symplectic structure on $\mathcal{U}$ and the canonical projection $p: (\mathcal{U}, \omega)\to (M, \pi)$ is a
symplectic realization.
\end{theoremzero}

\vspace{0.2cm} \textbf{Theorem 1: Reeb stability for symplectic foliations} \vspace{0.2cm}

Another main result proven in this thesis (Theorem 1 form chapter \ref{ChReeb}) is a normal form result for
symplectic foliations $(M,\mathcal{F},\omega)$ (i.e.\ regular Poisson structures), which is the Poisson-geometric
version of the Local Reeb Stability Theorem from foliation theory. Compared to Theorem 2, this is not a first order
normal form result, in the sense that the conditions of Theorem 1 depend on the holonomy of the foliation. We will use
the cohomological variation of the symplectic structure at the leaf $S$ through $x$, denoted by
\[[\delta_S\omega]:\nu_{S,x}\rmap H^2(\widetilde{S}),\]
where $\nu_{S,x}$ is the normal space to $S$ at $x$, and $\widetilde{S}$ is the holonomy cover of $S$.
\begin{theoremone}
Let $(M,\mathcal{F}, \omega)$ be a symplectic foliation and let $S\subset M$ be an embedded symplectic leaf. If $S$
is a finite type manifold with finite holonomy and surjective cohomological variation then, around $S$, the symplectic
foliation is isomorphic to its local model.
\end{theoremone}

The proof of Theorem 1 is substantially easier than that of Theorem 2: we use a version of Reeb's theorem for
non-compact leaves to linearize the foliation and then Moser's argument to put the symplectic structures on the leaves
in normal form.

\vspace{0.2cm} \textbf{Theorem 3: Formal rigidity around Poisson submanifolds} \vspace{0.2cm}

The next main result that we describe (Theorem 3 from chapter \ref{ChFormalRigidity}) is a formal rigidity result,
which generalizes Weinstein's formal linearization theorem from fixed points to arbitrary Poisson submanifolds.

\begin{theoremthree}
Let $\pi_1$ and $\pi_2$ be two Poisson structures on $M$ and let $S\subset M$ be an embedded Poisson submanifold
for both structures. If $\pi_1$ and $\pi_2$ have the same first order jet along $S$ and their common Lie algebroid
$A_S$ satisfies
\[H^{2}(A_S,\mathcal{S}^{k}(\nu_S^*))=0, \ \ \forall \  k\geq 2,\]
then the two structures are formally Poisson diffeomorphic.
\end{theoremthree}

This result relies on an equivalence criterion for Maurer-Cartan elements in complete graded Lie algebras. The proof of
the criterion uses a sequence of elements in the algebra that converges formally to an element whose exponential
realizes the gauge equivalence. Our original motivation for proving Theorem 3 was to understand the algebraic steps
needed to construct such a sequence, and to use this sequence in an analytic proof of Theorem 2. The outcome is
Theorem 4.

\vspace{0.2cm} \textbf{Theorem 4: Rigidity around Poisson submanifolds}\vspace{0.2cm}

Using the fast convergence method of Nash and Moser, in chapter \ref{ChRigidity} we prove a rigidity result for
integrable Poisson structure around compact Poisson submanifolds (Theorem 4 below). A Poisson structure $\pi$ is
called $C^p$-$C^1$-rigid around a submanifold $S$, if every Poisson structure $\widetilde{\pi}$ that is $C^p$-close
to $\pi$ around $S$ is isomorphic to $\pi$ around $S$ by a Poisson diffeomorphism $C^1$-close to the identity.

\begin{theoremfour}
Let $(M,\pi)$ be a Poisson manifold for which the Lie algebroid $T^*M$ is integrable by a Hausdorff Lie groupoid
whose $s$-fibers are compact and their de Rham cohomology vanishes in degree two. For every compact Poisson
submanifold $S$ of $M$ we have that
\begin{enumerate}[(a)]
\item $\pi$ is $C^p$-$C^1$-rigid around $S$,
\item up to isomorphism, $\pi$ is determined around $S$ by its first order jet at $S$.
\end{enumerate}
\end{theoremfour}

A consequence of  (b) is a strengthening of Theorem 2: namely, it suffices to assume existence of some compact $P$
integrating $A_S$ such that $H^2(P)=0$; $P$ does not have to be $1$-connected.

The proof of Theorem 4 is mainly inspired by Conn's proof. We deviate the most from his approach in the construction
of the tame homotopy operators. For this we prove a general result on tame vanishing of Lie algebroid cohomology with
coefficients: the Tame Vanishing Lemma, which is presented in the appendix \ref{ChFoli}. This result can be applied to
similar geometric problems. To illustrate this, we briefly discuss in the appendix \ref{ChFoli} a theorem of Hamilton on
rigidity of foliations.\\

\textbf{Theorem 5: Smooth deformations of the Lie-Poisson sphere}\vspace{0.2cm}

The last main result of this thesis is a rather unexpected application of Theorem 4 to the Poisson moduli space of the
Lie-Poisson spheres. Let $\mathfrak{g}$ be a compact semisimple Lie algebra. We define the \emph{Lie-Poisson
sphere} as the unit sphere in $\mathfrak{g}^*$ with respect to an $Aut(\mathfrak{g})$- invariant inner product.
Endowed with the restriction of the linear Poisson structure $\pi_{\mathfrak{g}}$, the sphere becomes a Poisson
manifold, denoted
\[(\mathbb{S}(\mathfrak{g}^*),\pi_{\mathbb{S}}).\]
Combining Theorem 4 with a nice geometric argument, in chapter \ref{ChDef} we describe all Poisson structures on
$\mathbb{S}(\mathfrak{g}^*)$ around $\pi_{\mathbb{S}}$. Using a Lie theoretical interpretation of the cohomology
of the leaves of $\pi_{\mathbb{S}}$, we determine which of these Poisson structures are isomorphic. The outcome is:
\begin{theoremfive}
\begin{enumerate}[(a)]
\item There exists a $C^p$-open $\mathcal{W}\subset \mathfrak{X}^2(\mathbb{S}(\mathfrak{g}^*))$ around $\pi_{\mathbb{S}}$, such that every Poisson structure in $\mathcal{W}$ is isomorphic to one
of the form $f\pi_{\mathbb{S}}$, where $f$ is a positive Casimir function. 
\item For two positive Casimirs $f$ and $g$, the Poisson manifolds $(\mathbb{S}(\mathfrak{g}^*),f\pi_{\mathbb{S}})$ and $(\mathbb{S}(\mathfrak{g}^*),g\pi_{\mathbb{S}})$
are isomorphic precisely when $f$ and $g$ are related by an outer automorphism of $\mathfrak{g}$.
\end{enumerate}
\end{theoremfive}

\noindent In other words, the Poisson moduli space is parameterized around $\pi_{\mathbb{S}}$ by
\begin{equation*}
\mathfrak{Casim}(\mathbb{S}(\mathfrak{g}^*),\pi_{\mathbb{S}})/Out(\mathfrak{g}),
\end{equation*}
where $\mathfrak{Casim}(\mathbb{S}(\mathfrak{g}^*),\pi_{\mathbb{S}})$ is the space of Casimirs and
$Out(\mathfrak{g})$ is the finite group of outer automorphism. Using classical invariant theory, we give a more
explicit description of this space at the end of chapter \ref{ChDef}.

This result represents the first computation of a Poisson moduli space in dimension greater or equal to three around a
degenerate (i.e.\ not symplectic) Poisson structure.\\

\noindent\underline{\textbf{Other contributions of the thesis}}\\

There are some smaller contribution of the thesis which we consider of independent interest. Some of these are
auxiliary results used in the proofs of the main theorems, others are new approaches to well-known results.

\vspace{0.2cm}

\textbf{A detailed proof of Conn's theorem}. In section \ref{Section_Conn} we revisit Conn's theorem presenting it as
a manifestation of a rigidity phenomenon. We rewrite the proof with some simplifications. The major novelly in our
approach lies in the construction of the homotopy operators, which uses classical Hodge theory. As in the proof of
Theorem 4, the construction of these operators follows from the Tame Vanishing Lemma, but for clarity of the
exposition we include the construction in this case. Actually, section \ref{Section_Conn} is a toy model of the proof of
Theorem 4, but which is still involved enough to capture the main technical difficulties.

\vspace{0.1cm}

\textbf{Existence of invariant tubular neighborhoods}. In subsection \ref{Subsection_inv_tub_nbd} we prove a result
(Lemma \ref{Lemma_tubular_neighborhood}) on the existence of invariant tubular neighborhoods around invariant
submanifolds of the base of a proper Lie groupoid. This result is used in the proof of Theorem 4. The proof relies on
results from \cite{Posthuma}, and to experts such as the authors of \emph{loc.cit.} our lemma should be obvious, yet
we couldn't find it in the literature.

\vspace{0.1cm}

\textbf{Integration of ideals}. In subsection \ref{subsection_representations} we prove that an ideal of a Lie algebroid
$A$ can be integrated as a representation to any $s$-connected Lie groupoid of $A$ (Lemma
\ref{Lemma_integrating_ideals}). This is precisely what happens also in classical Lie theory. This result is used later in
the proof of Theorem 4.

\vspace{0.1cm} \textbf{Symplectic realizations from transversals in the manifold of cotangent paths}. Section
\ref{Constructing symplectic realizations from transversals in the manifold of cotangent paths} presents a
finite-dimensional, explicit approach to the space of cotangent paths of a Poisson manifold. We give a simple formula
for the symplectic structure on transversals to the foliation induced by cotangent homotopy. This formula turns out to
be an efficient computational tool, which allows us to prove many of the properties of these transversals explicitly. The
results obtained in this section were mostly known from the infinite-dimensional construction of the Weinstein
groupoid as a symplectic reduction \cite{CaFe,CrFe2,CrFe-Conn}, yet it is likely that our approach can be further used
to give a more explicit description of the local groupoid of a Poisson manifold. The results of this section are used in the
proof of Theorem 2.

\vspace{0.1cm}

\textbf{Local Reeb Stability Theorem for non-compact leaves}. In the appendix \ref{section_Appendix_Reeb} of
chapter \ref{ChReeb} we prove a non-invariant version of the Local Reeb Stability Theorem around non-compact leaves
(Theorem \ref{Reeb_Theorem}). The proof is an adaptation of the proof of the classical result from \cite{MM}. This
result is used in the proof of Theorem 1.

\vspace{0.1cm}

\textbf{The local model around symplectic leaves}. Chapter \ref{ChNormalForms} contains several new descriptions
of the local model of a Poisson structure around a symplectic leaf. We also extend the algebraic framework developed
in \cite{CrFe-stab} in order to handle Vorobjev triples and their linearization. This framework is of interest on its own,
since it describes several other geometric structures that are transverse to the fibers of a bundle (flat Ehresmann
connections, Dirac structures) and their linearization.

\vspace{0.1cm}

\textbf{Equivalence of Maurer Cartan elements in complete graded Lie algebras}. In the appendix
\ref{section_equivalence} of chapter \ref{ChFormalRigidity}, we prove a general criterion (Theorem \ref{Teo1}) for
equivalence of Maurer-Cartan elements in graded Lie algebras endowed with a complete filtration. This result is used
in the proof of Theorem 3. The analogue of this criterion for differential graded associative algebras can be found in the
Appendix A of \cite{CMB}.

\vspace{0.1cm}

\textbf{Continuity of the volume function}. In subsection \ref{subsection_a_global_conflict} we prove that on a Poisson
manifold, integrable by a groupoid with compact $s$-fibers, the function that associates to a regular leaf its symplectic
volume multiplied with the number of elements of its holonomy group, can be extended continuously by zero to the
singular part (Lemma \ref{Lemma_volume_holonomy}). This result implies that there is no compact Poisson manifold
satisfying the conditions of Theorem 4.

\vspace{0.1cm}\textbf{The Tame Vanishing Lemma}. In the first part of the appendix \ref{ChFoli} of the thesis, we
present a general construction of tame homotopy operators for the complex computing Lie algebroid cohomology with
coefficients (the Tame Vanishing Lemma). This result is used in the proof of Theorem 4, and a particular case of this
construction in the proof of Conn's theorem from section \ref{Section_Conn}. When combined with the Nash-Moser
techniques, the Tame Vanishing Lemma is a very useful tool that can be applied in similar geometric problems.

\vspace{0.1cm}

\textbf{Hamilton's theorem on rigidity of foliations}. In the second part of the appendix \ref{ChFoli}, we revisit a
theorem of Richard S. Hamilton \cite{Ham3} on rigidity of foliations. We briefly rewrite Hamilton's proof using the
language of Lie algebroids, and we show that the Tame Vanishing Lemma implies ``tame infinitesimal rigidity'', which
is a crucial step in the proof of this result.

 \clearpage \pagestyle{plain}

\chapter{Basic results in Poisson geometry}\label{ChBasicPoisson}
\pagestyle{fancy}
\fancyhead[CE]{Chapter \ref{ChBasicPoisson}} 
\fancyhead[CO]{Basic results in Poisson geometry} 

This chapter contains an introduction to the basic notions in Poisson geometry and the proofs of two classical results of
this field: the existence of symplectic realizations and Conn's theorem.

\section{Preliminaries}\label{SPreliminaries}

\subsection{Poisson manifolds}\label{SSPoissonDefi}

In this subsection we recall some standard notions and results in Poisson geometry. As a reference to this subject, we
recommend the monographs \cite{DZ,Vaism}.

\begin{definition}
A \textbf{Poisson structure} on a smooth manifold $M$ is a Lie bracket $\{\cdot,\cdot\}$ on the space
$C^{\infty}(M)$, satisfying the Leibniz rule\index{Poisson manifold/structure}
\[\{f,gh\}=\{f,h\}g+\{f,g\}h, \ \  f,g,h \in C^{\infty}(M).\]
A Poisson structure can be given also by a bivector $\pi\in\mathfrak{X}^2(M)$, involutive with respect to the
Schouten bracket, i.e.\ $[\pi,\pi]=0$ (for the Schouten bracket, see section \ref{SSNotConv}). The bivector and the
bracket are related by:
\[\{f,g\}=\langle\pi,df\wedge dg\rangle,\ \  f,g \in C^{\infty}(M).\]
\end{definition}

The bivector $\pi\in\mathfrak{X}^2(M)$ induces a map $T^*M\to TM$ denoted by
\index{$\pi^{\sharp}$}\index{bivector}
\[\pi^{\sharp}:T^*M\rmap TM, \ \ \pi^{\sharp}(\alpha):=\pi(\alpha,\cdot).\]

A Poisson structure on $M$ can be thought of as partition of $M$ into symplectic submanifolds; more
precisely:\index{symplectic leaves}
\begin{proposition}[Theorem 2.12 \cite{Vaism}]
Let $(M,\pi)$ be a Poisson manifold. Then $\pi^{\sharp}(T^*M)$ is a completely integrable singular foliation, and the
Poisson structure induces symplectic forms on the leaves. More precisely, $M$ has a partition into immersed,
connected submanifolds, called symplectic leaves, such that the leaf $S$ passing through $x\in M$ satisfies
$T_xS=\pi^{\sharp}(T^*_xM)$ and carries a symplectic structure
\[\omega_{S}:=\pi_{|S}^{-1}\in \Omega^2(S).\]
\end{proposition}

The main ingredient in the proof of the proposition is that $\pi^{\sharp}(T^*M)$ is spanned by Hamiltonian vector
fields. The \textbf{Hamiltonian vector field} of a function $f\in C^{\infty}(M)$ is defined by\index{Hamiltonian vector
field}
\[H_f=\{f,\cdot\}=-[\pi,f] \in\mathfrak{X}(M).\]
Set theoretically, the symplectic leaf through $x$ is the set of points that can be reached starting from $x$ by flow
lines of Hamiltonian vector fields. The Hamiltonian vector fields satisfy $[H_{f},H_g]=H_{\{f,g\}}$ and
$L_{H_f}\pi=0$. In particular they are infinitesimal automorphisms of $(M,\pi)$. A function $f$ for which $H_{f}=0$
is called a \textbf{Casimir} function on $M$. Equivalently, a Casimir is a smooth function constant on the symplectic
leaves.\index{Casimir function}

\subsection{Examples}\label{SSExamplesPoisson}

\subsubsection*{Symplectic manifolds}
Symplectic geometry is the nondegenerate version of Poisson geometry. For a symplectic manifold $(M,\omega)$, the
inverse of the symplectic structure is a Poisson structure and conversely, the inverse of a nondegenerate Poisson
structure $\pi$ (i.e.\ for which $\pi^{\sharp}$ is invertible) is a symplectic structure. In other words, the equation
$d\omega=0$ is equivalent to $[\pi,\pi]=0$, for $\pi:=\omega^{-1}$.

\subsubsection*{Quotients of symplectic manifolds}\index{quotients of symplectic manifolds}
Let $(\Sigma,\omega)$ be a symplectic manifold and let $G$ be a Lie group with a free and proper action on $\Sigma$
by symplectomorphisms. The quotient space inherits a Poisson structure
\[(M,\pi):=(\Sigma,\omega)/G.\] If $p:\Sigma\to M$ denotes the projection, the Poisson bracket on $M$ is such that
\[\{f,g\}\circ p=\{f\circ p, g\circ p\}, \ f,g\in C^{\infty}(M).\]

\subsubsection*{Linear Poisson structures}\index{linear Poisson structure}
Let $(\mathfrak{g},[\cdot,\cdot])$ be a Lie algebra. The dual vector space $\mathfrak{g}^*$ carries a canonical
Poisson structure $\pi_{\mathfrak{g}}$, called a \textbf{linear Poisson structure}. It is defined by
\[\pi_{\mathfrak{g},\xi}:=\xi\circ [\cdot,\cdot]\in \Lambda^2\mathfrak{g}^*\cong \Lambda^2(T_{\xi}\mathfrak{g}^*).\]
Let $\{e_i\}$ be a basis of $\mathfrak{g}$ and let $\{x_i\}$ be the induced coordinates on $\mathfrak{g}^*$. The
corresponding structure constants of $\mathfrak{g}$ are the coefficients $C_{i,j}^k$ in
\[[e_i,e_j]=\sum_{k}C_{i,j}^ke_k.\]
Using these numbers, the bivector $\pi_{\mathfrak{g}}$ is given by
\[\pi_{\mathfrak{g}}=\frac{1}{2}\sum_{i,j,k}C_{i,j}^k x_k\frac{\partial}{\partial x_i}\wedge\frac{\partial}{\partial x_j}.\]
Conversely, if a Poisson structure
\[\pi=\frac{1}{2}\sum_{i,j}\pi_{i,j}(x)\frac{\partial}{\partial x_i}\wedge\frac{\partial}{\partial x_j}\]
on $\mathbb{R}^n$ has linear coefficients, i.e.\ $\pi_{i,j}(x)=\sum_{k}C_{i,j}^k x_k$, then the numbers $C_{i,j}^k$
form the structure constants of a Lie algebra.

The Hamiltonian vector field associated to $X\in \mathfrak{g}$, viewed as a linear function on $\mathfrak{g}^*$, is
the infinitesimal (right) coadjoint action of $X$ on $\mathfrak{g}^*$,
\[H_{X,\xi}=-ad^{*}_X(\xi)=\xi([X,\cdot])\in \mathfrak{g}^*\cong T_{\xi}\mathfrak{g}^*.\]

Let $G$ be a connected Lie group integrating $\mathfrak{g}$. Then $G$ acts by Poisson diffeomorphisms on
$(\mathfrak{g}^*,\pi_{\mathfrak{g}})$, via the (left) coadjoint action
\[G\times \mathfrak{g}^* \rmap \mathfrak{g}^*, \  \  (g,\xi) \mapsto Ad^{*}_{g^{-1}}(\xi):=\xi\circ Ad_{g^{-1}}.\]
For $\xi\in \mathfrak{g}^*$, the coadjoint orbit\index{coadjoint orbit}\index{coadjoint action} through $\xi$ coincides
with the symplectic leaf through $\xi$ and we denote it by $(O_{\xi},\omega_{\xi})$. If $G_{\xi}$ denotes the stabilizer
of $\xi$, then $G/G_{\xi}\cong O_{\xi}$. Consider the (left) $G$-equivariant map
\[p:G\rmap O_{\xi}, \ \ g\mapsto Ad_{g^{-1}}^{*}(\xi).\]
The pullback of $\omega_{\xi}$ by $p$ is given by
\begin{equation}\label{EQ_pull_back_coadjoint}
p^*(\omega_{\xi})=-d\xi^l,
\end{equation}
where $\xi^l_g:=l_{g^{-1}}^*(\xi)$; it is the left invariant extension of $\xi$. To see this, observe that both 2-forms
are left invariant, so it is enough to check the equality on $X,Y\in \mathfrak{g}$. The left side gives
\begin{align*}
p^*(\omega_{\xi})&(X,Y)=\omega_{\xi}(dp_{e}(X),dp_e(Y))=\omega_{\xi}(ad^{*}_X(\xi),ad^{*}_Y(\xi))=\\
&=\omega_{\xi}(\pi_{\mathfrak{g}}^{\sharp}(X),\pi_{\mathfrak{g}}^{\sharp}(Y))=\langle\pi_{\mathfrak{g}}^{\sharp}(Y), \omega_{\xi}^{\sharp}\circ \pi_{\mathfrak{g}}^{\sharp}(X)\rangle=\pi_{\mathfrak{g}}(Y,X)=-\xi([X,Y]).
\end{align*}
Denoting by $X^l,Y^l$ the left invariant extensions of $X,Y$, we have
\[d\xi^l(X^l,Y^l)=L_{X^l}(\xi^l(Y^l))-L_{Y^l}(\xi^l(X^l))-\xi^l[X^l,Y^l]=\xi([X,Y]),\]
where we have used that $\xi^l(Z^l)$ is constant. This implies (\ref{EQ_pull_back_coadjoint}).

\subsection{Maps and submanifolds}

A map between two Poisson manifolds
\[\varphi:(M_1,\pi_1)\rmap (M_2,\pi_2),\]
is called \textbf{a Poisson map}, if $\varphi_*(\pi_1)=\pi_2$, or equivalently, if\index{Poisson map}
\[\varphi^*:(C^{\infty}(M_2),\{\cdot,\cdot\})\rmap (C^{\infty}(M_1),\{\cdot,\cdot\})\]
is a Lie algebra homomorphism.\\

A \textbf{symplectic realization} of a Poisson manifold $(M,\pi)$ is a Poisson map from a symplectic manifold that is a
surjective submersion\index{symplectic realization}
\[\mu:(\Sigma,\omega)\rmap (M,\pi).\]

A submanifold $N$ of $(M,\pi)$ is called a \textbf{Poisson submanifold}, if\index{Poisson submanifold}
\[\pi_{|N}\in \mathfrak{X}^2(N).\]
Equivalently, $N$ is a Poisson submanifold if for every symplectic leaf $S$ of $M$, the intersection $N\cap S$ is open
in $S$. If $N$ satisfies this condition, then $\pi_{|N}$ is a Poisson structure on $N$ such that the inclusion map is
Poisson.\\

A \textbf{Poisson transversal} in $(M,\pi)$ is a submanifold $N\subset M$ that satisfies\index{Poisson transversal}
\begin{equation}\label{EQ_cosymplectic}
TM_{|N}=TN\oplus \pi^{\sharp}(TN^{\circ}),
\end{equation}
where $TN^{\circ}\subset T^*M_{|N}$ denotes the annihilator of $TN$. If $N$ satisfies this condition, then it carries
a canonical Poisson structure $\pi_N$ (see Proposition 1.4 in \cite{Wein}), determined by
\[\pi_N^{\sharp}(\xi_{|TN})=\pi^{\sharp}(\xi),\ \ \forall \ \xi \textrm{ such that } \pi^{\sharp}(\xi)\in TN.\]
Geometrically, Poisson transversality means that $N$ intersects every symplectic leaf $(S,\omega_S)$ transversally
and that $\omega_{S|N\cap S}$ is nondegenerate. The symplectic leaves of $\pi_N$ are the connected components of
these intersections. The standard name used for this notion is \textbf{cosymplectic submanifold}
\cite{CrFe2,Zambon}\index{cosymplectic submanifold}, yet we consider that ``Poisson transversal'' is more
appropriate. For example, a Poisson transversal in a symplectic manifolds is the same as symplectic submanifold, and it
would be awkward to call it a cosymplectic submanifold.

Small enough transversal submanifolds to the symplectic leaves, of complementary dimension, are Poisson transversal:
if $(S,\omega_S)$ a symplectic leaf, $x\in S$ and $N\subset M$ is a submanifold through $x$ such that
\[T_xS\oplus T_xN=T_xM,\]
then, around $x$, $N$ is a Poisson transversal. This holds since (\ref{EQ_cosymplectic}) is an open condition and it is
satisfied at $x$. Weinstein's splitting theorem\index{splitting theorem} \cite{Wein} gives a local decomposition of
$(M,\pi)$ around $x$: there are open neighborhoods of $x$, $M_x\subset M$, $S_x\subset S$ and $N_x\subset N$
(such that it is Poisson transversal), and a Poisson diffeomorphism (which fixes $S_x$ and $N_x$):
\[(M_x,\pi_{|M_x})\cong (S_{x},\omega_{S|S_{x}})\times (N_x,\pi_{N_x}).\]

The lemma below will be useful later on.
\begin{lemma}\label{Lemma_restricting_symplectic_realizations}
Let $\mu:(\Sigma,\omega)\to (M,\pi)$ be a symplectic realization, $N\subset M$ be a submanifold. Let $U\subset N$ be
the open where $N$ is Poisson transversal, and denote by $\Sigma_{U}:=\mu^{-1}(U)$. Then $\Sigma_U$ is the
nondegeneracy locus of $\omega_{|\mu^{-1}(N)}$, and moreover, $\mu$ restricts to a symplectic realization
\[\mu:(\Sigma_U,\omega_{|\Sigma_U})\rmap(U,\pi_{U}).\]
\end{lemma}
\begin{proof}
Denote also by $\Sigma_N:=\mu^{-1}(N)$. For $x\in \Sigma_N$, $\omega_{|\Sigma_N}$ is nondegenerate at $x$ if
and only if the map
\[(\omega^{-1})^{\sharp}:T_x\Sigma^{\circ}_N\rmap T_x\Sigma/T_x\Sigma_N\]
is an isomorphism. Since $d\mu$ induces an isomorphism between the spaces
\[(d\mu): T_x\Sigma/T_x\Sigma_N\rmap T_{\mu(x)}M/T_{\mu(x)}N,\]
and $\mu$ is Poisson, this is equivalent to invertability of the map
\[(d\mu)\circ(\omega^{-1})^{\sharp}\circ (d\mu)^*=(\pi)^{\sharp}:T_{\mu(x)}N^{\circ}\rmap T_{\mu(x)}M/T_{\mu(x)}N,\]
hence to $\mu(x)\in U$. This proves the first part.

Consider $x\in \Sigma_U$, $\xi\in T_{\mu(x)}^*N$ and denote by
\[X_{\xi}:=(\omega_{|\Sigma_U}^{-1})^{\sharp}(\mu^*(\xi))\in T_x\Sigma_U.\]
To prove that $\mu_{|\Sigma_U}$ is Poisson, we have to check that $d\mu(X_{\xi})=\pi_U^{\sharp}(\xi)$. By the
definition of $\pi_U$, there is some $\eta\in T^*_{\mu(x)}M$ such that $\eta_{|TN}=\xi$ and
$\pi^{\sharp}(\eta)=\pi_U^{\sharp}(\xi)$. Since $\mu$ is Poisson, the vector
$X_{\eta}:=(\omega^{-1})^{\sharp}(\mu^*(\eta))$ projects to $\pi^{\sharp}(\eta)$. In particular, $X_{\eta}\in
T_x\Sigma_U$, thus
\[\mu^*(\xi)=\mu^*(\eta)_{|\Sigma_U}=(\omega^{\sharp}X_{\eta})_{|\Sigma_U}=(\omega_{|\Sigma_U})^{\sharp}(X_{\eta}).\]
This shows that $X_{\eta}=X_{\xi}$, hence $X_{\xi}$ also projects to $\pi^{\sharp}(\eta)=\pi_U^{\sharp}(\xi)$.
\end{proof}

\subsection{Poisson cohomology}\label{SSPoissonCoho}

The notions of Casimir function and of Hamiltonian vector field fit into the framework of Poisson
cohomology.\index{cohomology, Poisson}

\begin{definition}
The \textbf{Poisson cohomology} of a Poisson manifold $(M,\pi)$ is the cohomology computed by the complex
\[(\mathfrak{X}^{\bullet}(M),d_{\pi}), \ \ d_{\pi}:=[\pi,\cdot].\]
The resulting Poisson cohomology groups are denoted by $H^{\bullet}_{\pi}(M)$.
\end{definition}

In low degrees these groups have a geometric interpretation:
\begin{itemize}
\item $H^0_{\pi}(M)$ is the space of Casimir functions.
\item $H^1_{\pi}(M)$ is the space of infinitesimal automorphisms of $\pi$ modulo Hamiltonian vector fields.
\item $H^2_{\pi}(M)$ is the space of infinitesimal deformations of $\pi$ modulo ``geometric deformations'', i.e.\ deformations coming from
diffeomorphisms.
\end{itemize}

To explain the interpretation given to $H^2_{\pi}(M)$, let $\pi_t$ be a family of Poisson structures, with $\pi_0=\pi$.
Taking the derivative at $t=0$ in $[\pi_t,\pi_t]=0$, we obtain that $d_\pi(\dot{\pi}_0)=0$. Now if
$\pi_t=\Phi_{t}^{*}(\pi)$, where $\Phi_t$ is a family of diffeomorphisms of $M$, with $\Phi_0=\textrm{Id}_M$, then
$\dot{\pi}_0=d_{\pi}(\dot{\Phi}_0)$. So $H^2_{\pi}(M)$ has the heuristical interpretation of being the ``tangent
space'' at $\pi$ to the moduli space of all Poisson structures on $M$.

For example, every Casimir function $f\in H^0_{\pi}(M)$ gives a class $[f\pi]\in H^{2}_{\pi}(M)$. This class
corresponds to the deformation of $\pi$ given by $\pi_t=e^{tf}\pi$. Geometrically, the leaves of $\pi_t$ are the same
as those of $\pi$, but the symplectic form on $S$ for $\pi_t$ is $e^{-tf}\omega_S$ (these deformations are relevant in
chapter \ref{ChDef}).

The space $\mathfrak{X}^{\bullet}(M)$, of multivector fields on $M$, can be regarded as the space of
multi-derivations of $C^{\infty}(M)$, i.e.\ skew-symmetric maps
\[C^{\infty}(M)\times\ldots \times C^{\infty}(M)\rmap C^{\infty}(M),\]
satisfying the Leibniz rule in each entry. Then, the Poisson differential can be given also by the formula\index{Koszul
formula}
\begin{align*}
d_{\pi}W(f_0, \ldots , f_{p})=&\sum_{i}(-1)^{i} \{f_i,W(f_0, \ldots , \widehat{f}_i, \ldots , f_{p})\}+\\
& + \sum_{i< j} (-1)^{i+j}W(\{f_i, f_j\}, \ldots , \widehat{f}_i, \ldots, \widehat{f}_j, \ldots , f_{p}),
\end{align*}
for $W\in\mathfrak{X}^{p}(M)$ and $f_0,\ldots,f_p\in C^{\infty}(M)$. So, Poisson cohomology can be computed using
a subcomplex of the Eilenberg-Chevalley complex of the Lie algebra $(C^{\infty}(M),\{\cdot,\cdot\})$.

The Poisson cohomology of $(M,\pi)$ is related to de Rham cohomology of $M$ by the chain map
\begin{equation}\label{EQ_Chain_deRham_Poisson}
(-1)^{\bullet+1}\Lambda^{\bullet}\pi^{\sharp}:(\Omega^{\bullet}(M),d)\rmap (\mathfrak{X}^{\bullet}(M),d_{\pi}).
\end{equation}
The induced map in cohomology we denote by
\[H(\pi^{\sharp}):H^{\bullet}(M)\rmap H^{\bullet}_{\pi}(M).\]

When $\pi$ is nondegenerate, i.e.\ $\pi=\omega^{-1}$ for a symplectic structure $\omega$, then
(\ref{EQ_Chain_deRham_Poisson}) is an isomorphism. So, in this case, the Poisson cohomology is isomorphic to the de
Rham cohomology.

In the next subsection we explain the infinitesimal deformations of $\pi$ coming from elements in
$H(\pi^{\sharp})(H^2(M))$.

\subsection{Gauge transformations}

Let $\omega$ be a closed 2-form on a Poisson manifold $(M,\pi)$. On the open where
\begin{equation}\label{EQ_map_gauge}
\textrm{Id}+\omega^{\sharp}\circ\pi^{\sharp}
\end{equation}
is invertible, one defines a new Poisson structure, denoted by $\pi^{\omega}$, called the \textbf{gauge
transformation} of $\pi$ by $\omega$\index{gauge transformation}
\[\pi^{\omega,\sharp}:=\pi^{\sharp}\circ (\textrm{Id}+\omega^{\sharp}\circ\pi^{\sharp} )^{-1}.\]
The fact that $\pi^{\omega}$ is Poisson will be explained when we discuss gauge transformations of Dirac structures.
Geometrically, the leaves of $\pi^{\omega}$ are the same as those of $\pi$, but the symplectic form on $S$ for
$\pi^{\omega}$ is $\omega_S+\omega_{|S}$.

The paths $t\mapsto \pi^{t\omega}$ is the deformation corresponding to the class
\[H(\pi^{\sharp})[\omega]\in H^2_{\pi}(M).\]
To see this, we differentiate the equation
\[\pi^{t\omega,\sharp}\circ (\textrm{Id}+t \omega^{\sharp}\circ \pi^{\sharp})=\pi^{\sharp},\]
and obtain
\[\frac{d}{dt}(\pi^{t\omega,\sharp})\circ (\textrm{Id}+t \omega^{\sharp}\circ \pi^{\sharp})=-\pi^{t\omega,\sharp}\circ\omega^{\sharp}\circ \pi^{\sharp}.\]
Multiplying with $(\textrm{Id}+t \omega^{\sharp}\circ \pi^{\sharp})^{-1}$ from the right, this gives
\begin{equation}\label{EQ_gauge_time_derivative}
\frac{d}{dt}\pi^{t\omega}=-(\Lambda^2\pi^{t\omega,\sharp})(\omega).
\end{equation}
At $t=0$, we obtain that $[\frac{d}{dt}\pi^{t\omega}_{|0}]=H(\pi^{\sharp})[\omega]$.\\

By skew-symmetry of $\pi$ and $\omega$, the dual of the map (\ref{EQ_map_gauge}) is
$(\textrm{Id}+\pi^{\sharp}\omega^{\sharp})$; thus, if one is invertible, then so is the other. This plays a role in the
following lemma, whose proof we explain in the next section.

\begin{lemma}\label{Lemma_gauge_cohomology}
Let $U$ be the open on which $(\mathrm{Id}+\omega^{\sharp}\pi^{\sharp})$ is invertible. The Poisson manifolds
$(U,\pi^{\omega})$ and $(U,\pi)$ have isomorphic cohomology; an isomorphism between their complexes is the chain
map
\[\Lambda^{\bullet}(\mathrm{Id}+\pi^{\sharp}\omega^{\sharp}):(\mathfrak{X}^{\bullet}(U),d_{\pi^{\omega}})\rmap (\mathfrak{X}^{\bullet}(U),d_{\pi})\]
\end{lemma}
We give now the Poisson geometric version of the Moser Lemma from symplectic geometry.

\begin{lemma}\label{Lemma_Moser_Poisson_vs_Symplectic}
Assume that the class $H(\pi^{\sharp})[\omega]\in H^2_{\pi}(M)$ vanishes, and let $X\in\mathfrak{X}(M)$ be a
solution to
\[d_{\pi}(X)=-(\Lambda^2\pi^{\sharp})(\omega).\]
The flow $\Phi_t$ of the time dependent vector field
\[X_t=(\mathrm{Id}+t\pi^{\sharp}\omega^{\sharp})^{-1}(X)\]
sends $\pi$ to $\pi^{t\omega}$, whenever it is defined. In particular, for $\omega=d\alpha$,
\[X_t=X=\pi^{\sharp}(\alpha).\]
\end{lemma}
\begin{proof}
Using Lemma \ref{Lemma_gauge_cohomology} and that
$(\textrm{Id}+t\pi^{\sharp}\omega^{\sharp})^{-1}\pi^{\sharp}=\pi^{t\omega,\sharp}$, we obtain
\[L_{X_t}\pi^{t\omega}=-d_{\pi^{t\omega}}X_t=-(\Lambda^2(\textrm{Id}+t\pi^{\sharp}\omega^{\sharp})^{-1})(d_{\pi}X)=(\Lambda^2\pi^{t\omega,\sharp})(\omega).\]
Using (\ref{EQ_gauge_time_derivative}), we obtain that $\Phi_t^*(\pi^{t\omega})$ is constant
\[\frac{d}{dt}\Phi_t^*(\pi^{t\omega})=\Phi_t^*(L_{X_t}\pi^{t\omega}+\frac{d}{dt}\pi^{t\omega})=0.\]
Since $\Phi_0=\textrm{Id}_M$, the result follows.
\end{proof}

\subsection{Dirac structures}\label{subsection_Dirac}
Dirac geometry is the common framework for several geometric structures: foliations, closed 2-forms and Poisson
bivectors. Our main application for Dirac structures is to the study of the local model for a Poisson manifold around a
symplectic leaf, in chapter \ref{ChNormalForms}; the resulting structure is globally Dirac and only on an open around
the leaf it is Poisson. As a reference to Dirac geometry we indicate \cite{BR,Courant}.

\subsubsection*{Definition and first properties}

Let $M$ be a manifold. The vector bundle $TM\oplus T^*M$ carries a symmetric, nondegenerate paring of signature
$(\mathrm{dim}(M),\mathrm{dim}(M))$,
\[(X+\xi,Y+\eta):=\xi(X)+\eta(Y), \ X+\xi,Y+\eta\in T_xM\oplus T_x^*M,\]
and its space of sections carries the so-called Dorfman bracket\index{Dorfman bracket}
\begin{equation}\label{EQ_Dorfman_bracket}
[X+\xi,Y+\eta]_D:=[X,Y]+L_X\eta-\iota_{Y}d\xi,
\end{equation}
for $X,Y\in\mathfrak{X}(M)$ and $\eta,\xi\in\Omega^1(M)$. This bracket satisfies the Jacobi identity, but fails to be
skew-symmetric.

\begin{definition}
A \textbf{Dirac structure} on $M$ is a maximal isotropic subbundle $L\subset (TM\oplus T^*M,(\cdot,\cdot))$, which
is involutive with respect to the Dorfman bracket.\index{Dirac structure}
\end{definition}
Note that the failure of the bracket of being skew-symmetric is given by
\[[X+\xi,Y+\eta]_D+[Y+\eta, X+\xi]_D=d(X+\xi,Y+\eta).\]
Therefore, the space of sections of a Dirac structure is a Lie algebra.

\subsubsection*{Presymplectic leaves}

Let $L$ be a Dirac structure on $M$. Denote by $p_T$ and $p_{T}^*$ the projections of $TM\oplus T^*M$ onto $TM$
and $T^*M$ respectively. As for Poisson structures, the singular distribution $p_T(L)$ is involutive and integrates to a
partition of $M$ into immersed submanifolds which carry closed 2-forms. These submanifolds are called
\textbf{presymplectic leaves}\index{presymplectic leaves} and the closed 2-form $\omega_S$ on the leaf $S\subset
M$ is given at $x$ by
\[\omega_S(X,Y):=\xi(Y)=-\eta(X), \ X+\xi,Y+\eta\in L_x.\]

\subsubsection*{Foliations}

A smooth distribution $B\subset TM$ induces a maximal isotropic subspace \[L_{B}:=B\oplus B^{\circ}\subset
TM\oplus T^*M,\] and $L_B$ is a Dirac structure if and only if $B$ is involutive. For such Dirac structures $L_B$, the
presymplectic leaves are the leaves of the foliation corresponding to $B$ and the closed 2-forms vanish. This class of
examples corresponds to Dirac structures $L$ such that $L= p_T(L)\oplus p_{T^*}(L)$.\index{foliation}

\subsubsection*{Poisson structures}

A bivector $\pi\in \mathfrak{X}^2(M)$ induces a maximal isotropic subspace
\[L_\pi:=\{\pi^{\sharp}(\xi)+\xi | \xi\in T^*M \},\]
and the condition that $L_{\pi}$ is Dirac is equivalent to $\pi$ being Poisson. In this case, the presymplectic leaves of
$L_{\pi}$ are the symplectic leaves of $\pi$ with the corresponding symplectic structures.

Under the isomorphism
\[T^*M\diffto L_{\pi},\ \ \xi\mapsto \xi+\pi^{\sharp}(\xi),\]
the Dorfman bracket becomes
\[[\alpha,\beta]_{\pi}:=L_{\pi^{\sharp}\alpha}\beta-L_{\pi^{\sharp}\beta}\alpha- d\pi(\alpha,\beta).\]
Poisson manifolds are Dirac structures $L$ satisfying $p_{T^*}(L)=T^*M$.

\begin{definition}
The \textbf{Poisson support} of a Dirac structure $L$ is
\[\mathrm{supp}(L):=\{x\in M | p_{T^*}(L_x)=T^*_xM\}.\]
It is maximal open set on which $L$ is Poisson.\index{Poisson support}
\end{definition}

\subsubsection*{Closed 2-forms}

A 2-form $\omega\in \Omega^2(M)$ gives a maximal isotropic subspace
\[L_\omega:=\{X+\iota_X\omega| X\in TM \},\]
which is Dirac if and only if $\omega$ is closed. In this case, if $M$ is connected, $(M,\omega)$ is the only
presymplectic leaf. This class of examples corresponds to Dirac structures $L$ such that $p_{T}(L)=TM$.


\subsubsection*{Dirac maps}

For Dirac structures there are both pullback and push forward maps, generalizing the pullback of foliations and
2-forms, and the push forward of Poisson bivectors. A map between Dirac manifolds
\[\varphi:(M_1,L_1)\rmap (M_2,L_2)\]
is a \textbf{forward Dirac map}, if for all $x\in M_1$\index{forward Dirac map},
\[L_{2,\varphi(x)}=\varphi_*(L_{1,x}):=\{\varphi_*(X)+\xi | X+\varphi^*(\xi)\in L_{1,x} \}.\]
The map $\varphi$ is called a \textbf{backward Dirac map}, if for all $x\in M_1$,\index{backward Dirac map}
\[L_{1,x}=\varphi^*(L_{2,\varphi(x)}):=\{X+\varphi^*(\xi) | \varphi_*(X)+\xi \in L_{2,\varphi(x)} \}.\]
A Poisson map is a forward Dirac map; the inclusion of a Poisson transversal of a Poisson manifold is a backward Dirac
map.

\subsubsection*{Rescaling}

We have also the operation of rescaling of Dirac structures, defined by
\[tL:=\{tX+\xi| X+\xi\in L \}, \ t\neq 0.\]
Foliations are the only Dirac structures invariant under rescaling.

\subsubsection*{Products}

Two Dirac structures $L_1$ and $L_2$ on $M$, such that $L_1+L_2=TM\oplus T^*M$ are called \textbf{transverse}
Dirac structures. If $L_1$ and $L_2$ satisfy $p_T(L_1)+p_T(L_2)=TM$, then we say that $L_1$ and $L_2$ are
$t$-\textbf{transverse}.\index{transverse Dirac structures}

Define the \textbf{product}\index{product of Dirac structures} of $t$-transverse Dirac structures $L_1$ and $L_2$ by
\[L_1*L_2:=\{X+\xi_1+\xi_2 | X+\xi_1\in L_1, X+\xi_2\in L_2\}.\]
Notice that, by $t$-transversality, the map
\[L_1\oplus L_2\rmap TM, \ (X_1+\xi_1)\oplus(X_2+\xi_2)\mapsto X_1-X_2,\]
is surjective, thus its kernel $K$ is a subbundle. We claim that the map
\[K\rmap L_1*L_2, \ (X+\xi_1)\oplus(X+\xi_2)\mapsto X+\xi_1+\xi_2,\]
is an isomorphism, therefore $L_1*L_2$ is a (smooth) subbundle of $TM\oplus T^*M$. The map is surjective by
definition. An element in its kernel is of the form $(\xi,-\xi)$, for $\xi\in L_1\cap L_2\cap T^*M$. Since $L_1$ and $L_2$
are isotropic,
\[\{0\}=(\xi,L_1+L_2)=\xi(p_T(L_1)+p_T(L_2))=\xi(TM),\] therefore $\xi=0$, and this proves injectivity. This
argument also shows that $L_1*L_2$ has the right dimension. The fact that it is isotropic and involutivity follow easily.
We conclude that $L_1*L_2$ is a new Dirac structure.

This operation has the usual algebraic properties
\begin{align*}
& TM*L=L*TM=L,\\
& L_1*L_2=L_2*L_1,\\
& L_1*(L_2*L_3)=(L_1*L_2)*L_3,
\end{align*}
whenever these products are defined.

Geometrically, taking the product of $L_1$ and $L_2$ amounts to intersecting the underlying singular foliations and
adding the restrictions of the presymplectic structures. Observe that the product has the following property
\[L_1*L_2\ \textrm{is Poisson}\ \Leftrightarrow \ L_1 \ \textrm{and}\ -L_2 \ \textrm{are transverse}.\]

This product generalizes several constructions:

\noindent $\bullet$ Addition of $2$-forms.

\noindent $\bullet$ Intersection of transverse foliations.

\noindent $\bullet$ The \textbf{gauge transformation}\index{gauge transformation} of a Dirac structure $L$ by a
closed 2-form $\omega$
\[L^{\omega}:=L* L_{\omega}=\{ X+\xi+\iota_X\omega | X+\xi\in L\}.\]
Observe that for $L_{\pi}$, where $\pi$ is a Poisson structure, the Poisson support of $L_{\pi}^{\omega}$ is given by
the open where the map (\ref{EQ_map_gauge}) is invertible. On this open,
\[L_{\pi}^{\omega}=L_{\pi^{\omega}},\]
and this proves that $\pi^{\omega}$ is indeed Poisson.

\noindent$\bullet$ The product of Poisson structures from \cite{LWX}. Transversality of the Dirac structures
$L_{\pi_1}$ and $-L_{\pi_2}$, with $\pi_1$ and $\pi_2$ Poisson, is equivalent to nondegeneracy of $\pi_1+\pi_2$. In
this case, we obtain a new Poisson structure $\pi_1*\pi_2$, which corresponds to the product
\[L_{\pi_1}*L_{\pi_2}=L_{\pi_1*\pi_2}.\] More explicitly, this Poisson tensor is given by
\[(\pi_1*\pi_2)^{\sharp}:=\pi_1^{\sharp}\circ (\pi_1^{\sharp}+\pi_2^{\sharp})^{-1}\circ \pi_2^{\sharp}.\]

\subsubsection*{Cohomology}

The cohomology of a Dirac structure $L$, is computed by the complex
\[(\Gamma(\Lambda^{\bullet}L^*),d_L),\]
where the differential is given by a Koszul-type formula\index{cohomology, Dirac}\index{Koszul formula}
\begin{align*}
d_{L}\alpha(X_0, \ldots , X_{p})=&\sum_{i}(-1)^{i} L_{p_T(X_i)}(\alpha(X_0, \ldots , \widehat{X}_i, \ldots , X_{p}))+\\
& + \sum_{i< j} (-1)^{i+j}\alpha([X_i, X_j]_D, \ldots , \widehat{X}_i, \ldots, \widehat{X}_j, \ldots , X_{p}).
\end{align*}

For a Poisson structure $\pi$, the cohomology of the associated Dirac structure $L_{\pi}$ is isomorphic to the Poisson
cohomology. The natural
identification $L_{\pi}^*\cong TM$ induces an isomorphism between the complexes.\\

For $\omega$ a closed $2$-form and $L$ a Dirac structure, it is easy to see that the isomorphism of bundles
\begin{equation}\label{EQ_gauge_dirac_map}
e_{\omega}: L\diffto L^{\omega}, \ \ X+\xi\mapsto X+\omega^{\sharp}(X)+\xi,
\end{equation}
induces an isomorphism between the Lie algebras
\[e_{\omega}:(\Gamma(L),[\cdot,\cdot]_D)\diffto (\Gamma(L^{\omega}),[\cdot,\cdot]_D).\]
Therefore, the dual map induces a chain-isomorphism
\[\Lambda^{\bullet} e_{\omega}^*:(\Gamma(\Lambda^{\bullet}L^{\omega,*}),d_{L^{\omega}})\rmap (\Gamma(\Lambda^{\bullet}L^*),d_{L}).\]
For a Poisson structure $\pi$, identifying $L_{\pi}^*\cong TM\cong L_{\pi^{\omega}}^*$, this map becomes
\[e_{\omega}^*=(\textrm{Id}+\omega^{\sharp}\pi^{\sharp})^*=(\textrm{Id}+\pi^{\sharp}\omega^{\sharp}).\]
This remark implies Lemma \ref{Lemma_gauge_cohomology}.

\subsection{Contravariant geometry}\label{SubSection_Contravariant}

The basic idea of contravariant geometry\index{contravariant geometry} in Poisson geometry is to replace the tangent
bundle $TM$ of a Poisson manifold $(M,\pi)$ by the cotangent bundle $T^*M$. The two are related by the bundle map
$\pi^{\sharp}$. The main structure that makes everything work is the Lie bracket $[\cdot,\cdot]_{\pi}$ on
$\Omega^1(M)$, which is the contravariant analogue of the Lie bracket on vector fields
\[[\alpha,\beta]_{\pi}:=L_{\pi^{\sharp}\alpha}\beta-L_{\pi^{\sharp}\beta}\alpha- d\pi(\alpha,\beta).\]
The Lie bracket $[\cdot,\cdot]_{\pi}$ has the following properties:
\begin{itemize}
\item the map $\pi^{\sharp}$ is a Lie algebra map
\begin{equation}\label{EQ_bracket_relation}
\pi^{\sharp}([\alpha,\beta]_{\pi})=[\pi^{\sharp}(\alpha),\pi^{\sharp}(\beta)];
\end{equation}
\item the de Rham differential is a Lie algebra map
\[d\{f, g\}=[df, dg]_{\pi};\]
\item it satisfies the Leibniz identity
\[ [\alpha, f\beta]_{\pi}= f[\alpha, \beta]_{\pi}+ L_{\pi^{\sharp}\alpha}(f) \beta.\]
\end{itemize}
In other words $(T^*M, [\cdot, \cdot]_{\pi}, \pi^{\sharp})$ is a Lie algebroid, called \textbf{the cotangent Lie
algebroid}\index{cotangent Lie algebroid} of $(M,\pi)$, and contravariant geometry is the geometry associated to this
Lie algebroid (this notion will be discussed in chapter \ref{CHLieAlgLieGroupoids}).

Here are the contravariant versions of some usual notions (see \cite{CrFe2, Fernandes}).

\vspace*{.1in} A \textbf{contravariant connection}\index{contravariant connection} on a vector bundle $E$ over
$(M,\pi)$ is a bilinear map
\[\nabla: \Omega^1(M)\times \Gamma(E)\rmap \Gamma(E), \ \ (\alpha, s)\mapsto \nabla_{\alpha}(s)\]
satisfying
\[ \nabla_{f\alpha}(s)= f\nabla_{\alpha}(s), \ \ \nabla_{\alpha}(fs)= f\nabla_{\alpha}(s)+ L_{\pi^{\sharp}\alpha}(f) s.\]
The standard operations with connections (duals, tensor products, etc.) have obvious contravariant versions. Note also that any classical
connection $\nabla$ on $E$ induces a contravariant connection on $E$
\begin{equation}\label{EQ_class_conn_cont_conn}
 \nabla_{\alpha}= \nabla_{\pi^{\sharp}\alpha}.
\end{equation}

A \textbf{cotangent path}\index{cotangent path} in $(M,\pi)$ is a path $\gamma: [0,1]\to M$ with a ``contravariant
speed'', that is a map $a:[0,1]\to T^*M$ sitting above $\gamma$ that satisfies
\[ \pi^{\sharp}(a(t))= \frac{d\gamma}{dt}(t).\]

Given a contravariant connection $\nabla$ on a vector bundle $E$, one has a well-defined notion of \textbf{derivative of
sections along cotangent paths}: Given a cotangent path $(a, \gamma)$ and a path $u: [0, 1]\to E$ sitting above
$\gamma$, $\nabla_{a}(u)$ is a new path in $E$ sitting above $\gamma$. Writing $u(t)= s_t(\gamma(t))$ for some time
dependent section $s_t$ of $E$,
\[\nabla_{a}(u)= \nabla_{a}(s_t)(x)+ \frac{d  s_t}{dt}(x), \ \ \ \textrm{at}\ x= \gamma(t).\]

Given a contravariant connection $\nabla$ on $T^*M$, the \textbf{contravariant torsion}\index{contravariant torsion}
of $\nabla$ is the tensor $T_{\nabla}$ defined by
\[ T_{\nabla}(\alpha, \beta)= \nabla_{\alpha}(\beta)- \nabla_{\beta}(\alpha)- [\alpha, \beta]_{\pi}.\]

Given a metric $g$ on $T^*M$, one has an associated \textbf{contravariant Levi-Civita
connection}\index{contravariant Levi-Civita connection}, which is the unique contravariant metric connection
$\nabla^{g}$ on $T^*M$ whose contravariant torsion vanishes:
\[g(\nabla^g_{\alpha}\beta,\gamma)+g(\beta,\nabla^g_{\alpha}\gamma)=L_{\pi^\sharp\alpha}g(\beta,\gamma),\ \ T_{\nabla^g}=0.\]
The corresponding \textbf{contravariant geodesics}\index{contravariant geodesics} are defined as the (cotangent)
paths $a$ satisfying $\nabla_{a}a= 0$. They are the integral curves of a vector field $\mathcal{V}_{\pi}^{g}$ on
$T^*M$, called the \textbf{contravariant geodesic vector field}\index{contravariant geodesic vector field}. In local
coordinates $(x, y)$, where $x$ are the coordinates on $M$ and $y$ the coordinates on the fibers,
\begin{equation}\label{EQ_local_spray}
\mathcal{V}_{\pi}^g(x, y)= \sum_{i, j} \pi_{i, j}(x)y_i \frac{\partial}{\partial x_j}- \sum_{i,j,k} \Gamma_{i, j}^{k}(x) y_iy_j \frac{\partial}{\partial y_k},
\end{equation}
where $\Gamma_{i, j}^k(x)$ are the coefficients in
\[\nabla^g_{dx_i}(dx_j)=\sum_k \Gamma_{i,j}^{k}(x) dx_k.\]
Geodesics and the geodesic vector field are actually defined for any contravariant connection $\nabla$ on $T^*M$, not necessarily of metric type
(for example, the connection from (\ref{EQ_class_conn_cont_conn})).

The geodesic vector is an example of a Poisson spray.
\begin{definition} A \textbf{Poisson spray}\index{Poisson spray} or \textbf{contravariant spray}\index{contravariant spray} on a Poisson manifold $(M,\pi)$ is a vector field $\mathcal{V}_{\pi}$ on $T^*M$
satisfying:
\begin{enumerate}
\item[(1)] $(dp)_{\xi}(\mathcal{V}_{\pi, \xi})= \pi^{\sharp}(\xi)$ for all $\xi\in T^*M$,
\item[(2)] $\mu_{t}^{*}(\mathcal{V}_{\pi}) = t\mathcal{V}_{\pi}$ for all $t> 0$,
\end{enumerate}
where $p:T^*M\to M$ is the canonical projection and $\mu_t: T^*M\to T^*M$ is the fiberwise multiplication by
$t>0$.
\end{definition}

Condition (1) means that the integral curves of $\mathcal{V}_{\pi}$ are cotangent paths, it also appears in
\cite{WeinLagrmech} under the name ``second order differential equation''. Actually, this definition is equivalent to
the local expression (\ref{EQ_local_spray}).

The discussion above implies:
\begin{corollary}
For every Poisson manifold $(M,\pi)$, contravariant sprays exist. Any contravariant geodesic vector field is a contravariant spray.
\end{corollary}

Recall also (cf. e.g. \cite{CrFe2}) that any classical connection $\nabla$ induces two contravariant connections, one on
$TM$ and one on $T^*M$, both denoted $\overline{\nabla}$:
\begin{equation}\label{EQ_two_conn}
\overline{\nabla}_{\alpha}(X)= \pi^{\sharp}\nabla_{X}(\alpha)+ [\pi^{\sharp}(\alpha),X], \ \overline{\nabla}_{\alpha}(\beta)= \nabla_{\pi^{\sharp}\beta}(\alpha)+ [\alpha, \beta]_{\pi}.
\end{equation}
The two are related as in the lemma below, a direct consequence of (\ref{EQ_bracket_relation}). 

\begin{lemma} For any classical connection $\nabla$,
\begin{equation}\label{EQ_conn_pisharp_rel}
\overline{\nabla}_{\alpha}(\pi^{\sharp}(\beta))=\pi^{\sharp}(\overline{\nabla}_{\alpha}(\beta)).
\end{equation}
\end{lemma}

The next lemma shows that if the classical connection $\nabla$ is a torsion-free, then the two contravariant connections are in duality.

\begin{lemma}\label{Lemma_ConjConn} If $\nabla$ is a torsion-free connection, then the connections from (\ref{EQ_two_conn}) satisfy the duality
relation
\[\langle\overline{\nabla}_{\alpha}(\beta),X\rangle+\langle \beta,\overline{\nabla}_{\alpha}(X)\rangle=L_{\pi^{\sharp}\alpha}\langle \beta, X\rangle.\]
\end{lemma}

\begin{proof} Using the formulas for $\overline{\nabla}$, the left hand side becomes
\begin{equation}\label{EQ_Aux0}
\langle \nabla_{\pi^{\sharp}\beta}(\alpha)+ [\alpha, \beta]_{\pi}, X\rangle+ \langle \beta, \pi^{\sharp}\nabla_{X}(\alpha)+ [\pi^{\sharp}\alpha,
X]\rangle.
\end{equation}
For the two terms involving $\nabla$ we find
\begin{align*}
\langle \nabla_{\pi^{\sharp}\beta}(\alpha)&, X\rangle + \langle \beta, \pi^{\sharp}\nabla_{X}(\alpha)\rangle = \langle \nabla_{\pi^{\sharp}\beta}(\alpha), X\rangle - \langle \nabla_{X}(\alpha), \pi^{\sharp}\beta \rangle =\\
&= L_{\pi^{\sharp}\beta}\langle \alpha, X \rangle- \langle \alpha, \nabla_{\pi^{\sharp}\beta}(X) \rangle- L_{X}\langle \alpha, \pi^{\sharp}\beta \rangle + \langle \alpha, \nabla_{X}(\pi^{\sharp}\beta)\rangle =\\
&= L_{\pi^{\sharp}\beta}\langle \alpha, X \rangle+L_{X}(\pi(\alpha,\beta))+ \langle \alpha, [X,\pi^{\sharp}\beta]\rangle,
\end{align*}
where we have used the antisymmetry of $\pi$, and then passed from $\nabla$ on $T^*M$ to its dual on $TM$, and used that $\nabla$ is
torsion-free. For the remaining two term in (\ref{EQ_Aux0}), using the definition of $[\cdot,\cdot]_{\pi}$, we find
\begin{align*}
\langle [\alpha, \beta]_{\pi}, X\rangle+&\langle \beta, [\pi^{\sharp}\alpha, X]\rangle =\\
&= \langle L_{\pi^{\sharp}\alpha}(\beta)-L_{\pi^{\sharp}\beta}(\alpha)- d\pi(\alpha, \beta), X\rangle+\langle \beta, L_{\pi^{\sharp}\alpha} X\rangle = \\
&= L_{\pi^{\sharp}\alpha}\langle \beta, X\rangle - L_{\pi^{\sharp}\beta}\langle \alpha,
X\rangle + \langle \alpha, [\pi^{\sharp}\beta, X]\rangle - L_X(\pi(\alpha,
\beta)).
\end{align*}
Adding up, the conclusion follows.
\end{proof}

This duality can be expressed also using the covariant derivatives.

\begin{lemma}\label{ConjConn2} If $\nabla$ is a torsion-free connection and $a:[0,1]\to T^*M$ a cotangent path with base path $\gamma$, then for any smooth paths $\theta$ in $T^*M$ and $v$ in $TM$, both above
$\gamma$, the following identity holds:
\[\langle\overline{\nabla}_{a}(\theta),v\rangle+\langle \theta,\overline{\nabla}_{a}(v)\rangle=\frac{d}{dt}\langle \theta, v\rangle.\]
\end{lemma}

\begin{proof}
Choose a time dependent 1-form $A= A(t, x)$ such that $a(t)= A(t, \gamma(t))$ and similarly a time dependent 1-form $\Theta$ corresponding to
$\theta$, and a time dependent vector field $V$ corresponding to $v$. Using the definition of the derivatives $\overline{\nabla}_{a}$ along
cotangent paths and then Lemma \ref{Lemma_ConjConn}, we obtain
\begin{align*}
\langle\overline{\nabla}_{a}(\theta),v\rangle(t)+\langle \theta,\overline{\nabla}_{a}(v)\rangle(t)&= \langle\overline{\nabla}_{A}(\Theta)+\frac{d\Theta}{dt},V\rangle(t,\gamma(t))+\\
+\langle \Theta,\overline{\nabla}_{A}(V)+\frac{dV}{dt}\rangle(t,\gamma(t))&=L_{\pi^{\sharp}A}\langle\Theta,V\rangle(t,\gamma(t))+\frac{d}{dt}\langle \Theta,V\rangle(t,\gamma(t)).
\end{align*}
Since $\pi^{\sharp}A(t,\gamma(t))= \frac{d\gamma}{dt}(t)$, we obtain the identity from the statement.
\end{proof}

\section{Existence of symplectic realizations: a new approach}\label{Ontheexistence}

The content of this section was published in \cite{CrMa_symplectic}.

\subsection{Statement of Theorem 0}

Recall that a \textbf{symplectic realization} of a Poisson manifold $(M, \pi)$ is a symplectic manifold $(\Sigma,
\omega)$ together with a surjective Poisson submersion
\[ \mu: (\Sigma, \omega) \rmap (M, \pi).\]
Although the existence of symplectic realizations is a fundamental result in Poisson geometry, the known proofs are
rather involved. Originally, the local result was proven in \cite{Wein} and a gluing argument was provided in
\cite{DWC}; the same procedure appears in \cite{Kar}. The path approach to symplectic groupoids \cite{CaFe,
CrFe2} gives a different proof. Here we present a direct, global, finite dimensional proof, based on the philosophy of
contravariant geometry.\index{symplectic realization}

\begin{theoremzero} 
Let $(M, \pi)$ be a Poisson manifold and $\mathcal{V}_{\pi}$ a contravariant spray with flow $\varphi_t$. There exists
an open $\mathcal{U}\subset T^*M$ around the zero section so that
\[ \omega:= \int_{0}^{1} (\varphi_t)^*\omega_{\textrm{can}} dt\]
is a symplectic structure on $\mathcal{U}$ and the canonical projection $p: (\mathcal{U}, \omega)\to (M, \pi)$ is a symplectic realization.
\end{theoremzero}

When $M$ is an open in $\mathbb{R}^n$ and $\pi_{i, j}$ are the components of $\pi$, the simplest contravariant
spray is $\mathcal{V}_{\pi}(x, y)= \sum_{i, j} \pi_{i, j}(x) y_i\frac{\partial}{\partial x_j}$, where $x$ are the
coordinates on $M$ and $(x, y)$ the induced coordinates on $T^*M$. It is not difficult to see that the resulting
$\omega$ coincides with the one constructed by Weinstein \cite{Wein}.

One may expect that the proof is ``just a computation''. Although that is true in principle, the computation is more
subtle than expected. In particular, we will make use of the principle of ``contravariant geometry'' which is intrinsic to
Poisson geometry. The fact that the proof cannot be so trivial and hides some interesting geometry was already
observed in the local case by Weinstein \cite{Wein}: the notion of contravariant spray, its existence, and the formula
for $\omega$ (giving a symplectic form on an small enough $\mathcal{U}$) all make sense for any bivector $\pi$, even
if it is not Poisson. But the fact that the push-down of (the inverse of) $\omega$ is $\pi$ can hold only for Poisson
bivectors. Nowadays, with all the insight we have gained from symplectic groupoids, we can say that we have the full
geometric understanding of this theorem; in particular, it can be derived from the path-approach to symplectic
groupoids of \cite{CaFe} and the resulting construction of local symplectic groupoids \cite{CrFe2} (see Remark
\ref{remark_symplectic_realizations_is_a_consequence}).

\subsection{The first steps of the proof of Theorem 0}

We will first look at $\omega$ at zeros $0_x\in T^{*}_{x}M$. At such points one has a canonical isomorphism
$T_{0_x}(T^*M)\cong T_xM\oplus T_{x}^{*}M$ denoted by $v\mapsto (\overline{v}, \theta_v)$, and the canonical
symplectic form is
\begin{equation}\label{omega-can-0}
\omega_{\textrm{can}, 0_x}(v, w)= \langle \theta_{w},
\overline{v}\rangle -\langle \theta_{v}, \overline{w}\rangle.
\end{equation}
From the properties of $\mathcal{V}_{\pi}$, it follows that $\varphi_t(0_x)= 0_x$ for all $t$ and all $x$, hence
$\varphi_{t}$ is well-defined on a neighborhood of the zero section, for all $t\in [0, 1]$ (see the Tube Lemma
\ref{TubeLemma}). The properties $\mathcal{V}_{\pi}$ also imply that
\[ (d\varphi_t)_{0_x} : T_{0_x}(T^*M)\rmap T_{0_x}(T^*M) \]
is given in components by:
\[ (\overline{v}, \theta_v)\mapsto (\overline{v}+ t\pi^{\sharp}\theta_{v}, \theta_{v}) .\]
From the definition of $\omega$ and the previous formula for $\omega_{\textrm{can}}$, we deduce that
\begin{equation}\label{omega-0}
\omega_{0_x}(v, w)= \langle \theta_{w}, \overline{v}\rangle- \langle
\theta_{v}, \overline{w}\rangle+ \pi(\theta_{v}, \theta_{w}).
\end{equation}
This shows that $\omega$ is nondegenerate at all zeros $0_x\in T^*M$. Hence, we can find a neighborhood of the zero
section $\mathcal{U}\subset T^*M$ such that $\varphi_t$ is defined on $\mathcal{U}$ for all $t\in [0, 1]$ and
$\omega|_{\mathcal{U}}$ is nondegenerate (again, use the Tube Lemma \ref{TubeLemma}). Fixing such an
$\mathcal{U}$, we still have to show that the map
\[ (dp)_{\xi}: T_{\xi}\mathcal{U}\rmap T_{p(\xi)}M \]
sends the bivector associated to $\omega$ to $\pi$. The fact that this holds at $\xi= 0_x$ follows immediately from
(\ref{omega-0}), so our task is to show that it holds at all $\xi\in \mathcal{U}$. Recall the following result (see e.g.
Theorem 1.9.7 \cite{DZ}).
\begin{theorem}[Libermann's Theorem]\label{Libermann's Theorem}\index{Libermann's Theorem}
Let $(\mathcal{U}, \omega)$ be a symplectic manifold and let $p:\mathcal{U}\to M$ be a surjective submersion with
connected fibers. The bivector $\omega^{-1}$ can be pushed down to a bivector on $M$ precisely when the
symplectic orthogonal of $T\mathcal{F}(p)=\ker(dp)$, denoted by $T\mathcal{F}(p)^{\perp}\subset T\mathcal{U}$, is
involutive.
\end{theorem}
What happens in our case is that the foliation tangent to $T\mathcal{F}(p)^{\perp}$ is
\begin{equation}\label{EQ_symplectic_orth}
\mathcal{F}(p)^{\perp}= \mathcal{F}(p_1),
\end{equation}
where $\mathcal{F}(p_1)$ is the foliation given by the fibers of \[p_1:= p\circ \varphi_1: \mathcal{U}\rmap M.\]
However, it turns out that the ingredients needed to prove this equality can be used to show directly that $p$ is a
Poisson map, without having to appeal to Libermann's result.

\subsection{A different formula for $\omega$}

In this subsection we give another description of $\omega$. The resulting formula generalizes (\ref{omega-0}) from
zeros $0_x$ to arbitrary $\xi\in \mathcal{U}$. It will depend on a connection on $TM$ which is used in order to handle
vectors tangent to $T^*M$. Hence, from now on, we fix such a connection $\nabla$ which we assume to be torsion-free
and let $\overline{\nabla}$ denote the corresponding contravariant connections defined by (\ref{EQ_two_conn}). With
respect to $\nabla$, any tangent vector $v\in T_{\xi}(T^*M)$ is determined by its projection to $M$ and by its vertical
component
\[ \overline{v}= (dp)_{\xi}(v)\in T_{p(\xi)}M , \ \ \theta_v= v- \textrm{hor}_{\xi}(\overline{v})\in T_{p(\xi)}^{*}M .\]
Of course, when $\xi= 0_x$, these coincide with the components mentioned in the introduction. The fact that $\nabla$ is torsion-free ensures the
following generalization of the formula (\ref{omega-can-0}) for $\omega_{\textrm{can}}$ at arbitrary $\xi$.

\begin{lemma}\label{torsion-free1} If $\nabla$ is torsion-free, then for any $v, w\in T_{\xi}(T^*M)$,
\begin{equation}\label{Omega_can}
\omega_{\mathrm{can}}(v, w)=\langle \theta_w, \overline{v}\rangle-\langle\theta_v,\overline{w}\rangle.
\end{equation}
\end{lemma}

\begin{proof}  Since $\nabla$ is torsion free, the associated horizontal distribution $H\subset T(T^*M)$ is Lagrangian with respect to $\omega_{\textrm{can}}$.
This implies the formula.
\end{proof}

As above, let $\mathcal{U}\subset T^*M$ be an open around the zero section on which $\varphi_t$ is defined up to
$t=1$ and $\omega$ is nondegenerate. To establish the generalization of (\ref{omega-0}), we introduce some notation.
Fix $\xi\in\mathcal{U}$ and consider
\[ a: [0, 1]\rmap \mathcal{U}, \ a(t)= \varphi_t(\xi),\]
which, by the properties of $\mathcal{V}_{\pi}$, is a cotangent path. We denote by $\gamma= p\circ a$ its base path.
The push forward by $\varphi_t$ of a tangent vector $v_0\in T_{\xi}\mathcal{U}$ gives a smooth path
\begin{equation}\label{v-t}
t\mapsto v_t:= (\varphi_t)_*(v_0)\in T_{a(t)}\mathcal{U}.
\end{equation}
The components of $v$ with respect to $\nabla$ are two paths above $\gamma$, one in $TM$ and one in $T^*M$, denoted by $\overline{v}$ and
$\theta_v$
\[ t\mapsto \overline{v}_t\in T_{\gamma(t)}M, \  \ t\mapsto \theta_{v_t}\in T_{\gamma(t)}^{*}M .\]
They are related by:

\begin{lemma}\label{lemma_cotangent_tangent}
For $a$, $\overline{v}$ and $\theta_v$ as above, we have that  $\overline{\nabla}_a
\overline{v}=\pi^{\sharp}\theta_{v}$.
\end{lemma}

\begin{proof} We start with a remark on derivatives along vector fields. For any smooth path of vectors $V$ tangent to $\mathcal{U}$ along $a$, $t\mapsto V(t)\in T_{a(t)}\mathcal{U}$,
one has the Lie derivative of $V$ along $\mathcal{V}_{\pi}$, again a smooth path of tangent vectors along $a$, defined by
\[ L_{\mathcal{V}_{\pi}}(V)(t)= \frac{d}{ds}_{|s= 0}  (d\varphi_{-s})_{a(s+t)}(V(s+t)) \in T_{a(t)}\mathcal{U}.\]
Note the following:
\begin{enumerate}
\item For vertical $V$, i.e.\ corresponding to a 1-form $\theta_V$ on $M$ along $\gamma$, we have that
$(dp)(L_{\mathcal{V}_{\pi}}(V))= -\pi^{\sharp}(\theta_V)$. This follows immediately from the first property of the spray (e.g. by a local
computation).
\item For horizontal $V$, $(dp)(L_{\mathcal{V}_{\pi}}(V))= \overline{\nabla}_{a}(\overline{V})$, where $\overline{V}= (dp)(V)$ is a tangent vector to $M$ along $\gamma$. To check this, one may assume that
$V$ is a global horizontal vector field on $M$, and one has to show that $(dp)_{\eta}(L_{\mathcal{V}_{\pi}}(V))=
\overline{\nabla}_{\eta}(\overline{V})$ for all $\eta\in T^*M$. Again, this follows immediately by a local computation.
\end{enumerate}
Hence, for an arbitrary $V$ (along $a$), using its components $(\overline{V}, \theta_V)$,
\[ dp(L_{\mathcal{V}_{\pi}}(V)) = -\pi^{\sharp}(\theta_V) + \overline{\nabla}_{a}(\overline{V}).\]
Finally, note that for $v$ as in (\ref{v-t}), $L_{\mathcal{V}_{\pi}}(v)= 0$.
\end{proof}

We have the following version of (\ref{omega-0}) at arbitrary $\xi$ in $\mathcal{U}$.

\begin{lemma}\label{Lemma_simple_frmla} Let $\xi\in \mathcal{U}$ and $v_0, w_0 \in T_{\xi}(T^*M)$.
Let $v=v_t$ be as before and let $\widetilde{\theta}_{v}$ be a path in $T^*M$, which is a solution of the differential
equation
\begin{equation}\label{diff-eq}
{\overline{\nabla}}_{a}(\widetilde{\theta}_v)= \theta_v .
\end{equation}
Similarly, consider $w= w_t$ and $\widetilde{\theta}_{w}$ corresponding to $w_0$. Then
\begin{equation}\label{frmla_lemma}
\omega(v_0, w_0)= (\langle \widetilde{\theta}_w, \overline{v}\rangle - \langle \widetilde{\theta}_v, \overline{w}\rangle- \pi(\widetilde{\theta}_{v},\widetilde{\theta}_w) )|_{0}^{1}.
\end{equation}
\end{lemma}

\begin{proof}
Since $\nabla$ is torsion-free, Lemma \ref{torsion-free1} implies that
\[ \omega(v_0, w_0)= \int_{0}^{1}(\langle \theta_w, \overline{v}\rangle - \langle \theta_v, \overline{w}\rangle ) dt.\]
Hence, it suffices to show that
\[ \langle \theta_w, \overline{v}\rangle - \langle \theta_v, \overline{w}\rangle = \frac{d}{dt}  (\langle \widetilde{\theta}_w, \overline{v}\rangle - \langle \widetilde{\theta}_v, \overline{w}\rangle- \pi(\widetilde{\theta}_{v},\widetilde{\theta}_w) ).\]
Using (\ref{diff-eq}) for $\theta_v$ and $\theta_w$, followed by Lemma \ref{ConjConn2}, we obtain
\[\langle \theta_w, \overline{v}\rangle - \langle \theta_v, \overline{w}\rangle = \frac{d}{dt}(\langle \widetilde{\theta}_w, \overline{v}\rangle - \langle \widetilde{\theta}_v, \overline{w}\rangle)- \langle \widetilde{\theta}_{w}, \overline{\nabla}_{a}(\overline{v})\rangle + \langle \widetilde{\theta}_{v},
\overline{\nabla}_{a}(\overline{w})\rangle.\]
The last two terms can be evaluated as follows:
\begin{align*}
-\langle \widetilde{\theta}_{w}, \overline{\nabla}_{a}(\overline{v})\rangle + \langle \widetilde{\theta}_{v},& \overline{\nabla}_{a}(\overline{w})\rangle =- \langle \widetilde{\theta}_{w}, \pi^{\sharp}(\theta_v)\rangle + \langle \widetilde{\theta}_{v}, \pi^{\sharp}(\theta_w)\rangle=\\
&= - \langle \widetilde{\theta}_{w}, \pi^{\sharp}({\overline{\nabla}}_{a}(\widetilde{\theta}_v))\rangle + \langle \widetilde{\theta}_{v}, \pi^{\sharp}({\overline{\nabla}}_{a}(\widetilde{\theta}_w))\rangle=\\
&=- \langle \widetilde{\theta}_{w}, {\overline{\nabla}}_{a}(\pi^{\sharp}(\widetilde{\theta}_v))\rangle - \langle \overline{\nabla}_a(\widetilde{\theta}_w), \pi^{\sharp}(\widetilde{\theta}_v) \rangle=\\
&=-\frac{d}{dt} \langle \widetilde{\theta}_w, \pi^{\sharp}(\widetilde{\theta}_v)\rangle=
-\frac{d}{dt} \pi(\widetilde{\theta}_{v},\widetilde{\theta}_w),
\end{align*}
where we have used Lemma (\ref{lemma_cotangent_tangent}), equation (\ref{diff-eq}), the fact that
$\pi^{\sharp}({\overline{\nabla}}_{a}(\widetilde{\theta}_v))= {\overline{\nabla}}_{a}(\pi^{\sharp}(\widetilde{\theta}_v))$, the antisymmetry of
$\pi^{\sharp}$ and Lemma \ref{ConjConn2}.
\end{proof}

\subsection{Proof of Theorem 0}

We start by proving the equality (\ref{EQ_symplectic_orth}). By counting dimensions, it suffices to prove that for
$\xi\in \mathcal{U}$, we have that $\omega(v_0, w_0)= 0$ for all
\begin{equation}\label{cond}
v_0\in T_{\xi}\mathcal{F}(p),\ \textrm{and} \ w_0\in T_{\xi}\mathcal{F}(p_1).
\end{equation}
Using the same notation as before, conditions \ref{cond} are equivalent to $\overline{v}(0)= 0$ and $\overline{w}(1)=
0$. We remark that (\ref{diff-eq}), as an equation on $\widetilde{\theta}_{v}$, is a linear ordinary differential
equation; hence it has solutions defined for all $t\in [0, 1]$ satisfying any given initial (or final) condition. Hence one
may arrange that $\widetilde{\theta}_{v}(0)= 0$, $\widetilde{\theta}_{w}(1)=0$. Lemma \ref{Lemma_simple_frmla}
immediately implies that $\omega(v_0, w_0)= 0$.

\vspace*{.2in}

Finally, we show that $p$ is a Poisson map. We have to show that, for $\xi\in \mathcal{U}$ and $\theta\in
T_{x}^{*}M$  ($x= p(\xi)$), the unique $v_0\in T_{\xi}\mathcal{U}$ satisfying
\begin{equation}\label{Poisson_map}
 p^{*}(\theta)_{\xi}=\iota_{v_0}(\omega)
\end{equation}
also satisfies $(dp)_{\xi}(v_0)= \pi^{\sharp}(\theta)$. From (\ref{Poisson_map}), it follows that $v_0\in
T_{\xi}\mathcal{F}(p)^{\perp}$, hence $v_0\in T_{\xi}\mathcal{F}(p_1)$, therefore $\overline{v}(1)= 0$. Next, we
evaluate (\ref{Poisson_map}) on an arbitrary $w_0\in T_{\xi}\mathcal{U}$. We use the formula \ref{frmla_lemma},
where $\widetilde{\theta}_v$ and $\widetilde{\theta}_w$ are chosen so that $\widetilde{\theta}_v(1)= 0$ and
$\widetilde{\theta}_w(0)= \eta\in T_{x}^{*}M$ is arbitrary. We find:
\[ \theta(\overline{w}(0))= \langle \widetilde{\theta}_v(0), \overline{w}(0)\rangle + \langle \eta, \pi^{\sharp}\widetilde{\theta}_v(0)- \overline{v}(0)\rangle.\]
Since this holds for all $w_0$ and all $\eta$, we deduce that $\theta= \widetilde{\theta}_v(0)$ and
$\pi^{\sharp}\widetilde{\theta}_v(0)=\overline{v}(0)$. Hence $\pi^{\sharp}(\theta)= \overline{v}(0)= (dp)_{\xi}(v_0)$.

\subsection{Some remarks}

Here are some remarks on possible variations. First of all, regarding the notion of contravariant spray, the first condition means that,
locally, $\mathcal{V}_{\pi}$ is of the form
\[ \mathcal{V}_{\pi}(x, y)= \sum_{i,j} \pi_{i, j}(x)y_i\frac{\partial}{\partial x_j}+ \sum_p \gamma^p(x, y) \frac{\partial}{\partial y_p}.\]
The second condition means that each $\gamma^p(x, y)$ is of the form $\sum_{q, r} \gamma^{p}_{q, r}(x) y_q y_r$.
While the first condition has been heavily used, the second was only needed to ensure that $\omega$ is well-defined
and nondegenerate at elements $0_x\in T^{*}_{x}M$.

Another remark is that one can show that $\mathcal{U}$ can be made into a local symplectic groupoid, with source map $p$ and target map $p_1$
(see also \cite{Kar}).

Let us also point out that we used that $\pi$ is Poisson only in the compatibility relation (\ref{EQ_conn_pisharp_rel}) which in turn was only
used at the end of the proof of Lemma \ref{Lemma_simple_frmla}. However, it is easy to keep track of the extra-terms that show up for general
bivectors $\pi$. At the right hand side of (\ref{EQ_conn_pisharp_rel}), one has to add the term $i_{\alpha\wedge \beta}(\chi_{\pi})$ where
$\chi_{\pi}= [\pi, \pi]$, and to (\ref{frmla_lemma}) the term $\int_{0}^{1} \chi_{\pi}(a, \widetilde{\theta}_v, \widetilde{\theta}_w) dt$. This
is useful for handling various twisted versions. For example, for a $\sigma$-twisted bivector $\pi$ on $M$ (i.e. satisfying $[\pi, \pi]=
\pi^{\sharp}(\sigma)$, where $\sigma$ is a given closed 3-form on $M$ \cite{WeinSev}), the interesting (twisted symplectic) 2-form on
$\mathcal{U}$ is the previously defined $\omega$ to which we add the 2-form $\omega_{\sigma}$ given by (compare with \cite{CaXu}):
\[ \omega_{\sigma} = \int_{0}^{1} \varphi_{t}^{*}(i_{\mathcal{V}_{\pi}}p^*(\sigma)) dt.\]

\section{Conn's theorem revisited}\label{Section_Conn}

Conjectured by Weinstein \cite{Wein}, Conn's\index{Conn's theorem} normal form theorem \cite{Conn} is nowadays
one of the classical results in Poisson geometry. Conn's original proof is analytic, based on the fast convergence
method of Nash and Moser. A geometric proof of this result is also available \cite{CrFe-Conn}. However, Conn's
argument seems stronger than the geometric one; in particular, it also implies a local rigidity property. We have
adapted both approaches (the geometric and the analytic) to the case of symplectic leaves. The outcomes are the
normal form Theorem \ref{Theorem_TWO} and the rigidity Theorem \ref{Theorem_FOUR}. In this section we present a
simplified version of Conn's original proof.

\subsection{The statement}

Let $(M,\pi)$ be a Poisson manifold. A point $x\in M$ is called a \textbf{fixed point}\index{fixed point} of $\pi$, if
$\pi_{x}=0$ or, equivalently, if $\{x\}$ is a symplectic leaf. Conn's theorem \cite{Conn} gives a normal form for $(M,
\pi)$ near $x$, built out of the \textbf{isotropy Lie algebra}\index{isotropy Lie algebra} $\mathfrak{g}_x$ at $x$.
Recall that $\mathfrak{g}_x= T^*_{x}M$, with Lie bracket defined by
\[[d_xf,d_xg]=d_x \{ f,g\}.\]
The linear Poisson structure $(\mathfrak{g}^*_x,\pi_{\mathfrak{g}_x})$ plays the role of the first order approximation of $\pi$ around $x$ in
the realm of Poisson geometry.

\begin{theorem}[Conn's Theorem \cite{Conn}]
If $\mathfrak{g}_x$ is semisimple of compact type, then a neighborhood of $x$ in $(M,\pi)$ is Poisson diffeomorphic to
a neighborhood of the origin in $(\mathfrak{g}_{x}^{*},\pi_{\mathfrak{g}_x})$.
\end{theorem}

\subsection{Rigidity of the linear Poisson structure}
Conn's theorem follows from a rigidity\index{rigidity} property of the linear Poisson structure, which we explain in the
sequel.

Throughout this section, we fix $(\mathfrak{g},[\cdot,\cdot])$ a compact semisimple Lie algebra and $(\cdot,\cdot)$ an
invariant inner product on $\mathfrak{g}^*$. Denote by $B_r\subset \mathfrak{g}^*$ the open ball of radius $r$
centered at the origin. Let $\{x_i\}$ be linear coordinates on $\mathfrak{g}^*$ corresponding to an orthonormal basis
of $\mathfrak{g}$. We define $C^n$-norms\index{C$^n$-norms} on the space of multivector fields on the closed ball
$\overline{B}_r$: for $f\in C^{\infty}(\overline{B}_r)$, let
\[\|f\|_{n,r}:=\sup_{0\leq |\alpha|\leq n}\sup_{x\in \overline{B}_r}|\frac{\partial^{|\alpha|}f}{\partial x^{\alpha}}(x)|,\]
and for $W\in \mathfrak{X}^\bullet(\overline{B}_r)$, let
$\|W\|_{n,r}:=\sup_{i_1,\ldots,i_k}\|W_{i_1,\ldots,i_k}\|_{n,r}$, where
\[W=\sum_{i_1,\ldots,i_k}W_{i_1,\ldots,i_k}\frac{\partial}{\partial x_{i_1}}\wedge\ldots\wedge\frac{\partial}{\partial x_{i_k}}.\]
For $R>r$, denote the restriction map $\mathfrak{X}^\bullet(\overline{B}_R)\to
\mathfrak{X}^\bullet(\overline{B}_r)$ by
$W\mapsto W_{|r}$. 

The linear Poisson structure $\pi_{\mathfrak{g}}$ has the following rigidity property.

\begin{theorem}\label{Theorem_rigidity_fixed_points}
For some $p>0$ and all $0<r<R$, there is a constant $\delta=\delta(r,R)>0$, such that for every Poisson structure
$\pi$ on $\overline{B}_R$ satisfying
\[\|\pi-\pi_{\mathfrak{g}}\|_{p,R}<\delta,\]
there is an open embedding $\varphi:B_r\hookrightarrow B_R$, which is a Poisson map between
\[\varphi:(B_r,\pi_{\mathfrak{g}})\rmap (B_R,\pi).\]
\end{theorem}
\begin{remark}\rm
For $p$ we find the (most probably not optimal) value:
\[p=7\lfloor \mathrm{dim}(\mathfrak{g})/2\rfloor+17.\]
\end{remark}

Conn's theorem is a consequence of this rigidity property.

\begin{proof}[Proof of rigidity $\Rightarrow$ linearization]
Let $\pi$ be a Poisson structure whose isotropy Lie algebra at a fixed point $x$ is $\mathfrak{g}$. Taking a chart
centered at $x$, we may assume that $\pi$ is defined on $\overline{B}_R$, for some $R>0$, and so, it is given by
\[\pi=\frac{1}{2}\sum_{i,j}\pi_{i,j}(x)\frac{\partial}{\partial x_i}\wedge\frac{\partial}{\partial x_j}, \ \textrm{with}\ \pi_{i,j}(0)=0.\]
The linear Poisson structure is
\[\pi_{\mathfrak{g}}=\frac{1}{2}\sum_{i,j,k}\frac{\partial\pi_{i,j}}{\partial x_k}(0)x_k\frac{\partial}{\partial x_i}\wedge\frac{\partial}{\partial x_j}.\]
The two structures can be connected by a smooth path of Poisson bivectors $\pi^t\in \mathfrak{X}^2(\overline{B}_{R})$, with $\pi^1=\pi$ and
$\pi^0=\pi_\mathfrak{g}$, given by
\[\pi^t:=\frac{1}{2t}\sum_{i,j}\pi_{i,j}(tx)\frac{\partial}{\partial x_i}\wedge\frac{\partial}{\partial x_j},\ t\in[0,1].\]
Let $0<r<R$. By Theorem \ref{Theorem_rigidity_fixed_points}, for $t>0$ such that
\[\|\pi^t-\pi_{\mathfrak{g}}\|_{p,R}<\delta(r,R),\]
there exists an open embedding $\varphi:B_r\hookrightarrow B_R$ such that $\pi_{\mathfrak{g}}=\varphi^*(\pi^t)$.
Notice that $\pi^t=t\mu_t^{*}(\pi)$, where $\mu_t$ denotes multiplication by $t>0$, and, since $\pi_\mathfrak{g}$ is
linear, that $\mu_{1/t}^*(\pi_\mathfrak{g})=t\pi_\mathfrak{g}$. Therefore
\[\pi_{\mathfrak{g}}=\frac{1}{t}\mu_{1/t}^*(\pi_{\mathfrak{g}})=\frac{1}{t}\mu_{1/t}^*\circ\varphi^*(\pi^t)=\mu_{1/t}^*\circ\varphi^*\circ\mu_t^{*}(\pi)=(\mu_t\circ\varphi\circ\mu_{1/t})^*(\pi). \]
Thus $\psi:=\mu_t\circ\varphi\circ\mu_{1/t}$ is an open embedding and a Poisson map between
\[\psi:(B_{tr},\pi_{\mathfrak{g}})\rmap (B_{tR},\pi).\]
\end{proof}

\begin{remark}\rm
In fact, by making $\delta$ small enough, we prove that $\varphi$ from Theorem \ref{Theorem_rigidity_fixed_points}
can be made arbitrary $C^1$-close to the identity. In particular, we may assume that $0\in \varphi(B_r)$, and so, in the
proof above, that $0\in\psi(B_{tr})$. Since $0$ is the only fixed point of $\pi_{\mathfrak{g}}$ the Poisson map $\psi$
sends $0$ to $0$; and this is how Conn states the conclusion of his theorem \cite{Conn}.
\end{remark}

\subsection{Tame homotopy operators}\label{SSTamehomotopy}

One of the main ingredients in the proof of Theorem \ref{Theorem_rigidity_fixed_points} is the existence of tame
homotopy operators for the second Poisson cohomology of $\pi_{\mathfrak{g}}$. Here is where our proof deviates the
most from Conn's. A more general construction, which applies to Lie algebroid cohomology with coefficients, is
presented in the appendix (The Tame Vanishing Lemma).

\subsubsection*{The Poisson cohomology as the invariant part of $\Omega_{\mathfrak{g}^*}^{\bullet}(G)$}

The results in this subsection hold for any Lie algebra $\mathfrak{g}$.

The space of multivector fields on $\mathfrak{g}^*$ can be identified with the space of forms on $\mathfrak{g}$ with values in
$C^{\infty}(\mathfrak{g}^*)$
\[\mathfrak{X}^{\bullet}(\mathfrak{g}^*)\cong \Lambda^{\bullet}\mathfrak{g}^*\otimes C^{\infty}(\mathfrak{g}^*).\]
With this, the Poisson differential becomes the differential computing the Eilenberg Chevalley cohomology of
$\mathfrak{g}$ with coefficients in the representation $C^{\infty}(\mathfrak{g}^*)$, and it can be written using the
Poisson bracket as
\begin{align*}
d_{\pi_{\mathfrak{g}}}\omega(X_0,\ldots,X_k)&=\sum_{i=0}^k(-1)^i\{X_i,\omega(X_0,\ldots,\hat{X}_i,\ldots,X_k)\}+\\
&+\sum_{i<j}(-1)^{i+j}\omega(\{X_i,X_j\},X_0,\ldots,\hat{X}_i,\ldots,\hat{X}_j,\ldots,X_k),
\end{align*}
for $\omega\in \Lambda^{k}\mathfrak{g}^*\otimes C^{\infty}(\mathfrak{g}^*)$, where we regard $X_i\in
\mathfrak{g}$ as linear functions on $\mathfrak{g}^*$.

Let $G$ be the 1-connected Lie group integrating $\mathfrak{g}$. Let $X^r\in\mathfrak{X}(G)$ be the right invariant
extension of $X\in\mathfrak{g}$, i.e.\ $X^r_g=dr_g(X)$. By our convention (see section \ref{SSNotConv}), the map
$X\mapsto X^r$ is a Lie algebra homomorphism. Therefore, the map associating to a form on $\mathfrak{g}$ its right
invariant extension
\[\Lambda^{\bullet}\mathfrak{g}^*\ni\omega\mapsto \omega^r\in \Omega^{\bullet}(G),\  \omega^r_g:=r_{g^{-1}}^*(\omega)\]
is a chain map between the Eilenberg Chevalley complex of $\mathfrak{g}$ and the de Rham complex of $G$
\[(\Lambda^{\bullet}\mathfrak{g}^*,d_{\mathfrak{g}})\rmap(\Omega^{\bullet}(G),d_G).\]
We have a version of this also with coefficients in $C^{\infty}(\mathfrak{g}^*)$. Let
$\Omega^{\bullet}_{\mathfrak{g}^*}(G)$ denote the space of sections of $\Lambda^{\bullet}T^*G\times
\mathfrak{g}^*\to G\times\mathfrak{g}^*$. Regarding the elements in $\Omega^{\bullet}_{\mathfrak{g}^*}(G)$ as
smooth maps
\[\omega:\mathfrak{g}^*\rmap \Omega^{\bullet}(G), \ \xi\mapsto \omega_{\xi},\]
$\Omega^{\bullet}_{\mathfrak{g}^*}(G)$ can be endowed with the ``pointwise de Rham differential'':
\[d_G(\omega)_{\xi}:=d_G(\omega_{\xi}),\ \xi\in \mathfrak{g}^*,\  \omega\in \Omega^{\bullet}_{\mathfrak{g}^*}(G).\]
Poisson cohomology is the $G$-invariant part of the resulting cohomology.
\begin{lemma}\label{lemma_commutes_with_diff_fixed_point}
The Poisson complex of $(\mathfrak{g}^*,\pi_{\mathfrak{g}})$ is isomorphic to the $G$-invariant part of the complex
$(\Omega_{\mathfrak{g}^*}^{\bullet}(G),d_G)$ via the map
\[J:(\mathfrak{X}^{\bullet}(\mathfrak{g}^*),d_{\pi_{\mathfrak{g}}})\rmap  (\Omega_{\mathfrak{g}^*}^{\bullet}(G),d_G),\]
\[J(\omega)_{(g,\xi)}(X_1,\ldots,X_k):=\omega_{Ad^{*}_{g^{-1}}(\xi)}(dr_{g^{-1}}(X_1),\ldots,dr_{g^{-1}}(X_k)),\]
for $\omega\in \mathfrak{X}^{k}(\mathfrak{g}^*)$, where the action of $G$ on
$\Omega_{\mathfrak{g}^*}^{\bullet}(G)$ is
\[(g\cdot \omega)_{\xi}:=r_{g}^*(\omega_{Ad^{*}_{g}(\xi)}).\]

\end{lemma}
\begin{proof}
The fact that $J$ is an isomorphism between
\[\Lambda^{\bullet}\mathfrak{g}^*\otimes C^{\infty}(\mathfrak{g}^*)\cong \Omega_{\mathfrak{g}^*}^{\bullet}(G)^G,\]
is straightforward; we only note that a left inverse of $J$ is the map
\[P:\Omega_{\mathfrak{g}^*}^{\bullet}(G)\rmap \Lambda^{\bullet}\mathfrak{g}^*\otimes C^{\infty}(\mathfrak{g}^*), \ P(\omega)_{\xi}:=\omega_{\xi|\Lambda^{\bullet}T_eG}.\]
On forms with constant coefficients, i.e.\ forms in $\Lambda^{\bullet}\mathfrak{g}^*$, the map $J$ is just the right
invariant extension, thus it commutes with the differentials. So is suffices to check the statement for $f\in
C^{\infty}(\mathfrak{g}^*)$. For $X\in\mathfrak{g}$, we have
\begin{align*}
(d_GJ(f))_{\xi}&(dr_g(X))=\frac{d}{d\epsilon}\left(J(f)_{\xi}(\exp(\epsilon X)g)\right)|_{\epsilon=0}=\\
&=\frac{d}{d\epsilon}f(Ad_{{\exp(-\epsilon X)}}^{*}Ad_{g^{-1}}^{*}(\xi))|_{\epsilon=0}=df(H_{X})_{Ad_{g^{-1}}^{*}(\xi)}=\\
&=\{X,f\}_{Ad_{g^{-1}}^{*}(\xi)}=(d_{\pi_{\mathfrak{g}}}f)_{Ad_{g^{-1}}^{*}(\xi)}(X)=J(d_{\pi_{\mathfrak{g}}}f)_{\xi}(dr_g(X)),
\end{align*}
and this finishes the proof.
\end{proof}

\subsubsection*{Homotopy operators for the de Rham complex of $G$}

The\index{homotopy operators} assumption that $\mathfrak{g}$ is compact and semisimple is equivalent to the
compactness of $G$. The simply connectedness of $G$ implies that $H^2(G)=0$ (see e.g. \cite{DK}). The invariant
inner product on $\mathfrak{g}$ extends to a bi-invariant Riemannian metric $m$ on $G$. Integration against the
volume density corresponding to $m$ gives an inner product on the space of forms
\[(\eta,\theta):=\int_{G}m(\eta,\theta)|dVol(m)|,\ \eta,\theta\in\Omega^{\bullet}(G).\]
Denote by $\delta$ the formal adjoint of the exterior derivative $d$, i.e.\ $\delta$ is the unique linear, first order partial differential
operator
\[\delta:\Omega^{\bullet+1}(G)\rmap \Omega^{\bullet}(G)\]
satisfying
\[(d\eta,\theta)=(\eta,\delta\theta),\ \ \ \  \eta\in\Omega^{\bullet}(G),\ \theta\in\Omega^{\bullet+1}(G).\]
Since $d^2=0$, also $\delta^2=0$. The corresponding Laplace-Beltrami operator is\index{Laplace-Beltrami operator}
\[\Delta:\Omega^{\bullet}(G)\rmap \Omega^{\bullet}(G),\ \ \ \Delta:=(d+\delta)^2=d\delta+\delta d.\]
Since $H^2(G)=0$, by classical Hodge theory (see e.g. \cite{Warner}), it follows that $\Delta$ is invertible in degree 2, and that its inverse
satisfies
\begin{equation}\label{EQ_Delta_inverse_tame}
|\Delta^{-1}\eta|^G_{n+2}\leq C_n|\eta|^G_{n},\ \ \ \ \eta\in \Omega^2(G),
\end{equation}
where $C_n$ are positive constants and $|\cdot|_n^G$ denote Sobolev norms on $\Omega^{\bullet}(G)$. The Sobolev
embedding theorem and (\ref{EQ_Delta_inverse_tame}) imply that $\Delta^{-1}$ satisfies similar inequalities for
$C^n$-norms $\|\cdot\|_n^G$ on $\Omega^{\bullet}(G)$
\[\|\Delta^{-1}\eta\|^G_{n}\leq C_n\|\eta\|^G_{n+s-1},\ \ \ \ \eta\in \Omega^2(G),\]
where $s=\lfloor\frac{1}{2}\mathrm{dim}(G)\rfloor$.

Consider the linear operators $H_1:=\delta\circ \Delta^{-1}$ and $H_2:=\Delta^{-1}\circ \delta$
\[\Omega^{1}(G)\stackrel{H_1}{\longleftarrow}\Omega^{2}(G)\stackrel{H_2}{\longleftarrow}\Omega^{3}(G).\]
Note that $d\delta$ and $\Delta$ commute (hence also $d\delta$ and $\Delta^{-1}$)
\[\Delta d\delta=(d\delta+\delta d)d\delta=d\delta d\delta+ \delta d^2\delta=d\delta d\delta=d\delta d\delta+d\delta^2 d=d\delta(d\delta+\delta d)=d\delta\Delta.\]
This implies that $H_1$ and $H_2$ are linear homotopy operators for the de Rham complex in degree 2
\begin{align*}
d H_1+H_2d=d\delta\Delta^{-1}+\Delta^{-1}\delta d=\Delta^{-1}(d\delta+\delta d)=\textrm{Id}_{\Omega^2(G)}.
\end{align*}
Since $\delta$ is of first order, and by (\ref{EQ_Delta_inverse_tame}), $H_1$ and $H_2$ satisfy
\begin{equation}\label{EQ_fixed_point_tame_homotopy}
\|H_1\eta\|^G_{n}\leq C_n\|\eta\|^G_{n+s},\ \ \|H_2\theta\|^G_{n}\leq C_n\|\theta\|^G_{n+s}.
\end{equation}
By invariance of the metric, also these operators are invariant
\begin{equation}\label{EQ_H_i_invariant}
l_g^*(H_1\eta)=H_1l_g^*(\eta), \ \ l_g^*(H_2\theta)=H_2l_g^*(\theta).
\end{equation}

\subsubsection*{Tame homotopy operators for the Poisson complex}

We are ready to construct the homotopy operators for the Poisson complex.\index{homotopy operators}
\begin{lemma}[Proposition 2.1 \cite{Conn}]\label{Lemma_fixed_point_homotopy_ops}
There are linear homotopy operators for the Poisson complex of $(\mathfrak{g}^*,\pi_{\mathfrak{g}})$ in degree two
\[\mathfrak{X}^1(\mathfrak{g}^*)\stackrel{h_1}{\longleftarrow}\mathfrak{X}^2(\mathfrak{g}^*)\stackrel{h_2}{\longleftarrow}\mathfrak{X}^3(\mathfrak{g}^*),\]
\[d_{\pi_{\mathfrak{g}}}h_1+h_2d_{\pi_{\mathfrak{g}}}=\textrm{Id}_{\mathfrak{X}^2(\mathfrak{g}^*)},\]
which induce homotopy operators for $(\overline{B}_r,\pi_{\mathfrak{g}|r})$ in degree two
\[\mathfrak{X}^1(\overline{B}_r)\stackrel{h_1^r}{\longleftarrow}\mathfrak{X}^2(\overline{B}_r)\stackrel{h_2^r}{\longleftarrow}\mathfrak{X}^3(\overline{B}_r),\]
\[h_1^r(\eta_{|r})=h_1(\eta)_{|r}, \ \ h_2^r(\theta_{|r})=h_2(\theta)_{|r}.\]
These operators satisfy the following inequalities
\[\|h_1^r\eta\|_{n,r}\leq C_n\|\eta\|_{n+s,r},\ \ \|h_2^r\theta\|_{n,r}\leq C_n\|\theta\|_{n+s,r},\]
with $s=\lfloor\frac{1}{2}\mathrm{dim}(\mathfrak{g})\rfloor$ and constants $C_n>0$ independent of $r$.
\end{lemma}
\begin{proof}
Using the maps $J$ and $P$ from Lemma \ref{lemma_commutes_with_diff_fixed_point}, we define
\[h_1(\eta):=P\circ H_1\circ J(\eta), \ \ h_2(\theta):=P\circ H_2\circ J(\theta).\]

Since $H_1$ and $H_2$ are homotopy operators for the de Rham cohomology of $G$ in degree 2, it follows that they
induce also homotopy operators for the complex $(\Omega_{\mathfrak{g}^*}(G),d_G)$ in degree 2. By
(\ref{EQ_H_i_invariant}), $H_1$ and $H_2$ preserve $\Omega_{\mathfrak{g}^*}(G)^G$ and Lemma
\ref{lemma_commutes_with_diff_fixed_point} gives an isomorphism between $(\Omega_{\mathfrak{g}^*}(G)^G,d_G)$
and the Poisson complex under which $H_1$ and $H_2$ correspond to $h_1$ and $h_2$. Thus, $h_1$ and $h_2$ are
homotopy operators for the Poisson complex in second degree.

Observe that for $\eta\in \mathfrak{X}^{\bullet}(\mathfrak{g}^*)\cong \Lambda^{\bullet}\mathfrak{g}^*\otimes C^{\infty}(\mathfrak{g}^*)$ and
$\xi\in \overline{B}_r$, by invariance of the metric, $J(\eta)_{\xi}$ depends only on $\eta_{|r}$. Therefore, a version of Lemma
\ref{lemma_commutes_with_diff_fixed_point} is true for $\overline{B}_r$ which gives an isomorphism between
\[(\mathfrak{X}^{\bullet}(\overline{B}_r),d_{\pi_{\mathfrak{g}}})\cong (\Omega_{\overline{B}_r}^{\bullet}(G)^G,d_G).\]
In particular, $h_1^r$ and $h_2^r$ are well-defined and, by a similar argument as before, they are homotopy operators in degree 2.

To check the inequalities, we endow first $\Omega_{\overline{B}_r}^{\bullet}(G)$ with $C^n$-norms
\[\|\omega\|^G_{n,r}:=\sup_{|\alpha|+k\leq n}\sup_{x\in \overline{B}_r}\|\frac{\partial^{|\alpha|}\omega}{\partial x^{\alpha}}(x)\|_k^G,\]
where $\|\cdot\|_n^G$ are some fixed $C^n$-norms on $G$. Using (\ref{EQ_fixed_point_tame_homotopy}), we compute
\begin{align*}
\|H_1(\eta)\|^G_{n,r}&=\sup_{|\alpha|+k\leq n}\sup_{x\in \overline{B}_r}\|\frac{\partial^{|\alpha|}H_1\eta}{\partial x^{\alpha}}(x)\|_k^G=\sup_{|\alpha|+k\leq n}\sup_{x\in \overline{B}_r}\|H_1\frac{\partial^{|\alpha|}\eta}{\partial x^{\alpha}}(x)\|_k^G\leq\\
&\leq \sup_{|\alpha|+k\leq n}\sup_{x\in \overline{B}_r}C_k\|\frac{\partial^{|\alpha|}\eta}{\partial x^{\alpha}}(x)\|_{k+s}^G\leq C_n \|\eta\|^G_{n+s,r},
\end{align*}
and similarly, for $H_2$ we obtain
\[\|H_2(\theta)\|^G_{n,r}\leq C_n \|\theta\|^G_{n+s,r}.\]
It is not difficult to show that $J$ and $P$ satisfy the inequalities
\[\|J(\eta)\|^G_{n,r}\leq C_n \|\eta\|_{n,r}, \ \ \|P(\omega)\|_{n,r}\leq C_n\|\omega\|^G_{n,r},\]
(see also Lemma \ref{Lemma_pullback_tame}) and combining these with the inequalities for $H_1$ and $H_2$, we
obtain the conclusion.
\end{proof}

\subsection{Inequalities}\label{Inequalities}

In this subsection we list several inequalities which will be used in the proof of Theorem
\ref{Theorem_rigidity_fixed_points}. Most of these results are standard, and for the proofs we refer to chapter
\ref{ChRigidity} where we present more general statements.

A usual convention when dealing with the Nash-Moser techniques (e.g.\ \cite{Ham}), already used in the previous
subsection, and which we also adopt here, is to denote all constants by the same symbol. In the results below we work
with ``big enough'' constants $C_n$ which depend continuously on $r>0$ (the radius of $\overline{B}_r$), on the Lie
algebra $\mathfrak{g}$ and the metric defined on it, and we also work with a ``small enough'' constant $\theta>0$,
which does not depend on $r$.

\subsubsection{Smoothing operators and interpolation inequalities}

The main technical tool used in the Nash-Moser method are the \textbf{smoothing operators}. This is a family of linear
operators\index{smoothing operators}
\[\{S_t^r:\mathfrak{X}^{k}(\overline{B}_r)\rmap \mathfrak{X}^{k}(\overline{B}_r) \}_{t>1,r>0}\]
that satisfy, for all $m\leq n$, $W\in \mathfrak{X}^{k}(\overline{B}_r)$, the inequalities
\begin{equation*}
\|S_t^r(W)\|_{n,r}\leq t^{m}C_{n}\|W\|_{n-m,r},\ \ \|S_t^r(W)-W\|_{n-m,r}\leq t^{-m}C_{n}\|W\|_{n,r},
\end{equation*}
with $C_n>0$ depending continuously on $r>0$. For their existence see Lemma \ref{Lemma_smoothing_operators}.

As remarked in \cite{Ham}, existence of smoothing operators implies that the classical \textbf{interpolation
inequalities} hold for the norms $\|\cdot\|_{n,r}$\index{interpolation inequalities}
\[\|W\|_{l,r}\leq C_{n} (\|W\|_{k,r})^{\frac{n-l}{n-k}}(\|W\|_{n,r})^{\frac{l-k}{n-k}}, \ \ \forall \ W\in\mathfrak{X}^{\bullet}(\overline{B}_r),\]
for $ k\leq l\leq n$, not all equal, with $C_n>0$ depending continuously on $r$.

\subsubsection*{Inequalities of some natural operations}

As a straightforward consequence of the interpolation inequalities, the Schouten bracket on
$\mathfrak{X}^{\bullet}(\overline{B}_r)$ satisfies (see Lemma \ref{L_Bracket})
\[\|[W,V]\|_{n,r}\leq C_n(\|W\|_{0,r}\|V\|_{n+1,r}+\|W\|_{n+1,r}\|V\|_{0,r}).\]

We consider also $C^n$-norms on the space of maps $C^{\infty}(\overline{B}_r,\mathfrak{g}^*)$, defined exactly as
on $\mathfrak{X}(\mathfrak{g}^*)$. Such a map is called a \textbf{local diffeomorphism},\index{local
diffeomorphism} if it can be extended on some neighborhood of $\overline{B}_r$ to an open embedding. Let $I_r$
denote the inclusion $\overline{B}_r\subset \mathfrak{g}^*$. A $C^1$-open around $I_r$ contains only local
diffeomorphisms.

\begin{lemma}[see Lemma \ref{Lemma_embedding}]\label{Lemma_embeddingfixed}
There exists $\theta>0$, such that, if $\psi\in C^{\infty}(\overline{B}_r,\mathfrak{g}^*)$ satisfies $\|\psi-I_r\|_{1,r}<\theta$, then $\psi$ is
a local diffeomorphism.
\end{lemma}

The composition satisfies the following inequalities.
\begin{lemma}[see Lemma \ref{Lemma_composition}]
There are constants $C_n> 0$, depending continuously on $0<s< r$, such that for and all $\varphi\in
C^{\infty}(\overline{B}_s,B_r)$ satisfying $\|\varphi-I_s\|_{1,s}<1$, and all $\psi\in
C^{\infty}(\overline{B}_r,\mathfrak{g}^*)$, the following hold
\begin{align*}
\|\psi\circ \varphi-I_s\|_{n,s}&\leq \|\psi-I_r\|_{n,r}+\|\varphi-I_s\|_{n,s}+\\
&+ C_n(\|\psi-I_r\|_{n,r}\|\varphi-I_s\|_{1,s}+\|\varphi-I_s\|_{n,s}\|\psi-I_r\|_{1,r}),\\
\|\psi\circ \varphi-\psi\|_{n,s}&\leq \|\varphi-I_s\|_{n,s}+\\
&+ C_n(\|\psi-I_r\|_{n+1,r}\|\varphi-I_s\|_{1,s}+\|\varphi-I_s\|_{n,s}\|\psi-I_r\|_{1,r}).
\end{align*}
\end{lemma}

These inequalities imply a criterion for convergence of compositions.
\begin{lemma}[see Lemma \ref{Lemma_convergent_embbedings}]\label{Lemma_convergent_embbedingsfixed}
There exists $\theta>0$, such that for all sequences
\[\{\varphi_{k}\in C^{\infty}(\overline{B}_{r_k},B_{r_{k-1}})\}_{k\geq 1},\] where $0<r<r_{k}<r_{k-1}< R$, which satisfy
\[\sigma_1:=\sum_{k\geq 1}\|\varphi_k-I_{r_k}\|_{1,r_k}<\theta\ \textrm{ and  }\  \sum_{k\geq 1}\|\varphi_k-I_{r_k}\|_{n,r_k}<\infty,\]
the sequence of maps
\[\psi_k:=\varphi_1\circ\ldots\circ \varphi_{k|r}:\overline{B}_{r}\rmap B_{R},\]
converges in all $C^n$-norms to a smooth map $\psi\in C^{\infty}(\overline{B}_r,B_{R})$ which satisfies
\[\|\psi-I_{r}\|_{1,r}\leq C\sigma_1,\]
for some constant $C>0$ depending continuously on $r$ and $R$.
\end{lemma}

We list some properties of the flow.

\begin{lemma}[see Lemmas \ref{Lemma_domain_of_flow}, \ref{Lemma_size of _the flow},
\ref{Lemma_tame_flow_pull_back}]\label{Lemma_domain_of_flowfixed} There exists $\theta>0$ such that for all
$0<s<r$ and all $X\in \mathfrak{X}^1(\overline{B}_r)$ with $\|X\|_{0,r}<(r-s)\theta$, the flow of $X$ is defined for
$t\in [0,1]$ between
\[\varphi^t_{X}: \overline{B}_{s}\rmap B_{r}.\]
Moreover, there are $C_n>0$, which depend continuously on $r$, such that if $X$ also satisfies $\|X\|_{1,r}<\theta$, then
$\varphi_X:=\varphi_X^1$ has the following properties
\begin{align*}
&\|\varphi_{X}-I_s\|_{n,s}\leq C_n\|X\|_{n,r},\\
&\|\varphi_{X}^*(W)\|_{n,s} \leq C_n (\|W\|_{n,r}+\|W\|_{0,r}\|X\|_{n+1,r}),\\
&\|\varphi_{X}^*(W)-W_{|s}\|_{n,s}\leq C_n(\|X\|_{n+1,r}\|W\|_{1,r}+\|X\|_{1,r}\|W\|_{n+1,r}),\\
&\|\varphi_{X}^*(W)-W_{|s}- \varphi_{X}^*([X,W])\|_{n,s} \leq \\
&\phantom{\|\varphi_{X}^*(W)-W_{|s}- \varphi_{X}^*}\leq C_n\|X\|_{0,r}(\|X\|_{n+2,r}\|W\|_{2,r}+\|X\|_{2,r}\|W\|_{n+2,r}),
\end{align*}
for all $W\in\mathfrak{X}^{\bullet}(\overline{B}_r)$.
\end{lemma}

\subsection{Proof of Theorem \ref{Theorem_rigidity_fixed_points}}

We will construct $\psi$ as a limit of open embeddings $\psi_k:\overline{B}_{r}\hookrightarrow B_{R}$, chosen such
that the difference $\psi_k^*(\pi)-\pi_{\mathfrak{g}|{r}}$ converges to zero. Denote by
\[\alpha:=2(s+3), \  p:=7s+17,\]
where $s=\lfloor\frac{1}{2}\mathrm{dim}(\mathfrak{g})\rfloor$, and consider the sequences 
\[\begin{array}{ccc}
  \epsilon_0:=\mathrm{min}\{\frac{1}{4},\frac{R-r}{2}\}, & r_0:=R, &  t_0:={\delta}^{-1/\alpha}, \\
  \epsilon_{k+1}:=\epsilon_k^{3/2}, & r_{k+1}:=r_k-\epsilon_k, & t_{k+1}:=t_k^{3/2},
\end{array}\]
where $\delta=\delta(r,R)\in (0,1)$ will be fixed later on. By our choice of $\epsilon_0$,
\begin{align*}
\sum_{k=0}^{\infty} \epsilon_k=\sum_{k=0}^{\infty} \epsilon_0^{(3/2)^k}<\sum_{k=0}^{\infty} \epsilon_0^{1+\frac{k}{2}}=\frac{\epsilon_0}{1-\sqrt{\epsilon_0}}\leq (R-r),
\end{align*}
hence $r<r_k< R$, for all $k\geq 1$.

Denote the smoothing operators by $S_{k}:=S_{t_k}^{r_k}$. 
Consider the sequences
\[\{X_k\in\mathfrak{X}^1(\overline{B}_{r_k})\}_{k\geq 0}, \ \ \{\pi_k\in\mathfrak{X}^{2}(\overline{B}_{r_k})\}_{k\geq 0},\]
defined by the following recursive procedure:
\[\pi_0:=\pi,\,\,\,\,\, Z_k:=\pi_k-\pi_{\mathfrak{g}|{r_k}},\,\,\,\,\, X_k:=S_k(h_1^{r_k}(Z_k)),\,\,\,\,\, \pi_{k+1}:=\varphi_{X_k}^*(\pi_k).\]

We will prove by induction the following statements (note that the first ensures that the procedure is well-defined for all $k$):\\
\noindent $(a_k)$: the flow of $X_k$ at time one is defined as a map
\[\varphi_{X_k}:\overline{B}_{r_{k+1}}\rmap B_{r_k}.\]
\noindent $(b_k)$: $Z_k$ satisfies the inequalities
\[\|Z_k\|_{s,r_k}\leq t_k^{-\alpha}, \ \ \|Z_k\|_{p,r_k}\leq t_k^{\alpha}.\]

By hypothesis, $(b_0)$ holds
\[\|Z_0\|_{p,R}=\|\pi-\pi_{\mathfrak{g}}\|_{p,R}<\delta=t_{0}^{-\alpha}.\]
We will show that $(b_k)$ implies $(a_k)$ and $(b_{k+1})$.

First we give a bound for the norms of $X_k$ in terms of the norms of $Z_k$
\begin{align}\label{EQ_X_n_lfixed}
\|X_k\|_{m,r_k}&=\|S_{k}(h_1^{r_k}(Z_k))\|_{m,r_k}\leq C_mt_k^{l}\|h_1^{r_k}(Z_k)\|_{m-l,r_k}\leq\\
\nonumber &\leq C_mt_k^{l}\|Z_k\|_{m+s-l,r_k},
\end{align}
for all $0\leq l\leq m$. In particular, for $m=l$, using $(b_k)$, we obtain
\begin{align}\label{EQ_X_n_nfixed}
\|X_k\|_{m,r_k}&\leq C_mt_k^{m-\alpha}.
\end{align}
Taking $\delta$ such that $\delta^{-1/\alpha}=t_0>\epsilon_0^{-1}$, we have that $t_k>\epsilon_k^{-1}$. Using that $\alpha>3$ and
(\ref{EQ_X_n_lfixed}), we obtain
\begin{equation}\label{EQ_X_12fixed}
\|X_k\|_{1,r_k}\leq Ct_k^{1-\alpha}\leq Ct_0^{-1}t_k^{-1}<C\delta^{1/\alpha}\epsilon_k.
\end{equation}
By shrinking $\delta$, we get that $\|X_k\|_{1,r_k}\leq \theta\epsilon_k$, so $X_k$ satisfies the hypothesis of Lemma
\ref{Lemma_domain_of_flowfixed}, in particular $(a_k)$ holds.

We deduce now an inequality for all norms $\|Z_{k+1}\|_{n,r_{k+1}}$, with $n\geq s$
\begin{align}\label{EQ_Z_nfixed}
\|Z_{k+1}&\|_{n,r_{k+1}}=\|\varphi_{X_k}^{*}(Z_k)+\varphi_{X_k}^{*}(\pi_{\mathfrak{g}})-\pi_{\mathfrak{g}}\|_{n,r_{k+1}}\leq\\
\nonumber &\leq C_n(\|Z_k\|_{n,r_{k}}+\|X_k\|_{n+1,r_k}\|Z_k\|_{0,r_k}+\|X_k\|_{n+1,r_k}\|\pi_{\mathfrak{g}}\|_{n+1,r_k})\leq\\
\nonumber&\leq C_n(\|Z_k\|_{n,r_{k}}+\|X_k\|_{n+1,r_k})\leq C_nt_k^{s+1}\|Z_k\|_{n,r_k},
\end{align}
where we used Lemma \ref{Lemma_domain_of_flowfixed}, the inductive hypothesis and inequality (\ref{EQ_X_n_lfixed}) with $m=n+1$ and $l=s+1$. For
$n=p$, using also that $s+1+\alpha< \frac{3}{2}\alpha-1$ and by shrinking $\delta$, this gives the second part of $(b_{k+1})$:
\begin{align*}
\|Z_{k+1}\|_{p,r_{k+1}}&\leq Ct_k^{s+1+\alpha}< Ct_k^{\frac{3}{2}\alpha-1}\leq Ct_0^{-1}t_{k+1}^{\alpha}=C\delta^{1/\alpha}t_{k+1}^{\alpha}\leq t_{k+1}^{\alpha}.
\end{align*}

To prove the first part of $(b_{k+1})$, we write $Z_{k+1}=V_k+\varphi_{X_k}^{*}(U_{k})$, where
\[V_k:=\varphi_{X_k}^{*}(\pi_{\mathfrak{g}})-\pi_{\mathfrak{g}}-\varphi_{X_k}^{*}([X_k,\pi_{\mathfrak{g}}]),\ \ U_{k}:=Z_{k}-[\pi_{\mathfrak{g}},X_k].\]
Using Lemma \ref{Lemma_domain_of_flowfixed} and inequality (\ref{EQ_X_n_nfixed}), these terms can be bounded by
\begin{align}\label{EQ_Vfixed}
\|V_k\|_{s,r_{k+1}}&\leq C\|\pi_{\mathfrak{g}}\|_{s+2,r_k}\|X_k\|_{0,r_k}\|X_k\|_{s+2,r_k}\leq  Ct_k^{s+2-2\alpha},\\
\label{EQ_Ufixed}\|\varphi_{X_k}^{*}(U_{k})\|_{s,r_{k+1}}&\leq C(\|U_k\|_{s,r_k}+\|U_k\|_{0,r_k}\|X_k\|_{s+1,r_{k}})\leq\\
\nonumber&\leq C(\|U_k\|_{s,r_k}+t_k^{s+1-\alpha}\|U_k\|_{0,r_k}).
\end{align}
To compute the $C^s$-norm for $U_k$, we rewrite it as
\begin{align*}
U_k&=Z_k-[\pi_{\mathfrak{g}},X_k]=[\pi_{\mathfrak{g}},h_1^{r_k}(Z_k)]+h_2^{r_k}([\pi_{\mathfrak{g}},Z_k])-[\pi_{\mathfrak{g}},X_k]=\\
&=[\pi_{\mathfrak{g}},(I-S_k)h_1^{r_k}(Z_k)]-\frac{1}{2}h_2^{r_k}([Z_k,Z_k]).
\end{align*}
By tameness of the Lie bracket, the first term can be bounded by
\begin{align*}
\|[\pi_{\mathfrak{g}},(I-S_k)& h_1^{r_k}(Z_k)]\|_{s,r_k}\leq C \|(I-S_k)h_1^{r_k}(Z_k)\|_{s+1,r_k}\leq\\
&\leq C t_k^{2s+1-p} \|h_1^{r_k}(Z_k)\|_{p-s,r_k}\leq C t_k^{2s+1-p} \|Z_k\|_{p,r_k}\leq\\
&\leq C t_k^{2s+1-p+\alpha}= C t_k^{-\frac{3}{2}\alpha-1},
\end{align*}
and using also the interpolation inequalities, for the second term we obtain
\begin{align*}
\|\frac{1}{2}h_2^{r_k}([Z_k,Z_k])&\|_{s,r_k}\leq C\|[Z_k,Z_k]\|_{2s,r_k}\leq C\|Z_k\|_{0,r_k}\|Z_k\|_{2s+1,r_k}\leq\\
&\leq C t_{k}^{-\alpha}\|Z_k\|_{s,r_k}^{\frac{p-(2s+1)}{p-s}}\|Z_k\|_{p,r_k}^{\frac{s+1}{p-s}}\leq C t_k^{-\alpha(1+\frac{p-(3s+2)}{p-s})}.
\end{align*}
Since $-\alpha(1+\frac{p-(3s+2)}{p-s})\leq -\frac{3}{2}\alpha-1$, these two inequalities imply that
\begin{equation}\label{EQ_U_sfixed}
\|U_k\|_{s,r_k}\leq Ct_k^{-\frac{3}{2}\alpha-1}.
\end{equation}
Using (\ref{EQ_X_n_nfixed}), we bound the $C^0$-norm of $U_k$ by
\begin{equation}\label{EQ_U_0fixed}
\|U_k\|_{0,r_{k}}\leq \|Z_k\|_{0,r_{k}}+\|[\pi_{\mathfrak{g}},X_{k}]\|_{0,r_k}\leq t_k^{-\alpha}+C\|X_{k}\|_{1,r_k}\leq Ct_{k}^{1-\alpha}.
\end{equation}
Using (\ref{EQ_Vfixed}), (\ref{EQ_Ufixed}), (\ref{EQ_U_sfixed}), (\ref{EQ_U_0fixed}) and that $s+2-2\alpha= -\frac{3}{2}\alpha-1$ we get
\begin{align*}
\|Z_{k+1}\|_{s,r_{k+1}}\leq C(t_k^{s+2-2\alpha}+t_k^{-\frac{3}{2}\alpha-1})\leq Ct_k^{-\frac{3}{2}\alpha-1}\leq C\delta^{1/\alpha}t_{k}^{-\frac{3}{2}\alpha}= C\delta^{1/\alpha}t_{k+1}^{-\alpha},
\end{align*}
and by taking $\delta$ even smaller, this finishes the induction.

From (\ref{EQ_Z_nfixed}) we conclude that for each $n\geq 1$, there is $k_n\geq 0$, such that
\[\|Z_{k+1}\|_{n,r_{k+1}}\leq t_k^{s+2}\|Z_{k}\|_{n,r_{k}}, \ \ \forall\ k\geq k_n.\]
By iterating this, we obtain
\[t_k^{s+2}\|Z_{k}\|_{n,r_{k}}\leq (t_kt_{k-1}\ldots t_{k_n})^{s+2}\|Z_{k_n}\|_{n,r_{k_n}}.\]
On the other hand we have that
\[t_kt_{k-1}\ldots t_{k_n}=t_{k_n}^{1+\frac{3}{2}+\ldots+(\frac{3}{2})^{k-k_n}}\leq t_{k_n}^{2(\frac{3}{2})^{k+1-k_n}}=t_k^3.\]
Therefore, we obtain a bound valid for all $k>k_n$
\[\|Z_{k}\|_{n,r_{k}}\leq t_{k}^{2(s+2)}\|Z_{k_n}\|_{n,r_{k_n}}.\]
Consider now $m>s$ and denote by $n:=4m-3s$. Applying the interpolation inequalities, for $k> k_{n}$, we obtain
\begin{eqnarray*}
\|Z_k\|_{m,r_k}&\leq& C_m  \|Z_k\|_{s,r_k}^{\frac{n-m}{n-s}}\|Z_k\|_{n,r_k}^{\frac{m-s}{n-s}}=C_m  \|Z_k\|_{s,r_k}^{\frac{3}{4}}\|Z_k\|_{n,r_k}^{\frac{1}{4}}\leq \\
&\leq& C_m t_k^{-\frac{3}{4}\alpha+\frac{1}{2}(s+2)}\|Z_{k_{n}}\|_{n,r_{k_{n}}}^{\frac{1}{4}}= C_m t_k^{-(s+\frac{7}{2})} \|Z_{k_{n}}\|_{n,r_{k_{n}}}^{\frac{1}{4}}.
\end{eqnarray*}
This inequality and inequality (\ref{EQ_X_n_lfixed}), for $l=s$, imply
\begin{align*}
\|X_{k}\|_{m,r_k}\leq C_m t_k^{s}\|Z_k\|_{m,r_k}\leq  t_k^{-\frac{7}{2}}\left( C_m \|Z_{k_{n}}\|_{n,r_{k_{n}}}^{\frac{1}{4}}\right).
\end{align*}
This shows that the series $\sum_{k\geq 0}\|X_k\|_{m,r_k}$ converges for all $m$. By Lemma \ref{Lemma_domain_of_flowfixed}, also $\sum_{k\geq
0}\|\varphi_{X_k}-I_{r_{k+1}}\|_{m,r_{k+1}}$ converges for all $m$ and, by (\ref{EQ_X_12fixed}),
\begin{equation*}
\sigma_1:=\sum_{k\geq 1}\|\varphi_{X_k}-I_{r_{k+1}}\|_{1,r_{k+1}}\leq C \sum_{k\geq 1}\|X_k\|_{1,r_k}\leq C \sum_{k\geq 1}\delta^{1/\alpha}\epsilon_k\leq C \delta^{1/\alpha}.
\end{equation*}
So we may assume that $\sigma_1\leq \theta$, and we can apply Lemma \ref{Lemma_convergent_embbedingsfixed} to conclude that the sequence of
maps
\[\psi_k:=\varphi_{X_0}\circ\ldots\circ \varphi_{X_{k|r}}:\overline{B}_r\rmap B_R,\]
converges uniformly in all $C^n$-norms to a map $\psi:\overline{B}_r\to B_R$ that satisfies $\|\psi-I_r\|_{1,r}\leq
C\delta^{1/\alpha}$. Hence, by shrinking $\delta$, we may apply Lemma \ref{Lemma_embeddingfixed}, to conclude
that $\psi$ is a local diffeomorphism. Since $\psi_{k}$ converges in the $C^1$-topology to $\psi$ and
$\psi_k^*(\pi)=(d\psi_k)^{-1}(\pi_{\psi_k})$, it follows that $\psi_k^*(\pi)$ converges in the $C^0$-topology to
$\psi^*(\pi)$. On the other hand, $Z_{k|r}=\psi_k^*(\pi)-\pi_{\mathfrak{g}|r}$ converges to $0$ in the $C^0$-norm,
hence $\psi^*(\pi)=\pi_{\mathfrak{g}|r}$. So $\psi$ is a Poisson map and a local diffeomorphism between
\[\psi:(\overline{B}_r,\pi_{\mathfrak{g}|r})\rmap (B_R,\pi).\]

\subsection{The implicit function theorem point of view}

By the heuristical interpretation of Poisson cohomology from subsection \ref{SSPoissonCoho}, the infinitesimal
condition corresponding to rigidity is the vanishing of the second Poisson cohomology. This fits very well in the
framework of \emph{implicit function theorems}. The various results of this type allow one to prove existence results
by analyzing the linearized equations underlying the problem (i.e.\ the infinitesimal data). A natural question is if such a
result is applicable in the case of Conn's theorem. The standard implicit function theorem is for finite dimensional
spaces, but there are versions also for infinite dimensional ones. A generalization to Banach spaces (e.g.\ spaces of
$C^n$-maps), with rather mild extra assumptions, is well known and proved very useful. In the case of Fr\'echet
spaces (e.g.\ spaces of $C^{\infty}$-maps), the situation is more subtle and the precise conditions are more involved.
The main general such result is Hamilton's implicit function theorem \cite{Ham}, which is discussed briefly in the
appendix \ref{Section_Richard}. Here we explain that, although intuitively, Conn's theorem is the manifestation of an
implicit function phenomenon, Hamilton's result cannot be applied (at least not in the way outlined in \cite{Desol}).
Still, Hamilton's implicit function theorem provides more insight into the problem; in particular the inequalities listed
in subsection \ref{Inequalities} assert that the natural operations (e.g.\ the Schouten bracket, flows of vector fields,
pullbacks) are ``tame'' maps in Hamilton's category of ``tame Fr\'echet spaces''.

The approach presented here is similar to the one in \cite{Desol}. For $0<r<R$, let $\mathcal{E}(R,r)$ denote the
space of embeddings of $\overline{B}_R$ in $\mathfrak{g}^*$ whose image contains $\overline{B}_r$ in the interior.
Consider the ``nonlinear complex''
\begin{equation}\label{EQ_12}
\mathcal{E}(R,r)\stackrel{P}{\rmap}\mathfrak{X}^2(\overline{B}_r)\stackrel{Q}{\rmap}\mathfrak{X}^3(\overline{B}_r),
\end{equation}
where the maps $P$ and $Q$ are given by
\begin{equation}\label{EQ_11}
P(\varphi):=\varphi_*(\pi_{\mathfrak{g}})_{|\overline{B}_r},\ \ Q(W):=\frac{1}{2}[W,W].
\end{equation}
We call this a nonlinear complex, because $ Q\circ P(\varphi)=0$, for all $\varphi\in\mathcal{E}(R,r)$. Exactness of the
complex around $\pi_{\mathfrak{g}}$ corresponds to the following rigidity notion: every bivector $\pi$ on
$\overline{B}_r$, close enough to $\pi_{\mathfrak{g}}$, such that $Q(\pi)=0$ (i.e.\ $\pi$ is Poisson) is of the form
$P(\varphi)$, for some $\varphi\in \mathcal{E}(R,r)$ (i.e.\ $\pi$ it is isomorphic to the restriction of
$\pi_{\mathfrak{g}}$ to $\varphi^{-1}(\overline{B}_r)$). In a finite dimensional situation, for this exactness property
to hold, it is enough to require infinitesimal exactness at $\varphi=I_R$ (see e.g.\ chapter 3, section 11.C of
\cite{Serre}, Part II). In our case, identifying the tangent space at $I_R$ of $\mathcal{E}(R,r)$ with
$\mathfrak{X}(\overline{B}_R)$, the linearized complex is
\begin{equation}\label{EQ_10}
\mathfrak{X}(\overline{B}_R)\stackrel{d_{\pi_{\mathfrak{g}}}\circ|_{r} }{\rmap}\mathfrak{X}^2(\overline{B}_r)\stackrel{d_{\pi_{\mathfrak{g}}}}{\rmap}\mathfrak{X}^3(\overline{B}_r).
\end{equation}
One easily checks that the maps $e\circ h_1$ and $h_2$ are homotopy operators for (\ref{EQ_10}), where $h_1$ and
$h_2$ are the operators from subsection \ref{SSTamehomotopy} and $e:\mathfrak{X}(\overline{B}_r)\to
\mathfrak{X}(\overline{B}_R)$ is a splitting of the restriction map. So the infinitesimal sequence splits, and moreover,
by choosing $e$ to be a tame extension operator (see e.g.\ Corollary 1.3.7 \cite{Ham}), we obtain tame homotopy
operators.

To apply Hamilton's result \cite{Ham2}, first one needs to check that the maps in (\ref{EQ_11}) are smooth and
``tame'' (which is routine, and follows from results in \cite{Ham}), and second, that the infinitesimal complex splits at
\emph{all points} near $I_R$. At $\varphi\in \mathcal{E}(R,r)$, this complex is isomorphic to
\begin{equation}\label{EQ_13}
\mathfrak{X}(\overline{B}_R)\stackrel{d_{P(\varphi)}\circ|_r\circ\varphi_*}{\rmap}\mathfrak{X}^2(\overline{B}_r)\stackrel{d_{P(\varphi)}}{\rmap}\mathfrak{X}^3(\overline{B}_r),
\end{equation}
and its exactness implies vanishing of $H^2_{P(\varphi)}(\overline{B}_r)$. Yet this cohomology doesn't vanish on any
neighborhood of $I_R$! This, together with the fact that the nonlinear complex is not exact in general, is explained
below in the case of $\mathfrak{so}(3)$. So, this natural attempt to apply Hamilton's implicit function theorem doesn't
work. In \cite{Desol}, the author works instead of $\mathcal{E}(R,r)$ with the space of diffeomorphisms of
$\mathfrak{g}^*$ with support in $\overline{B}_R$. Yet, the same problem occurs when trying to apply Hamilton's
result, as one can easily check. Nevertheless, this point of view gives more insight into the problem. Of course, there
are other geometric situations in which Hamilton's theorem is the right tool to prove rigidity (see e.g.\
\cite{Ham2,Ham3} and also Appendix \ref{Section_Richard}).

\subsubsection{Smooth deformations of the ball in $\mathfrak{so}(3)^*$}

Consider the linear Poisson structure on $\mathbb{R}^3$ corresponding to $\mathfrak{so}(3)^*$
\[\pi_{\mathfrak{so}(3)}:=x\frac{\partial}{\partial y}\wedge \frac{\partial}{\partial z}+y\frac{\partial}{\partial z}\wedge\frac{\partial}{\partial x}+z\frac{\partial}{\partial x}\wedge \frac{\partial}{\partial y}.\]
The following shows that (\ref{EQ_13}) is not exact on a neighborhood of $I_R$.

\begin{proposition}
There exist open embeddings $\varphi:\overline{B}_R\hookrightarrow \mathbb{R}^3$, arbitrary close to the identity,
such that
\[H^2_{\pi}(\overline{B}_r)\neq 0, \ \ \textrm{for }\ \pi:=\varphi_*(\pi_{\mathfrak{so}(3)})_{|_r}.\]
\end{proposition}

\begin{proof}
Consider coordinates $(\rho,\theta,z)$ on $\mathbb{R}^3\backslash\{(0,0,z)|z\in\mathbb{R}\}$, where $\rho>0$ is the
distance to the origin, $\theta\in S^1$ is the angle with the $xz$-plane and $z$ satisfies $-\rho<z<\rho$. In these
coordinates the Poisson structure becomes
\[\pi_{\mathfrak{so}(3)}=\frac{\partial}{\partial z} \wedge \frac{\partial}{\partial \theta}.\]
Consider the bivector
\[W:=\chi(\rho)\frac{\partial}{\partial \rho} \wedge \frac{\partial}{\partial z},\]
where $\chi:[0,\infty)\to \mathbb{R}$ is a smooth function satisfying $\chi(\rho)=0$ for $\rho\in[0,r]$ and
$\chi(\rho)>0$ for $\rho>r$. Note that, in the coordinates $(x,y,z)$, $W$ extends to
\[U:=\mathbb{R}^3\backslash \{(0,0,z)| |z|\geq r\},\] and that it vanishes on $B_r$. Clearly $[\pi_{\mathfrak{so}(3)},W]=0$.

We claim that, for an open $\widetilde{U}\subset U$ that contains a circle of the form
$S^1_{\rho_0,z_0}:=\{(\rho_0,\theta,z_0) | \theta\in S^1\}$, for $\rho_0>r$, the class of $[W]\in
H^2_{\pi_{\mathfrak{so}(3)}}(\widetilde{U})$ is nontrivial. Assume that $X=A\frac{\partial}{\partial
\rho}+B\frac{\partial}{\partial \theta}+C\frac{\partial}{\partial z}$ is a vector field on $\widetilde{U}$ such that
$[\pi_{\mathfrak{so}(3)},X]=W$. In particular, this equation implies that
\[\frac{\partial A}{\partial \theta}(\rho_0,\theta,z_0)=\chi(\rho_0).\]
Integrating $\int_{\theta=0}^{\theta=2\pi}$, we obtain a contradiction: $0=2\pi\chi(\rho_0)$.

Thus, every embedding $\varphi:\overline{B}_R\hookrightarrow \mathbb{R}^3$, such that for some $\rho_0>r$:
\[S^1_{\rho_0,z_0}\subset \varphi^{-1}(\overline{B}_r)\subset U,\]
satisfies that $H^2_{\pi}(\overline{B}_r)\neq 0$, for $\pi=\varphi_*(\pi_{\mathfrak{so}(3)})_{|\overline{B}_r}$. For
example, the family:
\begin{equation}\label{famdiff}
\varphi_{\epsilon}(x,y,z):=(e^{-\epsilon}x,e^{-\epsilon}y,e^{{\epsilon}}z),\ \ {\epsilon}>0.
\end{equation}
\end{proof}

The class used in the proof comes from a nontrivial deformation $\widetilde{\pi}$ of $\pi_{\mathfrak{so}(3)}$ on $U$,
which we explain in the sequel. Geometrically, the Poisson structure $\widetilde{\pi}$ coincides with
$\pi_{\mathfrak{so}(3)}$ on $B_r$, but the leaves outside $\overline{B}_r$ are noncompact, and spin around $\partial
\overline{B}_r$. To give $\widetilde{\pi}$ explicitly, we describe a regular Poisson structure on an open in
$\mathbb{R}^3$ as a pair $(\alpha,\eta)$ consisting of a non-vanishing 1-form and a non-vanishing 2-form such that
\[\alpha\wedge d\alpha=0, \ \ \alpha\wedge\eta\neq 0.\]
The foliation is the kernel of $\alpha$ and the symplectic structure is the restriction of $\eta$ to the leaves.

The linear Poisson structure $\pi_{\mathfrak{so}(3)}$ is represented on $\mathbb{R}^3\backslash\{0\}$ by the pair
\[(d\rho,\omega), \ \ \omega=\frac{1}{\rho^2}\left(xdy\wedge dz+ydz\wedge dx+zdx\wedge dy\right).\]
We define $\widetilde{\pi}$ on $U\backslash\{0\}$ in terms of the pair
\[(d\rho+\chi(\rho)d\theta, \omega),\]
where $\chi$ is defined in the proof above. Clearly, $\widetilde{\pi}$ extends to $U$, and coincides on $B_r$ with
$\pi_{\mathfrak{so}(3)}$. Since the kernel of $\omega$ is spanned by $\frac{\partial}{\partial \rho}$, it follows that
$(d\rho+\chi(\rho)d\theta)\wedge \omega$ is a volume form on $U\backslash\{0\}$.

Using $\widetilde{\pi}$, we prove that the nonlinear complex (\ref{EQ_11}) is not exact.

\begin{proposition}
There are Poisson structures $\pi$ on $\overline{B}_r$, arbitrary close to $\pi_{\mathfrak{so}(3)}$, such that
$(\overline{B}_r,\pi)$ cannot be embedded in $(\mathbb{R}^3,\pi_{\mathfrak{so}(3)})$.
\end{proposition}
\begin{proof}
Consider the vector field $X=\frac{\partial}{\partial \theta}-\chi(\rho)\frac{\partial}{\partial \rho}$ on $U\backslash
\overline{B}_r$. Note that $X$ is tangent to the symplectic leaves of $\widetilde{\pi}$, and that its flow lines are of
the form
\[(\rho(t),t+\theta_0, z_0),\]
where $\rho(t)$ is a decreasing function which tends to $r$. In particular, the boundary leaf
$S:=U\cap\partial\overline{B}_r$ has nontrivial holonomy.

Consider the smooth deformation of $\pi_{\mathfrak{so}(3)}$
\[\pi_{\epsilon}:=\varphi_{\epsilon*}(\widetilde{\pi})_{|\overline{B}_r},\ \ {\epsilon}>0,\]
where $\varphi_{\epsilon}$ is given in (\ref{famdiff}). For $\epsilon>0$, $(\overline{B}_r,\pi_{\epsilon})$ cannot be
embedded in $(\mathbb{R}^3,\pi_{\mathfrak{so}(3)})$, since the leaf $\varphi_{\epsilon}(S)\cap\overline{B}_r$ has
nontrivial holonomy.
\end{proof}

\section{Some notations and conventions}\label{SSNotConv}

We use the following conventions:\\
$\bullet$ the Lie bracket of vector fields is defined by
\[L_X(L_Y(f))-L_Y(L_X(f))=L_{[X,Y]}(f),\]
$\bullet$ the Schouten bracket on $\mathfrak{X}^{\bullet}(M)$ is defined by \index{Schouten
bracket}\index{multivector fields}
\begin{align}
&[X_1\wedge\ldots \wedge X_p,Y_1\wedge\ldots \wedge Y_q]=\label{EQ_Schouten_bracket}\\
\nonumber &=\sum_{i,j}(-1)^{i+j}[X_i,Y_j]\wedge X_1\wedge\ldots\wedge \widehat{X}_i\wedge\ldots \wedge X_p\wedge Y_1\wedge\ldots\wedge \widehat{Y}_j\wedge\ldots \wedge Y_q,
\end{align}
\noindent$\bullet$ we define the Lie bracket on the Lie algebra $\mathfrak{g}$ of a Lie group $G$ using right
invariant vector fields:
\[[X,Y]^r=[X^r,Y^r],\ \ [X,Y]^l=-[X^l,Y^l],\]
where for $Z\in\mathfrak{g}$ we denote by $Z^r, Z^l\in\mathfrak{X}(G)$ the right (respectively left) invariant
extension to $G$:\index{right invariant}
\[Z^r_g:=dr_g(Z), \ \ Z^l_g:=dl_g(Z).\]
Denoting as usual\index{adjoint action}
\[Ad_g(Y):=\frac{d}{d\epsilon}_{|\epsilon=0}g\exp(\epsilon Y) g^{-1}, \ \textrm{ and } ad_X(Y):=\frac{d}{d\epsilon}_{|\epsilon=0}Ad_{\exp(\epsilon X)} (Y),\]
the Lie bracket also satisfies
\[ad_X(Y)=[Y,X].\]

\index{Tube lemma}The Tube Lemma from topology states that a neighborhood of $K\times \{x\}$ in $K\times X$,
where $K$ and $X$ are topological spaces, with $K$ compact, contains a tube $K\times U$, where $U$ is a
neighborhood of $x$. We will often use this result in the following form:
\begin{TubeLemma}\label{TubeLemma}
Let $M$ be a topological space and let $\{U_t\}_{t\in[0,1]} \subset M$ be a family of opens, such that
$\cup_{t}(\{t\}\times U_t)$ is open in $[0,1]\times M$. Then $\cap_{t\in [0,1]}U_t$ is open in $M$.
\end{TubeLemma}

\clearpage \pagestyle{plain}

\chapter{Lie algebroids and Lie groupoids}\label{CHLieAlgLieGroupoids}
\pagestyle{fancy}
\fancyhead[CE]{Chapter \ref{CHLieAlgLieGroupoids}} 
\fancyhead[CO]{Lie algebroids and Lie groupoids} 

In this chapter we recall some results about Lie algebroids and Lie groupoids, and we also prove some lemmas on the
subject to be used in the following chapters. For the general theory, we recommend \cite{MM,MackenzieGT}.

\section{Lie groupoids}

\subsection{Definition of a Lie groupoid}

One usually thinks about a group as the space of symmetries of an object. Groupoids generalize this idea, by putting
together the symmetries of several objects and also ``external symmetries'' between different objects. Groupoids that
come with a smooth structure, compatible with the groupoid operations, are called Lie groupoids.

More precisely, a \textbf{Lie groupoid}\index{Lie groupoid} is given by two smooth manifolds $\mathcal{G}$ and
$M$, whose points are called \textbf{arrows} and \textbf{objects} respectively, and smooth structure maps $s$, $t$,
$u$, $\iota$ and $m$, as follows
\begin{itemize}
\item $s,t:\mathcal{G}\to M$, are surjective submersions, called \textbf{source} and \textbf{target},
\item $u:M\to \mathcal{G}$, denoted also by $u(x)=1_{x}$, is called the \textbf{unit} map,
\item $\iota:\mathcal{G}\to \mathcal{G}$, denoted also by $\iota(g)=g^{-1}$, is the \textbf{inversion map},
\item $m:\mathcal{G}\times_{M} \mathcal{G}\to \mathcal{G}$, denoted also by $m(g,h)=gh$, is called the \textbf{multiplication} and
is defined on \[\mathcal{G}\times_{M} \mathcal{G}=\{(g,h) | s(g)=t(h)\}.\] $\mathcal{G}\times_{M} \mathcal{G}$
is a smooth manifold since $s,t$ are submersions.
\end{itemize}
These structure maps are required to satisfy the natural axioms
\begin{itemize}
\item $s(1_x)=t(1_x)=x$,
\item $t(g^{-1})=s(g)$, $s(g^{-1})=t(g)$,
\item $1_{t(g)}g=g1_{s(g)}=g$,
\item $gg^{-1}=1_{t(g)}$, $g^{-1}g=1_{s(g)}$,
\item $s(gh)=s(h)$, $t(gh)=t(g)$, whenever $s(g)=t(h)$,
\item $g(hk)=(gh)k$, whenever $s(g)=t(h)$ and $s(h)=t(k)$,
\end{itemize}
for all $x\in M$ and $g,h,k\in\mathcal{G}$.

We usually denote groupoids by $\mathcal{G}\rightrightarrows M$.\\

There are natural examples of groupoids (e.g.\ coming from foliation theory) for which the space of arrows is not a
Hausdorff manifold; that is, $\mathcal{G}$ satisfies all the axioms of a manifold (it has a countable atlas with smooth
transition functions), except for Hausdorffness. Our overall assumptions are that $M$ is a smooth Hausdorff manifold
and that $\mathcal{G}$ is smooth manifold but not necessarily Hausdorff.

\subsection{Examples}\label{Examples Lie groupoids}

\subsubsection*{Lie groups} A Lie group is a Lie groupoid over a point.

\subsubsection*{The action groupoid}
Let $G$ be a Lie group acting on an manifold $M$. One forms the corresponding \textbf{action groupoid} $G\ltimes M
\rightrightarrows M$ with structure maps\index{action groupoid}
\begin{align*}
s(g,x)=x, \ \ &t(g,x)=gx, \ \ 1_x=(1_{G},x),\ \ (g,x)^{-1}=(g^{-1},gx), \\
&(g,x)(h,y)=(gh,y)\ \  \textrm{where} \ \ x=hy.
\end{align*}

\subsubsection*{The fundamental groupoid}\index{fundamental groupoid}
The \textbf{fundamental groupoid} of a manifold $M$, denoted by $\Pi_1(M)\rightrightarrows M$, is the space of paths
in $M$ modulo homotopy relative to the endpoints. The source (respectively target) of a path is its initial (respectively
final) point, the composition is concatenation, the units are represented by constant paths, and the inverse of a path is
the same path with reversed orientation.

\subsubsection*{The gauge groupoid}\index{gauge groupoid}
Let $P\to M$ be a principal bundle with structure group $G$. The \textbf{gauge groupoid} of $P$, denoted by
$P\times_G P\rightrightarrows M$, has as space of arrows between $x,y\in M$, equivalence classes $[p_x,p_y]$ with
$p_x\in P_x$, $p_y\in P_y$, where
\[[p_x,p_y]=[q_x,q_y]\ \Leftrightarrow  \ q_x=p_xg,\  q_y=p_yg \textrm{  for some  }g\in G.\]
The units are $1_x=[p_x,p_x]$, inversion $[p_x,p_y]\to [p_y,p_x]$ and composition is
\[[p_x,p_y][q_y,q_z]=[p_x,q_z g],\textrm{  where  }g\in G \textrm{  satisfies  }p_y=q_y g.\]

\subsubsection*{The holonomy groupoid}\index{holonomy groupoid}\index{holonomy}
To a foliation $\mathcal{F}$ on $M$, one associates the \textbf{holonomy groupoid}
\[\textrm{Hol}(\mathcal{F})\rightrightarrows M.\]
Points on different leaves of $\mathcal{F}$ have no arrows between them, and the arrows between $x$ and $y$, both
in the same leaf $L$, are the holonomy classes of paths from $x$ to $y$. Recall \cite{MM} that, given a path
$\gamma:[0,1]\to L$, with $\gamma(0)=x$ and $\gamma(1)=y$, the \textbf{holonomy} of $\gamma$ is the germ of a
diffeomorphism
\[hol({\gamma}): T_x\rmap T_y,\]
where $T_x$ and $T_y$ are small transversal submanifolds to $L$ at $x$ and $y$ respectively. The map
$hol({\gamma})$ can be computed using parallel transport of the Ehresmann connection induced by $\mathcal{F}$ on
a tubular neighborhood $T\subset M$ of $L$. Two paths $\gamma_1,\gamma_2:x\to y$ and said to be in the same
\textbf{holonomy class}, if, after restricting to some open of $x$ in $T_x$, $hol(\gamma_1)=hol(\gamma_2)$. The space
of arrows $\textrm{Hol}(\mathcal{F})$ is a smooth manifold \cite{MM}, but may fail to be Hausdorff.

\subsection{Terminology and some properties of Lie groupoids}
For a Lie groupoid $\mathcal{G}\rightrightarrows M$, we have that
\begin{itemize}
\item $\mathcal{G}(x,y):=s^{-1}(x)\cap t^{-1}(y)$ - the set of arrows from $x$ to $y$ - is a (Hausdorff) submanifold of $\mathcal{G}$.

\item $\mathcal{G}(x,x):=\mathcal{G}_x$ - called the \textbf{isotropy group} at $x$ - is a Lie group.\index{isotropy group}

\item $\mathcal{G}x:=t(s^{-1}(x))$- called the \textbf{orbit} through $x$ - is an immersed submanifold of $M$.\index{orbit}

\item $\mathcal{G}(x,-):=s^{-1}(x)$ - the $s$-\textbf{fiber} over $x$ - is a principal $\mathcal{G}_x$ bundle over $\mathcal{G}x$ with projection the target map
$t: \mathcal{G}(x,-)\to \mathcal{G}x$.\index{s-fiber}

\item If $X\subset M$ is a submanifold transverse to the orbits of $\mathcal{G}$, then
\[\mathcal{G}_{|X}:=s^{-1}(X)\cap t^{-1}(X)\rightrightarrows X\]
is a Lie groupoid. In particular, this holds for opens $X\subset M$.

\item A submanifold $S\subset M$ is called \textbf{invariant}, if it is a union of orbits.\index{invariant submanifold}


\item $\mathcal{G}$ is called \textbf{proper}, if it is Hausdorff and the map\index{proper groupoid}
\[(s,t):\mathcal{G}\rmap M\times M,\]
is proper (i.e.\ preimages of compact sets are compact).

\item $\mathcal{G}$ is called \textbf{transitive} if the map\index{transitive groupoid}
\[(s,t):\mathcal{G}\rmap M\times M,\]
is a surjective submersion. In this case, $\mathcal{G}$ is isomorphic to a gauge groupoid
\[\mathcal{G}\cong P\times_G P,\]
where the principal $G$-bundle $P$ is any $s$-fiber. 
\end{itemize}

\subsection{Existence of invariant tubular neighborhoods}\label{Subsection_inv_tub_nbd}

In chapter \ref{ChRigidity}, we will use the following lemma.

\begin{lemma}\label{Lemma_tubular_neighborhood}\index{invariant tubular neighborhood}
Let $\mathcal{G}\rightrightarrows M$ be a proper Lie groupoid with connected $s$-fibers and let $S\subset M$ be a
compact invariant submanifold. There exists a tubular neighborhood $E\subset M$ (where $E\cong T_SM/TS$) and a
metric on $E$, such that, for all $r>0$, the closed tube $E_r:=\{v\in E: |v|\leq r\}$ is $\mathcal{G}$-invariant.
\end{lemma}

The proof is inspired by the proof of Theorem 4.2 from \cite{Posthuma}, and it uses the following result (see Definition
3.11, Proposition 3.13 and Proposition 6.4 in  \emph{loc.cit.}).

\begin{lemma}
On the base of a proper Lie groupoid there exist Riemannian metrics such that every geodesic which emanates
orthogonally from an orbit stays orthogonal to any orbit it passes through. Such metrics are called \textbf{adapted}.
\end{lemma}

\begin{proof}[Proof of Lemma \ref{Lemma_tubular_neighborhood}]
Let $g$ be an adapted metric on $M$ and let $E:=TS^{\perp}\subset T_SM$ be the normal bundle. By rescaling $g$,
we may assume that
\begin{enumerate}[(1)]
\item the exponential is defined on $E_2$ and on $\textrm{int}(E_2)$ it is an open embedding;
\item for all $r\in (0,1]$ we have that
\[\exp(E_r)=\{p\in M: d(p,S)\leq r\},\]
where $d$ denotes the distance induced by the Riemannian structure.
\end{enumerate}
For these assertions, see the proof of Theorem 20, Chapter 9 in \cite{Spivak}.

Let $v\in E_1$ with $|v|:=r$ and base point $x$. We claim that the geodesic $\gamma(t):=\exp(tv)$ at $t=1$ is normal
to $\exp (\partial E_r)$ at $\gamma(1)$, i.e.
\[T_{\gamma(1)}\exp (\partial E_r)=\dot{\gamma}(1)^{\perp}.\]
Let $\overline{B}_r(x)$ be the closed ball of radius $r$ centered at $x$. By the Gauss Lemma
\[\dot{\gamma}(1)^{\perp}=T_{\gamma(1)}\partial \overline{B}_r(x),\]
and by (2), $\overline{B}_r(x)\subset \exp(E_r)$, therefore the boundaries of $\overline{B}_r(x)$ and $\exp(E_r)$ are
tangent at $\gamma(1)$. This proves the claim.

By assumption, $S$ is a union of orbits, therefore the geodesics $\gamma(t):=\exp(tv)$, for $v\in E$, start normal to the
orbits of $\mathcal{G}$, and by the property of the metric, they continue to be orthogonal to the orbits. Hence, the
claim implies that the orbit are tangent to $\exp(\partial E_r)$, for $r\in (0,1)$. Since the orbits are connected, it follows
that $\exp(\partial E_r)$ is invariant for all $r\in (0,1)$. Let $\lambda: [0,\infty)\to [0,1)$ be a diffeomorphism that is the
identity around $0$, and define the embedding $E\hookrightarrow M$ by $v\mapsto \exp (\lambda(|v|)/|v|v)$.
\end{proof}

\subsection{Representations}

A \textbf{representation}\index{representation, Lie groupoid} of a Lie groupoid $\mathcal{G}\rightrightarrows M$ is
a vector bundle $p: V\to M$ endowed with a linear action of $\mathcal{G}$ on $V$, i.e.\ a map
\[\mu:\mathcal{G}\times_{M} V \rmap V, \ \ \mu(g,v):=g\cdot v\in V_{t(g)}, \ \textrm{ for }g,v \textrm{ such that } s(g)=p(v),\]
which is linear in $v$ and satisfied the usual axioms
\[1_{p(v)}\cdot v=v, \ \ (gh)\cdot v=g\cdot(h\cdot v).\]

\section{Lie algebroids}\label{Section_Lie_algebroids}

\subsection{Definition of a Lie algebroid}

A \textbf{Lie algebroid}\index{Lie algebroid} is a vector bundle $A\to M$ with a Lie bracket $[\cdot,\cdot]$ on its
space of sections $\Gamma(A)$ and a vector bundle map $\rho:A\to TM$, called the \textbf{anchor}\index{anchor},
which satisfy the Leibniz rule
\[[X,f Y]=f[X,Y]+\rho(X)(f)Y,\ \  \ X,Y\in\Gamma(A), f\in C^{\infty}(M).\]
This axiom implies that $\rho$ is a Lie algebra homomorphism $\Gamma(A)\to \mathfrak{X}(M)$.

\subsection{The Lie algebroid of a Lie groupoid}

The infinitesimal counterpart of a Lie groupoid $\mathcal{G}\rightrightarrows M$ is a Lie algebroid
\[A=Lie(\mathcal{G})\rmap M.\]
As a vector bundle, $A$ is the pullback by the unit map $u:M\to \mathcal{G}$ of the space of $s$-vertical vectors on
$\mathcal{G}$
\[A=u^*(T^s\mathcal{G}), \textrm{ where } T^s\mathcal{G}:=\ker (ds:T\mathcal{G}\to TM).\]
The anchor is the differential of the target map $\rho:=dt_{|A}$. To define the Lie bracket, let $r_g$ denote right
multiplication by an arrow $g:x\to y$,
\[r_g:\mathcal{G}(y,-)\diffto \mathcal{G}(x,-), \ h\mapsto hg.\]
An $s$-vertical vector field $X\in \Gamma(T^s\mathcal{G})$ is called \textbf{right invariant}\index{right invariant}, if
\[dr_g(X_h)=X_{hg},\ \ \textrm{ whenever  }s(h)=t(g).\]
The space of right invariant vector fields is denoted by $\Gamma(T^s\mathcal{G})^\mathcal{G}$. A section
$X\in\Gamma(A)$ extends to a unique right invariant vector field $X^r$,
\[X^r_g:=dr_g(X_{t(g)}), \ g\in \mathcal{G},\]
and the assignment $X\mapsto X^r$ defines an isomorphism between $\Gamma(A)\cong
\Gamma(T^s\mathcal{G})^\mathcal{G}$. This isomorphism induces the Lie bracket on $\Gamma(A)$.

Compared with the theory of Lie algebras and Lie groups, a major difference is that not every Lie algebroid is the Lie algebroid of a Lie groupoid.
\begin{definition}
A Lie algebroid $A\to M$, for which there is a Lie groupoid $\mathcal{G}\rightrightarrows M$ such that $A\cong
Lie(\mathcal{G})$, is called \textbf{integrable}.\index{integrable Lie algebroid}
\end{definition}
More on integrability will be presented in section \ref{Section_Integrability}.

\subsection{Examples}\label{Subsection_Examples_of_Lie_algebroids}

\subsubsection*{Lie algebras} A Lie algebroid over a point is the same as a Lie algebra.

\subsubsection*{The tangent bundle and foliations}
The tangent bundle of a manifold $M$ is a Lie algebroid with $\rho=\textrm{Id}_{TM}$. A Lie groupoid integrating
$TM$ is the fundamental groupoid $\Pi_1(M)$.

The tangent bundle $T\mathcal{F}\subset TM$ of a foliation $\mathcal{F}$ is a Lie algebroid. A Lie groupoid that
integrates it is the holonomy groupoid $\textrm{Hol}(\mathcal{F})\rightrightarrows M$.

\subsubsection*{The action Lie algebroid}
Let $\mathfrak{g}$ be a Lie algebra acting on a manifold $M$ by the Lie algebra homomorphism
$\alpha:\mathfrak{g}\to \mathfrak{X}(M)$. This defines a Lie algebroid $\mathfrak{g}\ltimes M$ called the
\textbf{action algebroid}\index{action algebroid}. The vector bundle is $\mathfrak{g}\times M\to M$, the anchor is
given by $\alpha$ and the bracket is uniquely determined by the Leibniz rule and the fact that on constant sections it
coincides with that of $\mathfrak{g}$. If the action integrates to the action of a Lie group $G$ of $\mathfrak{g}$, then
$\mathfrak{g}\ltimes M$ is the Lie algebroid of the action groupoid $G\ltimes M$.

\subsubsection*{Transitive Lie algebroids}
A Lie algebroid $A\to M$ with surjective anchor is called \textbf{transitive}\index{transitive Lie algebroid}. If $A$ is
integrable and $M$ is connected, then any Lie groupoid integrating $A$ is automatically transitive, hence a gauge
groupoid.

Conversely, let $p:P\to M$ be a principal $G$-bundle. The Lie algebroid of the gauge groupoid $P\times_G P$, as a
vector bundle, is $A(P):=TP/G$\index{gauge groupoid}. The anchor is induced by $dp:TP\to TM$ and the Lie bracket
comes from the identification between sections of $A(P)$ and invariant vector fields on $P$: \[\Gamma(A(P))\cong
\mathfrak{X}(P)^G.\]

\subsubsection*{The cotangent Lie algebroid of a Poisson manifold}
For a Poisson manifold $(M,\pi)$, $T^*M$ is a Lie algebroid, called \textbf{the cotangent Lie
algebroid}\index{cotangent Lie algebroid}, (see section \ref{SPreliminaries}) with anchor $\pi^{\sharp}$ and Lie
bracket
\begin{equation}\label{br}
[\alpha,\beta]_{\pi}:=L_{\pi^{\sharp}(\alpha)}(\beta)-L_{\pi^{\sharp}(\beta)}(\alpha)-d\pi(\alpha,\beta).
\end{equation}

\subsubsection*{Dirac structure}

A Dirac structure $L\subset TM\oplus T^*M$ is a Lie algebroid with anchor $p_T:L\to TM$, and the Dorfman bracket (\ref{EQ_Dorfman_bracket}).

For $\pi$ Poisson, the cotangent Lie algebroid $T^*M$ and the associated Dirac structure $L_{\pi}$ are isomorphic
\[(T^*M,[\cdot,\cdot]_{\pi},\pi^{\sharp})\diffto (L_{\pi},[\cdot,\cdot]_D,p_T), \ \xi\mapsto \pi^{\sharp}(\xi)+\xi.\]

\subsection{Some properties of Lie algebroids}\label{subsection_some_properties_of_Lie_algebroids}
Let $A$ be a Lie algebroid over $M$.

\subsubsection*{The isotropy Lie algebra}

\index{isotropy Lie algebra}For $x\in M$, the Lie bracket on $\Gamma(A)$ induces a natural Lie algebra structure on
the kernel of the anchor at $x$, called the \textbf{isotropy Lie algebra} at $x$ and denoted by
$\mathfrak{g}_x:=\ker(\rho)_x$. For $X,Y\in \mathfrak{g}_x$, the bracket is defined by
\[[X,Y]:=[\widetilde{X},\widetilde{Y}]_{|x},\]
where $\widetilde{X},\widetilde{Y}\in\Gamma(A)$ are extensions of $X$, respectively of $Y$.

In the integrable case, when $A=Lie(\mathcal{G})$, the isotropy Lie algebra at $x$ is the Lie algebra of the isotropy
group $\mathcal{G}_x$ at $x$.

\subsubsection*{The orbits}
\index{orbit}The subbundle $\rho(A)\subset TM$ is a singular involutive distribution, which integrates to a partition of
$M$ by immersed submanifolds, called the \textbf{orbits} of $A$. The orbit through $x$, denoted by $Ax$, satisfies
$T_x(Ax)=\rho(A_x)$. If $A$ is integrable, with $A=Lie(\mathcal{G})$, then the orbits of $A$ coincide with the
connected components of the orbits of $\mathcal{G}$.

\subsubsection*{Restrictions}

If $N\subset M$ is a submanifold satisfying
\[\rho(A_x)\subset T_xN, \ \forall\  x\in N,\]
one can restrict $A$, to $N$, i.e.\ there is a unique bracket on $\Gamma(A_{|N})$ such that the restriction map is a
Lie algebra homomorphism. With this bracket,
\[(A_{|N},[\cdot,\cdot]_N,\rho_{|N})\]
becomes a Lie algebroid over $N$ called the \textbf{restriction} of $A$ to $N$. In particular, for an orbit $O=Ax$, the
restricted Lie algebroid $A_{|O}$ is transitive.

\subsubsection*{A-paths}

The orbit $Ax$ can be described as the set of points $y\in M$ for which there exists a smooth path $\gamma:[0,1]\to M$ joining
$x$ to $y$, and a smooth path $a:[0,1]\to A$ lying over $\gamma$, such that
\[\rho(a)(t)=\frac{d\gamma}{dt}(t), \ \  t\in[0,1].\]
Such a pair $(a,\gamma)$ is called an $A$-\textbf{path}\index{A-path}.

For the cotangent algebroid of a Poisson manifold, this notion coincides with that of a cotangent path introduced in
section \ref{SPreliminaries}.

\subsubsection*{The fundamental
integration of $A$}

For a Lie groupoid $\mathcal{G}$, the union of the connected components of the $s$-fibers containing the
corresponding unit is an open subgroupoid of $\mathcal{G}$, denoted by $\mathcal{G}^{\circ}$. We call this the
$s$-\textbf{connected} subgroupoid of $\mathcal{G}$.\index{s-connected}

Putting together the universal covers of the $s$-fibers of $\mathcal{G}^{\circ}$, we obtain a Lie groupoid
$\widetilde{\mathcal{G}}$ with the following property: it is the unique Lie groupoid (up to isomorphism) with
1-connected $s$-fibers integrating the Lie algebroid $A=Lie(\mathcal{G})$.

Any integration with 1-connected $s$-fibers of a Lie algebroid $A$ is called the \textbf{fundamental
integration}\index{fundamental integration} of $A$.

\subsection{Representations}\label{subsection_representations}

Let $(A,[\cdot,\cdot],\rho)$ be an algebroid over a manifold $M$. An $A$-\textbf{connection} on a vector bundle $V\to
M$ is a bilinear map\index{A-connection}
\[\nabla:\Gamma(A)\times \Gamma(V)\rmap \Gamma(V),\]
satisfying
\[\nabla_{fX}(\lambda)=f\nabla_X(\lambda),\ \ \nabla_{X}(f\lambda)=f\nabla_X(\lambda)+L_{\rho(X)}(f)\lambda.\]

A \textbf{representation}\index{representation, Lie algebroid} of $A$ is an $A$-connection that is flat:
\[\nabla_{[X,Y]}=\nabla_X\circ \nabla_Y - \nabla_Y\circ \nabla_X.\]

Let $\mathcal{G}$ be a Lie groupoid with Lie algebroid $A$. A representation $V$ of $\mathcal{G}$ can be
differentiated to a representation of $A$. Explicitly, for $X\in\Gamma(A)$ and $\lambda\in\Gamma(V)$, the
$A$-connection is given by (see Definition 3.6.8 \cite{MackenzieGT})
\begin{align}\label{EQ_differentiating_reps}
\nabla_X(\lambda)_x:=\frac{d}{dt}\left(\phi^{t}(X,x)^{-1}\cdot \lambda_{\varphi^t(X,x)}\right)_{|t=0},
\end{align}
where $g\mapsto \phi^t(X,g)$ is the flow of the right invariant extension $X^r$ of $X$ to $\mathcal{\mathcal{G}}$
and $x\mapsto \varphi^t(X,x)$ is the flow of $\rho(X)$ on $M$. Since $dt(X^r)=\rho(X)$, we also have that
$t(\phi^t(X,x))=\varphi^t(X,x)$.

Conversely, a representation of $A$ can be integrated to a representation of the $s$-fiber 1-connected groupoid of $A$,
but, in general, not to $\mathcal{G}$.

We prove now a result that will be used in chapter \ref{ChRigidity}, namely that representations on ideals of $A$ can
be integrated to any $s$-connected Lie groupoid. We call a subbundle $I\subset A$ an \textbf{ideal}\index{ideal, Lie
algebroid}, if $\rho(I)=0$ and $\Gamma(I)$ is an ideal of the Lie algebra $\Gamma(A)$. Using the Leibniz rule, one
easily sees that, if $A\neq I$, then the second condition implies the first. An ideal $I$ is naturally a representation of
$A$, with $A$-connection given by the Lie bracket \[\nabla_X(Y):=[X,Y], \ \ X\in\Gamma(A), \ Y\in\Gamma(I).\]

\begin{lemma}\label{Lemma_integrating_ideals}
Let $\mathcal{G}\rightrightarrows M$ be a Hausdorff Lie groupoid with Lie algebroid $A$ and let $I\subset A$ be an
ideal. If the $s$-fibers of $\mathcal{G}$ are connected, then the representation of $A$ on $I$ given by the Lie bracket
integrates to $\mathcal{G}$.
\end{lemma}

\begin{proof}

First observe that $\mathcal{G}$ acts on the possibly singular bundle of isotropy Lie algebras $\ker(\rho)\to M$ via the formula:
\begin{align}\label{EQ_Action}
g\cdot Y=\frac{d}{d\epsilon}\left(g\phantom{\cdot} exp(\epsilon Y) g^{-1} \right)_{|\epsilon=0},  \forall\  Y\in \ker(\rho)_{s(g)}.
\end{align}
Let $N(I)\subset \mathcal{G}$ be the subgroupoid consisting of elements $g$ that satisfy $g\cdot I_{s(g)}\subset
I_{t(g)}$. We will prove that $N(I)=\mathcal{G}$, and that the induced action of $\mathcal{G}$ on $I$ differentiates
to the Lie bracket.

Recall that a derivation\index{derivation on a vector bundle} on a vector bundle $E\to M$ (see section 3.4
\cite{MackenzieGT}) is a pair $(D,V)$, with $D$ a linear operator on $\Gamma(E)$ and $V$ a vector field on $M$,
satisfying
\[D(f\alpha)=fD(\alpha)+V(f)\alpha,  \ \forall\ \alpha\in\Gamma(E), f\in C^{\infty}(M) .\]
The flow of a derivation $(D,V)$, denoted by $\phi_D^t$, is a vector bundle map covering the flow $\varphi_V^t$ of
$V$, $\phi_D^t(x):E_x\to E_{\varphi_V^t(x)}$ (whenever $\varphi_V^t(x)$ is defined) that is the solution to the
following differential equation
\[\phi_{D}^0=\textrm{Id}_E,\ \ \frac{d}{dt}(\phi_D^t)^*(\alpha)=(\phi_D^t)^*(D \alpha),\]
where $(\phi_D^t)^*(\alpha)_x=\phi_D^{-t}(\varphi_V^{t}(x)) \alpha_{\varphi_V^{t}(x)}$.

For $X\in\Gamma(A)$, denote by $\phi^t(X,g)$ the flow of the corresponding right invariant vector field on
$\mathcal{G}$, and by $\varphi^t(X,x)$ the flow of $\rho(X)$ on $M$. Conjugation by $\phi^t(X)$ is an automorphism
of $\mathcal{G}$ covering $\varphi^t(X,x)$, which we denote by
\[C(\phi^t(X)):\mathcal{G}\rmap \mathcal{G}, \ g\mapsto \phi^t(X,t(g))g\phi^t(X, s(g))^{-1}.\]
Since $C(\phi^t(X))$ sends the $s$-fiber over $x$ to the $s$-fiber over $\varphi^t(X,x)$, its differential at the identity $1_x$ gives an
isomorphism
\[Ad(\phi^t(X)):A_{x}\rmap A_{\varphi^t(X,x)}, \ Ad(\phi^t(X))_x:=dC(\phi^t(X))_{|A_x}.\]
On $\ker(\rho)_x$, we recover the action (\ref{EQ_Action}) of $g=\phi^t(X,x)$. We have that:
\begin{align}\label{EQ_adjoint_diff}
\frac{d}{dt} &(Ad(\phi^t(X))^*Y)_x=\frac{d}{dt}Ad(\phi^{-t}(X,\varphi^t(X,x)))Y_{\varphi^t(X,x)}=\\
\nonumber &=-\frac{d}{ds}\left(Ad(\phi^{-t}(X,\varphi^t(X,x)))Ad(\phi^s(X,\varphi^{t-s}(X,x)))Y_{\varphi^{t-s}(X,x)}\right)_{|s=0}=\\
\nonumber &=Ad(\phi^{-t}(X,\varphi^t(X,x)))[X,Y]_{\varphi^t(X,x)}=Ad(\phi^t(X))^*([X,Y])_x,
\end{align}
for $Y\in\Gamma(A)$, where we have used the adjoint formulas from Proposition 3.7.1 \cite{MackenzieGT}. This shows
that $Ad(\phi^{t}(X))$ is the flow of the derivation $([X,\cdot],\rho(X))$ on $A$. Since $I$ is an ideal, the derivation
$[X,\cdot]$ restricts to a derivation on $I$, and therefore $I$ is invariant under $Ad(\phi^{t}(X))$. This proves that for
all $Y\in I_x$,
\[Ad(\phi^t(X,x))Y=\phi^t(X,x)\cdot Y\in I.\]
So $N(I)$ contains all the elements in $\mathcal{G}$ of the form $\phi^t(X,x)$. The set of such elements contains an
open neighborhood $O$ of the unit section in $\mathcal{G}$. Since the $s$-fibers of $\mathcal{G}$ are connected,
$O$ generates $\mathcal{G}$ (see Proposition 1.5.8 \cite{MackenzieGT}). Therefore, $N(I)=\mathcal{G}$ and so,
(\ref{EQ_Action}) defines an action of $\mathcal{G}$ on $I$.

Using that $\phi^{-t}(X,\varphi^t(X,x))=\phi^{t}(X,x)^{-1}$, equation (\ref{EQ_adjoint_diff}) gives
\begin{align*}
\frac{d}{dt}_{|t=0}\left(\phi^{t}(X,x)^{-1}\cdot Y_{\varphi^t(X,x)}\right)=
[X,Y]_x,\  \forall  \ X\in\Gamma(A), Y\in\Gamma(I),
\end{align*}
which, by (\ref{EQ_differentiating_reps}), implies that the action differentiates to the Lie bracket.
\end{proof}

\section{Cohomology}\label{Section_COHOMOLOGY}

In this section we recall some of the cohomology theories associated to Lie groupoids/algebroids and also some vanishing results.

\subsection{Differentiable cohomology of Lie groupoids}\label{subsection_differentiable_cohomology_of_Lie_groupoids}

\index{differentiable cohomology}The \textbf{differentiable cohomology} \cite{Haefliger} of a Lie groupoid
$\mathcal{G}$ with coefficients in a representation $V\to M$ is computed by the complex
of smooth maps $c:\mathcal{G}^{(p)}\to V$, where
\[\mathcal{G}^{(p)}:=\{(g_1,\ldots, g_p)\in \mathcal{G}^p | s(g_i)=t(g_{i+1}), i=1,\ldots, p-1\}\]
with $c(g_1,\ldots,g_p)\in V_{t(g_1)}$, and with differential given by
\begin{align*}
dc(g_1, \ldots ,g_{p+1}) &= g_1c(g_2, \ldots , g_{p+1})+\\
& + \sum_{i=1}^p(-1)^i c(g_1, \ldots , g_ig_{i+1}, \ldots , g_{p+1}) + (-1)^{p+1}c(g_1,\ldots,g_p).
\end{align*}
The resulting cohomology groups will be denoted $H^{\bullet}_{\mathrm{diff}}(\mathcal{G},V)$.

We recall a vanishing result for proper Lie groupoids.

\begin{proposition}[Proposition 1 \cite{Cra}]
For any proper Lie groupoid $\mathcal{G}$, and any representation $V$,
\[H^{k}_{\mathrm{diff}}(\mathcal{G},V)=0, \ \forall\ k\geq 1.\]\index{proper groupoid}
\end{proposition}

\subsection{Lie algebroid cohomology}\label{subsection_Lie_algebroid_cohomology}

Let $(A,[\cdot,\cdot],\rho)$ be a Lie algebroid over a manifold $M$ and let $(V,\nabla)$ be a representation of $A$. The
cohomology of $A$ with coefficient in $V$,
\[H^{\bullet}(A,V),\]
is the cohomology of the complex
\[\Omega^{\bullet}(A,V):=\Gamma(\Lambda^{\bullet}A^*\otimes V),\]
with differential given by the classical Koszul formula:\index{cohomology, Lie algebroid}\index{Koszul formula}
\begin{align*}
d_{\nabla}\alpha(X_0, \ldots , X_{p})=&\sum_{i}(-1)^{i} \nabla_{X_i}(\alpha(X_0, \ldots , \widehat{X}_i, \ldots , X_{p}))+\\
& + \sum_{i< j} (-1)^{i+j}\alpha([X_i, X_j], \ldots , \widehat{X}_i, \ldots, \widehat{X}_j, \ldots , X_{p}).
\end{align*}
For trivial coefficients, we denote the differential by $d_A$ and the cohomology groups by $H^{\bullet}(A)$.

\subsubsection*{Examples}

\begin{itemize}
\item
For $A=TM$, $H^{\bullet}(TM)=H^{\bullet}(M)$ is the de Rham cohomology.

\item For $A=\mathfrak{g}$, a Lie algebra, we obtain the Eilenberg Chevalley cohomology of $\mathfrak{g}$ with coefficients in $V$.

\item If $A=T^*M$ is the cotangent Lie algebroid of a Poisson manifold $(M,\pi)$, then $\Omega^{\bullet}(T^*M)=\mathfrak{X}^{\bullet}(M)$ and $d_{T^*M}=d_{\pi}$; thus we obtain the Poisson
cohomology\index{cohomology, Poisson}
$$H^{\bullet}(T^*M)= H^{\bullet}_{\pi}(M).$$

\end{itemize}

\subsubsection*{Invariant forms on the Lie groupoid}
Let $\mathcal{G}$ be a Lie groupoid with a representation $V$. We consider the complex of $s$-foliated
forms\index{s-foliated forms} on $\mathcal{G}$ with values in $s^*(V)$
\[(\Omega^{{\bullet}}(T^s{\mathcal{G}}, s^*(V)),d\otimes I_V).\]
This is the complex of forms on the (foliation) Lie algebroid $T^s{\mathcal{G}}\to \mathcal{G}$ with coefficients in
the representation $(s^*(V),\nabla)$; the connection is
\[\nabla_X(\lambda)_{|s^{-1}(x)}:=L_{X|s^{-1}(x)}(\lambda_{|s^{-1}(x)}),\]
where this is just the Lie derivative of the function $\lambda_{|s^{-1}(x)}:s^{-1}(x)\to V_{x}$.

A form $\omega\in\Omega^{\bullet}(T^s{\mathcal{G}}, s^*(V))$ is called \textbf{invariant}\index{invariant forms},
if it satisfies
\[(r_g^*\otimes g)(\omega_{hg})=\omega_h,\  \forall \ h,g\in {\mathcal{G}}, \textrm{ with }s(h)=t(g),\]
where $r_g^*\otimes g$ is the linear isomorphism $\eta\mapsto g\cdot\eta\circ dr_g$. Denote the space of $V$-valued invariant forms on
${\mathcal{G}}$ by $\Omega^{\bullet}(T^s{\mathcal{G}},s^*(V))^{\mathcal{G}}$.

There is a canonical isomorphism between $V$-values forms on $A$ and $V$-valued invariant forms on
$T^s\mathcal{G}$, given by
\begin{equation}\label{EQ_definition_of_J}
J:\Omega^{\bullet}(A,V)\rmap\Omega^{\bullet}(T^s{\mathcal{G}},s^*(V)), \ \ J(\eta)_g:=(r_{g^{-1}}^*\otimes g^{-1})(\eta_{t(g)}).
\end{equation}
A left inverse for $J$ (i.e.\ a map $P$ such that $P\circ J=I$) is
\begin{equation}\label{EQ_definition_of_P}
P:\Omega^{\bullet}(T^s{\mathcal{G}},s^*(V))\rmap \Omega^{\bullet}(A,V), \ \ P(\omega)_x:=\omega_{u(x)}.
\end{equation}

Moreover, it is well known that $J$ is a chain map (e.g.\ Theorem 1.2 \cite{WeinXu}). For completeness we include a
proof.

\begin{lemma}\label{Lemma_forms_on_A_invariant_forms}
The map $J$ is a chain map, therefore it is an isomorphism of complexes
\[J:(\Omega^{\bullet}(A,V),d_{\nabla})\cong (\Omega^{\bullet}(T^s{\mathcal{G}},s^*(V))^{\mathcal{G}},d\otimes I_V).\]
\end{lemma}

\begin{proof}
As before, for $X\in\Gamma(A)$ denote by $X^r$ its right invariant extension to $\mathcal{G}$, by $g\mapsto
\phi^t(X,g)$ the flow of $X^r$ and let $\varphi^t(X,x):=t(\phi^t(X,x))$. Using right invariance of $\phi^t(X,g)$ and
formula (\ref{EQ_differentiating_reps}), we obtain
\begin{align*}
J(\nabla_X(\lambda))_g&=g^{-1}(\nabla_X(\lambda))_{t(g)}=g^{-1}\frac{d}{d\epsilon}_{|\epsilon=0}\phi^{\epsilon}(X,t(g))^{-1}\lambda_{\varphi_{\epsilon}(X,t(g))}=\\
\nonumber&=\frac{d}{d\epsilon}_{|\epsilon=0}\phi^{\epsilon}(X,g)^{-1}\lambda_{\varphi^{\epsilon}(X,t(g))}=
\frac{d}{d\epsilon}_{|\epsilon=0}J(\lambda)_{\phi^{\epsilon}(X,g)}=L_{X^r}J(\lambda)_g,
\end{align*}
for all $X\in \Gamma(A)$ and $\lambda\in\Gamma(V)$. Clearly, $J$ satisfies
\begin{equation*}
J(\eta(X_1,\ldots ,X_q))=J(\eta)(X^r_1,\ldots , X^r_q),
\end{equation*}
for $X_1,\ldots, X_q\in \Gamma(A)$ and $\eta\in\Omega^q(A,E)$. Recall also that $[X^r,Y^r]=[X,Y]^r$, for $X,
Y\in\Gamma(A)$. Using these facts and that $d_{\nabla}$ and $d\otimes I_V$ are both expressed using the Koszul
formula, one obtains that the equation $J\circ d_{\nabla}(\eta)=d\otimes I_V\circ J(\eta)$ holds when evaluated on right
invariant vector fields. Since such vector fields span $T^s\mathcal{G}$, the conclusion follows.
\end{proof}

\subsection{The Van Est map}\label{Subsection_Van_Est}

Let $\mathcal{G}\rightrightarrows M$ be a Hausdorff Lie groupoid with Lie algebroid $A$, and let $V$ be a
representation of $\mathcal{G}$. Recall \cite{Cra} that the differentiable cohomology of $\mathcal{G}$ with
coefficients in $V$ and the Lie algebroid cohomology of $A$ with values in $V$ are related by the so-called \textbf{Van
Est map}\index{Van Est map}:
\[\Phi_V:H^{p}_{\mathrm{diff}}(\mathcal{G},V) \rmap  H^{p}(A,V).\]
For trivial coefficients, denote this map by $\Phi$.

The Van Est map satisfies:

\begin{theorem}[Theorem 4 and Corollary 2 in \cite{Cra}]
If the $s$-fibers of $\mathcal{G}$ are homologically $n$-connected, then $\Phi_V$ is an
isomorphism in degrees $p\leq n$, and is injective for $p=n+1$. Moreover, in this case, $[\omega]\in H^{n+1}(A)$ is in the image of $\Phi$ if and only if
\[\int_{\gamma}J(\omega)=0,\]
for all loops $\gamma:\mathbb{S}^{n+1}\to s^{-1}(x)$, for some $x\in M$, where $J(\omega)\in
\Omega^{n+1}(T^s\mathcal{G})^{\mathcal{G}}$ is given by (\ref{EQ_definition_of_J}).
\end{theorem}

\section{Symplectic groupoids}\label{Section_symplectic_groupoids}

\begin{definition}
A \textbf{symplectic groupoid}\index{symplectic groupoid} is Lie groupoid $\mathcal{G}\rightrightarrows M$ with a
symplectic structure $\omega$ for which the graph of the multiplication is Lagrangian
\[\{(g,h,gh)|s(g)=t(h)\}\subset (\mathcal{G},\omega)\times (\mathcal{G},\omega)\times (\mathcal{G},-\omega).\]
\end{definition}

\subsection{Some properties of symplectic groupoids}

We list below some remarkable properties of symplectic groupoids \cite{MackXu,DWC}.
\begin{itemize}
\item  The symplectic structure is \textbf{multiplicative}, i.e.\ it satisfies
\[m^*(\omega)=p_1^*(\omega)+p_2^*(\omega),\]
where $m:\mathcal{G}\times_{M} \mathcal{G}\to \mathcal{G}$ is the multiplication and $p_1,p_2:\mathcal{G}\times_{M} \mathcal{G}\to \mathcal{G}$ are the natural projections.
\item We have that $2\mathrm{dim}(M)=\mathrm{dim}(\mathcal{G})$.
\item The multiplicativity of $\omega$ and the dimension relation imply that the units $M\subset \mathcal{G}$ form a Lagrangian submanifold.
\item As for any Lie groupoid, we have a canonical decomposition of the tangent space along the units
\[T\mathcal{G}_{|M}=A\oplus TM.\]
By the dimension formula, $A$ and $T^*M$ have the same rank. Moreover, since $u^*(\omega)=0$, it follows that
$\omega^{\sharp}_{|M}:TM\to A^*$ is an isomorphism, and we use its dual $-\omega^{\sharp}_{|M}$ to identify
$A\cong T^*M$.
\item The anchor of $A\cong T^*M$, when viewed as a map
$T^*M\to TM$ is antisymmetric, thus it gives a bivector $\pi\in\mathfrak{X}^2(M)$. In fact, $\pi$ is Poisson and
the algebroid structure on $T^*M$ is that of the cotangent Lie algebroid\index{cotangent Lie algebroid} of
$(M,\pi)$.
\item The source map
\[s:(\mathcal{G},\omega)\rmap (M,\pi),\]
is a Poisson map and the target map is an anti-Poisson map.
\item The $s$-fibers and the $t$-fibers are symplectic orthogonal,
\[T^s\mathcal{G}^{\perp}=T^t\mathcal{G}, \ \ T^t\mathcal{G}^{\perp}=T^s\mathcal{G}.\]
\item The inversion map $\iota:\mathcal{G}\to \mathcal{G}$ satisfies $\iota^*(\omega)=-\omega$.
\end{itemize}

A Poisson manifold $(M,\pi)$ is called \textbf{integrable}\index{integrable Poisson manifold} if there is a symplectic
groupoid $(\mathcal{G},\omega)$ over $M$ that induces on $M$ the original Poisson structure. As remarked before, if
such a groupoid $\mathcal{G}$ exists, then $\mathcal{G}$ integrates the cotangent Lie algebroid $T^*M$. The
converse also holds:
\begin{theorem}[Theorem 5.2 \cite{MackXu}]\label{Theorem_Integration_Poisson}
Suppose that the cotangent Lie algebroid $T^*M$ of a Poisson manifold $(M,\pi)$ is integrable. Then the $s$-fiber
1-connected groupoid of $T^*M$ admits a natural symplectic structure that makes it into a symplectic groupoid and
that induces $\pi$ on $M$.
\end{theorem}

\subsection{Examples}\label{Subsection_examples_of_symplectic_groupoids}

\subsubsection*{Symplectic manifolds}

A symplectic manifold $(S,\omega_S)$ is integrated by the symplectic pair groupoid \[(S\times
\overline{S},\omega)\rightrightarrows S, \ \ \omega:=p_1^*(\omega_S)-p_2^{*}(\omega_S).\] The $s$-fiber 1-connected
groupoid of $(S,\omega_S)$ is the symplectic fundamental groupoid
\[(\Pi_1(S),\omega)\rightrightarrows S, \ \ \omega:=s^*(\omega_S)-t^{*}(\omega_S).\]

\subsubsection*{Trivial Poisson structure}

Recall that the cotangent bundle $p:T^*M\to M$ of a manifold $M$ carries the \textbf{canonical symplectic
structure}\index{canonical symplectic structure}
\[\omega_{\mathrm{can}}:=-d\alpha, \ \ \alpha_{\xi}:=p^*(\xi), \ \  \forall \  \xi \in T^*M,\]
where $\alpha\in \Omega^1(T^*M)$ is called the \textbf{tautological} 1-\textbf{form}\index{tautological 1-form}. This
makes $T^*M$ into a symplectic groupoid $(T^*M,\omega_{\mathrm{can}})\rightrightarrows M$, integrating the
zero Poisson structure on $M$.

\subsubsection*{Quotients of symplectic manifolds}

Let $(\Sigma,\omega)$\index{quotients of symplectic manifolds} be a symplectic manifold and let $G$ be a Lie group
acting by symplectomorphisms on $\Sigma$. The action is called \textbf{Hamiltonian}\index{Hamiltonian action}, if
there is a map $\mu:\Sigma\to \mathfrak{g}^*$, called the \textbf{moment map}\index{moment map}, such that
\[\iota_{\widetilde{X}}\omega=d\langle \mu,X\rangle, \ X\in\mathfrak{g},\]
where $\widetilde{X}$ is the infinitesimal action of $X$ on $\Sigma$. The following lemma is a consequence of results
in \cite{FeOrRa}, and it describes a symplectic groupoid integrating the quotient $(M,\pi)=(\Sigma,\omega)/G$.

\begin{lemma}\label{Lemma_symplectic_groupoid_reduction}
Let $(\Sigma,\omega)$ be a symplectic manifold endowed with a proper, free Hamiltonian action of a Lie group $G$ and
equivariant moment map $\mu:\Sigma\to \mathfrak{g}^*$. Then $(M,\pi)=(\Sigma,\omega)/G$ is integrable by the
symplectic groupoid
\[(\Sigma\times_{\mu}\Sigma)/G\rightrightarrows M,\]
with symplectic structure $\Omega$ induced by $(s^*(\omega)-t^*(\omega))_{|\Sigma\times_{\mu}\Sigma}$.
\end{lemma}
\begin{proof}
Consider the symplectic pair groupoid $(\Sigma\times\Sigma,\widetilde{\Omega})\rightrightarrows \Sigma$, with
$\widetilde{\Omega}:=s^*(\omega)-t^*(\omega)$. Then $G$ acts on $\Sigma\times\Sigma$ by symplectic groupoid
automorphisms, with equivariant moment map $J:=s^*\mu-t^*\mu$, which is also a groupoid cocycle. By Proposition
4.6 \cite{FeOrRa}, the Marsden-Weinstein reduction
\[\Sigma\times\Sigma//G=J^{-1}(0)/G\]
is a symplectic groupoid integrating the Poisson manifold $(M,\pi)=(\Sigma,\omega)/G$. In our case,
$J^{-1}(0)=\Sigma\times_{\mu}\Sigma$, and the symplectic form $\Omega$ is obtained as follows: the 2-form
$\widetilde{\Omega}_{|\Sigma\times_{\mu}\Sigma}$ is $G$-invariant and its kernel is spanned by the vectors coming
from the infinitesimal action of $\mathfrak{g}$. Therefore, it is the pullback of a symplectic form $\Omega$ on
$(\Sigma\times_{\mu}\Sigma)/G$.
\end{proof}

\subsubsection*{Linear Poisson structures}


Consider the linear Poisson structure\index{linear Poisson structure} $(\mathfrak{g}^*,\pi_{\mathfrak{g}})$
corresponding to a Lie algebra $\mathfrak{g}$. Let $G$ be a Lie group integrating $\mathfrak{g}$, and consider the
(left) coadjoint action of $G$ on $\mathfrak{g}^*$, $(g,\xi)\mapsto Ad_{g^{-1}}^*(\xi)$\index{coadjoint action}. A
symplectic Lie groupoid integrating $\pi_{\mathfrak{g}}$ is the corresponding action groupoid\index{action
groupoid}
\[(G\ltimes \mathfrak{g}^*,\omega_{\mathfrak{g}})\rightrightarrows (\mathfrak{g}^*,\pi_{\mathfrak{g}}),\]
with $\omega_{\mathfrak{g}}=d\widetilde{\theta}_{\mathrm{MC}}$ and
\[\widetilde{\theta}_{\mathrm{MC}}(X,\eta):=\langle\xi,dl_{g^{-1}}(X)\rangle, \ \ X\in T_{g}G, \ \eta\in T_{\xi}\mathfrak{g}^*.\]
Of course, $\widetilde{\theta}_{\mathrm{MC}}$ comes from the right invariant Maurer-Cartan form on $G$,
\[\theta_{\mathrm{MC}}\in\Omega^1(G,\mathfrak{g}), \ \ \theta_{\mathrm{MC},g}(X):=dl_{g^{-1}}(X).\]
Under the diffeomorphism $G\times\mathfrak{g}^*\diffto T^*G$, $(g,\xi)\mapsto l_{g^{-1}}^*(\xi)$, the 1-form
$\widetilde{\theta}_{\mathrm{MC}}$ becomes the tautological 1-form on $T^*G$, thus $\omega_{\mathfrak{g}}$
becomes the negative of the canonical symplectic structure on $T^*G$.

We compute $\omega_{\mathfrak{g}}$ explicitly. Let $X,Y\in\mathfrak{g}$ and denote by $X^l, Y^l$ their left
invariant extensions, and let $\alpha,\beta\in\mathfrak{g}^*$ which we regard as constant vector fields on
$\mathfrak{g}^*$. Then we have that
\begin{align*}
d\widetilde{\theta}_{\mathrm{MC}}(X^l+\alpha,Y^l+\beta)=&L_{X^l+\alpha}(\xi\mapsto \langle\xi,Y\rangle)-L_{Y^l+\beta}(\xi\mapsto \langle\xi,X\rangle)-\\
&-\widetilde{\theta}_{\mathrm{MC}}([X^l+\alpha,Y^l+\beta]).
\end{align*}
By our convention, $[X^l,Y^l]=-[X,Y]^l$, and clearly $\alpha$ and $\beta$ commute with each other and with $X^l$
and $Y^l$. Hence, the symplectic form is given by
\[\omega_{\mathfrak{g}}(U+\alpha, V+\beta)_{(g,\xi)}=\xi([dl_{g^{-1}}(U),dl_{g^{-1}}(V)])+\alpha(dl_{g^{-1}}(V))-\beta(dl_{g^{-1}}(U)).\]

We prove now that the source map is Poisson. The Hamiltonian vector field of $X\in \mathfrak{g}$, regarded as a
linear function on $(\mathfrak{g}^*,\pi_{\mathfrak{g}})$, is (see section \ref{SPreliminaries})
\[H_{X,\xi}:=-ad_X^*(\xi)\in \mathfrak{g}^*\cong T_{\xi}\mathfrak{g}^*.\]
We claim that the Hamiltonian vector field of $s^*(X)$ on $G\times \mathfrak{g}^*$ is
\[H_{s^*(X),(g,\xi)}=X^l_g-ad_X^*(\xi).\]
This follows by using the formula for $\omega_{\mathfrak{g}}$
\begin{align*}
\omega_{\mathfrak{g}}(X^l_g-ad_X^*(\xi), V&+\beta)_{(g,\xi)}=\xi([X,dl_{g^{-1}}(V)])-\langle ad_{X}^*(\xi),dl_{g^{-1}}(V)\rangle+\\
&+\beta(X)=\beta(X)=ds^*(X)(V+\beta).
\end{align*}
Since $ds(H_{s^*(X)})=H_X$, it follows that $s$ is Poisson.

Similarly, $t$ is anti-Poisson. Observe that the diagonal action of $G$ on the product
$(G\times\mathfrak{g}^*,-d\widetilde{\theta}_{\mathrm{MC}})$ is symplectic and its orbits are precisely the fibers
of $t$. Therefore, $t$ induces a Poisson diffeomorphism between
\begin{equation}\label{EQ_local_model_fixed_point_reduction}
(G\times\mathfrak{g^*},-d\widetilde{\theta}_{\mathrm{MC}})/G \cong (\mathfrak{g}^*,\pi_{\mathfrak{g}}).
\end{equation}

\begin{remark}\rm
Observe that the complex of $s$-foliated forms on $G\ltimes\mathfrak{g}^*$ coincides with the complex
$(\Omega_{\mathfrak{g}^*}^{\bullet}(G), d_G)$ from section \ref{Section_Conn}, of smooth maps
$\mathfrak{g}^*\to\Omega^{\bullet}(G)$ with the pointwise de Rham differential. Moreover, the invariant forms on
$G\ltimes\mathfrak{g}^*$ form the invariant part of $\Omega_{\mathfrak{g}^*}^{\bullet}(G)$, therefore Lemma
\ref{lemma_commutes_with_diff_fixed_point} is a special case of Lemma \ref{Lemma_forms_on_A_invariant_forms}.
\end{remark}

\section{Integrability}\label{Section_Integrability}

In this section we present some results related to integrability of Lie algebroids, and in particular of Poisson manifolds.
We start by explaining the necessary and sufficient conditions for integrability from \cite{CrFe1}, and also some of the
techniques involved in the proofs.

\subsection{Criterion for integrability of Lie algebroids}

Consider an integrable Lie algebroid $A$,\index{integrable Lie algebroid} and let $\mathcal{G}$ be the $s$-fiber
1-connected groupoid of $A$. The $s$-fiber at $x\in M$ $\mathcal{G}(x,-)$ is a principal $\mathcal{G}_x$-bundle over
the orbit $\mathcal{G}x$, which coincides with $Ax$, since $\mathcal{G}(x,-)$ is connected. From the long exact
sequence of homotopy groups associated to this fibration:
\[\ldots \to \pi_2(\mathcal{G}_x,1_x)\to \pi_2(\mathcal{G}(x,-),1_x)\to \pi_2(Ax,x)\stackrel{\tilde{\partial}}{\to}\pi_1(\mathcal{G}_x,1_x)\to \{1\},\]
we obtain a group homomorphism $\tilde{\partial}:\pi_2(Ax,x)\to \pi_1(\mathcal{G}_x,1_x)$. On the other hand,
$\pi_1(\mathcal{G}_x,1_x)$ sits naturally as a discrete subgroup of the center of $G(\mathfrak{g}_x)$, the 1-connected
group of the isotropy Lie algebra\index{isotropy Lie algebra} $\mathfrak{g}_x$. Therefore, we can regard
$\tilde{\partial}$ as a map with values in the center of $G(\mathfrak{g}_x)$
\begin{equation*}
\widetilde{\partial}:\pi_2(Ax,x)\rmap Z(G(\mathfrak{g}_x)).
\end{equation*}

The crucial remark in \cite{CrFe1} is that also for a possibly non-integrable Lie algebroid $A$, there is a group
homomorphism
\[\partial:\pi_2(Ax,x)\rmap Z(G(\mathfrak{g}_x)),\]
called the \textbf{monodromy map}\index{monodromy map}, which coincides with $\widetilde{\partial}$ in the
integrable case. Denote the image of $\partial$ by
\[\widetilde{\mathcal{N}}_x(A)\subset Z(G(\mathfrak{g}_x)).\]
Discreteness of this group is equivalent to discreteness of the so-called \textbf{monodromy group}\index{monodromy
group}, defined by
\[\mathcal{N}_x(A):=\{X\in Z(\mathfrak{g}_x) | \exp(X)\in \widetilde{\mathcal{N}}_x(A)\}\subset Z(\mathfrak{g}_x),\]
where $Z(\mathfrak{g}_x)$ denotes the center of $\mathfrak{g}_x$. Note that the terminology we adopt here differs
form that in \cite{CrFe1}, where $\widetilde{\mathcal{N}}_x(A)$ is called the monodromy group and
$\mathcal{N}_x(A)$ is called the reduced monodromy group.

Denote by
\[\mathcal{N}(A):=\bigcup_{x\in M}\mathcal{N}_x(A)\subset A.\]
We have the following criterion for integrability of Lie algebroids.
\begin{theorem}[Theorem 4.1 \cite{CrFe1}]\label{Criterion of Marius and Rui}
A Lie algebroid $A$ is integrable if and only if there exists an open $U\subset A$ around the zero section $M$, such
that
\[\mathcal{N}(A)\cap U=M.\]
\end{theorem}

\subsection{The space of $A$-paths}

In this subsection, we recall some results from \cite{CrFe1}.

The condition that $a:I=[0,1]\to A$ is an $A$-path is equivalent to\index{A-path}
\[a(t)dt:TI\rmap A\]
being a Lie algebroid homomorphism.

An $A$-\textbf{homotopy}\index{A-homotopy, A-homotopic} is a family of $A$-paths $a_{\epsilon}$, for $\epsilon\in
I$, for which there exists a Lie algebroid map of the form
\[adt+bd\epsilon:TI\times TI \rmap A,\]
where $b=b(\epsilon,t):I\times I\to A$ satisfies $b(\epsilon,0)=0$ and $b(\epsilon,1)=0$.

Two $A$-paths $a_0$ and $a_1$ are called $A$-\textbf{homotopic}, if there is an $A$-homotopy $a_{\epsilon}$
connecting them. This relation is an equivalence relation $\sim$ on the space of $A$-paths. The quotient space
\[\mathcal{G}(A):=\{A-\textrm{paths}\}/\sim\]
carries a groupoid structure, with source/target map the initial/end point
\[s([a])=p(a(0))\in M, \ t([a])=p(a(1))\in M,\]
with units the constant 0-paths and composition essentially given by concatenation. In fact,
$\mathcal{G}(A)\rightrightarrows M$ is a topological groupoid with 1-connected $s$-fibers, and it is called the
\textbf{Weinstein groupoid}\index{Weinstein groupoid} of $A$. In the integrable case, it carries a natural smooth
structure and therefore it is the fundamental integration of $A$.

To describe the topology, consider the space of $C^1$-paths
\[\widetilde{P}(A):=C^1(I,A).\]
The space of $C^1$-$A$-paths, denoted by $P(A)\subset \widetilde{P}(A)$, is a smooth Banach submanifold, and the
equivalence relation $\sim$ of $A$-homotopy is given by the leaves of a finite codimension foliation $\mathcal{F}(A)$
on $P(A)$ (for Banach manifolds and foliations on Banach manifolds see \cite{Lang}). The topology of
$\mathcal{G}(A)$ is the quotient topology. There exists at most one smooth structure on $\mathcal{G}(A)$ for which
the projection map $P(A)\to \mathcal{G}(A)$ is a submersion. This exists precisely when $A$ is integrable!

\section{Symplectic realizations from transversals in the manifold of cotangent paths}
\label{Constructing symplectic realizations from transversals in the manifold of cotangent paths}

In this section we specialize to the case when $A=T^*M$\index{cotangent path} is the cotangent algebroid of a
Poisson manifold $(M,\pi)$. Using the space of cotangent paths, we describe a general method for constructing
symplectic realizations. In particular, we describe the symplectic structure on the Weinstein groupoid
\[\mathcal{G}(M,\pi):=\mathcal{G}(T^*M),\]
whenever this is smooth.

We will use the same notations as in \cite{CrFe1,CrFe2}. We consider:
\begin{itemize}
\item $\widetilde{\mathcal{X}}= \widetilde{P}(T^*M)$ is the space of all $C^1$-paths in $T^*M$. Recall \cite{CrFe1} that
$\widetilde{\mathcal{X}}$ has a natural structure of Banach
manifold.
\item $\mathcal{X}= P(T^*M)$ is the space of all cotangent paths which, by Lemma 4.6 in \cite{CrFe1}, is a Banach submanifold of $\widetilde{\mathcal{X}}$.
\item $\mathcal{F}= \mathcal{F}(T^*M)$ is the foliation on $\mathcal{X}$ given by the equivalence relation of cotangent homotopy\index{cotangent homotopy} (i.e.\ $T^*M$-homotopy)
it is a smooth foliation on $\mathcal{X}$ of finite codimension. In (\ref{the-fol}) we will recall the description of
$\mathcal{F}$ via its involutive distribution.
\end{itemize}
Thinking of $\widetilde{\mathcal{X}}$ as the cotangent space of the space $P(M)$ of paths in $M$, it comes with a canonical symplectic structure
$\widetilde{\Omega}$. To avoid issues regarding symplectic structures on Banach manifolds let us just define $\widetilde{\Omega}$ explicitly:
\[\widetilde{\Omega}(X_a,Y_a)=\int_0^1 \omega_{\mathrm{can}}(X_a,Y_a)_{a(t)}dt,\textrm{ for } a\in \widetilde{\mathcal{X}}, X,Y\in T_a\widetilde{\mathcal{X}},\]
where $X_a, Y_a$ are interpreted as paths in $T(T^*M)$ sitting above $a$ and where $\omega_{\mathrm{can}}$ is the canonical symplectic form on
$T^*M$. It can be checked directly that $\widetilde{\Omega}$ is closed; we only need its restriction to $\mathcal{X}$:
\[ \Omega:= \widetilde{\Omega}|_{\mathcal{X}}\in \Omega^2(\mathcal{X}).\]
We will prove that the kernel of $\Omega$ is precisely $T\mathcal{F}$ and that $\Omega$ is invariant under the
holonomy of $\mathcal{F}$; these properties ensures that $\Omega$ descends to a symplectic form on the leaf space
$\mathcal{G}(M, \pi)$ (whenever smooth).

We will use the contravariant geometry developed in subsection \ref{SubSection_Contravariant}:
\begin{itemize}
\item let $\nabla$ be a torsion-free connection on $M$;
\item consider the two contravariant connections\index{contravariant connection} defined by (\ref{EQ_two_conn}) on $T^*M$ and $TM$, both denoted by $\overline{\nabla}$;
\item we decompose a tangent vector $X$ to $T_x^*M$ as a pair $(\overline{X}, \theta_X)\in T_xM\oplus T^*_xM$ with respect to $\nabla$;
\end{itemize}

Similarly, a tangent vector $X\in T_a\widetilde{\mathcal{X}}$ is represented by a pair $(\overline{X}, \theta_X)$,
where $\overline{X}$ is a $C^1$-path in $TM$ and $\theta_X$ a $C^1$-path in $T^*M$, both sitting above the base
path $\gamma= p\circ a$. The tangent space
\[ T_a\mathcal{X}\subset T_a\widetilde{\mathcal{X}}\]
corresponds to those pairs $(\overline{X},\theta_X)$ satisfying (see \cite{CrFe1}):
\begin{equation}\label{EQ_tangent_to_the_cotangent}
\overline{\nabla}_a(\overline{X})=\pi^{\sharp}(\theta_X).
\end{equation}
Note that this equation forces $\overline{X}$ to be of class $C^2$.

To describe the distribution $T\mathcal{F}$, let $a\in \mathcal{X}$ with base path $\gamma$ and let
\[\mathcal{E}_{\gamma}:=\{\beta\in C^2(I,T^*M) | p\circ\beta=\gamma \}.\]
By (\ref{EQ_conn_pisharp_rel}), each $\beta\in\mathcal{E}_{\gamma}$ induces a vector in $T_a\mathcal{X}$, with components
\[ X_{\beta}:= (\pi^{\sharp}(\beta),\overline{\nabla}_a(\beta))\in T_a\mathcal{X}.\]
With these, the foliation $\mathcal{F}$ has tangent bundle (see \cite{CrFe1})
\begin{equation}\label{the-fol}
T\mathcal{F}_{a}= \{ X_{\beta} | \beta\in \mathcal{E}_{\gamma}, \beta(0)=0, \beta(1)=0\}.
\end{equation}

Next, we give a very useful formula for $\Omega$.
\begin{lemma}\label{Omega_Formula_simpla}
Let $a\in \mathcal{X}$ with base path $\gamma$. For $X=(\overline{X},\theta_X)$, $Y=(\overline{Y},\theta_Y)\in T_a\mathcal{X}$ choose
$\beta_X,\beta_Y\in \mathcal{E}_{\gamma}$ such that $\theta_X=\overline{\nabla}_a(\beta_X)$ and $\theta_Y=\overline{\nabla}_a(\beta_Y)$. Then
\[\Omega(X,Y)=\langle \beta_Y,\overline{X}\rangle|_0^1-\langle \beta_X,\overline{Y}\rangle|_0^1-\pi(\beta_X,\beta_Y)|_0^1.\]
In particular, for $Y=X_{\beta}$, with $\beta\in\mathcal{E}_{\gamma}$, we have that $\Omega(X,X_\beta)=\langle\beta,\overline{X}\rangle|_0^1$.
\end{lemma}
\begin{proof}
The proof is the same as that of Lemma \ref{Lemma_simple_frmla}, which, besides the properties of the connections, uses only the fact that the components of the paths $v$ and $w$ satisfy the equation from Lemma \ref{lemma_cotangent_tangent}, which is exactly equation (\ref{EQ_tangent_to_the_cotangent}) for $X$ and $Y$.
\end{proof}

\begin{corollary}\label{CoroKerOmega}
Let $a\in \mathcal{X}$ with base path $\gamma$. Then
\[\ker(\Omega_{a})=T\mathcal{F}_a=\{X_\beta: \beta\in \mathcal{E}_{\gamma}, \beta(0)=0, \beta(1)=0\}.\]
\end{corollary}
\begin{proof}
Let $X=(\overline{X},\theta_X)\in \ker(\Omega_{a})$. Then, for all $\xi\in\mathcal{E}_{\gamma}$ we have that
$\Omega_a(X,X_{\xi})=0$, hence, by the previous lemma, $\overline{X}(0)=0$ and $\overline{X}(1)=0$. Let $\beta\in
\mathcal{E}_{\gamma}$ be the unique solution to the equation $\theta_X=\overline{\nabla}_a(\beta)$ with
$\beta(0)=0$. Observe that by (\ref{EQ_conn_pisharp_rel}), both $\overline{X}$ and $\pi^{\sharp}(\beta)$ satisfy the
equation
\[\overline{\nabla}_a(Z)=\pi^{\sharp}(\theta_X),\quad Z(0)=0.\]
Therefore they must be equal, and so $X=X_\beta$. Again by the lemma, for all $Y=(\overline{Y},\theta_Y)\in
T_a\mathcal{X}$, we have that $\langle\overline{Y}(1),\beta(1)\rangle=0$. On the other hand, $\overline{Y}(1)$ can
take any value (see the lemma below), thus $\beta(1)=0$ and this shows that $X\in T\mathcal{F}_a$. The other
inclusion follows directly from Lemma \ref{Omega_Formula_simpla}.
\end{proof}

Consider now the maps $\tilde{s}, \tilde{t}: \mathcal{X}\to M$ which assign to a path $a$ the starting (respectively
ending) point of its base path $\gamma$.

\begin{lemma}\label{LemmaOrto} $\tilde{s}$ and $\tilde{t}$ are submersions and their fibers
are orthogonal with respect to $\Omega$:
\[(\ker d\tilde{s}_a)^{\perp}=\ker d\tilde{t}_a,\quad (\ker d\tilde{t}_a)^{\perp}=\ker d\tilde{s}_a.\]
\end{lemma}

\begin{proof}
For $V_0\in T_{\gamma(0)}M$, let $\overline{V}$ be the solution (above $\gamma$) to
\[\overline{\nabla}_a(\overline{V})=0,\ \ \overline{V}(0)=V_0.\]
Then $(\overline{V},0)\in T_a\mathcal{X}$ and $d\tilde{s}_a(\overline{V},0)=V_0$. Thus $\tilde{s}$ is a submersion,
and similarly, one shows that $\tilde{t}$ is a submersion. For the second part, note that
\[\ker d\tilde{s}_a=\{(\overline{X},\theta_X) | \overline{X}(0)=0\}, \  \ \ \ker d\tilde{t}_a=\{(\overline{X},\theta_X) | \overline{X}(1)=0\}.\]
For $X=(\overline{X},\theta_X)\in \ker d\tilde{t}_a$, $Y=(\overline{Y},\theta_Y)\in \ker d\tilde{s}_a$, consider
$\beta_X,\beta_Y\in \mathcal{E}_{\gamma}$ the solutions to
\[\overline{\nabla}_a(\beta_X)=\theta_X,\  \beta_X(1)=0,\ \ \  \overline{\nabla}_a(\beta_Y)=\theta_Y,\  \beta_Y(0)=0.\]
Lemma \ref{Omega_Formula_simpla} implies that $\Omega_a(X,Y)=0$. Conversely, let $X=(\overline{X},\theta_X)\in
(\ker d\tilde{s}_a)^{\perp}$. For $\xi\in \mathcal{E}_{\gamma}$, with $\xi(0)=0$ we have that $X_{\xi}\in \ker
d\tilde{s}_a$, therefore, by assumption, $\Omega_a(X,X_{\xi})=0$. Thus, by Lemma \ref{Omega_Formula_simpla} we
have that
\[0=\Omega_a(X,X_{\xi})=\langle \xi(1),\overline{X}(1)\rangle.\]
But $\xi(1)$ is arbitrary, hence $\overline{X}(1)=0$, i.e. $X\in \ker d\tilde{t}_a$. So $(\ker d\tilde{s}_a)^{\perp}=\ker
d\tilde{t}_a$. The other equality is proven similarly.
\end{proof}

We collect the main properties of $\Omega$.

\begin{proposition}\label{PropTrnasversal}
Let $\mathcal{T}$ be a transversal to $\mathcal{F}$. Then the
following hold:
\begin{enumerate}
\item[(a)]
$\Omega_{|\mathcal{T}}$ is a symplectic form on $\mathcal{T}$, and is invariant under the action of the holonomy
of $\mathcal{F}$ on $\mathcal{T}$
\item[(b)] the sets $U_s=\tilde{s}(\mathcal{T})$ and $U_t=\tilde{t}(\mathcal{T})$ are open in
$M$, and
\begin{eqnarray*}
&&\sigma:=\tilde{s}_{|\mathcal{T}}:(\mathcal{T},\Omega_{|\mathcal{T}})\to
(U_s,\pi_{|U_s})\textrm{
is a Poisson map and}\\
&&\tau:=\tilde{t}_{|\mathcal{T}}:(\mathcal{T},\Omega_{|\mathcal{T}})\to
(U_t,\pi_{|U_t})\textrm{ is anti-Poisson}
\end{eqnarray*}
\item[(c)] $\ker(\sigma)^{\perp}=\ker(\tau)$ and
$\ker(\tau)^{\perp}=\ker(\sigma)$.
\end{enumerate}
\end{proposition}

\begin{proof}
Since $\mathcal{F}$ is of finite codimension \cite{CrFe1}, there are no issues regarding the meaning of symplectic
forms on our $\mathcal{T}$, and $\Omega|_{\mathcal{T}}$ is clearly symplectic. Actually, the entire (a) is a standard
fact about kernels of closed 2-forms, at least in the finite dimensions; it applies to our situation as well: from the
construction of holonomy by patching together foliation charts, the second part is a local issue: given a product
$B\times \mathcal{T}$ of a ball $B$ in a Banach space and a finite dimensional manifold $\mathcal{T}$ (for us a small
ball in a Euclidean space) and a closed two-form $\Omega$ on $B\times \mathcal{T}$, if
\[ \ker(\Omega_{x, y})= T_{x}B\times\{0_y\} \subset T_xB\times T_{y}\mathcal{T}, \ \ \forall \ (x, y)\in B\times \mathcal{T},\]
then $\Omega_x= \Omega|_{\{x\}\times \mathcal{T}}\in \Omega^{2}(\mathcal{T})$ does not depend on $x\in B$ (since $\Omega$ is closed).

We prove part (b) for $\sigma$, for $\tau$ it follows similarly. Since
\[d\tilde{s}_a:T_a\mathcal{X}=T_a\mathcal{T}\oplus T_a\mathcal{F}\to T_{\gamma(0)}M\]
is surjective and $T_a\mathcal{F}\subset \ker{d\tilde{s}_a}$, it follows that $\tilde{s}_{|\mathcal{T}}$ is a
submersion onto the open $\tilde{s}(\mathcal{T})=U_s$. To show that $\sigma$ is Poisson, we first describe the
Hamiltonian vector field of $\sigma^*(f)$, for $f\in C^{\infty}(U_s)$. Consider the vector field on $\mathcal{X}$:
\[\widetilde{H}_{f,a}:=X_{(1-t)df_{\gamma(t)}}=(\overline{\nabla}_a((1-t)df_{\gamma(t)}),(1-t)\pi^{\sharp}(df_{\gamma(t)}))\in T_a\mathcal{X}.\]
Clearly $d\tilde{s}(\widetilde{H}_{f})=\pi^{\sharp}(df)$, and by Lemma \ref{Omega_Formula_simpla} $\widetilde{H}_{f}$ also satisfies
\[\Omega(\widetilde{H}_{f},Y)_a=\langle df_{\gamma(0)},\overline{Y}(0)\rangle=d(\widetilde{s}^*(f))(Y),\ \  Y\in T_a\mathcal{X}.\]
Thus $\Omega(\widetilde{H}_{f},\cdot)=d(\widetilde{s}^*(f))$. Decomposing
$\widetilde{H}_{f|\mathcal{T}}:=H_f+V_f$, where $H_f$ is tangent to $\mathcal{T}$ and $V_f$ is tangent to
$\mathcal{F}$ and using that $T\mathcal{F}=\ker\Omega$, we obtain
\[\Omega_{|\mathcal{T}}(H_f,\cdot)=d(\sigma^*(f)), \ \ d\sigma(H_f)=\pi^{\sharp}(df).\]
Thus $\sigma$ is Poisson. Part (c) follows from Lemma \ref{LemmaOrto} and Corollary \ref{CoroKerOmega}.
\end{proof}

\begin{remark}\label{remark_symplectic_realizations_is_a_consequence}\rm
In fact, Theorem 0 is a consequence of the results of this section. Namely, if $\mathcal{V}_{\pi}$ is a contravariant
spray with flow $\varphi_t$, then, on an open $\mathcal{U}$ of the zero section, the map
\[\phi_{\mathcal{V}_{\pi}}:\mathcal{U}\rmap \mathcal{X}, \ \ \xi\mapsto a(t):=\varphi_t(\xi)\]
parameterizes a transversal $\mathcal{T}$ to $\mathcal{F}$. Pulling back $\Omega$ by $\phi_{\mathcal{V}_{\pi}}$,
one obtains the formula from Theorem 0 for the symplectic structure $\omega$ on $\mathcal{U}$.
\end{remark}

\clearpage \pagestyle{plain}

\chapter{Reeb stability for symplectic foliations}\label{ChReeb}
\pagestyle{fancy}
\fancyhead[CE]{Chapter \ref{ChReeb}} 
\fancyhead[CO]{Reeb stability for symplectic foliations} 

In this chapter we prove Theorem \ref{Theorem_ONE}, a normal form result for symplectic foliations around finite type
symplectic leaves.

\section{Introduction}

A \textbf{symplectic foliation}\index{symplectic foliation} on a manifold $M$ is a (regular) foliation $\mathcal{F}$,
endowed with a 2-form $\omega$ on $T\mathcal{F}$, such that $\omega$ restricts to a symplectic form on every leaf
$S$ of $\mathcal{F}$
\[\omega_{|S}\in\Omega^2(S).\]
Equivalently, a symplectic foliation $(M,\mathcal{F},\omega)$ is a \textbf{regular}\index{regular Poisson structure}
Poisson structure on $M$, i.e.\ a Poisson structure $\pi$ of constant rank. The symplectic leaves of $\pi$ correspond to
the leaves of
$\mathcal{F}$ endowed with the restriction of $\omega$.\\

In this chapter we prove Theorem \ref{Theorem_ONE}, a normal form result for symplectic foliations, analogous to the
Local Reeb Stability Theorem (see e.g.\ \cite{MM}). A similar result  (Theorem \ref{Theorem_TWO}) for possibly
non-regular Poisson manifolds, and which generalizes Conn's theorem, will be presented in chapter
\ref{ChNormalForms}. Compared to Theorem \ref{Theorem_TWO}, in the regular case, the proof is substantially
easier and also the hypothesis is weaker; in particular, the compactness assumption on the leaf can be replaced with
the requirement that it is a finite type manifold\index{finite type manifold} (i.e.\ a manifold admitting a Morse function with a finite number of critical points).\\

The local model of a symplectic foliation $(M,\mathcal{F},\omega)$ around a leaf $(S,\omega_S)$ is a refinement of the
one from the Local Reeb Stability Theorem: the foliation is ``linearized'' (as in Reeb's theorem) and the 2-form
$\omega$ is made ``affine'', which, loosely speaking, means that one linearizes $\omega-\omega_S$.

Accordingly, the proof has two steps. In the first step, we apply a version of Reeb's theorem for non-compact leaves,
whose proof we include in the appendix (it is a straightforward adaptation of the proof from \cite{MM} of Reeb's
theorem, yet we couldn't find it elsewhere in the literature). In the second step, we take care of the 2-form: first, we
show that its variation can be linearized cohomologically, and then, using a leaf-wise Moser deformation argument, we
bring the leafwise symplectic structures to the normal form.

\section{Local Reeb Stability for non-compact leaves}\label{section_local_reeb_stability_for_non-compact_leaves}
We start by recalling the local model for foliations that appears in the Local Reeb Stability Theorem. We give two
descriptions of it, the first using the linear holonomy and the second in terms of the Bott connection on the normal
bundle to the leaf. For the theory of foliations, including the basic properties of holonomy, we recommend \cite{MM}.

\subsubsection*{The holonomy cover}

Let $(M,\mathcal{F})$ be a foliated manifold and let $S\subset M$ be an embedded leaf. Denote the normal bundle to
the $S$ by
\[\nu_S:=TM_{|S}/TS.\]

The \textbf{holonomy group}\index{holonomy group} at $x\in S$ is the isotropy group at $x$ of the holonomy
groupoid\index{holonomy groupoid} of the foliation (see \cite{MM} and subsection \ref{Examples Lie groupoids}), and
is denoted by:
\[H:=\textrm{Hol}(\mathcal{F})_x.\]
Then $H$ is the quotient of $\pi_1(S,x)$ by the normal subgroup $K$, of classes of paths with trivial
holonomy\index{holonomy}:
\[1\rmap K\rmap \pi_1(S,x)\rmap H\rmap 1.\]
The \textbf{holonomy cover}\index{holonomy cover} of $S$ is the covering space $\widetilde{S}\to S$ corresponding
to the subgroup $K$; equivalently $\widetilde{S}$ is the $s$-fiber over $x$ of $\textrm{Hol}(\mathcal{F})$. The
holonomy cover $\widetilde{S}$ has the structure of a principal $H$-bundle over $S$.

The holonomy group $H$ acts on every transversal to the foliation at $x$ by germs of diffeomorphisms, and
differentiating this action gives a linear representation of $H$ on the normal space at $x$
\[V:=\nu_{S,x}.\]
This is called the \textbf{linear holonomy action}\index{linear holonomy} of $H$. The \textbf{linear holonomy
group}\index{linear holonomy group}, denoted by $H_{\mathrm{lin}}$, is the quotient of $H$ that acts effectively on
$V$, i.e.\ it is defined as
\[H_{\mathrm{lin}}:=H/K_{\mathrm{lin}},\]
where $K_{\mathrm{lin}}:=\ker(H\to \mathrm{Gl}(V))$. The \textbf{linear holonomy cover}\index{linear holonomy
cover} of $S$, denoted by $\widetilde{S}_{\textrm{lin}}\to S$, is the covering space corresponding to
$K_{\mathrm{lin}}$, i.e.\
\[\widetilde{S}_{\textrm{lin}}:=\widetilde{S}/K_{\mathrm{lin}}.\]

The local model\index{local model, foliation} for the foliation around $S$ is the foliated manifold
\[(\widetilde{S}\times_{H}V,\mathcal{F}_{N}),\]
where the leaves of the foliation $\mathcal{F}_N$ are
\[(\widetilde{S}\times Hv)/H,\ \ \textrm{ for } \ v\in V.\]
Clearly, it can be described also using the linear holonomy cover
\[(\widetilde{S}\times_{H}V,\mathcal{F}_{N})\cong (\widetilde{S}_{\mathrm{lin}}\times_{H_{\mathrm{lin}}}V,\mathcal{F}_{N}).\]

\subsubsection*{The Bott connection}\index{Bott connection}

The foliation induces on $\nu_S$ a flat linear connection, called the \textbf{Bott connection}; defined as follows:
\[\nabla_{X}(V):=[\widetilde{X},\widetilde{V}]_{|S} \mod \ TS , \ \ X\in \mathfrak{X}(S), \ V\in \Gamma(\nu_S),\]
where $\widetilde{X}$ and $\widetilde{V}$ are vector fields on some open around $S$, such that $\widetilde{X}$ is
tangent to the foliation and extends $X$, and the restriction of $\widetilde{V}$ to $S$ is a representative of $V$. The
local model for the foliation\index{local model, foliation} around $S$, also called the \textbf{linearization of the
foliation}\index{linearization, foliation} at $S$, is the foliated manifold
\[(\nu_S,\mathcal{F}_{\nabla}),\]
where $\mathcal{F}_{\nabla}$ is the foliation tangent to the horizontal bundle of $\nabla$.

The linear holonomy action on $V$ can be given in terms of the parallel transport induced by $\nabla$, namely
\[[\gamma]\cdot v= T_{\nabla}(\gamma)v, \]
for $v\in V$ and $\gamma$ a closed loop at $x$. This implies that parallel transport also induces an isomorphism
between the two descriptions of the local model:
\begin{equation}\label{transport}
T:(\widetilde{S}_{\textrm{lin}}\times_{H_{\textrm{lin}}}\nu_{S,x},\mathcal{F}_N)\rmap(\nu_S,\mathcal{F}_{\nabla}),\  [\gamma,v]\mapsto T_{\nabla}(\gamma)v.
\end{equation}

\subsubsection*{Local Reeb Stability}

Similar to the discussion in \cite{CrStr} on linearization of proper Lie groupoids around orbits, a local normal form
theorem for foliations might have three different conclusions:\index{Reeb stability}
\begin{enumerate}[1)]
\item
$\mathcal{F}$ is \textbf{linearizable} around $S$, i.e.\ there exist open neighborhoods $U\subset M$ and $O\subset
\widetilde{S}\times_{H}V$ of $S$, and a foliated diffeomorphism, which is the identity on $S$, between
\[(U,\mathcal{F}_{|U})\cong (O,\mathcal{F}_{N|O}),\]
\item $\mathcal{F}$ is \textbf{semi-invariant linearizable} around $S$, i.e.\ there exists an open neighborhood
$U\subset M$ of $S$, and a foliated diffeomorphism, which is the identity on $S$, between
\[(U,\mathcal{F}_{|U})\cong (\widetilde{S}\times_{H}V,\mathcal{F}_{N}),\]
\item $\mathcal{F}$ is \textbf{invariant linearizable} around $S$, i.e.\ there exists a saturated open neighborhood of $S$,
$U\subset M$, and a foliated diffeomorphism, which is the identity on $S$, between
\[(U,\mathcal{F}_{|U})\cong (\widetilde{S}\times_{H}V,\mathcal{F}_{N}).\]
\end{enumerate}
Accordingly, we have three versions of the Local Reeb Stability Theorem:

\begin{theorem}[General Local Reeb Stability]\label{Reeb_Theorem}
Let $(M,\mathcal{F})$ be a foliated manifold and let $S\subset M$ be an embedded leaf with a finite holonomy group.
Then
\begin{enumerate}[1)]
\item The foliation is linearizable around $S$.
\item If $S$ is a finite type manifold, then the foliation is semi-invariant linearizable around $S$.
\item If $S$ is compact, then the foliation is invariant linearizable around $S$.
\end{enumerate}
\end{theorem}

A manifold is called of \textbf{finite type}\index{finite type manifold}, if it admits a proper Morse function with a finite
number of critical points. Equivalently, a manifold is of finite type if and only if it is diffeomorphic to the interior of a
compact manifold with boundary. This condition appears in the study of linearization of proper Lie groupoids in
\cite{Wein3,CrStr}, and seems to be the most suited for semi-invariant linearizability. For an example of a foliation
which is not semi-invariant linearizable around a non-finite type leaf see Example 5.3 \cite{Wein3}.

We prove the theorem in the appendix of this chapter. For the first part, we adapt the proof from \cite{MM} of Reeb's
theorem. The other two versions follow easily by analyzing the local model. Version 3) is the classical Reeb stability
theorem, and the compactness assumption on the leaf is essential for invariant linearization; a good illustration of this
is the foliation from Example 3.5 \cite{Wein3}.

\begin{remark}\rm
Note that if the foliation is linearizable around $S$, then also its holonomy groupoid\index{holonomy groupoid} is
linearizable around $S$ (in the sense of \cite{Wein3}). This suggests that Theorem \ref{Reeb_Theorem} might be a
consequence of the normal form theorem for proper Lie groupoids around orbits \cite{Wein3,Zung,CrStr}.
Nevertheless, these results do not apply directly, because they are about proper Hausdorff Lie groupoids. The
holonomy groupoid, even though it is always a smooth manifold, it might be non-Hausdorff (see the example in the
appendix of this chapter). Even if a neighborhood of the leaf is integrable by a Hausdorff groupoid, it is not trivial that
finiteness of the holonomy group of the leaf implies properness of the holonomy group (of a possibly smaller open).
\end{remark}

\section{A normal form theorem for symplectic foliations}

In this section we describe the local model for a symplectic foliation around a leaf and we prove Theorem
\ref{Theorem_ONE}, a normal form result for finite type leaves.

\subsection{The local model}\label{subsection_model_sympl_foli}

\subsubsection*{The foliation as a Lie algebroid}

Let $(M,\mathcal{F},\omega)$ be a symplectic foliation. Recall from subsection
\ref{subsection_Lie_algebroid_cohomology}, that the cohomology of the Lie algebroid $T\mathcal{F}$ is computed
using the complex
\[(\Omega^{\bullet}(T\mathcal{F}),d_{\mathcal{F}}),\]
where, in this case, $d_{\mathcal{F}}$ is the leafwise de Rham differential. We denote by
\[H^{\bullet}(\mathcal{F})\]
the corresponding cohomology groups. Since $\omega$ is closed, it defines a class \[[\omega]\in H^{2}(\mathcal{F}).\]

The normal bundle to the foliation
\[\nu:=TM/T\mathcal{F},\]
carries a canonical flat $T\mathcal{F}$-connection
\[\nabla:\Gamma(T\mathcal{F})\times\Gamma(\nu)\rmap \Gamma(\nu), \ \ \nabla_X(\overline{Y}):=\overline{[X,Y]},\]
where, for a vector field $Z$ on $M$, we denote by $\overline{Z}$ its image in $\Gamma(\nu)$. This makes $\nu$ into
a representation of the Lie algebroid $T\mathcal{F}$. For a leaf $S$, the Bott connection\index{Bott connection} on
$\nu_S$ defined in the previous section is just the pullback to the Lie sub-algebroid $TS$ of this representation. The
dual connection defines a representation of $T\mathcal{F}$ on $\nu^*$.

Recall that there is a canonical map between the cohomology groups
\[d_{\nu}:H^{2}(\mathcal{F})\rmap H^2(\mathcal{F},\nu^*),\]
constructed as follows (see e.g.\ \cite{CrFe2}): for $[\theta]\in H^{2}(\mathcal{F})$, consider
$\widetilde{\theta}\in\Omega^2(M)$ an extension of $\theta$, and define $d_{\nu}[\theta]$ to be the class of the
2-form
\[(X,Y)\mapsto d\widetilde{\theta}(X,Y,-)\in \nu^*, \ \ X,Y\in T\mathcal{F}.\]

The image of the class of $[\omega]$ under this map will be used to construct the local model. We fix a representative
$\Omega\in \Omega^2(T\mathcal{F},\nu^*)$,
\[d_{\nu}[\omega]=[\Omega]\in H^2(\mathcal{F},\nu^*).\]

\subsubsection*{The vertical derivative}

Consider $S$ a leaf of $M$. Restricting $\Omega$ to $TS$, we obtain a cocycle
\[\Omega_{|TS}\in \Omega^2(S,\nu^*_S).\]
We regard this as a 2-form, denoted by $\delta_S$, on (the manifold) $\nu_S$, as follows:
\[\delta_S\omega\in \Omega^2(\nu_S), \ \ \ \ \delta_S\omega_v:=p^*\langle \Omega_{|TS}, v\rangle.\]
We call $\delta_S\omega$ the \textbf{vertical derivative}\index{vertical derivative} of $\omega$ at $S$. Observe that
$\delta_S\omega$ vanishes on vertical vectors, and so it is determined by its restriction to the horizontal distribution
$T\mathcal{F}_{\nabla}$, corresponding to the Bott connection. Moreover, the fact that $\Omega_{|TS}$ is closed
with respect to $d_{\nabla}$, is equivalent to the fact that $\delta_S\omega$ is closed along the leaves of
$\mathcal{F}_{\nabla}$. This allows us to regard $\delta_{S}\omega$ as a closed foliated 2-form, i.e.\ as a cocycle in
\[(\Omega^{\bullet}(T\mathcal{F}_{\nabla}), d_{\mathcal{F}}).\]

\subsubsection*{The local model}

The local model\index{local model, symplectic foliation} of Poisson manifolds around symplectic leaves was first
defined by Vorobjev \cite{Vorobjev,Vorobjev2} (see also chapter \ref{ChNormalForms}). Here, we present a simple
description of the local model for regular Poisson manifolds.

The \textbf{first order jet} of $\omega$ at $S$ is defined as the foliated 2-form on $\mathcal{F}_{\nabla}$,
\[j^1_{S}(\omega):=p^*(\omega_S)+\delta_S\omega\in \Omega^2(T\mathcal{F}_{\nabla}).\]
By the above, $j^1_{S}(\omega)$ is closed along the leaves of $\mathcal{F}_{\nabla}$. Clearly, $j^1_{S}(\omega)$
restricts to $\omega_S$ on the leaf $S$ (viewed as the zero section). Therefore, the open $N(\omega)\subset \nu_S$
consisting of points where $j^1_{S}(\omega)$ is nondegenerate contains $S$.

\begin{definition}
The \textbf{local model} of $(M,\mathcal{F},\omega)$ around $S$ is
\[(N(\omega),\mathcal{F}_{\nabla|N(\omega)}, j^1_{S}(\omega)_{|N(\omega)}).\]
The symplectic foliation $(M,\mathcal{F},\omega)$ is called \textbf{linearizable} around $S$, if there exist open
neighborhoods of $S$, $O\subset M$ and $U\subset N(\omega)$, and an isomorphism of symplectic foliations
\[(O,\mathcal{F}_{|O},\omega_{|O})\cong (U,\mathcal{F}_{\nabla|U}, j^1_{S}(\omega)_{|U}),\]
which is the identity on $S$.
\end{definition}

The construction of the local model depends on the choice of the 2-form $\Omega$. Yet, this choice has no essential
influence on the outcome (see also the more general Proposition \ref{proposition_splitting_equivalent}).

\begin{proposition}\label{Proposition_independence_of_extension}
Different choices of $\Omega\in \Omega^2(T\mathcal{F},\nu^*)$ satifying $d_{\nu}[\omega]=[\Omega]$ produce local
models that are isomorphic around $S$ by a diffeomorphism that fixes $S$.
\end{proposition}

We first need a version of the Moser Lemma for symplectic foliations.
\begin{lemma}\label{Moser_lemma_for_symplectic foliations}
Let $(M,\mathcal{F},\omega)$ be a symplectic foliation, and $S\subset M$ an embedded leaf. Let
\[\alpha\in \Omega^1(T\mathcal{F})\]
be a 1-form which vanishes on $S$. Then, $\omega+d_{\mathcal{F}}\alpha$ is nondegenerate in a neighborhood $U$
of $S$, and the resulting symplectic foliation \[(U,\mathcal{F}_{|U},\omega+d_{\mathcal{F}}\alpha)\] is isomorphic
around $S$ to $(M,\mathcal{F},\omega)$ by a diffeomorphism that fixes $S$.
\end{lemma}
\begin{proof}
Since $\alpha$ vanishes on $S$, it follows that $(\omega+d_{\mathcal{F}}\alpha)_{|S}=\omega_S$, thus it is
nondegenerate on $S$. Therefore, on some open $U$ around $S$, we have that $\omega+d_{\mathcal{F}}\alpha$ is
nondegenerate along the leaves of $\mathcal{F}$. Moreover, by the Tube Lemma \ref{TubeLemma}, we can choose
$U$ such that
\[\omega_t:= \omega+td_{\mathcal{F}}\alpha\in \Omega^2(T\mathcal{F})\]
is nondegenerate along the leaves of $\mathcal{F}_{|U}$, for all $t\in [0,1]$. Consider the time dependent vector field
$X_t$ on $U$, tangent to $\mathcal{F}$, determined by
\[\iota_{X_t}\omega_t=-\alpha, \ \ X_t\in \Gamma(T\mathcal{F}_{|U}).\]
Since $X_t$ vanishes along $S$, again by the Tube Lemma \ref{TubeLemma}, there is an open neighborhood of $S$,
$V\subset U$, such that $\Phi^t_{X}$, the flow of $X_t$, is defined up to time 1 on $V$. We claim that $\Phi^1_X$
gives the desired isomorphism. Clearly $\Phi^1_X$ preserves the foliation and is the identity on $S$. On each leaf
$\widetilde{S}$ of the foliation, we have that
\begin{align*}
\frac{d}{dt}\Phi_X^{t*}(\omega_{t|\widetilde{S}})=\Phi_X^{t*}(L_{X_t}\omega_{t|\widetilde{S}}+d_{\mathcal{F}}\alpha_{|\widetilde{S}})=\Phi_X^{t*}(d\iota_{X_t}\omega_{t|\widetilde{S}}+d\alpha_{|\widetilde{S}})=0.
\end{align*}
Thus $\Phi_X^{t*}(\omega_{t|\widetilde{S}})$ is constant, and since $\Phi_X^0=\textrm{Id}$, we have that
\[\Phi_X^{1*}((\omega+d_{\mathcal{F}}\alpha)_{|\widetilde{S}})=\omega_{|\widetilde{S}}.\]
So, $\Phi^1_X$ is an isomorphism onto its image between the symplectic foliations
\[\Phi_X^1:(V,\mathcal{F}_{|V},\omega_{|V})\rmap (U,\mathcal{F}_{|U},\omega_{|U}+d_{\mathcal{F}}\alpha_{|U}).\]
\end{proof}

\begin{proof}[Proof of Proposition \ref{Proposition_independence_of_extension}]
Let $\Omega_1$ and $\Omega_2$ be such 2-forms. Since
\[[\Omega_{1|S}]=[\Omega_{2|S}]=d_{\nu}[\omega]_{|S}\in H^2(S,\nu^*_S),\]
it follows that $\Omega_{2|S}-\Omega_{1|S}$ is exact, i.e.\ there is a 1-form
\[\lambda\in \Omega^1(S,\nu^*_S),\ \textrm{ such that }\ d_{\nabla}\lambda=\Omega_{2|S}-\Omega_{1|S}.\]
Let $\widetilde{\lambda}$ be the foliated 1-form on $\nu_S$ induced by $\lambda$, i.e.\
\[\widetilde{\lambda}\in\Omega^1(T\mathcal{F}_{\nabla}), \ \ \widetilde{\lambda}_v:=p^*(\langle\lambda,v\rangle)_{|T\mathcal{F}_{\nabla}}.\]
Now if $j^1_S(\omega)$ denotes the first jet constructed using $\Omega_1$, then the one corresponding to
$\Omega_2$ is
\[j^1_S(\omega)+\widetilde{d_{\nabla}\lambda}=j^1_S(\omega)+d_{\mathcal{F}_{\nabla}}\widetilde{\lambda}.\]
Since $\widetilde{\lambda}$ vanishes on $S$, the result follows from the previous lemma.
\end{proof}

\subsubsection*{Cohomological variation}

The vertical derivative of $\omega$ encodes the variation at $S$ of the symplectic forms on the leaves. To define the
cohomological variation, note first that the leaf $\widetilde{S}_{\mathrm{lin}}$ covers all leaves of the local model,
via the maps
\[p_v:\widetilde{S}_{\mathrm{lin}}\rmap S_v, \ \ p_v([\gamma])=T([\gamma,v]),\ \ v\in \nu_{S,x},\]
where $S_v$ is the leaf through $v$ of the local model, and $T$ is the map defined in (\ref{transport}) using parallel
transport. This defines a linear map that associates to $v\in \nu_{S,x}$ the closed 2-form
\begin{equation}\label{EQ_variation}
p_v^*(\delta_S\omega)\in \Omega^2(\widetilde{S}_{\textrm{lin}}).
\end{equation}
By the proof of Proposition \ref{Proposition_independence_of_extension}, we see that the cohomology class of
$p^*_v(\delta_S\omega)$ does not depend on the 2-form $\Omega$ used to construct $\delta_{S}\omega$. The induced
linear map to the cohomology of $\widetilde{S}_{\textrm{lin}}$, will be called the \textbf{cohomological
variation}\index{cohomological variation} of $\omega$ at $S$
\[[\delta_S\omega]:\nu_{S,x}\rmap H^2(\widetilde{S}_{\textrm{lin}}), \ [\delta_S\omega](v):=[p^*_v(\delta_S\omega)].\]

\subsection{Theorem \ref{Theorem_ONE}, statement and proof}

We state now the main result of this chapter:

\begin{mtheorem}\label{Theorem_ONE}
Let $(M,\mathcal{F}, \omega)$ be a symplectic foliation, $S\subset M$ an embedded symplectic leaf and consider
$x\in S$. If $S$ is a finite type manifold, the holonomy group at $x$ is finite, and the cohomological variation is
surjective
\[[\delta_S\omega]:\nu_{S,x}\rmap H^2(\widetilde{S}_{\textrm{lin}}),\]
then, the symplectic foliation is linearizable around $S$.
\end{mtheorem}

\begin{proof}
Denote as before $V:=\nu_{S,x}$ and by $H$ the holonomy group at $x$. Since $H$ is finite and $S$ is a finite type
leaf, the second version of Reeb stability applies. So, an open around $S$ in $M$ is diffeomorphic to the foliated
manifold
\[(\widetilde{S}\times_H V,\mathcal{F}_N),\]
where $\widetilde{S}$ is the holonomy cover of $S$, which, in this case, coincides with
$\widetilde{S}_{\textrm{lin}}$. The symplectic leaves are
\[\{(S_v,\omega_v)\}_{v\in V}, \ \ S_v=(\widetilde{S}\times Hv)/H\]
Denote the pullback of $\omega_v$ to $\widetilde{S}$ via the map $y\mapsto [y,v]$ by $\widetilde{\omega}_v$. Then
$\{\widetilde{\omega}_v\}_{v\in V}$ is a smooth family of symplectic structures on $\widetilde{S}$, which is
$H$-equivariant:
\[\widetilde{\omega}_{gv}=g^*(\widetilde{\omega}_v), \ \textrm{for all }g\in H, \ v\in V.\]
Note that a smooth family of forms $\{\eta_{v}\in \Omega^{\bullet}(\widetilde{S})\}_{v\in V}$ descends to a foliated
form on $(\widetilde{S}\times_HV,\mathcal{F}_N)$ if and only if it is $H$-equivariant.

Consider the extension $\Omega\in \Omega^2(\widetilde{S}\times_H V)$ of $\omega$ that vanishes on vectors tangent
to the fibers $\{y\}\times V\hookrightarrow\widetilde{S}\times_H V$ of the projection to $S$. The variation
$\delta_S\omega$ corresponding to $\Omega$ is given by the $H$-equivariant family:
\[\delta_{S}\omega_v:=\frac{d}{d\epsilon}(\widetilde{\omega}_{\epsilon v})_{|\epsilon=0}\in \Omega^2(\widetilde{S}),\]
thus, the local model is represented by the $H$-equivariant family of 2-forms:
\[j^1_{S}(\omega)_v=p^*(\omega_S)+\delta_S\omega_v\in \Omega^2(\widetilde{S}).\]

Consider the $H$-equivariant map
\[f:V\rmap H^2(\widetilde{S}),\ \ f(v)=[\widetilde{\omega}_v]-[\widetilde{\omega}_0].\]
Clearly, $f(0)=0$ and its differential at $0$ is the cohomological variation
\[df_{0}(v)=[\delta_{S}\omega]v, \ \ \forall \ v\in V.\]
So, by our hypothesis, $f$ is a submersion around $0$. We show that $f$ can be equivariantly linearized. For this we
construct an $H$-equivariant embedding
\[\chi:U\hookrightarrow V,\]
where $U$ is an $H$-invariant open neighborhoods of $0$ in $V$, such that
\[\chi(0)=0 \ \textrm{ and }\ f(\chi(v))=df_0(v).\]
Since $H$ is finite, we can find an $H$-equivariant projection $p_K:V\to K$, where $K:=\ker(df_0)$. The differential at
$0$ of the $H$-equivariant map
\[(f,p_K):V\rmap H^2(\widetilde{S})\times K, \ v\mapsto(f(v),p_K(v))\]
is $(df_{0},p_K)$. So $(f,p_K)$ is a diffeomorphism when restricted to a small open $U_0$ in $V$ around $0$ and,
since $0$ is a fixed point for the action of $H$, we may assume $U_0$ to be $H$-invariant. Define the embedding as
follows
\[\chi:U\hookrightarrow V, \ \ \chi=(f,p_K)^{-1}\circ(df_{0},p_K),\]
where $U:=(df_{0},p_K)^{-1}(U_0)$. Clearly $U$ is $H$-invariant, $\chi(0)=0$, $\chi$ is $H$-equivariant, and since
\[(f(\chi(v)),p_K(\chi(v)))=(df_{0}(v),p_K(v)),\]
we also obtain that $f(\chi(v))=df_{0}(v)$.

By $H$-equivariance, $\chi$ induces a foliation preserving embedding
\[\widetilde{\chi}: (\widetilde{S}\times_H U,\mathcal{F}_N) \hookrightarrow (\widetilde{S}\times_H V,\mathcal{F}_N), \ \ \widetilde{\chi}([y,v])=[y,\chi(v)]\]
that restricts to a diffeomorphism between the leaf $S_v$ and the leaf $S_{\chi(v)}$. The pullback of $\omega$ under
$\widetilde{\chi}$ is the $H$-equivariant family
\[ \overline{\omega}:=\{\overline{\omega}_v:=\widetilde{\omega}_{\chi(v)}\}_{v\in U},\]
Since $d_0\chi=\textrm{Id}_V$, it follows that $\overline{\omega}$ has the same variation as $\omega$ at $S$
\begin{align*}
\delta_S\overline{\omega}_v=\frac{d}{d\epsilon}_{|\epsilon=0}\overline{\omega}_{\epsilon v}=\frac{d}{d\epsilon}_{|\epsilon=0}\widetilde{\omega}_{\chi(\epsilon v)}=\delta_S\omega_v.
\end{align*}
On the other hand, we obtain that
\begin{align*}
[\overline{\omega}_v]-[\overline{\omega}_0]&=[\widetilde{\omega}_{\chi(v)}]-[\widetilde{\omega}_0]=f(\chi(v))=df(0)v=[\delta_S\omega]v.
\end{align*}
Equivalently, this relation can be rewritten as
\[\overline{\omega}_v=j^1_{S}(\omega)_v+\eta_v, \forall v\in U,\]
where $\{\eta_v\}_{v\in U}$ is an $H$-equivariant family of exact 2-forms that vanishes for $v=0$. For such a family,
one can choose a smooth family of primitives, $\alpha_v\in \Omega^1(\widetilde{S})$ for $v\in U$, such that
$\alpha_0=0$ (we prove this in the appendix Lemma \ref{Appendix_Lemma}). Also, we may assume that $\alpha_v$
$H$-equivariant, by averaging:
\[\frac{1}{|H|}\sum_{g\in H}(g^{-1})^*(\alpha_{gv}).\]
So, we obtain that
\begin{equation}\label{EQ_cutarica}
\overline{\omega}_v=j^1_{S}(\omega)_v+d\alpha_v,
\end{equation}
with $\alpha_0=0$ and $\alpha$ is $H$-equivariant. Regarding $\alpha$ is a foliated 1-form,
\[\alpha\in\Omega^1(T\mathcal{F}_N),\]
we can rewrite (\ref{EQ_cutarica}) as
\[\overline{\omega}=j^1_S\omega+d_{\mathcal{F}_N}\alpha,\]
therefore Lemma \ref{Moser_lemma_for_symplectic foliations} concludes the proof.
\end{proof}

\section{Appendix}\label{section_Appendix_Reeb}

\subsection{Proof of Theorem \ref{Reeb_Theorem}, part 1)}\label{subsection_Proof_reeb}

Since $S$ is embedded, by restricting to a tubular neighborhood, we may assume that the foliation is on a vector
bundle $p:E\to S$, for which $S$, identified with the zero section, is a leaf. Then the holonomy of paths in $S$ is
represented by germs of a diffeomorphism between the fibers of $E$.

Each point in $S$ has an open neighborhood $U\subset E$ such that
\begin{itemize}
\item $S\cap U$ is 1-connected,
\item for $x\in S\cap U$, $E_x\cap U$ is a connected neighborhood of $x$,
\item for every $x,y\in S\cap U$, the holonomy along any path in $S\cap U$ connecting them is defined as a diffeomorphism between
\[hol_{x}^y:E_x\cap U\diffto E_y\cap U.\]
\end{itemize}

Let $\mathcal{U}$ be locally finite cover of $S$ by opens $U\subset E$ of the type just described, such that for all
$U,U'\in\mathcal{U}$, $U\cap U'\cap S$ is connected (or empty), and such that each $U\in \mathcal{U}$ is relatively
compact.

We fix $x_0\in S$, $U_0\in\mathcal{U}$ an open containing $x_0$, and denote by
\[V:=E_{x_0}.\]
Consider a path $\gamma$ in $S$ starting at $x_0$ and with endpoint $x$. Cover the path by a chain of opens in
$\mathcal{U}$
\[\xi=(U_{0},\ldots, U_{k(\xi)}),\]
such that there is a partition
\[0=t_0<t_1<\ldots t_{k-1}<t_k=1,\]
with $\gamma([t_{j-1},t_{j}])\subset U_{j}$. Since the holonomy transformations inside $U_{j}$ are all trivial, and all
the intersections $U_i\cap U_j\cap S$ are connected, it follows that the holonomy of $\gamma$ only depends on the
chain $\xi$ and is defined as an embedding
\[hol(\gamma)=hol_{x_0}^x(\xi):O(\xi)\hookrightarrow E_x,\]
where $O(\xi)\subset V$ is an open neighborhood of $x_0$, which is independent of $x\in U_{k(\xi)}$. Denote by
$\mathcal{Z}$ the space of all chains in $\mathcal{U}$
\[ \xi=(U_{0}, \ldots, U_{k(\xi)}),\ \  \textrm{with} \ U_{l}\cap U_{l+1}\neq \emptyset.\]

Denote by $H:= \textrm{Hol}(S,x_0)$, the holonomy group at $x_0$, and let
\[K:=\ker(\pi_1(S,x_0)\to H).\]
The holonomy cover $\widetilde{S}\to S$ can be described as the space of all paths $\gamma$ in $S$ starting at $x_0$,
and two such paths $\gamma_1$ and $\gamma_2$ are equivalent if they have the same endpoint, and the homotopy
class of $\gamma_2^{-1}\circ \gamma_1$ lies in $K$. The projection is then given by $[\gamma]\mapsto \gamma(1)$.
Denote by $\widetilde{x}_0$ the point in $\widetilde{S}$ corresponding to the constant path at $x_0$. So, we can
represent each point in $\widetilde{S}$ (not uniquely!) by a pair $(\xi,x)$ with $\xi \in \mathcal{Z}$ and endpoint $x\in
U_{k(\xi)}\cap S$.

The group $H$ acts freely on $\widetilde{S}$ by pre-composing paths. For every $g\in H$ fix a chain
$\xi_g\in\mathcal{Z}$, such that $(\xi_g,x_0)$ represents $\widetilde{x}_0g$. Consider the open
\[\widetilde{O}_0:=\bigcap_{g\in H}O(\xi_{g})\subset V,\]
on which all holonomies $hol_{x_0}^{x_0}(\xi_{g})$ are defined, and a smaller open $\widetilde{O}_1\subset
\widetilde{O}_0$ around $x_0$, such that $hol_{x_0}^{x_0}(\xi_{g})$ maps $\widetilde{O}_1$ into
$\widetilde{O}_0$. Hence the composition
\begin{equation*}
hol_{x_0}^{x_0}(\xi_g)\circ hol_{x_0}^{x_0}(\xi_h): \widetilde{O}_{1}\hookrightarrow V,
\end{equation*}
is well defined. Since the germs of $hol_{x_0}^{x_0}(\xi_g)\circ hol_{x_0}^{x_0}(\xi_h)$ and
$hol_{x_0}^{x_0}(\xi_{gh})$ are the same, by shrinking $\widetilde{O}_1$ if necessary, we may assume that
\begin{equation}\label{composition}
hol_{x_0}^{x_0}(\xi_g)\circ hol_{x_0}^{x_0} (\xi_h)=hol_{x_0}^{x_0}(\xi_{gh}): \widetilde{O}_{1}\hookrightarrow V.
\end{equation}
Consider the following open
\[O:=\bigcap_{g\in H}hol_{x_0}^{x_0}(\xi_g)(\widetilde{O}_{1}).\]
Then $O\subset \widetilde{O}_{1}$, and for $h\in H$, we have that
\begin{align*}
hol_{x_0}^{x_0}(\xi_{h})(O)&\subseteq  \bigcap_{g\in H}hol_{x_0}^{x_0}(\xi_{h})\circ hol_{x_0}^{x_0}(\xi_{g})(\widetilde{O}_{1})
=\\
&=\bigcap_{g\in H}hol_{x_0}^{x_0}(\xi_{hg})(\widetilde{O}_{1})=\bigcap_{g\in H}hol_{x_0}^{x_0}(\xi_{g})(\widetilde{O}_{1})= O.
\end{align*}
So $hol_{x_0}^{x_0}(\xi_h)$ maps $O$ to $O$, and by (\ref{composition}) it follows that the holonomy transport along
$\xi_g$ defines an action of $H$ on $O$, which we denote by
\[hol(g):=hol_{x_0}^{x_0}(\xi_g):O\rmap O.\]
Since $H$ is a finite group acting on $O$ with a fixed point $x_0$, by Bochner's Linearization Theorem (see e.g.\
\cite{DK}), we can linearize the action around $x_0$. So, by shrinking $O$ if necessary, the action is isomorphic to the
linear holonomy action of $H$ on $V$. In particular, this implies that $O$ contains arbitrarily small $H$-invariant open
neighborhoods of $x_0$.

Since $\mathcal{U}$ is a locally finite cover by relatively compact opens, there are only finitely many chains in
$\mathcal{Z}$ of a certain length. Denote by $\mathcal{Z}_n$ the set of chains of length at most $n$. Let $c$ be such
that $\xi_g\in \mathcal{Z}_c$ for all $g\in H$.

By the above and by the basic properties of holonomy, there exist open neighborhoods $\{O_n\}_{n\geq 1}$ of $x_0$ in
$O$:
\[\ldots \subset O_{n+1}\subset O_n\subset O_{n-1}\subset \ldots \subset O_1\subset O\subset V,\]
satisfying the following:
\begin{enumerate}[1)]
\item for every chain $\xi\in \mathcal{Z}_n$, $O_n\subset O({\xi})$,
\item for every two chains $\xi,\xi'\in \mathcal{Z}_n$ and $x\in U_{k(\xi)}\cap U_{k(\xi')}\cap S$, such that the pairs $(\xi,x)$ and
$(\xi',x)$ represent the same element in $\widetilde{S}$, we have that
\[hol_{x_0}^x(\xi)=hol_{x_0}^x(\xi'):O_n\hookrightarrow E_x,\]
\item $O_n$ is $H$-invariant,
\item for every $g\in H$, $\xi\in\mathcal{Z}_{n}$ and $x\in U_{k(\xi)}\cap S$, we have that
\[hol_{x_0}^x(\xi_g\cup \xi)=hol_{x_0}^x(\xi)\circ hol(g): O_{n+c}\hookrightarrow E_{x}.\]
\end{enumerate}

Denote by $\widetilde{S}_n$ the set of points in $\widetilde{x}\in \widetilde{S}$ for which every element in the orbit
$\widetilde{x}H$ can be represented by a pair $(\xi,x)$ with $\xi\in\mathcal{Z}_n$. Note that for $n\geq c$,
$\widetilde{S}_n$ is nonempty, $H$-invariant, open and connected. Consider the following $H$-invariant open
neighborhood of $\widetilde{S}\times\{x_0\}$:
\[\mathcal{V}:=\bigcup_{n\geq c}\widetilde{S}_n\times O_{n+c}\subset \widetilde{S}\times V.\]
On $\mathcal{V}$ we define the map
\[\widetilde{\mathcal{H}}:\mathcal{V}\rmap E,  \ \ \ \widetilde{\mathcal{H}}(\widetilde{x},v):=hol_{x_0}^x(\xi)(v),\]
for $(\widetilde{x},v)\in\widetilde{S}_n\times O_{n+c}$ and $(\xi,x)$ is pair representing $\widetilde{x}$ with
$\xi\in\mathcal{Z}_{n}$ and $x\in U_{k(\xi)}$. By the properties of the opens $O_n$, $\widetilde{\mathcal{H}}$ is
well defined. Since the holonomy transport is by germs of diffeomorphisms and preserves the foliation, it follows that
$\widetilde{\mathcal{H}}$ is a foliated local diffeomorphism, which sends the trivial foliation on $\mathcal{V}$ with
leaves $\mathcal{V}\cap \widetilde{S}\times\{v\}$ to $\mathcal{F}_{|E}$.

We  prove now that $\widetilde{\mathcal{H}}$ is $H$-invariant. Let $(\widetilde{x},v)\in\widetilde{S}_n\times
O_{n+c}$ and $g\in H$. Consider chains $\xi$ and $\xi'$ in $\mathcal{Z}_{n}$ representing $\widetilde{x}$ and
$\widetilde{x}g$ respectively, with $x\in U_{k(\xi)}\cap U_{k(\xi')}\cap S$. Then $\xi'$ and $\xi_g\cup \xi$ both belong
to $\mathcal{Z}_{n+c}$ and $(\xi',x)$, $(\xi_g\cup\xi,x)$ both represent $\widetilde{x}g\in\widetilde{S}$. Hence, by
the properties 2) and 4) of the opens $O_n$, we obtain $H$-invariance:
\begin{align*}
\widetilde{\mathcal{H}}(\widetilde{x}g,hol(g^{-1})v)&=hol_{x_0}^x(\xi')(hol(g^{-1})v)=hol_{x_0}^x(\xi_g\cup \xi)(hol(g^{-1})v)=\\
&=hol_{x_0}^x(\xi)\circ hol(g)\circ hol(g^{-1})v=hol_{x_0}^x(\xi)(v)=\widetilde{\mathcal{H}}(\widetilde{x},v).
\end{align*}

Since the action of $H$ on $\mathcal{V}$ is free and preserves the foliation on $\mathcal{V}$, we obtain an induced
local diffeomorphism of foliated manifolds:
\[\mathcal{H}: \mathcal{V}/H\subset \widetilde{S}\times_{H} V\rmap E.\]

We prove now that $\mathcal{H}$ is injective. Let $(\widetilde{x},v), (\widetilde{x}',v')\in \mathcal{V}$ be such that
\[\widetilde{\mathcal{H}}(\widetilde{x},v)=\widetilde{\mathcal{H}}(\widetilde{x}',v').\]
Denoting by $x=p(\widetilde{\mathcal{H}}(\widetilde{x},v))=p(\widetilde{\mathcal{H}}(\widetilde{x}',v'))$, we have
that $\widetilde{\mathcal{H}}(\widetilde{x},v)$, $\widetilde{\mathcal{H}}(\widetilde{x}',v')\in E_x$. Hence
$\widetilde{x}$ and $\widetilde{x}'$, both lie in the fiber of $\widetilde{S}\to S$ over $x$, thus there is a unique $g\in
H$ with $\widetilde{x}'=\widetilde{x}g$. Let $n,m\geq c$ be such that $(\widetilde{x},v)\in \widetilde{S}_n\times
O_{n+c}$ and $(\widetilde{x}',v')\in \widetilde{S}_m\times O_{m+c}$, and assume also that $n\leq m$. Consider
$\xi\in\mathcal{Z}_n$ and $\xi'\in \mathcal{Z}_m$ such that $(\xi,x)$ represents $\widetilde{x}$ and $(\xi',x)$
represents $\widetilde{x}'$. Then we have that
\begin{equation}\label{EQ1}
hol_{x_0}^x(\xi)(v)=hol_{x_0}^x(\xi')(v').
\end{equation}
Since both $(\xi',x)$ and $(\xi_g\cup \xi,x)$ represent $\widetilde{x}'\in \widetilde{S}$, and both have length $\leq
m+c$, again by the properties 2) and 4) we obtain
\[hol_{x_0}^x(\xi')(v')=hol_{x_0}^x(\xi_g\cup \xi)(v')=hol_{x_0}^x(\xi)(hol(g)(v')).\]
Since $hol_{x_0}^x(\xi)$ is injective, (\ref{EQ1}) implies that $v=hol(g)(v')$. So, we obtain
\[(\widetilde{x},v)=(\widetilde{x}'g^{-1},hol(g)(v')),\]
which proves injectivity of $\mathcal{H}$.

\subsection{Proof of  Theorem \ref{Reeb_Theorem}, part 2), 3)}

To prove the second version of the theorem, we will use the following property of finite type manifolds
\begin{theorem}[Theorem 5.1. \cite{Wein3}]\label{Weinstein}
Let $f : X \to Y$ be a submersion that is equivariant with respect to the action of a compact group $G$. If $y\in Y$ is a
fixed point and if $O:= f^{-1}(y)$ is a manifold of finite type, then there exists a $G$-invariant open neighborhood
$U\subset X$ of $O$, and a $G$-equivariant retraction $\rho:U\to O$ such that
\[(\rho,f) :U \rmap O\times f(U)\]
is a $G$-equivariant diffeomorphism.
\end{theorem}

Using part 1) of Theorem \ref{Reeb_Theorem}, part 2) is a consequence of the following property of the local model
around finite type leaves:

\begin{proposition}
Let $U\subset \widetilde{S}\times_H V$ be an open neighborhood containing $S=(\widetilde{S}\times\{0\})/H$. If $S$
is a finite type manifold, then there exists a smaller open $\widetilde{U}\subset U$, containing $S$, and an
isomorphism of foliated manifolds
\[(\widetilde{U},\mathcal{F}_{N|\widetilde{U}})\cong (\widetilde{S}\times_H V,\mathcal{F}_N),\]
which is the identity on $S$.
\end{proposition}
\begin{proof} The open $U$ is of the form $X/H$, where $X\subset \widetilde{S}\times V$ is an $H$-invariant open containing
$\widetilde{S}\times\{0\}$. Consider the $H$-equivariant projection
\[f:X\rmap V, \ \ f(x,v)=v.\]
Clearly, $f$ is a submersion, and its fiber above $0\in V$ is $\widetilde{S}\times\{0\}$. Since $S$ is of finite type, it
follows that also its finite cover $\widetilde{S}$ is a finite type manifold (just pull back to $\widetilde{S}$ an
appropriate Morse function on $S$). Hence, by Theorem \ref{Weinstein}, there exists an $H$-invariant open $X'\subset
X$ containing $\widetilde{S}\times\{0\}$ and an $H$-equivariant retraction $\rho:X'\to \widetilde{S}$ such that
\[(\rho,f):X'\rmap \widetilde{S}\times f(X')\]
is an $H$-equivariant diffeomorphism. Since $f(X')$ is an open neighborhood of $0$ in $V$, we can find $W\subset
f(X')$, a smaller open around $0$, diffeomorphic to $V$ by an $H$-equivariant map
\[\chi:W\rmap V.\]
Denote by $\widetilde{X}:=X'\cap f^{-1}(W)$. Then the map
\[(\rho, \chi\circ f):\widetilde{X}\rmap \widetilde{S}\times V,\]
is an $H$-equivariant diffeomorphism that restrict to a diffeomorphism between the leaf
$(\widetilde{S}\times\{v\})\cap \widetilde{X}$ and the leaf $\widetilde{S}\times\{\chi(v)\}$. Hence, it induces a
foliated diffeomorphism between
\[(\widetilde{U},\mathcal{F}_{N|\widetilde{U}})\cong (\widetilde{S}\times_H V,\mathcal{F}_N),\]
where $\widetilde{U}:=\widetilde{X}/H\subset U$.
\end{proof}

The classical Reeb stability theorem is a direct consequence of part 2) of the theorem. To see this, observe first that
compactness of $S$ implies compactness of all the leaves in the local model. Let $U$ be an open neighborhood of $S$
on which $\mathcal{F}$ is semi-invariant linearizable. We prove that $U$ is invariant. Consider $x\in U$, and let $S'$
be the leaf of $\mathcal{F}$ through $x$. The connected component containing $x$ of $S'\cap U$, denoted by $S'_x$,
is the leaf of $\mathcal{F}_{|U}$ through $x$. Thus $S'_x$ is compact. Since it is also open in $S'$, we have that
$S'=S'_x$. This shows that $S'\subset U$, hence $U$ is invariant.

\subsection{Example: trivial holonomy, non-Hausdorff holonomy groupoid}\label{trivial holonomy, non-Hausdorff holonomy group}

We construct a foliation with an embedded leaf $S$ with trivial holonomy, and such that the holonomy groupoid of any
saturated neighborhood of $S$ is non-Hausdorff. The example is of the form \[(M\times_{\Gamma} V,\mathcal{F}),\]
where $M$ and $V$ are connected manifolds, $\Gamma$ is a discrete group acting on $M$ and $V$, the action on
$M$ is free and proper (such that $M/\Gamma$ is a manifold), and the foliation $\mathcal{F}$ has leaves
\[M_{v}:=M\times_{\Gamma} \Gamma v,\  \textrm{for}\  v\in V.\]
The projection $M\times_{\Gamma} V\to M/\Gamma$ is a fiber bundle, and freeness of the action on $M$ implies that
all the fibers $\{m\}\times V\hookrightarrow M\times_{\Gamma} V$ are diffeomorphic to $V$. Moreover, each fiber is
transversal to the foliation, and therefore can be used to compute holonomy. Using these transversals, it is easy to see
that the following are equivalent:
\[M_v =\ \textrm{embedded}\ \Leftrightarrow \ M_v=\ \textrm{closed} \ \Leftrightarrow\  \Gamma v\subset V =\ \textrm{discrete}.\]

Observe that the foliation, viewed as a Lie algebroid, is integrable by the (Hausdorff) Lie groupoid
\[\mathcal{G}:=(M\times M\times V)/\Gamma \rightrightarrows M\times_{\Gamma} V,\]
whose $s$-fibers are all diffeomorphic to $M$, thus connected. Since the holonomy groupoid
$\mathrm{Hol}(\mathcal{F})$ is the smallest Lie groupoid integrating $T\mathcal{F}$ \cite{MM}, it is a quotient of
$\mathcal{G}$ by a normal subgroupoid (a bundle of groups) $\mathcal{K}$
\[\mathcal{K}\rmap \mathcal{G}\rightrightarrows \mathrm{Hol}(\mathcal{F}).\]
Now, the projection
\[\textrm{P}:\mathcal{G}\rightrightarrows \mathrm{Hol}(\mathcal{F})\]
works as follows: for $[m_1,m_0,v]\in \mathcal{G}$, consider a path $m_t$ in $M$ starting at $m_0$ and ending at
$m_1$, and define $\textrm{P}([m_1,m_0,v])$ to be the holonomy class of the path $[m_t,v]$ in $M_v$.

For $v\in V$, let $\Gamma_v\subset\Gamma$ be the stabilizer of $v$, and denote by $K_{v}$ the following subgroup of
$\Gamma_v$:
\[K_v:=\{g\in \Gamma_v| \textrm{germ}_v(g:V\to V)=\textrm{germ}_v(\textrm{Id}_V:V\to V)\}.\]
We claim that the subgroupoid $\mathcal{K}$ is given by
\[\mathcal{K}_{[m,v]}:=\{[mg,m,v]|g\in K_v\}\subset \mathcal{G}_{[m,v]}.\]
To see this, consider a closed loop $\gamma_t:=[m_t,v]\in M_{v}$, with $m_0=m$. Since $[m_1,v]=[m,v]$, for some
$g\in\Gamma_v$, we have that $m_1=mg$. Then, the holonomy transport along this path is multiplication by $g$:
\[hol(\gamma):\{m\}\times V\rmap \{m\}\times V, \ \ [m,w]\mapsto [mg,w]=[m,gw].\]
This implies that $hol(\gamma)$ is trivial, if and only if $g\in K_{v}$. Since $hol(\gamma)=\textrm{P}([mg,m,v])$, this
implies the claim.

Hausdorffness translates to the following property of the action:
\begin{lemma}
Let $\Gamma_{\circ}\subset \Gamma$ be the kernel of the homomorphism $\Gamma\to \textrm{Diff}(V)$. Then
$\mathrm{Hol}(\mathcal{F})$ is Hausdorff if and only if $K_{v}=\Gamma_{\circ}$ for all $v\in V$.
\end{lemma}
\begin{proof}
If $K_{v}=\Gamma_{\circ}$, then the holonomy groupoid is the Hausdorff groupoid
\[\mathrm{Hol}(\mathcal{F})=(M/\Gamma_{\circ}\times M/\Gamma_{\circ}\times V)/(\Gamma/\Gamma_{\circ})\rightrightarrows M/\Gamma_{\circ}\times_{\Gamma/\Gamma_{\circ}}V\cong M\times_{\Gamma}V.\]
Conversely, assume that there is some $g\in K_{v}\backslash \Gamma_{\circ}$. Then, both opens
\[\{w | g\in K_w\}\  \textrm{and} \ \{w |gw\neq  w\}\]
are nonempty. By connectedness of $V$, we find a sequence $w_n$ in the first open which converges to a point $w$ on
the boundary, i.e.\ $g\in K_{w_n}$, but $g\notin K_{w}$. Then $[mg,m,w_n]$ is a sequence in $\mathcal{K}$, which
converges to the element $[mg,m,w]\notin \mathcal{K}$. Thus $\mathcal{K}$ is not closed, and this implies that the
quotient is not Hausdorff.
\end{proof}

For our example, let $\Gamma$ be the free group in generators $\{x_n| n\geq 1\}$. For every $n\geq 1$, let
$\varphi_n:\mathbb{R}\diffto\mathbb{R}$ be a diffeomorphism such that
\begin{itemize}
\item $\varphi_n(0)=0$ and $\textrm{germ}_0(\varphi_n)=\textrm{germ}_0(\textrm{Id})$,
\item $\varphi_n(\frac{1}{n})=\frac{1}{n}$ and $\textrm{germ}_{1/n}(\varphi_n)\neq \textrm{germ}_{1/n}(\textrm{Id})$.
\end{itemize}
By freeness, $\Gamma$ has a unique action on $\mathbb{R}$ which sends $x_n$ to $\varphi_n$, and since every
element of $\Gamma$ is a product of $x_{n}^{\pm1}$, it follows that
\begin{equation}\label{EQ_Example_foliations}
K_0=\Gamma, \ \ \ K_{\frac{1}{n}}\neq \Gamma.
\end{equation}
Let $M$ be the universal cover of $S:=\mathbb{R}^2\backslash\{(n,0)| n\geq 1\}$. Since the fundamental group of
$S$ is isomorphic to $\Gamma$, it follows that $\Gamma$ has a free and proper action on $M$ (of course $M\cong
\mathbb{R}^2$). We consider the foliated manifold
\[(M\times_{\Gamma}\mathbb{R}, \mathcal{F}_N).\]

Since $0$ is a fixed point for the action of $\Gamma$ on $\mathbb{R}$, it follows that $M_{0}\cong S$ is an embedded
leaf, and since $K_{0}=\Gamma$, the holonomy group of $S$ is trivial. By the lemma above and by
(\ref{EQ_Example_foliations}), the holonomy groupoid is non-Hausdorff. Moreover, we claim that there is no saturated
neighborhood of $S$ for which the holonomy groupoid is Hausdorff. For this, notice that every saturated neighborhood
is of the form $M\times_{\Gamma}U$, where $U\subset \mathbb{R}$ is a $\Gamma$-invariant neighborhood of $0$.
Since $\frac{1}{n}\in U$, for $n$ big enough, the same argument implies the claim. The fact that $S$ has a
non-invariant neighborhood with a Hausdorff (even proper) holonomy groupoid follows from Theorem
\ref{Reeb_Theorem}.

\subsection{On primitives of smooth families of exact forms}

We devote this subsection of the appendix to proving that, on a finite type manifold, a smooth family of exact forms
admits a smooth family of primitives. This result was used in the proof of Theorem \ref{Theorem_ONE}. Alan
Weinstein, after reading the version of the thesis sent for review, has pointed out that a variation of this lemma
appeared in \cite{GLSW}. Also, a slight modification to the argument given in \emph{loc.cit.} implies our result. For
completeness, we have decided to keep the proof in the thesis.

Recall that a \textbf{finite good cover}\index{finite good cover} of a manifold of dimension $m$, is a finite cover by
opens diffeomorphic to $\mathbb{R}^m$ such that every nonempty intersection is diffeomorphic to $\mathbb{R}^m$.

\begin{lemma}
A finite type manifold admits a finite good cover.
\end{lemma}
\begin{proof}
For a finite type manifold $M$, there exists a manifold with boundary $\overline{M}$ such that the interior of
$\overline{M}$ is diffeomorphic to $M$ (see \cite{Wein3}). A good cover of a manifold can be given by a locally finite
cover consisting of geodesically convex opens with respect to some metric (see the proof of Theorem 5.1 in \cite{Tu}).
Endow $\overline{M}$ with a metric which is a product metric on a collar neighborhood of the boundary
\[C=\partial\overline{M}\times[0,1).\]
Consider an open cover $\mathcal{V}_0$ of $M$ and a finite open cover $\mathcal{U}_{0}$ of
$\partial\overline{M}$, both consisting of geodesically convex opens. By our choice of the metric, the finite family
\[\mathcal{U}:=\{U\times (0,1): U\in \mathcal{U}_{0}\},\]
consists also of geodesically convex opens. Since $M\backslash C$ is compact, we can find a finite collection
$\mathcal{V}_1\subset \mathcal{V}_0$ covering it. Hence $\mathcal{U}_1\cup \mathcal{V}_1$ is a finite open cover
of $M$ by geodesically convex opens, therefore it is a finite good cover.
\end{proof}

\begin{lemma}\label{Appendix_Lemma}
Let $M$ and $P$ be smooth manifolds, such that $M$ admits a finite good cover. Let $\{\omega_x\}$ be a family of
$k$ forms on $M$ depending smoothly on $x\in P$. If $\omega_x$ is exact for all $x\in P$, then there exists a family of
$k-1$ forms $\theta_x$, depending smoothly on $x\in P$, such that $d\theta_x=\omega_x$ for all $x\in P$. Moreover, if
$\omega_{x_0}=0$, for some $x_0\in P$, one can choose $\theta_x$ such that $\theta_{x_0}=0$.
\end{lemma}
\begin{proof}
Let $\mathfrak{U}:=\{U_{i}\}_{i\in I}$ be a finite good cover of $M$, i.e.\ $U_{i_0i_1\ldots i_k}:=U_{i_0}\cap
U_{i_1}\cap\ldots\cap U_{i_k}$ is either empty or diffeomorphic to $\mathbb{R}^m$, for all $i_0,\ldots, i_{k}\in I$.
First we show that it suffices to prove the statement for the \v Cech complex
\[(C^{\bullet}(\mathfrak{U}),\delta)\]
associated to $\mathfrak{U}$. For this consider the double complex
\[(C^{\bullet}(\mathfrak{U},\Omega^{\bullet}), D),\ \ C^{p}(\mathfrak{U},\Omega^{q}):=\prod_{i_0<\ldots<i_p}\Omega^q(U_{i_0\ldots i_p}),\]
endowed with the differential $D:=\delta+D''$, where \[\delta:C^{p}(\mathfrak{U},\Omega^{q})\rmap
C^{p+1}(\mathfrak{U},\Omega^{q})\] is the usual \v Cech differential and $D''=(-1)^pd$. The \v Cech complex with
coefficients in the sheaf $\Omega^{q}$
\[0\rmap \Omega^q(M)\stackrel{r}{\rmap} C^{0}(\mathfrak{U},\Omega^{q})\stackrel{\delta}{\rmap} C^{1}(\mathfrak{U},\Omega^{q})\stackrel{\delta}{\rmap} C^{2}(\mathfrak{U},\Omega^{q})\stackrel{\delta}{\rmap} \ldots,\]
is exact, and the homotopy operators $k$ for this complex constructed in Proposition 8.5 \cite{Tu} have the property
that they send smooth families to smooth families. Similarly, using the explicit homotopies for the de Rham complex of
$\mathbb{R}^{m}$ from the standard proof of the Poincar\'{e} lemma (see e.g. \cite{Tu}, Chapter 1 \S 4), we obtain
homotopy operators, denoted $h$, for the complex
\[0\rmap C^p(\mathfrak{U})\stackrel{i}{\rmap} C^{p}(\mathfrak{U},\Omega^{0})\stackrel{d}{\rmap} C^{p}(\mathfrak{U},\Omega^{1})\stackrel{d}{\rmap} C^{p}(\mathfrak{U},\Omega^{2})\stackrel{d}{\rmap} \ldots.\]
Also these operators have the property that they send smooth families to smooth families. By a standard argument for
double complexes (see e.g.\ Proposition 8.8 \cite{Tu}), the maps $r$ and $i$ induce isomorphisms in cohomology
\[r:H^{\bullet}(M)\diffto H_D\{C^{\bullet}(\mathfrak{U},\Omega^{\bullet})\},\ \
i:\check{H}^{\bullet}(\mathfrak{U})\diffto H_D\{C^{\bullet}(\mathfrak{U},\Omega^{\bullet})\},\] where
$H_D\{C^{\bullet}(\mathfrak{U},\Omega^{\bullet})\}$ is the cohomology of the double complex. Using the maps $k$
and $h$ one can construct an explicit isomorphisms
\begin{equation}\label{EQ_faranumar}
H^{\bullet}(M)\diffto \check{H}^{\bullet}(\mathfrak{U}),
\end{equation}
which represents the map $i^{-1}\circ r$. Moreover, this map can be constructed at the level of the closed forms/\v
Cech cocycles (see Proposition 9.8 \cite{Tu}); and the explicit formula one gets implies that it maps smooth families to
smooth families. Using these maps, one reduces the problem to the \v Cech complex.

To simplify the discussion, we will describe what is happening in the case of 2-forms. Let $\omega_x\in\Omega^2(M)$
be a smooth family of closed 2-forms. Consider the elements:
\[\omega_1:=h r\omega\in C^{0}(\mathfrak{U},\Omega^{1}), \ \ \omega_2:=h \delta \omega_1\in C^{1}(\mathfrak{U},\Omega^{0}).\]
Since $\omega$ is closed, we have that
\[d\omega_1=dhr\omega=r\omega-hdr\omega=r\omega-hrd\omega=r\omega.\]
Repeating this computation, we obtain that
\[d\omega_2=dh\delta \omega_1=\delta\omega_1-hd\delta\omega_1=\delta\omega_1-h\delta(r \omega)=\delta\omega_1.\]
In particular, we obtain that
\[d\delta \omega_2=\delta d\omega_2=\delta^2 \omega_1=0,\]
hence $\delta \omega_2=i(\omega_3)$, for some $\omega_3\in C^2(\mathfrak{U})$. Clearly $\omega_{3,x}$ depends
smoothly on $x\in P$, and also $\omega_3$ is closed, since $i$ is injective and
\[i\delta(\omega_3)=\delta^2\omega_2=0.\]
The map $\omega\mapsto \omega_3$ is a lift of the isomorphism from (\ref{EQ_faranumar}).

Assuming that $\omega$ is exact, we will prove that $\omega_{3}$ is exact. Write $d\theta_x=\omega_x$, where we
don't make any assumption on the dependence of $\theta$ on $x$. Let
\[\theta_1:=hr\theta \in C^{0}(\mathfrak{U},\Omega^{0}).\]
Then we have that
\[d\theta_1=d hr\theta=r\theta -h d r \theta= r\theta - h r \omega = r\theta - \omega_1,\]
thus
\[d(\omega_2+\delta\theta_1)=d\omega_2 - \delta \omega_1=0,\]
and so $\omega_2+\delta\theta_1=i(\theta_2)$, for some $\theta_2\in C^1(\mathfrak{U})$. This is the primitive of
$\omega_3$:
\[i \delta \theta_2=\delta (\omega_2+\delta \theta_1)=i(\omega_3).\]

Assuming that the statement is true for the \v Cech cohomology, we can write $\omega_3=\delta\alpha$, for a smooth
family $\alpha_x\in C^1(\mathfrak{U})$. Then we have that
\[\delta (i(\alpha)-\omega_2)=0,\]
therefore, denoting by $\alpha_1:=k(i(\alpha)-\omega_2)$, we obtain
\[\delta \alpha_1=i(\alpha)-\omega_2.\]
Applying $d$ to both sides, we get
\[d\delta \alpha_1=-d\omega_2=-\delta\omega_1,\]
hence
\[\delta (d\alpha_1+\omega_1)=0.\]
Denote now $\alpha_2:=k(d\alpha_1+\omega_1)$. Then
\[\delta \alpha_2=d\alpha_1+\omega_1,\]
hence
\[r\omega=d\omega_1=d\delta\alpha_2=r d \alpha_2,\]
and since $r$ is injective, it follows that $\alpha_2$ is a smooth primitive for $\omega$.

Note also that, if $\omega_{x_0}=0$, for some $x_0\in P$ then, by construction, $\omega_{3,x_0}=0$. If also
$\alpha_{x_0}=0$, then again by construction, $\alpha_{2,x_0}=0$.\\

It remains now to prove the statement for the \v Cech complex. Consider a smooth family $\eta_x\in
C^{k}(\mathfrak{U})$, such that
\[[\eta_x]=0\in \check{H}^k(\mathfrak{U}),\ \ \forall \ x\in P.\]
Let $\Psi:C^{k}(\mathfrak{U})\to C^{k-1}(\mathfrak{U})$ be a linear map, such that if $\lambda\in
C^{k}(\mathfrak{U})$ is exact, then $\lambda=\delta (\Psi\lambda)$; the existence of such a map is a trivial linear
algebra exercise. Define $\alpha$ by $\alpha_x:=\Psi(\eta_x)$, for $x\in P$. Since the vector spaces
$C^{k}(\mathfrak{U})$ are finite dimensional, $\alpha$ is clearly smooth, and since $\eta_x$ is exact, it follows that
$\delta(\alpha)=\eta$. Moreover, if $\eta_{x_0}=0$, then also $\alpha_{x_0}=0$.
\end{proof}

\clearpage \pagestyle{plain}

\chapter{A normal form theorem around symplectic leaves}\label{ChNormalForms}
\pagestyle{fancy}
\fancyhead[CE]{Chapter \ref{ChNormalForms}} 
\fancyhead[CO]{A normal form theorem around symplectic leaves} 

In this chapter we prove a normal form theorem around symplectic leaves, which is the Poisson-geometric version of
the Local Reeb Stability Theorem (from foliation theory) and of the Slice Theorem (from equivariant geometry). Conn's
linearization theorem corresponds to the case of one-point leaves (fixed points). The content of this chapter was
published in \cite{CrMa}.

\section{Introduction}

Let $(M,\pi)$ be a Poisson manifold, $x$ a point in $M$ and let $S$ be the symplectic leaf through $x$. We will use the
\textbf{Poisson homotopy bundle}\index{Poisson homotopy bundle} of the leaf $S$
\[ P_x\rmap S,\]
whose structure group $G_x$ we call the \textbf{Poisson homotopy group}\index{Poisson homotopy group} at $x$.
The bundle $P_x$ is the space of cotangent paths starting at $x$ modulo cotangent homotopy, i.e.\ it is the $s$-fiber
over $x$ of the Weinstein groupoid
\[\mathcal{G}(M,\pi)\rightrightarrows M.\]

We state now the main result of this chapter (more details and reformulations are presented in subsection \ref{Subsection_Full_statement_of_the_theorem_and_reformulations}).

\begin{mtheorem}\label{Theorem_TWO}
Let $(M, \pi)$ be a Poisson manifold and let $S$ be a compact leaf. If the Poisson homotopy bundle over $S$ is a
smooth, compact manifold with vanishing second de Rham cohomology group, then, in a neighborhood of $S$, $\pi$ is
Poisson diffeomorphic to its first order model at $S$.
\end{mtheorem}

\subsection{Some comments on the proof}

The proof uses ideas similar to the ones in \cite{CrFe-Conn}: a Moser-type argument reduces the problem to a
cohomological one (Theorem \ref{theorem-1}); a Van Est argument and averaging reduces the cohomological problem
to an integrability problem (Theorem \ref{TheoremRedInt}) which, in turn, can be reduced to the existence of special
symplectic realizations (Theorem \ref{theorem-step-2.2}); the symplectic realization is built as in subsection
\ref{Constructing symplectic realizations from transversals in the manifold of cotangent paths}, by working on the
Banach manifold of cotangent paths (subsection \ref{Step 2.3: the needed symplectic realization}).

There have been various attempts to generalize Conn's linearization theorem to arbitrary symplectic leaves. While the desired conclusion was clear (the same as in our theorem), the assumptions (except for the compactness of $S$) are more subtle. Of course, as for any (first order) local form result, one looks for assumptions on the first jet of $\pi$ along $S$. Here are a few remarks on the assumptions.\\

$1.$ {\it Compactness assumptions}. It was originally believed that such a result could follow by first applying Conn's
theorem to a transversal to $S$. The expected assumption was, next to compactness of $S$, that the isotropy Lie
algebra $\mathfrak{g}_x$ ($x\in S$) is semisimple of compact type. The failure of such a result was already pointed out
in \cite{WadeLent}. A refined conjecture with the compactness assumptions from our theorem appeared in
\cite{CrFe0}. The idea is the following: while the condition that $\mathfrak{g}_x$ is semisimple of compact type is
equivalent to the fact that all (connected) Lie groups integrating the Lie algebra $\mathfrak{g}_x$ are compact, one
should require compactness (and smoothness) of only one group associated to $\mathfrak{g}_x$- namely, of the
Poisson homotopy group $G_x$. This is an important difference because
\begin{itemize}
\item our theorem may be applied even when $\mathfrak{g}_x$ is abelian; 
\item actually, under the assumptions of the theorem, $\mathfrak{g}_x$ can be semisimple and compact only when the leaf is a point!
\end{itemize}

$2.$ {\it Vanishing of $H^2(P_x)$}. The compactness condition on the Poisson homotopy bundle is natural also when
drawing an analogy with other local normal form results like local Reeb stability or the slice theorem. However,
compactness alone is not enough (see the first example in subsection \ref{Subsection_Examples_linearization}). The
subtle condition is $H^2(P_x)=0$ and its appearance is completely new in the context of normal forms:

\begin{itemize}
\item In Conn's theorem, it is not visible (it is automatically satisfied!).
\item In the classical cases (foliations, actions) such a condition is not needed.
\end{itemize}
What happens is that the vanishing condition is related to integrability phenomena \cite{CrFe1, CrFe2}. In contrast
with the case of foliations and of group actions, Poisson manifolds give rise to Lie algebroids that may fail to be
integrable. To clarify the role of this assumption, we mention here that:

{\it It implies integrability.}  The main role of this assumption is that it forces the Poisson manifold to be (Hausdorff)
integrable around the leaf. Actually, under such an integrability assumption, the normal form is much easier to
establish, and the vanishing condition is not needed- see Proposition \ref{main-cor}, which can also be deduced from
Zung's linearization theorem \cite{Zung}. Note however that such an integrability condition refers to the germ of $\pi$
around $S$ (and not the first order jet, as desired!); and, of course, Conn's theorem does not make such an assumption.

{\it It implies vanishing of the second Poisson cohomology.} Next to integrability, the vanishing condition also implies the vanishing of the
second Poisson cohomology group $H^2_{\pi}(U)$ (of arbitrarily small neighborhoods $U$ of $S$)- which is known to be relevant to infinitesimal
deformations (see e.g. \cite{CrFe0}). We would like to point out that the use of $H^2_{\pi}(U)=0$ only simplifies our argument but is not essential.
A careful analysis shows that one only needs a certain class in $H^2_{\pi}(U)$ to vanish, and this can be shown using only integrability.
This is explained at the end of subsection \ref{Step 2.1: Reduction to integrability}, when concluding the proof of Proposition
\ref{main-cor} mentioned above.\\

\subsection{Normal form theorems}\label{subsection_normal_form_theorems}

We recall some classical normal form theorems in differential geometry. We give weaker versions of these results, so
that the assumptions depend only on the first jet of the structures involved. Our theorem is of the same nature. For
example, in the slice theorem, properness of the action at the orbit is a sufficient hypothesis, but this is a condition on
the germ of the action, and similarly, in Reeb stability, the holonomy group depends on the germ of the foliation around
the leaf.

\subsubsection*{The Slice Theorem}
\label{The Slice Theorem}

Let $G$ be a Lie group acting on a manifold $M$, $x\in M$, and let $\mathcal{O}$ be the orbit through $x$. The Slice
Theorem (see e.g.\ \cite{DK}) gives a normal form for the $G$-manifold $M$ around $\mathcal{O}$. It is built out of
the isotropy group $G_x$ at $x$ and its canonical representation $\nu_x= T_xM/T_x\mathcal{O}$. Explicitly, the local
model is:
\[ G\times_{G_x} \nu_x= (G\times \nu_x)/G_{x}\]
which is a $G$-manifold and admits $\mathcal{O}$ as the orbit corresponding to $0\in \nu_x$.

\begin{theorem}[The Slice Theorem]
If $G$ is compact, then a $G$-invariant neighborhood of $\mathcal{O}$ in $M$ is diffeomorphic, as a $G$-manifold, to
a $G$-invariant neighborhood of $\mathcal{O}$ in $G\times_{G_x} \nu_x$.\index{slice theorem}
\end{theorem}

\noindent It is instructive to think of the building blocks of the local model as a triple
\[(G_x, G\rmap \mathcal{O}, \nu_x),\]
consisting of the Lie group $G_x$, the principal $G_x$-bundle $G$ over $\mathcal{O}$ and a representation $\nu_x$ of
$G_x$. This triple represents the first order data (first jet) at $\mathcal{O}$ of the $G$-manifold $M$, while the
associated local model, represents its first order approximation.

\subsubsection*{Local Reeb Stability}

\index{Reeb stability}Let $\mathcal{F}$ be a foliation on a manifold $M$, $x\in M$, and denote by $L$ the leaf
through $x$. The Local Reeb Stability Theorem is a normal form result for the foliation around $L$. We recall here a
weaker version (see also chapter \ref{ChReeb}). Denote by $\widetilde{L}$ the universal cover of $L$, and consider
the linear holonomy representation of $\Gamma_x:=\pi_1(L,x)$ on $\nu_x= T_xM/T_xL$. The local model is
$$\widetilde{L}\times_{\Gamma_x} \nu_x=(\widetilde{L}\times \nu_x)/\Gamma_x$$
with leaves $\widetilde{L}\times_{\Gamma_x} (\Gamma_xv)$ for $v\in \nu_x$ and $L$ corresponds to $v=0$.

\begin{theorem}[Local Reeb Stability]
If $L$ is compact and $\Gamma_x$ is finite, then a saturated neighborhood of $L$ in $M$ is diffeomorphic, as a
foliated manifold, to a neighborhood of $L$ in $\widetilde{L}\times_{\Gamma_x} \nu_x$.
\end{theorem}

Again, the local model is build out of a triple $$(\Gamma_x, \widetilde{L}\rmap L, \nu_x),$$ consisting of the discrete
group $\Gamma_x$, the principal $\Gamma_x$-bundle $\widetilde{L}$ and a representation $\nu_x$ of $\Gamma_x$.
The triple should be thought of as the first order data along $L$ associated to the foliated manifold $M$, and the local
model should be thought of as its first order approximation (see also chapter \ref{ChReeb}).

\subsubsection*{Conn's Linearization Theorem}\label{Conn's theorem}

Also Conn's theorem\index{Conn's theorem} \cite{Conn} can be put in a similar setting. Let $(M,\pi)$ be a Poisson
manifold and let $x\in M$ be a fixed point of $\pi$. Let $\mathfrak{g}_x$ be the isotropy Lie algebra at $x$ and let
$G_x:=\mathcal{G}(\mathfrak{g}_x)$ be the 1-connected Lie group integrating $\mathfrak{g}_x$. The local model of
$M$ around $x$ is the linear Poisson structure corresponding to $\mathfrak{g}_x$, and, as discussed in subsection
\ref{Subsection_examples_of_symplectic_groupoids}, this can be constructed as the quotient of the symplectic
manifold (see (\ref{EQ_local_model_fixed_point_reduction}))
\[(\mathfrak{g}_x^*,\pi_{\mathfrak{g}_x})=(G_x\times\mathfrak{g}_x^*,-d\widetilde{\theta}_{\mathrm{MC}})/G_x.\]
The local data is again a triple
\[ (G_x, G_x\rmap \{x\}, \mathfrak{g}_{x}^{*} )\]
and the assumption of Conn's theorem is equivalent to compactness of $G_x$.

\subsection{The local model}\label{The local model}

In this subsection we explain the local model around symplectic leaves, which generalizes the linear Poisson structure
from the case of fixed points.\index{local model, Poisson} The construction given below is standard in symplectic
geometry and goes back to the local forms of Hamiltonian spaces around the level sets of the moment map (cf. e.g.
\cite{GS,SymplecticFibrations}) and it also shows up in the work of Montgomery \cite{Montgomery}.

The starting data is again a triple. It consists of a symplectic manifold $(S,\omega_S)$, which will be the symplectic
leaf, a principal $G$-bundle $P\to S$, which will be the Poisson homotopy bundle, and the coadjoint representation of
$G$ 
\[(G,P\rmap (S,\omega_S),\mathfrak{g}^*).\]
As before, $G$ acts diagonally on $P\times\mathfrak{g}^*$. As a manifold, the local model is:
\[ P\times_{G}\mathfrak{g}^*= (P\times \mathfrak{g}^*)/G.\]
To describe the Poisson structure, we choose a connection 1-form on $P$, $\theta\in \Omega^1(P, \mathfrak{g})$, i.e.\ $\theta$ must satisfy
\[\theta(\widetilde{X})=X, \ \ r_g^*(\theta)=Ad_{g^{-1}}\circ\theta, \ \ X\in \mathfrak{g}, \ g\in G,\]
where $\widetilde{X}$ the vector field on $P$ representing the infinitesimal action of $X$. The $G$-equivariance of
$\theta$ implies that the $1$-form $\widetilde{\theta}$ on $P\times \mathfrak{g}^*$ defined by
\[\widetilde{\theta}_{(q,\xi)}=\langle\xi,\theta_q\rangle\]
is $G$-invariant. Consider now the $G$-invariant 2-form
\[\Omega:=p^*(\omega_S)-d\widetilde{\theta}\in \Omega^2(P\times \mathfrak{g}^*).\]
Let $\Sigma\subset P\times \mathfrak{g}^*$ be the (open) set on which $\Omega$ is nondegenerate. Then $\Sigma$
contains $P\times\{0\}$ and $(\Sigma,\Omega)$ is a symplectic manifold on which $G$ acts freely and properly. The
action is Hamiltonian, with equivariant moment map
\[\mu:\Sigma\rmap \mathfrak{g}^*, \ \mu(q,\xi)=\xi.\]
Using $G$-invariance of $\widetilde{\theta}$, this follows from the computation
\begin{align*}
\iota_{\widetilde{X}-ad_X^*}\Omega=-L_{{\widetilde{X}-ad_X^*}}\widetilde{\theta}+d\iota_{\widetilde{X}-ad_X^*}\widetilde{\theta}=d\langle X, \mu\rangle.
\end{align*}

The local model is the quotient of this symplectic manifold\index{quotients of symplectic manifolds}
\[(N(P),\pi_P):=(\Sigma,\Omega)/G.\]
We regard $N(P)$ as an open in $P\times_G\mathfrak{g}^*$ and we make the identification $S\cong P\times_G \{0\}$.
Then $(N(P),\pi_P)$ contains $(S,\omega_S)$ as the symplectic leaf
\[(S,\omega_S)=(P\times \{0\},\Omega_{|P\times \{0\}})/G\subset (N(P),\pi_P).\]

We will prove in Proposition \ref{proposition_splitting_equivalent} below, that different choices of connections induce
Poisson structures that are isomorphic around $S$.

\begin{example}[Torus bundles]\label{torus-bundle}\rm
To understand the role of the bundle $P$, it is instructive to look at the case when $G= T^q$ is a $q$-torus. As a
foliated manifold, the local model is:
\[ P\times_{G} \mathfrak{g}^*= S\times \mathbb{R}^q,\]
and the symplectic leaves, are just copies of $S$:
\[ S\times \{y\}\ \ \ y\in \mathbb{R}^q.\]
To complete the description of the local model as a Poisson manifold, we need to specify the symplectic forms
$\omega_y$ on $S$- and this is where $P$ comes in. Principal $T^q$-bundles are classified by $q$ integral cohomology
classes $c_1, \ldots , c_q\in H^2(S,\mathbb{Z})$. The choice of the connection $\theta$ above corresponds to a choice
of representatives $\omega_i\in \Omega^2(S)$ for $c_i$. The symplectic structure on the leaf $S\times\{y\}$ of the
local model is given by
\[ \omega_{y}= \omega_{S}+ y_1\omega_1+ \ldots + y_q\omega_q,\ \ y= (y_1, \ldots , y_q).\]
\end{example}

\subsubsection*{Integrability}

Since $(N(P),\pi_P)$ is constructed as the quotient of the Hamiltonian space $\Sigma$, by Lemma
\ref{Lemma_symplectic_groupoid_reduction}, it is integrable by the Lie groupoid
$(\Sigma\times_{\mu}\Sigma)/G\rightrightarrows N(P)$. In fact, this groupoid is just the restriction to $N(P)$ of the
action groupoid
\[\mathcal{G}(P):=(P\times P\times\mathfrak{g}^*)/G\rightrightarrows P\times_G\mathfrak{g}^*,\]
corresponding to the representation of the gauge groupoid\index{gauge groupoid} $P\times_GP$ on
$P\times_G\mathfrak{g}^*$. Explicitly, its structure maps are given by
\[u([p,\xi])=[p,p,\xi], \ s([p_1,p_2,\xi])=[p_2,\xi], \ t([p_1,p_2,\xi])=[p_1,\xi],\] \[[p_1,p_2,\xi]^{-1}=[p_2,p_1,\xi], \ [p_1,p_2,\xi]\cdot[p_2,p_3,\xi]=[p_1,p_3,\xi].\]

Note that, if $P$ is compact, then $N(P)$ contains arbitrarily small opens of the form $P\times_GV$, with $V\subset
\mathfrak{g}^*$ a $G$-invariant open around $0$. Denote the coadjoint orbits by $O_{\xi}:=G\xi$. The symplectic
leaves of $P\times_G V$ are
\[ P\times_{G}O_{\xi} \subset P\times_{G} V, \ \ \  \xi\in V.\]
The opens $P\times_GV$ are also $\mathcal{G}(P)$-invariant, and the restriction of $\mathcal{G}(P)$ to $P\times_G
V$ is $(P\times P\times V)/G$. In particular, all its $s$-fibers are diffeomorphic to $P$. This proves the following:

\begin{proposition}\label{Lemma_small_neighborhoods}
The local model $(N(P),\pi_P)$ associated to a principal bundle $P$ over a symplectic manifold $(S,\omega_S)$ is
integrable by a Hausdorff symplectic groupoid. If $P$ is compact, then there are arbitrarily small invariant opens $U$
containing $S$, such that the $s$-fibers over points in $U$ are diffeomorphic to $P$.
\end{proposition}

\subsection{Full statement of Theorem \ref{Theorem_TWO} and reformulations}\label{Subsection_Full_statement_of_the_theorem_and_reformulations}

\subsubsection*{Smoothness of the Poisson homotopy bundle}

Recall that the Poisson homotopy bundle\index{Poisson homotopy bundle} $P_x$ of $(M,\pi)$ at $x$, is the $s$-fiber
over $x$ of the Weinstein groupoid of the cotangent Lie algebroid $T^*M$\index{cotangent Lie
algebroid}\index{Weinstein groupoid}
\[ \mathcal{G}(M,\pi)=\frac{\textrm{cotangent\ paths}}{\textrm{cotangent\ homotopy}}.\]
As such, $P_x$ is the space of cotangent paths starting at $x$ modulo cotangent homotopies. The Poisson homotopy
group $G_x$\index{Poisson homotopy group} are those cotangent paths that start and end at $x$. Composition of
cotangent paths defines a free action of $G_x$ on $P_x$, with quotient identified with the symplectic leaf $S$ through
$x$, via the target map
\[ P_x\longrightarrow S,\ \ [a]\mapsto p(a(1)).\]

Regarding the smoothness of $P_x$, one remarks that it is a quotient of the (Banach) manifold of cotangent paths of
class $C^1$. We are interested in a smooth structure which make the corresponding quotient map into a submersion.
Of course, there is at most one such smooth structure on $P_x$; when it exists, we say that $P_x$ is smooth.
Completely analogously, one makes sense of the smoothness of $G_x$. However, smoothness of $P_x$, Hausdorffness
of $P_x$, smoothness of $G_x$ and Hausdorffness of $G_x$ are all equivalent (see \cite{CrFe2}), and they are
governed by the \textbf{monodromy map}\index{monodromy map} at $x$, which is a group homomorphism (see also
section \ref{Section_Integrability})
\begin{equation}
\label{partial}
\partial_x: \pi_2(S,x)\rmap Z(G(\mathfrak{g}_x))
\end{equation}
with values in the center of the 1-connected Lie group $G(\mathfrak{g}_x)$ of $\mathfrak{g}_x$.

From \cite{CrFe1,CrFe2}, we recall:

\begin{proposition}\label{smoothness-all-in-one}
The Poisson homotopy bundle $P_x$ at $x$ is smooth if and only if the image of $\partial_x$ is a discrete subgroup of
$G(\mathfrak{g}_x)$.

In this case, $G_x$ is a Lie group with Lie algebra $\mathfrak{g}_x$ and $P_x$ is a smooth principal $G_x$-bundle
over $S$. Under the natural inclusion $\pi_1(G_x)\subset Z(G(\mathfrak{g}_x))$, the monodromy map becomes the
boundary in the homotopy long exact sequence associated to $P_x\to S$
\[\ldots \rmap \pi_2(S)\stackrel{\partial_x}{\rmap}\pi_1(G_x)\rmap \ldots\]
\end{proposition}

\subsubsection*{Complete statement of Theorem \ref{Theorem_TWO}}
The local model of $(M,\pi)$ around the leaf $S$ through $x$ is defined as follows.

\begin{definition}\label{first-order-jet-def1}
Assuming that $P_x$ is smooth, \textbf{the first order local model}\index{local model, Poisson} of $(M,\pi)$ around
the leaf $(S,\omega_S)$ through $x$ is defined as the local model (from subsection \ref{The local model}) associated to
the Poisson homotopy bundle
\[(G_x,P_x\rmap (S,\omega_S),\mathfrak{g}^*_x).\]
\end{definition}

The fact that the Poisson homotopy bundle encodes the first jet of $\pi$ along $S$, and that the first order local model
is a first order approximation of $\pi$ along $S$, is explained in subsection
\ref{Subsection_the_local_model_the_general_case}.

We can complete now the statement of Theorem \ref{Theorem_TWO}.
\begin{theoremtwo}[complete version]
Let $(M, \pi)$ be a Poisson manifold, $S$ a compact symplectic leaf and $x\in S$. If $P_x$, the Poisson homotopy bundle at $x$, is smooth,
compact, with
\begin{equation}\label{subtle-cond}
H^2(P_x; \mathbb{R})= 0,
\end{equation}
then there exists a Poisson diffeomorphism between an open neighborhood of $S$ in $(M,\pi)$ and an open
neighborhood of $S$ in the first order local model, which is the identity on $S$.
\end{theoremtwo}

\begin{remark}\rm
The open neighborhood of $S$ in $M$ can be chosen to be saturated. Indeed, by Proposition
\ref{Lemma_small_neighborhoods}, the local model has arbitrarily small saturated open neighborhoods of $S$, with all
leaves compact.
\end{remark}
Comparing with the classical results from foliation theory and group actions, the surprising condition is
(\ref{subtle-cond}). As we shall soon see, this condition is indeed necessary. However, as the next proposition shows,
this condition is not needed in the Hausdorff integrable case.

\begin{proposition}\label{main-cor}
In Theorem \ref{Theorem_TWO}, if $S$ admits a neighborhood that is Hausdorff integrable, then the assumption
(\ref{subtle-cond}) can be dropped.
\end{proposition}

Note the conceptual difference between the proposition and Theorem \ref{Theorem_TWO}: the assumptions of the
proposition are on the restriction of $\pi$ to an open around $S$, while those of Theorem 2 depend only on the first jet
of $\pi$ at $S$. As a consequence, the proof of the proposition is considerably easier (in fact, it can be derived from
Zung's results on linearization of proper groupoids \cite{Zung}).

Note also that, in contrast with the proposition, if $M$ is compact, then the conditions of Theorem \ref{Theorem_TWO}
cannot hold {\it at all} points $x\in M$, since it would follow that $\mathcal{G}(M, \pi)$ is compact and that its
symplectic form is exact (see \cite{CrFe0} and also subsection \ref{subsection_a_global_conflict} for a different
approach).

\subsubsection*{Reformulations}

Since $P_x$ is the $s$-fiber of the Weinstein groupoid, it is 1-connected. When $P_x$ is smooth, by Proposition \ref{smoothness-all-in-one}, the homotopy long exact sequence gives
\begin{equation}\label{EQ_short_exact_sequence}
1\rmap \pi_2(P_x)\rmap \pi_2(S)\stackrel{\partial_x}{\rmap}\pi_1(G_x)\rmap 1,\ \pi_1(S)\cong \pi_0(G_x),
\end{equation}
where we also used that, by Hopf's theorem, $\pi_2(G_x)=1$. Moreover, since \[\pi_1(G_x)\cong
\textrm{Im}(\partial_x)\subset Z(G(\mathfrak{g}_x)),\] the connected component of the identity satisfies
\begin{equation}\label{EQ_short_exact_sequence1}
G_x^{\circ}\cong G(\mathfrak{g}_x)/\mathrm{Im}(\partial_x).
\end{equation}

We next reformulate the conditions of Theorem \ref{Theorem_TWO}, as, in their present form, they may be difficult to
check in explicit examples.

\begin{proposition}\label{Prop_restatement_1} The conditions of Theorem \ref{Theorem_TWO} are equivalent to:
\begin{enumerate}[(a)]
\item The leaf $S$ is compact.
\item The Poisson homotopy group $G_x$ is smooth and compact.
\item The dimension of the center of $G_x$ equals to the rank of $\pi_2(S, x)$.
\end{enumerate}
\end{proposition}

\begin{proof} We already know that smoothness of $P_x$ is equivalent to that of $G_x$ while, under this
assumption, compactness of $P_x$ is clearly equivalent to that of $S$ and $G_x$. Hence, assuming (a) and (b), we still
have to show that (c) is equivalent to $H^2(P_x)=0$. Since $\mathfrak{g}_x$ is compact, it decomposes as (see
\cite{DK})
\[\mathfrak{g}_x=\mathfrak{k}\oplus\zeta,\]
where $\mathfrak{k}=[\mathfrak{g}_x,\mathfrak{g}_x]$ is semisimple of compact type and
$\zeta=Z(\mathfrak{g}_x)$ is abelian. Therefore $G(\mathfrak{g}_x)= K\times \zeta$, with $K$ compact 1-connected,
and $\partial_x$ maps to
\[Z(G(\mathfrak{g}_x))=Z\times \zeta,\]
where $Z= Z(K)$ is a finite group. Since by (\ref{EQ_short_exact_sequence}) $\pi_2(P_x)$ can be identified with
$\mathrm{ker}(\partial_x:\pi_2(S)\to Z\times\zeta)$, applying $\otimes_{\mathbb{Z}}\mathbb{R}$, we obtain an exact
sequence
\[ 0\rmap \pi_2(P_x)\otimes_{\mathbb{Z}}\mathbb{R}\rmap \pi_2(S)\otimes_{\mathbb{Z}}\mathbb{R}\stackrel{\partial_{\mathbb{R}}}{\rmap} \zeta.\]
Since $P_x$ is 1-connected, by the Hurewicz theorem, the first term is isomorphic to $H_2(P_x; \mathbb{R})$. Finally,
by (\ref{EQ_short_exact_sequence1}), $(K\times \zeta)/\mathrm{Im}(\partial_x)$ is compact, so also $(Z\times
\zeta)/\mathrm{Im}(\partial_x)$ is compact, and thus $\partial_{\mathbb{R}}$ is surjective. We obtain an exact
sequence
\[ 0\rmap H_2(P_x)\rmap \pi_2(S)\otimes_{\mathbb{Z}}\mathbb{R}\stackrel{\partial_{\mathbb{R}}}{\rmap} \zeta \rmap 0.\]
Therefore, condition (c) is equivalent to the vanishing of $H_2(P_x)$.
\end{proof}

Next, using the monodromy group\index{monodromy group}, one can also get rid of the $G_x$. Recall \cite{CrFe2,
CrFe1} that the monodromy group at $x$ is the following subgroup of $Z(\mathfrak{g}_x)$
\[\mathcal{N}_x=\{X\in Z(\mathfrak{g}_x)| \exp(X)\in \mathrm{Im}(\partial_x)\}\subset Z(\mathfrak{g}_x).\]

\begin{proposition}\label{Prop_restatement_2} The conditions of Theorem \ref{Theorem_TWO} are equivalent to:
\begin{enumerate}[(a)]
\item The leaf $S$ is compact with finite fundamental group.
\item The isotropy Lie algebra $\mathfrak{g}_x$ is of compact type.
\item $\mathcal{N}_x$ is a lattice in $Z(\mathfrak{g}_x)$.
\item The dimension of $Z(\mathfrak{g}_x)$ equals the rank of $\pi_2(S,x)$.
\end{enumerate}
\end{proposition}

\begin{proof}
Let $\widetilde{\mathcal{N}}_x$ be the image of $\partial_x$ and $\zeta=Z(\mathfrak{g}_x)$.

We show first that conditions (a), (b) and (c) are equivalent to $P_x$ being smooth and compact. Discreteness of
$\mathcal{N}_x$ (from (c)) is equivalent to that of $\widetilde{\mathcal{N}}_x$ \cite{CrFe2}, hence to smoothness of
$P_x$. Compactness of $P_x$ is equivalent to the following three conditions: $S$ being compact, $\pi_0(G_x)$ being
finite and the connected component of the identity $G_{x}^{\circ}$ being compact. Using
(\ref{EQ_short_exact_sequence}), the first two are equivalent to (a). Compactness of $G_x^{\circ}$ implies (b), which
is equivalent to the decomposition $G(\mathfrak{g}_x)=K\times\zeta$, for a 1-connected compact Lie group $K$. By
(\ref{EQ_short_exact_sequence1}), $G_x^{\circ}\cong (K\times\zeta)/\widetilde{\mathcal{N}}_x$, and we claim that
its compactness is equivalent to compactness of $\zeta/\mathcal{N}_x$ (hence to $\mathcal{N}_x$ being of maximal
rank in $\zeta$). To see this, note that $\zeta/\mathcal{N}_x$ injects naturally into $(K\times
\zeta)/\widetilde{\mathcal{N}}_x$ and that there is a surjection $K\times (\zeta/\mathcal{N}_x)\to (K\times
\zeta)/\widetilde{\mathcal{N}}_x$.

Under the compactness and smoothness assumptions, the proof from Proposition \ref{Prop_restatement_1} applies to
conclude that (d) is equivalent $H^2(P_x)=0$.
\end{proof}

\subsection{The first order data}\label{Subsection_the_first_order_data}

In this subsection we show that the Poisson homotopy bundle and its infinitesimal version, the transitive Lie algebroid
of the leaf, encode the first order data of a Poisson structure around the leaf.

\subsubsection*{The first order jet}

Let $M$ be a manifold and let $S\subset M$ be an embedded submanifold. We are interested in first jets of Poisson
structures around $S$ that have $S$ as a symplectic leaf. So, by replacing $M$ with a tubular neighborhood of $S$, we
may also assume that $S$ is closed in $M$. Consider the following objects
\begin{itemize}
\item the subalgebra of $\mathfrak{X}^{\bullet}(M)$ consisting of multivector fields tangent to $S$
\[\mathfrak{X}^{\bullet}_{S}(M):=\{W\in \mathfrak{X}^{\bullet}(M) | W_{|S}\in \mathfrak{X}^{\bullet}(S)\},\]
\item the ideal $I_S \subset C^{\infty}(M)$ of functions that vanish on $S$.
\end{itemize}
Now, $I^k_S\mathfrak{X}^{\bullet}(M)$ are ideals in $\mathfrak{X}^{\bullet}_{S}(M)$; therefore, the quotients
inherit Lie brackets such that the jet maps are graded Lie algebra homomorphisms
\[ j_{|S}^{k}: (\mathfrak{X}^{\bullet}_{S}(M),[\cdot,\cdot])\longrightarrow (J^k_S(\mathfrak{X}^{\bullet}_{S}(M)),[\cdot,\cdot]),\]
where
\[J^k_S(\mathfrak{X}^{\bullet}_{S}(M)):=\mathfrak{X}^{\bullet}_{S}(M)/I^{k+1}_S\mathfrak{X}^{\bullet}(M).\]
Of course, $J^0_S(\mathfrak{X}^{\bullet}_S(M))=\mathfrak{X}^{\bullet}(S)$ and $j^0_{|S}$ is the restriction map
$W\mapsto W_{|S}$.

\begin{definition}\label{Defintion_first_jet_Poisson_symplectic}
A first jet of a Poisson on $M$ with symplectic leaf $(S,\omega_S)$ is an element $\tau\in
J^1_S(\mathfrak{X}^{2}_S(M))$, satisfying $\tau_{|S}=\omega_S^{-1}\in \mathfrak{X}^{2}(S)$ and
$[\tau,\tau]=0$. We denote by $J^1_{(S,\omega_S)}\mathrm{Poiss}(M)$ the space of such elements.
\end{definition}
If $\pi$ is a Poisson structure on $M$ with symplectic leaf $(S,\omega_S)$, then $\tau:=j^1_{|S}\pi$ is such an
element. The fact that every such $\tau$ comes from a Poisson structure defined on a neighborhood of $S$ is surprising
and will be explained in the next subsection. On the contrary, if one only requires that $\tau_{|S}$ be Poisson, this is
no longer true (see Example \ref{Example_Not_every_X_is_Y}).

\subsubsection*{The abstract Atiyah sequence}

An \textbf{abstract Atiyah sequence}\index{Atiyah sequence} over a manifold $S$ is simply a transitive Lie algebroid
$A$ over $S$, thought of as the exact sequence of Lie algebroids:
\begin{equation}
\label{abstract-Atiyah}
0\rmap K\rmap A\stackrel{\rho}{\rmap} TS\rmap 0,
\end{equation}
where $\rho$ is the anchor map of $A$ and $K=\ker(\rho)$. As discussed in subsection
\ref{Subsection_Examples_of_Lie_algebroids}, any principal $G$-bundle $p: P\to S$ gives rise to such a sequence,
known as the \textbf{Atiyah sequence} associated to $P$:
\begin{equation}\label{concrete_Atiyah_sequence}
 0\longrightarrow P\times_{G} \mathfrak{g}^*\longrightarrow TP/G\stackrel{(dp)}{\longrightarrow} TS \rmap 0.
\end{equation}

\begin{definition}
The abstract Atiyah sequence (\ref{abstract-Atiyah}) is said to \textbf{integrable} if it is isomorphic to one given by a
principal bundle (\ref{concrete_Atiyah_sequence}).
\end{definition}

This notion was already considered in \cite{Almeida} without any reference to Lie algebroids. However, it is clear that
integrability of an abstract Atiyah sequence is equivalent to its integrability as a transitive Lie algebroid
\cite{Mackenzie, MackenzieLL,CrFe1}. In the integrable case there exists a unique (up to isomorphism) 1-connected
principal bundle integrating it (the $s$-fiber of the Weinstein groupoid).

Given $(M,\pi)$ a Poisson manifold and $(S,\omega_S)$ a symplectic leaf, one can associate to the leaf a transitive Lie
algebroid
\[A_S:=T^*M_{|S},\]
which is the restriction to $S$ of the cotangent Lie algebroid $(T^*M,[\cdot,\cdot]_{\pi},\pi^{\sharp})$. The
corresponding Atiyah sequence is
\begin{equation}
\label{pre-Atiyah} 0\rmap \nu_{S}^{*}\rmap A_S \stackrel{(\omega_S^{-1})^{\sharp}}{\rmap} TS \rmap 0,
\end{equation}
where $\nu_{S}^{*}\subset T^*M_{|S}$ is the conormal bundle of $S$,  i.e.\ the annihilator of $TS$.

Now, smoothness of the Poisson homotopy bundle $P_x$ (for $x\in S$) is equivalent to integrability of $A_S$ and, in
this case, $P_x$ is the 1-connected principal bundle of $A_S$.

The abstract Atiyah sequence also encodes the first jet of $\pi$ at $S$. For a proof of the following result, see the more
general Proposition \ref{Proposition_first_order_data}.

\begin{proposition}\label{1st-jet-Atiyah}
Consider a submanifold $S$ of $M$ endowed with a symplectic structure $\omega_{S}$. There is a 1-1 correspondence
between elements in $J^{1}_{(S, \omega_S)}\textrm{Poiss}(M)$ and Lie brackets on $A_S$ making (\ref{pre-Atiyah})
into an abstract Atiyah sequence. For $\pi$ a Poisson structure with $(S,\omega_S)$ as a symplectic leaf, under this
correspondence, $j^1_{|S}\pi$ is mapped to the restriction to $S$ of the cotangent Lie algebroid of $\pi$.
\end{proposition}

\subsection{The local model: the general case}\label{Subsection_the_local_model_the_general_case}

The local model (from Theorem \ref{Theorem_TWO}) can be described also when $P_x$ is not smooth, using its
infinitesimal counterpart: the transitive Lie algebroid $A_S$. This was explained by Vorobjev
\cite{Vorobjev,Vorobjev2}, but here we indicate a different approach using the linear Poisson structure on $A_S^*$.
The proofs of the claims made in this
subsection are given at the end of subsection \ref{subsection_linearizing_Poisson_structures}.\\

The starting data is a symplectic manifold $(S,\omega_S)$ and a transitive Lie algebroid $(A,[\cdot,\cdot]_A,\rho)$ over
$S$, with Atiyah sequence
\begin{equation}\label{EQ_Atiyah1}
0\rmap K\rmap A\stackrel{\rho}{\rmap} TS\rmap 0.
\end{equation}
Similar to the linear Poisson structure on the dual of a Lie algebra, the dual vector bundle $A^*$ carries a linear
Poisson structure $\pi_{\mathrm{lin}}(A)$, with Poisson bracket determined by
\[\{p^*(f),p^*{(g)}\}=0, \  \  \  \{\widetilde{\alpha},p^*(g)\}=p^*(L_{\rho(\alpha)}g),\ \ \ \{\widetilde{\alpha},\widetilde{\beta}\}=\widetilde{[\alpha,\beta]},\]
for all $f,g\in C^{\infty}(S)$ and $\alpha,\beta\in \Gamma(A)$, where by $\widetilde{\alpha},\widetilde{\beta}\in
C^{\infty}(A^*)$ we denote the corresponding fiberwise linear functions on $A^*$.

The following lemma will be proved at the end of subsection \ref{subsection_linearizing_Poisson_structures}.
\begin{lemma}\label{Lemma_pi_la_omega_S}
The gauge transform $\pi_{\mathrm{lin}}^{p^*(\omega_S)}(A)$ is well-defined on $A^*$.
\end{lemma}

Using a splitting $\sigma: A\to K$ of (\ref{EQ_Atiyah1}), we identify $K^*\cong\sigma^*(K^*)\subset A^*$. Let
$N(A)\subset K^*$ be the open where $K^*$ is Poisson transversal\index{Poisson transversal} for
$\pi_{\mathrm{lin}}^{p^*(\omega_S)}(A)$ (see section \ref{SPreliminaries}), and let $\pi_A$ be the corresponding
Poisson structure on $N(A)$. The following will be proved at the end of subsection
\ref{subsection_linearizing_Poisson_structures}:

\begin{proposition}\label{proposition_restricting_to_cosymplectic}
The Poisson manifold $(N(A),\pi_A)$ contains $(S,\omega_S)$ as a symplectic leaf and the transitive Lie algebroid
$A_S$ of $S$ is isomorphic to $A$, via the maps
\[A_{S}=T^*K^*_{|S}\cong T^*S\oplus K\stackrel{(\omega_S^{-1,\sharp},I)}\rmap TS\oplus K\cong A.\]
\end{proposition}

We will call the Poisson manifold $(N(A),\pi_A)$ the \textbf{local model}\index{local model, Poisson} associated to the
pair $(A,\omega_S)$. As proven by Vorobjev in \cite{Vorobjev}, the local model doesn't depend on the splitting used to
define it. We will give our own proof of this result at the end of subsection
\ref{subsection_linearizing_Poisson_structures}.

\begin{proposition}\label{proposition_splitting_equivalent}
If $(N_{\sigma_1}(A),\pi_A(\sigma_1))$ and $(N_{\sigma_2}(A),\pi_A(\sigma_2))$ are two local models, constructed
with the aid of two splittings $\sigma_1,\sigma_2:A\to K$, then there exists a Poisson diffeomorphism
\[\varphi:(U_1,\pi_{A|U_1}(\sigma_1))\rmap (U_2,\pi_{A|U_2}(\sigma_2)),\]
where $U_1$ and $U_2$ are open neighborhoods of $S$ in $N_{\sigma_1}(A)$ and $N_{\sigma_2}(A)$ respectively,
which is the identity along $S$.
\end{proposition}

We can drop now the assumption on the smoothness of the Poisson homotopy bundle in Definition
(\ref{first-order-jet-def1}) of the local model.

\begin{definition}\label{definition_local_model_nonintegrable}
The \textbf{first order local model} of $(M,\pi)$ around the symplectic leaf $(S,\omega_S)$ is the Poisson manifold
\[(N(A_S),\pi(A_S)),\ \ \textrm{ with } \ N(A_S)\subset \nu_S,\]
associated to the Lie algebroid $A_S:=T^*M_{|S}$ and to $(S,\omega_S)$.
\end{definition}

In the case when $A$ is integrable, we obtain the same local model as in subsection \ref{The local model}.

\begin{proposition}\label{Proposition_reconcile_3}
Let $P$ be a principal $G$-bundle over a symplectic manifold $(S,\omega_S)$, with Atiyah sequence
\[0\rmap K\rmap A\rmap TS\rmap 0,\]
where $K:=P\times_G\mathfrak{g}$ and $A=TP/G$. There is a bijection between
\begin{itemize}
\item $\theta\in\Omega^1(P,\mathfrak{g})$, connection 1-forms on $P$,
\item $\sigma:A\to K$, splittings of the Atiyah sequence,
\end{itemize}
such that the Poisson manifold $(N(P),\pi_P)$ constructed in subsection \ref{The local model} with the aid of $\theta$
and the Poisson manifold $(N(A),\pi_{A})$, constructed using the corresponding $\sigma$, coincide.
\end{proposition}
\begin{proof}
The first part is clear: a connection 1-form $\theta\in\Omega^1(P,\mathfrak{g})$ is the same as a $G$-invariant
splitting of the exact sequence
\[0\rmap P\times\mathfrak{g}\rmap TP\rmap p^*(TS)\rmap 0,\]
which is the same as a splitting
\[\sigma:A=TP/G\rmap K=P\times_G\mathfrak{g}\]
of the Atiyah sequence of $P$. We fix such a pair $(\theta,\sigma)$.

We prove now that the linear Poisson structure $\pi_{\mathrm{lin}}(A)$ on $A^*$ is obtained as the quotient of the
cotangent bundle of $P$,
\begin{equation}\label{EQ_diagram}
(T^*P,\omega_{\mathrm{can}})/G\cong(A^*,\pi_{\mathrm{lin}}(A)).
\end{equation}
For $Z\in \mathfrak{X}(P)$, we denote by $\widetilde{Z}\in C^{\infty}(T^*P)$ the induced fiberwise linear function
on $T^*P$. The Poisson brackets on $T^*P$ and on $A^*=(TP/G)^*$ are both uniquely determined by the relation
\[\{\widetilde{X},\widetilde{Y}\}=\widetilde{[X,Y]},\]
imposed in the first case for all vector fields $X$ and $Y$ on $P$, while in the second case only for $G$-invariant $X$
and $Y$; thus (\ref{EQ_diagram}) holds.

Note that $p^*(\omega_S)+\omega_{\mathrm{can}}$ is a symplectic structure on $T^*P$ (this does not require
$\omega_S$ to be nondegenerate). Applying the gauge transformation by $p^*(\omega_S)$ to (\ref{EQ_diagram}), we
obtain also a symplectic realization for $\plin^{p^*(\omega_S)}(A)$
\[(T^*P,p^*(\omega_S)+\omega_{\mathrm{can}})\rmap (A^*,\plin^{p^*(\omega_S)}(A)).\]
Notice that $\theta$ induces a $G$-invariant inclusion $i=\theta^*:P\times\mathfrak{g}^*\to T^*P$, which covers
$\sigma^*:K^*\to A^*$. The 1-form $\widetilde{\theta}$ constructed in subsection \ref{The local model} is just
$\widetilde{\theta}=i^*(\alpha_{\mathrm{can}})$, where $\alpha_{\mathrm{can}}$ is the tautological 1-form on
$T^*P$; therefore
\[i^*(p^*(\omega_S)+\omega_{\mathrm{can}})=p^*(\omega_S)-d\widetilde{\theta}=\Omega.\]
So we can apply Lemma \ref{Lemma_restricting_symplectic_realizations}, which says that $N_A$, the open in $K^*$
where $\plin^{p^*(\omega_S)}(A)$ is Poisson transverse, coincides with $N_P=\Sigma/G$, where $\Sigma$ is the open
where $\Omega$ is nondegenerate, and moreover, that the induced Poisson structure $\pi_A$ on $N_A$ is the push
forward of $\Omega^{-1}$; which, by construction, is $\pi_P$.
\end{proof}

\subsubsection*{The local model as a linearization}

Using Definition \ref{definition_local_model_nonintegrable} to construct the local model might be difficult, even in
some simple cases. We describe here a direct approach, but the proofs are
left for subsection \ref{subsection_linearizing_Poisson_structures}, after we develop some stronger algebraic tools.\\

Let $(M,\pi)$ be a Poisson manifold and $(S,\omega_S)$ an embedded symplectic leaf. Let $\nu_S$ be the normal
bundle of $S$ in $M$, and consider a tubular neighborhood of $S$ in $M$
\[\Psi:\nu_S\rmap M.\]
Denote by $E:=\Psi(\nu_S)$, by $p:E\to S$ the induced projection and by $\mu_t:E\to E$ the multiplication by $t$.
Consider the path of Poisson structures
\begin{equation}\label{EQ_linearization_explicit}
\pi_t:=t\mu_t^*(\pi^{(t-1)p^*(\omega_{S})}),\ t\neq 0.
\end{equation}
Now $\pi_1=\pi$, and we will show in Remark \ref{Remark_why_pite_is_smooth} that, on some open around $S$,
$\pi_t$ is defined for all $t\in(0,1]$ and that it extends smoothly at $t=0$.

\begin{definition}\label{definition_first_order_approx}
The Poisson structure
\[\pi_0:=\lim_{t\to 0}\pi_t,\]
defined on an neighborhood of $S$, is called the \textbf{first order approximation} of $\pi$ around $S$ corresponding
to the tubular neighborhood $\Psi$.
\end{definition}

The following will be proved at the end of subsection \ref{subsection_linearizing_Poisson_structures}.

\begin{proposition}\label{Proposition_reconcile}
Let $A_S=T^*M_{|S}$ be the transitive Lie algebroid of $S$
\[0\rmap \nu^*_S\rmap A_S\rmap TS\rmap 0,\]
and consider the corresponding local model $\pi_{A_S}$ on $N(A_S)\subset \nu_S$, constructed with the aid of the
splitting
\[\sigma:=d\Psi_{|S}^*:A_S\rmap \nu^*_S. \]
Then, $\pi_0$ is defined on $\Psi(N(A_S))$, and $\Psi$ is a Poisson diffeomorphism
\[\Psi:(N(A_S),\pi_{A_S})\diffto (\Psi(N(A_S)),\pi_0).\]
\end{proposition}

\subsection{Examples}\label{Subsection_Examples_linearization}

\subsubsection{The condition $H^2(P_x)=0$}

We give now an example in which all conditions of Theorem \ref{Theorem_TWO} are satisfied, except for the vanishing
of $H^2(P_x)$, and in which the conclusion fails. The Poisson structure we construct is of the following type: take a
product $S\times\mathfrak{g}^*$, with $S$ symplectic and $\mathfrak{g}^*$ the dual of a Lie algebra, and then
multiply the first Poisson tensor by a Casimir function on $\mathfrak{g}^*$. This class of Poisson structures (and their
associated local models) are discussed in \cite{WadeLent}.

Consider the unit sphere $(\mathbb{S}^2,\pi_{\mathbb{S}^2})\subset \mathbb{R}^3$, with coordinates denoted $(u,
v, w)$, endowed with the inverse of the area form
\[\pi_{\mathbb{S}^2}^{-1}= \omega_{\mathbb{S}^2}=  (udv\wedge dw+ vdw\wedge du+ wdu\wedge dv) .\]
Consider also the linear Poisson structure on $\mathfrak{so}(3)^*\cong \{(x,y,z)\in \mathbb{R}^3\}$
\[ \pi_{\mathfrak{so}(3)}=x\frac{\partial}{\partial y}\wedge\frac{\partial}{\partial z}+y\frac{\partial}{\partial z}\wedge\frac{\partial}{\partial x}+z\frac{\partial}{\partial x}\wedge\frac{\partial}{\partial y}.\]
Its symplectic leaves are the spheres $\mathbb{S}^{2}_{r}$ of radius $r>0$ and symplectic form
\[ \omega_r= \frac{1}{r^2} (x dy\wedge dz+ y dz\wedge dx+ z dx\wedge dy),\]
and the origin. Finally, consider the product of these two Poisson manifolds
\[ M= \mathbb{S}^2\times \mathbb{R}^3,  \ \pi_0= \pi_{\mathbb{S}^2}+ \pi_{\mathfrak{so}(3)}.\]
The symplectic leaves of $(M,\pi_0)$ are
\[(S,\omega_S):=(\mathbb{S}^{2}\times \{0\}, \omega_{\mathbb{S}^2})\  \textrm{and}\ (\mathbb{S}^{2}\times \mathbb{S}^{2}_{r}, \omega_{\mathbb{S}^2}+\omega_r), \ r>0.\]
The abstract Atiyah sequence of $S$ is the product of Lie algebroids
\[A_S=TS\oplus \mathfrak{so}(3).\]
Hence, for $x\in S$, the Poisson homotopy group equals
\[G_x = G(\mathfrak{so}(3))\cong \textrm{SU}(2)(\cong \mathbb{S}^3),\]
and the Poisson homotopy bundle is
\[P_x= \mathbb{S}^2\times \textrm{SU}(2){\rmap}\mathbb{S}^2.\]
Using the trivial connection on $P_x$, one finds that $(M, \pi_0)$ coincides with the resulting local model. Note that all
the conditions of Theorem \ref{Theorem_TWO} are satisfied, except for the vanishing of $H^2(P_x)$.

Let us now modify $\pi_0$ without modifying $j^{1}_{|S}\pi_0$; we consider
\[ \pi= (1+ r^2)\pi_{\mathbb{S}^2}+ \pi_{\mathfrak{so}(3)}.\]
Note that $\pi$ has the same leaves as $\pi_0$, but with different symplectic forms:
\[(S,\omega_S) \ \textrm{and} \  (\mathbb{S}^2\times \mathbb{S}^{2}_{r},\frac{1}{1+ r^2} \omega_{\mathbb{S}^2}+ \omega_r), \ r>0.\]

To show that $\pi$ and $\pi_0$ are not equivalent, we compute symplectic areas/volumes. For this, note that, under the
parametrization
\[[0,2\pi]\times [-\pi/2,\pi/2]\ni (\varphi,\theta)\mapsto r(\cos\varphi\cos\theta,\sin\varphi\cos\theta,\sin\theta)\in \mathbb{S}_r,\]
the symplectic form becomes $\omega_r=r\cos\theta d\varphi\wedge d\theta$. Therefore
\[\int_{\mathbb{S}_r}\omega_r=4\pi r.\]

We claim that $\pi$ is not Poisson diffeomorphic around $S$ to $\pi_0$. Assume it is. Then, for any $r$ small enough,
we find $r'$ and a symplectomorphism
\[ \phi: (\mathbb{S}^2\times \mathbb{S}^{2}_{r}, \frac{1}{1+ r^2} \omega_{\mathbb{S}^2}+ \omega_r)\diffto (\mathbb{S}^2\times \mathbb{S}^{2}_{r'}, \omega_{\mathbb{S}^2}+ \omega_{r'}).\]
Comparing the symplectic volumes, we find
\[\frac{r}{1+ r^2}=\frac{1}{(4\pi)^{2}}\int_{\mathbb{S}^2\times \mathbb{S}^{2}_{r}}\frac{1}{1+ r^2} \omega_{\mathbb{S}^2}\wedge\omega_r=\frac{1}{(4\pi)^{2}}\int_{\mathbb{S}^2\times \mathbb{S}^{2}_{r'}}\omega_{\mathbb{S}^2}\wedge\omega_{r'}=r'.\]
On the other hand, $\phi$ sends the first generator $\sigma_1$ of $\pi_2(\mathbb{S}^{2}\times \mathbb{S}^{2}_{r})$ into a combination $m\sigma_1+ n\sigma_2$ with
$m$ and $n$ integers. Computing the symplectic areas of these elements, we obtain:
\[\frac{1}{1+ r^2}=\frac{1}{4\pi} \int_{\sigma_1} (\frac{1}{1+r^2} \omega_{\mathbb{S}^{2}}+ \omega_r)= \frac{1}{4\pi}\int_{m\sigma_1+ n\sigma_2} (\omega_{\mathbb{S}^{2}}+ \omega_{r'})=m+nr'.\]
These two equalities imply that $mr^2+nr+m=1$. This cannot hold for all $r$ (even small enough), because it forces
$r$ to be an algebraic number.

Another characteristic which tells apart $\pi_0$ from $\pi$ is integrability. Since $\pi_0$ is the local model of $\pi$
around $S$, and $P_x$ is smooth and compact, by Proposition \ref{main-cor} and Lemma
\ref{Lemma_small_neighborhoods}, integrability of $\pi$ around $S$ is equivalent to $\pi$ and $\pi_0$ being
isomorphic around $S$.

We compute the monodromy groups\index{monodromy group} of $\pi$ at $y=(p,q)\in
\mathbb{S}^2\times\mathbb{S}^2_r$. The isotropy Lie algebra $\mathfrak{g}_y$ is one-dimensional, and the normal
bundle to the leaf at $y$ is spanned by $q$ regarded as a tangent vector at $y$. Since the Poisson structure is regular
around $y$, we can use the formula (\ref{monodromy-regular}) below to compute the monodromy
map\index{monodromy map}
\begin{align*}
\partial_y(\sigma_1)(q)&=\frac{d}{dr}\int_{\sigma_1}\left(\frac{1}{1+r^2}\omega_{\mathbb{S}^2}+\omega_r\right)=\frac{d}{dr}\left(\frac{4\pi}{1+ r^2}\right)=\frac{-8\pi r}{(1+r^2)^2},\\
\partial_y(\sigma_2)(q)&=\frac{d}{dr}\int_{\sigma_2}\left(\frac{1}{1+r^2}\omega_{\mathbb{S}^2}+\omega_r\right)=\frac{d}{dr}(4\pi r)=4\pi.
\end{align*}
So, the monodromy group at $y$ is
\[\mathcal{N}_y=\left\{\frac{2r}{(1+r^2)^2} 4\pi m + 4\pi n\ | \ m,n\in \mathbb{Z}\right\},\]
and, for $r$ a transcendental number, it is not discrete. This shows that $\pi$ is not integrable on any open
neighborhood of $S$.

\subsubsection*{The regular case}\label{second-main-example}

Let $(M,\pi)$ be a regular Poisson\index{regular Poisson structure} manifold with an embedded leaf $(S,\omega_S)$.
The local model around $S$ from this chapter coincides with the local model constructed in chapter \ref{ChReeb},
where we regard the Poisson manifold as a symplectic foliation $(M,\mathcal{F},\omega)$. To see this, assume that the
symplectic foliation is on the normal bundle $\nu_S$, and that the leaves are transverse to the fibers of $\nu_S$. Let
$\Omega$ be the extension of $\omega$ that vanishes on vertical vectors in $\nu_S$. Then, the symplectic foliation
corresponding to the path $\pi_t$ from Definition \ref{definition_first_order_approx} is
\[(\nu_S,\mu_t^*(\mathcal{F}),p^*(\omega_S)+\frac{\mu_t^*(\Omega)-p^*(\omega_S)}{t}).\]
We have that $\lim_{t\to 0}\mu_t^*(\mathcal{F})=\mathcal{F}_{\nabla}$ is the foliation corresponding to the Bott
connection\index{Bott connection}, and
\[p^*(\omega_S)+\lim_{t\to 0}\frac{\mu_t^*(\Omega)-p^*(\omega_S)}{t}=p^*(\omega_S)+\delta_S\omega=j^1_S\omega,\]
where $\delta_S\omega$ is the vertical derivative of $\omega$ at $S$. So $\pi_0$ corresponds to the local model from
chapter \ref{ChReeb}
\[(\nu_S,\mathcal{F}_{\nabla},j^1_S\omega).\]

Let us discuss the local model also in the integrable case. Let $x\in S$ and denote by $\widetilde{S}$ the universal
cover of $S$. The isotropy Lie algebra $\mathfrak{g}_x= \nu_{x}^{*}$ is abelian and, as in the general case (see
(\ref{EQ_short_exact_sequence})), the Poisson homotopy group fits into a short exact sequence
\[1\rmap G_{x}^{\circ}\rmap G_x\rmap \pi_1(S,x) \rmap 1.\]
Hence $G_{x}^{\circ}$ is abelian and the Poisson homotopy bundle is a principal $G_{x}^{\circ}$-bundle over the
universal cover  $\widetilde{S}$. In conclusion, as a foliated manifold we obtain the same local model as in Reeb
stability
\[ P_x\times_{G_x} \mathfrak{g}_x^*\cong \widetilde{S}\times_{\pi_1(S,x)} \nu_x.\]
To describe the symplectic forms on this foliation, consider
\[\widetilde{S}_{\textrm{lin}}\rmap S, \ \ \ \ \ H_{\textrm{lin}}\]
the linear holonomy cover, respectively the linear holonomy group at $x$ (see chapter \ref{ChReeb}). Recall that the
vertical derivative\index{vertical derivative} of $\omega$ can be viewed as a $H_{\textrm{lin}}$-equivariant linear
map (see (\ref{EQ_variation})), which we denote by the same symbol
\[\delta_S\omega:\nu_{x}\rmap \Omega^2_{\textrm{cl}}(\widetilde{S}_{\textrm{lin}}),\]
where $\Omega^2_{\textrm{cl}}(\widetilde{S}_{\textrm{lin}})$ is the space of closed $2$-forms on
$\widetilde{S}_{\textrm{lin}}$. By pulling back to $\widetilde{S}$, we obtain a $\pi_1(S,x)$-equivariant map, which
we also denote by
\[\delta_S\omega:\nu_{x}\rmap \Omega^2_{\textrm{cl}}(\widetilde{S}).\]
On easily sees that the pullback of the symplectic forms from the local model to $\widetilde{S}\times \nu_{x}$ are
given by the family of $2$-forms
\[(\widetilde{S}\times \{y\},p^*(\omega_S)+\delta_S\omega_y), \ \ y\in \nu_{x}.\]

To make the connection with the Poisson homotopy bundle, let us recall the simpler description of the monodromy
map\index{monodromy map} for regular Poisson structures from \cite{CrFe2}. Let $y\in\nu_x$ and let $\sigma$ be a
2-sphere in $S$, with $\sigma(N)=x$, where $N\in \mathbb{S}^2$ denotes the north pole. Consider a small variation
$\sigma_{\epsilon}$ of 2-spheres, such that
\begin{itemize}
\item $\sigma_{\epsilon}(\mathbb{S}^2)\subset S_{\epsilon}$, where $(S_{\epsilon},\omega_{\epsilon})$ is a symplectic leaf,
\item $\sigma_0=\sigma$,
\item the vector $\frac{d}{d\epsilon}_{|\epsilon=0}\sigma_{\epsilon}(N)$ represents $y$.
\end{itemize}
Then the monodromy map on $\sigma$ is:
\begin{equation}\label{monodromy-regular}
 \partial(\sigma)(y)= \frac{d}{d\epsilon} _{|_{\epsilon= 0}}\int_{\sigma_{\epsilon}} \omega_{\epsilon}.
\end{equation}
Now, since the projections $\widetilde{S}\to \widetilde{S}_{\textrm{lin}}\to S$ induce isomorphisms between
\[\pi_2(\widetilde{S},x)\diffto \pi_2(\widetilde{S}_{\textrm{lin}},x)\diffto \pi_2(S,x),\]
we can lift the monodromy map to a map
\[\pi_2(\widetilde{S},x)\cong \pi_2(\widetilde{S}_{\textrm{lin}},x)\stackrel{\partial}{\rmap} \nu_x^*.\]
This map is just integration of the vertical derivative \cite{CrFe2}
\[\partial(\sigma)(y)=\int_{\sigma}\delta_S\omega_y.\]
Choosing linear coordinates $y=(y_1,\ldots,y_q)$ on $\nu_x$, the 2-forms $j^1_S\omega_{y}$ become
\[ j^1_S\omega_{y}= p^*(\omega_S)+ y_1\omega_1+ \ldots +y_q\omega_q,\]
where $\omega_i\in \Omega^2(\widetilde{S})$ are closed 2-forms representing the components of the monodromy map
$\partial: \pi_2(\widetilde{S},x)\to \nu_{x}^{*}$,
\begin{equation}\label{EQ_components_monodromy}
\partial(\sigma)=(\int_{\sigma}\omega_1,\ldots,\int_{\sigma}\omega_q),\ [\sigma]\in \pi_2(\widetilde{S},x).
\end{equation}
Since $\widetilde{S}$ is simply connected, the Hurewicz theorem gives
\[\pi_2(\widetilde{S},x)\cong H_2(\widetilde{S},\mathbb{Z});\]
therefore (\ref{EQ_components_monodromy}) determines uniquely the cohomology classes $[\omega_i]\in
H^2(\widetilde{S})$. Of course, this is related to the fact that the cohomological variation $[\delta_S\omega]$ of
$\omega$ is independent
of the choices made (see subsection \ref{subsection_model_sympl_foli}).\\

Recall the two conditions of Theorem \ref{Theorem_ONE}, the normal form result for symplectic foliations:
\begin{itemize}
\item $S$ is a manifold of finite type,
\item the cohomological variation of $\omega$ is a surjective map
\[[\delta_S\omega]:\nu_x\rmap H^2(\widetilde{S}_{\textrm{lin}}).\]
\end{itemize}
We will discuss now the relation between Theorem \ref{Theorem_TWO}, in the case of regular Poisson structures, and
Theorem \ref{Theorem_ONE}. First, we restate the conditions of Theorem \ref{Theorem_TWO} in a similar fashion.

\begin{proposition}\label{proposition_cohomological_variation}
The cohomological variation of $\omega$, viewed as a map
\begin{equation}\label{EQ_variation_1}
[\delta_S\omega]:\nu_x\rmap H^2(\widetilde{S}),
\end{equation}
satisfies:
\begin{enumerate}[(a)]
\item it is surjective if and only if $P_x$, the Poisson homotopy bundle at $x$, is smooth with $H^2(P_x)=0$,
\item it is injective if and only if $G_x^{\circ}$, the connected component of the Poisson homotopy group at $x$, is compact.
\end{enumerate}
\end{proposition}
\begin{proof}
Under the identification $\pi_2(S,x)\cong \pi_2(\widetilde{S},x)\cong H_2(\widetilde{S},\mathbb{Z})$, the discussion
above implies that the dual of $[\delta_{S}\omega]$ is the map
\[\partial_{\mathbb{R}}:\pi_2(S,x)\otimes_{\mathbb{Z}}\mathbb{R}\rmap \nu_x^*.\]
So, surjectivity of $[\delta_S\omega]$ is equivalent to $\ker(\partial)\otimes_{\mathbb{Z}}\mathbb{R}=0$. This
condition implies discreteness of $\mathcal{N}_x$, the image of $\partial$, which is equivalent to smoothness of
$P_x$. Under the smoothness assumption, by (\ref{EQ_short_exact_sequence}) and the fact that $\pi_1(P_x)=1$, we
have that
\[\ker(\partial)\otimes_{\mathbb{Z}}\mathbb{R}=\pi_2(P_x)\otimes_{\mathbb{Z}}\mathbb{R}=H_2(P_x,\mathbb{Z})\otimes_{\mathbb{Z}}\mathbb{R}=H_2(P_x).\]
For the second part, surjectivity of $\partial_{\mathbb{R}}$ is equivalent to the fact that $\mathcal{N}_x$ spans
$\nu_x^*$. By (\ref{EQ_short_exact_sequence1}), $G_x^{\circ}\cong \nu_x^*/\mathcal{N}_x$ as topological groups,
and this holds also in the non-integrable case (see \cite{CrFe2}). It is easy to see that compactness of
$\nu_x^*/\mathcal{N}_x$ is equivalent to the fact that $\mathcal{N}_x$ spans $\nu_x^*$. 
\end{proof}

We conclude:
\begin{corollary}\label{corollary_conditions_regular}
For regular Poisson structures, the conditions of Theorem \ref{Theorem_TWO} are equivalent to compactness of
$\widetilde{S}$ and to (\ref{EQ_variation_1}) being a linear isomorphism.
\end{corollary}

The following relates the assumptions of Theorems \ref{Theorem_ONE} and \ref{Theorem_TWO}.
\begin{lemma}\label{lemma_surj_implies_surj}
Let $(M,\mathcal{F},\omega)$ be a symplectic foliation and let $S$ be a symplectic leaf. If $p:\widetilde{S}\to
\widetilde{S}_{\mathrm{inf}}$ is a finite cover, and the map
\[[\delta_S\omega]:\nu_x\rmap H^2(\widetilde{S})\]
is surjective, then also
\[[\delta_S\omega]:\nu_x\rmap H^2(\widetilde{S}_{\mathrm{lin}})\]
is surjective.
\end{lemma}
\begin{proof}
By assumption, we have that the composition
\[\nu_x\stackrel{[\delta_S\omega]}{\rmap}H^2(\widetilde{S}_{\textrm{lin}})\stackrel{p^*}{\rmap}H^2(\widetilde{S}) \]
is surjective, and we need to show that the first map is surjective. For this it suffices to prove that $p^*$ is injective.
Now $\widetilde{S}\to \widetilde{S}_{\textrm{lin}}$ is principal $K_{\mathrm{lin}}$-bundle, where
$K_{\mathrm{lin}}$ is the kernel of the map
\[\pi_1(S,x)\rmap H_{\mathrm{inf}}\rmap 0.\]
Let $\eta\in \Omega^{\bullet}_{\mathrm{cl}}(\widetilde{S}_{\textrm{lin}})$ be a form such that
$p^*(\eta)=d\alpha$. Since $K_{\mathrm{lin}}$ is finite, by averaging $\alpha$, we may assume that $\alpha$ is
$K_{\mathrm{lin}}$-invariant, thus $\alpha=p^*(\beta)$. This implies injectivity of $p^*$ in cohomology, and
concludes the proof.
\end{proof}

The corollary and the lemma above together, show that Theorem \ref{Theorem_ONE} implies Theorem
\ref{Theorem_TWO} for regular Poisson structures. Nevertheless, this is not completely true, since Theorem
\ref{Theorem_ONE} is not a \emph{first order normal form theorem}, whereas Theorem
\ref{Theorem_TWO} is. Here are a few remarks which clarify this:\\

$1.$ The condition from Theorem \ref{Theorem_ONE}, that the holonomy group is finite, is not a first order condition;
it depends on the germ of the foliation around $S$. The first jet of the foliation sees only the linear holonomy group,
whose finiteness is not sufficient for linearization. The corresponding first order condition is, as in Corollary
\ref{corollary_conditions_regular}, that the fundamental group is finite. This corresponds also to the weaker version of
Reeb stability from subsection \ref{subsection_normal_form_theorems}.

$2.$ Theorem \ref{Theorem_ONE} has a more subtle condition, which is not a first order condition in the Poisson
world, namely that the Poisson structure is regular around the leaf. Now, chapter \ref{ChReeb} is entirely about
symplectic foliations, so this condition is automatically satisfied, but if we are dealing with general Poisson structures,
then this is not detectable from the first jet. For example, the Poisson structure $\pi=x^2\frac{\partial}{\partial
x}\wedge \frac{\partial}{\partial y}$ on $\mathbb{R}^2$ has $0$ as a fixed point, and its first order approximation at
$0$ is regular (it is the trivial Poisson structure), but $\pi$ is not regular around $0$. Now, Theorem
\ref{Theorem_TWO} implies the following statement, which is not a consequence of Theorem \ref{Theorem_ONE}:
\begin{corollary}
If $\pi$ is a Poisson structure whose local model around a leaf $S$ is regular, and it satisfies the conditions from
Corollary \ref{corollary_conditions_regular}, then, around $S$, $\pi$ is regular and isomorphic to its local model.
\end{corollary}

\subsubsection*{On the condition $H^2(P_x)=0$ and integrability in the regular case}

We give now an example which explains the importance of the condition
\[H^2(P_x)=0\]
for symplectic foliations. To simplify the discussion, we will assume that
\[S\textrm{ is compact and simply connected}.\]
Then, by Reeb stability, $M= S\times \mathbb{R}^q$ with the trivial foliation, and the Poisson structure on $M$ is
determined by a family $\{\omega_y\in \Omega^2(S)\}_{y\in\mathbb{R}^q}$ of symplectic forms. We look at the leaf
$S\cong S\times\{0\}$ and denote $\omega_S:=\omega_0$. The vertical derivative of $\omega$ is simply
\[\delta_S\omega_y:=y_1\omega_1+ \ldots +y_q\omega_q, \ \ \textrm{where}\ \omega_k=\frac{\partial}{\partial y_k}\omega_{|y=0},\]
and so the local model is
\begin{equation}\label{EQ_linear_regular_simplyconnected}
j^1_S\omega_y=\omega_S+y_1\omega_1+ \ldots +y_q\omega_q.
\end{equation}
Let's assume that the cohomological variation is not surjective
\begin{equation}\label{EQ_cohomological_variation}
[\delta_S\omega]:\mathbb{R}^q\rmap H^2(S),
\end{equation}
which by Proposition \ref{proposition_cohomological_variation} means that either $P_x$ is not smooth or that
$H^2(P_x)\neq 0$. Then we can find a closed $2$-form $\lambda$ on $S$, with $[\lambda]\in H^2(S)$ not in the linear
span of $[\omega_1], \ldots, [\omega_q]$. Consider the Poisson structure corresponding to the family of $2$-forms
\begin{equation}\label{EQ_forms}
\omega_y:=\omega_S+y_1\omega_1+\ldots+y_q\omega_q+ y_{1}^{2} \lambda.
\end{equation}
We claim that it is not isomorphic to (\ref{EQ_linear_regular_simplyconnected}), its linearization around $S$, by a
diffeomorphism which fixes $S$. Otherwise, we could find a diffeomorphism of the form $(x,y)\mapsto(\phi_y(x),
\tau(y))$ with $\tau(0)= 0$, $\phi_0(x)= x$ and such that
\[ \phi_{y}^{*}\omega_S+ \sum \tau_i(y)\phi_{y}^{*}\omega_i= \omega_S+ \sum y_i\omega_i+ y_{1}^{2}\lambda.\]
Since $\phi_{y}^{*}$ is the identity in cohomology ($\phi_y$ is isotopic to $\phi_0$), we get, for $y_1\neq 0$ a
contradiction:
\[[\lambda]=\frac{1}{y_1^2} \sum (\tau_i(y)-y_i)[\omega_i].\]

Related to Proposition \ref{main-cor}, let us now assume that $P_x$ is smooth, but the map
(\ref{EQ_cohomological_variation}) is still not surjective. This implies that the monodromy group\index{monodromy
group} $\mathcal{N}_x\subset\mathbb{R}^q$ is discrete, and so, by a linear change of coordinates on
$\mathbb{R}^q$, we may assume that
\[\mathcal{N}_x=\mathbb{Z}^p\times\{0\} \subset \mathbb{R}^q.\]
As a result, using also that $\pi_2(S,x)\cong H_2(S,\mathbb{Z})$, we see that
\[[\omega_1],\ldots,[\omega_p]\in H^2(S,\mathbb{Z}) \textrm{ are linearly independent,}\]
and that
\[[\omega_{p+1}]=\ldots=[\omega_q]=0.\]
As a side remark, by Proposition \ref{proposition_cohomological_variation} (b), $p=q$ is equivalent to compactness of
$P_x$. Let us choose also $\lambda$ such that $[\lambda]\in H^2(S;\mathbb{Z})$ and, as before, which it is not in the
$\mathbb{R}$-span of the $[\omega_i]$'s.

We claim that the Poisson structure corresponding to the family of forms (\ref{EQ_forms}) is not integrable on any
open neighborhood of $S$ (remark that this is not a direct consequence of Proposition \ref{main-cor}, since we are not
assuming that $P_x$ is compact). This follows by computing the monodromy groups using (\ref{monodromy-regular})
at $(x,y)$
\[\mathcal{N}_{y}=\left\{(\int_{\sigma}(\omega_1+2y_1\lambda),\int_{\sigma}\omega_2,\ldots,\int_{\sigma}\omega_q) | [\sigma]\in \pi_2(S,x)\right\}.\]
The conditions on $[\omega_i]$, $[\lambda]$ imply existence of $[\sigma_1]$, $[\sigma]\in\pi_2(S,x)$, such that
\[\int_{\sigma_1}\omega_1=\int_{\sigma}\lambda=C, \ \ \  \int_{\sigma_1}\lambda=\int_{\sigma}\omega_i=\int_{\sigma_1}\omega_j=0,\  j\neq 1,\]
with $C$ a nonzero integer. Then $\mathcal{N}_y$ contains
\[\left\{(nC+2y_1m C, 0,\ldots,0) | m,n\in\mathbb{Z}\right\},\]
thus it is not discrete for $y_1\notin\mathbb{Q}$.

\subsubsection*{Duistermaat-Heckman variation formula}\index{Duistermaat-Heckman formula}

Next, we indicate the relationship of our results with the theorem of Duistermaat and Heckman on the linear variation
in cohomology of the reduced symplectic forms \cite{DH}. We first recall (a simplified version of) this result.

Let $(\Sigma, \Omega)$ be a symplectic manifold endowed with a Hamiltonian action of a torus $T$ and with proper
moment map $\mu: \Sigma\to \mathfrak{t}^*$. Let $\xi_0\in \mathfrak{t}^*$ be a regular value of $\mu$. For
simplicity, we will assume that the action of $T$ on $\mu^{-1}(\xi_0)$ is free. Then, there exists $U$, a ball around
$\xi_0$ consisting of regular values of $\mu$, such that $T$ acts freely on $\mu^{-1}(U)$. We replace $\Sigma$ with
$\mu^{-1}(U)$.

The symplectic quotients
\[ S_{\xi}:= \mu^{-1}(\xi)/T,\ \ \xi\in U \]
come with symplectic forms denoted $\omega_{\xi}$. There are canonical isomorphisms
\[ H^2(S_{\xi})\cong H^2(S_{\xi_0}),\ \ \  \textrm{for}\ \xi\in U,\]
and the Duistermaat-Heckman theorem asserts that, in cohomology,
\begin{equation}\label{EQ_DH}
 [\omega_{\xi}]= [\omega_{\xi_0}]+ \langle c, \xi- \xi_0\rangle,
\end{equation}
where $c$ is the Chern class of the $T$-bundle $\mu^{-1}(\xi_0)\to S_{\xi_0}$.

This is related to our theorem applied to the regular Poisson manifold
\[(M,\mathcal{F},\omega):=(\Sigma,\Omega)/T,\]
with symplectic leaves the $(S_{\xi},\omega_{\xi})$'s. By Lemma \ref{Lemma_symplectic_groupoid_reduction}, $M$
is integrable by the symplectic groupoid
\[\mathcal{G}:=(\Sigma\times_\mu \Sigma)/T\rightrightarrows M,\]
with symplectic structure induced by $pr_1^*(\Omega)-pr_2^*(\Omega)\in \Omega^2(\Sigma\times \Sigma)$. The
isotropy groups of $\mathcal{G}$ are all isomorphic to $T$ and the $s$-fibers are isomorphic (as principal
$T$-bundles) to the fibers of $\mu$. Taking $U$ small enough and using a $T$-invariant Ehresmann connection, on
proves that all fibers of $\mu$ are diffeomorphic as $T$-bundles (see \cite{DH} for details). So, if we are assuming that
$\mu^{-1}(\xi_0)$ is 1-connected, then $\mathcal{G}$ is the 1-connected symplectic groupoid integrating $M$. In
particular, the Poisson homotopy bundle corresponding to $S_{\xi_0}$ is the $T$-bundle $\mu^{-1}(\xi_0)\to
S_{\xi_0}$, and the Chern class $c$ is the cohomological variation $[\delta_S\omega]$ (see also Example
\ref{torus-bundle}). In this case the condition $H^2(\mu^{-1}(\xi_0))=0$ is not required, since we can apply directly
Proposition \ref{main-cor}, to conclude that the local model holds around $S_{\xi_0}$,
\[\omega_{\xi}=\omega_{\xi_0}+\delta_S\omega_{\xi-\xi_0},\]
which implies (\ref{EQ_DH}).

\subsubsection*{Linear Poisson structures}\index{linear Poisson structure}

Consider $(\mathfrak{g}^*,\pi_{\mathfrak{g}})$, the linear Poisson structure on the dual of a Lie algebra
$\mathfrak{g}$. As discussed in subsection \ref{Subsection_examples_of_symplectic_groupoids}, this Poisson
structure is integrable and its 1-connected symplectic groupoid is
\[\mathcal{G}(\mathfrak{g}^*,\pi_{\mathfrak{g}})= (G\ltimes \mathfrak{g}^*,\omega_{\mathfrak{g}})\rightrightarrows \mathfrak{g}^*,\]
where $G$ is the 1-connected group of $\mathfrak{g}$. Thus, for every $\xi\in \mathfrak{g}^{*}$, the Poisson
homotopy bundle is the principal $G_{\xi}$-bundle
\[G\rmap O_{\xi},\]
where $G_{\xi}\subset G$ is the stabilizer of $\xi$ and $O_{\xi}$ is the coadjoint orbit (and also the symplectic leaf)
through $\xi$. Since $H^2(G)=0$, the hypothesis of Theorem \ref{Theorem_TWO} reduces to $G$ being compact, or
equivalently, to $\mathfrak{g}$ being semisimple of compact type. Note also that the resulting local form around
$O_{\xi}$ implies the linearizability of the transversal Poisson structure \cite{Wein} to $O_{\xi}$, which fails
for general Lie algebras \cite{Wein}-Errata.\\

Of course, one may wonder about a direct argument. This is possible, and actually one needs slightly weaker
conditions:
\[O_{\xi} \textrm{ is \textbf{embedded} and } \xi \textrm{ is \textbf{split}.} \]
The split condition \cite{SymplecticFibrations} means that there is a $G_{\xi}$-equivariant projection
\[\sigma:\mathfrak{g}\rmap \mathfrak{g}_{\xi}.\]
The proof of this can be found in \cite{SymplecticFibrations,Montgomery}. We give here a direct argument. Observe
that the splitting induces a $G$-invariant principal connection on the Poisson homotopy bundle $G\to O_{\xi}$:
\[\theta\in\Omega^1(G;\mathfrak{g}_{\xi}), \quad \theta_g=l_{g^{-1}}^*(\sigma).\]
By (\ref{EQ_pull_back_coadjoint}), we have that $p^*(\omega_{\xi})=-d\xi^l$, where $\xi^l\in\Omega^1(G)$ is the left
invariant extension of $\xi$, so the 2-form $\Omega\in\Omega^2(G\times\mathfrak{g}^*_{\xi})$ used to define the
local model (subsection \ref{The local model}) is
\[\Omega=-d\xi^l-d\widetilde{\theta}.\]
Now, $\Omega$ is not just (right) $G_{\xi}$-invariant, but also (left) $G$-invariant, therefore its nondegeneracy locus
is of the form $G\times U$, where $U\subset \mathfrak{g}_{\xi}^*$ is an open $G_{\xi}$-invariant neighborhood of
$0$; and the local model is of the form
\begin{equation}\label{EQ_local_model_linear}
(G\times U,\Omega)/G_{\xi}\cong (G\times_{G_{\xi}} U,\pi_{\xi}).
\end{equation}

The projection $\sigma$ also gives a $G$-invariant tubular neighborhood of $O_{\xi}$
\[\Psi:G\times_{G_{\xi}}\mathfrak{g}_{\xi}^{*}\rmap \mathfrak{g}^{*},\quad [g,\eta]\mapsto Ad_{g^{-1}}^*(\xi+\sigma^{*}(\eta)),\]
which is a diffeomorphism on some open around $O_{\xi}$. As a side remark, the open where the differential of $\Psi$
is invertible coincides with $G\times_{G_{\xi}}U$ (see Theorem 2.3.7 \cite{SymplecticFibrations}). We show now that
the local model holds around $O_{\xi}$, by proving that $\Psi$ is a Poisson map (Theorem 1, section 1.3
\cite{Montgomery})
\[\Psi:(G\times_{G_{\xi}}U,\pi_{\xi})\rmap (\mathfrak{g}^{*},\pi_{\mathfrak{g}}).\]
Notice that $\sigma$ induces also a map into the symplectic groupoid of $\pi_{\mathfrak{g}}$
\[\psi:G\times\mathfrak{g}_{\xi}^{*}\rmap G\times \mathfrak{g}^{*},\ \ (g,\eta)\mapsto (g,\xi+\sigma^{*}(\eta)),\]
which satisfies $t\circ \psi(g,\xi)=\Psi([g,\xi])$, for $t$ the target map; and $\Omega$ is the pull-back of the symplectic
structure on $G\ltimes\mathfrak{g}^*$
\[\Omega=-d\xi^l-d\widetilde{\theta}=-d\langle\xi+\sigma^*,\theta_{\mathrm{MC}}\rangle=\psi^*(-d\widetilde{\theta}_{\mathrm{MC}}).\]
Therefore the symplectic realization (\ref{EQ_local_model_linear}) is just the restriction of the symplectic realization
(\ref{EQ_local_model_fixed_point_reduction}) of $\pi_{\mathfrak{g}}$, thus Lemma
\ref{Lemma_restricting_symplectic_realizations} implies that $\Psi$ is Poisson.

\section{Poisson structures around a symplectic leaf: the algebraic framework}\label{Poisson structures around a symplectic leaf: the algebraic framework}

A Poisson structure on a tubular neighborhood of a symplectic leaf can be described by a so-called Vorobjev triple
\cite{Vorobjev,Vorobjev2}: a vertical Poisson structure, an Ehresmann connection and a horizontal 2-form. This triple
encodes the behavior of the Poisson tensor around the leaf; in particular the linearized structure corresponds to (and
can be defined by) the linearization of the components of the triple. In this section we present an improvement of the
algebraic framework from \cite{CrFe-stab}, which is used the handle theses triples. In particular, using the algebraic
tools we develop, we reprove several results from \cite{Vorobjev,Vorobjev2}. Also, we explain why the first order
approximation from Definition \ref{definition_first_order_approx} deserves this name.

\subsection{The graded Lie algebra $(\widetilde{\Omega}_E,[\cdot,\cdot]_{\ltimes})$}\label{The graded Lie algebra}

Since we are interested in the local behavior of Poisson structures around an embedded symplectic leaf, we may
restrict our attention to a tubular neighborhood. Throughout this section $p: E\to S$ will be a vector bundle over a
manifold $S$. Actually, we don't use the linear structure of $E$ until subsection \ref{The dilation operators and jets
along $S$}, and the whole discussion works for any surjective submersion (except for Lemma \ref{OmegaS-Central},
where one has to assume that $p$ has
connected fibers).\\

We will use the following notations: given any vector bundle $F\to S$, denote the space of $F$-valued forms on $S$ by:
\[ \Omega^{\bullet}(S, F):= \Gamma(\Lambda^{\bullet}T^*S\otimes F)= \Omega^{\bullet}(S)\otimes_{C^{\infty}(S)}\Gamma(F).\]
More generally, for any $C^{\infty}(S)$-module $\mathfrak{X}$, denote by
\[\Omega^{\bullet}(S, \mathfrak{X}):=\Omega^{\bullet}(S)\otimes_{C^{\infty}(S)} \mathfrak{X},\]
the space of antisymmetric forms on $S$ with values in $\mathfrak{X}$.

Consider the vertical subbundle of $TE$
\[ V:=\ker(dp)\subset TE,\]
and the subalgebra of $(\mathfrak{X}^{\bullet}(E),[\cdot,\cdot])$ of vertical multivector fields
\[\mathfrak{X}^{\bullet}_{\mathrm{V}}(E)=\Gamma(\Lambda^{\bullet} V) \subset \mathfrak{X}^{\bullet}(E).\]
We recall also the grading:
\[ \textrm{deg}(X):= |X|-1= q- 1\ \ \ \textrm{for}\ X\in \mathfrak{X}^{q}(E).\]
The Dirac structure corresponding to $V$ is
\[L_V:=V\oplus V^{\circ}\subset TE\oplus T^*E,\]
where $V^{\circ}$ is the annihilator of $V$, which can also be identified with the bundle $p^*(T^*S)$. Now,
$\Gamma(V\oplus V^{\circ})$ carries the Dorfman bracket (see subsection \ref{subsection_Dirac}), which here will be
denoted simply by $[\cdot,\cdot]:=[\cdot,\cdot]_D$. Exactly like the Schouten bracket (see
(\ref{EQ_Schouten_bracket})), this bracket can be extended to
\[\Omega_E:=\Gamma(\Lambda^{\bullet}(V\oplus V^{\circ})).\]
We view $\Omega_E$ as a bigraded space, whose elements of bidegree $(p, q)$ are
\[\Omega_E^{p,q}=\Gamma(\Lambda^qV\otimes \Lambda^p V^{\circ}).\]
Then $(\Omega_E,[\cdot,\cdot])$ becomes a graded Lie algebra, with degree
\[ \textrm{deg}(\varphi\otimes X):=|\varphi|+|X|-1= p+q- 1\ \ \ \textrm{for}\ \varphi\otimes X\in\Omega_E^{p,q}.\]
Also, one can think of $\Omega_E$ as the space of forms on $S$, with values in the Lie algebra
$\mathfrak{X}^{\bullet}_{\mathrm{V}}(E)$
\[\Omega^{p,q}_{E}=\Omega^{p}(S,\mathfrak{X}^{q}_{\mathrm{V}}(E))=\Omega^p(S)\otimes_{C^{\infty}(S)}\mathfrak{X}^{q}_{\mathrm{V}}(E),\]
and as such, the Lie bracket is also given by (see \cite{CrFe-stab})
\begin{align*}
[\alpha,\beta]&(X_1,\ldots,X_{p+p'})=\\
&=\sum_{\sigma}(-1)^{|\sigma|+p'(q-1)}[\alpha(X_{\sigma_1},\ldots,X_{\sigma_p}),\beta(X_{\sigma_{p+1}},\ldots,X_{\sigma_{p+p'}})],
\end{align*}
for $X_i\in\mathfrak{X}(S)$, $\alpha\in \Omega_E^{p,q}$, $\beta\in \Omega_E^{p',q'}$, where the sum is over all
$(p,p')$-shuffles $\sigma$. For decomposable elements
\[\varphi\otimes X, \ \  \psi\otimes Y, \ \ \varphi,\psi\in \Omega^{\bullet}(S), \ X,Y\in \mathfrak{X}^{\bullet}_{\mathrm{V}}(E),\]
this formula reduces to
\[[\varphi\otimes X,\psi\otimes Y]=(-1)^{|\psi|(|X|-1)}\varphi\wedge\psi\otimes[X,Y].\]

We will need an extension $\widetilde{\Omega}_E$ of $\Omega_E$. Consider algebra
\[(\mathfrak{X}_{\mathrm{P}}(E),[\cdot,\cdot]),\]
of \textbf{projectable vector fields}\index{projectable vector fields} on $E$, i.e.\ vector fields $X\in \mathfrak{X}(E)$
with the property that there is a vector field on $S$, denoted by $p_S(X)\in \mathfrak{X}(S)$, such that
$dp(X)=p_S(X)$. The Schouten bracket of a projectable vector field and of a vertical multivector field is a vertical
multivector field, therefore
\[\mathfrak{X}^{\bullet}_{\mathrm{P},\mathrm{V}}(E):=\mathfrak{X}_{\mathrm{P}}(E)+\mathfrak{X}^{\bullet}_{\mathrm{V}}(E)\]
is a subalgebra of $\mathfrak{X}^{\bullet}(E)$, which fits in the short exact sequence
\[0\rmap (\mathfrak{X}^{\bullet}_{\mathrm{V}}(E),[\cdot,\cdot])\rmap (\mathfrak{X}^{\bullet}_{\mathrm{P},\mathrm{V}}(E),[\cdot,\cdot])\rmap (\mathfrak{X}(S),[\cdot,\cdot])\rmap 0.\]

Consider the following bigraded vector space,
\[ \widetilde{\Omega}_{E}:= \Omega^{\bullet}(S, \mathfrak{X}^{\bullet}_{\mathrm{P},\mathrm{V}}(E))=\Omega_E +\Omega^{\bullet}(S)\otimes_{C^{\infty}(S)}\mathfrak{X}_{\mathrm{P}}(E) \subset \Omega^{\bullet}(E, \Lambda^{\bullet} TE), \]
which in bidegree $(p, q)$ is given by
\begin{eqnarray*}
\nonumber \widetilde{\Omega}_{E}^{p,q}= \left\{
\begin{array}{rl}
\Omega^{p}(S, \mathfrak{X}^{q}_{\mathrm{V}}(E)) & \text{if } q\neq 1\\
\Omega^{p}(S, \mathfrak{X}_{\mathrm{P}}(E)) & \text{if } q = 1\\
\end{array} \right..
\end{eqnarray*}

Also the space $\widetilde{\Omega}_E$ fits in a short exact sequence of vector spaces:
\[ 0\rmap \Omega_E\rmap \widetilde{\Omega}_E\stackrel{p_S}{\rmap} \Omega^{\bullet}(S, TS) \rmap 0.\]
Next, we show that this is naturally a sequence of graded Lie algebras. On $\Omega^{\bullet}(S, TS)$ we consider the
\textbf{Fr\"{o}hlicher-Nijenhuis-Bracket}\index{Fr\"{o}hlicher-Nijenhuis bracket}, denoted $[\cdot, \cdot]_{F}$,
which we recall using section 13 of \cite{Kumpera-Spencer}. The key-point is that $\Omega^{\bullet}(S, TS)$ can be
identified with the space of derivations of the graded algebra $(\Omega^{\bullet}(S),\wedge)$, which commute with
the de Rham differential, and, as a space of derivations, it inherits a natural Lie bracket. In more detail, for $u=
\alpha\otimes X\in \Omega^{\bullet}(S, TS)$, the operator $L_u:= [i_u, d]$ on $\Omega^{\bullet}(S)$ is given by:
\[ L_{u}(\omega)= \alpha\wedge L_{X}(\omega)+ (-1)^{|\alpha|} d\alpha \wedge i_X(\omega) .\]
The resulting commutator bracket on $\Omega^{\bullet}(S, TS)$ is:
\[ [u, v]_{F}= L_{u}(\beta)\otimes Y-(-1)^{|\alpha||\beta|}L_{v}(\alpha)\otimes X+ \alpha\wedge\beta\otimes[X,Y],\]
for $u= \alpha\otimes X, v= \beta\otimes Y\in \Omega^{\bullet}(S, TS)$. With these $(\Omega^{\bullet}(S,
TS),[\cdot,\cdot]_F)$ is a graded Lie algebra, with grading $\textrm{deg}= r$ on $\Omega^r(S, TS)$.

Denote the element corresponding to the identity map of $TS$ by
\[ \gamma_S=\textrm{Id}_{TS}\in \Omega^1(S, TS).\]
This  element is central in $(\Omega^{\bullet}(S, TS),[\cdot,\cdot]_{F})$ and satisfies
\begin{equation} \label{gamma-represents-d}
L_{\gamma_S}=d:\Omega^{\bullet}(S)\rmap \Omega^{\bullet +1}(S).
\end{equation}

Next, the operations involving $\Omega^{\bullet}(S, TS)$ have the following lifts to $E$:
\begin{itemize}
\item With the short exact sequence
\[ 0\rmap \Omega^{\bullet}(S, \mathfrak{X}_{\mathrm{V}}(E))\rmap \Omega^{\bullet}(S, \mathfrak{X}_{\mathrm{P}}(E))\stackrel{p_S}{\rmap} \Omega^{\bullet}(S, TS)\rmap 0\]
in mind, there is a natural lift of $[\cdot, \cdot]_{F}$ to the middle term, which we denote by the same symbol.
Actually, realizing
\[\Omega^{\bullet}(S, \mathfrak{X}_{\mathrm{P}}(E))\stackrel{p^*}{\hookrightarrow} \Omega^{\bullet}(E, TE),\]
this is the restriction of the Fr\"{o}hlicher-Nijenhuis bracket on $\Omega^{\bullet}(E, TE)$.
\item The action $L$ of $\Omega^{\bullet}(S, TS)$ on $\Omega^{\bullet}(S)$ lifts to an action of
$\Omega^{\bullet}(S, \mathfrak{X}_{\mathrm{P}}(E))$ on $\Omega_E$, for $u=\alpha\otimes X\in
\Omega^{\bullet}(S, \mathfrak{X}_{\mathrm{P}}(E))$ and $v=\omega\otimes Y\in \Omega_E$, we have:
\[ L_{u}(v)= L_{p_S(u)}(\omega)\otimes Y+ \alpha\wedge \omega\otimes [X, Y].\]
\end{itemize}

The following shows how to put these operations together.

\begin{proposition}
 $\widetilde{\Omega}_{E}$ is a graded Lie algebra with bracket
\[[u,v]_{\ltimes}=\left\{
\begin{array}{ll}
[u,v]         &\textrm{for } u, v\in \Omega_E, \\
L_{u}(v) & \textrm{for } u\in \Omega^{\bullet}(S, \mathfrak{X}_{\mathrm{P}}(E)), v\in\Omega_E,\\
\phantom{} [u,v]_{F} & \textrm{for } u,v\in \Omega^{\bullet}(S, \mathfrak{X}_{\mathrm{P}}(E)).
\end{array}
 \right.
\]
Moreover, we have a short exact sequence of graded Lie algebras:
\[0\rmap(\Omega_{E}^{\bullet},[\cdot,\cdot])\rmap (\widetilde{\Omega}^{\bullet}_E,[\cdot,\cdot]_{\ltimes})\stackrel{p_S}{\rmap} (\Omega^{\bullet}(S, TS),[\cdot,\cdot]_F)\rmap 0.\]
\end{proposition}

\begin{proof}
We first check that the definitions agree on overlaps. Consider
\[u=\alpha\otimes X\in \Omega^{\bullet}(S,\mathfrak{X}_{\mathrm{P}}(E)), \ v=\beta\otimes Y \in \Omega^{\bullet}(S,\mathfrak{X}_{\mathrm{V}}(E)),\ w=\gamma\otimes Z\in \Omega_E.\]
Using that $p_S(v)=0$, we obtain
\begin{align*}
[u,v]_F&=L_{p_S(u)}(\beta)\otimes X-(-1)^{|\alpha||\beta|}L_{p_S(v)}(\alpha)\otimes Y+\alpha\wedge\beta\otimes [X,Y]\\
&=L_u(v),\\
L_v(w)&=L_{p_S(v)}(\gamma)\otimes Z+\beta\wedge\gamma\otimes [Y,Z]=[v,w],
\end{align*}
and for $z\in \Omega^{\bullet}(S,\mathfrak{X}_{\mathrm{V}}(E))$, $[v,z]=[v,z]_F$. Thus the bracket is well-defined.

To prove the graded Jacobi identity, it suffices to check that the algebra
\[(\Omega^{\bullet}(S,\mathfrak{X}_{\mathrm{P}}(E)),[\cdot,\cdot]_F)\]
acts by graded derivations on $(\Omega_E,[\cdot,\cdot])$, i.e.\
\begin{align*}
L_u&\circ L_v-(-1)^{(|u|-1)(|v|-1)}L_v\circ L_u=L_{[u,v]_F},\\
L_u([w,z])&=[L_{u}(w),z]+(-1)^{(|u|-1)(|w|-1)}[w,L_u(z)].
\end{align*}
for all $u,v\in \Omega^{\bullet}(S,\mathfrak{X}_{\mathrm{P}}(E))$ and $w,z\in \Omega_E$. We use the previous
notations for the elements. By the formula for $[\cdot,\cdot]_F$, we have that
\begin{align*}
L_{[u,v]_F}(w)=&L_{p_S([u,v]_F)}(\gamma)\otimes Z+L_{p_S(u)}(\beta)\wedge\gamma\otimes[Y,Z]-\\
&-(-1)^{|\alpha||\beta|}L_{p_S(v)}(\alpha)\wedge\gamma\otimes [X,Z]+\alpha\wedge\beta\wedge\gamma\otimes[[X,Y],Z].
\end{align*}
Using that $(\Omega^{\bullet}(S,TS),[\cdot,\cdot]_F)$ acts on $\Omega^{\bullet}(S)$, and that $p_S$ is an algebra
homomorphism, we split the first term in two parts
\[L_{p_S([u,v]_F)}(\gamma)\otimes Z=L_{p_S(u)}(L_{p_S(v)}(\gamma))\otimes Z- (-1)^{|\alpha||\beta|}L_{p_S(v)}(L_{p_S(u)}(\gamma))\otimes Z,\]
and similarly, using Jacobi, we split also the last term
\begin{align*}
\alpha\wedge\beta&\wedge\gamma\otimes[[X,Y],Z]=\\
&=\alpha\wedge\beta\wedge\gamma\otimes[X,[Y,Z]]-(-1)^{|\alpha||\beta|}\beta\wedge\alpha\wedge\gamma\otimes[Y,[X,Z]].
\end{align*}
Adding up the terms where the sign $-(-1)^{|\alpha||\beta|}$ doesn't appear, we obtain
\begin{align*}
L_{p_S(u)}(L_{p_S(v)}(\gamma))&\otimes Z+\alpha\wedge L_{p_S(v)}(\gamma)\otimes [X,Z]+\\
&+L_{p_{S}(u)}(\beta\wedge\gamma)\otimes[Y,Z]+\alpha\wedge\beta\wedge\gamma\otimes [X,[Y,Z]]=\\
&=L_u(L_{p_S(v)}(\gamma)\otimes Z+\beta\wedge\gamma\otimes[Y,Z])=L_u\circ L_v(w),
\end{align*}
and similarly the rest of the terms add up to $-(-1)^{|\alpha||\beta|}L_v\circ L_u(w)$. This finishes the proof of the first
relation.

Denoting $z=\delta\otimes T$, the second relation follows from the computation
\begin{align*}
L_u([w,z])&=(-1)^{|\delta|(|Z|-1)}L_u(\gamma\wedge\delta\otimes[Z,T])=\\
&=(-1)^{|\delta|(|Z|-1)}(L_{p_S(u)}(\gamma\wedge\delta)\otimes[Z,T]+\alpha\wedge\gamma\wedge\delta\otimes[X,[Z,T]])=\\
&=(-1)^{|\delta|(|Z|-1)}(L_{p_S(u)}(\gamma)\wedge\delta\otimes[Z,T]+\alpha\wedge\gamma\wedge\delta\otimes[[X,Z],T]+\\
&+(-1)^{|\alpha||\gamma|}\gamma\wedge L_{p_S(u)}(\delta)\otimes[Z,T]+\alpha\wedge\gamma\wedge\delta\otimes[Z,[X,T]])=\\
&=[L_{p_S(u)}(\gamma)\otimes Z+\alpha\wedge\gamma\otimes[X,Z],\delta\otimes T]+\\
&+(-1)^{|\alpha|(\gamma+|Z|-1)}[\gamma\otimes Z,L_{p_S(u)}(\delta)\otimes T+\alpha\wedge\delta\otimes[X,T]]=\\
&=[L_{u}(w),z]+(-1)^{(|u|-1)(|w|-1)}[w,L_u(z)].\phantom{+\alpha\wedge\delta\otimes[X,T]]====}\qedhere
\end{align*}
\end{proof}

A bit more on the structure of $(\widetilde{\Omega}_E,[\cdot,\cdot]_{\ltimes})$ is given in the lemma below.

\begin{lemma}\label{OmegaS-Central}
The center of $\widetilde{\Omega}_E$ are the constant functions on $E$, and the center of $\Omega_E$ is
$\Omega^{\bullet}(S)$. Moreover, $\Omega^{\bullet}(S)$ is an ideal in $\widetilde{\Omega}_E$.
\end{lemma}
\begin{proof}
The fact that $\Omega^{\bullet}(S)$ is in the center of $\Omega_E$ follows straightaway from the definition of the Lie
bracket. For $\omega\in \Omega^{p,q}_E$, a central element, we have
\[0=[X,\omega](X_1,\ldots,X_p)= [X,\omega(X_1,\ldots,X_p)],\]
for every $X\in\mathfrak{X}_{\mathrm{V}}(E)$ and $X_1,\ldots,X_p\in\mathfrak{X}(S)$. Hence,
$\omega(X_1,\ldots,X_p)$ is a multivector field which commutes with all vertical vector fields, and this implies that it is
the pullback of a function on $S$ (this requires the fibers of $p$ to be connected). Thus $\omega\in
\Omega^{\bullet}(S)$.

The fact that $\Omega^{\bullet}(S)$ is an ideal in $\widetilde{\Omega}_E$, follows by definition; just notice that the
induced representation descends to $\Omega^{\bullet}(S,TS)$:
\[[u,\omega]_{\ltimes}=L_u(\omega)=L_{p_S(u)}(\omega).\]
Thus an element $\omega$ in the center of $\widetilde{\Omega}_E$, belongs to $\Omega^{\bullet}(S)$, and satisfies
$L_X\omega=0$, for all $X\in \mathfrak{X}(S)$. This implies that $\omega$ is a constant.
\end{proof}

\subsection{Horizontally nondegenerate Poisson structures}\label{The graded Lie algebra}

\subsubsection*{Ehresmann connections}

As an illustration of the use of $\widetilde{\Omega}_E$, we look at \textbf{Ehresmann connections}\index{Ehresmann
connection} on $E$. Such a connection can be described either as a subbundle $H\subset TE$, complementary to $V$,
or equivalently by a $C^{\infty}(S)$-linear map which associates to a vector field $X$ on $S$ its horizontal lift
$\mathrm{hor}(X)$ to $E$. Since $\mathrm{hor}(X)$ projects back to $X$, we see that the connection can also be
described by an element
\[\Gamma\in\widetilde{\Omega}_E^{1,1}, \textrm{ such that } p_S(\Gamma)=\gamma_S.\]
Also the curvature $R_{\Gamma}$ of $\Gamma$
\[R_{\Gamma}(X,Y)=[\mathrm{hor}(X),\mathrm{hor}(Y)]-\mathrm{hor}([X,Y]),\]
can be described using $\widetilde{\Omega}_E$; it is just
\[R_{\Gamma}=\frac{1}{2}[\Gamma,\Gamma]_{\ltimes}\in\Omega^{2,1}_E.\]
Moreover, with the identification $\Omega^{\bullet,0}_E\cong \Omega^{\bullet}(H)$, the operator
\[d_{\Gamma}:\Omega^{\bullet,0}_E\rmap \Omega^{\bullet+1,0}_E, \ \ \alpha\mapsto [\Gamma,\alpha]_{\ltimes}\]
is just the horizontal derivative of horizontal forms
\begin{equation}\label{EQ_just_the_horizontal_derivative}
d_{\Gamma}(\alpha)(X_1,\ldots,X_n)=d\alpha(\mathrm{hor}(X_1),\ldots,\mathrm{hor}(X_n)).
\end{equation}
This can be easily checked on functions. In general, by decomposing $\alpha$ as a sum of elements of the form
$\omega\otimes f$, with $\omega\in\Omega(S)$ and $f\in C^{\infty}(E)$, and using that both $d_{\Gamma}$ and $d$
act as derivations, this follows from (\ref{gamma-represents-d})
\[d_{\Gamma}(\omega)=[\Gamma,\omega]_{\ltimes}=L_{\Gamma}(\omega)=L_{p_S(\Gamma)}(\omega)=L_{\gamma_S}(\omega)=d\omega.\]

\subsubsection*{Dirac elements and Dirac structures}

We introduce the following generalization of flat Ehresmann connections.

\begin{definition} A \textbf{Dirac element}\index{Dirac element} is an element $\gamma\in \widetilde{\Omega}_{E}^{2}$, satisfying
\[ [\gamma, \gamma]_{\ltimes}=0, \ \ \ \ p_{S}(\gamma)= \gamma_S.\]
We use the following notations for the components of $\gamma$:
\begin{itemize}
\item $\gamma^{\mathrm{v}}$ for the $(0,2)$ component- an element in $\mathfrak{X}_{\mathrm{V}}^{2}(E)$
\item $\Gamma_{\gamma}$ for the $(1, 1)$ component- an Ehresmann connection on $E$
\item $\mathbb{F}_{\gamma}$ for the $(2, 0)$ component- an element in $\Omega^2(S, C^{\infty}(E))$.
\end{itemize}
We denote by $H_{\gamma}\subset TE$ the horizontal distribution corresponding to $\Gamma_{\gamma}$.
\end{definition}

Before explaining the geometric meaning of such elements, we recall:

\begin{definition}[see  \cite{FernandesBrahic, Wade}]
A Dirac structure $L\subset TE\oplus T^*E$ is called \textbf{horizontally nondegenerate}\index{horizontally nondeg.
Dirac}\index{Dirac structure} if $L\cap (V\oplus V^{\circ})= \{0\}$.
\end{definition}

The following proposition is our interpretation of Corollary 2.8 \cite{FernandesBrahic} and of Theorem 2.9
\cite{Wade}.
\begin{proposition}
For $\gamma=(\gamma^{\mathrm{v}},\Gamma_{\gamma},\mathbb{F}_{\gamma})\in\widetilde{\Omega}^2_E$, a
Dirac element, define
\[L_{\gamma}:=\mathrm{Graph}(\gamma^{\textrm{v}\sharp}:H_{\gamma}^{\circ}\to V)\oplus \mathrm{Graph}(\mathbb{F}_{\gamma}^{\sharp}:H_{\gamma}\to V^{\circ})\subset TE\oplus T^*E.\]
The assignment $\gamma\mapsto L_{\gamma}$ is a one to one correspondence between Dirac elements and
horizontally nondegenerate Dirac structures on $E$.
\end{proposition}

\begin{proof}
We start by describing the inverse map. Let $L\subset TE\oplus T^*E $ be a horizontally nondegenerate Dirac
structure. Since $L\oplus L_V=TE\oplus T^*E$, where $L_V=V\oplus V^{\circ}$, we can decompose every $X\in TE$
uniquely as
\[X=(V_X-\xi_X)\oplus (H_X+\xi_X), \ \ V_X\in V,\ \xi_X\in V^{\circ}, \ H_X+\xi_X\in L.\]
Let $\Gamma_{L}$ be the Ehresmann connection with horizontal projection $X\mapsto H_X$, and let $H\subset TE$
be the corresponding distribution. Since $L$ in isotropic, and $\xi_X\in V^{\circ}$, we have that
\[0=(H_X+\xi_X,H_Y+\xi_Y)=(X+\xi_X,Y+\xi_Y)=\xi_X(Y)+\xi_Y(X),\]
thus $\xi$ is given by a skew-symmetric 2-form $\mathbb{F}_L\in\Omega^2(E)$, i.e.\
$\xi_X=\mathbb{F}_{L}^{\sharp}(X)$. Since $L\cap V^{\circ}=0$, we have that $\xi_{|V}=0$, and since $\xi_X\in
V^{\circ}$, it follows that $\mathbb{F}_L\in\Omega^{2,0}_E$. Similarly, an $\eta\in T^*E$ decomposes as
\[\eta=(-V_{\eta}+\theta_{\eta})\oplus (V_{\eta}+\varphi_{\eta}), \ \ V_{\eta}\in V,\ \theta_{\eta}\in V^{\circ}, \ V_{\eta}+\varphi_{\eta}\in L.\]
Using again that $L$ is isotropic, we obtain that for all $X\in TE$,
\[0=(H_X+\xi_X,V_{\eta}+\varphi_{\eta})=\varphi_{\eta}(H_X).\]
Thus $\varphi_{\eta}\in H^{\circ}$. So, $\eta=\theta_{\eta}+\varphi_{\eta}$ is the decomposition of $\eta$
corresponding to $\Gamma_{L}$. By a similar argument as before, this shows that $V_{|H^{\circ}}=0$, and that
there is a vertical bivector $\pi^{\textrm{v}}_{L}\in\Omega^{0,2}_E$, such that
$V_{\xi}=\pi^{\textrm{v},\sharp}_{L}(\xi)$. By comparing dimensions, we obtain that $L=L_{\gamma}$, where
\[\gamma=(\pi_L^{\mathrm{v}},\Gamma_{L},\mathbb{F}_{L})\in \widetilde{\Omega}^2_L.\]

We still have to check that $L_{\gamma}$ being closed under the Dorfman bracket is equivalent to
$[\gamma,\gamma]_{\ltimes}=0$. This equation splits into 4 components
\[0=[\gamma,\gamma]_{\ltimes}=[\gamma^{\mathrm{v}},\gamma^{\mathrm{v}}]\oplus
2[\Gamma_{\gamma},\gamma^\textrm{v}]_{\ltimes}\oplus2(R_{\Gamma_{\gamma}}+[\gamma^\textrm{v},\mathbb{F}_{\gamma}])\oplus 2[\Gamma_{\gamma},\mathbb{F}_{\gamma}]_{\ltimes}\in\]
\[\in\Omega^{0,3}_E\oplus\Omega^{1,2}_E\oplus\Omega^{2,1}_E\oplus\Omega^{3,0}_E=\Omega^{3}_E.\]

Let's consider the condition
\begin{equation}\label{EQ_condition1}
[\gamma^{\mathrm{v},\sharp}(\alpha)+\alpha,\gamma^{\mathrm{v},\sharp}(\beta)+\beta]_D\in L_{\gamma}, \ \ \forall \alpha,\beta\in \Gamma(H_{\gamma}^{\circ}).
\end{equation}
Explicitly, this expression is given by:
\[[\gamma^{\mathrm{v},\sharp}(\alpha),\gamma^{\mathrm{v},\sharp}(\beta)]+\iota_{\gamma^{\mathrm{v},\sharp}(\alpha)}d\beta-\iota_{\gamma^{\mathrm{v},\sharp}(\beta)} d\alpha+d\gamma^{\mathrm{v}}(\alpha,\beta),\]
and denote the form part by $[\alpha,\beta]_{\gamma^{\mathrm{v}}}$. Since the vector part is vertical,
(\ref{EQ_condition1}) is equivalent to the following two conditions
\begin{equation}\label{EQ_condition2}
[\gamma^{\mathrm{v},\sharp}(\alpha),\gamma^{\mathrm{v},\sharp}(\beta)]=\gamma^{\mathrm{v},\sharp}([\alpha,\beta]_{\gamma^{\mathrm{v}}}),
\end{equation}
\begin{equation}\label{EQ_condition3}
[\alpha,\beta]_{\gamma^{\mathrm{v}}}\in \Gamma(H_{\gamma}^{\circ}) ,
\end{equation}
for all $\alpha,\beta\in \Gamma(H_{\gamma}^{\circ})$. Now for $\beta\in \Gamma(V^{\circ})$, and any 1-form
$\alpha$, we have that
\[[\alpha,\beta]_{\gamma^{\mathrm{v}}}=\iota_{\gamma^{\mathrm{v},\sharp}(\alpha)}d\beta\in \Gamma(H_{\gamma}^{\circ}).\]
Therefore, if (\ref{EQ_condition2}) holds, then it holds for all 1-forms $\alpha,\beta\in\Omega^1(E)$. On the other
hand, one easily sees (e.g.\ by applying (\ref{EQ_condition2}) to $\alpha=df,\beta=dg$) that this is equivalent to
$\gamma^{\mathrm{v}}$ being Poisson i.e.\ $[\gamma^{\mathrm{v}},\gamma^{\mathrm{v}}]=0$. For
$\alpha,\beta\in\Gamma(H_{\gamma}^{\circ})$, contracting $[\alpha,\beta]_{\gamma^{\mathrm{v}}}$ with
$\mathrm{hor}(X)$, for $X\in\mathfrak{X}(S)$, one obtains
\begin{align*}
-\iota_{\gamma^{\mathrm{v},\sharp}(\alpha)}&L_{\mathrm{hor}(X)}(\beta)+\iota_{\gamma^{\mathrm{v},\sharp}(\beta)}L_{\mathrm{hor}(X)}(\alpha)+L_{\mathrm{hor}(X)}(\gamma^{\mathrm{v}}(\alpha,\beta))=\\
&=-\gamma^{\mathrm{v}}(\alpha,L_{\mathrm{hor}(X)}(\beta))-\gamma^{\mathrm{v}}(L_{\mathrm{hor}(X)}(\alpha),\beta)+L_{\mathrm{hor}(X)}(\gamma^{\mathrm{v}}(\alpha,\beta))=\\
&=[\mathrm{hor}(X),\gamma^{\mathrm{v}}](\alpha,\beta).
\end{align*}
Since $[\mathrm{hor}(X),\gamma^{\mathrm{v}}]$ is vertical, condition (\ref{EQ_condition3}) is equivalent to
\[[\mathrm{hor}(X),\gamma^{\mathrm{v}}]=[\Gamma_{\gamma},\gamma^{\mathrm{v}}]_{\ltimes}(X)=0, \ \ \forall X\in \mathfrak{X}(S).\]
Thus we have shown that (\ref{EQ_condition1}) is equivalent to the vanishing of the first two components in
$[\gamma,\gamma]_{\ltimes}$.

Consider the condition
\begin{equation}\label{EQ_condition4}
[\mathrm{hor}(X)+\mathbb{F}_{\gamma}^{\sharp}(X),\gamma^{\mathrm{v},\sharp}(\alpha)+\alpha]_D\in L_{\gamma}, \ \ \forall\ X\in \mathfrak{X}(S),\ \alpha\in \Gamma(H_{\gamma}^{\circ}).
\end{equation}
Explicitly, this is given by
\[[\mathrm{hor}(X),\gamma^{\mathrm{v},\sharp}(\alpha)]+L_{\mathrm{hor}(X)}\alpha- \iota_{\gamma^{\mathrm{v},\sharp}(\alpha)}d\iota_{\mathrm{hor}(X)} \mathbb{F}_{\gamma}.\]
Using that $[\gamma^{\mathrm{v},\sharp}(\alpha),\mathrm{hor}(X)]$ is vertical, for the last term we have
\begin{align*}
\iota_{\gamma^{\mathrm{v},\sharp}(\alpha)}d\iota_{\mathrm{hor}(X)} \mathbb{F}_{\gamma}=L_{\gamma^{\mathrm{v},\sharp}(\alpha)}\iota_{\mathrm{hor}(X)} \mathbb{F}_{\gamma}=\iota_{[\gamma^{\mathrm{v},\sharp}(\alpha),\mathrm{hor}(X)]}\mathbb{F}_{\gamma}+\\
+\iota_{\mathrm{hor}(X)}L_{\gamma^{\mathrm{v},\sharp}(\alpha)} \mathbb{F}_{\gamma}=\iota_{\mathrm{hor}(X)}\iota_{\gamma^{\mathrm{v},\sharp}(\alpha)} d\mathbb{F}_{\gamma}.
\end{align*}
Thus we have to show that
\[[\mathrm{hor}(X),\gamma^{\mathrm{v},\sharp}(\alpha)]+L_{\mathrm{hor}(X)}\alpha- \iota_{\mathrm{hor}(X)}\iota_{\gamma^{\mathrm{v},\sharp}(\alpha)} d\mathbb{F}_{\gamma}\in L_{\gamma}.\]
The first term is in $V$ and the last in $V^{\circ}$, therefore (\ref{EQ_condition4}) is equivalent to the following two
conditions
\begin{equation}\label{EQ_condition5}
[\mathrm{hor}(X),\gamma^{\mathrm{v},\sharp}(\alpha)]=\gamma^{\mathrm{v},\sharp}(L_{\mathrm{hor}(X)}\alpha),
\end{equation}
\begin{equation}\label{EQ_condition6}
L_{\mathrm{hor}(X)}\alpha- \iota_{\mathrm{hor}(X)}\iota_{\gamma^{\mathrm{v},\sharp}(\alpha)} d\mathbb{F}_{\gamma}\in \Gamma(H_{\gamma}^{\circ}) ,
\end{equation}
for all $\alpha\in \Gamma(H_{\gamma}^{\circ})$ and $X\in\mathfrak{X}(S)$. The first is equivalent to
\[[\mathrm{hor}(X),\gamma^v]^{\sharp}(\alpha)=0,\] which, as before, translates to
$[\Gamma_{\gamma},\gamma^{\mathrm{v}}]_{\ltimes}=0$. Before looking at (\ref{EQ_condition6}), note that, using
the Koszul formula to compute $d\mathbb{F}_{\gamma}$ and that $\mathbb{F}_{\gamma}$ vanishes on vertical
vectors, we obtain
\begin{equation}\label{EQ_condition0}
d\mathbb{F}_{\gamma}(V,\mathrm{hor}(X),\mathrm{hor}(Y))=L_V(\mathbb{F}_{\gamma}(X,Y)), \ \ \forall \  V\in \mathfrak{X}_{\mathrm{V}}(E).
\end{equation}
So, if we contract (\ref{EQ_condition6}) with $\mathrm{hor}(Y)$, we get
\begin{align*}
 \iota_{\mathrm{hor}(Y)}&L_{\mathrm{hor}(X)}\alpha-
 d\mathbb{F}_{\gamma}(\gamma^{\mathrm{v},\sharp}(\alpha),\mathrm{hor}(X),\mathrm{hor}(Y))=\\
&=\alpha([\mathrm{hor}(X),\mathrm{hor}(Y)])-L_{\gamma^{\mathrm{v},\sharp}(\alpha)}(\mathbb{F}_{\gamma}(X,Y))=\\
&=\alpha(R_{\Gamma_{\gamma}}(X,Y)+[\gamma^{\mathrm{v}},\mathbb{F}_{\gamma}](X,Y)).
\end{align*}
Thus (\ref{EQ_condition6}) is equivalent to
$[\gamma^{\mathrm{v}},\mathbb{F}_{\gamma}]+R_{\Gamma_{\gamma}}=0$. We have shown that
(\ref{EQ_condition4}) is equivalent to the vanishing of the second and third term in $[\gamma,\gamma]_{\ltimes}$.

Finally, we analyze the condition
\begin{equation}\label{EQ_condition7}
[\mathrm{hor}(X)+\mathbb{F}_{\gamma}^{\sharp}(X),\mathrm{hor}(Y)+\mathbb{F}_{\gamma}^{\sharp}(Y)]_D\in L_{\gamma}, \ \ \forall \ X,Y\in \mathfrak{X}(S).
\end{equation}
This can be rewritten as
\begin{align*}
[\mathrm{hor}(X)&,\mathrm{hor}(Y)]+L_{\mathrm{hor}(X)}\iota_{\mathrm{hor}(Y)}\mathbb{F}_{\gamma}-\iota_{\mathrm{hor}(Y)}d(\iota_{\mathrm{hor}(X)}\mathbb{F}_{\gamma})=\\
&=\mathrm{hor}([X,Y])+R_{\Gamma_{\gamma}}(X,Y)+\iota_{[X,Y]}\mathbb{F}_{\gamma}+\iota_{\mathrm{hor}(Y)}\iota_{\mathrm{hor}(X)}d\mathbb{F}_{\gamma}.
\end{align*}
So, (\ref{EQ_condition7}) is equivalent to
\[R_{\Gamma_{\gamma}}(X,Y)+\iota_{\mathrm{hor}(Y)}\iota_{\mathrm{hor}(X)}d\mathbb{F}_{\gamma}\in L_{\gamma}, \ \ \forall \ X,Y\in \mathfrak{X}(S).\]
Since the vector part is vertical, this also can be restated as two conditions
\begin{equation}\label{EQ_condition8}
\iota_{\mathrm{hor}(Y)}\iota_{\mathrm{hor}(X)}d\mathbb{F}_{\gamma}\in \Gamma(H_{\gamma}^{\circ}) ,
\end{equation}
\begin{equation}\label{EQ_condition9}
\gamma^{\mathrm{v},\sharp}(\iota_{\mathrm{hor}(Y)}\iota_{\mathrm{hor}(X)}d\mathbb{F}_{\gamma})=-R_{\Gamma_{\gamma}}(X,Y),
\end{equation}
for all $X, Y\in\mathfrak{X}(S)$. By (\ref{EQ_just_the_horizontal_derivative}), we have that
\[d\mathbb{F}_{\gamma}(\mathrm{hor}(X),\mathrm{hor}(Y),\mathrm{hor}(Z))=[\Gamma_{\gamma},\mathbb{F}_{\gamma}]_{\ltimes}(X,Y,Z),\]
so (\ref{EQ_condition8}) is equivalent to $[\Gamma_{\gamma},\mathbb{F}_{\gamma}]_{\ltimes}=0$. Applying
$\alpha$ to (\ref{EQ_condition9}), by (\ref{EQ_condition0}) we get
\[\alpha([\gamma^{\mathrm{v}},\mathbb{F}_{\gamma}](X,Y))=-\alpha(R_{\Gamma_{\gamma}}(X,Y)),\]
thus (\ref{EQ_condition9}) is equivalent to
$R_{\Gamma_{\gamma}}+[\gamma^\textrm{v},\mathbb{F}_{\gamma}]=0$. So, we have shown that
(\ref{EQ_condition7}) is equivalent to the vanishing of the last two components of $[\gamma,\gamma]_{\ltimes}$. This
finishes the proof.
\end{proof}

\begin{remark}\rm
Notice that the vertical Poisson structure $\gamma^{\mathrm{v}}$ corresponding to the Dirac element $\gamma$
can be described using the product of Dirac structures (see subsection \ref{subsection_Dirac})
\[L_{\gamma^{\mathrm{v}}}=L_{\gamma}*L_{V},\]
where $L_V=V\oplus V^{\circ}$. Moreover, the fact that $L_{\gamma}*L_{V}$ is Poisson, is equivalent to
$L_{\gamma}$ being horizontally nondegenerate.
\end{remark}

As a consequence of the proof of the proposition, we note the following:

\begin{corollary}\label{corollary_bracket_explicit}
Under the isomorphism
\[H_{\gamma}\oplus H_{\gamma}^{\circ}\diffto L_{\gamma}, \ \ (X,\alpha)\mapsto (X+\gamma^{\mathrm{v},\sharp}(\alpha), \alpha+\mathbb{F}_{\gamma}^{\sharp}(X)),\]
the Lie algebroid $(L_{\gamma},[\cdot,\cdot]_{D},p_{T})$ becomes
\[(H_{\gamma}\oplus H_{\gamma}^{\circ},[\cdot,\cdot]_{\gamma},\rho_{\gamma}),\]
with Lie bracket
\begin{align*}
[\mathrm{hor}(X),\mathrm{hor}(Y)]_{\gamma}&=\mathrm{hor}([X,Y])+\iota_{\mathrm{hor}(Y)}\iota_{\mathrm{hor}(X)}d\mathbb{F}_{\gamma}\\
[\mathrm{hor}(X),\alpha]_{\gamma}&=L_{\mathrm{hor}(X)}\alpha+\iota_{\gamma^{\mathrm{v},\sharp}(\alpha)}\iota_{\mathrm{hor}(X)} d\mathbb{F}_{\gamma}\\
[\alpha,\beta]_{\gamma}&=\iota_{\gamma^{\mathrm{v},\sharp}(\alpha)}d\beta- \iota_{\gamma^{\mathrm{v},\sharp}(\beta)}d\alpha+d\gamma^{\mathrm{v}}(\alpha,\beta),
\end{align*}
and anchor
\[\rho_\gamma(\mathrm{hor}(X))=\mathrm{hor}(X), \ \ \rho_\gamma(\alpha)=\gamma^{\mathrm{v},\sharp}(\alpha),\]
for all $X,Y\in\mathfrak{X}(S)$,  $\alpha,\beta\in\Gamma(H^{\circ}_{\gamma})$.
\end{corollary}

The complex computing the cohomology of a horizontally nondegenerate Dirac structure (see subsection
\ref{subsection_Dirac}) can be computed as follows.

\begin{proposition}\label{Proposition_Dirac_cohomology}
Let $\gamma$ be a Dirac element, and $L_{\gamma}$ the corresponding Dirac structure on $E$. Using the
isomorphism $L_{\gamma}^*=V\oplus V^{\circ}$, given by the pairing on $TE\oplus T^*E$, the complex computing
the cohomology of $L_{\gamma}$ becomes
\[(\Omega^{\bullet}_E, d_{\gamma}), \ \ d_{\gamma}:=[\gamma,\cdot]_{\ltimes}.\]
 \end{proposition}

\begin{proof}
We identify $L_{\gamma}$ with the Lie algebroid $H_{\gamma}\oplus H_{\gamma}^{\circ}$ (as in the previous
corollary), and $H_{\gamma}^*\cong V^{\circ}$, $(H_{\gamma}^{\circ})^{*}\cong V$. With these identifications,
we obtain that $\Omega_E=\Gamma(\Lambda^{\bullet}(V\oplus V^{\circ}))$, which is the complex computing the Lie
algebroid cohomology\index{cohomology, Lie algebroid}\index{cohomology, Dirac} of $H_{\gamma}\oplus
H_{\gamma}^{\circ}$. Denote its differential by $d_{L_{\gamma}}$. Since both $d_{L_{\gamma}}$ and
$d_{\gamma}$ act by derivations, it suffices to check that they coincide on
\[W\in \Omega^{0,1}_E=\mathfrak{X}_{\mathrm{V}}(E) \ \ \textrm{and}\ \ \eta\in\Omega^{1,0}_E=\Omega^1(S,C^{\infty}(E)).\]
First observe that
\begin{align*}
&d_{\gamma}W=[\gamma,W]_{\ltimes}=[\gamma^{\mathrm{v}},W]\oplus [\Gamma_{\gamma},W]_{\ltimes}\oplus [\mathbb{F}_{\gamma},W]\in \Omega^{0,2}_E\oplus\Omega^{1,1}_E\oplus \Omega^{2,0}_E,\\
&d_{\gamma}\eta=[\gamma,\eta]_{\ltimes}=[\gamma^{\mathrm{v}},\eta]\oplus [\Gamma_{\gamma},\eta]_{\ltimes}\in \Omega^{1,1}_E\oplus\Omega^{0,2}_E.
\end{align*}
We compute $d_{L_{\gamma}}$ on $W$ and $\eta$ using the formulas from the previous corollary, and we compare
the two differentials by evaluating them on elements $\alpha,\beta\in \Gamma(H_{\gamma}^{\circ})$ and
$\mathrm{hor}(X),\mathrm{hor}(Y)\in\Gamma(H_{\gamma})$, for $X,Y\in \mathfrak{X}(S)$:
\begin{align*}
(d_{L_{\gamma}}W)(\alpha,\beta)&=L_{\gamma^{\mathrm{v},\sharp}(\alpha)}\beta(W)-L_{\gamma^{\mathrm{v},\sharp}(\beta)}\alpha(W)-\\
&-W(\iota_{\gamma^{\mathrm{v},\sharp}(\alpha)}d\beta-\iota_{\gamma^{\mathrm{v},\sharp}(\beta)}d\alpha+d \gamma^{\mathrm{v}}(\alpha,\beta))=\\
&=\iota_{\gamma^{\mathrm{v},\sharp}(\alpha)}d\iota_W{\beta}-\iota_{\gamma^{\mathrm{v},\sharp}(\beta)}d\iota_W\alpha-\\
&-\iota_W\iota_{\gamma^{\mathrm{v},\sharp}(\alpha)}d\beta+\iota_W\iota_{\gamma^{\mathrm{v},\sharp}(\beta)}d\alpha-L_W \gamma^{\mathrm{v}}(\alpha,\beta)=\\
&=\iota_{\gamma^{\mathrm{v},\sharp}(\alpha)}L_W{\beta}-\iota_{\gamma^{\mathrm{v},\sharp}(\beta)}L_W\alpha-L_W \gamma^{\mathrm{v}}(\alpha,\beta)=\\
&=\gamma^{\mathrm{v}}(\alpha,L_W{\beta})+\gamma^{\mathrm{v}}(L_W\alpha,\beta)-L_W \gamma^{\mathrm{v}}(\alpha,\beta)=\\
&=[\gamma^{\mathrm{v}},W](\alpha,\beta)=[\gamma,W]_{\ltimes}(\alpha,\beta).
\end{align*}
\begin{align*}
(d_{L_{\gamma}}W)(\mathrm{hor}(X),\beta)&=L_{\mathrm{hor}(X)}\beta(W)-W(L_{\mathrm{hor}(X)}\beta-\iota_{\mathrm{hor}(X)}\iota_{\gamma^{\mathrm{v},\sharp}(\beta)}d\mathbb{F}_{\gamma})=\\
&=[\mathrm{hor}(X),W](\beta)=[\Gamma_{\gamma},W]_{\ltimes}(X,\beta)=[\gamma,W]_{\ltimes}(X,\beta),
\end{align*}
where we used that $d\mathbb{F}$ vanishes on two vertical vectors.
\begin{align*}
(d_{L_{\gamma}}W)(\mathrm{hor}(X),\mathrm{hor}(Y))&=-d\mathbb{F}_{\gamma}(\mathrm{hor}(X),\mathrm{hor}(Y),W)=-L_W(\mathbb{F}_{\gamma}(X,Y))=\\
&=[\mathbb{F}_{\gamma},W]_{\ltimes}(X,Y)=[\gamma,W]_{\ltimes}(X,Y),
\end{align*}
where we used (\ref{EQ_condition0}).
\begin{align*}
(d_{L_{\gamma}}\eta)&(\alpha,\beta)=0=[\gamma,\eta]_{\ltimes}(\alpha,\beta),\\
(d_{L_{\gamma}}\eta)&(\mathrm{hor}(X),\beta)=-L_{\gamma^{\mathrm{v},\sharp}(\beta)}\eta(\mathrm{hor}(X))=[\gamma^{\mathrm{v}},\eta](X,\beta)=[\gamma,\eta]_{\ltimes}(X,\beta),\\
(d_{L_{\gamma}}\eta)&(\mathrm{hor}(X),\mathrm{hor}(Y))=L_{\mathrm{hor}(X)}\eta(\mathrm{hor}(Y))-
L_{\mathrm{hor}(Y)}\eta(\mathrm{hor}(X))-\\
&\phantom{12345}-\eta([\mathrm{hor}(X),\mathrm{hor}(Y)])=d\eta(\mathrm{hor}(X),\mathrm{hor}(Y))=\\
&\phantom{12345}=[\Gamma_{\gamma},\eta]_{\ltimes}(X,Y)=[\gamma,\eta]_{\ltimes}(X,Y),
\end{align*}
where we used (\ref{EQ_just_the_horizontal_derivative}).
\end{proof}

\subsection{Horizontally nondegenerate Poisson structures}

Observe that, for a Dirac element $\gamma$, the Poisson support of $L_{\gamma}$ is
\[\mathrm{supp}(L_{\gamma})=\{e\in E\ |\ \mathbb{F}_{\gamma,e}^{\sharp}:H_{\gamma}\to H_{\gamma}^{\circ}\ \textrm{ is invertible}\}.\]
We will refer to this open also as the \textbf{Poisson support}\index{Poisson support} of $\gamma$.

\begin{definition} A Poisson bivector $\pi\in \mathfrak{X}^2(E)$ which satisfies
\[ V_e+ \pi^{\sharp}(V_{e}^{\circ})= T_eE, \ \forall \ e\in E,\]
is called \textbf{horizontally nondegenerate}\index{horizontally nondeg. Poisson}.
\end{definition}
Clearly, a Poisson structure $\pi$ on $E$ is horizontally nondegenerate, if and only if $L_{\pi}$ is a horizontally
nondegenerate Dirac structure. We will denote the corresponding Dirac element as follows:
\[\gamma_{\pi}=(\pi^{\mathrm{v}},\Gamma_{\pi},\mathbb{F}_{\pi}).\]
We describe the triple explicitly. The nondegeneracy of $\pi$ implies that
\[ H_{\pi}= \pi^{\sharp}(V^{\circ})\]
gives an Ehresmann connection on $E$, which corresponds to $\Gamma_{\pi}\in \widetilde{\Omega}_{E}^{1, 1}$.
With respect to the resulting decomposition $TE=V\oplus H_{\pi}$, the mixed component of $\pi$ vanishes
\[ \pi= \pi^{\mathrm{v}}+ \pi^{\mathrm{h}}\in \Lambda^2V\oplus \Lambda^2H_{\pi},\]
and $\pi^{\mathrm{v}}$ is the desired $(0, 2)$-component. The bivector $\pi^{\mathrm{h}}$ is nondegenerate as a
map $V^{\circ}\to H_{\pi}$, and its inverse represents $\mathbb{F}_{\pi}\in \Omega^2(S,C^{\infty}(E))$. Explicitly,
\begin{equation}\label{2Form}
\mathbb{F}_{\pi}(dp(\pi^{\mathrm{h} \sharp}\eta),dp(\pi^{\mathrm{h}
\sharp}\mu))=-\pi^{\mathrm{h}}(\eta,\mu),\ \ \forall  \ \eta,\mu\in V^{\circ}.
\end{equation}
Vorobjev's Theorem 2.1 in \cite{Vorobjev} can be summarized as follows:

\begin{proposition}\label{Dirac-Poisson} There is a 1-1 correspondence between
\begin{enumerate}
\item Dirac elements $\gamma\in \widetilde{\Omega}_{E}^2$ with $\textrm{supp}(L_{\gamma})=E$.
\item Horizontally nondegenerate Poisson structures $\pi$ on $E$.
\end{enumerate}
\end{proposition}

The following reformulation of Proposition 4.3 \cite{CrFe-stab} is a direct consequence of Proposition
\ref{Proposition_Dirac_cohomology}.

\begin{proposition}\label{PoissonCohomology}\index{cohomology, Poisson}
Let $\pi$ be a horizontally nondegenerate Poisson structure on $E$ with corresponding Dirac element $\gamma$. Then
there is an isomorphism of complexes
\[ \tau_{\pi}: (\mathfrak{X}^{\bullet}(E),d_{\pi})\diffto (\Omega_{E}^{\bullet},d_{\gamma}), \]
given by
\begin{equation}\label{tau-explicit}
\tau_{\pi}=\wedge^{\bullet} f_{\pi,*}: \mathfrak{X}^{\bullet}(E)\diffto \Omega^{\bullet}_E,
\end{equation}
where $f_{\pi}$ is the bundle isomorphism
\[f_{\pi}:=(I,-\mathbb{F}_{\pi}^{\sharp},):V\oplus H_{\pi}=TE\diffto V\oplus V^{\circ}.\]
\end{proposition}

\begin{proof}
By Proposition \ref{Proposition_Dirac_cohomology}, both complexes compute the cohomology of the Dirac structure
$L_{\pi}$, so we need only identify the resulting isomorphism. This is the dual of the composition of Lie algebroid
isomorphisms
\[H_{\pi}\oplus H_{\pi}^{\circ}\diffto L_{\pi}\diffto T^*E, \]
\[ (X,\eta)\mapsto (X+\pi^{\mathrm{v}}(\eta),\eta+\mathbb{F}_{\pi}^{\sharp}(X))\mapsto \eta+\mathbb{F}_{\pi}^{\sharp}(X),\]
which, by skew-symmetry of $\mathbb{F}_{\pi}$, is $f_{\pi}$.
\end{proof}

We discuss now deformations of horizontally nondegenerate Poisson structures. Let $\{\gamma_t\}_{t\in[0,1]}$ be a
smooth family of Dirac elements, and let $U\subset E$ be an open contained in the support of $\gamma_t$ for all
$t\in[0,1]$. Denote by $\pi_t$ the corresponding horizontally nondegenerate Poisson structure on $U$. Then
$\frac{d}{dt}\pi_t$ is a closed element in the Poisson cohomology of $\pi_t$. We describe this co-cycle in terms of
$\gamma_t$.

\begin{lemma}\label{Lemma_derivative_of_pite}
We have that
\[\tau_{\pi_t}\left(\frac{d}{dt}\pi_t\right)=\frac{d}{dt}\gamma_t.\]
\end{lemma}
\begin{proof}
Decomposing $\pi_t=\pi_t^{\mathrm{v}}+\pi_t^{\mathrm{h}}$, we see that $\dot{\pi}_t^{\mathrm{v}}$ is vertical,
therefore we have to show that
\[\tau_{\pi_t}(\dot{\pi}_t^{\mathrm{h}})=\dot{\Gamma}_{\pi_t}+\dot{\mathbb{F}}_{\pi_t}.\]
We use the following notations
\[A_t:=\pi_t^{\mathrm{h},\sharp}:T^*U\to TU, \  B_t:=\mathbb{F}_{\pi_t}^{\sharp}:TU\to T^*U, \  C_t:=A_t\circ B_t:TU\to TU.\]
Clearly, $C_t$ is just the horizontal projection corresponding to $\Gamma_{\pi_t}$. Also, we regard $\Gamma_{\pi_t}$
and $\mathbb{F}_{\pi_t}$ as skew-symmetric elements in \[\mathrm{Hom}(TU\oplus T^*U,T^*U\oplus TU),\] and as
such they can be written as
\[\Gamma_{\pi_t}=\left(
                   \begin{array}{cc}
                     0 & -C_t^* \\
                     C_t & 0 \\
                   \end{array}
                 \right), \ \
\mathbb{F}_{\pi_t}=\left(
                     \begin{array}{cc}
                       B_t & 0 \\
                       0 & 0 \\
                     \end{array}
                   \right)
\]
The map $f_{\pi_t}:TU\to T^*U\oplus TU$, is given by
\[f_{\pi_t}=\left(
              \begin{array}{c}
                -B_t\\
                I-C_t \\
              \end{array}
            \right).\]
With this formalism, on $\mathfrak{X}^2(U)\subset \mathrm{Hom}(T^*U,TU)$, we have
\[\tau_{\pi_t}(W)=f_{\pi_t}\circ W\circ f_{\pi_t}^*.\]
Now, we compute $\tau_{\pi_t}(\dot{A}_t)$. First we have that
\begin{align*}
\dot{A}_t\circ f_{\pi_t}^*&= \dot{A}_t(B_t,I-C_t^* )=(\dot{A}_tB_t,\dot{A}_t-\dot{A}_tC_t^*)=\\
&=(\dot{C}_t-A_t\dot{B}_t,-(\dot{A}_t-C_t\dot{A}_t)^*)=\\
&=(\dot{C}_t-A_t\dot{B}_t,-(\dot{C}_tA_t)^*)=(\dot{C}_t-A_t\dot{B}_t,A_t\dot{C}_t^*),
\end{align*}
where we have used that $C_tA_t=A_t$. Therefore
\begin{align*}
\tau_{\pi_t}(\dot{A}_t)&=\left(
              \begin{array}{c}
                -B_t\\
                I-C_t \\
              \end{array}
            \right)(\dot{C}_t-A_t\dot{B}_t,A_t\dot{C}_t)=\\
            &=\left(
                                                           \begin{array}{cc}
                                                              B_t(-\dot{C}_t+A_t\dot{B}_t)  & -B_tA_t\dot{C}_t^* \\
                                                             (I-C_t)(\dot{C}_t-A_t\dot{B}_t)  & (I-C_t)A_t\dot{C}_t\\
                                                           \end{array}
                                                         \right).
\end{align*}
For the $(1,1)$-entry, we note that $B_t\dot{C}_t=0$, since $\dot{C}_t$ takes values vertical vectors; and
$B_tA_t=C_t^*$, $B_tC_t=B_t$, therefore
\[B_tA_t\dot{B}_t=-(\dot{B}_tC_t)^*=\dot{B}_t+(B_t\dot{C}_t)^*=B_t.\]
For the $(1,2)$ entry, using also that $C_t^2=C_t$ and that $C_t\dot{C}_t=0$ (since $C_t$ vanishes on vertical
vectors), we obtain
\[-B_tA_t\dot{C}_t^*=-C_t^*\dot{C}_t^*=-(\dot{C}_tC_t)^*=-\dot{C}_t^*+(C_t\dot{C}_t)^*=-\dot{C}_t^* \]
Using $(I-C_t)A_t=0$, we fill in the remaining entries and obtain the result
\[
\tau_{\pi_t}(\dot{A}_t)=\left(
                                                           \begin{array}{cc}
                                                              \dot{B}_t & -\dot{C}_t^* \\
                                                             \dot{C}_t  & 0\\
                                                           \end{array}
                                                         \right)=\dot{\Gamma}_{\pi_t}+\dot{\mathbb{F}}_{\pi_t}.\qedhere\]
\end{proof}

To trivialize a smooth family $\{\pi_t\}_{t\in [0,1]}$ of Poisson structures on $U$, means to find a smooth family of
diffeomorphisms $\Phi_t:U\to U$ that satisfy
\[\Phi_0=\mathrm{Id}, \ \ \ \Phi_t^*(\pi_t)=\pi_0.\]
Let $X_t$ denote the time dependent vector field generating the family $\Phi_t$
\[\frac{d}{dt}\Phi_t(x)=X_{t}(\Phi_t(x)).\]
Then $X_t$ satisfies the so-called \textbf{homotopy equation}\index{homotopy equation}:
\begin{equation}\label{EQ_homotopy_equation}
\frac{d}{dt}\pi_t=[\pi_t,X_t].
\end{equation}
Conversely, if we find a time dependent vector field $X_t$ satisfying (\ref{EQ_homotopy_equation}), then its flow,
defined as the solution to the differential equation
\[\Phi_0(x)=x, \ \ \frac{d}{dt}\Phi_t(x)=X_{t}(\Phi_t(x)),\]
will send $\pi_0$ to $\pi_t$, i.e.\ $\Phi_t^*(\pi_t)=\pi_0$, whenever it is defined.

Lemma \ref{Lemma_derivative_of_pite} and Proposition \ref{PoissonCohomology} imply the following compact version
of Proposition 2.14 \cite{Vorobjev2}.

\begin{lemma}\label{HomotopyEquation}
Let $\{\gamma_t\}_{t\in [0,1]}$ be a family of Dirac elements on $E$, $U\subset E$ an open included in the
intersection of their supports, and let $\pi_t$ be the corresponding family of horizontally nondegenerate Poisson
structures. Then a time dependent vector field $X_t\in\mathfrak{X}(U)$ satisfies the homotopy equation
(\ref{EQ_homotopy_equation}), if and only if \[V_t:=\tau_{\pi_t}(X_t)\in \Omega^1_U\] satisfies
\begin{equation}
\label{homotopy-equation}\frac{d}{dt}\gamma_t= [\gamma_t,V_t]_{\ltimes}.
\end{equation}
\end{lemma}

\subsection{The dilation operators and jets along $S$} \label{The dilation operators and jets along $S$}

For $t\in\mathbb{R}$, $t\neq0$ let $\mu_t:E\to E$ be the fiberwise multiplication by $t$. Pullback by $\mu_t$ gives an
automorphism
\[\mu_t^*:(\widetilde{\Omega}_E,[\cdot,\cdot]_{\ltimes})\diffto (\widetilde{\Omega}_E,[\cdot,\cdot]_{\ltimes}),\]
which preserves $\Omega_E$ and acts as the identity on $\Omega^{\bullet}(S)$.

Define the \textbf{dilation operators}\index{dilation operators} as follows:
\[ \varphi_t:\widetilde{\Omega}_E\diffto \widetilde{\Omega}_E, \ \varphi_t(u)=t^{q-1}\mu_t^{*}(u),\textrm{ for }u\in \widetilde{\Omega}_E^{\bullet,q}.\]

\begin{remark}\label{phi-local-coordinates}\rm
It is useful to describe this operation in local coordinates. Choose $x_i$ coordinates for $S$ and $y_{\alpha}$ linear
coordinates on the fibers of $E$. An arbitrary element in $\Omega^p(S,\mathfrak{X}^q(E))$ is a sum of elements of
type
\[ a(x, y) dx_I \otimes \partial_{x_J}\wedge \partial_{y_K} \]
where $I$, $J$ and $K$ are multi-indices with $|I|= p$, $|J|+ |K|= q$ and $a=a(x, y)$ is a smooth function. Elements in
$\Omega_{E}$ contain only terms with $|J|= 0$. The elements in $\widetilde{\Omega}_E$ are also allowed to contain
terms with $|J|= 1$, but those terms must have $|K|=0$, and the coefficient $a$ only depending on $x$.  Applying
$\varphi_t$ to such an element we find
\begin{equation}\label{EQ_how_they_act}
 t^{|J|-1} a(x, ty) dx_I \otimes \partial_{x_J}\wedge \partial_{y_K}.
 \end{equation}
\end{remark}

\begin{lemma}
The dilation operator $\varphi_t$ preserves the bidegree, is an automorphism of the graded Lie algebra
$\widetilde{\Omega}_E$ and preserves $\Omega_E$.
\end{lemma}

\begin{proof}
Due to its functoriality, $\mu_{t}^{*}$ has these properties. Since
\[[\widetilde{\Omega}_E^{p,q},\widetilde{\Omega}_E^{p',q'}]_{\ltimes}\subset \widetilde{\Omega}_E^{p+p',q+q'-1},\]
also the multiplication by $t^{q-1}$ has the same properties. Hence also the composition of the two operations, i.e.\
$\varphi_t$, has the desired properties.
\end{proof}

We introduce also the following subspaces of $\widetilde{\Omega}_E$, for $l\in \mathbb{Z}$:
\[ \textrm{gr}_{l}(\widetilde{\Omega}_E)= \{ u\in \widetilde{\Omega}_E: \varphi_t(u)= t^{l-1} u\}\subset \widetilde{\Omega}_E,\]
\[ J^{l}_{S}(\widetilde{\Omega}_E)= \textrm{gr}_{0}(\widetilde{\Omega}_E)\oplus \ldots \oplus \textrm{gr}_{l}(\widetilde{\Omega}_E)\subset \widetilde{\Omega}_E.\]
These spaces vanish for $l< 0$ (see (\ref{EQ_how_they_act})). The elements in $\textrm{gr}_{0}$ are called
\textbf{constant}, those in $\textrm{gr}_{1}$ are called \textbf{linear}, while those in $\textrm{gr}_{l}$ are called
\textbf{homogeneous of degree $l$}. Similarly, one defines $\textrm{gr}_{l}(\Omega_E)$. We have canonical
isomorphisms (see e.g.\ (\ref{EQ_how_they_act}))
\begin{equation}\label{computation-grading}
\textrm{gr}_{l}(\Omega_{E}^{p, q})= \Omega^p(S, \Lambda^qE\otimes \mathcal{S}^lE^*),
\end{equation}
where $\mathcal{S}^lE^*$ denotes the $l$-th symmetric power of $E^*$, and we identify:\\
$\bullet$ sections of $E$ with fiberwise constant vertical vector fields on $E$,\\
$\bullet$ sections of $\mathcal{S}^lE^*$ with degree $l$ homogeneous polynomial functions on $E$.\\
Moreover, $\textrm{gr}_{l}(\widetilde{\Omega}_{E}^{p, q})$ coincides with $\textrm{gr}_{l}(\Omega_{E}^{p,
q})$ except for $l= 1$, $q= 1$, when
\[\textrm{gr}_{1}(\widetilde{\Omega}_{E}^{p, 1})= \Omega^p(S, \mathfrak{X}_{\textrm{lin}}(E)),\]
where $\mathfrak{X}_{\textrm{lin}}(E)$ is the space of linear vector fields on $E$, i.e. projectable vector fields whose
flow is fiberwise linear.

Our next aim is to introduce the partial derivative operators along $S$,
\[ d_{S}^{l}: \widetilde{\Omega}_E\rmap \textrm{gr}_l(\widetilde{\Omega}_E).\]
To define and handle them, we use the formal power series expansion of $t\varphi_{t}(u)$ with respect to $t$.
Although $\varphi_t$ is not defined at $t= 0$, by (\ref{EQ_how_they_act}) it is clear that, for any $u\in
\widetilde{\Omega}_E$, the map
\[ \mathbb{R}^{*}\ni t\mapsto t\varphi_t(u)\in \widetilde{\Omega}_E \]
admits a smooth extension to $\mathbb{R}$. Hence, the following makes sense.

\begin{definition}
Define the \textbf{$n$-th order derivatives} of $u\in \widetilde{\Omega}_E$ along $S$, denoted by $d^{n}_{S}u$, as
the coefficients of the formal power expansion of $\varphi_t$ at $t=0$
\[ \varphi_{t}(u)\cong t^{-1} u|_{S}+ d_{S}u+ t d^{2}_{S} u+ \ldots .\]
In other words,
\[ d_{S}^{n}(u)= \frac{1}{n!}\frac{d^n}{dt^n}_{|t= 0} t\varphi_t(u) \in \widetilde{\Omega}_E.\]
We also use the notation $u|_{S}:=d^0_Su$. The \textbf{$n$-th order jet} of $u$ along $S$ is
\[ j^{n}_{S}(u)= \sum_{k= 0}^{n} d_{S}^{k} (u).\]
\end{definition}

\begin{lemma}\label{TheFormula}
We have that $d_{S}^{n}(u)\in \mathrm{gr}_{n}(\widetilde{\Omega}_E)$ and $j^{n}_{S}(u)\in
J^{n}_{S}(\widetilde{\Omega}_E)$.
\end{lemma}

\begin{proof}
Since $\varphi_{r}\circ \varphi_{s}=\varphi_{rs}$, we have that
$\varphi_{r}(s\varphi_{s}(u))=r^{-1}[\xi\varphi_{\xi}(u)]_{|\xi=rs}$. Taking $n$ derivatives at $s=0$, we obtain the
result.
\end{proof}

The power series description, together with the properties of $\varphi_t$, are very useful in avoiding computations. For
instance, using that $\varphi_t$ preserves $[\cdot, \cdot]_{\ltimes}$ and comparing coefficients, we obtain a Newton
type formula.

\begin{lemma}\label{Newton-formula}
For any $u, v\in \widetilde{\Omega}_E$, we have that
\[ d_{S}^{l}[u, v]_{\ltimes}= \sum_{p+q= l+1} [d^{p}_{S}u, d^{q}_{S}v]_{\ltimes} .\]
\end{lemma}

As an illustration of our constructions let us look again at connections. We have already seen that an Ehresmann
connection on $E$ can be seen as an element $\Gamma\in \widetilde{\Omega}_{E}^{1, 1}$. We have that $\Gamma$
is linear as an element of $\widetilde{\Omega}_{E}$ if and only if it is a linear connection. For the direct implication:
the properties of $\varphi_t$ immediately imply that the $\ltimes$- bracket with $\Gamma$ preserves
$\textrm{gr}_{0}(\widetilde{\Omega}^{\bullet, 1}_{E})= \Omega^{\bullet}(S, E)$ hence it induces a covariant
derivative
\[ d_{\Gamma}:= [\Gamma, \cdot ]_{\ltimes}: \Omega^{\bullet}(S, E)\rmap \Omega^{\bullet+ 1}(S, E).\]

\subsection{The first jet of a Poisson structure around a leaf}\label{subsection_linearizing_Poisson_structures}

Throughout this subsection, we let $\pi$ be a horizontally nondegenerate Poisson structure on $E$, with corresponding
Dirac element $\gamma=(\pi^{\mathrm{v}},\Gamma_{\pi},\mathbb{F}_{\pi})$. After the first lemma, we will make
the extra assumption that $S$, viewed as the zero section of $E$, is a symplectic leaf of $\pi$, and we will describe the
linearization of $\pi$ around $S$ using the algebraic tools we have just developed. Actually, as one can easily see,
everything what we prove works more generally, for any horizontally nondegenerate Dirac structure on $E$, which
admits $S$ as a presymplectic leaf.\\

The fact that $S$ is a symplectic leaf can be characterized algebraically (see also Proposition 6.1 \cite{CrFe-stab}).

\begin{lemma} \label{lma-last} $S$ is a symplectic leaf of $\pi$ if and only if $\gamma|_{S}$ lives in bi-degree $(2,
0)$. In this case, the symplectic form is
\[ \omega_S=\gamma|_{S}\in \textrm{gr}_{0}(\Omega_{E}^{2, 0})= \Omega^2(S).\]
\end{lemma}

\begin{proof}
Recall that the Dirac structure $L_{\pi}$ is spanned by $\mathrm{hor}(X)+\mathbb{F}_{\pi}^{\sharp}(X)$ and
$\pi^{\mathrm{v},\sharp}(\alpha)+\alpha$, with $X\in \mathfrak{X}(S)$ and $\alpha\in \Gamma(H^{\circ})$. Since
$\pi$ has rank at least $\mathrm{dim}(S)$, it follows that $S$ is a symplectic leaf if and only if the vector part of these
elements belongs to $TS\subset TE_{|S}$, which means that $\mathrm{hor}(X)_{|S}=X$ and
$\pi^{\mathrm{v},\sharp}_{|S}(\alpha)=0$. These conditions are equivalent to $\pi^{\mathrm{v}}_{|S}=0$ and
$\Gamma_{\pi|S}=0$, i.e.\ $\gamma_{|S}\in \Omega^{2,0}_E$. In this case, we have that
\[\pi_{|S}=\pi^{\mathrm{v}}_{|S}+\pi^{\mathrm{h}}_{|S}=\pi^{\mathrm{h}}_{|S}\in \mathfrak{X}^2(S),\]
and so
\[\omega_S=(\pi^{\mathrm{h}}_{|S})^{-1}=\mathbb{F}_{\pi|S}=\gamma_{|S}.\]
\end{proof}

From now on, we will assume that $(S,\omega_S)$ is a symplectic leaf of $\pi$, or equivalently that
$\gamma_{|S}=\omega_S$. We look now at the next term in the Taylor expansion of $\gamma$ around $S$, whose
components we denote as follows
\[d^1_S\gamma=(\plin^{\mathrm{v}},\Gamma_{\mathrm{lin}},\mathbb{F}_{\mathrm{lin}})\in \textrm{gr}_1(\widetilde{\Omega}_{E}^2),\]
\[\plin^{\mathrm{v}}\in \Gamma (\Lambda^2E\otimes E^*),\ \Gamma_{\mathrm{lin}}\in  \Omega^1(S, \mathfrak{X}_{\textrm{lin}}(E)),\ \mathbb{F}_{\mathrm{lin}}\in \Omega^2(S, E^*).\]

\begin{lemma}\label{Lemma_d_and_j_are_Dirac}
We have that $d^1_S\gamma$ and $j^1_S\gamma$ are Dirac elements on $E$.
\end{lemma}
\begin{proof}
By the Newton formula from Lemma \ref{Newton-formula}, we have that
\[0=d^1_S[\gamma,\gamma]_{\ltimes}=2[\omega_S,d^2_S\gamma]+[d^1_S\gamma,d^1_S\gamma]_{\ltimes}=[d^1_S\gamma,d^1_S\gamma]_{\ltimes},\]
where we have used also that $d^2_S\gamma\in\Omega_E^2$, and so, by Lemma \ref{OmegaS-Central}, $\omega_S$
commutes with $d^2_S\gamma$. Also, we have that $p_S(\varphi_t(u))=p_S(u)$, for all $u\in \widetilde{\Omega}_E$,
therefore $p_S(d_S\gamma)=\gamma_S$. This shows that $d^1_S\gamma$ is a Dirac element. For $j^1_S\gamma$,
note that $[\omega_{S},\omega_S]=0$, $p_S(j^1_S\gamma)=p_S(d^1_S\gamma)=\gamma_S$, and also
\[ [d^1_S\gamma,\omega_S]_{\ltimes}=L_{p_S(d^1_S\gamma)}\omega_S=L_{\gamma_S}\omega_S=d\omega_S=0.\]
These imply that $j^1_S\gamma_S$ is also Dirac.
\end{proof}

Observe that the 2-form part of $j^1_S\gamma$ is $\omega_S+\mathbb{F}_{\mathrm{lin}}$. Since
$\mathbb{F}_{\mathrm{lin}}$ vanishes along $S$, it follows that the open where
$\omega_S+\mathbb{F}_{\mathrm{lin}}$ is invertible contains $S$. This is precisely the support of $j^1_S\gamma$,
which we denote by
\[N:=\mathrm{supp}(j^1_S\gamma).\]
We denote the corresponding Poisson structure by $\plin(S)$.

\begin{proposition}\label{Proposition_reconcile_2}
The Poisson manifold $(N, \plin(S))$ is the first order approximation of $\pi$ along $S$ from Definition
\ref{definition_first_order_approx}.
\end{proposition}
\begin{proof}
Recall that the path $\pi_t$ from Definition \ref{definition_first_order_approx} is given by
\[\pi_t=t\mu_t^*(\pi^{(t-1)\omega_S}),\ t\neq 0 .\]
This is not well-defined everywhere on $E$ as a Poisson structure, but it is a well-defined Dirac structure $L_t$ on $E$,
which is given by
\[L_t=t\mu_t^*(L_{\pi}^{(t-1)\omega_S}),\ t\neq 0,\]
where we use the rescaling of Dirac structures from subsection \ref{subsection_Dirac}. The fact that $S$ is a
symplectic leaf of $\pi$ implies that $L_{\pi}$ is horizontally nondegenerate on some open $U$ containing $S$. We
denote the corresponding Dirac element by
\[\gamma=(\pi^{\mathrm{v}},\Gamma_{\pi},\mathbb{F}_{\pi})\in \widetilde{\Omega}^2_{U}.\]
We claim that $L_t$ is horizontally nondegenerate on $\frac{1}{t}U$, and we compute its Dirac element $\gamma_t\in
\widetilde{\Omega}^2_{\frac{1}{t}U}$. Recall that $L_{\pi}$ is spanned by
\[X+\mathbb{F}_{\pi}^{\sharp}(X),\ \  \pi^{\mathrm{v},\sharp}(\alpha)+\alpha, \ \  X\in H_{\gamma}, \ \alpha\in H_{\gamma}^{\circ}.\]
Since $\omega_S$ is horizontal, it follows that $L_t$ is spanned by
\[t\mu_t^*(X)+\left(\mu_t^*(\mathbb{F}_{\pi})+(t-1)\omega_S\right)^{\sharp}\mu_t^*(X),\  t\mu_t^*(\pi^{\mathrm{v}})^{\sharp}\mu_t^*(\alpha)+\mu_t^*(\alpha),\]
with $X\in H_{\gamma}$ and $\alpha\in H_{\gamma}^{\circ}$. Observe that $H_{t}:=\mu_t^*(H_{\gamma})$ is the
horizontal distribution corresponding to $\mu_t^*(\Gamma_{\pi})$, and that
$H^{\circ}_{t}=\mu_t^*(H_{\gamma}^{\circ})$, both defined on $\frac{1}{t}U$. Therefore, $L_t$ is spanned by the
elements
\[X+t^{-1}\left(\mu_t^*(\mathbb{F}_{\pi})-\omega_S+t\omega_S\right)^{\sharp}(X),\ \  t\mu_t^*(\pi^{\mathrm{v}})^{\sharp}(\alpha)+\alpha,\]
with $X\in H_{t}$ and $\alpha\in H_{t}^{\circ}$. This shows that $L_t$ is the Dirac structure $L_{\gamma_t}$ on
$\frac{1}{t}U$, corresponding to the Dirac element
\[\gamma_t:=t\mu_t^*(\pi^{\mathrm{v}})+\mu_t^*(\Gamma_{\pi})+t^{-1}\mu_t^*(\mathbb{F}_{\pi})+(1-t^{-1})\omega_S=\gamma_{|S}+\frac{t\varphi_t(\gamma)-\gamma_{|S}}{t}.\]
That $\gamma_t$ is Dirac can also be checked directly:
\begin{align*}
p_S(\gamma_t)&=p_S(\varphi_t(\gamma)+(1-t^{-1})\omega_S)=p_S(\gamma)=\gamma_S,\\
[\gamma_t,\gamma_t]_{\ltimes}&=[\varphi_t(\gamma),\varphi_t(\gamma)]_{\ltimes}+2(1-t^{-1})[\varphi_t,\omega_S]_{\ltimes}=\\
&=\varphi_t([\gamma,\gamma]_{\ltimes})+2(1-t^{-1})d\omega_S=0.
\end{align*}
Also, it is clear that $\gamma_t$ extends smoothly at $t=0$, with $\gamma_0=j^1_S\gamma$, and therefore also
$\pi_0=\plin(S)$. This finishes the proof.
\end{proof}

\begin{remark}\rm\label{Remark_why_pite_is_smooth}
The proof above implies the assertions made before Definition \ref{definition_first_order_approx}, namely that $\pi_t$
extends smoothly at $t=0$, and that on some neighborhood of $S$, $\pi_t$ is defined for all $t\in [0,1]$. The fact that
such an open exists, follows by the Tube Lemma and the fact that the support of $\gamma_t$ contains $S$ for all $t$
(simply because $\gamma_{t|S}=\omega_S$).
\end{remark}

We prove now that $d^1_S\gamma$ encodes the structure of the Lie algebroid $A_S$, the restriction of the cotangent
Lie algebroid of $\pi$ to $S$ (see also \cite{Vorobjev}).

\begin{lemma}\label{Lemma_algebroid_cutarica}
The Dirac element $d^1_S\gamma$ induces a Lie algebroid structure on $A:=TS\oplus E^*$, with anchor the first
projection and Lie bracket $[\cdot,\cdot]_A$ given by
\begin{equation}\label{EQ_Lie_bracket}
 [\alpha, \beta]_{A}=\pi^{\mathrm{v}}_{\mathrm{lin}}(\alpha,\beta), \ [X,\alpha]_{A}= \nabla_X(\alpha), \ [X, Y]_{A}= [X, Y]+ \mathbb{F}_{\mathrm{lin}}(X,Y),
 \end{equation}
for $X,Y\in\mathfrak{X}(S)$, $\alpha,\beta\in\Gamma(E^*)$, where $\nabla$ is the covariant derivative corresponding
to $\Gamma_{\mathrm{lin}}$. Moreover the map
\[A=TS\oplus E^*\stackrel{(\omega_S^{\sharp},I)}{\rmap}T^*S\oplus E^*=A_S\]
is an isomorphism of Lie algebroids, where $A_S$ is the Lie algebroid corresponding to the leaf $S$ of $\pi$.
\end{lemma}

\begin{proof}
We will use Corollary \ref{corollary_bracket_explicit}, which gives the Lie bracket on the Lie algebroid
$B:=H_{\pi}\oplus H_{\pi}^{\circ}$, and an isomorphism between $B\cong L_{\pi}$. Observe that $H_{\pi|S}=TS$
and $H^{\circ}_{\pi|S}=E^*\subset T^*E_{|S}$, thus $\mathrm{hor}(X)_{|S}=X$, and since
$\pi^{\mathrm{v}}_{|S}=0$, it follows that $S$ is an orbit of $B$, and that the anchor of $B_{|S}$ is the first
projection. Moreover, one easily checks the following equalities
\[\left(L_{\mathrm{hor}(X)}(\alpha)\right)_{|S}=\nabla_X(\alpha_{|S}),\ \  \left(\iota_{\mathrm{hor}(Y)}\iota_{\mathrm{hor}(X)}d\mathbb{F}_{\pi}\right)_{|S}=\mathbb{F}_{\mathrm{lin}}(X,Y), \]
\[ d\pi^{\mathrm{v}}(\alpha,\beta)_{|S}=\pi^{\mathrm{v}}_{\mathrm{lin}}(\alpha_{|S},\beta_{|S}), \ \forall \ \alpha,\beta\in\Gamma(H^{\circ}_{\pi}), \ X,Y\in\mathfrak{X}(S),\]
where we use the inclusion $E^*\subset T^*E_{|S}$. These equations show that also the Lie bracket of $B_{|S}$
coincides with that of $A$, thus $B_{|S}=A$ as Lie algebroids. The induced isomorphism between $B_{|S}$ and
$L_{\pi|S}$ becomes $(X,\alpha)\mapsto (X,\alpha+\omega_{S}^{\sharp}(X))$, and composing it with the isomorphism
$L_{\pi|S}\cong A_S$, $(X,\eta)\mapsto \eta$, we obtain the one from the statement.
\end{proof}

The Dirac structure $L_{d_S^1\gamma}$ corresponding to the element $d_S^1\gamma$ satisfies
\[\mu_t^*(L_{d_S^1\gamma})=tL_{d_S^1\gamma}, \ \ t\neq 0.\]
Such Dirac structures deserve a name.

\begin{definition}
A horizontally nondegenerate Dirac structure $L$ on $E$ is called \textbf{a horizontally nondegenerate linear Dirac
structure}, if it satisfies
\[\mu_t^*(L)=tL, \ \ \forall t\neq 0.\]
\end{definition}

Of course, such Dirac structures correspond to linear Dirac elements. Note that any Lie algebroid structure on
$TS\oplus E^*$ with anchor the first projection, is given by formulas as in (\ref{EQ_Lie_bracket}), for some linear
element
\[\gamma_{\mathrm{lin}}=(\plin^{\mathrm{v}},\Gamma_{\mathrm{lin}},\mathbb{F}_{\mathrm{lin}})\in \textrm{gr}_1(\widetilde{\Omega}_{E}^2),\]
such that $p_S(\gamma_{\mathrm{lin}})=\gamma_S$, i.e.\  $\Gamma_{\mathrm{lin}}$ is a linear connection. We
prove now that this element is Dirac, and that this is a bijective correspondence.

\begin{proposition}\label{proposition_one_to_one_lin_Dirac}
There is a one to one correspondence between
\begin{itemize}
\item Lie algebroid structure on structure on $A=TS\oplus E^*$ with anchor $pr_1$,
\item horizontally nondegenerate linear Dirac structures on $E$.
\end{itemize}
Explicitly,
\begin{itemize}
\item the Lie algebroid corresponding to $L$ is
\[A:=L_{|S},\]
\item the Dirac structure corresponding to $A$ is such that the inclusion
\[(E,L)\subset (A^*,\plin(A))\]
is a backward Dirac map.
\end{itemize}
\end{proposition}

\begin{proof}
Both structures are encoded by elements $\gamma_{\mathrm{lin}}\in\textrm{gr}_1(\widetilde{\Omega}_{E}^2)$,
satisfying $p_S(\gamma_{\mathrm{lin}})=\gamma_S$. Also, as in the proof of the previous lemma, if
$\gamma_{\mathrm{lin}}$ is Dirac, then the Lie algebroid $A:=L_{\gamma_{\mathrm{lin}}|S}$ has bracket
determined by $\gamma_{\mathrm{lin}}$ via formulas (\ref{EQ_Lie_bracket}). Thus it is enough to prove that for a
Lie algebroid structure on $A$, determined by $\gamma_{\mathrm{lin}}$, the inclusion
$(E,L_{\gamma_{\mathrm{lin}}})\subset (A^*,\plin(A))$ is a backward Dirac map. Integrability of
$L_{\gamma_{\mathrm{lin}}}$ (or equivalent $[\gamma_{\mathrm{lin}},\gamma_{\mathrm{lin}}]_{\ltimes}=0$)
will follow automatically.

Denote by $L\subset TE\oplus T^*E$ the pullback of $L_{\plin(A)}$, i.e.\ $L$ is spanned by
\[\alpha+\plin^{\sharp}(A)(\theta), \ \textrm{where} \ \theta\in T^*A^*_{|E} \ \textrm{satisfies}\ \theta_{|TE}=\alpha, \ \plin^{\sharp}(A)(\theta)\in TE.\]
A priori, we know that, at every point in $E$, $L$ is a maximal isotropic subspace, but we don't know if it is a smooth
subbundle. If this is the case, then it is automatically a Dirac structure. It is enough to check that
$L_{\gamma_{\mathrm{lin}}}\subset L$, since, by comparing dimensions, $L_{\gamma_{\mathrm{lin}}}=L$, so $L$
is smooth and $L_{\gamma_{\mathrm{lin}}}$ is Dirac.

The cotangent bundle of $A^*=T^*S\oplus E$ is spanned by the differentials of functions of the following type
\begin{itemize}
\item $\widetilde{X}\in C^{\infty}(A^*)$, the linear function corresponding to $X\in\mathfrak{X}(S)$,
\item $\widetilde{\xi}\in C^{\infty}(A^*)$, the linear function corresponding to $\xi\in\Gamma(E^*)$,
\item $\widetilde{f}\in C^{\infty}(A^*)$, the pullback of $f\in C^{\infty}(S)$,
\end{itemize}
and the nonzero Poisson brackets of such functions are
\[\{\widetilde{X},\widetilde{Y}\}=\widetilde{[X,Y]}+\widetilde{\mathbb{F}_{\mathrm{lin}}(X,Y)}, \ \{\widetilde{X},\widetilde{\xi}\}=\widetilde{\nabla_X(\xi)}, \]
\[\{\widetilde{X},\widetilde{f}\}=\widetilde{L_X(f)}, \ \{\widetilde{\xi},\widetilde{\eta}\}=\widetilde{\pi^{\mathrm{v}}_{\mathrm{lin}}(\xi,\eta)}.\]
This shows that $H_{\widetilde{X}}$ preserves the function in $p^*(C^{\infty}(E))$, where $p:A^*\to E$ is the
projection, or equivalently it is $p$-projectable, and clearly it projects to $\mathrm{hor}(X)$. This implies the following
relation
\begin{align*}
\plin(A)(d\widetilde{X}&+p^*(\mathbb{F}_{\mathrm{lin}}^{\sharp}(X)),d\widetilde{Y})=\widetilde{[X,Y]}+\widetilde{\mathbb{F}_{\mathrm{lin}}(X,Y)}-p^*(\mathbb{F}_{\mathrm{lin}}^{\sharp}(X))(H_{\widetilde{Y}})=\\
&=\widetilde{[X,Y]}+\widetilde{\mathbb{F}_{\mathrm{lin}}(X,Y)}-p^*(\mathbb{F}_{\mathrm{lin}}(X,\mathrm{hor}(Y)))=\widetilde{[X,Y]}.
\end{align*}
Restricting to $E\subset A^*$, this gives
\[\plin(A)(d\widetilde{X}+p^*(\mathbb{F}_{\mathrm{lin}}^{\sharp}(X)),d\widetilde{Y})_{|E}=0.\]
Since the differentials $d\widetilde{Y}$ span $(TE)^{\circ}$, this implies that
\[\plin(A)^{\sharp}(d\widetilde{X}+p^*(\mathbb{F}_{\mathrm{lin}}^{\sharp}(X)))_{|E}\in TE.\]
We claim that this vector field equals $\mathrm{hor}(X)$. Since $H_{\widetilde{X}}$ projects to $\mathrm{hor}(X)$,
it suffices to check that $\plin(A)^{\sharp}(p^*(\mathbb{F}_{\mathrm{lin}}^{\sharp}(X)))_{|E}=0$. But this is
clear, since $p^*(\mathbb{F}_{\mathrm{lin}}^{\sharp}(X))$ is in the span of $d\widetilde{f}$, and
\[\plin(A)(d\widetilde{f},d\widetilde{\xi})=0, \ \ \ \plin(A)(d\widetilde{f},d\widetilde{g})=0.\]
Since $d\widetilde{X}_{|E}=0$, we have shown that the following elements are in $L$
\[\mathrm{hor}(X)+\mathbb{F}_{\mathrm{lin}}^{\sharp}(X)\in L, \ \forall \ X\in\mathfrak{X}(S).\]
Consider now $\alpha\in \Gamma(H^{\circ})\subset \Omega^1(E)$. Since $H_{\widetilde{X}}$ projects to
$\mathrm{hor}(X)$, it follows that $p^*(\alpha)(H_{\widetilde{X}})=0$, or equivalently that
$\plin(A)(p^*(\alpha),d\widetilde{X})=0$. As before, this implies that $\plin^{\sharp}(A)(p^*(\alpha))_{|E}\in TE$.
Also, for any $\xi\in \Gamma(E^*)$, by the formulas for the Poisson bracket, we have that
\[\plin(A)(p^*(\alpha),d\widetilde{\xi})_{|E}=\pi^{\mathrm{v}}_{\mathrm{lin}}(\alpha,d\xi),\]
hence $\plin(A)^{\sharp}(p^*(\alpha))_{|E}=\pi_{\mathrm{lin}}^{\mathrm{v},\sharp}(\alpha)$. So, we obtain that
\[\pi_{\mathrm{lin}}^{\mathrm{v},\sharp}(\alpha)+\alpha\in L, \ \ \forall \alpha\in\Gamma(H^{\circ}),\]
and this finishes the proof.
\end{proof}

The Lie algebroid $A_S$ of the leaf has a canonical representation on the conormal bundle of the leaf, $\nu_S^*=E^*$,
since this coincides with the kernel of its anchor. The cohomology of $A_S$ with coefficients in this representation (see
subsection \ref{subsection_Lie_algebroid_cohomology}) can also be computed using our algebraic tools.

\begin{lemma}\label{cohom-restricted-algebroid}
For any $l\geq 0$, the complex $(\Omega^{\bullet}(A_S, \mathcal{S}^l\nu_S^*), d_{A_S})$ computing the
cohomology of $A_S$ with coefficients in the $l$-th symmetric power of $\nu_{S}^{*}= E^*$ is canonically
isomorphic to the complex $(\textrm{gr}_{l}(\Omega_{E}^{\bullet}), [d^{1}_{S}\gamma, \cdot]_{\ltimes})$.
\end{lemma}

\begin{proof}
We will use Proposition \ref{Proposition_isomorphic_complexes} from chapter \ref{ChFormalRigidity}, which states
that the complex $(\Omega^{\bullet}(A_S, \mathcal{S}^l\nu_S^*), d_{A_S})$ is canonically isomorphic to the complex
\begin{equation}\label{EQ_complex}
(I^{l}_S\mathfrak{X}^{\bullet}(E)/I^{l+1}_S\mathfrak{X}^{\bullet}(E),d_{\pi}^l),
\end{equation}
where $I_S$ is the vanishing ideal of $S$, $I^k_S$ is its $k$-th power and $d_{\pi}^l$ is the differential induced by
$d_{\pi}=[\pi,\cdot]$. Since the map $\tau_{\pi}:\mathfrak{X}^{\bullet}(E)\to \Omega^{\bullet}_E$ comes from a
vector bundle isomorphism, it restricts to an isomorphism
\[I^{l}_S\mathfrak{X}^{\bullet}(E)\cong I^{l}_S\Omega^{\bullet}_E,\]
thus, by Proposition \ref{PoissonCohomology}, it induces an isomorphism between (\ref{EQ_complex}) and
\begin{equation}\label{EQ_complex2}
(I^{l}_S\Omega_E^{\bullet}/I^{l+1}_S\Omega_E^{\bullet},d_{\gamma}^l),
\end{equation}
where $d_{\gamma}^l$ is the differential induced by $d_{\gamma}=[\gamma,\cdot]_{\ltimes}$. As vector spaces, this
quotient can be identified with $\textrm{gr}_{l}(\Omega_{E}^{\bullet})$, since
\[I^{l}_S\Omega_E^{\bullet}=I^{l+1}_S\Omega_E^{\bullet}\oplus  \textrm{gr}_{l}(\Omega_{E}^{\bullet}),\]
with projection $d_S^l:I^{l}_S\Omega_E^{\bullet}\to \textrm{gr}_{l}(\Omega_{E}^{\bullet})$. Using the
Newton-type formula from Lemma \ref{Newton-formula}, we obtain the expected differential on
$\textrm{gr}_{l}(\Omega_{E}^{\bullet})$
\[d_{\gamma}^l(u)=d_S^l[\gamma,u]_{\ltimes}=[d_S^1\gamma,d_S^lu]_{\ltimes}=[d_S^1\gamma,u]_{\ltimes}, \ \ u=d_S^lu\in \textrm{gr}_{l}(\Omega_{E}^{\bullet}).\qedhere\]
\end{proof}

We end this section by proving the results stated in subsection \ref{Subsection_the_local_model_the_general_case}.

\begin{proof}[Proof of Lemma \ref{Lemma_pi_la_omega_S}]
We have to check that the map $I+p^*(\omega_S)^{\sharp}\circ \plin^{\sharp}(A)$ is invertible. Let $X,Y\in TE$. We
have that
\[\langle p^*(\omega_S)^{\sharp}\circ \plin^{\sharp}(A)\circ p^*(\omega_S)^{\sharp}(X), Y\rangle=-\plin(A)(\xi_X,\xi_Y),\]
where $\xi_X:=p^*(\omega_S)^{\sharp}(X)$ and $\xi_Y:=p^*(\omega_S)^{\sharp}(Y)$. Now, $\xi_X,\xi_Y$ are in the
span of $d \widetilde{f}$, for $f\in C^{\infty}(S)$, and since $\plin(A)(d\widetilde{f},d\widetilde{g})=0$, it follows
that $\plin(A)(\xi_X,\xi_Y)=0$. Hence
\[p^*(\omega_S)^{\sharp}\circ \plin^{\sharp}(A)\circ p^*(\omega_S)^{\sharp}=0.\]
This shows that $p^*(\omega_S)^{\sharp}\circ \plin^{\sharp}(A)$ is nilpotent, thus $I+p^*(\omega_S)^{\sharp}\circ
\plin^{\sharp}(A)$ is invertible. Moreover, the resulting Poisson structure is given by
\[\plin^{p^*(\omega_S),\sharp}(A)=\plin^{\sharp}(A)-\plin^{\sharp}(A)\circ p^*(\omega_S)^{\sharp}\circ \plin^{\sharp}(A).\qedhere\]
\end{proof}

\begin{proof}[Proof of Proposition \ref{proposition_restricting_to_cosymplectic}]
Using the splitting $\sigma$, we identify the Lie algebroid $A$ with $TS\oplus E^*$ with anchor the first projection,
where $E=K^*$. Let $\gamma_{\mathrm{lin}}$ be the corresponding Dirac element. By Proposition
\ref{proposition_one_to_one_lin_Dirac}, we have that the inclusion $(E,L_{\gamma_{\mathrm{lin}}})\subset
(A^*,\plin(A))$ is a backward Dirac map. Gauge transforming both sides by $p^*(\omega_S)$ preserves this relation,
thus $(E,L_{\gamma_{\mathrm{lin}}+\omega_S})\subset (A^*,\plin^{p^*(\omega_S)}(A))$ is a backward Dirac map.
The Poisson support of $\gamma_{\mathrm{lin}}+\omega_S$ coincides with the open $N(A)$, where $E$ is Poisson
transversal, and it includes $S$. The induced Poisson structure $\pi_A$, is the one corresponding to
$\gamma_{\mathrm{lin}}+\omega_S$, therefore, by Lemma \ref{lma-last} it has $(S,\omega_S)$ as a symplectic leaf,
and by Lemma \ref{Lemma_algebroid_cutarica} it has the expected Lie algebroid.
\end{proof}

\begin{proof}[Proof of Proposition \ref{Proposition_reconcile}]
Clearly, we may assume that $\Psi=\mathrm{Id}$ and that $E=\nu_S=M$. Denote by $\gamma$ the corresponding
Dirac element, defined on some open around $S$. By the proof of Proposition
\ref{proposition_restricting_to_cosymplectic}, we have that $\pi_{A_S}$ coincides with the Poisson structure
$\plin(S)$ corresponding to the Dirac element $j^1_S\gamma=\gamma_{|S}+d^1_S\gamma$, and by Proposition
\ref{Proposition_reconcile_2}, this coincides with $\pi_0$. We also note that, on $N(A_S)$, $\pi_0$ can be given as the
limit $\lim_{t\to 0}\pi_t$. To see this, note that, by the continuity of the family $\gamma_t$ from the proof of
Proposition \ref{Proposition_reconcile_2}, for every point $e\in N(A_S)=\mathrm{supp}(j^1_S\gamma)$, we find
$\epsilon>0$, such that $e\in \mathrm{supp}(\gamma_t)$ for all $t\in [0,\epsilon)$, hence $\pi_t$ is defined around $e$
for $t\in [0,\epsilon)$, and so $\pi_{0,e}=\lim_{t\to 0}\pi_{t,e}$.
\end{proof}

\begin{proof}[Proof of Proposition \ref{proposition_splitting_equivalent}]
Let $\sigma:A\to E^*=K$ be a splitting inducing a decomposition $A=TS\oplus E^*$. We denote by $[\cdot,\cdot]_A$
the resulting Lie bracket on $TS\oplus E^*$ and the induced linear Dirac element by
\[\gamma_{\mathrm{lin}}=(\plin^{\mathrm{v}},\Gamma_{\mathrm{lin}},\mathbb{F}_{\mathrm{lin}})\in \textrm{gr}_1(\widetilde{\Omega}_{E}^2).\]
A second splitting is given by an element $\lambda\in \Omega^1(S,E^*)=\textrm{gr}_1(\Omega_{E}^{1,0})$, i.e.\
\[\sigma_{\lambda}:A=TS\oplus E^*\rmap E^*,  \ (X,\alpha)\mapsto \lambda(X)+\alpha.\]
Let $[\cdot,\cdot]_{\lambda}$ be the corresponding Lie bracket and the linear Dirac element be
\[\gamma_{\mathrm{lin}}^{\lambda}=(\plin^{\mathrm{v},\lambda},\Gamma_{\mathrm{lin}}^{\lambda},\mathbb{F}_{\mathrm{lin}}^{\lambda})\in \textrm{gr}_1(\widetilde{\Omega}_{E}^2).\]
The Lie algebroid structures are related by the bundle isomorphism
\[\varphi_{\lambda}:(TS\oplus E^*,[\cdot,\cdot]_{\lambda})\diffto (TS\oplus E^*,[\cdot,\cdot]_{A}), \  \ (X,\alpha)\mapsto (X,\lambda(X)+\alpha).\]
We determine now the Dirac element $\gamma_{\mathrm{lin}}^{\lambda}$. For $\alpha,\beta\in\Gamma(E^*)$, we
have
\begin{align*}
[\alpha,\beta]_{\lambda}&=\varphi_{\lambda}^{-1}[\varphi_{\lambda}(\alpha),\varphi_{\lambda}(\beta)]_{A}=\plin^{\mathrm{v}}(\alpha,\beta),
\end{align*}
thus $\plin^{\mathrm{v},\lambda}=\plin^{\mathrm{v}}$. For $X\in\mathfrak{X}(S)$ and $\alpha\in\Gamma(E^*)$,
we have
\begin{align*}
[X,\alpha]_{\lambda}&=\varphi_{\lambda}^{-1}[X+\lambda(X),\alpha]_{A}=\nabla_X(\alpha)+\plin^{\mathrm{v}}(\lambda(X),\alpha)=\nabla^{\lambda}_X(\alpha).
\end{align*}
The extra term can be also written as
$\plin^{\mathrm{v}}(\lambda(X),\alpha)=L_{-[\lambda,\plin^{\mathrm{v}}](X)}(\alpha)$, therefore the corresponding
linear connection is
\[\Gamma_{\mathrm{lin}}^{\lambda}=\Gamma_{\mathrm{lin}}-[\lambda,\plin^{\mathrm{v}}]\in \textrm{gr}_1(\widetilde{\Omega}_{E}^{1,1}).\]
For $X,Y\in\mathfrak{X}(S)$, we have that
\begin{align*}
[X,Y]_{\lambda}&=\varphi_{\lambda}^{-1}[X+\lambda(X),Y+\lambda(Y)]_{A}=[X,Y]+\mathbb{F}_{\mathrm{lin}}(X,Y)+\\
&+\nabla_X(\lambda(Y))-\nabla_Y(\lambda(X))+\plin^{\mathrm{v}}(\lambda(X),\lambda(Y))-\lambda([X,Y])=\\
&=[X,Y]+\mathbb{F}_{\mathrm{lin}}(X,Y)+[\Gamma_{\mathrm{lin}},\lambda]_{\ltimes}(X,Y)+\plin^{\mathrm{v}}(\lambda(X),\lambda(Y)),
\end{align*}
where we used (\ref{EQ_just_the_horizontal_derivative}). Also the last term can be expressed using the bracket
\[\plin^{\mathrm{v}}(\lambda(X),\lambda(Y))=\frac{1}{2}[\lambda,[\lambda,\plin^{\mathrm{v}}]](X,Y).\]
So, for the 2-from we obtain
\[\mathbb{F}_{\mathrm{lin}}^{\lambda}=\mathbb{F}_{\mathrm{lin}}-[\lambda,\Gamma_{\mathrm{lin}}]_{\ltimes}+\frac{1}{2}[\lambda,[\lambda,\plin^{\mathrm{v}}]].\]
These formulas have a very simple interpretation: the operator
\[ad_{\lambda}:=-[\lambda,\cdot]_{\ltimes}:\widetilde{\Omega}^{p,q}_{E}\rmap \widetilde{\Omega}^{p+1,q-1}_E\]
is nilpotent, therefore its exponential is well defined
\[\exp(ad_{\lambda})(u)=\sum_{n\geq 0}\frac{ad_{\lambda}^n}{n!}(u),\]
and the formulas reduce to $\gamma_{\mathrm{lin}}^{\lambda}=\exp(ad_{\lambda})(\gamma_{\mathrm{lin}})$.
Since $[\lambda,\omega_S]=0$, it follows that the Dirac elements of the two local models are also related by
\[\omega_S+\gamma_{\mathrm{lin}}^{\lambda}=\exp(ad_{\lambda})(\omega_S+\gamma_{\mathrm{lin}}).\]
Consider now the smooth family joining them
\[\gamma_t:=\exp(t ad_{\lambda})(\omega_S+\gamma_{\mathrm{lin}}).\]
Either by using that $\exp(t ad_{\lambda})$ is a Lie algebra automorphism, or that $\gamma_t$ corresponds to the
splitting $\sigma_{t\lambda}$, we deduce that $\gamma_t$ is Dirac, for all $t$. Since $\exp(ad_{\lambda})$ preserves
the spaces $\textrm{gr}_l(\widetilde{\Omega}_{E}^2)$, it follows that $\gamma_{t|S}=\omega_S$, thus $S$ is in the
support of $\gamma_t$, for all $t$. By the Tube Lemma, there exists an open neighborhood $U$ of $S$, such that
$U\subset \mathrm{supp}(\gamma_t)$ for all $t\in[0,1]$. Let $\pi_t$ be the corresponding Poisson structures on $U$.
The time dependent vector field
\[X_t:=\tau_{\pi_t}^{-1}(\lambda)\in\mathfrak{X}(U)\]
vanishes on $S$, so we can choose $V\subset U$, and open neighborhood of $S$, such that the flow of $X_t$ is defined
for all $t\in [0,1]$ as a map $\Phi_t:V\to U$. If we prove that $X_t$ satisfies the homotopy equation
(\ref{EQ_homotopy_equation}), we are done since this equation implies that $\Phi_1$ is a Poisson diffeomorphism
between
\[\Phi_1:(V,\pi_0)\rmap (\Phi_1(V),\pi_1),\]
which is the identity on $S$. Equivalently, by Lemma \ref{HomotopyEquation}, we have to check that $\lambda$
satisfies (\ref{homotopy-equation}), and this is straightforward:
\[\frac{d}{dt}\gamma_t=\frac{d}{dt}\exp(t ad_{\lambda})(\gamma_0)=-[\lambda,\exp(t ad_{\lambda})(\gamma_0)]_{\ltimes}=[\gamma_t,\lambda]_{\ltimes}.\qedhere\]
\end{proof}

\section{Proof of Theorem \ref{Theorem_TWO}, Step 1: Moser's path method}
\label{Proof of the main theorem; step 1: Moser path method}

In this section we use Moser's path method to reduce the proof of Theorem \ref{Theorem_TWO} to some cohomological
equations. The main outcome is Theorem \ref{theorem-1}.

Let $(M, \pi)$ be a Poisson manifold and let $(S,\omega_S)$ be an embedded symplectic leaf. We start by describing
the relevant cohomologies. They are all relatives of the Poisson cohomology groups $H^{\bullet}_{\pi}(M)$. The first
one is, intuitively, the Poisson cohomology of the germ of $(M, \pi)$ around $S$:
\[H^{\bullet}_{\pi}(M)_{S}= \lim_{S\subset U} H^{\bullet}_{\pi_{|U}}(U),\]
where the limit is the direct limit over all opens $U$ around $S$.

The \textbf{Poisson cohomology restricted to $S$},\index{cohomology, Poisson} denoted by
\[H^{\bullet}_{\pi, S}(M),\]
is defined by the complex $(\mathfrak{X}^{\bullet}_{|S}(M), {d_{\pi}}_{|S})$, where
$\mathfrak{X}^{\bullet}_{|S}(M)= \Gamma(\Lambda^{\bullet} TM_{|S})$. Of course, this is just the cohomology
$H^{\bullet}(A_S)$ of the transitive Lie algebroid\index{cohomology, Lie algebroid}
\[ A_{S}:= T^{*}M_{|S},\]
the restriction of the cotangent Lie algebroid of $\pi$ to $S$.

The last relevant cohomology is a version of $H^{\bullet}_{\pi, S}(M)$ with coefficients
\[ H^{\bullet}_{\pi, S}(M, \nu_{S}^{*}):= H^{\bullet}(A_S, \nu_{S}^{*}),\]
where $\nu_S^*$, the conormal bundle of $S$, is canonically a representation of $A_S$. This is also isomorphic to the
cohomology of the complex of multivector fields which vanish at $S$, modulo those which vanish up to first order along
$S$ (see Proposition \ref{Proposition_isomorphic_complexes} in chapter \ref{ChFormalRigidity}).

\begin{theorem}\label{theorem-1}
Let $S$ be an embedded symplectic leaf of a Poisson manifold $(M, \pi)$ and let $\plin(S)$ be the first order
approximation of $\pi$ along $S$ associated to some tubular neighborhood of $S$ in $M$. If
\[ H^{2}_{\pi}(M)_{S}=0, \  H^{1}_{\pi, S}(M)= 0, \  H^{1}_{\pi, S}(M, \nu_{S}^{*})= 0,\]
then, around $S$, $\pi$ and $\plin(S)$ are Poisson diffeomorphic, by a Poisson diffeomorphism which is the identity on
$S$.
\end{theorem}

The rest of this section is devoted to this proof of the theorem, followed by a slight improvement that will be used in the
proof of Proposition \ref{main-cor}.

First of all, by using a tubular neighborhood, we may assume that $M= E$ is a vector bundle over $S$ and that $\pi$ is
horizontally nondegenerate on $E$. Let $\gamma\in \widetilde{\Omega}_{E}^{2}$ be the associated Dirac element
and $j^{1}_{S}\gamma$ its linearization. Consider the path of Dirac elements from the proof of Proposition
\ref{Proposition_reconcile_2}
\begin{equation}\label{EQ_path}
\gamma_t=\gamma_{|S}+\frac{\varphi_t(\gamma)-\gamma_{|S}}{t}, \ \ \textrm{with}\ \gamma_0=j^1_S\gamma, \ \textrm{and} \ \gamma_1=\gamma.
\end{equation}
Let $U$ be an open around $S$ included in the support of $\gamma_t$, for all $t\in[0,1]$, and denote by $\pi_t$ the
corresponding Poisson structure on $U$ (see (\ref{EQ_linearization_explicit})).

We are looking for a family $\Phi_{t}$ of diffeomorphisms defined on a neighborhood of $S$ in $U$, for $t\in [0, 1]$,
such that $\Phi_{t|S}= \textrm{Id}$, $\Phi_0=\textrm{Id}$ and
\begin{equation}
\label{desired-homotopy}
\Phi_{t}^{*}\pi_t=\pi_0=\plin(S)
\end{equation}
for all $t\in [0,1]$. Then $\Phi_1$ will be the desired isomorphism. We will define $\Phi_t$ as the flow of a time
dependent vector field $X_t$, i.e.\ as the solution of:
\[ \frac{d}{dt} \Phi_t(x)= X_t(\mu_t(x)), \ \ \Phi_0(x)= x.\]
Hence we are looking for the time dependent $X_t$ defined on an around $S$. The first condition we require is that
$X=0$ along $S$. This implies that $\Phi_{t|S}= \textrm{Id}$, in particular $\Phi_t$ is defined for all $t\in [0,1]$ on
$S$. Hence by the Tube Lemma, $\Phi_t$ is well-defined up to time 1 on an open $O\subset U$ containing $S$. Finally,
since (\ref{desired-homotopy}) holds at $t=0$, it suffices to require its infinitesimal version
\begin{equation}\label{homotopy-equation0}
\frac{d}{dt}\pi_t=[\pi_t,X_t].
\end{equation}
By Lemma \ref{HomotopyEquation}, this is equivalent to finding a time dependent element $V_t=\tau_{\pi_t}(X_t)\in
\Omega^1_U$, such that
\begin{equation}\label{EQ_homotopy_EQ}
\frac{d}{dt}\gamma_t=[\gamma_t,V_t]_{\ltimes}.
\end{equation}
There is one equation for each $t$ but, since $\gamma_t$ is of a special type, one can reduce everything to a single
equation.

\begin{lemma} Assume that there exists $Z\in \Omega^{1}_{U}$ such that $j_{S}^{1}Z= 0$ and
\begin{equation}\label{EquationV}
[\gamma, Z]_{\ltimes}= \frac{d}{dt}_{|t=1}\gamma_{t}.
\end{equation}
Then $V_t:= t^{-1}\varphi_t(Z)$ satisfies the homotopy equations (\ref{EQ_homotopy_EQ}).
\end{lemma}

\begin{proof} The condition that the first jet of $Z$ along $S$ vanishes, ensures that $V_t$
is a smooth family defined also at $t=0$ and that $V_t$ vanishes along $S$. We check the homotopy equations at all
$t\in (0, 1]$. For the left hand side
\[[\gamma_t,V_t]_{\ltimes}=[\varphi_t(\gamma)+(1-\frac{1}{t})\omega_S,V_t]_{\ltimes}=[\varphi_t(\gamma),\frac{1}{t}\varphi_{t}(Z)]_{\ltimes})=\frac{1}{t}\varphi_t([\gamma,Z]_{\ltimes}),\]
where we used that $\omega_S$ lies in the center of $\Omega_E$ (Lemma \ref{OmegaS-Central}), and that $\varphi_t$
commutes with the brackets. Using the assumption on $Z$, we find
\[ [\gamma_t,V_t]_{\ltimes}=t^{-1}\varphi_t(\dot{\gamma}_{1}).\]
Hence (\ref{EQ_homotopy_EQ}) will follow for $t\in(0,1]$ if we prove the following equation
\begin{equation}\label{Auxiliary}
t^{-1}\varphi_t(\dot{\gamma}_{1})= \dot{\gamma}_t.
\end{equation}
The explicit formula for $\gamma_{t}$ (\ref{EQ_path}) implies that
\[\varphi_t(\gamma_s)=\gamma_{ts}+(1-t^{-1})\omega_S.\]
Taking the derivative with respect to $s$, at $s=1$, we obtain the result.
\end{proof}

The equation for $Z$ in the last lemma lives in the cohomology of $(\Omega_{U}^{\bullet},d_{\gamma})$. Note that
the right hand side of the equation is indeed closed. This follows by taking the derivative at $t= 1$ in $[\gamma_t,
\gamma_t]_{\ltimes}=0$. By Proposition \ref{PoissonCohomology} and the first assumption of Theorem
\ref{theorem-1}, after shrinking $U$ if necessary, we may assume that $\dot{\gamma}_1$ is exact. This ensures the
existence of $Z$. To conclude the proof of the theorem, we still have to show that $Z$ can be corrected so that
$j_{S}^{1}Z= 0$. We will do so by finding $F\in \Omega_{E}^{0}=C^{\infty}(E)$, such that
\[ j^{1}_{S}([\gamma, F]_{\ltimes})= j^{1}_{S}(Z).\]
Then $Z'= Z- [\gamma,F]_{\ltimes}$ will be the correction of $Z$. Since the condition only depends on $j^{1}_{S}F$,
it suffices to look for $F$ of type
\[ F= F_0+ F_1\in \textrm{gr}_{0}(\Omega_{E}^{0})\oplus \textrm{gr}_{1}(\Omega_{E}^{0})= C^{\infty}(S)\oplus \Gamma(E^*).\]
So $F|_{S}= F_0$, $d_{S}^{1}F= F_1$. Using the Newton formula (Lemma \ref{Newton-formula}) and that
$\omega_S, F_0$ are central in $\Omega_E$ (Lemma \ref{OmegaS-Central}), we find
\[j^1_S[\gamma,F]_{\ltimes}=[d^1_S\gamma, F_0]_{\ltimes}+ [d^1_S\gamma, F_1]_{\ltimes}\in \textrm{gr}_{0}(\Omega_{E}^{1})\oplus \textrm{gr}_{1}(\Omega_{E}^{1}),\]
so we have to solve the following cohomological equations:
\begin{equation}\label{F0-F1}
[d_{S}^{1}\gamma, F_0]_{\ltimes}= Z|_{S},\ \ \  [d_{S}^{1}\gamma, F_1]_{\ltimes}= d_{S}^{1}Z.
\end{equation}
By Lemma \ref{cohom-restricted-algebroid}, the relevant cohomologies are precisely the ones assumed to vanish in the theorem. So 
it is enough to show that $Z|_{S}$ and $d_{S}^{1}Z$ are closed with respect to the differential
$[d_{S}^{1}\gamma,\cdot]_{\ltimes}$. These are precisely the first order consequences of the equation
(\ref{EquationV}) that $Z$ satisfies. By (\ref{Auxiliary}), we have that
$\varphi_t(\dot{\gamma}_1)=t\dot{\gamma}_t$, hence $j^1_S(\dot{\gamma}_1)=0$. Applying the Newton formula to
(\ref{EquationV}) gives
\begin{equation}\label{Equation_Z}
[d_{S}^{1}\gamma, Z|_S]_{\ltimes}= 0,\ \  [d_{S}^{1}\gamma, d_{S}^{1}Z]_{\ltimes}+ [d_{S}^{2}\gamma, Z|_{S}]_{\ltimes}= 0.
\end{equation}
Hence $Z|_{S}$ is closed, and so we can find $F_0$ satisfying the first equation in (\ref{F0-F1}). In particular, this
shows that
\[Z|_{S}=[d_{S}^{1}\gamma, F_0]_{\ltimes}=L_{p_S(d_{S}^{1}\gamma)}F_0=dF_0\in\Omega(S).\]
Hence $Z|_{S}$ commutes with $\Omega_E$, and since $d_{S}^{2}\gamma\in\Omega_E$, the second equation of
(\ref{Equation_Z}) becomes $[d_{S}^{1}\gamma, d_{S}^{1}Z]_{\ltimes}=0$. This finishes the proof.

\begin{remark}\rm \label{remark-corollary}
The previous arguments reveal a certain cohomology class related to the linearization problem, which we will describe
now more explicitly. Consider a tubular neighborhood $p:E\to S$. Since any two tubular neighborhoods are isotopic, it
follows that the class in the de Rham cohomology of the ``germ of $M$ around $S$''
\[[p^*(\omega_S)]\in H^{\bullet}(M)_{S}= \lim_{S\subset U} H^{\bullet}(U),\]
is independent of $p$. The chain map between the de Rham complex and the Poisson complex (see subsection
\ref{SSPoissonCoho}) induced by $\pi^{\sharp}$, induces a map
\[H(\pi^{\sharp}):H^{\bullet}(M)_{S}\rmap H^{\bullet}_{\pi}(M)_{S}.\]
Consider the class of $p^*(\omega_S)$ under this map, denoted by
\[[\pi_{|S}]:=H(\pi^{\sharp})[p^*(\omega_S)]=-[\Lambda^2\pi^{\sharp}(p^*(\omega_S))]\in H^2_{\pi}(M)_S.\]

\begin{definition}
\textbf{The linearization class} associated to an embedded leaf $S$ of a Poisson manifold $(M,\pi)$ is defined by
\[l_{\pi,S}:=[\pi]-[\pi_{|S}]\in H^2_{\pi}(M)_S.\]
\end{definition}

\begin{corollary} \label{remark-corollary2} In the Theorem \ref{theorem-1}, the condition
$H^{2}_{\pi}(M)_{S}= 0$ can be replaced by the condition $l_{\pi,S}=0$.
\end{corollary}

\begin{proof} In the proof above we only used the vanishing of
$[\dot{\gamma}_{1}]$, hence, by Lemma \ref{Lemma_derivative_of_pite} it suffices to show that
$[\dot{\pi}_1]=l_{\pi,S}$. Recall that (\ref{EQ_linearization_explicit})
\[\pi_t^{\sharp}=t\mu_t^*\left(\pi^{\sharp}(\mathrm{Id}+(t-1)p^*(\omega_S)^{\sharp}\pi^{\sharp})^{-1}\right).\]
Since $\mu_{e^t}$ is the flow of the Liouville vector field $\mathcal{E}=\sum y_i\frac{\partial}{\partial y_i}$, we
obtain
\[\dot{\pi}_1=\pi+[\mathcal{E},\pi]+\dot{u}_1,\]
where $u_t^{\sharp}=\pi^{\sharp}(\mathrm{Id}+(t-1)p^*(\omega_S)^{\sharp}\pi^{\sharp})^{-1}$. We compute
$\dot{u}_1$ as follows
\begin{align*}
0&=\frac{d}{dt}_{|t=1}\left(u^{\sharp}_t(\mathrm{Id}+(t-1)p^*(\omega_S)^{\sharp}\pi^{\sharp})\right)=\\
&=\dot{u}^{\sharp}_1+\pi^{\sharp}p^*(\omega_S)^{\sharp}\pi^{\sharp}=\dot{u}_1-\Lambda^2\pi^{\sharp}(p^*(\omega_S)).
\end{align*}
This finishes the proof.
\end{proof}
\end{remark}

\section{Proof of Theorem \ref{Theorem_TWO}, Step 2: Integrability}
\label{Proof of the main theorem; step 2: Integrability}

In this section we show that the conditions of Theorem \ref{Theorem_TWO} imply the integrability of the Poisson
structure around the symplectic leaf which, in turn, implies that the cohomological conditions from Theorem
\ref{theorem-1} are satisfied. As in the geometric proof of Conn's linearization theorem \cite{CrFe-Conn}, this step
will be divided into three sub-steps, corresponding to the three subsections:
\begin{itemize}
\item Step 2.1: Integrability implies the needed cohomological conditions.
\item Step 2.2: Existence of ``nice'' symplectic realizations implies integrability.
\item Step 2.3: Proof of the existence of such ``nice'' symplectic realizations.
\end{itemize}
In the first sub-step we will also finish the proof of Proposition \ref{main-cor}.

By integrability we mean here the integrability of the cotangent Lie algebroids $T^*M$, or, equivalently, the
integrability of the Poisson manifold $(M,\pi)$ by a symplectic groupoid (see section \ref{Section_Integrability}).

\subsection{Step 2.1: Reduction to integrability}
\label{Step 2.1: Reduction to integrability}

The idea of using integrability in order to prove the vanishing of the cohomologies from Theorem \ref{theorem-1} is
very simple. First of all, all the cohomologies involved are Lie algebroid cohomologies. Then, with the Lie groupoid at
hand, one can try to integrate algebroid cocycles (in the relevant cohomologies) to the groupoid level. This is the idea of
Van Est maps, used by Van Est in the case of Lie groups to prove Lie's Theorem III. The generalization to Lie groupoids
was carried out in \cite{Cra} (Theorem 4, which we recall in subsection \ref{Subsection_Van_Est}). At the groupoid
level, if this is compact, or just proper, one can use averaging in order to prove the triviality of groupoid cocycles
(Proposition 1 in \cite{Cra}, which we recall in subsection
\ref{subsection_differentiable_cohomology_of_Lie_groupoids}).\index{integrable Poisson manifold} The outcome is the
following.

\begin{theorem}\label{TheoremRedInt}
Let $(M,\pi)$ be a Poisson manifold, $x\in M$, and let $S$ be the symplectic leaf through $x$. If $P_x$, the homotopy
bundle at $x$, is smooth and compact, then
\[\ H^{1}_{\pi, S}(M)= 0, \ H^{1}_{\pi, S}(M, \nu_{S}^{*})= 0.\]
If moreover $H^2(P_x)= 0$ and $S$ admits an open neighborhood $U$ whose associated Weinstein groupoid
$\mathcal{G}(U, \pi|_{U})$ is smooth and Hausdorff, then also
\[ H^{2}_{\pi}(M)_{S}=  0.\]
\end{theorem}

\begin{proof}
The first part of the proof is completely similar to that of Theorem 2 in \cite{CrFe-Conn}. The conditions on $P_x$
imply that $\mathcal{G}(A_S)$, the 1-connected groupoid of $A_S= T^{*}M_{|S}$, is smooth and compact. Hence, its
differentiable cohomology with any coefficients vanishes. Since its $s$-fibers are 1-connected, the Van Est map with
coefficients is an isomorphism in degree 1. So, also the cohomology of $A_S$, with any coefficients, vanishes in degree
1.\index{differentiable cohomology}\index{Van Est map}

For the second part, it suffices to show that, for any open $W\subset U$ containing $S$, there exists a smaller one $V$,
containing $S$, such that $H^{2}_{\pi}(V)= 0$. Proceeding as in the first part, it suffices to find $V$ for which
$\mathcal{G}(V, \pi_{|V})$ has $s$-fibers which are compact and cohomologically 2-connected. Let
$\mathcal{G}\subset \mathcal{G}(U,\pi_{|U})$ be the set of arrows with source and target inside $W$, and for which
both the $s$-fiber and the $t$-fiber are diffeomorphic to $P_x$. By local Reeb stability applied to the foliation by the
$s$-fibers (and $t$-fibers respectively), we see that all four conditions are open, therefore
$\mathcal{G}\subset\mathcal{G}(U,\pi_{|U})$ is open, and by assumption, all arrows above $S$ are in $\mathcal{G}$.
By the way $\mathcal{G}$ was defined, we see that it is an open subgroupoid over the invariant open
$V:=s(\mathcal{G})\supset S$, thus it integrates $T^*V$. Since all its $s$-fibers are 1-connected, we have that
$\mathcal{G}=\mathcal{G}(V,\pi_{|V})$, and this finishes the proof.
\end{proof}

\begin{proof}[End of the proof of Proposition \ref{main-cor}]
We will adapt the previous argument, making use of Corollary \ref{remark-corollary2}. We have to show that
\[l_{\pi,S}:=[\pi]-[\pi_{|S}]=0 \in H^2_{\pi}(M)_S.\]
We show that, for any tubular neighborhood $p: W\to S$ of $S$ with $W\subset U$, there exists a smaller one $V$,
containing $S$, such that
\begin{equation*}
[\pi]- [\pi_{|S}]=0 \in H^{2}_{\pi}(V).
\end{equation*}
Let $V$ be as in the previous proof, which we assume to be connected (take the component containing $S$ in $V$).
Since $\mathcal{G}(V,\pi_{|V})$ is still proper, it suffices to show that the class above is in the image of the Van Est
map of $\mathcal{G}(V,\pi_{|V})$ (in degree $2$). By Corollary 2 \cite{Cra} (which we recall in subsection
\ref{Subsection_Van_Est}), this image consists of elements $[\omega]\in H^{2}_{\pi}(V)$, such that
\[\int_{\gamma}J(\omega)=0,\]
for all $2$-spheres $\gamma$ in the $s$-fibers of $\mathcal{G}(V,\pi_{|V})$. Hence it suffices to check this for
$[\pi]-[\pi_{|S}]$. Let $\widetilde{\omega}_S= p^*(\omega_S)_{|V}$ and let $\Omega$ be the symplectic form on
$\mathcal{G}(V,\pi_{|V})$. We compute now $J(\Lambda^2\pi^{\sharp}(\widetilde{\omega}_S))$. Note that
\[\Lambda^2\pi^{\sharp}(\widetilde{\omega}_S)(\alpha,\beta)=\widetilde{\omega}_S(\pi^{\sharp}(\alpha),\pi^{\sharp}(\beta)).\]
Recall from section \ref{Section_symplectic_groupoids}, that, under the identification between the Lie algebroid of
$\mathcal{G}(V,\pi_{|V})$ and $T^*V$ given by $-\Omega^{\sharp}$, we have that $dt_{|T^*V}=\pi^{\sharp}$.
Thus, for $\alpha,\beta \in T_g^s\mathcal{G}(V,\pi_{|V})$, we obtain
\begin{align*}
J(\Lambda^2&\pi^{\sharp}(\widetilde{\omega}_S))(\alpha,\beta)=\widetilde{\omega}_S(\pi^{\sharp}(dr_g(\alpha)),\pi^{\sharp}(dr_g(\beta)))=\\
&=\widetilde{\omega}_S(dt\circ dr_g(\alpha),dt\circ dr_g(\beta))=\widetilde{\omega}_S(dt(\alpha),dt(\beta))=t^*(\widetilde{\omega}_S)(\alpha,\beta).
\end{align*}
Hence, $J(\Lambda^2\pi^{\sharp}(\widetilde{\omega}_S))$ is the restriction to the $s$-fibers of
$t^*(\widetilde{\omega}_S)$. On the other hand, it is well known (see e.g.\ \cite{WeinXu}), and can also be easily
checked using the properties listed in section \ref{Section_symplectic_groupoids}, that $J(\pi)$ is the restriction to the
$s$-fibers of $\Omega$. Thus we obtain that $J([\pi]-[\pi_{|S}])$ is the restriction to the $s$-fibers of the closed
$2$-form
\[\omega:=\Omega+t^*(\widetilde{\omega}_{S}).\]
Notice also that $\omega$ vanishes over the $s$-fibers over points $x\in S$. This follows from the fact that $t$ is an
anti-Poisson map, and so $\Omega_{|s^{-1}(x)}=-t^*(\omega_S)$. Now let $\gamma$ a 2-sphere in an $s$-fiber of
$\mathcal{G}(V)$. Since $V$ is connected, we can find a homotopy between $\gamma$ and a 2-sphere $\gamma_1$
which lies in an $s$-fiber over $S$. Since $\omega$ is closed, we have that
$\int_{\gamma}\omega=\int_{\gamma_1}\omega$, and since the restriction of $\omega$ to $s$-fibers over $S$
vanishes, it follows that $\int_{\gamma_1}\omega=0$. This ends the proof.
\end{proof}

\subsection{Step 2.2: Reduction to the existence of ``nice'' symplectic realizations}

Next, we show that the integrability condition from \ref{TheoremRedInt} is implied by the existence of a symplectic
realization with some specific properties.\index{symplectic realization}

We will use the following notation. Given a symplectic realization $\mu$, we denote by $\mathcal{F}(\mu)$ the
foliation defined by the fibers of $\mu$. By Libermann's theorem, the symplectic orthogonal
$T\mathcal{F}(\mu)^{\perp}$ is involutive, and we denote by $\mathcal{F}(\mu)^{\perp}$ the underlying foliation.

\begin{theorem}\label{theorem-step-2.2} Let $(M, \pi)$ be a Poisson manifold and let $S$ be a symplectic leaf. Assume that there exists a symplectic realization
\[ \mu: (\Sigma, \Omega)\rmap (U, \pi_{|U})\]
of some open neighborhood $U$ of $S$ in $M$ such that any leaf of the foliation $\mathcal{F}^{\perp}(\mu)$ which
intersects $\mu^{-1}(S)$ is compact and 1-connected.

Then there exists an open neighborhood $V\subset U$ of $S$ such that the Weinstein groupoid $\mathcal{G}(V,
\pi_{|V})$ is Hausdorff and smooth.
\end{theorem}

\begin{proof}
We may assume that all leaves of $\mathcal{F}(\mu)^{\perp}$ are compact and 1-connected. Otherwise, we replace
$\Sigma$ by $\Sigma'$ and $U$ by $U'= \mu(\Sigma')$, where $\Sigma'$ is defined as the set of points $y\in\Sigma$
with the property that the leaf of $\mathcal{F}(\mu)^{\perp}$ through $y$ is compact and 1-connected. Local Reeb
stability implies that $\Sigma'$ is open in $\Sigma$. The hypothesis implies that $\mu^{-1}(S)\subset \Sigma'$ and,
since $\mu$ is open, $U'$ is an open neighborhood of $S$. Clearly, we may also assume that $U= M$.

Hence we have a symplectic realization
\[ \mu: (\Sigma, \Omega)\rmap (M, \pi)\]
with the property that all leaves of $\mathcal{F}(\mu)^{\perp}$ are compact and 1-connected. We claim that
$\mathcal{G}(M, \pi)$ has the desired properties. By Theorem 8 in \cite{CrFe2}, if the symplectic realization $\mu$ is
complete, then $\mathcal{G}(M, \pi)$ is smooth. The compactness assumption on the leaves of
$\mathcal{F}(\mu)^{\perp}$ implies that $\mu$ is complete since the Hamiltonian vector fields of type
$X_{\mu^*(f)}$ are tangent to these leaves. For Hausdorffness, we take a closer look to the argument of \cite{CrFe2}.
It is based on a natural isomorphism of groupoids
\[\mathcal{G}(M, \pi)\times_{M}\Sigma \cong \mathcal{G}(\mathcal{F}(\mu)^{\perp}),\]
where the left hand side is the fibered product over $s$ and $\mu$, and the right hand side is the homotopy groupoid of
the foliation $\mathcal{F}(\mu)^{\perp}$- obtained by putting together the homotopy groupoids of all the leaves. Since
homotopy groupoids are always smooth, \cite{CrFe2} concluded that $\mathcal{G}(M, \pi)$ is smooth. In our case, the
leaves of $\mathcal{F}(\mu)^{\perp}$ are 1-connected, hence the homotopy groupoid is a subgroupoid of $M\times
M$. So it is Hausdorff, and this implies that also $\mathcal{G}(M, \pi)$ is Hausdorff.
\end{proof}

\subsection{Step 2.3: the needed symplectic realization}
\label{Step 2.3: the needed symplectic realization}

To finish the proof, we still have to provide a symplectic realization as in Theorem \ref{theorem-step-2.2}. This will be
constructed using the methods from section \ref{Constructing symplectic realizations from transversals in the
manifold of cotangent paths}. We will use the notations:\index{cotangent path}
\begin{itemize}
\item $\mathcal{X}= P(T^*M)$ is the Banach manifold of cotangent paths.
\item $\mathcal{F}= \mathcal{F}(T^*M)$ is the foliation on $\mathcal{X}$ given by the equivalence relation of cotangent homotopy; it is a smooth foliation of finite
codimension.
\item $\mathcal{Y}= \tilde{s}^{-1}(S)\subset \mathcal{X}$, the submanifold of $\mathcal{X}$ sitting above $S$.
Note that this is the same as the manifold $P(A_S)$ of $A$-paths of the algebroid $A_S=T^*M_{|S}$.
\item $\mathcal{F}_{\mathcal{Y}}= \mathcal{F}_{|\mathcal{Y}}$, the restriction of $\mathcal{F}$ to $\mathcal{Y}$.
This coincides with the foliation $\mathcal{F}(A_S)$ associated to the algebroid $A_S$ \cite{CrFe1}, and
$\mathcal{Y}/\mathcal{F}_{\mathcal{Y}}$ is the groupoid $\mathcal{G}(A_S)$ of $A_S$.
\item We will denote $B=\mathcal{G}(A_S)$. By the assumptions of Theorem \ref{Theorem_TWO}, $B$ is smooth and compact.
\end{itemize}

By Proposition \ref{PropTrnasversal}, every transversal $\mathcal{T}$ to $\mathcal{F}$ inherits a symplectic
structure $\Omega_{|\mathcal{T}}$, such that $\sigma=\tilde{s}_{|\mathcal{T}}:\mathcal{T}\to M$ is a symplectic
realization of its image. Moreover, $\Omega_{|\mathcal{T}}$ is invariant under the induced holonomy action of
$\mathcal{F}$. Our strategy is to produce such a $\mathcal{T}$ and an equivalence relation $\sim$ on $\mathcal{T}$,
weaker than the holonomy relation, such that the quotient $\mathcal{T}/\sim$ is smooth. Then
$\Omega_{|\mathcal{T}}$ will descend to a symplectic structure on $\mathcal{T}/\sim$ and $\sigma$ (which is
invariant under holonomy) will induce the required symplectic realization.

As in the appendix in \cite{CrFe-Conn}, we will use the following technical lemma:

\begin{proposition}
\label{technical} Let $\mathcal{F}$ be a foliation of finite codimension on a Banach manifold $\mathcal{X}$ and let
$\mathcal{Y}\subset \mathcal{X}$ be a submanifold which is saturated with respect to $\mathcal{F}$ (i.e.\ each leaf
of $\mathcal{F}$ which hits $\mathcal{Y}$ is contained in $\mathcal{Y}$). Assume that:
\begin{itemize}
\item[(H0)] The holonomy groups of the foliation $\mathcal{F}$ at the points of $\mathcal{Y}$ are trivial.
\item[(H1)] $\mathcal{F}_{\mathcal{Y}}:= \mathcal{F}_{|\mathcal{Y}}$ is induced by a submersion $p: \mathcal{Y}\to B$, with $B$-compact.
\item[(H2)] The fibration $p: \mathcal{Y}\to B$ is locally trivial.
\end{itemize}
Then one can find:
\begin{enumerate}[(i)]
\item a transversal $\mathcal{T}\subset \mathcal{X}$ to $\mathcal{F}$ such that $\mathcal{T}_{\mathcal{Y}}:= \mathcal{Y}\cap \mathcal{T}$ is a complete transversal to $\mathcal{F}_{\mathcal{Y}}$ (i.e.\ it intersects each leaf of $\mathcal{F}_{\mathcal{Y}}$ at least once).
\item a retraction $r: \mathcal{T}\to \mathcal{T}_{\mathcal{Y}}$.
\item an action of the holonomy of $\mathcal{F}_{\mathcal{Y}}$ on $r$ along the leaves of $\mathcal{F}$.
\end{enumerate}
Moreover, the quotient of $\mathcal{T}$ by this action is a smooth Hausdorff manifold.
\end{proposition}

In our case, once we make sure that the lemma can be applied, the resulting quotient $\Sigma$ of $\mathcal{T}$ will
produce the symplectic realization required in Theorem \ref{theorem-step-2.2}. To see this, notice that by
construction, $\sigma^{-1}(S)= \mathcal{G}(A_S)\subset \Sigma$, and by Proposition \ref{PropTrnasversal}, the
symplectic orthogonals to the $\sigma$ fibers are the $\tau$ fibers. Now on $\mathcal{G}(A_S)$, $\tau$ becomes the
target map, thus all its fibers are diffeomorphic to $P_x$, in particular they are compact and 1-connected.

(H1) is clear since $B$ is smooth and compact. We first prove a lemma:

\begin{lemma}\label{Localy-triv}
Let $M$ be a finite dimensional manifold, $x_0\in M$, and denote by $\mathrm{Path}(M, x_0)$ the Banach manifold of
$C^2$-paths in $M$ starting at $x_0$. Then
\[ \epsilon: \mathrm{Path}(M, x_0)\rmap M, \ \gamma\mapsto \gamma(1)\]
is a locally trivial fiber bundle.
\end{lemma}

\begin{proof}
For $x\in M$, we construct open neighborhoods $U\subset V$ and a smooth family of diffeomorphisms
\[\phi_{y,t}:M\diffto M, \ \  \mathrm{ for }\ y\in U,\  t\in\mathbb{R},\]
such that $\phi_{y,t}$ is supported inside $V$, $\phi_{y,0}=\mathrm{Id}_M$ and $\phi_{y,1}(x)=y$. Then the
required trivialization over $U$ is given by
\[\tau_U:\epsilon^{-1}(x)\times U\diffto \epsilon^{-1}(U),\quad \tau_U(\gamma, y)(t)=\phi_{y,t}(\gamma(t)),\]
with inverse
\[\tau^{-1}_{U}(\gamma)(t)=(\phi^{-1}_{\gamma(1),t}(\gamma(t)),
\gamma(1)).\] The construction of such maps is clearly a local issue, thus we may assume that $M=\mathbb{R}^m$,
with $x=0$ and $U=B_1(0)$, $V=B_2(0)$, the balls of radii 1 and 2 respectively. Consider $f\in
C^{\infty}(\mathbb{R}^m)$, supported inside $B_2(0)$, with $f_{|B_1(0)}=1$. Let $\phi_{y,t}$ be the flow at time
$t$ of the compactly supported vector field $X_y:=f\overrightarrow y$, where $\overrightarrow y$ represents the
constant vector field on $\mathbb{R}^m$ corresponding to $y\in B_1(0)$. Then $\phi_{y,t}$ satisfies all requirements.
\end{proof}

Next, we denote by $\mathrm{Path}^{s}(\mathcal{G}(A_S), 1)$ the Banach manifold of $C^2$-paths $\gamma$ in
$\mathcal{G}(A_S)$ starting at some unit $1_x$ and satisfying $s\circ\gamma=x$. Proposition 1.1 of \cite{CrFe1}
identifies our bundle $p: \mathcal{Y}\to B$ with the bundle
\begin{equation}\label{map_epsilon}
\widetilde{\epsilon}: \mathrm{Path}^s(\mathcal{G}(A_S), 1)\rmap \mathcal{G}(A_S), \ \gamma\mapsto \gamma(1).
\end{equation}
Hence, the following implies (H2).

\begin{lemma} For a source locally trivial Lie groupoid $\mathcal{G}$, the map $\widetilde{\epsilon}$ (\ref{map_epsilon}) is a locally trivial fiber bundle.
\end{lemma}

\begin{proof}
Consider $g_0\in\mathcal{G}$, $x_0=s(g_0)$. Consider a local trivialization of $s$ over an open $U$, with $x_0\in U$,
$\tau:s^{-1}(U)\cong U\times s^{-1}(x_0)$. Since the unit map is transverse to $s$, we may assume that
$\tau(1_x)=(x,1_{x_0})$ for all $x\in U$. Left composing with $\tau$ induces a diffeomorphism
$\tau_{*}:\widetilde{\epsilon}^{-1}(s^{-1}(U))\cong U\times \mathrm{Path}(s^{-1}(x_0),1_{x_0})$, under which
$\widetilde{\epsilon}$ becomes
\[\textrm{Id}\times \epsilon: U\times\mathrm{Path}(s^{-1}(x_0),1_{x_0})\rmap U\times s^{-1}(x_0),\]
where $\epsilon$ is the map from the previous lemma. By that lemma, for every $g_0\in s^{-1}(x_0)$, we can find
$V\subset s^{-1}(x_0)$ an open around $g_0$, over which $\epsilon$ can be trivialized. Thus, $\widetilde{\epsilon}$
can be trivialized over $\tau^{-1}(U\times V)$.
\end{proof}

We still have to check (H0). First we note the following.

\begin{lemma} \label{lemma-Getzler1} For any leaf $\mathcal{L}$ of $\mathcal{F}$ inside $\mathcal{Y}$,
$\pi_1(\mathcal{L})\cong \pi_2(P_x)$.
\end{lemma}

\begin{proof}
The foliation $\mathcal{F}_{\mathcal{Y}}$ is given by the fibers of $p:\mathcal{Y}\to B$, which, as remarked before,
is isomorphic to the bundle $\widetilde{\epsilon}:\mathrm{Path}^s(\mathcal{G}(A_S), 1)\to \mathcal{G}(A_S)$. So, a
leaf $\mathcal{L}$ will be identified with $\widetilde{\epsilon}^{-1}(g)$, for some $g\in \mathcal{G}(A_S)$. Let
$y:=s(g)$ and $P_y:=s^{-1}(y)$. Then $\mathcal{L}$ is a fiber of $\epsilon:\mathrm{Path}(P_y,1_y)\to P_y$, which,
by Lemma \ref{Localy-triv}, is a locally trivial fiber bundle. Since $\mathrm{Path}(P_y, 1_y)$ is contractible, using the
long exact sequence in homotopy, we find that $\pi_1(\mathcal{L})\cong \pi_2(P_y)$. Since $\mathcal{G}(A_S)$ is
transitive, $P_y$ and $P_x$ are diffeomorphic.
\end{proof}

Of course, if $\pi_2(P_x)$ were assumed to be trivial, then condition (H0) would follow automatically. Since $P_x$ is
1-connected, by the Hurewicz theorem, the hypothesis that $H^2(P_x)= 0$ is equivalent to $\pi_2(P_x)$ being finite.
We show that this is enough to ensure triviality of the holonomy groups.

\begin{lemma} \label{lemma-Getzler2}
The holonomy group of the foliated manifold $(\mathcal{X}, \mathcal{F})$ is trivial at any point $a\in \mathcal{Y}$.
\end{lemma}

\begin{proof} Let $a\in \mathcal{Y}$ and let $\Gamma$ be the holonomy group at $a$.
Let $\mathcal{T}$ be a transversal of $(\mathcal{X}, \mathcal{F})$ through $a$. Since $\Gamma$ is finite,
$\mathcal{T}$ can be chosen small enough so that the holonomy transformations $\mathrm{hol}_{u}$ define an
action of $\Gamma$ on $\mathcal{T}$. Denote by $\mathcal{T}_{\mathcal{Y}}:= \mathcal{T}\cap \mathcal{Y}$.
Since the holonomy of $(\mathcal{Y}, \mathcal{F}_{\mathcal{Y}})$ is trivial, by making $\mathcal{T}$ smaller, we
may assume:
\begin{enumerate}
\item[C1:] the action of $\Gamma$ on $\mathcal{T}_{\mathcal{Y}}$ is trivial.
\end{enumerate}
Note also that the submersion $\sigma: \mathcal{T}\to M$ satisfies:
\begin{enumerate}
\item[C2:] $\sigma$ is $\Gamma$-invariant.
\item[C3:] $\sigma^{-1}(\sigma(a))\subset \mathcal{T}_{\mathcal{Y}}$.
\end{enumerate}
For C2, just note that, for $a'\in \mathcal{T}$, since $a'$ and $\mathrm{hol}_{u}(a')$ are in the same leaf (i.e.\
cotangent-homotopic), they have the same starting point. C3 is also clear, since $\sigma^{-1}(S)\subset \mathcal{Y}$.

We have to show that the action of $\Gamma$ is trivial in a neighborhood of $a$ in $\mathcal{T}$. Since
$\Gamma$ is finite, it suffices to show that the induced infinitesimal action of $\Gamma$ on $T_{a}\mathcal{T}$ is trivial. 
At the infinitesimal level, we have a short exact sequence
\[ \ker(d\sigma)_a\stackrel{i}{\rmap} T_{a}\mathcal{T} \stackrel{(d\sigma)_a}{\rmap} T_{\sigma(a)}M.\]
This is a sequence of $\Gamma$-modules, where $\Gamma$ acts trivially on the first and the last term. For the last
map, this follows from C2. For the first map, C3 implies that $\ker(d\sigma)_a\subset T_a\mathcal{T}_{\mathcal{Y}}$
on which $\Gamma$ acts trivially by C1. Since $\Gamma$ is finite, the action on the middle term must be trivial as well
(use e.g.\ an equivariant splitting).
\end{proof}

\clearpage \pagestyle{plain}

\chapter{Formal equivalence around Poisson submanifolds}\label{ChFormalRigidity}
\pagestyle{fancy}
\fancyhead[CE]{Chapter \ref{ChFormalRigidity}} 
\fancyhead[CO]{Formal equivalence around Poisson submanifolds}

In this chapter we study the normal form problem for Poisson structures from the formal point of view. The main result
(Theorem \ref{Theorem_THREE}) is a formal rigidity theorem for Poisson structures around Poisson submanifolds,
which generalizes Weinstein's formal linearization around fixed points with semisimple isotropy Lie algebra
\cite{Wein}. The first order jet of a Poisson bivector at a Poisson submanifold encodes in the structure of the restricted
Lie algebroid. We find cohomological conditions on this Lie algebroid which insure that all Poisson structure with the
same first order jet are formally diffeomorphic. This result the formal version of Theorem \ref{Theorem_FOUR} form
chapter \ref{ChRigidity}, which treats the smooth rigidity problem. The content of this chapter was published in
\cite{MaFormal}.

\section{Statement of Theorem \ref{Theorem_THREE}}\label{section_statement_3}

Let $(M,\pi)$ be a Poisson manifold. Recall that an immersed submanifold \[\iota:S\rmap M\] is called a \textbf{Poisson
submanifold}\index{Poisson submanifold} of $M$ if $\pi$ is tangent to $S$. Recall also (see subsection
\ref{subsection_some_properties_of_Lie_algebroids}) that the cotangent Lie algebroid\index{cotangent Lie algebroid}
of $\pi$ can be restricted to $S$, inducing a Lie algebroid structure on
\[A_S:=T^*M_{|S}.\]
This fits in a short exact sequence of Lie algebroids
\begin{equation}\label{shortexact2}
0\rmap \nu_S^*\rmap A_S\rmap T^*S\rmap 0,
\end{equation}
where $\nu_S^*\subset A_S$ is the conormal bundle of $S$ and $T^*S$ is the cotangent algebroid of $\pi_{|S}$. In
particular, we obtain a representation of $A_S$ on $\nu_S^*$ and thus, also on its symmetric powers
$\mathcal{S}^{k}(\nu_S^*)$. Using the explicit formula for the Lie bracket of $T^*M$ (\ref{br}), we see that the Lie
algebroid structures on $A_S$ and the sequence (\ref{shortexact2}), they both depend only on the first jet of $\pi$ at
$S$, denoted by $j^1_{|S}\pi$.

The cohomology relevant for formal deformations of $\pi$ around $S$ is version of the Poisson cohomology of $(M,\pi)$
restricted to $S$ with coefficients (see section \ref{Proof of the main theorem; step 1: Moser path method}). In terms
of the Lie algebroid $A_S$, this is defined as\index{cohomology, Poisson}
\[H^{\bullet}_{\pi,S}(M,\mathcal{S}^{k}(\nu_S^*)):=H^{\bullet}(A_S,\mathcal{S}^{k}(\nu_S^*)).\]
An equivalent description of these groups, in terms of jets of multivector fields, will be given in subsection \ref{The
cohomology of the restricted algebroid}.

The main result of this chapter is the following:
\begin{mtheorem}\label{Theorem_THREE}
Let $\pi_1$ and $\pi_2$ be two Poisson structures on $M$, such that $S\subset M$ is an embedded Poisson
submanifold for both, and such that they have the same first order jet along $S$. If their common algebroid $A_S$
satisfies
\[H^{2}(A_S,\mathcal{S}^{k}(\nu_S^*))=0, \ \ \forall \  k\geq 2,\]
then the two structures are formally Poisson diffeomorphic. More precisely, there exists a diffeomorphism
\[\psi:\mathcal{U}\diffto \mathcal{V},\]
with $d\psi_{|TM_{|S}}=\mathrm{Id}_{TM_{|S}}$, where $\mathcal{U}$ and $\mathcal{V}$ are open
neighborhoods of $S$, such that $\pi_{1|\mathcal{U}}$ and $\psi^*(\pi_{2|\mathcal{V}})$ have the same infinite jet
along $S$:
\[j_{|S}^{\infty}(\pi_{1|\mathcal{U}})=j^{\infty}_{|S}(\psi^*(\pi_{2|\mathcal{V}})).\]
\end{mtheorem}

Applying Theorem \ref{Theorem_THREE} to the linear Poisson structure on the dual of a compact, semisimple Lie
algebra, we obtain the following result.

\begin{corollary}\label{C_Dual_ss_Lie}
Consider the linear Poisson structure $(\mathfrak{g}^*,\pi_{\mathfrak{g}})$ corresponding to a semisimple Lie
algebra of compact type $\mathfrak{g}$, and let $\mathbb{S}({\mathfrak{g}}^*)\subset \mathfrak{g}^*$ be the unit
sphere in $\mathfrak{g}^*$ centered at the origin, with respect to some invariant inner product. Any Poisson
structure $\pi$, which is defined on some open around $\mathbb{S}({\mathfrak{g}}^*)$, and that satisfies
\[j^1_{|\mathbb{S}({\mathfrak{g}}^*)}(\pi_{\mathfrak{g}})=j^1_{|\mathbb{S}({\mathfrak{g}}^*)}(\pi),\]
is formally Poisson diffeomorphic to $\pi_{\mathfrak{g}}$ around
$\mathbb{S}({\mathfrak{g}}^*)$.\index{Lie-Poisson sphere}
\end{corollary}

Recall from chapter \ref{ChNormalForms} that around an embedded symplectic leaf $(S,\omega_S)$ we have the
notion of a first order approximation of $\pi$ around $S$, which we denote by $\plin(S)$. This Poisson structure has the
same first order jet as $\pi$ at $S$, and is constructed only using data encoded in $j^1_{|S}\pi$.

Theorem \ref{Theorem_THREE} implies a version of Theorem \ref{Theorem_TWO} in the formal category:

\begin{theorem}\label{Theorem1}
Let $(M,\pi)$ be a Poisson manifold and let $S\subset M$ be an embedded symplectic leaf. If the cohomology groups
\[H^2(A_S,\mathcal{S}^{k}(\nu_S^*))\]
vanish for all $k\geq 2$, then $\pi$ is formally Poisson diffeomorphic to its first order approximation around $S$.
\end{theorem}

In many cases, we prove that these cohomological obstructions vanish, and we obtain the following corollaries:

\begin{corollary}\label{Theorem2}
Let $(M,\pi)$ be a Poisson manifold and let $S\subset M$ be an embedded symplectic leaf. Assume that the Poisson
homotopy bundle of $S$ is a smooth principal bundle with vanishing second de Rham cohomology group,
and\index{differentiable cohomology} that its structure group $G$ satisfies
\[H^2_{\mathrm{diff}}(G,\mathcal{S}^k(\mathfrak{g}))=0,\ \ \forall\  k\geq 2,\]
where $\mathfrak{g}$ is the Lie algebra of $G$ and $H^{\bullet}_{\mathrm{diff}}(G,\mathcal{S}^k(\mathfrak{g}))$
denotes the differentiable cohomology of $G$ with coefficients in the $k$-th symmetric power of the adjoint
representation. Then $\pi$ is formally Poisson diffeomorphic to its first order approximation around $S$.
\end{corollary}

Differentiable cohomology of compact groups vanishes, therefore:

\begin{corollary}\label{Corollary_2}
Let $(M,\pi)$ be a Poisson manifold and let $S\subset M$ be an embedded symplectic leaf. Assume that the Poisson
homotopy bundle of $S$ is a smooth principal bundle with vanishing second de Rham cohomology and compact
structure group. Then $\pi$ is formally Poisson diffeomorphic to its first order approximation around $S$.
\end{corollary}

A bit more technical is the following:
\begin{corollary}\label{Theorem3}
Let $(M, \pi)$ be a Poisson manifold and let $S\subset M$ be an embedded symplectic leaf whose isotropy Lie algebra
is reductive. If the abelianization algebroid
\[A_S^{\mathrm{ab}}:=A_S/[\nu_S^*,\nu_S^*]\]
is integrable by a simply connected principal bundle with vanishing second de Rham cohomology and compact
structure group, then $\pi$ is formally Poisson diffeomorphic to its first order approximation around $S$.
\end{corollary}

\begin{corollary}\label{Corollary_3}
Let $(M, \pi)$ be a Poisson manifold and let $S\subset M$ be an embedded symplectic leaf through $x\in M$. If the
isotropy Lie algebra at $x$ is semisimple, $\pi_1(S,x)$ is finite and $\pi_2(S,x)$ is torsion, then $\pi$ is formally Poisson
diffeomorphic to its first order approximation around $S$.
\end{corollary}

\subsection{Some related results}

The first order approximation of a Poisson manifold around a fixed point is the linear Poisson structure corresponding
to the isotropy Lie algebra. Formal linearization in this setting was proven by Weinstein in \cite{Wein}, under the
hypothesis that the isotropy Lie algebra is semisimple. This case is covered also by our Corollary \ref{Corollary_3}.

A weaker version of Theorem \ref{Theorem_THREE} is stated in \cite{VorCo}. Instead of embedded Poisson
submanifolds, the authors of \emph{loc.cit.} work with compact symplectic leaves and also their conclusion is a bit
weaker, they prove that for each $k$ there exists a diffeomorphism that identifies the Poisson structures up to order
$k$ (Theorem 7.1 \emph{loc.cit.}). Compactness of the leaf is a too strong assumption for formal equivalence. For
example, Corollary 7.4 \emph{loc.cit.} concludes that hypotheses similar to those of our Corollary \ref{Corollary_3}
imply vanishing of the cohomological obstructions; but, as remarked by the authors themselves, these assumptions
together with the compactness assumption imply that the leaf is a point.

In order to prove Theorem \ref{Theorem_THREE}, we first reduce the problem to an equivalence criterion for Maurer
Cartan elements in complete graded Lie algebras. In the context of differential graded associative algebras, this
criterion can be found in Appendix A \cite{CMB}.

In order to prove vanishing of cohomology and to obtain the corollaries enumerated above, we use techniques like
Whitehead's Lemma for semisimple Lie algebras, spectral sequences for Lie algebroids, but also the more powerful
techniques developed in \cite{Cra} such as the Van Est map and vanishing of cohomology of proper groupoids.

The statement of Corollary \ref{Corollary_2} resembles the most that of Theorem \ref{Theorem_TWO}; the only
hypothesis that is weakened is the compactness of the leaf. Yet, the assumptions of Corollary \ref{Corollary_2} are too
strong in the formal setting; in particular, they imply compactness of the semisimple part of the isotropy Lie algebra.
The more technical Corollary \ref{Theorem3} generalizes this result by allowing the semisimple part to be
noncompact, and therefore, we consider this result to be the appropriate analogue of Theorem \ref{Theorem_TWO} in
the formal category.

\subsection{The first order data}\label{The first order data}

Let $(M,\pi)$ be a Poisson manifold and let $\iota:S\to M$ be an embedded Poisson submanifold. There are several ways
to encode the first order data associated to $S$, which we explain below. Since we are interested in local properties of
$\pi$ around $S$, we may assume that $S$ is also closed in $M$ (this happens if we replace $M$ with a tubular
neighborhood of $S$).

\subsubsection*{The first order jet}

We use the same notations as in subsection \ref{Subsection_the_first_order_data}: $I_S$ is the vanishing ideal of $S$,
$\mathfrak{X}^{\bullet}_S(M)$ is the algebra of multivector fields tangent to $S$, and the jet spaces are defined by
\[J^k_S(\mathfrak{X}^{\bullet}_S(M))=\mathfrak{X}^{\bullet}_S(M)/I^{k+1}_S\mathfrak{X}^{\bullet}(M).\]
Recall also that the jet spaces inherit a graded Lie algebra structure.

Denote by $\pi_S:=\pi_{|S}$ the Poisson structure induced on $S$. Similar to Definition
(\ref{Defintion_first_jet_Poisson_symplectic}), we consider:
\begin{definition}\label{Defintion_first_jet_Poisson_Poisson}
A first jet of a Poisson on $M$ with Poisson submanifold $(S,\pi_S)$, is an element $\tau\in
J^1_S(\mathfrak{X}^{2}_S(M))$, satisfying
\[ \tau_{|S}=\pi_S\ \ \textrm{ and }\ \ [\tau,\tau]=0.\]
We denote by $J^1_{(S,\pi_S)}\mathrm{Poiss}(M)$ the space of such elements.
\end{definition}

\subsubsection*{An extension of Poisson algebras}\index{Poisson algebra}

The fact that $S$ is a Poisson submanifold is equivalent to $I_S$ being an ideal of the Lie algebra
$(C^{\infty}(M),\{\cdot,\cdot\})$. The quotient space $C^{\infty}(M)/I_S$ carries a natural Poisson algebra structure,
which is canonically isomorphic to the Poisson algebra corresponding to $\pi_S$
\[(C^{\infty}(S),\{\cdot,\cdot\}).\]
We regard this algebra as the $0$-th order approximation of the Poisson manifold $(M,\pi)$ around $S$. This gives a
recipe for constructing higher order approximations. For example, the first order approximation fits into an exact
sequence of Poisson algebras
\begin{equation}\label{shortexact}
0\rmap (I_S/I^{2}_S,\{\cdot,\cdot\})\rmap(C^{\infty}(M)/I^{2}_S,\{\cdot,\cdot\})\rmap (C^{\infty}(S),\{\cdot,\cdot\}) \rmap 0.
\end{equation}

\subsubsection*{An extension of Lie algebroids}

The cotangent Lie algebroid $(T^*M,[\cdot,\cdot]_{\pi},\pi^{\sharp})$ can be restricted to $S$, i.e.\ there is a unique
Lie algebroid structure on $A_S:=T^*M_{|S}$ such that the restriction map $\Gamma(T^*M)\to \Gamma(A_S)$ is a
Lie algebra homomorphism. Endowed with this structure, $A_S$ fits in the short exact sequence of Lie algebroids:
\begin{equation}\label{shortexact_alg}
0\rmap (\nu_S^*,[\cdot,\cdot])\rmap (A_S,[\cdot,\cdot])\rmap (T^*S,[\cdot,\cdot]_{\pi_S})\rmap 0,
\end{equation}
where the last term is the cotangent Lie algebroid of $(S,\pi_S)$. The representation of $A_S$ on $\nu_S^*$
corresponds to the flat $A_S$-connection:
\[\nabla:\Gamma(A_S)\times \Gamma(\nu_S^*)\rmap \Gamma(\nu_S^*), \ \ \nabla_{\alpha}(\eta):=[\alpha,\eta].\]

These three constructions are equivalent ways to encode the first order data of $\pi$ around $S$.

\begin{proposition}\label{Proposition_first_order_data}
Let $(S,\pi_S)$ be Poisson manifold and let $\iota:S\to M$ be a closed embedding. There is a one to one correspondence
between
\begin{itemize}
\item first jets of Poisson structures with $(S,\pi_S)$ as a Poisson submanifold,
\item Poisson algebra structures on the commutative algebra $C^{\infty}(M)/I^2_S$ that fit in the short exact
sequence (\ref{shortexact}),
\item Lie algebroid structures on the vector bundle $A_S=T^*M_{|S}$ that fit in the short exact sequence (\ref{shortexact_alg}).
\end{itemize}
\end{proposition}
\begin{proof}
We first show that the three structures are encoded by the same type of objects. Using a tubular neighborhood, we may
replace $M$ with $\nu_S$. This gives canonical identifications $A_S=T^*S\oplus \nu_S^*$ and
$C^{\infty}(\nu_S)/I^2_S=C^{\infty}(S)\oplus \Gamma(\nu_S^*)$. As a commutative algebra, $C^{\infty}(S)\oplus
\Gamma(\nu_S^*)$ is the square-zero extension of $C^{\infty}(S)$ by $\Gamma(\nu_S^*)$, i.e.\
$\Gamma(\nu_S^*)^2=0$. When regarding an element $\theta\in \Gamma(\nu_S^*)$ as a linear function on $\nu_S$,
we denote it by $\overline{\theta}$; and when we view it as the restriction of a 1-forms on $\nu_S$ to $S$, we simply
denote it by $\theta$. Using the inclusion $\nu_S\subset (T\nu_S)_{|S}$, we have that
\[d\overline{\theta}_{|\nu_S}=\theta.\]
We encode an element $\tau\in J^1_{S}(\mathfrak{X}^2_S(\nu_S))$, such that $\tau_{|S}=\pi_S$, by a triple:
\[(R,\nabla,\pi^{\mathrm{v}}).\]
To explain each component, write $\tau=j^1_{|S}W$ for a bivector $W$ on $\nu_S$. Then:\\
$\bullet$ $R\in\mathfrak{X}^2(S)\otimes \Gamma(\nu_S^*)$ is a bivector with values in $\nu_S^*$ defined by
\[R(\alpha,\beta):= d(W(p^*(\alpha),p^*(\beta)))_{|\nu_S},\ \ \alpha,\beta\in \Gamma(T^*S),\]
$\bullet$ $\nabla:\Gamma(T^*S)\otimes \Gamma(\nu_S^*)\to \Gamma(\nu_S^*)$ is a contravariant
connection\index{contravariant connection} on $\nu_S^*$ for the cotangent Lie algebroid of $\pi_S$ defined by
\[\nabla_{\alpha}(\theta):=d(W(p^*(\alpha),d\overline{\theta}))_{|\nu_S},\ \ \alpha\in\Gamma(T^*S), \theta\in\Gamma(\nu^*_S),\]
$\bullet$ $\pi^{\mathrm{v}}\in \Gamma(\nu_S^*\otimes \Lambda^2\nu_S)$ is defined by
\[\pi^{\mathrm{v}}(\eta,\theta):=d(W(d\overline{\eta},d\overline{\theta}))_{|\nu_S},\ \eta,\theta\in\Gamma(\nu^*_S).\]
It is straightforward to check that these elements are $C^{\infty}(S)$-linear in both arguments, except for $\nabla$ in
the second argument, in which it satisfies the rule of a contravariant derivative (see subsection
\ref{SubSection_Contravariant}):
\[\nabla_{\alpha}f\theta=f\nabla_{\alpha}\theta+L_{\pi_S^{\sharp}(\alpha)}(f)\theta.\]

Similarly, a bilinear operator on $A_S$ that fits in (\ref{shortexact_alg}) and satisfies the Leibniz rule is also given by
the same structures, namely the ``bracket'' is
\begin{equation*}
[(\alpha,\eta),(\beta,\theta)]=([\alpha,\beta]_{\pi_S},R(\alpha,\beta)+\nabla_{\alpha}\theta-\nabla_{\beta}\eta+\pi^{\mathrm{v}}(\eta,\theta)).
\end{equation*}
for $\alpha,\beta\in\Gamma(T^*S)$ and $\eta,\theta\in\Gamma(\nu_S^*)$.

Also, a biderivation of the commutative algebra $C^{\infty}(S)\oplus \Gamma(\nu_S^*)$ that fits into
(\ref{shortexact}) is given by the same type of elements:
\[\{(f,\overline{\eta}),(g,\overline{\theta})\}=(\{f,g\}_S,\overline{R(df,dg)+\nabla_{df}(\theta)-\nabla_{dg}(\eta)+\pi^{\mathrm{v}}(\eta,\theta)}),\]
for $(f,\eta),(g,\theta) \in C^{\infty}(S)\oplus\Gamma(\nu_S^*)$. It is a lengthly but straightforward computation, which
we omit, to check that the equation $[\tau,\tau]=0$ is equivalent to the Jacobi identity for $[\cdot,\cdot]$ on
$\Gamma(A_S)$ and also to the Jacobi identity for $\{\cdot,\cdot\}$ on $C^{\infty}(S)\oplus\Gamma(\nu_S^*)$.

In the case when $\tau$ is the first jet of a Poisson structure $\pi$, notice that the bracket on $A_S$ is the one
obtained by restricting the cotangent Lie algebroid of $\pi$ to $S$, and that the Poisson algebra on
$C^{\infty}(S)\oplus\Gamma(\nu_S^*)$ is the one of $\pi$ modulo the ideal $I^2_S$.
\end{proof}

If $\pi_S$ is nondegenerate, i.e.\ $\pi_S=\omega_S^{-1}$ for a symplectic form $\omega_S$ on $S$, then every first jet
of a Poisson tensor $\tau\in J^1_{(S,\pi_S)}\mathrm{Poiss}(M)$ can be extended to a Poisson structure on a
neighborhood of $S$. This follows from the results in subsection \ref{Subsection_the_local_model_the_general_case}:
by the previous proposition, $\tau$ determines a Lie algebroid structure on $A_S$, which, by nondegeneracy of $\pi_S$,
is transitive; by Proposition \ref{proposition_restricting_to_cosymplectic}, the local model $\pi_{A_S}$ constructed
out of $A_S$ and $\omega_S$ has $(S,\omega_S)$ as a symplectic leaf and $A_S$ as its Lie algebroid, hence, again by
the previous proposition, $j^1_{|S}\pi_{A_S}=\tau$.

In the case when $\pi_S$ is degenerate this is no longer true. The example below illustrates this.

\begin{example}\label{Example_Not_every_X_is_Y}\rm
Consider $S:=\mathbb{R}^2$ as the submanifold $\{z=0\}\subset M:=\mathbb{R}^3$. We define a Poisson algebra
structure on the commutative algebra
\[C^{\infty}(M)/I^2_S=C^{\infty}(M)/(z^2)=C^{\infty}(S)\oplus z C^{\infty}(S)\]
with the property that $\{f,g\}\in (z)$, for all $f,g\in C^{\infty}(M)/(z^2)$. Explicitly,
\begin{align*}
\{f,g\}&=z(\frac{\partial{f}}{\partial x}\frac{\partial{g}}{\partial y}-\frac{\partial{f}}{\partial y}\frac{\partial{g}}{\partial
x}+x\frac{\partial{f}}{\partial x}\frac{\partial{g}}{\partial z}-x\frac{\partial{f}}{\partial z}\frac{\partial{g}}{\partial x}) \ \mathrm{modulo} \
(z^2).
\end{align*}
It is easy to check that $\{\cdot,\cdot\}$ satisfies the Jacobi identity. We obtain an extension of Poisson algebras
\[0\rmap z C^{\infty}(S)\rmap C^{\infty}(S)\oplus z C^{\infty}(S)\rmap C^{\infty}(S)\rmap 0,\]
with zero Poisson bracket on $C^{\infty}S$. The total space of the corresponding Lie algebroid $A_S$ is
$\mathbb{R}^3\times S\to S$ with a global frame given by $\{dx_{|S}, dy_{|S}, dz_{|S}\}$. The anchor is zero, thus
the bracket is determined by the relations
\[[dx_{|S},dy_{|S}]=dz_{|S}, \ \ [dy_{|S},dz_{|S}]=0,\ \ [dx_{|S},dz_{|S}]=xdz_{|S}.\]

We claim that there is no Poisson structure on $M$ (nor on any open neighborhood of $S$) that extends this first order
data. The bracket of such a Poisson structure must be of the form:
\[\{x,y\}=z+z^2h,\ \ \{y,z\}=z^2k, \ \ \{x,z\}=xz+z^2l,\]
for some smooth functions $h,k,l$. Computing the Jacobiator of $x$, $y$, $z$, and putting together all terms that
vanish up third order on $S$, we obtain
\begin{align*}
J\{x,y,z\}&=\{x,\{y,z\}\}+\{y,\{z,x\}\}+\{z,\{x,y\}\}=\\
&=\{x,z^2k\}-\{y,xz+z^2l\}+\{z,z+z^2h\}=\\
&=2z^2k(x,y,0)+z^2-xz^2k(x,y,0)+z^3a(x,y,z)=\\
&=z^2((2-x)k(x,y,0)+1)+z^3a(x,y,z),
\end{align*}
where $a$ is a smooth function. Note that $J$ cannot vanish, because
\[\frac{\partial^2J}{\partial z^2}(2,y,0)=2.\]
\end{example}

\section{The algebra of formal vector fields}\label{Formal equivalence of Poisson structures around symplectic
leaves}

Let $S\subset M$ be a closed embedded submanifold. We describe now the algebraic framework that controls formal
jets of multivector fields along $S$.

The ideal $I_S$ induces a filtration $\mathcal{F}$ on $\mathfrak{X}^{\bullet}_S(M)$:\index{filtration}
\[\mathfrak{X}^{\bullet}_S(M)\supset \mathcal{F}^{\bullet}_0\supset\mathcal{F}^{\bullet}_1\supset\ldots \mathcal{F}^{\bullet}_k\supset\mathcal{F}^{\bullet}_{k+1}\supset\ldots.\]
\[\mathcal{F}^{\bullet}_k=I^{k+1}_S\mathfrak{X}^{\bullet}(M),\ \   k\geq 0.\]
It is readily checked that
\begin{equation}\label{Filtration_bracket}
[\mathcal{F}_k,\mathcal{F}_l]\subset \mathcal{F}_{k+l},\ \ [\mathfrak{X}_S^{\bullet}(M),\mathcal{F}_k]\subset \mathcal{F}_k,
\end{equation}
therefore the jet spaces
\[J^{k}_S(\mathfrak{X}^{\bullet}_S(M)):=\mathfrak{X}_S^{\bullet}(M)/\mathcal{F}^{\bullet}_k\]
inherit a graded Lie algebra structure\index{graded Lie algebra}. Let $J^{\infty}_S(\mathfrak{X}^{\bullet}_S(M))$
be the completion of $\mathfrak{X}^{\bullet}_S(M)$ with respect to the filtration $\mathcal{F}$; it is defined by the
inverse limit
\[J^{\infty}_S(\mathfrak{X}^{\bullet}_S(M)):=\varprojlim \mathfrak{X}^{\bullet}_S(M)/\mathcal{F}^{\bullet}_{k}=\varprojlim J^{k}_S(\mathfrak{X}^{\bullet}_S(M)).\]
By (\ref{Filtration_bracket}), it follows that also $J^{\infty}_S(\mathfrak{X}^{\bullet}_S(M))$ inherits a graded Lie
algebra structure, for which the natural projections
\[ j^{k}_{|S}:J^{\infty}_S(\mathfrak{X}^{\bullet}_S(M))\rmap J^{k}_S(\mathfrak{X}^{\bullet}_S(M)),\textrm{ for }k\geq 0\]
are Lie algebra homomorphisms. The algebra $(J^{\infty}_S(\mathfrak{X}^{\bullet}_S(M)),[\cdot,\cdot])$ will be
called the \textbf{algebra of formal multivector fields along} $S$. Consider also the graded Lie algebra homomorphism
\[j^{\infty}_{|S}:\mathfrak{X}^{\bullet}_S(M)\rmap J^{\infty}_S(\mathfrak{X}^{\bullet}_S(M)).\]
By a version of Borel's Theorem (see e.g.\ \cite{Moerdijk}) about existence of smooth sections with a specified infinite
jet along a submanifold, $j^{\infty}_{|S}$ is surjective. Observe that $J^{\infty}_S(\mathfrak{X}^{\bullet}_S(M))$
inherits a filtration $\hat{\mathcal{F}}$ from $\mathfrak{X}^{\bullet}_S(M)$, given by
\[\hat{\mathcal{F}}^{\bullet}_{k}=j^{\infty}_{|S}\mathcal{F}^{\bullet}_{k},\]
and which satisfies the corresponding equations (\ref{Filtration_bracket}).

The adjoint action of an element $X\in\hat{\mathcal{F}}_1^1$
\[ad_X: J^{\infty}_S(\mathfrak{X}^{\bullet}_S(M))\rmap J^{\infty}_S(\mathfrak{X}^{\bullet}_S(M)),\ \ ad_X(Y):=[Y,X]\]
increases the degree of the filtration by 1. Therefore, the partial sums
\[\sum_{i=0}^n\frac{ad_{X}^i}{i!}(Y)\]
are constant modulo $\hat{\mathcal{F}}_k$ for $n\geq k$ and all $Y\in J^{\infty}_S(\mathfrak{X}^{\bullet}_S(M))$.
This and the completeness of the filtration on $\hat{\mathcal{F}}$ show that the exponential of $ad_X$
\[e^{ad_X}:J^{\infty}_S(\mathfrak{X}^{\bullet}_S(M))\rmap J^{\infty}_S(\mathfrak{X}^{\bullet}_S(M)),\ \ e^{ad_{X}}(Y):=\sum_{n\geq 0} \frac{ad_{X}^n}{n!}(Y)\]
is well defined. It is readily checked that $e^{ad_X}$ is a graded Lie algebra isomorphism with inverse $e^{-ad_X}$
and that it preserves the filtration.

These isomorphisms have a geometric interpretation.
\begin{lemma}\label{gauge_real}
For $X\in\hat{\mathcal{F}}_1^1$, there exists $\psi:M\diffto M$ a diffeomorphism of $M$, with
$\psi_{|S}=\mathrm{Id}_S$ and $d\psi_{|S}=\mathrm{Id}_{TM_{|S}}$, such that
\[j^{\infty}_{|S}(\psi^*(W))=e^{ad_X}(j^{\infty}_{|S}(W)), \ \ \forall \ W\in
\mathfrak{X}^{\bullet}_S(M).\]
\end{lemma}
\begin{proof}
By Borel's Theorem, there exists a vector field $V$ on $M$, such that \[X=-j^{\infty}_{|S}(V).\] We claim that $V$ can
be chosen to be complete. Let $g$ be a complete metric on $M$ and let $\phi:M\to [0,1]$ be a smooth function, such
that $\phi=1$ on the set $\{x|g_x(V_x,V_x)\leq 1/2\}$ and $\phi=0$ on the set $\{x|g_x(V_x,V_x)\geq 1\}$. Since
$V_{|S}=0$, $\phi V$ has the same germ as $V$ around $S$, therefore $j^{\infty}_{|S}(\phi V)=X$. Since $\phi V$ is
bounded, it is complete. So, just replace $V$ with $\phi V$.

We will show that $\psi:=\varphi_V$, the flow of $V$ at time 1 satisfies all requirements. Since $j^1_{|S}(V)=0$, it is
clear that $\psi_{|S}=\textrm{Id}_S$ and $d\psi_{|S}=\textrm{Id}_{TM_{|S}}$.

For $W\in \mathfrak{X}_S^{\bullet}(M)$, we denote by $W_s:=(\varphi_V^s)^{*}(W)$ the pullback of $W$ by the
flow of $V$ at time $s$. Using that $W_s$ satisfies the differential equation $\frac{d}{ds}W_s=[V,W_s]=-ad_V(W_s)$,
we obtain that
\begin{align*}
\frac{d}{ds}\left(\sum_{i=0}^k\frac{s^iad_{V}^i}{i!}(W_s)\right)&=\sum_{i=0}^k\left(\frac{s^{i-1}ad_{V}^i}{(i-1)!}(W_s)-\frac{s^iad_{V}^i}{i!}ad_V(W_s)\right)=\\
&=-\frac{s^k ad_V^{k+1}}{k!}(W_s).
\end{align*}
This shows that the sum
\[\sum_{i=0}^k\frac{s^iad_{V}^i}{i!}(W_s)\]
modulo $\mathcal{F}_{k+1}$ is independent of $s$; therefore,
\[W-\sum_{i=0}^k\frac{ad_{V}^i}{i!}(\psi^*(W))\in\mathcal{F}_{k+1}.\]
Applying $j^{\infty}_{|S}$ to this equation yields
\[j^{\infty}_{|S}(W)-\sum_{i=0}^k\frac{(-1)^iad_{X}^i}{i!}j^{\infty}_{|S}(\psi^*(W))\in\hat{\mathcal{F}}_{k+1},\]
hence, the conclusion
\[j^{\infty}_{|S}(W)=e^{-ad_X}j^{\infty}_{|S}(\psi^*(W)).\qedhere\]
\end{proof}

\section{The cohomology of the restricted algebroid}\label{The cohomology of the restricted algebroid}

Let $(M,\pi)$ be a Poisson manifold and $S\subset M$ a closed, embedded Poisson submanifold. By the discussion in
the previous subsection, we have that
\[[\pi,\mathcal{F}^{\bullet}_k]\subset \mathcal{F}^{\bullet}_k.\]
Hence $\mathcal{F}^{\bullet}_k$ is a subcomplex of the Poisson complex
\[(\mathcal{F}^{\bullet}_k, d_{\pi})\subset (\mathfrak{X}^{\bullet}(M), d_{\pi}).\]
Denote the complexes obtained by taking consecutive quotients by
\[(\mathcal{F}^{\bullet}_k/\mathcal{F}^{\bullet}_{k+1},d^k_{\pi}).\]
For $k=0$, this computes the Poisson cohomology relative to $S$. Observe that the differential on these complexes
depends only on the first jet of $\pi$ along $S$. Therefore, following the philosophy of subsection \ref{The first order
data}, it can be described in terms of the Lie algebroid $A_S$.\index{cohomology, Poisson}

\begin{proposition}\label{Proposition_isomorphic_complexes}
The following two complexes are isomorphic
\[(\mathcal{F}^{\bullet}_k/\mathcal{F}^{\bullet}_{k+1},d_\pi^k)\cong (\Omega^{\bullet}(A_S,\mathcal{S}^k(\nu_S^*)),d_{\nabla^k}),\ \  \forall  k\geq 0.\]
\end{proposition}
\begin{proof}
Using that the space $\Gamma(\nu_S^*)$ is spanned by differentials of elements in $I_S$, it is easy to check that the
following map is well-defined and surjective
\begin{align*}
\tau_k&:\mathcal{F}^{\bullet}_k=I^{k}_S\mathfrak{X}^{\bullet}(M)\rmap \Omega^{\bullet}(A_S,\mathcal{S}^k(\nu_S^*))=\Gamma(\Lambda^{\bullet}(T_SM)\otimes\mathcal{S}^{k}(\nu_S^*)),\\
&\tau_k(f_1\ldots f_kW)=W_{|S}\otimes df_{1|S} \odot \ldots \odot df_{k|S},
\end{align*}
where $f_1,\ldots,f_k\in I_S$ and $W\in\mathfrak{X}^{\bullet}(M)$. Moreover, its kernel is precisely
$\mathcal{F}^{\bullet}_{k+1}=I^{k+1}_S\mathfrak{X}^{\bullet}(M)$. So, it suffices to prove that
\begin{equation}\label{Commuting_with_differential}
\tau_k([\pi,W])=d_{\nabla^k}(\tau_k(W)),\ \ \forall \ W\in \mathcal{F}^{\bullet}_k.
\end{equation}
The algebroid $A_S$ has anchor $\rho=\pi^{\sharp}_{|S}$ and its bracket is determined by
\[[d\phi_{|S},d\psi_{|S}]_{A_S}:=d\{\phi,\psi\}_{|S},\ \ \forall\ \phi,\psi\in C^{\infty}(M).\]
Moreover, $\nabla^0$ is given by
\[\nabla^0:\Gamma(A_S)\times C^{\infty}(S)\rmap C^{\infty}(S),\ \ \nabla^0_{\eta}(h)=L_{\rho(\eta)}(h).\]
Since both differentials $d_{\pi}$ and $d_{\nabla^k}$ act by derivations and $\nabla^k$ is obtained by extending
$\nabla^1$ by derivations, it suffices to prove (\ref{Commuting_with_differential}) for $\phi\in C^{\infty}(M)$,
$X\in\mathfrak{X}^{1}(M)$ (for $k=0$) and $f\in I_S$ (for $k=1$).

Let $\phi\in C^{\infty}(M)$ and $\eta\in\Gamma(A_S)$. Using that $\pi$ is tangent to $S$, we obtain that
(\ref{Commuting_with_differential}) holds for $\phi$:
\[\tau_0([\pi,\phi])(\eta)=[\pi,\phi]_{|S}(\eta)=d\phi_{|S}(\pi^{\sharp}_{|S}(\eta))=L_{\rho(\eta)}(\tau_0(\phi))=d_{\nabla^0}(\tau_0(\phi))(\eta).\]
Let $X\in \mathfrak{X}^{1}(M)$, $\alpha,\beta\in C^{\infty}(M)$ and $\eta:=d\alpha_{|S}$,
$\theta:=d\beta_{|S}\in\Gamma(A_S)$. Then
\begin{align*}
\tau_0([\pi,&X])(\eta,\theta)=[\pi,X]_{|S}(d\alpha_{|S},d\beta_{|S})=\\
=&(\{X(\alpha),\beta\}+\{\alpha,X(\beta)\}-X(\{\alpha,\beta\}))_{|S}=\\
=&\pi^{\sharp}_{|S}(d\alpha_{|S})(X_{|S}(d\beta_{|S}))-\pi^{\sharp}_{|S}(d\beta_{|S})(X_{|S}(d\alpha_{|S}))-X_{|S}(d\{\alpha,\beta\}_{|S})=\\
=&L_{\rho(\eta)}(\tau_0(X)(\theta))-L_{\rho(\theta)}(\tau_0(X)(\eta))-\tau_0(X)([\eta,\theta]_{A_S})=\\
=&d_{\nabla^0}(\tau_0(X))(\eta,\theta),
\end{align*}
thus (\ref{Commuting_with_differential}) holds for $X$.

Consider now $f\in I_S$ and $\eta:=d\alpha_{|S}\in\Gamma(A_S)$, with $\alpha\in C^{\infty}(M)$. The following shows
that (\ref{Commuting_with_differential}) holds also for $f$
\begin{align*}
\tau_1([\pi,f])(\eta)&=\tau_1([\pi,f])(d\alpha_{|S})=\tau_1([\pi,f](d\alpha))=\tau_1(\{\alpha,f\})=\\
&=d\{\alpha,f\}_{|S}=[\eta,df_{|S}]_{A_S}=\nabla^1_{\eta}(\tau(f))=d_{\nabla^1}(\tau(f))(\eta).\phantom{123456789}\qedhere
\end{align*}
\end{proof}

\section{The proofs}

\subsection{Proof of Theorem \ref{Theorem_THREE}}

By replacing $M$ with a tubular neighborhood of $S$, we may assume that $S$ is closed in $M$. Denote by
\[\gamma:=j^{\infty}_{|S}\pi_1, \gamma':=j^{\infty}_{|S}\pi_2\in J^{\infty}_S(\mathfrak{X}^2_S(M)).\]
By Proposition \ref{Proposition_isomorphic_complexes}, we can recast the hypothesis as follows
\begin{align*}
[\gamma,\gamma]=0,\ \ [\gamma',\gamma']=0,\ \ \gamma-\gamma'\in \hat{\mathcal{F}}_1,\ \
H^2(\hat{\mathcal{F}}^{\bullet}_{k}/\hat{\mathcal{F}}^{\bullet}_{k+1},d_{\gamma})=0,\ \ \forall \ k\geq 1,
\end{align*}
where $d_{\gamma}=[\gamma,\cdot ]$. These conditions are expressed in terms of a graded Lie algebra
$\mathcal{L}^{\bullet}$ with a complete filtration, namely:
\[\mathcal{L}^{\bullet}:=J_S^{\infty}(\mathfrak{X}^{\bullet+1}_S(M)).\]
We will prove in the appendix (Theorem \ref{Teo1}) a result about equivalence of Maurer-Cartan elements in complete
graded Lie algebras. In our case, this result implies the existence of a formal vector field $X\in
\hat{\mathcal{F}}_1^1$, so that $\gamma=e^{ad_X}(\gamma')$. By Lemma \ref{gauge_real}, there exists a
diffeomorphism $\psi$ of $M$, such that $j^{\infty}_{|S}(\psi^*(W))=e^{ad_X}j^{\infty}_{|S}(W)$, for all
$W\in\mathfrak{X}_S^{\bullet}(M)$. This concludes the proof:
\[j^{\infty}_{|S}(\psi^*(\pi_2))=e^{ad_X}j^{\infty}_{|S}(\pi_2)=e^{ad_X}(\gamma')=\gamma=j^{\infty}_{|S}(\pi_1).\]

\subsection{On the existence of Poisson structures with a specified infinite jet}

The proof of Theorem \ref{Theorem_THREE} presented above can be used to obtain an existence result for Poisson
bivectors with a specified infinite jet. Let $S$ be a closed embedded submanifold of $M$. An element $\hat{\pi}\in
J^{\infty}_S(\mathfrak{X}^{2}_S(M))$ satisfying the equation $[\hat{\pi},\hat{\pi}]=0$ will be called a \textbf{formal
Poisson bivector}. Observe that
\[\hat{\pi}_{|S}:=j^0_{|S}\hat{\pi} \in \mathfrak{X}^2(S)\]
is a Poisson structure on $S$. By the results in subsection \ref{The first order data}, its first jet
\[j^1_{|S}(\hat{\pi})=\hat{\pi}\ \textrm{modulo}\ \hat{\mathcal{F}}_1\]
determines a Lie algebroid $A_S$ on $T^*M_{|S}$. Assuming that $S$ is a symplectic leaf of $\hat{\pi}$ (i.e.\ that
$\hat{\pi}_{|S}$ is nondegenerate), consider $\plin(S)$ the local model corresponding to $A_S$ and
$\omega_S:=\hat{\pi}_{|S}^{-1}$ from subsection \ref{Subsection_the_local_model_the_general_case}. Then
$\plin(S)$ is a Poisson structure on an open around $S$, whose first order jet coincides with that of $\hat{\pi}$. If the
cohomology groups
\[H^{2}(A_S;\mathcal{S}^{k}(\nu_S^*))\]
vanish for all $k\geq 2$, then, by the proof of Theorem \ref{Theorem_THREE}, we find a diffeomorphism $\psi$, such
that
\[j^{\infty}_{|S}(\psi^*(\plin(S)))=\hat{\pi}.\]
Thus $\pi:=\psi^*(\plin(S))$ gives a Poisson structure defined on an open around $S$ whose infinite jet is $\hat{\pi}$.
Hence, we proved the following statement.
\begin{corollary}
Let $\hat{\pi}\in J^{\infty}_S(\mathfrak{X}^{2}_S(M))$ be a formal Poisson structure, for which $S$ is a symplectic
leaf. If the algebroid $A_S$ induced by $j^1_{|S}\hat{\pi}$ satisfies
\[H^{2}(A_S;\mathcal{S}^{k}(\nu_S^*))=0, \ \ \forall\  k\geq 2,\]
then, on some open around $S$, there exists a Poisson structure $\pi$ such that \[\hat{\pi}=j^{\infty}_{|S}\pi.\]
\end{corollary}

\subsection{Proofs of the criteria}

In this section we prove the corollaries from the section \ref{section_statement_3}. First, we summarize some of the
results from \cite{Cra} in the form of a lemma (see also section \ref{Section_COHOMOLOGY} for the cohomologies
involved and some of the vanishing results).

\begin{lemma}\label{vanishing_proposition}
Let $\mathcal{G}$ be a Lie groupoid over $N$ with Lie algebroid $A$ and $E\to N$ a representation of
$\mathcal{G}$.\index{differentiable cohomology}\index{cohomology, Lie algebroid}
\begin{itemize}
\item[(1)] If the $s$-fibers of $\mathcal{G}$ are cohomologically 2-connected, then \[H^2(A,E)\cong H^2_{\mathrm{diff}}(\mathcal{G},E).\]
\item[(2)] If $\mathcal{G}$ is proper, then \[H^2_{\mathrm{diff}}(\mathcal{G},E)=0.\]
\item[(3)] If $\mathcal{G}$ is transitive, then
\[H^2_{\mathrm{diff}}(\mathcal{G},E)\cong H^2_{\mathrm{diff}}(\mathcal{G}_x,E_x),\ x\in N,\]
where $\mathcal{G}_x:=s^{-1}(x)\cap t^{-1}(x)$.
\end{itemize}
\end{lemma}
\begin{proof}
For (1) see Theorem 4 \cite{Cra}. For (2) see Proposition 1 \cite{Cra}. For (3), since $\mathcal{G}$ is transitive, it is
Morita equivalent to $\mathcal{G}_x$ \cite{MM}. By Theorem 1 \cite{Cra}, Morita equivalences induce isomorphisms
in differentiable cohomology.
\end{proof}

\begin{proof}[Proof of Corollary \ref{C_Dual_ss_Lie}]

Recall (subsection \ref{Subsection_examples_of_symplectic_groupoids}) that the cotangent Lie algebroid of
$(\mathfrak{g}^*,\pi_{{\mathfrak{g}}})$ is integrable by the action groupoid $G\ltimes \mathfrak{g}^*$, where $G$
denotes the 1-connected Lie group of $\mathfrak{g}$. By assumption, $G$ is compact. The symplectic leaves of
$\pi_{\mathfrak{g}}$ are the coadjoint orbits, and since the inner product is invariant,
$\mathbb{S}({\mathfrak{g}^*})$ is a union of symplectic leaves; thus, a Poisson submanifold. Also, the algebroid
$A_{\mathbb{S}({\mathfrak{g}^*})}$ is integrable by the action groupoid $G\ltimes
\mathbb{S}({\mathfrak{g}}^*)$. Since $G$ is simply connected it follows that $H^2(G)=0$ (see Theorem 1.14.2
\cite{DK}). On the other hand, the $s$-fibers of $G\ltimes \mathbb{S}({\mathfrak{g}^*})$ are diffeomorphic to $G$;
and so, the assumptions of Lemma \ref{vanishing_proposition} (1) are satisfied. Hence, for any representation $E\to
\mathbb{S}({\mathfrak{g}})$ of $G\ltimes \mathbb{S}({\mathfrak{g}})$ we have that
\[H^2(A_{\mathbb{S}({\mathfrak{g}^*})}; E)\cong H^2_{\mathrm{diff}}(G\ltimes \mathbb{S}({\mathfrak{g}^*});E).\]
Since $G\ltimes \mathbb{S}({\mathfrak{g}^*})$ is compact it is proper, by Lemma \ref{vanishing_proposition} (2),
we have that $H^2_{\mathrm{diff}}(G\ltimes \mathbb{S}({\mathfrak{g}^*});E)=0$. The corollary follows from
Theorem \ref{Theorem_THREE}.
\end{proof}
\begin{proof}[Proof of Corollary \ref{Theorem2}]
Let $P$ be the Poisson homotopy bundle of $S$, and let $G$ be its structure group. By hypothesis, $P$ is smooth,
$1$-connected and has vanishing second de Rham cohomology. Let $\mathcal{G}:=P\times_{G}P$ be the gauge
groupoid of $P$. The $s$-fibers of $\mathcal{G}$ are diffeomorphic to $P$; thus, $\mathcal{G}$ satisfies the
assumptions of Lemma \ref{vanishing_proposition} (1). Therefore
\[H^{2}(A_S;\mathcal{S}^k(\nu_S^*))\cong H^{2}_{\mathrm{diff}}(\mathcal{G};\mathcal{S}^k(\nu_S^*)).\]
Since $\mathcal{G}$ is transitive, by Lemma \ref{vanishing_proposition} (3), we have that
\[H^{2}_{\mathrm{diff}}(\mathcal{G};\mathcal{S}^k(\nu_S^*))\cong H^{2}_{\mathrm{diff}}(G;\mathcal{S}^k(\nu_{S,x}^*)).\]
Since $\nu_{S,x}^*\cong \mathfrak{g}$ as $G$ representations, Theorem \ref{Theorem1} implies the conclusion.
\end{proof}
\begin{proof}[Proof of Corollary \ref{Corollary_2}]
By Lemma \ref{vanishing_proposition} (2), the differentiable cohomology of compact groups vanishes; thus, Corollary
\ref{Theorem2} implies the result.
\end{proof}
\begin{proof}[Proof of Corollary \ref{Theorem3}]
Let $x\in S$ and denote by $\mathfrak{g}_x:=\nu_{S,x}^*$ the isotropy Lie algebra of the transitive algebroid $A_S$.
By hypothesis, $\mathfrak{g}_x$ is reductive, i.e.\ it splits as a direct product of a semisimple Lie algebra and its
center $\mathfrak{g}_x=\mathfrak{s}_x\oplus \mathfrak{z}_x$, where
$\mathfrak{s}_x=[\mathfrak{g}_x,\mathfrak{g}_x]$ and $\mathfrak{z}_x=Z(\mathfrak{g}_x)$ is the center of
$\mathfrak{g}_x$. Since $\mathfrak{g}:=\nu_S^*$ is a Lie algebra bundle, it follows that this splitting is in fact
global:
\[\mathfrak{g}=[\mathfrak{g},\mathfrak{g}]\oplus Z(\mathfrak{g})=\mathfrak{s}\oplus\mathfrak{z}.\]
Since $\mathfrak{s}$ is an ideal of $A_S$, we obtain a short exact sequence of algebroids
\[0\rmap \mathfrak{s}\rmap  A_S\rmap  A_S^{\textrm{ab}}\rmap  0,\]
with $A_S^{\textrm{ab}}=A_S/\mathfrak{s}$. Similar to the spectral sequence for Lie algebra extensions (see e.g.\
\cite{Serre}), there is a spectral sequence for extensions of Lie algebroids (see \cite{MackenzieLL}, Theorem 5.5 and
the remark following it), which, in our case, converges to $H^{\bullet}(A_S;\mathcal{S}^{k}(\mathfrak{g}))$, with
\[E_2^{p,q}=H^p(A_S^{\textrm{ab}};H^q(\mathfrak{s};\mathcal{S}^k(\mathfrak{g})))\Rightarrow H^{p+q}(A_S;\mathcal{S}^{k}(\mathfrak{g})).\]
Since $\mathfrak{s}$ is in the kernel of the anchor, $H^q(\mathfrak{s};\mathcal{S}^k(\mathfrak{g}))$ is indeed a
vector bundle, with fiber
$H^q(\mathfrak{s};\mathcal{S}^k(\mathfrak{g}))_x=H^q(\mathfrak{s}_x;\mathcal{S}^k(\mathfrak{g}_x))$ and it
inherits a representation of $A_S^{\textrm{ab}}$. Since $\mathfrak{s}_x$ is semisimple, by the Whitehead Lemma
we have that $H^1(\mathfrak{s}_x;\mathcal{S}^k(\mathfrak{g}_x))=0$ and
$H^2(\mathfrak{s}_x;\mathcal{S}^k(\mathfrak{g}_x))=0$. Therefore,
\begin{equation}\label{IncaUna}
H^{2}(A_S;\mathcal{S}^{k}(\mathfrak{g}))\cong H^2(A_S^{\textrm{ab}};\mathcal{S}^k(\mathfrak{g})^{\mathfrak{s}}),
\end{equation}
where $\mathcal{S}^k(\mathfrak{g}_x)^{\mathfrak{s}_x}$ is the $\mathfrak{s}_x$ invariant part of
$\mathcal{S}^k(\mathfrak{g}_x)$. By hypothesis, $A_S^{\textrm{ab}}$ is integrable by a 1-connected principal
bundle $P^{\mathrm{ab}}$ with vanishing second de Rham cohomology and compact structure group $T$. Therefore,
by (\ref{IncaUna}) and by applying Lemma \ref{vanishing_proposition}, we obtain
\begin{align*}
H^{2}(A_S;\mathcal{S}^{k}(\mathfrak{g}))&\cong H^2(A_S^{\textrm{ab}};\mathcal{S}^k(\mathfrak{g})^{\mathfrak{s}})\cong
H^2_{\mathrm{diff}}(P^{\mathrm{ab}}\times_{T}P^{\mathrm{ab}};\mathcal{S}^k(\mathfrak{g})^{\mathfrak{s}}) \cong \\
&\cong H^2_{\mathrm{diff}}(T;\mathcal{S}^k(\mathfrak{g}_x)^{\mathfrak{s}_x})=0.
\end{align*}
Again, Theorem \ref{Theorem1} concludes the proof.
\end{proof}

\begin{proof}[Proof of Corollary \ref{Corollary_3}]
Assume that $\mathfrak{g}_{x}$ is semisimple, $\pi_1(S,x)$ is finite and $\pi_2(S,x)$ is a torsion group. With the
notation from above, we have that $A_S^{\textrm{ab}}\cong TS$. $TS$ is integrable and the 1-connected principal
bundle integrating it is $\widetilde{S}$, the universal cover of $S$. Finiteness of $\pi_1(S)$ is equivalent to
compactness of the structure group of $\widetilde{S}$. By the Hurewicz theorem
$H_2(\widetilde{S},\mathbb{Z})\cong \pi_2(\widetilde{S})$ and since $\pi_2(\widetilde{S})=\pi_2(S)$ is torsion, we
have that $H^2(\widetilde{S})=0$. So the result follows from Corollary \ref{Theorem3}.
\end{proof}

\section{Appendix: Equivalence of MC-elements in complete GLA's}\label{section_equivalence}

In this appendix we discuss some general facts about graded Lie algebras endowed with a complete filtration, with the
aim of proving a criterion for equivalence of Maurer-Cartan elements (Theorem \ref{Teo1}), which was used in the
proof of Theorem \ref{Theorem_THREE}. Some of the constructions given here can be also found in Appendix B.1 of
\cite{Bursztyn} in the more general setting of differential graded Lie algebras with a complete filtration. In fact all our
constructions can be adapted to this setup, in particular also Theorem \ref{Teo1}. The analog of Theorem \ref{Teo1},
in the case of differential graded associative algebras can be found in the Appendix A of \cite{CMB}.

\begin{definitions}\rm
A \textbf{graded Lie algebra}\index{graded Lie algebra} (\textbf{GLA}) consists of a $\mathbb{Z}$-graded vector
space $\mathcal{L}^{\bullet}$ endowed with a graded bracket $[\cdot,\cdot]:\mathcal{L}^{p}\times
\mathcal{L}^{q}\to\mathcal{L}^{p+q}$, which is graded commutative and satisfies the graded Jacobi identity:
\[[X,Y]=-(-1)^{|X||Y|}[Y,X] ,\ \ [X,[Y,Z]]=[[X,Y],Z]+(-1)^{|X||Y|}[Y,[X,Z]].\]

An element $\gamma\in \mathcal{L}^1$ is called a \textbf{Maurer Cartan element}\index{Maurer Cartan element} if
$[\gamma,\gamma]=0$.

A \textbf{filtration}\index{filtration} on a GLA is a decreasing sequence of homogeneous subspaces:
\[\mathcal{L}^{\bullet}\supset\mathcal{F}_{0}\mathcal{L}^{\bullet}\supset \ldots \supset\mathcal{F}_n\mathcal{L}^{\bullet}\supset \mathcal{F}_{n+1}\mathcal{L}^{\bullet}\supset\ldots,\]
satisfying
\[[\mathcal{F}_{n}\mathcal{L},\mathcal{F}_{m}\mathcal{L}]\subset
\mathcal{F}_{n+m}\mathcal{L},\ \ [\mathcal{L},\mathcal{F}_{n}\mathcal{L}]\subset
\mathcal{F}_{n}\mathcal{L}.\]

A filtration $\mathcal{F}\mathcal{L}$ is called \textbf{complete}\index{complete filtration}, if $\mathcal{L}$ is
isomorphic to the projective limit $\varprojlim \mathcal{L}/\mathcal{F}_{n}\mathcal{L}$.
\end{definitions}

An example of a GLA with a complete filtration appeared in section \ref{Formal equivalence of Poisson structures
around symplectic leaves}: starting from a manifold $M$ with a closed embedded submanifold $S\subset M$, we
constructed the algebra of formal vector fields along $S$,
\[(J^{\infty}_S(\mathfrak{X}^{\bullet+1}_S(M)),[\cdot,\cdot]),\]
with filtration given by the powers of the vanishing ideal of $S$. So, the index of the filtration is the order to which
elements vanish along $S$.

For a general GLA with a complete filtration $\mathcal{FL}$, define the \textbf{order} of an element as follows:
\begin{align*}
&\mathcal{O}:\mathcal{L}\rmap \{0, 1, \ldots, \infty\},\\
&\mathcal{O}(X)=\left\{\begin{array}{lll} 0, &\textrm{if}& X\in \mathcal{L}\backslash
\mathcal{F}_{1}\mathcal{L},\\ n, &\textrm{if}& X\in \mathcal{F}_n\mathcal{L}\backslash \mathcal{F}_{n+1}\mathcal{L},\\
\infty, &\textrm{if} & X=0.\end{array}\right.
\end{align*}
The order has the following properties:
\begin{itemize}
\item $\mathcal{O}(X)=\infty$ if and only if $X=0$,
\item $\mathcal{O}(X+Y)\geq \mathcal{O}(X)\wedge\mathcal{O}(Y):=\mathrm{min}\{\mathcal{O}(X),\mathcal{O}(Y)\}$,
\item $\mathcal{O}(\alpha X)\geq \mathcal{O}(X), \ \ \forall \ \alpha\in\mathbb{R}$,
\item $\mathcal{O}([X,Y])\geq \mathcal{O}(X)+\mathcal{O}(Y)$.
\end{itemize}
Completeness of the filtration implies a criterion for convergence.
\begin{lemma}\label{LemConv}
Let $\{X_n\}_{n\geq 0}\in\mathcal{L}$ be a sequence of elements such that \[\lim_{n\to
\infty}\mathcal{O}(X_n)=\infty.\]
There exists a unique element $X\in \mathcal{L}$, denoted $X:=\sum_{n\geq 0}
X_n$ such that
\[X-\sum_{k=0}^nX_k\in\mathcal{F}_m\mathcal{L},\]
for all $n$ big enough.
\end{lemma}

Note that $\mathfrak{g}(\mathcal{L}):=\mathcal{F}_{1}\mathcal{L}^{0}$ forms a Lie subalgebra of
$\mathcal{L}^0$. For $X\in \mathfrak{g}(\mathcal{L})$, consider the operator
\[Y\mapsto ad_X(Y):=[X,Y].\]
These operators satisfy \[\mathcal{O}(ad_X(Y))\geq \mathcal{O}(Y)+1,\] for all $Y\in\mathcal{L}$, therefore, by
Lemma \ref{LemConv}, the exponential of $ad_X$ is well-defined
\[Ad(e^X):\mathcal{L}^{\bullet}\to\mathcal{L}^{\bullet},\quad Ad(e^X)Y:=e^{ad_X}(Y)=\sum_{n\geq 0}\frac{ad_X^n}{n!}(Y).\]
Moreover, $Ad(e^X)$ is a GLA-automorphism of $\mathcal{L}^{\bullet}$. By Lemma \ref{LemConv}, the
Campbell-Hausdorff formula converges for $X,Y\in\mathfrak{g}(\mathcal{L})$
\begin{eqnarray}\label{Campbell-Hausdorff}
&&X*Y=X+Y+\sum_{k\geq 1}\frac{(-1)^{k}}{k+1}D_k(X,Y),\ \ \textrm{where}\\
\nonumber && D_k(X,Y)=\sum_{ l_i + m_i > 0}\frac{ad_X^{\phantom{o}l_1}}{l_1!}\circ\frac{ad_Y^{\phantom{o}m_1}}{m_1!}\circ\ldots
\circ\frac{ad_X^{\phantom{o}l_k}}{l_k!}\circ\frac{ad_Y^{\phantom{o}m_k}}{m_k!}(X).
\end{eqnarray}
We will use the notation $\mathcal{G}(\mathcal{L})=\{e^X| X\in \mathfrak{g}(\mathcal{L})\}$, i.e.\
$\mathcal{G}(\mathcal{L})$ is the same space as $\mathfrak{g}(\mathcal{L})$, but we just denote its elements by
$e^X$. The universal properties of the Campbell-Hausdorff formula (\ref{Campbell-Hausdorff}), imply that
$\mathcal{G}(\mathcal{L})$ endowed with the product $e^Xe^Y=e^{X*Y}$ forms a group. Moreover, $Ad$ gives an
action of $\mathcal{G}(\mathcal{L})$ on $\mathcal{L}$ by graded Lie algebra automorphisms, which preserves the
order:
\begin{itemize}
\item $Ad(e^{X*Y})=Ad(e^Xe^Y)=Ad(e^X)\circ Ad(e^Y)$
\item $Ad(e^X)([U,V])=[Ad(e^X)U,Ad(e^X)V]$,
\item $\mathcal{O}(Ad(e^X)(U))=\mathcal{O}(U)$,
\end{itemize}
for all $X,Y\in\mathfrak{g}(\mathcal{L})$ and all $U,V\in \mathcal{L}$.

For later use, we give the following straightforward estimates:
\begin{lemma}\label{PropGroup}
For all $X,Y,X',Y'\in \mathfrak{g}(\mathcal{L})$ and $U\in\mathcal{L}$, we have that
\begin{itemize}
\item[(a)] $\mathcal{O}(X*Y-X'*Y')\geq \mathcal{O}(X-X')\wedge\mathcal{O}(Y-Y'),$
\item[(b)] $\mathcal{O}(Ad(e^X)U-Ad(e^Y)U)\geq \mathcal{O}(X-Y)$.
\end{itemize}
\end{lemma}

Let $\gamma$ be a MC-element. Notice that $[\gamma,\gamma]=0$, implies that $d_{\gamma}:=ad_{\gamma}$ is a
differential on $\mathcal{L}^{\bullet}$. The fact that $\mathcal{F}_k\mathcal{L}$ are ideals implies that
$(\mathcal{F}_k\mathcal{L}^{\bullet},d_{\gamma})$ are subcomplexes of $(\mathcal{L}^{\bullet},d_{\gamma})$.
The induced differential on the quotient of consecutive complexes depends only on $\gamma$ modulo
$\mathcal{F}_1$, and their cohomology groups will be denoted
\[H^n_{\gamma}(\mathcal{F}_{k}\mathcal{L}^{\bullet}/\mathcal{F}_{k+1}\mathcal{L}^{\bullet}).\]

For $e^X\in\mathcal{G}(\mathcal{L})$, $Ad(e^X)\gamma$ is again a MC-element, and we will call $\gamma$ and
$Ad(e^X)\gamma$ \textbf{gauge equivalent}. The next lemma gives a linear approximation of the action
$\mathcal{G}(\mathcal{L})$ on MC-elements.
\begin{lemma}\label{LemAprox}
For $\gamma$ a MC-element and $e^X\in\mathcal{G}(\mathcal{L})$, we have that
\[\mathcal{O}(Ad(e^X)\gamma-\gamma+d_{\gamma}X)\geq 2 \mathcal{O}(X).\]
\end{lemma}

We have the following criterion for gauge equivalence.
\begin{theorem}\label{Teo1}
Let $(\mathcal{L}^{\bullet},[\cdot,\cdot])$ be a GLA with a complete filtration $\mathcal{F}_n\mathcal{L}$. Let
$\gamma,\gamma'$ be two Maurer Cartan elements such that $\mathcal{O}(\gamma-\gamma')\geq 1$ and
\[\ \ H^1_{\gamma}(\mathcal{F}_{q}\mathcal{L}^{\bullet}/\mathcal{F}_{q+1}\mathcal{L}^{\bullet})=0,\ \ \forall\ q\geq \mathcal{O}(\gamma-\gamma').\]
Then $\gamma$ and $\gamma'$ are gauge equivalent, i.e.\ there exists an element $e^X\in \mathcal{G}(\mathcal{L})$
such that $\gamma=Ad(e^X)\gamma'$.
\end{theorem}

\begin{proof}
Denote by $p:=\mathcal{O}(\gamma-\gamma')$. By hypothesis, for $q\geq p$, we can find homotopy operators
\[h^q_1:\mathcal{F}_{q}\mathcal{L}^1\rmap \mathcal{F}_{q}\mathcal{L}^0\ \  \textrm{and}\ \  h^q_2:\mathcal{F}_{q}\mathcal{L}^2\rmap
\mathcal{F}_{q}\mathcal{L}^1\] such that $h^q_1(\mathcal{F}_{q+1}\mathcal{L}^1)\subset \mathcal{F}_{q+1}\mathcal{L}^0$,
$h^q_2(\mathcal{F}_{q+1}\mathcal{L}^2)\subset \mathcal{F}_{q+1}\mathcal{L}^1$ and
\[(d_{\gamma}h^q_1+h^q_2d_{\gamma}-\textrm{Id})(\mathcal{F}_{q}\mathcal{L}^1)\subset \mathcal{F}_{q+1}\mathcal{L}^1.\]

We first prove an estimate. Let $q\geq p$ and $\tilde{\gamma}$ a MC-element, with
$\mathcal{O}(\tilde{\gamma}-\gamma)\geq q$. Then for ${\widetilde{X}}:=h^{q}_1(\tilde{\gamma}-\gamma)$, we
claim that the following hold:
\begin{equation}\label{EQ3}
\mathcal{O}({\widetilde{X}})\geq {q}, \ \ \mathcal{O}(Ad(e^{\widetilde{X}})\tilde{\gamma}-\gamma)\geq {q+1}.
\end{equation}
The first holds by the properties of $h^{q}_1$. To prove the second, we compute:
\begin{align*}
\mathcal{O}(Ad(e^{\widetilde{X}})\tilde{\gamma}-\gamma)&\geq \mathcal{O}(Ad(e^{\widetilde{X}})\tilde{\gamma}-\tilde{\gamma}+d_{\tilde{\gamma}}({\widetilde{X}}))\wedge \mathcal{O}(\tilde{\gamma}-d_{\tilde{\gamma}}({\widetilde{X}})-\gamma)\geq\\
&\geq 2 \mathcal{O}({\widetilde{X}})\wedge\mathcal{O}([\gamma-\tilde{\gamma},{\widetilde{X}}])\wedge\mathcal{O}(\tilde{\gamma}-\gamma-d_{\gamma}({\widetilde{X}}))\geq\\
&\geq 2q\wedge(\mathcal{O}(\gamma-\tilde{\gamma})+\mathcal{O}({\widetilde{X}}))\wedge\mathcal{O}(\tilde{\gamma}-\gamma-d_{\gamma}({\widetilde{X}}))\geq\\
&\geq 2q\wedge\mathcal{O}((\textrm{Id}-d_{\gamma}h^{q}_1)(\tilde{\gamma}-\gamma)),
\end{align*}
where, for the second inequality, we used Lemma \ref{LemAprox}. The last term can be evaluated as follows:
\begin{align*}
\mathcal{O}((\textrm{Id}-d_{\gamma}h^{q}_1)(\tilde{\gamma}-\gamma))&\geq \mathcal{O}((\textrm{Id}-d_{\gamma}h^{q}_1-h^{q}_2d_{\gamma})(\tilde{\gamma}-\gamma))\wedge
\mathcal{O}(h^{q}_2(d_{\gamma}(\tilde{\gamma}-\gamma)))\\
&\geq (q+1)\wedge \mathcal{O}(h^{q}_2(d_{\gamma}(\tilde{\gamma}-\gamma))).
\end{align*}
Since $d_{\gamma}(\tilde{\gamma}-\gamma)=-\frac{1}{2}[\tilde{\gamma}-\gamma,\tilde{\gamma}-\gamma]$, we
have that \[\mathcal{O}(d_{\gamma}(\tilde{\gamma}-\gamma))\geq 2q\geq q+1,\] hence
$h^{q}_2(d_{\gamma}(\tilde{\gamma}-\gamma))\in \mathcal{F}_{q+1}\mathcal{L}^1$, and this proves (\ref{EQ3}).

We construct a sequence of MC-elements $\{\gamma_k\}_{k\geq 0}$ and a sequence of group elements
$\{e^{X_k}\}_{k\geq 1}\in\mathcal{G}(\mathcal{L})$ by the following recursive formulas:
\begin{align*}
\gamma_0&:=\gamma',\\
X_{k}&:=h_1^{p+k-1}(\gamma_{k-1}-\gamma), \ \ \textrm{ for }k\geq 1,\\
\gamma_{k}&:=Ad(e^{X_{k}})\gamma_{k-1},  \ \ \textrm{ for }k\geq 1.
\end{align*}
To show that this formulas give indeed well-defined sequences, we have to check that
$\gamma_{k-1}-\gamma\in\mathcal{F}_{p+k-1}\mathcal{L}^1$. This holds for $k=1$, and in general, by applying
inductively, at each step $k\geq 1$, the estimate (\ref{EQ3}) with $\widetilde{\gamma}=\gamma_{k-1}$ and
$q=p+k-1$, we obtain:
\[\mathcal{O}(X_{k})\geq p+k-1, \ \ \mathcal{O}(\gamma_{k}-\gamma)\geq p+k.\]

Using Lemma \ref{PropGroup} (a), we obtain that
\[\mathcal{O}(X_k*X_{k-1}\ldots*X_1-X_{k-1}\ldots*X_1)\geq \mathcal{O}(X_k)\geq {p+k-1},\]
therefore, by Lemma \ref{LemConv}, the product $X_k*X_{k-1}*\ldots*X_1$ converges to some element $X$. Applying
Lemma \ref{PropGroup} (a) $k$ times, we obtain
\[\mathcal{O}(X_k*X_{k-1}\ldots*X_1)\geq \mathcal{O}(X_{k})\wedge\mathcal{O}(X_{k-1})\wedge\ldots \wedge\mathcal{O}(X_1)\geq 1,\]
thus $X\in\mathfrak{g}(\mathcal{L})$. On the other hand, we have that
\begin{eqnarray*}
\mathcal{O}(Ad(e^X)\gamma'-\gamma)&\geq& \mathcal{O}(Ad(e^X)\gamma'-\gamma_{k})\wedge \mathcal{O}(\gamma_{k}-\gamma)\geq\\
&\geq& \mathcal{O}(Ad(e^X)\gamma'-Ad(e^{X_k*\ldots
*X_1})\gamma')\wedge
(p+k)\geq\\
&\geq&\mathcal{O}(X-X_k*\ldots *X_1)\wedge (p+k),
\end{eqnarray*}
where for the last estimate we have used Lemma \ref{PropGroup} (b). If we let $k\to \infty$ we obtain the conclusion:
$Ad(e^X)\gamma'=\gamma$.
\end{proof}

\clearpage \pagestyle{plain}

\chapter{Rigidity around Poisson submanifolds}\label{ChRigidity}

\pagestyle{fancy}
\fancyhead[CE]{Chapter \ref{ChRigidity}} 
\fancyhead[CO]{Rigidity around Poisson submanifolds}

Conn's original proof of the linearization of Poisson structures around fixed points is analytical; it uses the Nash-Moser
fast convergence method (see chapter \ref{ChBasicPoisson}). The geometric approach to Conn's theorem from
\cite{CrFe-Conn}, allowed us to obtain Theorem \ref{Theorem_TWO}, which generalizes Conn's result to arbitrary
symplectic leaves. In this chapter we close the circle and provide an analytical approach to this case of symplectic
leaves. We obtain a local rigidity result (Theorem \ref{Theorem_FOUR}) for Poisson structure that are integrable by a
Hausdorff Lie groupoid whose $s$-fibers are compact and have vanishing second de Rham cohomology. Moreover, we
prove that such structures are determined by their first jet around compact Poisson submanifolds, and this gives an
improvement of Theorem \ref{Theorem_TWO}, which is noticeable already for fixed points. The content of this chapter
is available as a preprint at \cite{MarRig}.

\section{Statement of Theorem \ref{Theorem_FOUR}}\label{introduction_rigi}

Recall that the \textbf{compact-open}\index{compact-open} $C^k$-topology on the space of smooth sections
$\Gamma(F)$ of a fiber bundle $F\to M$ is generated by the opens
\[\mathcal{O}(K,U):=\{f\in\Gamma(F) | j^k(f)(K)\subset U\},\]
where $K\subset M$ is a compact set and $U$ is an open subset of $J^k(F)$, the $k$-th jet bundle of $F$. For a second
manifold $N$, the space of smooth maps from $M$ to $N$ is endowed with the compact-open $C^k$-topology, via the
identification
\[C^{\infty}(M,N)=\Gamma(M\times N\to M).\]
The rigidity property that we will study in this chapter is defined in terms of the compact-open topology:
\begin{definition}\label{Definition_CpC1}
Let $(M,\pi)$ be a Poisson manifold and $S\subset M$ a compact submanifold. We say that $\pi$ is
$C^p$-$C^1$-\textbf{rigid around} $S$, if there are small enough open neighborhoods $U$ of $S$, such that for all
opens $O$ with $S\subset O\subset \overline{O}\subset  U$, there exist
\begin{itemize}
\item an open $\mathcal{V}_{O}\subset \mathfrak{X}^2(U)$ around $\pi_{|U}$ in the compact-open $C^p$-topology,
\item a function $\widetilde{\pi}\mapsto \psi_{\widetilde{\pi}}$, which associates to a Poisson structure $\widetilde{\pi}\in
\mathcal{V}_{O}$ an embedding $\psi_{\widetilde{\pi}}:\overline{O}\hookrightarrow M$,
\end{itemize}
such that $\psi_{\widetilde{\pi}}$ restricts to a Poisson diffeomorphism
\[\psi_{\widetilde{\pi}}:(O,\widetilde{\pi}_{|O})\diffto (\psi_{\widetilde{\pi}}(O),\pi_{|\psi_{\widetilde{\pi}}(O)}),\]
and $\psi$ is continuous at $\widetilde{\pi}=\pi$ (with $\psi_{\pi}=\mathrm{Id}_{\overline{O}}$), with respect to the
$C^p$-topology on the space of Poisson structures and the $C^1$-topology on $C^{\infty}(\overline{O},M)$.
\end{definition}

The main result of this chapter is the following rigidity theorem for integrable Poisson manifolds.

\begin{mtheorem}\label{Theorem_FOUR}
Let $(M,\pi)$ be a Poisson manifold for which the cotangent Lie algebroid is integrable by a Hausdorff Lie groupoid
whose $s$-fibers are compact and their de Rham cohomology vanishes in degree two. There exists $p>0$, such that for
every compact Poisson submanifold $S$ of $M$, the following hold
\begin{enumerate}[(a)]
\item $\pi$ is $C^p$-$C^1$-rigid around $S$,
\item the first order jet of $\pi$ at $S$ determines $\pi$ around $S$.
\end{enumerate}
\end{mtheorem}

For $p$ we find the (most probably not optimal) value:
\[p=7(\lfloor \mathrm{dim}(M)/2\rfloor+5).\]

In part (b) we prove that every Poisson structure $\widetilde{\pi}$ defined on an open containing $S$, that satisfies
$j^1_{|S}\pi=j^1_{|S}\widetilde{\pi}$, is isomorphic to $\pi$ around $S$ by a diffeomorphism which is the identity on
$S$ up to first order.

\subsection{An improvement of Theorem \ref{Theorem_TWO}}

As an immediate consequence of Theorem \ref{Theorem_FOUR}, we obtain the following improvement of Theorem
\ref{Theorem_TWO}, which also includes rigidity:

\begin{theorem}\label{The_normal_form_theorem}
Let $(M,\pi)$ be a Poisson manifold and let $(S,\omega_S)$ be a compact symplectic leaf. If the Lie algebroid
$A_S:=T^*M_{|S}$ is integrable by a compact Lie groupoid whose $s$-fibers have vanishing de Rham cohomology in
degree two, then
\begin{enumerate}[(a)]
\item $\pi$ is Poisson diffeomorphic around $S$ to its local model around $S$,
\item $\pi$ is $C^p$-$C^1$-rigid around $S$.
\end{enumerate}
\end{theorem}
\begin{proof}
Since $A_S$ is transitive, the hypothesis implies the existence of a principal $G$-bundle $P\to S$, such that $A_S\cong
TP/G$. Let $\Psi:\nu_S\to M$ be a tubular neighborhood of $S$ in $M$, and consider the corresponding first order
approximation of $\pi$ around $S$ (constructed in subsection \ref{Subsection_the_local_model_the_general_case}),
which is a Poisson structure $\plin(S)$ on an open around $S$. By Proposition \ref{Proposition_reconcile}, $\Psi$ sends
$\pi_{A_S}$, the local model constructed using $A_S$ and the splitting $d\Psi_{|S}^*$, to $\plin(S)$. By Proposition
\ref{Proposition_reconcile_3}, $(N(A_S),\pi_{A_S})$ coincides with $(N(P),\pi_P)$, the local model corresponding to
$P$ (constructed in subsection \ref{The local model}). By Proposition \ref{Lemma_small_neighborhoods}, there are
arbitrarily small opens $U\subset N(P)$, containing $S$, such that $(U,\pi_{P|U})$ is integrable by a Hausdorff Lie
groupoid whose $s$-fibers are diffeomorphic to $P$. Since $P$ is compact and $H^2(P)=0$, the Poisson manifold
$(U,\pi_{P|U})$ satisfies the conditions of Theorem \ref{Theorem_FOUR}, thus also $\plin(S)$ will satisfy these
conditions on $\Psi(U)$. By part (a) of Theorem \ref{Theorem_FOUR}, $\plin(S)_{|\Psi(U)}$ is $C^p$-$C^1$-rigid
around $S$, and since $\plin(S)$ and $\pi$ have the same first order jet at $S$, part (b) of the theorem implies that they
are isomorphic around $S$. In particular, also $\pi$ is $C^p$-$C^1$-rigid around $S$.
\end{proof}

Part (a) of this theorem is a slight improvement of Theorem \ref{Theorem_TWO}: both results require the same
conditions on a Lie groupoid, here this groupoid could be any integration of $A_S$, but in Theorem \ref{Theorem_TWO}
it has to be the fundamental integration of $A_S$ (i.e.\ the $s$-fiber 1-connected).

Already in the case of fixed points, part (a) is stronger than Conn's theorem. Namely, a Lie algebra is integrable by a
compact group with vanishing second de Rham cohomology if and only if it is compact and its center is at most
one-dimensional (Lemma \ref{Lemma_cohomology_lie_groups}). The case when the center is trivial is Conn's result,
and the one-dimensional case is a consequence of a result of Monnier and Zung on smooth Levi decomposition of
Poisson manifolds \cite{Monnier}.

Another class of Poisson structures, where the difference in strength between these two results is easily
noticeable, is that of trivial symplectic foliations. This is discussed in subsection \ref{Trivial_symplectic_foliations}.\\

For a Poisson submanifold $S$ of $(M,\pi)$, the structure encoded by $j^1_{|S}\pi$ can be organized as an extension
of Lie algebroids (see subsection \ref{The first order data})
\begin{equation}\label{EQ_I_1}
0\rmap \nu_S^*\rmap A_S\rmap T^*S\rmap 0,
\end{equation}
where $T^*S$ is the cotangent Lie algebroid of the Poisson manifold $(S,\pi_{|S})$. One might try to follow the same
line of reasoning as in the proof of Theorem \ref{The_normal_form_theorem} above, and use Theorem
\ref{Theorem_FOUR} to obtain a normal form result around Poisson submanifolds. Unfortunately, around general
Poisson submanifolds, a first order local model does not seem to exist. As explained in Example
\ref{Example_Not_every_X_is_Y}, there are Lie algebroid extensions as in (\ref{EQ_I_1}) which do not arise as the first
jet of Poisson structures. Nevertheless, one can use Theorem \ref{Theorem_FOUR} to prove normal form results
around particular classes of Poisson submanifolds.

\subsection{About the proof}

The proof of Theorem \ref{Theorem_FOUR} is inspired mainly by Conn's paper \cite{Conn}. As explained in chapter
\ref{ChBasicPoisson}, Conn uses a technique due to Nash and Moser to construct a sequence of changes of
coordinates in which $\pi$ converges to the linear structure $\pi_{\mathfrak{g}_x}$. At every step the new
coordinates are found by solving some equations which are regarded as belonging to the complex computing the
Poisson cohomology of $\pi_{\mathfrak{g}_x}$. To account for the ``loss of derivatives'' phenomenon during this
procedure he uses smoothing operators. Finally, he proves uniform convergence of these changes of coordinates and of
their higher derivatives on some ball around $x$.

Conn's proof has been formalized in \cite{Miranda,Monnier} into an abstract Nash Moser normal form theorem. It is
likely that part (a) of our Theorem \ref{Theorem_FOUR} could be proven using Theorem 6.8 in \cite{Miranda}. Due to
some technical issues (see Remark \ref{Remark_SCI}), we cannot apply this result to conclude neither part (b) of our
Theorem \ref{Theorem_FOUR} nor Theorem \ref{The_normal_form_theorem}, therefore we follow a direct approach.

We prove Theorem \ref{Theorem_FOUR} in section \ref{section_technical_rigidity}, using the Nash-Moser method. We
simplify Conn's argument by giving coordinate free statements and working with flows of vector fields. For the expert:
we gave up on the polynomial-type inequalities using instead only inequalities that assert tameness of certain maps,
i.e.\ we work in Hamilton's category of tame Fr\'echet spaces.

Our proof deviates the most from Conn's when constructing the homotopy operators. Conn recognizes the Poisson
cohomology of $\pi_{\mathfrak{g}_x}$ as the Chevalley-Eilenberg cohomology of $\mathfrak{g}_x$ with coefficients
in the Fr\'echet space of smooth functions. By passing to the Lie group action on the corresponding Sobolev spaces, he
proves existence of tame (in the sense of Hamilton \cite{Ham}) homotopy operators for this complex. We, on the other
hand, regard this cohomology as Lie algebroid cohomology, and prove a general tame vanishing result for the
cohomology of Lie algebroids integrable by groupoids with compact $s$-fibers. We call this result the Tame Vanishing
Lemma, and we devote the appendix to its proof. This general result can be applied to similar geometric problems.
Such an application is presented also in the appendix, where we briefly review an unpublished paper of Hamilton on
rigidity of foliations.

\section{Remarks, examples and applications}\label{section_examples}

In this section we give a list of examples and applications for our two theorems and we also show some links with other
results from the literature.

\subsection{A global conflict}\label{subsection_a_global_conflict}

Theorem \ref{Theorem_FOUR} does not exclude the case when the Poisson submanifold $S$ is the total space $M$,
and the conclusion is that a compact Poisson manifold for which the Lie algebroid $T^*M$ is integrable by a compact
groupoid $\mathcal{G}$ whose $s$-fibers have vanishing $H^2$ is globally rigid. Nevertheless, this result is useless,
since the only example of such a manifold is $\mathbb{S}^1$, for which the trivial Poisson structure is clearly rigid. In
the case when $\mathcal{G}$ has 1-connected $s$-fibers, this conflict was pointed out in \cite{CrFe0}, and we explain
below the general case.

In symplectic geometry, this non-rigidity phenomenon can be easily remarked: if $(M,\omega)$ is a compact symplectic
manifold, then $t\omega$ (for $t>0$) is a nontrivial deformation of $\omega$, since it gives a different symplectic
volume.

Let $(M,\pi)$ be a Poisson manifold, for which $T^*M$ is integrable by a compact Lie groupoid $\mathcal{G}$ whose
$s$-fibers have trivial second de Rham cohomology. We will prove in Theorem \ref{Theorem_tame_vanishing}, that its
second Poisson cohomology vanishes. In particular the class $[\pi]$ is trivial, so there exists a vector field $X$ such
that $L_X(\pi)=\pi$. This implies that the flow of $X$ gives a Poisson diffeomorphism
\begin{equation}\label{EQ_545}
\varphi_X^t:(M,\pi) \diffto (M,e^{-t}\pi).
\end{equation}

By compactness of $\mathcal{G}$, the leaves of $M$ are compact. Let $M^{\mathrm{reg}}$ be the open in $M$
where $\pi$ has maximal rank. We will prove in the lemma below, that the regular leaves (i.e.\ leaves in
$M^{\mathrm{reg}}$) have finite holonomy. For $(S,\omega_S)$ a regular leaf, denote by $\mathrm{Hol}(S)$ its
holonomy group, and by $\mathrm{Vol}(S)$ the symplectic volume of $S$. We also prove in the lemma that the
function
\[\mathrm{vh}:M^{\mathrm{reg}}\rmap \mathbb{R}, \ \ \mathrm{vh}(x):=\mathrm{Vol}(S)|\mathrm{Hol}(S)|, \ \ \textrm{for }x\in S\]
extends continuously to $M$, with $\mathrm{vh}(x)=0$, for $x\notin M^{\mathrm{reg}}$. Since $\mathrm{vh}$ is
defined in Poisson geometric terms, by (\ref{EQ_545}), we have that $\mathrm{vh}\circ
\varphi_X^t=e^{tk}\mathrm{vh}$, where $2k$ denotes the maximal rank of $\pi$. Thus $\mathrm{vh}$ is
unbounded unless $\pi=0$. If $\pi=0$, then $\mathcal{G}\to M$ is a bundle of tori, so by the cohomological condition,
its fibers are at most one-dimensional. Hence also $M$ is at most 1-dimensional, thus $M$ is a point or
$M=\mathbb{S}^1$.

\begin{lemma}\label{Lemma_volume_holonomy}
The function $\mathrm{vh}$ extends continuously to $M$, with $\mathrm{vh}(x)=0$, for $x\notin
M^{\mathrm{reg}}$.
\end{lemma}
\begin{proof}
We may assume that $M$ is connected. We will show that every leaf of $M$ has a saturated neighborhood $U$, such
that $U^{\mathrm{reg}}$ is open and dense in $U$. Thus, since $M$ is connected, all leaves that are locally of
maximal dimension, are also globally of maximal dimension, and $M^{\mathrm{reg}}$ is open and dense in $M$.

Let $(S,\omega_S)$ be a symplectic leaf of $M$. We have that $\mathcal{G}_{|S}$ integrates $A_S$, therefore, by
Theorem \ref{The_normal_form_theorem}, the local model holds around $S$. So, for a compact, connected principal
$G$-bundle $P$ (the connected component of an $s$-fiber of $\mathcal{G}_{|S}$) we have that $(M,\pi)$ is Poisson
isomorphic around $S$ to an open around $S$ in the local model $(N(P),\pi_P)$. Recall from subsection \ref{The local
model} that the local model is constructed as the quotient of the symplectic manifold:
\[(\Sigma,\Omega)/G=(N(P),\pi_P),\]
where $\Omega=p^*(\omega_S)-d\widetilde{\theta}\in \Omega^2(P\times\mathfrak{g}^*)$ and $\Sigma$ is the open
in $P\times\mathfrak{g}^*$ where $\Omega$ is nondegenerate. The symplectic leaves of $\pi_P$ are of the form
\[(O_{\xi},\omega_{\xi}),\ \ \ O_{\xi}:=P\times_{G}\{G\xi\},\]
hence they are the base of the principal $G_{\xi}$-bundle
\[p_{\xi}:P\times\{\xi\}\rmap O_{\xi},\]
where $G_{\xi}$ is the stabilizer of $\xi$. We claim that
\begin{equation}\label{EQ_volume}
p_{\xi}^*(\omega_{\xi})=\Omega_{|P\times \{\xi\}}.
\end{equation}
This follows from a general fact about a symplectic realization whose proof is straightforward: Let
$\mathcal{F}^{\perp}$ be the foliation on $\Sigma$ whose tangent bundle is the symplectic orthogonal of the vertical
bundle. A leaf $L$ of $\mathcal{F}^{\perp}$ is mapped to a symplectic leaf $(S,\omega_S)$, and the pullback of
$\omega_S$ to $L$ is $\Omega_{|L}$. To see that we are in this situation, first recall that $T\mathcal{F}^{\perp}$ is
spanned by the Hamiltonian vector fields $H_{p^*(f)}$, for $f\in C^{\infty}(N(P))$. Since the action of $G$ is
Hamiltonian with $G$-equivariant moment map $\mu(q,\xi)=\xi$, it follows that these vector fields are tangent to the
fibers of $\mu$, i.e.\ to the submanifolds $P\times\{\xi\}$, for $\xi\in \mathfrak{g}^*$. Comparing dimensions, we have
that $T\mathcal{F}^{\perp}=(TP\times\mathfrak{g}^*)_{|\Sigma}$, thus (\ref{EQ_volume}) follows.

We will show that $\mathrm{vh}$ extends to a continuous map on $P\times_{G}\mathfrak{g}^*$. Let $T$ be a
maximal torus in $G$ and let $\mathfrak{t}$ be its Lie algebra. By compactness of $G$, we can consider an invariant
metric on $\mathfrak{g}$. This metric allows us to regard $\mathfrak{t}^*$ as a subspace in $\mathfrak{g}^*$ (i.e.\
the orthogonal to $\mathfrak{t}^{\circ}$), and also gives an isomorphism between the adjoint and the coadjoint
representation. For the adjoint representation it is well know (see e.g.\ \cite{DK}) that every orbit hits $\mathfrak{t}$,
hence also every orbit of the coadjoint action hits $\mathfrak{t}^*$. An element $\xi\in\mathfrak{t}^*$ is called
regular if $\mathfrak{g}_{\xi}=\mathfrak{t}$, where $\mathfrak{g}_{\xi}$ is the Lie algebra of $G_{\xi}$, and we
denote by $\mathfrak{t}^{*\mathrm{reg}}$ the set of regular elements. Then $\mathfrak{t}^{*\mathrm{reg}}$ is
open and dense in $\mathfrak{t}^*$ and it coincides with the set of elements $\xi$ for which $G_{\xi}/T$ is finite (see
e.g.\ \cite{DK}). Thus, for $\xi\in \mathfrak{t}^*$, a leaf $O_{\xi}$ has maximal dimension if and only if $\xi\in
\mathfrak{t}^{*\mathrm{reg}}$, hence the regular part of $\pi_P$ equals
\[N(P)^{\mathrm{reg}}=(P \times_G G\cdot\mathfrak{t}^{*\mathrm{reg}})\cap N(P).\]
This implies also the claims made about $M^{\mathrm{reg}}$ at the beginning of the proof.

Now, we fix $\xi\in \mathfrak{t}^{*\mathrm{reg}}$. By Theorem 3.7.1 \cite{DK} we have that
$(G^{\circ})_{\xi}=T$, therefore also $G_{\xi}^{\circ}=T$. Since $P$ is connected, the last terms in the long exact
sequence in homotopy associated to $p_{\xi}$ are
\begin{equation}\label{EQ_last_term_short}
\ldots\rmap \pi_1(O_{\xi})\stackrel{\theta}{\rmap} \pi_0(G_{\xi})\rmap 1.
\end{equation}
Thus we obtain a surjective group homomorphism $\theta:\pi_1(O_{\xi})\to G_{\xi}/T$. Explicitly, let $[q,\xi]\in O_{\xi}$
and $\gamma(t)$ be a closed loop at this point. Consider $\widetilde{\gamma}(t)$ a lift of $\gamma$ to $P$, with
$\widetilde{\gamma}(0)=q$. Since $p_{\xi}(\gamma(1),\xi)=[q,\xi]$, it follows that $\widetilde{\gamma}(1)=qg$, for
some $g\in G_{\xi}$. The map in (\ref{EQ_last_term_short}) is given by $\theta(\gamma)=[g]\in G_{\xi}/T$.

Next, we compute the holonomy group of $O_{\xi}$. Notice first that
\[T_{\xi}G\cdot{\xi}=\mathfrak{g}_{\xi}^{\circ}=\mathfrak{t}^{\circ}\subset \mathfrak{g}^*\cong T_{\xi}\mathfrak{g}^*,\]
and, since $\mathfrak{t}^{*}=\mathfrak{t}^{\circ,\perp}$, it follows that $\xi+\mathfrak{t}^*$ is transverse at
$\xi$ to the coadjoint orbit. Hence also the submanifold
\[\mathcal{T}:=\{q\}\times (\xi+\mathfrak{t}^*)\subset P\times_G\mathfrak{g}^*\]
is transverse to $O_{\xi}$ at $[q,\xi]$. Let $\gamma$ be a loop in $O_{\xi}$ based at $[q,\xi]$, and let
$\widetilde{\gamma}$ be a lift to $P$. Observe that, for $\eta\in\mathfrak{t}^*$, the path
\[t\mapsto [\widetilde{\gamma}(t),\xi+\eta]\in P\times_G\mathfrak{g}^*,\]
stays in the leaf $O_{\xi+\eta}$, and that the map $[q,\xi+\eta]\mapsto [\widetilde{\gamma}(1),\xi+\eta]$ is the
holonomy of $\gamma$ on $\mathcal{T}$. Writing $\widetilde{\gamma}(1)=qg$, for $g\in G_{\xi}$, it follows that the
holonomy of $\gamma$ is corresponds to the action of $g$ on $\mathfrak{t}^*$. This and the surjectivity of $\theta$
imply that
\[\mathrm{Hol}(O_{\xi})\cong G_{\xi}/Z_{G}(T),\]
where $Z_G(T)$ denotes the set of elements in $G$ which commute with all elements in $T$. In particular, the
holonomy groups are finite.

Since every coadjoint orbit hits $\mathfrak{t}^*$, it follows that the quotient map $P\times \mathfrak{g}^*\to
P\times_{G}\mathfrak{g}^*$ induces a surjective map
\[\mathrm{pr}: P/T\times\mathfrak{t}^*\rmap  P\times_{G}\mathfrak{g}^*.\]
Clearly, $\mathrm{pr}$ is a proper map, thus it suffices to show that $\mathrm{vh}\circ \mathrm{pr}$ extends
continuously. Note that, for $\xi\in \mathfrak{t}^{*\mathrm{reg}}$, the map $\mathrm{pr}$ restricts to a
$|G_{\xi}/T|$-covering projection of the leaf
\[\overline{p}_{\xi}:P/T\times \{\xi\}\rmap P/G_{\xi}\cong O_{\xi}.\]
Thus we have that
\begin{align*}
\mathrm{Vol}(P/T\times\{\xi\},\overline{p}_{\xi}^*(\omega_{\xi}))&=|G_{\xi}/T|\mathrm{Vol}(O_{\xi},\omega_{\xi})=\frac{|G_{\xi}/T|}{|G_{\xi}/Z_G(T)|}\mathrm{vh}(O_{\xi})=\\
 &=|Z_G(T)/T|\mathrm{vh}(O_{\xi}).
\end{align*}
Hence it suffices to show that the map
\begin{equation}\label{EQ_volume_1}
\mathfrak{t}^*\ni\xi\mapsto \mathrm{Vol}(P/T\times\{\xi\},\overline{p}_{\xi}^*(\omega_{\xi}))
\end{equation}
is continuous. By (\ref{EQ_volume}), we have that the pullback of $\overline{p}_{\xi}^*(\omega_{\xi})$ to
$P\times\{\xi\}$ is given by $\Omega_{|P\times\{\xi\}}=p^*(\omega_S)-d\langle \xi,\theta\rangle$, in particular it
depends smoothly on $\xi$. Hence also $\overline{p}_{\xi}^*(\omega_{\xi})$ depends smoothly on $\xi$, and so the
map (\ref{EQ_volume_1}) is continuous. To conclude the proof, we have to check that this map vanishes for $\xi\notin
\mathfrak{t}^{*\mathrm{reg}}$. For such an $\xi$, since $\mathrm{dim}(G_{\xi}/T)>0$, we have that
\[2l=\mathrm{dim}(O_{\xi})=\mathrm{dim}(P/G_{\xi})<\mathrm{dim}(P/T)=2k.\]
This finishes the proof, since
\[\Lambda^k\overline{p}_{\xi}^*(\omega_{\xi})=\overline{p}_{\xi}^*(\Lambda^k\omega_{\xi})=0.\qedhere\]
\end{proof}

\subsection{$C^p$-$C^1$-rigidity and isotopies}\label{subsection_CpC1isotopies}

In the definition of $C^p$-$C^1$-rigidity we may assume that the maps $\psi_{\widetilde{\pi}}$ are isotopic to the
inclusion $\textrm{Id}_{\overline{O}}$ of $\overline{O}$ in $M$ by a smooth path of embeddings. This follows from
the continuity of $\psi$ and the fact that $\textrm{Id}_{\overline{O}}$ has a path connected $C^1$-neighborhood in
$C^{\infty}(\overline{O},M)$ consisting of embeddings. To construct such an open, consider the exponential map
$\exp:TM\to M$ of a complete metric on $M$. Then $\exp$ induces a ``chart'' on the Fr\'echet manifold
$C^{\infty}(\overline{O},M)$:
\[\chi:\Gamma(TM_{|\overline{O}})\rmap C^{\infty}(\overline{O},M), \ \chi(X)(p):=\exp(X_p),\]
which is continuous with respect to all the $C^k$-topologies (see \cite{Ham}). One can take
$\mathcal{U}:=\chi(\mathcal{O})$, where $\mathcal{O}$ is a convex, $C^1$-open neighborhood of zero in
$\Gamma(TM_{|\overline{O}})$. If $\mathcal{O}$ is $C^1$-small enough, $\mathcal{U}$ consists of embeddings.

\subsection{The case of fixed points}\label{subsection_fixed_points}\index{fixed point}
Consider a Poisson manifold $(M,\pi)$ and let $x\in M$ be a fixed point of $\pi$. To apply Theorem
\ref{The_normal_form_theorem} in this setting, the isotropy Lie algebra at $x$ needs to be integrable by a compact Lie
group with vanishing second de Rham cohomology. Such Lie algebras have the following structure:

\begin{lemma}\label{Lemma_cohomology_lie_groups}
A Lie algebra $\mathfrak{g}$ is integrable by a compact Lie group with vanishing second de Rham cohomology if and
only if it is of the form
\[\mathfrak{g}=\mathfrak{k} \textrm{ or }\mathfrak{g}=\mathfrak{k}\oplus\mathbb{R},\]
where $\mathfrak{k}$ is a semisimple Lie algebra of compact type.
\end{lemma}
\begin{proof}
It is well known (e.g.\ \cite{DK}) that a compact Lie algebra $\mathfrak{g}$ decomposes as a direct product
$\mathfrak{g}=\mathfrak{k}\oplus\mathfrak{z}$, where $\mathfrak{k}=[\mathfrak{g},\mathfrak{g}]$ is semisimple
of compact type and $\mathfrak{z}$ is the center of $\mathfrak{g}$. Let $G$ be a compact, connected Lie group
integrating $\mathfrak{g}$. Its cohomology can be computed using invariant forms, hence $H^1(G)\cong
H^1(\mathfrak{g})$. By Hopf's theorem, 
$G$ is homotopy equivalent to a product of odd-dimensional spheres, therefore $H^2(G)=\Lambda^2 H^1(G)$. 
This shows that:
\begin{equation}\label{EQ_5}
H^2(G)= \Lambda^2 H^1(G)\cong \Lambda^2 H^1(\mathfrak{g})=\Lambda^2(\mathfrak{g}/[\mathfrak{g},\mathfrak{g}])^*=\Lambda^2 \mathfrak{z}^*.
\end{equation}
So $H^2(G)=0$ implies that $\mathrm{dim}(\mathfrak{z})\leq 1$.

Conversely, let $K$ be the compact, 1-connected Lie group integrating $\mathfrak{k}$. Take $G=K$ in the first case
and $G=K\times S^1$ in the second. By (\ref{EQ_5}), in both cases we obtain that $H^2(G)=0$.
\end{proof}

So for fixed points Theorem \ref{The_normal_form_theorem} gives:
\begin{corollary}
Let $(M,\pi)$ be a Poisson manifold with a fixed point $x$ for which the isotropy Lie algebra $\mathfrak{g}_x$ is
compact and its center is at most one-dimensional. Then $\pi$ is rigid around $x$, and an open around $x$ is Poisson
diffeomorphic to an open around $0$ in the linear Poisson manifold $(\mathfrak{g}_x^*,\pi_{\mathfrak{g}_x})$.
\end{corollary}

The linearization result in the semisimple case is Conn's theorem \cite{Conn}, and in the case when the isotropy has a
one-dimensional center, it follows from the smooth Levi decomposition theorem of Monnier and Zung \cite{Monnier}.

This fits into Weinstein's notion of a nondegenerate Lie algebra \cite{Wein}. Recall that a Lie algebra $\mathfrak{g}$
is called \textbf{nondegenerate}\index{nondegenerate Lie algebra} if every Poisson structure that has isotropy Lie
algebra $\mathfrak{g}$ at a fixed point $x$ is Poisson-diffeomorphic around $x$ to the linear Poisson structure
$(\mathfrak{g}^*,\pi_{\mathfrak{g}})$ around $0$.

A Lie algebra $\mathfrak{g}$ for which $\pi_{\mathfrak{g}}$ is rigid around $0$ is necessarily nondegenerate. To see
this, consider a Poisson bivector $\pi$ whose linearization at $0$ is $\pi_{\mathfrak{g}}$. We have a smooth path of
Poisson bivectors $\pi^t$, with $\pi^1=\pi$ and $\pi^0=\pi_\mathfrak{g}$, given by $\pi^t=t\mu_t^*(\pi)$, where
$\mu_t$ denotes multiplication by $t>0$. If $\pi_\mathfrak{g}$ is rigid around $0$, then, for some $r>0$ and some
$t>0$, there is a Poisson isomorphism between
\[\psi: (B_{r},\pi^t)\diffto (\psi(B_{r}),\pi_{\mathfrak{g}}).\]
Now $\xi:=\psi(0)$ is a fixed point of $\pi_{\mathfrak{g}}$, which is the same as an element in
$(\mathfrak{g}/[\mathfrak{g},\mathfrak{g}])^{*}$. It is easy to see that translation by $\xi$ is a Poisson isomorphism
of $\pi_{\mathfrak{g}}$, therefore by replacing $\psi$ with $\psi-\xi$ we may assume that $\psi(0)=0$. By linearity of
$\pi_{\mathfrak{g}}$, we have that $\mu_t^*(\pi_{\mathfrak{g}})=\frac{1}{t}\pi_{\mathfrak{g}}$, and this shows
that
\[\pi=\frac{1}{t}\mu_{1/t}^*(\pi^t)=\frac{1}{t}\mu_{1/t}^*(\psi^*(\pi_{\mathfrak{g}}))=\mu_{1/t}^*\circ\psi^*\circ\mu_t^*(\pi_{\mathfrak{g}}),\]
which implies that $\pi$ is linearizable by the map
\[\mu_t\circ\psi\circ\mu_{1/t}:(B_{tr},\pi)\rmap (t\psi(B_r),\pi_{\mathfrak{g}}),\]
which maps $0$ to $0$. Thus $\mathfrak{g}$ is nondegenerate.

\subsection{The Lie-Poisson sphere}\label{subsection_Poisson_sphere}\index{Lie-Poisson sphere}

Let $\mathfrak{g}$ be a semisimple Lie algebra of compact type and let $G$ be the compact, 1-connected Lie group
integrating it. Recall from subsection \ref{Subsection_examples_of_symplectic_groupoids} that the linear Poisson
structure $(\mathfrak{g}^*,\pi_{\mathfrak{g}})$ is integrable by the symplectic groupoid $(G\ltimes
\mathfrak{g}^*,\omega_{\mathfrak{g}})\rightrightarrows \mathfrak{g}^*$, whose $s$-fibers are diffeomorphic to
$G$. Since $H^2(G)=0$, we can apply Theorem \ref{Theorem_FOUR} to any compact Poisson submanifold of
$\mathfrak{g}^*$. Taking the sphere with respect to some invariant metric, we obtain a strengthening of Corollary
\ref{C_Dual_ss_Lie}.
\begin{proposition}\label{proposition_rigidity_of_spheres}
Let $\mathfrak{g}$ be a semisimple Lie algebra of compact type and denote by
$\mathbb{S}(\mathfrak{g}^*)\subset\mathfrak{g}^*$ the unit sphere centered at the origin with respect to some
invariant inner product. Then $\pi_{\mathfrak{g}}$ is $C^p$-$C^1$-rigid around $\mathbb{S}(\mathfrak{g}^*)$
and, up to isomorphism, it is determined around $\mathbb{S}(\mathfrak{g}^*)$ by its first order jet.
\end{proposition}

Using this rigidity result, we will give in chapter \ref{ChDef} a description of an open around
$\pi_{\mathfrak{g}|\mathbb{S}(\mathfrak{g}^*)}$ in the moduli space of all Poisson structures on
$\mathbb{S}(\mathfrak{g}^*)$.

\subsection{Relation with stability of symplectic leaves} \label{subsection_relation_stability}

Recall from \cite{CrFe-stab} that a symplectic leaf $(S,\omega_S)$ of a Poisson manifold $(M,\pi)$ is said to be
$C^p$-\textbf{strongly stable} if for every open $U$ around $S$ there exists an open neighborhood
$\mathcal{V}\subset \mathfrak{X}^2(U)$ of $\pi_{|U}$ with respect to the compact-open $C^p$-topology, such that
every Poisson structure in $\mathcal{V}$ has a leaf symplectomorphic to $(S,\omega_S)$. Recall also
\begin{theorem}[Theorem 2.2 \cite{CrFe-stab}]
If $S$ is compact and the Lie algebroid $A_S:=T^*M_{|S}$ satisfies $H^2(A_S)=0$, then $S$ is a strongly stable leaf.
\end{theorem}

If $\pi$ is $C^p$-$C^1$-rigid around $S$, then $S$ is a strongly stable leaf. Also, the hypothesis of our Theorem
\ref{The_normal_form_theorem} imply those of Theorem 2.2 \emph{loc.cit.} To see this, let $P\to S$ be a principal
$G$-bundle for which $A_S\cong TP/G$. Then
\[H^{\bullet}(A_S)\cong H^{\bullet}(\Omega(P)^G).\]
If $G$ is compact and connected then
\[H^{\bullet}(\Omega(P)^G)\cong H^{\bullet}(P)^G\subset H^{\bullet}(P),\]
hence $H^2(P)=0$ implies $H^2(A_S)=0$.

On the other hand, $H^2(A_S)=0$ doesn't imply rigidity, counterexamples can be found even for fixed points.
Weinstein proves \cite{Wein4} that a noncompact semisimple Lie algebra $\mathfrak{g}$ of real rank at least two is
degenerate, so $\pi_{\mathfrak{g}}$ is not rigid (see subsection \ref{subsection_fixed_points}). However, $0$ is a
stable point for $\pi_{\mathfrak{g}}$, because by Whitehead's Lemma $H^2(\mathfrak{g})=0$.

According to Theorem 2.3 \cite{CrFe-stab}, the condition $H^2(A_S)=0$ is also necessary for strong stability of the
symplectic leaf $(S,\omega_S)$ in the case of Poisson structures of ``first order'', i.e.\ Poisson structures that are
isomorphic to their local model around $S$. So for this type of Poisson structures, $H^2(A_S)=0$ is also necessary for
rigidity.

For Poisson structures with trivial underlying foliation, we will prove below that the hypotheses of Theorem
\ref{The_normal_form_theorem} and of Theorem 2.2 \cite{CrFe-stab} are equivalent.

\subsection{Trivial symplectic foliations}\label{Trivial_symplectic_foliations}
In this subsection we discuss rigidity and linearization of regular Poisson structures $\pi$ on $S\times\mathbb{R}^n$
with symplectic foliation
\[\{(S\times\{y\},\omega_{y}:=\pi_{|S\times\{y\}}^{-1})\}_{y\in \mathbb{R}^n},\]
where $\omega_y$ is a smooth family of forms on $S$. Denote by $(S,\omega_S)$ the symplectic leaf for $y=0$. As
explained in subsection \ref{Subsection_Examples_linearization}, the first order approximation around $S$
corresponds to the family of 2-forms
\[j^1_{S}(\omega)_{y}:=\omega_S+\delta_S\omega_y,\]
where $\delta_S\omega_y$ is the vertical derivative of $\omega$\index{vertical derivative}
\[\delta_S\omega_y:=\frac{d}{d\epsilon}(\omega_{\epsilon y})_{|\epsilon=0}=y_1\omega_1+\ldots+y_n\omega_n.\]
The first order approximation is defined on an open $U\subset S\times\mathbb{R}^n$ containing $S$, such that
$j^1_{S}(\omega)_y$ is nondegenerate along $U\cap (S\times\{y\})$, for all $y\in\mathbb{R}^n$.

By Lemma \ref{Lemma_algebroid_cutarica}, the splitting $T^*M_{|S}=T^*S\times\mathbb{R}^n$ and the
isomorphism of $\omega_S^{\sharp}:TS\diffto T^*S$ give an identification of $A_S$ with $TS\times\mathbb{R}^n$;
and the corresponding Lie bracket is:
\begin{align}\label{EQ_Bracket_AS}
[(&X,f_1,\ldots,f_n),(Y,g_1,\ldots,g_n)]=\\
\nonumber &=([X,Y],X(g_1)-Y(f_1)+\omega_1(X,Y),\ldots,X(g_n)-Y(f_n)+\omega_n(X,Y)).
\end{align}
The conditions in Theorem \ref{The_normal_form_theorem} are easier to verify in this case.\index{cohomological
variation}
\begin{lemma}\label{L_trivial_foliation}
If $S$ is compact, then the following are equivalent:
\begin{enumerate}[(a)]
\item $A_S$ is integrable by a compact principal bundle $P$, with $H^2(P)=0$,
\item $H^2(A_S)=0$,
\item The cohomological variation $[\delta_S\omega]:\mathbb{R}^n\to H^2(S)$ satisfies:
\begin{enumerate}
\item[(c$_1$)] it is surjective,
\item[(c$_2$)] its kernel is at most 1-dimensional,
\item[(c$_3$)] the map $H^1(S)\otimes \mathbb{R}^n\to H^3(S)$, $\eta\otimes y\mapsto \eta\wedge [\delta_S\omega_y]$ is injective.
\end{enumerate}
\end{enumerate}
\end{lemma}
\begin{proof}
The fact that (a) implies (b) was explained in subsection \ref{subsection_relation_stability}. Now, we show that (b) and
(c) are equivalent. The complex computing $H^{\bullet}(A_S)$ is given by:
\[\Omega^{k}(A_S):=\bigoplus_{p+q=k} \Omega^p(S)\otimes \Lambda^{q}\mathbb{R}^n,\]
and the differential acts on elements of degree one and two as follows:
\begin{align*}
&d_{A_S}(\lambda,\sum_{i}\mu_i e_i)=(d\lambda-\sum_{i}\mu_i\omega_i,\sum_i d\mu_i\otimes e_i,0),\\
&d_{A_S}(\alpha,\sum_i \beta_i\otimes e_i,\sum_{i,j}\gamma_{i,j}e_i\wedge e_j)=\\
&=(d\alpha+\sum_i\beta_i\wedge\omega_i,\sum_i(d\beta_i-\sum_j \gamma_{i,j}\omega_j)\otimes e_i, \sum_{i,j}d\gamma_{i,j}\otimes e_i\wedge e_j,0),
\end{align*}
where $\{e_i\}$ is the canonical basis of $\mathbb{R}^n$. We use the exact sequences:
\[0\rmap K\rmap H^2(A_S)\rmap \Lambda^2\mathbb{R}^n\rmap H^2(S)\otimes \mathbb{R}^n,\]
\[\mathbb{R}^n\rmap H^2(S)\rmap K\rmap H^1(S)\otimes \mathbb{R}^n\rmap H^3(S),\]
where the map $H^2(A_S)\to \Lambda^2\mathbb{R}^n$ is $[\alpha,\beta\otimes u, v\wedge w]\mapsto v\wedge w$;
the map $\Lambda^2 \mathbb{R}^n\to H^2(S)\otimes \mathbb{R}^n$, denoted by $[\delta_S\omega]\otimes
\mathrm{Id}$, sends $v\wedge w\mapsto [\delta_S\omega_v]\otimes w-[\delta_S\omega_w]\otimes v$; the map
$\mathbb{R}^n\to H^2(S)$ is $[\delta_S\omega]$; the map $H^2(S)\to K$ is $[\alpha]\mapsto [\alpha,0,0]$; the map
$K\to H^1(S)\otimes\mathbb{R}^n$ is $[\alpha,\beta\otimes v,0]\mapsto [\beta]\otimes v$; the last map is the one from
(c$_3$). When proving exactness, the only nontrivial part of the computation is at $\Lambda^2\mathbb{R}^n$. This is
based on a simple fact from linear algebra:
\begin{equation}\label{EQ_4}
\textrm{ker}([\delta_S\omega]\otimes
\mathrm{Id})=\Lambda^2( \textrm{ker}[\delta_S\omega]).
\end{equation}
So, an element in $\textrm{ker}([\delta_S\omega]\otimes \mathrm{Id})$ can be written as a sum of the form $\sum
v\wedge w$, with $v,w\in \textrm{ker}[\delta_S\omega]$. Writing $\delta_S\omega_{v}=d\eta$,
$\delta_S\omega_{w}=d\theta$, for $\eta,\theta\in\Omega^1(S)$, one easily checks that $$(\eta\wedge \theta,
\eta\otimes w-\theta\otimes v, v\wedge w)\in\Omega^2(A_S)$$ is closed. This implies exactness at
$\Lambda^2\mathbb{R}^n$. So $H^2(A_S)$ vanishes if and only if (c$_1$) and (c$_3$) hold and the map
$[\delta_S\omega]\otimes \mathrm{Id}$ is injective. By (\ref{EQ_4}), injectivity is equivalent to (c$_2$).

We prove that (b) and (c) imply (a). Part (c$_1$) implies that, by taking a different basis of $\mathbb{R}^n$, we may
assume that $[\omega_1],\ldots,[\omega_n]\in H^2(S,\mathbb{Z})$. Let $P\to S$ be a principal $T^n$ bundle with
connection $(\theta_1,\ldots,\theta_n)$ and curvature $(-\omega_1,\ldots,-\omega_n)$. We claim that the Lie algebroid
$TP/T^n$ is isomorphic to $A_S$ and therefore $A_S$ is integrable by the compact gauge groupoid
\[P\times_{T^n}P\rightrightarrows S.\]
A section of $TP/T^n$ (i.e.\ a $T^n$-invariant vector field on $P$) can be decomposed uniquely as
\[\widetilde{X}+\sum f_i\partial_{\theta_i},\]
where $\widetilde{X}$ is the horizontal lift of a vector field $X$ on $S$, $f_1,\ldots,f_n$ are functions on $S$ and
$\partial_{\theta_i}$ is the unique vertical vector field on $P$ which satisfies
\[\theta_j(\partial_{\theta_i})=\delta_{i,j}.\]
Using (\ref{EQ_Bracket_AS}) for the bracket on $A_S$ and that $d\theta_i=-p^*(\omega_i)$, it is straightforward to
check that the following map is a Lie algebroid isomorphism:
\[TP/T^n\diffto A_S, \ \  \widetilde{X}+\sum f_i\partial_{\theta_i}\mapsto (X,f_1,\ldots,f_n).\]
Using that $T^n$ acts trivially on $H^2(P)$ and (b), we obtain the conclusion:
\[H^2(P)=H^2(P)^{T^n}\cong H^2(TP/T^n)\cong H^2(A_S)=0.\qedhere \]
\end{proof}

So for trivial symplectic foliations the conditions in Theorem \ref{The_normal_form_theorem} and in Theorem 2.2
\cite{CrFe-stab} coincide.

\begin{corollary}\label{corollary_trivial_foliations}
Let $\{\omega_y\in\Omega^2(S)\}_{y\in\mathbb{R}^n}$ be a smooth family of symplectic structures on a compact
manifold $S$. If the cohomological variation at $0$
\[[\delta_S\omega]:\mathbb{R}^n\rmap H^2(S)\]
satisfies the conditions from Lemma \ref{L_trivial_foliation}, then the Poisson manifold
\[(S\times\mathbb{R}^n,\{\omega_y^{-1}\}_{y\in\mathbb{R}^n})\]
is isomorphic to its local model and $C^p$-$C^1$-rigid around $S\times\{0\}$.
\end{corollary}

By Theorem \ref{Theorem_ONE} condition (c$_1$) from Lemma \ref{L_trivial_foliation} (i.e.\ surjectivity of
$[\delta_S\omega]$) is sufficient for the symplectic foliation to be isomorphic to its first order approximation, but by
Theorem 2.3 \cite{CrFe-stab}, all conditions are necessary for rigidity.\\

In the case of trivial foliations we also have an improvement compared to Theorem \ref{Theorem_TWO}. Recall from
Corollary \ref{corollary_conditions_regular} that, in this case, the condition of this result are equivalent to:
\begin{itemize}
\item $S$ is compact with finite fundamental group,
\item $[\delta_S\omega]:\mathbb{R}^n\to H^2(\widetilde{S})$ is an isomorphism,
\end{itemize}
where $\widetilde{S}$ denotes the universal cover of $S$. The first condition implies that $H^1(S)=0$ (so (c$_1$)
holds), and, as in the proof of Lemma \ref{lemma_surj_implies_surj}, the second implies that the pullback induces an
isomorphism $H^2(S)\to H^2(\widetilde{S})$, hence also (c$_2$) and (c$_3$) hold.\\

The conditions of Lemma \ref{L_trivial_foliation} suggest the following construction of rigid Poisson manifolds.
\begin{example}\rm
Let $(S,\omega_S)$ be a compact symplectic manifold. Choose a basis of $H^2(S)$ represented by closed 2-forms
$\omega_1, \ldots,\omega_n$. We define a symplectic foliation on an open around $S\times\{[\omega_S]\}$ in $S\times
H^2(S)$:
\[\{(S\times\{C\},\omega_{C})\}_{C\in U}, \ \textrm{with}\  \omega_C=\omega_S+y_1 \omega_1+\ldots+y_n \omega_n,\]
where the coefficients $y_i$ are such that $C=[\omega_C]$. By applying the leafwise Moser argument from Lemma
\ref{Moser_lemma_for_symplectic foliations}, one proves that different bases of $H^2(S)$ induce Poisson structures
that are isomorphic around $S\times\{[\omega_S]\}$. By Lemma \ref{L_trivial_foliation}, Corollary
\ref{corollary_trivial_foliations}, and Theorem 2.3 \cite{CrFe-stab}, we have that this Poisson structure is rigid around
$S\times\{[\omega_S]\}$ if and only if the cup product is an injective map between the spaces:
\[\wedge: H^1(S)\otimes H^2(S)\rmap H^3(S).\]
This is certainly satisfied if $H^1(S)=0$.
\end{example}

\section{Proof of Theorem \ref{Theorem_FOUR}}\label{section_technical_rigidity}

We start by preparing the setting needed for the Nash-Moser method: we fix norms on the spaces that we will work
with, we construct smoothing operators adapted to our problem and recall the interpolation inequalities. Next, we
prove a series of inequalities which assert tameness of some natural operations like Lie derivative, flow and pullback,
and we also prove some properties of local diffeomorphisms. The section ends with the proof of Theorem
\ref{Theorem_FOUR}.

\begin{remark}\rm
A common convention when dealing with the Nash-Moser techniques (e.g.\ \cite{Ham}), which we also adopt, is to
denote all constants by the same symbol: In the preliminary results below, we work with ``big enough'' constants $C$
and $C_n$, and with ``small enough'' constants $\theta>0$; these depend on the trivialization data for the vector
bundle $E$ and on the smoothing operators. In the proof of Proposition \ref{Proposition_technical}, $C_n$ depends
also on the Poisson structure $\pi$.
\end{remark}

\subsection{The ingredients of the tame category}\label{subsection_norms}

We will use some terminology from \cite{Ham}. A Fr\'echet space $F$ endowed with an increasing family of
semi-norms $\{\|\cdot\|_{n}\}_{n\geq 0}$ generating its topology will be called a \textbf{graded Fr\'echet
space}\index{graded Fr\'echet space}. A linear map $T:F_1\to F_2$ between two graded Fr\'echet spaces is called
\textbf{tame}\index{tame} of degree $d$ and base $b$, if it satisfies inequalities of the form
\[\|Tf\|_{n}\leq C_n\|f\|_{n+d},\ \forall \ n\geq b, f\in F_1.\]

Let $E\to S$ be a vector bundle over a compact manifold $S$ and fix a metric on $E$. For $r>0$, consider the closed
tube in $E$ of radius $r$
\[E_r:=\{v\in E: |v|\leq r\}.\]
The space $\mathfrak{X}^{\bullet}(E_r)$, of multivector fields on $E_r$, endowed with the
$C^n$-norms\index{C$^n$-norms} $\|\cdot\|_{n,r}$ is a graded Fr\'echet space. We recall here the construction of
these norms. Fix a finite open cover of $S$ by domains of charts $\{\chi_i:O_i\diffto \mathbb{R}^d\}_{i=1}^{N}$ and
vector bundle isomorphisms
\[\widetilde{\chi}_i:E_{|O_i}\diffto \mathbb{R}^d\times \mathbb{R}^D\]
covering $\chi_i$. We will assume that $\widetilde{\chi}_i(E_{r|O_i})=\mathbb{R}^d\times \overline{B}_r$ and that
the family
\[\{O_i^{\delta}:=\chi_i^{-1}(B_{\delta})\}_{i=1}^N\]
covers $S$ for all $\delta\geq 1$. Moreover, we assume that the cover satisfies
\begin{equation}\label{EQ_6}
\textrm{if}\ \ O_i^{3/2}\cap O_j^{3/2}\neq \emptyset\  \ \textrm{then}\ \ O_j^1\subset O_i^4.
\end{equation}
This holds if $\chi_i^{-1}:B_4\to O_i$ is the exponential corresponding to some metric on $S$, with injectivity radius
bigger than $4$.

For $W\in\mathfrak{X}^{\bullet}(E_r)$, denote its local expression in the chart $\widetilde{\chi}_i$ by
\[W_i(z):=\sum_{1\leq i_1<\ldots<i_p\leq d+D}W_{i}^{i_1,\ldots,i_p}(z)\frac{\partial}{\partial z_{i_1}}\wedge \ldots\wedge \frac{\partial}{\partial z_{i_p}},\]
and let the $C^n$-norm of $W$ be given by
\[\|W\|_{n,r}:=\sup_{i, i_1,\ldots, i_p}\left\{|\frac{\partial^{|\alpha|}}{\partial{z}^{\alpha}} W_{i}^{i_1,\ldots,i_p}(z)|:z\in B_{1}\times B_r, 0\leq |\alpha|\leq n \right\}.\]
For $s<r$ the restriction maps are norm decreasing
\[\mathfrak{X}^{\bullet}(E_{r})\ni W\mapsto W_{|{s}}:=W_{|E_{s}}\in \mathfrak{X}^{\bullet}(E_{s}), \ \ \  \|W_{|{s}}\|_{n,s}\leq \|W\|_{n,r}.\]

We will work also with the closed subspaces of multivector fields on $E_{r}$ whose first jet vanishes along $S$, which
we denote by
\[\mathfrak{X}^{k}(E_r)^{(1)}:=\{W\in \mathfrak{X}^{k}(E_r) : j^1_{|S}W=0\}.\]

The main technical tool used in the Nash-Moser method are the smoothing operators. We will call a family $\{S_t:F\to
F\}_{t>1}$ of linear operators on the graded Fr\'echet space $F$ \textbf{smoothing operators}\index{smoothing
operators} of degree $d\geq 0$, if there exist constants $C_{n,m}>0$, such that for all $n,m\geq 0$ and $f\in F$, the
following inequalities hold:
\begin{equation}\label{EQ_7}
\|S_t(f)\|_{n+m}\leq t^{m+d}C_{n,m}\|f\|_{n},\ \ \|S_t(f)-f\|_{n}\leq t^{-m}C_{n,m}\|f\|_{n+m+d}.
\end{equation}

The construction of such operators is standard, but since we are dealing with a family of Fr\'echet spaces
$\{\mathfrak{X}^{k}(E_r)\}_{0<r\leq 1}$, we give the dependence of $C_{n,m}$ on $r$.
\begin{lemma}\label{Lemma_smoothing_operators}
The family of graded Fr\'echet spaces $\{(\mathfrak{X}^{k}(E_r),\|\cdot\|_{n,r})\}_{r\in (0,1]}$ has a family of
smoothing operators of degree $d=0$
\[\{S_t^r:\mathfrak{X}^{k}(E_r)\rmap \mathfrak{X}^{k}(E_r) \}_{t>1,0<r\leq 1},\]
which satisfy (\ref{EQ_7}) with constants of the form $C_{n,m}(r)=C_{n,m}r^{-(n+m+k)}$.

Similarly, the family $\{(\mathfrak{X}^{k}(E_r)^{(1)},\|\cdot\|_{n,r})\}_{r\in (0,1]}$ has smoothing operators
\[\{S_t^{r,1}:\mathfrak{X}^{k}(E_r)^{(1)}\rmap \mathfrak{X}^{k}(E_r)^{(1)} \}_{t>1,0<r\leq 1},\]
of degree $d=1$ and constants $C_{n,m}(r)=C_{n,m}r^{-(n+m+k+1)}$.
\end{lemma}
\begin{proof}
The existence of smoothing operators of degree zero on the Fr\'echet space of sections of a vector bundle over a
compact manifold (possibly with boundary) is standard (see \cite{Ham}). We fix such a family
\[\{S_t:\mathfrak{X}^{k}(E_{1})\to \mathfrak{X}^{k}(E_{1})\}_{t>1}.\] Consider also the rescaling operators, defined for $\rho\neq 0$ by
\[ \mu_{\rho}: E_{R}\rmap E_{\rho R}, \ \ \mu_{\rho}(v):=\rho v.\]
For $r<1$, we define $S_t^r$ by conjugation with these operators:
\[S_t^r:\mathfrak{X}^{k}(E_r)\rmap \mathfrak{X}^{k}(E_r),\ \ S_t^r:=\mu_{r^{-1}}^*\circ S_t\circ \mu_{r}^*.\]
It is straightforward to check that $\mu_{\rho}^{*}$ satisfies the following inequality
\[\|\mu_{\rho}^*(W)\|_{n,R}\leq \textrm{max}\{\rho^{-k},\rho^{n}\}\|W\|_{n,\rho R}, \ \forall\ W\in\mathfrak{X}^{k}(E_{\rho R}).\]
Using this we obtain that $S_t^r$ satisfies (\ref{EQ_7}) with $C_{n,m}(r)=C_{n,m}r^{-(n+m+k)}$.

To construct the operators $S_t^{r,1}$, we first define a tame projection \[P: \mathfrak{X}^{k}(E_r)\to
\mathfrak{X}^{k}(E_r)^{(1)}.\] Choose $\{\lambda_i\}_{i=1}^N$ a smooth partition of unity on $S$ subordinated to
the cover $\{O_i^{1}\}_{i=1}^N$, and let $\{\widetilde{\lambda}_i\}_{i=1}^N$ be the pullback to $E$. For
$W\in\mathfrak{X}^{k}(E_r)$, denote its local representatives by \[W_i:=\widetilde{\chi}_{i,*}(W_{|E_{r|O_i}})\in
\mathfrak{X}^{k}(\mathbb{R}^d\times \overline{B}_r).\] Define $P$ as follows:
\[P(W):=\sum\widetilde{\lambda}_i\cdot \widetilde{\chi}_{i,*}^{-1}(W_i-T_y^1(W_i)),\]
where $T^1_y(Z)$ is the first degree Taylor polynomial of $Z$ in the fiber direction
\[T^1_y(Z)(x,y):=Z(x,0)+\sum y_j\frac{\partial Z}{\partial y_j}(x,0).\]
If $W\in \mathfrak{X}^{k}(E_r)^{(1)}$, then $T^1_y(W_i)=0$; so $P$ is a projection. It is easy to check that $P$ is
tame of
degree $1$, that is, there are constants $C_{n}>0$ such that 
\[\|P(W)\|_{n,r}\leq C_n\|W\|_{n+1,r}.\]
Define the smoothing operators on $\mathfrak{X}^{k}(E_r)^{(1)}$ as follows:
\[S_t^{r,1}:\mathfrak{X}^{k}(E_r)^{(1)}\rmap \mathfrak{X}^{k}(E_r)^{(1)},\ \ S_t^{r,1}:=P\circ S_t^{r}.\]
Using tameness of $P$, the inequalities for $S_t^{r,1}$ are straightforward.
\end{proof}

The norms $\|\cdot\|_{n,r}$ satisfy the classical interpolation inequalities with constants that are polynomials in
$r^{-1}$.\index{interpolation inequalities}
\begin{lemma}\label{L_interpolation}
The norms $\|\cdot\|_{n,r}$ satisfy:
\[\|W\|_{l,r}\leq C_{n}r^{k-l} (\|W\|_{k,r})^{\frac{n-l}{n-k}}(\|W\|_{n,r})^{\frac{l-k}{n-k}},\  \forall \ r\in(0,1]\]
for all $0\leq k\leq l\leq n$, not all equal and all $W\in\mathfrak{X}^{\bullet}(E_r)$.
\end{lemma}
\begin{proof}
By \cite{Conn}, these inequalities hold locally i.e.\ for the $C^n$-norms on the space
$C^{\infty}(\overline{B}_1\times \overline{B}_r)$. Applying these to the components of the restrictions to the charts
$(E_{r|O^{1}_i},\widetilde{\chi}_i)$ of an element in $\mathfrak{X}^{\bullet}(E_r)$, we obtain the conclusion.
\end{proof}

\subsection{Tameness of some natural operators}
In this subsection we prove some tameness properties of the Schouten bracket, the pullback, and the flows of vector
fields.
\subsubsection*{The tame Fr\'echet Lie algebra of multivector fields} We prove that \[(\mathfrak{X}^{\bullet}(E_r),[\cdot,\cdot],\{\|\cdot\|_{n,r}\}_{n\geq 0})\] forms is a graded tame Fr\'echet Lie algebra.
\begin{lemma}\label{L_Bracket}
The Schouten bracket on $\mathfrak{X}^{\bullet}(E_r)$ satisfies
\[\|[W,V]\|_{n,r}\leq C_nr^{-(n+1)}(\|W\|_{0,r}\|V\|_{n+1,r}+\|W\|_{n+1,r}\|V\|_{0,r}),\  \forall \ r\in(0,1].\]
\end{lemma}
\begin{proof}
By a local computation, the bracket satisfies inequalities of the form:
\[\|[W,V]\|_{n,r}\leq C_n\sum_{i+j=n+1}\|W\|_{i,r}\|V\|_{j,r}.\]
Using the interpolation inequalities, a term in this sum can be bounded by:
\[\|W\|_{i,r}\|V\|_{j,r}\leq C_nr^{-(n+1)}(\|W\|_{0,r}\|V\|_{n+1,r})^{\frac{j}{n+1}}(\|V\|_{0,r}\|W\|_{n+1,r})^{\frac{i}{n+1}}.\]
The following simple inequality (to be used also later) implies the conclusion
\begin{equation}\label{EQ_simple}
x^{\lambda}y^{1-\lambda}\leq x+y, \ \  \forall\  x,y\geq 0, \lambda\in[0,1].
\end{equation}
\end{proof}

\subsubsection*{The space of local diffeomorphisms}
Next, we consider the space of smooth maps $E_r\to E$ that are $C^1$-close to the inclusion $I_r:E_r\hookrightarrow
E$. We call a map $\varphi:E_r\to E$ a \textbf{local diffeomorphism}\emph{local diffeomorphism} if it can be extended
on some open to a diffeomorphism onto its image. Since $E_r$ is compact, this is equivalent to injectivity of
$d\varphi:TE_r\to TE$. To be able to measure $C^n$-norms of such maps, we work with the following open
neighborhood of $I_r$ in $C^{\infty}(E_r;E)$:
\[\mathcal{U}_r:=\left\{\varphi:E_r\rmap E: \varphi(E_{r|\overline{O}_i^{1}})\subset E_{|O_i}, 1\leq i\leq N\right\}.\]
Denote the local representatives of a map $\varphi\in\mathcal{U}_r$ by
\[\varphi_i:\overline{B}_{1}\times \overline{B}_{r}\rmap \mathbb{R}^{d}\times\mathbb{R}^{D}.\]
Define $C^n$-distances between maps $\varphi,\psi\in \mathcal{U}_r$ as follows
\[\mathrm{d}(\varphi,\psi)_{n,r}:=\sup_{1\leq i\leq N} \left\{|\frac{\partial^{|\alpha|}}{\partial z^{\alpha}}(\varphi_i-\psi_i)(z)|: z\in B_{1}\times B_{r},0\leq |\alpha|\leq n\right\}.\]
To control compositions of maps, we will also need the following $C^n$-distances
\[\mathrm{d}(\varphi,\psi)_{n,r,\delta}:=\sup_{1\leq i\leq N} \left\{|\frac{\partial^{|\alpha|}}{\partial z^{\alpha}}(\varphi_i-\psi_i)(z)|: z\in B_{\delta}\times B_{r},0\leq |\alpha|\leq n\right\},\]
which are well-defined only on the opens
\[\mathcal{U}^{\delta}_r:=\left\{\chi\in\mathcal{U}_r: \chi(E_{r|\overline{O}_i^{\delta}})\subset E_{|O_i}\right\}.\]
Similarly, we define also on $\mathfrak{X}^{\bullet}(E_r)$ norms $\|\cdot\|_{n,r,\delta}$ (these measure the
$C^n$-norms in all our local charts on $B_{\delta}\times B_{r}$).

These norms and distances are equivalent:
\begin{lemma}\label{Lemma_equivalent_norms}
There exist $C_n>0$, such that $\forall$  $r\in(0,1]$ and $\delta\in [1,4]$
\[\mathrm{d}(\varphi,\psi)_{n,r}\leq \mathrm{d}(\varphi,\psi)_{n,r,\delta}\leq C_n\mathrm{d}(\varphi,\psi)_{n,r}, \ \forall \ \varphi,\psi\in\mathcal{U}_{r}^{\delta},\]
\[\|W\|_{n,r}\leq \|W\|_{n,r,\delta}\leq C_n\|W\|_{n,r},\ \forall \ W\in\mathfrak{X}^{\bullet}(E_r).\]
\end{lemma}

We also use the simplified notations: 
\[\mathrm{d}(\psi)_{n,r}:=\mathrm{d}(\psi,I_r)_{n,r}, \ \ \mathrm{d}(\psi)_{n,r,\delta}:=\mathrm{d}(\psi,I_r)_{n,r,\delta}.\]

The lemma below is used to check that compositions are defined.
\begin{lemma}\label{Lemma_control_embedding}
There exists a constant $\theta>0$, such that for all $r\in (0,1]$, $\epsilon\in (0,1]$, $\delta\in [1,4]$ and all
$\varphi\in\mathcal{U}_{r}$, satisfying $\mathrm{d}(\varphi)_{0,r}<\epsilon\theta$,
\[\varphi(E_{r|\overline{O}_i^\delta})\subset E_{r+\epsilon|O_i^{\delta+\epsilon}}.\]
\end{lemma}

We prove now that $I_r$ has a $C^1$-neighborhood of local diffeomorphisms.
\begin{lemma}\label{Lemma_embedding}
There exists $\theta>0$, such that, for all $r\in(0,1]$, if $\psi\in \mathcal{U}_r$ satisfies
$\mathrm{d}(\psi)_{1,r}<\theta$, then $\psi$ is a local diffeomorphism.
\end{lemma}
\begin{proof}
By Lemma \ref{Lemma_control_embedding}, if we shrink $\theta$, we may assume that
\begin{equation}\label{EQ_D3}
\psi(E_{r|\overline{O}_i^1})\subset E_{|O_i^{3/2}},\ \ \psi(E_{r|\overline{O}_i^4})\subset E_{|O_i}.
\end{equation}
In a local chart, we write $\psi$ as follows
\[\psi_i:=\textrm{Id}+g_{i}: \overline{B}_4\times \overline{B}_{r}\rmap \mathbb{R}^d\times\mathbb{R}^D.\]
By Lemma \ref{Lemma_equivalent_norms}, if we shrink $\theta$, we may also assume that
\begin{equation}\label{EQ_D1}
|\frac{\partial g_i}{\partial z_j}(z)|<\frac{1}{2(d+D)},  \forall \ z\in \overline{B}_4\times \overline{B}_{r}.
\end{equation}
This ensures that $\textrm{Id}+(dg_i)_z$ is close enough to $\textrm{Id}$ so that it is invertible for all $z\in
\overline{B}_1\times \overline{B}_{r}$; thus, $(d\psi)_p$ is invertible for all $p\in E_r$.

We check now injectivity of $\psi$. Let $p^i\in E_{r|O^1_i}$ and $p^j\in E_{r|O^1_j}$ be such that
$\psi(p^i)=q=\psi(p^j)$. Then, by (\ref{EQ_D3}), $q\in E_{|O^{3/2}_i}\cap E_{|O^{3/2}_j}$, so, by the property
(\ref{EQ_6}) of the opens, we know that $O^{1}_j\subset O^{4}_i$, hence $p^i,p^j\in E_{r|O^4_i}$. Denoting by
$w^i:=\widetilde{\chi}_i(p^i)$ and $w^j:=\widetilde{\chi}_i(p^j)$ we have that $w^i,w^j\in \overline{B}_4\times
\overline{B}_{r}$. Since $w^i+g_i(w^i)=w^j+g_i(w^j)$, using (\ref{EQ_D1}), we obtain
\begin{align*}
|w^i-&w^j|=|g_i(w^i)-g_i(w^j)|=\\
&=|\int_0^1\sum_{k=1}^{D+d}\frac{\partial g_i}{\partial z_k}(tw^i+(1-t)w^j)(w^i_k-w^j_k)dt|\leq \frac{1}{2}|w^i-w^j|.
\end{align*}
Thus $w^i=w^j$, and so $p^i=p^j$. This finishes the proof.
\end{proof}

The composition satisfies the following tame inequalities.
\begin{lemma}\label{Lemma_composition}
There are constants $C_n> 0$ such that for all $1\leq \delta\leq \sigma\leq 4$ and all $0<s\leq r\leq 1$, we have that if
$\varphi\in\mathcal{U}_{s}$ and $\psi\in\mathcal{U}_{r}$ satisfy
\[\varphi(E_{s|\overline{O}_i^{\delta}})\subset E_{r|O_i^{\sigma}}, \ \psi(E_{r|\overline{O}_i^{\sigma}})\subset E_{|O_i}, \ \forall\  0\leq i\leq N,\]
and $\mathrm{d}(\varphi)_{1,s}<1$, then the following inequalities hold:
\begin{align*}
\mathrm{d}(\psi\circ \varphi)_{n,s,\delta}\leq &\mathrm{d}(\psi)_{n,r,\sigma}+\mathrm{d}(\varphi)_{n,s,\delta}+\\
&+C_ns^{-n}(\mathrm{d}(\psi)_{n,r,\sigma}\mathrm{d}(\varphi)_{1,s,\delta}+\mathrm{d}(\varphi)_{n,s,\delta}\mathrm{d}(\psi)_{1,r,\sigma}),\\
\mathrm{d}(\psi\circ \varphi,\psi)_{n,s,\delta}\leq& \mathrm{d}(\varphi)_{n,s,\delta}+\\
&+C_ns^{-n}(\mathrm{d}(\psi)_{n+1,r,\sigma}\mathrm{d}(\varphi)_{1,s,\delta}+\mathrm{d}(\varphi)_{n,s,\delta}\mathrm{d}(\psi)_{1,r,\sigma}).
\end{align*}
\end{lemma}
\begin{proof}
Denote the local expressions of $\varphi$ and $\psi$ as follows:
\[\varphi_i:=\textrm{Id}+g_{i}:\overline{B}_{\delta}\times\overline{B}_{s}\rmap B_{\sigma}\times B_{r},\]
\[\psi_{i}:=\textrm{Id}+f_i:\overline{B}_{\sigma}\times \overline{B}_{r}\rmap \mathbb{R}^d\times \mathbb{R}^D.\]
Then for all $z\in \overline{B}_{\delta}\times\overline{B}_{s}$, we can write
\[\psi_i(\varphi_i(z))-z=f_i(z+g_i(z))+g_i(z).\]
By computing the $\frac{\partial^{|\alpha|}}{\partial z^{\alpha}}$ of the right hand side, for a multi-index $\alpha$
with $|\alpha|=n$, we obtain an expression of the form
\begin{align*}
\frac{\partial^{|\alpha|}g_i}{\partial z^{\alpha}}(z)+\frac{\partial^{|\alpha|}f_{i}}{\partial z^{\alpha}}(\varphi_{i}(z))+\sum_{\beta,\gamma_1,\ldots,\gamma_p}\frac{\partial^{|\beta|}f_{i}}{\partial z^{\beta}}(\varphi_i(z))\frac{\partial^{|\gamma_1|}g^{j_1}_{i}}{\partial z^{\gamma_1}}(z)\ldots\frac{\partial^{|\gamma_p|}g^{j_p}_{i}}{\partial z^{\gamma_p}}(z),
\end{align*}
where the multi-indices in the sum satisfy
\begin{equation}\label{EQ_E2}
1\leq p\leq n,\ 1\leq|\beta|,|\gamma_j|\leq n, \ |\beta|+\sum_{j=1}^p(|\gamma_j|-1)=n.
\end{equation}
The first two terms can be bounded by $\mathrm{d}(\psi)_{n,r,\sigma}+\mathrm{d}(\varphi)_{n,s,\delta}$. For the
last term we use the interpolation inequalities to obtain
\[\|f_i\|_{|\beta|,r,\sigma}\leq C_ns^{1-|\beta|}\|f_i\|_{1,r,\sigma}^{\frac{n-|\beta|}{n-1}}\|f_i\|_{n,r,\sigma}^{\frac{|\beta|-1}{n-1}},\]
\[\|g_i\|_{|\gamma_i|,s,\delta}\leq C_ns^{1-|\gamma_i|}\|g_i\|_{1,s,\delta}^{\frac{n-|\gamma_i|}{n-1}}\|g_i\|_{n,s,\delta}^{\frac{|\gamma_i|-1}{n-1}}.\]
Multiplying all these, and using (\ref{EQ_E2}), the sum is bounded by
\[ C_ns^{1-n}\|g_i\|_{1,s,\delta}^{p-1}(\|f_i\|_{1,r,\sigma}\|g_i\|_{n,s,\delta})^{\frac{n-|\beta|}{n-1}}(\|f_i\|_{n,r,\sigma}\|g_i\|_{1,s,\delta})^{\frac{|\beta|-1}{n-1}}.\]
Lemma \ref{Lemma_equivalent_norms} implies that $\|g_i\|_{1,s,\delta}<C$. Dropping this term, the first part follows
using inequality (\ref{EQ_simple}).

For the second part, write for $z\in \overline{B}_{\delta}\times\overline{B}_{s}$:
\[\psi_i(\varphi_i(z))-\psi_i(z)=f_i(z+g_i(z))-f_i(z)+g_i(z).\]
We compute $\frac{\partial^{|\alpha|}}{\partial z^{\alpha}}$ of the right hand side, for $\alpha$ a multi-index with
$|\alpha|=n$:
\begin{align*}
\frac{\partial^{|\alpha|}f_{i}}{\partial z^{\alpha}}&(\varphi_{i}(z))-\frac{\partial^{|\alpha|}f_{i}}{\partial z^{\alpha}}(z)+\frac{\partial^{|\alpha|}g_{i}}{\partial z^{\alpha}}(z)+\\
&+ \sum_{\beta,\gamma_1,\ldots,\gamma_p}\frac{\partial^{|\beta|}f_{i}}{\partial z^{\beta}}(\varphi_{i}(z))\frac{\partial^{|\gamma_1|}g^{j_1}_{i}}{\partial z^{\gamma_1}}(z)\ldots\frac{\partial^{|\gamma_p|}g^{j_p}_{i}}{\partial z^{\gamma_p}}(z).
\end{align*}
where the multi-indices in the sum satisfy (\ref{EQ_E2}). The last term we bound as before, and the third by
$\mathrm{d}(\varphi)_{n,s,\delta}$. Writing the first two terms as
\begin{align*}
\sum_{j=1}^{d+D}\int_{0}^{1}\frac{\partial^{|\alpha|+1}f_{i}}{\partial z_j\partial z^{\alpha}}(z+tg_{i}(z))g^j_{i}(z)dt,
\end{align*}
they are less than $C\mathrm{d}(\psi)_{n+1,r,\sigma}\mathrm{d}(\varphi)_{0,s,\delta}$. Adding up, the result
follows.
\end{proof}

We give now conditions for infinite compositions of maps to converge.
\begin{lemma}\label{Lemma_convergent_embbedings}
There exists $\theta>0$, such that for all sequences
\[\{\varphi_{k}\in \mathcal{U}_{r_k}\}_{k\geq 1},\ \ \varphi_k:E_{r_k}\rmap E_{r_{k-1}},\]
where $0<r<r_{k}<r_{k-1}\leq r_0<1$, which satisfy
\[\sigma_0:=\sum_{k\geq 1}\mathrm{d}(\varphi_k)_{0,r_k}<\theta,\ \ \ \sigma_n:=\sum_{k\geq 1}\mathrm{d}(\varphi_k)_{n,r_k}<\infty,\  \forall \ n\geq 1,\]
the sequence of maps
\[\psi_k:=\varphi_1\circ\ldots\circ \varphi_k:E_{r_k}\rmap E_{r_0},\]
converges in all $C^n$-norms on $E_r$ to a map $\psi:E_r\to E_{r_0}$, with $\psi\in\mathcal{U}_r$. Moreover, there
are $C_n>0$, such that if $\mathrm{d}(\varphi_k)_{1,r_k}<1$, $\forall$ $k\geq 1$, then
\[\mathrm{d}(\psi)_{n,r}\leq e^{C_nr^{-n}\sigma_n}C_nr^{-n}\sigma_n.\]
\end{lemma}
\begin{proof}
Consider the following sequences of numbers:
\[\epsilon_k:=\frac{\mathrm{d}(\varphi_k)_{0,r_k}}{\sum_{l\geq 1}\mathrm{d}(\varphi_l)_{0,r_l}}, \ \ \delta_k:=2-\sum_{l=1}^k\epsilon_l.\]
We have that $\mathrm{d}(\varphi_k)_{0,r_k}\leq \epsilon_k\theta$. So, by Lemma \ref{Lemma_control_embedding},
we may assume that
\[\varphi_k(E_{r_k|\overline{O}_i^2})\subset E_{r_{k-1}|O_i},\ \ \varphi_{k}(E_{r_k|\overline{O}_i^{\delta_k}})\subset E_{r_{k-1}|O^{\delta_{k-1}}_i},\]
and this implies that
\[\psi_{k-1}(E_{r_{k-1}|\overline{O}_i^{\delta_{k-1}}})\subset E_{r_0|O_i}.\]
So we can apply Lemma \ref{Lemma_composition} to the pair $\psi_{k-1}$ and $\varphi_k$ for all $k> k_0$. The first
part of Lemma \ref{Lemma_composition} and Lemma \ref{Lemma_equivalent_norms} imply an inequality of the form
\begin{align*}
1+\mathrm{d}(\psi_{k})_{n,r_k,\delta_k}\leq (1+\mathrm{d}(\psi_{k-1})_{n,r_{k-1},\delta_{k-1}})(1+C_nr^{-n}\mathrm{d}(\varphi_{k})_{n,r_{k}}).
\end{align*}
Iterating this inequality, we obtain that
\begin{align*}
1+\mathrm{d}(\psi_{k})_{n,r_k,\delta_k}&\leq (1+\mathrm{d}(\psi_{k_0})_{n,r_{k_0},\delta_{k_0}})\prod_{l=k_0+1}^k(1+C_nr^{-n}\mathrm{d}(\varphi_{l})_{n,r_{l}})\leq\\
&\leq (1+\mathrm{d}(\psi_{k_0})_{n,r_{k_0},\delta_{k_0}})e^{C_nr^{-n}\sum_{l> k_0}\mathrm{d}(\varphi_{l})_{n,r_{l}}}\leq\\
&\leq (1+\mathrm{d}(\psi_{k_0})_{n,r_{k_0},\delta_{k_0}})e^{C_nr^{-n}\sigma_n}.
\end{align*}
The second part of Lemma \ref{Lemma_composition} and Lemma \ref{Lemma_equivalent_norms} imply
\begin{align*}
 \mathrm{d}(\psi_{k},\psi_{k-1})_{n,r}&\leq (1+\mathrm{d}(\psi_{k-1})_{n+1,r_{k-1},\delta_{k-1}})C_nr^{-n}\mathrm{d}(\varphi_{k})_{n,r_{k},\delta_{k}}\leq\\
\nonumber &\leq (1+\mathrm{d}(\psi_{k_0})_{n+1,r_{k_0},\delta_{k_0}})e^{C_{n+1}r^{-1-n}\sigma_{n+1}}C_nr^{-n}\mathrm{d}(\varphi_{k})_{n,r_{k}}.
\end{align*}
This shows that the sum $\sum_{k\geq 1}\mathrm{d}(\psi_{k},\psi_{k-1})_{n,r}$ converges for all $n$, hence the
sequence $\{\psi_{k|E_r}\}_{k\geq 1}$ converges in all $C^n$-norms to a smooth function $\psi:E_{r}\to E_{r_0}$.

If $\mathrm{d}(\varphi_k)_{1,r_k}<1$ for all $k\geq 1$, then we can take $k_0=0$. So we obtain
\begin{align*}
1+\mathrm{d}(\psi_{k})_{n,r_k,\delta_k}\leq\prod_{l=1}^k(1+C_nr^{-n}\mathrm{d}(\varphi_{l})_{n,r_{l}})\leq e^{C_nr^{-n}\sum_{1}^k\mathrm{d}(\varphi_{l})_{n,r_{l}}}\leq e^{C_nr^{-n}\sigma_n}.
\end{align*}
Using the trivial inequality $e^x-1\leq xe^x$ for $x\geq 0$, the result follows.
\end{proof}

\subsubsection*{Tameness of the flow} The $C^0$-norm of a vector field controls the size of the domain of its flow.

\begin{lemma}\label{Lemma_domain_of_flow}
There exists $\theta>0$ such that for all $0<s<r\leq 1$ and all $X\in \mathfrak{X}^1(E_r)$ with
$\|X\|_{0,r}<(r-s)\theta$, we have that $\varphi_X^t$, the flow of $X$, is defined for all $t\in [0,1]$ on $E_{s}$ and
belongs to $\mathcal{U}_{s}$.
\end{lemma}
\begin{proof}
We denote the restriction of $X$ to a chart by $X_i\in \mathfrak{X}^{1}(\mathbb{R}^d \times \overline{B}_r)$.
Consider $p\in\overline{B}_1\times \overline{B}_{s}$. Let $t\in(0,1]$ be such that the flow of $X_i$ is defined up to
time $t$ at $p$ and such that for all $\tau\in [0,t)$ it satisfies $\varphi^{\tau}_{X_{i}}(p)\in B_2\times B_r$. Then we
have that
\begin{align*}
|\varphi^t_{X_{i}}(p)-p|=|\int_0^td \left(\varphi^{\tau}_{X_{i}}(p)\right)|\leq\int_0^t|X_i(\varphi^{\tau}_{X_{i}}(p))|d\tau\leq \|X_i\|_{0,r,2}\leq C\|X\|_{0,r},
\end{align*}
where for the last step we used Lemma \ref{Lemma_equivalent_norms}. Hence, if $\|X\|_{0,r}<(r-s)/C$, we have that
$\varphi^t_{X_{i}}(p)\in B_2\times B_r$, and this implies the result.
\end{proof}

We prove now that the map which associates to a vector field its flow is tame (this proof was inspired by the proof of
Lemma B.3 in \cite{Miranda}).
\begin{lemma}\label{Lemma_size of _the flow}
There exists $\theta>0$ such that for all $0<s<r\leq 1$, and all $X\in \mathfrak{X}^1(E_r)$ with
\[\|X\|_{0,r}<(r-s)\theta,\ \ \|X\|_{1,r}<\theta\]
we have that $\varphi_X:=\varphi_X^1$ belongs to $\mathcal{U}_{s}$ and it satisfies:
\[\mathrm{d}(\varphi_X)_{0,s}\leq C_0\|X\|_{0,r},\ \ \mathrm{d}(\varphi_{X})_{n,s}\leq r^{1-n}C_n\|X\|_{n,r},\ \forall\ n\geq 1.\]
\end{lemma}

\begin{proof}
By Lemma \ref{Lemma_domain_of_flow}, for $t\in[0,1]$, we have that $\varphi^t_X\in\mathcal{U}_{s}$, and by its
proof that the local representatives take values in $B_2\times B_r$
\[\varphi^t_{X_i}:=\textrm{Id}+g_{i,t}: \overline{B}_1\times \overline{B}_{s}\rmap B_2\times B_r.\]
We will prove by induction on $n$ that $g_{i,t}$ satisfies inequalities of the form:
\begin{equation}\label{EQ_poly}
\|g_{i,t}\|_{n,s}\leq C_nP_{n}(X),
\end{equation}
where $P_n(X)$ denotes the following polynomials in the norms of $X$
\[P_0(X)=\|X\|_{0,r}, \ P_1(X)=\|X\|_{1,r},\]
\[ P_n(X)=\sum_{\stackrel{j_1+\ldots+j_p=n-1}{1\leq j_k\leq n-1}} \| X \|_{j_1+1,r}\ldots \|X\|_{j_p+1,r}.\]
Observe that (\ref{EQ_poly}) implies the conclusion, since by the interpolation inequalities and the fact that
$\|X\|_{1,r}<\theta \leq 1$ we have that
\[\|X\|_{j_k+1,r}\leq C_nr^{-j_k}(\|X\|_{1,r})^{1-\frac{j_k}{n-1}}(\|X\|_{n,r})^{\frac{j_k}{n-1}}\leq C_n r^{-j_k}\|X\|_{n,r}^{\frac{j_k}{n-1}},\]
hence
\[P_n(X) \leq C_nr^{1-n}\|X\|_{n,r}.\]

The map $g_{i,t}$ satisfies the ordinary differential equation
\[\frac{d g_{i,t}}{dt}(z)=\frac{d \varphi^t_{X_i}}{dt}(z)=X_i(\varphi_{X_i}^t(z))=X_i(g_{i,t}(z)+z).\]
Since $g_{i,0}=0$, it follows that
\begin{equation}\label{EQ_for_the_difference}
g_{i,t}(z)=\int_0^tX_i(z+g_{i,\tau}(z))d\tau.
\end{equation}
Using also Lemma \ref{Lemma_equivalent_norms}, we obtain the result for $n=0$:
\[\|g_{i,t}\|_{0,s} \leq \|X\|_{0,r,2}\leq C_0\|X\|_{0,r}.\]
We will use the following version of the Gronwall inequality: if $u:[0,1]\to \mathbb{R}$ is a continuous map and there
are positive constants $A$, $B$ such that
\[u(t)\leq A+B\int_{0}^tu(\tau)d\tau,\]
then $u$ satisfies $u(t)\leq Ae^B$. For the proof, we first compute the derivative
\[\frac{d}{dt}\left( e^{-tB}(A+B\int_{0}^tu(s)ds)\right)= e^{-tB} B\left(-A-B\int_{0}^tu(s)ds+u(t)\right)\leq 0.\]
Thus, the function is non-increasing, and therefore
\[e^{-tB}\left( A+B\int_{0}^tu(s)ds\right)\leq A,\]
which implies the inequality:
\[u(t)\leq Ae^{tB}\leq Ae^B.\]

Computing the partial derivative $\frac{\partial}{\partial z_j}$ of equation (\ref{EQ_for_the_difference}) we obtain
\begin{align*}
\frac{\partial g_{i,t}}{\partial z_j}(z)&=\int_0^t\left(\frac{\partial X_{i}}{\partial z_j}(z+g_{i,\tau}(z))+\sum_{k=1}^{D+d}\frac{\partial X_i}{\partial z_k}(z+g_{i,\tau}(z))\frac{\partial g^k_{i,\tau}}{\partial z_j}(z)\right) d\tau.
\end{align*}
Therefore, using again Lemma \ref{Lemma_equivalent_norms}, the function $|\frac{\partial g_{i,t}}{\partial z_j}(z)|$
satisfies:
\begin{align*}
|\frac{\partial g_{i,t}}{\partial z_j}(z)|\leq C\|X\|_{1,r}+(D+d)\|X\|_{1,r}\int_0^t|\frac{\partial g_{i,\tau}}{\partial z_j}(z)|d\tau.
\end{align*}
The case $n=1$ follows now by Gronwall's inequality:
\[\|\frac{\partial g_{i,t}}{\partial z_j}\|_{0,s}\leq C \|X\|_{1,r}e^{(D+d)\|X\|_{1,r}}\leq C\|X\|_{1,r}.\]

For a multi-index $\alpha$, with $|\alpha|=n\geq 2$, applying $\frac{\partial^{|\alpha|}}{\partial z^{\alpha}}$ to
(\ref{EQ_for_the_difference}), we obtain
\begin{align}\label{EQ_3}
\frac{\partial^{|\alpha|}g_{i,t}}{\partial z^{\alpha}}(z)&=\int_{0}^t\sum_{2\leq |\beta|\leq |\alpha|} \frac{\partial^{|\beta|}X_{i}}{\partial z^{\beta}}(z+g_{i,\tau}(z))\frac{\partial^{|\gamma_1|}g_{i,\tau}^{i_1}}{\partial z^{\gamma_1}}(z)\ldots\frac{\partial^{|\gamma_p|}g_{i,\tau}^{i_p}}{\partial z^{\gamma_p}}(z)d\tau+\\
\nonumber&+\int_0^t\sum_{j=1}^{D+d}\frac{\partial X_i}{\partial z_j}(z+g_{i,\tau}(z))\frac{\partial^{|\alpha|}g_{i,\tau}^{j}}{\partial z^{\alpha}}(z) d\tau,
\end{align}
where the multi-indices satisfy
\[1\leq |\gamma_k|\leq n-1,\ \  (|\gamma_1|-1)+\ldots + (|\gamma_p|-1)+|\beta|=n.\]
Since $|\gamma_k|\leq n-1$, we can apply induction to conclude that
\[\|\frac{\partial^{|\gamma_k|}g_{i,\tau}^{i_k}}{\partial z^{\gamma_k}}\|_{0,s}\leq P_{|\gamma_k|}(X).\]
So, the first part of the sum can be bounded by
\begin{align}\label{EQ_2}
C_n\sum_{\stackrel{j_0+\ldots+j_p=n-1}{1\leq j_k\leq n-1}} \| X \|_{j_0+1,1}P_{j_1+1}(X)\ldots P_{j_p+1}(X).
\end{align}
It is easy to see that the polynomials $P_k(X)$ satisfy:
\begin{equation}\label{EQ_mult_prop_P}
P_{u+1}(X)P_{v+1}(X)\leq C_{u,v}P_{u+v+1}(X),
\end{equation}
therefore (\ref{EQ_2}) is bounded by $C_nP_n(X)$. Using this in (\ref{EQ_3}), we obtain
\[|\frac{\partial^{|\alpha|} g_{i,t}}{\partial z^{\alpha}}(z)|\leq C_nP_n(X)+(D+d)\|X\|_{1,r}\int_0^t|\frac{\partial^{|\alpha|} g_{i,\tau}}{\partial z^{\alpha}}(z)|d\tau.\]
Applying Gronwall's inequality, we obtain the conclusion.
\end{proof}

Next, we show how to approximate pullbacks by flows of vector fields.
\begin{lemma}\label{Lemma_tame_flow_pull_back}
There exists $\theta >0$, such that for all $0<s<r\leq 1$ and all $X\in \mathfrak{X}^1(E_r)$ with
$\|X\|_{0,r}<(r-s)\theta$ and $\|X\|_{1,r}<\theta$, we have that
\begin{align*}
\|\varphi_{X}^*(W)&\|_{n,s} \leq C_n r^{-n}(\|W\|_{n,r}+\|W\|_{0,r}\|X\|_{n+1,r}),\\
\|\varphi_{X}^*(W)&-W_{|s}\|_{n,s}\leq C_nr^{-2n-1}(\|X\|_{n+1,r}\|W\|_{1,r}+\|X\|_{1,r}\|W\|_{n+1,r}),\\
\|\varphi_{X}^*(W)&-W_{|s}- \varphi_{X}^*([X,W])\|_{n,s} \leq \\
&\leq C_nr^{-3(n+2)}\|X\|_{0,r}(\|X\|_{n+2,r}\|W\|_{2,r}+\|X\|_{2,r}\|W\|_{n+2,r}),
\end{align*}
for all $W\in\mathfrak{X}^{\bullet}(E_r)$, where $C_n>0$ is a constant depending only on $n$.
\end{lemma}
\begin{proof}
As in the proof above, the local expression of $\varphi_X$ is defined as follows:
\[\varphi_{X_i}=\textrm{Id}+g_{i}:\overline{B}_1\times \overline{B}_{s}\rmap B_2\times B_r.\]
Let $W\in\mathfrak{X}^{\bullet}(E_r)$, and denote by $W_i$ its local expression on $E_{r|\overline{O}_2^i}$:
\[W_i:=\sum_{J=\{j_1<\ldots<j_k\}} W_i^J(z)\frac{\partial}{\partial z_{j_1}}\wedge \ldots \wedge \frac{\partial}{\partial z_{j_k}}\in\mathfrak{X}^{\bullet}(\overline{B}_2\times\overline{B}_r).\]
The local representative of $\varphi_X^*(W)$, is given for $z\in \overline{B}_1\times \overline{B}_{s}$ by
\[(\varphi_X^*W)_i=\sum_{J} W_i^J(z+g_i(z))(\textrm{Id}+d_zg_i)^{-1}\frac{\partial}{\partial z_{j_1}}\wedge \ldots \wedge (\textrm{Id}+d_zg_i)^{-1}\frac{\partial}{\partial z_{j_k}}.\]
By the Cramer rule, the matrix $(\textrm{Id}+d_zg_i)^{-1}$ has entries of the form
\[\Psi\left(\frac{\partial g_i^{l}}{\partial z_{j}}(z)\right)det(\textrm{Id}+d_zg_i)^{-1},\]
where $\Psi$ is a polynomial in the variables $Y^l_j$, which we substitute by $\frac{\partial g_i^{l}}{\partial
z_{j}}(z)$. Therefore, any coefficient of the local expression of $\varphi_X^*(W)_i$, will be a sum of elements of the
form
\[W_i^{J}(z+g_i(z))\Psi\left(\frac{\partial g_i^{l}}{\partial z_{j}}(z)\right)det(\textrm{Id}+d_zg_i)^{-k}.\]
When computing $\frac{\partial^{|\alpha|}}{\partial z^{\alpha}}$ of such an expression with $|\alpha|=n$, using an
inductive argument, one proves that the outcome is a sum of terms of the form
\begin{equation}\label{EQ_9}
\frac{\partial^{|\beta|}W_{i}^J}{\partial z^{\beta}}(z+g_i(z))\frac{\partial^{|\gamma_1|} g_i^{v_{1}}}{\partial z^{\gamma_1}}(z)\ldots \frac{\partial^{|\gamma_p|} g_i^{v_{p}}}{\partial z^{\gamma_p}}(z)det(\textrm{Id}+d_zg_i)^{-M}
\end{equation}
with coefficients depending only on $\alpha$ and on the multi-indices, which satisfy
\[0\leq p,\ 0\leq M,  \ 1\leq |\gamma_j|, \ |\beta|+(|\gamma_1|-1)+\ldots+(|\gamma_p|-1)=n.\]
By Lemma \ref{Lemma_size of _the flow}, $\|g_i\|_{1,s}<C\theta$, so by shrinking $\theta$ if needed, we find that
\[det(\textrm{Id}+d_zg_i)^{-1}<2, \ \forall z\in\overline{B}_1\times \overline{B}_{s}.\]
Using this, Lemma \ref{Lemma_equivalent_norms} for $W$, and $|\frac{\partial g_i^{l}}{\partial z_{j}}(z)|\leq C$, we
bound (\ref{EQ_9}) by
\[C_n\sum_{j,j_1,\ldots, j_p}\|W\|_{j,r}\|g_i\|_{j_1+1,s}\ldots\|g_i\|_{j_p+1,s},\]
where the indexes satisfy
\[0\leq j,\ 0\leq j_k,\ j+j_1+\ldots +j_p=n.\]
The term with $p=0$ can simply be bounded by $C_n\|W\|_{n,r}$. For the rest, we use the bound $\|g_i\|_{j_k+1,s}\leq
P_{j_k+1}(X)$ from the proof of Lemma \ref{Lemma_size of _the flow}. The multiplicative property
(\ref{EQ_mult_prop_P}) of the polynomials $P_l(X)$ implies
\[\|\varphi_X^*(W)\|_{n,s}\leq C_n\sum_{j=0}^n\|W\|_{j,r}P_{n-j+1}(X).\]
Applying interpolation to $W_{j,r}$ and to a term of $P_{n-j+1}(X)$ we obtain
\begin{align*}
\|W\|_{j,r}&\leq C_nr^{-j}\|W\|_{0,r}^{1-j/n}\|W\|_{n,r}^{j/n},\\
\|X\|_{j_k+1,r}&\leq C_nr^{-j_k}\|X\|_{1,r}^{1-j_k/n}\|X\|_{n+1,r}^{j_k/n}\leq C_nr^{-j_k} \|X\|_{n+1,r}^{j_k/n}.
\end{align*}
Multiplying these terms and using (\ref{EQ_simple}), we obtain the first inequality:
\begin{align*}
\|W\|_{j,r}\|X\|_{j_1+1,r}\ldots\|X\|_{j_p+1,r}&\leq C_nr^{-n}(\|W\|_{0,r}\|X\|_{n+1,r})^{1-j/n}\|W\|_{n,r}^{j/n}\leq\\
&\leq C_nr^{-n}(\|W\|_{n,r}+\|W\|_{0,r}\|X\|_{n+1,r}).
\end{align*}

For the second inequality, denote
\[W_t:=\varphi^{t*}_{X}(W)-W_{|s}\in\mathfrak{X}^{\bullet}(E_{s}).\]
Then $W_0=0$, $W_1=\varphi_{X}^*(W)-W_{|s}$, and $\frac{d}{dt}W_t=\varphi^{t*}_{X}([X,W])$, therefore
\[\varphi_{X}^*(W)-W_{|s}=\int_{0}^1\varphi^{t*}_{X}([X,W])dt.\]
By the first part, we obtain
\[\|\varphi_{X}^*(W)-W_{|s}\|_{n,s}\leq C_nr^{-n}(\|[X,W]\|_{n,r}+\|[X,W]\|_{0,r}\|X\|_{n+1,r}).\]
Using now Lemma \ref{L_Bracket} and that $\|X\|_{1,r}\leq \theta$ we obtain the second part:
\[\|\varphi_{X}^*(W)-W_{|s}\|_{n,s}\leq C_nr^{-2n-1}(\|X\|_{n+1,r}\|W\|_{1,r}+\|W\|_{1,r}\|X\|_{n+1,r}).\]

For the last inequality, denote
\[W_t:=\varphi^{t*}_{X}(W)-W_{|s}- t\varphi_{X}^{t*}([X,W]).\]
We have that $W_0=0$, $W_1=\varphi^{*}_{X}(W)-W_{|s}- \varphi_{X}^*([X,W])$, and
\[\frac{d}{dt}W_t=-t\varphi_{X}^{t*}([X,[X,W]]),\]
therefore
\[W_1=-\int_{0}^1t\varphi_{X}^{t*}([X,[X,W]])dt.\]
Using again the first part, it follows that
\begin{equation}\label{EQ_1}
\|W_1\|_{n,s}\leq C_nr^{-n}(\|[X,[X,W]]\|_{n,r}+\|[X,[X,W]]\|_{0,r}\|X\|_{n+1,r}).
\end{equation}
Applying twice Lemma \ref{L_Bracket}, for all $k\leq n$ we obtain that
\begin{align*}
\|[X,[X,W]]&\|_{k,r}\leq C_{n}(r^{-(k+3)}\|X\|_{k+1,r}(\|X\|_{0,r}\|W\|_{1,r}+\|X\|_{1,r}\|W\|_{0,r})+\\
&+r^{-(2k+3)}\|X\|_{0,r}(\|X\|_{0,r}\|W\|_{k+2,r}+\|X\|_{k+2,r}\|W\|_{0,r}))\leq\\
&\leq C_{n}r^{-(2k+5)}\|X\|_{0,r}(\|W\|_{k+2,r}\|X\|_{0,r}+\|W\|_{2,r}\|X\|_{k+2,r}),
\end{align*}
where we have used the interpolation inequality
\[\|X\|_{1,r}\|X\|_{k+1,r}\leq C_nr^{-(k+2)}\|X\|_{0,r}\|X\|_{k+2,r}.\]
The first term in (\ref{EQ_1}) can be bounded using this inequality for $k=n$. For $k=0$, using also that
$\|X\|_{1,r}\leq \theta$ and the interpolation inequality
\[\|X\|_{2,r}\|X\|_{n+1,r}\leq C_nr^{-(n+1)}\|X\|_{1,r}\|X\|_{n+2,r},\]
we can bound the second term in (\ref{EQ_1}), and this concludes the proof:
\[\|[X,[X,W]]\|_{0,r}\|X\|_{n+1,r}\leq C_nr^{-(n+6)}\|W\|_{2,r}\|X\|_{0,r}\|X\|_{n+2,r}.\qedhere\]
\end{proof}

\subsection{An invariant tubular neighborhood and tame homotopy operators}

We start now the proof of Theorem \ref{Theorem_FOUR}. Let $(M,\pi)$ and $S\subset M$ be as in the statement. Let
$\mathcal{G}\rightrightarrows M$ be a Lie groupoid integrating $T^*M$. By restricting to the connected components
of the identities in the $s$-fibers of $\mathcal{G}$ \cite{MM}, we may assume that $\mathcal{G}$ has connected
$s$-fibers.

By Lemma \ref{Lemma_tubular_neighborhood}, $S$ has an invariant tubular neighborhood $E\cong \nu_S$ endowed
with a metric such that the closed tubes $E_r:=\{v\in E |  |v|\leq r\}$, for $r>0$, are also $\mathcal{G}$-invariant. We
endow $E$ with all the structure from subsection \ref{subsection_norms}.

Since $E$ is invariant, the cotangent Lie algebroid of $(E,\pi)$ is integrable by $\mathcal{G}_{|E}$, which has
compact $s$-fibers with vanishing $H^2$. Therefore, by the Tame Vanishing Lemma and Corollaries
\ref{corollary_unu} and \ref{corollary_unu_prim} from the appendix, there are linear homotopy operators
\[\mathfrak{X}^1(E)\stackrel{h_1}{\longleftarrow}\mathfrak{X}^2(E)\stackrel{h_2}{\longleftarrow}\mathfrak{X}^3(E),\]
\[[\pi,h_1(V)]+h_2([\pi,V])=V,\ \ \forall \  V\in \mathfrak{X}^2(E),\]
which satisfy:
\begin{itemize}
\item they induce linear homotopy operators $h_1^r$ and $h_2^r$ on $(E_r,\pi_{|r})$;
\item there are constants $C_{n}>0$ such that, for all $r\in(0,1]$,
\[\|h_1^{r}(X)\|_{n,r}\leq C_{n} \|X\|_{n+s,r},\ \ \|h_2^{r}(Y)\|_{n,r}\leq C_{n} \|Y\|_{n+s,r},\]
for all $X\in\mathfrak{X}^2(E_r)$, $Y\in\mathfrak{X}^3(E_r)$, where
$s=\lfloor\frac{1}{2}\mathrm{dim}(M)\rfloor+1$;
\item they induce homotopy operators in second degree on the subcomplex of vector fields vanishing along $S$.
\end{itemize}

\subsection{The Nash-Moser method}

We fix radii $0<r<R<1$. Let $s$ be as in the previous subsection, and let
\[\alpha:=2(s+5), \ \ \  \ \ p:=7(s+4).\]
The integer $p$ is the one from the statement of Theorem \ref{Theorem_FOUR}. Consider a second Poisson structure
$\widetilde{\pi}$ defined on $E_R$. To $\widetilde{\pi}$ we associate the following inductive procedure:

\noindent\textbf{Procedure P$_0$}:
\begin{itemize}
\item consider the number
\[t(\widetilde{\pi}):=\|\pi-\widetilde{\pi}\|_{p,R}^{-1/\alpha},\]
\item consider the sequences of numbers
\[\begin{array}{ccc}
  \epsilon_0:=(R-r)/4, & r_0:=R, &  t_0:=t(\widetilde{\pi}), \\
  \epsilon_{k+1}:=\epsilon_k^{3/2}, & r_{k+1}:=r_k-\epsilon_k, & t_{k+1}:=t_k^{3/2},
\end{array}\]
\item consider the sequences of Poisson bivectors and vector fields
\[\{\pi_k\in\mathfrak{X}^{2}(E_{r_k})\}_{k\geq 0},\ \ \ \{X_k\in\mathfrak{X}^1(E_{r_k})\}_{k\geq 0}, \]
defined inductively by
\begin{equation}\label{EQ_procedure}
\pi_0:=\widetilde{\pi}, \ \ \ \pi_{k+1}:=\varphi_{X_k}^*(\pi_k),\ \ \ X_k:=S_{t_k}^{r_k}(h_1^{r_k}(\pi_k-\pi_{|{r_k}})),
\end{equation}
\item consider the sequence of maps
\[\psi_k:=\varphi_{X_0}\circ\ldots\circ \varphi_{X_{k}}:E_{r_{k+1}}\rmap E_R.\]
\end{itemize}
By our choice of $\epsilon_0$, observe that $r<r_k< R$ for all $k\geq 1$:
\begin{align*}
\sum_{k=0}^{\infty} \epsilon_k=\sum_{k=0}^{\infty} \epsilon_0^{(3/2)^k}<\sum_{k=0}^{\infty} \epsilon_0^{1+\frac{k}{2}}=\frac{\epsilon_0}{1-\sqrt{\epsilon_0}}\leq (R-r),
\end{align*}
For \textbf{Procedure P$_0$} to be well-defined, we need that:
\begin{itemize}
\item [$(C_k)$] the time one flow of $X_k$ to be defined as a map \[\varphi_{X_k}:E_{r_{k+1}}\rmap E_{r_{k}}.\]
\end{itemize}

For part (b) of Theorem \ref{Theorem_FOUR}, we consider also \textbf{Procedure P$_1$} associated to
$\widetilde{\pi}$ such that $j^1_{|S}\widetilde{\pi}=j^1_{|S}\pi$. We define \textbf{Procedure P$_1$} the same as
\textbf{Procedure P$_0$} except that in (\ref{EQ_procedure}) we use the smoothing operators $S_{t_k}^{r_k,1}$.

For \textbf{Procedure P$_1$} to be well-defined, besides for condition $(C_k)$, one also needs that
$h_1^{r_k}(\pi_k-\pi_{|{r_k}})\in \mathfrak{X}^1(E_{r_k})^{(1)}$. This is automatically satisfied because the
operators $h_1^{r_k}$ preserve the space of tensors vanishing up to first order, and because
$j^1_{|S}(\pi_k-\pi_{|{r_k}})=0$. This last claim can be proven inductively: By hypothesis,
$j^1_{|S}(\pi_0-\pi_{|R})=0$. Assume that $j^1_{|S}(\pi_k-\pi_{|{r_k}})=0$, for some $k\geq 0$. Then also $X_k\in
\mathfrak{X}^1(E_{r_k})^{(1)}$, hence the first order jet of $\varphi_{X_k}$ along $S$ is that of the identity, and so
\[j^1_{|S}(\pi_{k+1})=j^1_{|S}(\pi_{k})=j^1_{|S}(\pi).\]
Therefore $j^1_{|S}(\pi_{k+1}-\pi_{|{r_{k+1}}})=0$.\\

These procedures converge to the map $\psi$ from Theorem \ref{Theorem_FOUR}.

\begin{proposition}\label{Proposition_technical}
There exists $\delta>0$ and an integer $d\geq 0$, such that both procedures \textbf{P$_0$} and \textbf{P$_1$} are
well defined for every $\widetilde{\pi}$ satisfying
\begin{equation}\label{EQ_8}
\|\widetilde{\pi}-\pi\|_{p,R}<\delta (r(R-r))^{d}.
\end{equation}
Moreover, the sequence $\psi_{k|r}$ converges uniformly on $E_r$ with all its derivatives to a local diffeomorphism
$\psi$, which is a Poisson map between
\[\psi:(E_r,\pi_{|r})\rmap (E_R,\widetilde{\pi}),\]
and which satisfies
\begin{equation}\label{EQ_continuity}
\mathrm{d}(\psi)_{1,r}\leq \|\pi-\widetilde{\pi}\|^{1/\alpha}_{p,R}.
\end{equation}
If $j^1_{|S}\widetilde{\pi}=j^1_{|S}\pi$, then $\psi$ obtained by \textbf{Procedure P$_1$} is the identity along $S$
up to first order.
\end{proposition}
\begin{proof}
We will prove the statement for the two procedures simultaneously. We denote by $S_k$ the used smoothing operators,
that is, in \textbf{P$_0$} we let $S_{k}:=S_{t_k}^{r_k}$ and in \textbf{P$_1$} we let $S_{k}:=S_{t_k}^{r_k,1}$.
In both cases, the following hold:
\begin{align*}
\|S_k(X)\|_{m,r_k}&\leq C_{m}r^{-c_m}{t}^{l+1}_k\|X\|_{m-l,r_k},\\
\|S_k(X)-X\|_{m-l,r_k}&\leq C_{m}r^{-c_m}t^{-l}_k\|X\|_{m+1,r_k}.
\end{align*}

For the procedures to be well-defined and to converge, we need that $t_0=t(\widetilde{\pi})$ is big enough, more
precisely, it must satisfy a finite number of inequalities of the form
\begin{equation}\label{EQ_0}
t_0=t(\widetilde{\pi})>C(r(R-r))^{-c}.
\end{equation}
Taking $\widetilde{\pi}$ such that it satisfies (\ref{EQ_8}), it suffices to ask that $\delta$ is small enough and $d$ is
big enough, such that a finite number of inequalities of the form
\[\delta((R-r)r)^d< \frac{1}{C}(r(R-r))^c\]
hold, and then $t_0$ will satisfy (\ref{EQ_0}).

Since $t_0>4(R-r)^{-1}=\epsilon_0^{-1}$, it follows that
\[t_k>\epsilon_k^{-1},\ \ \forall\ k\geq 0.\]

We will prove inductively that the bivectors
\[Z_k:=\pi_k-\pi_{|{r_k}}\in \mathfrak{X}^2(E_{r_{k}})\]
satisfy the inequalities ($a_k$) and ($b_k$)
\begin{equation*}
(a_k)\ \ \ \ \|Z_k\|_{s,r_k}\leq t_k^{-\alpha}, \ \ \ \ \  \ \  (b_k) \ \ \ \ \|Z_k\|_{p,r_k}\leq t_k^{\alpha}.
\end{equation*}
Since $t_0^{-\alpha}=\|Z_0\|_{p,R}$, $(a_0)$ and $(b_0)$ hold. Assuming that $(a_k)$ and $(b_k)$ hold for some
$k\geq 0$, we will show that condition $(C_k)$ holds (i.e.\ the procedure is well-defined up to step $k$) and also that
$(a_{k+1})$ and $(b_{k+1})$ hold.

First we give a bound for the norms of $X_k$ in terms of the norms of $Z_k$:
\begin{align}\label{EQ_X_n_l}
\|X_k\|_{m,r_k}&=\|S_{k}(h_1^{r_k}(Z_k))\|_{m,r_k}\leq C_mr^{-c_m}t_k^{1+l}\|h_1^{r_k}(Z_k)\|_{m-l,r_k}\leq\\
\nonumber &\leq C_mr^{-c_m}t_k^{1+l}\|Z_k\|_{m+s-l,r_k},\ \ \forall \ \ 0\leq l\leq m.
\end{align}
In particular, for $m=l$, we obtain
\begin{align}\label{EQ_X_n_n}
\|X_k\|_{m,r_k}&\leq C_mr^{-c_m}t_k^{1+m-\alpha}.
\end{align}
Since $\alpha>4$ and $t_k>\epsilon_k^{-1}$, this inequality implies that
\begin{equation}\label{EQ_X_12}
\|X_k\|_{1,r_k}\leq Cr^{-c}t_k^{2-\alpha}\leq Cr^{-c}t_0^{-1}t_k^{-1}<Cr^{-c}t_0^{-1}\epsilon_k.
\end{equation}
Since $t_0>Cr^{-c}/\theta$, we have that $\|X_k\|_{1,r_k}\leq \theta\epsilon_k$, and so by Lemma
\ref{Lemma_domain_of_flow} $(C_k)$ holds. Moreover, $X_k$ satisfies the inequalities from Lemma \ref{Lemma_size
of _the flow} and Lemma \ref{Lemma_tame_flow_pull_back}.

Next, we deduce an inequality for all norms $\|Z_{k+1}\|_{n,r_{k+1}}$ with $n\geq s$:
\begin{align}\label{EQ_Z_n}
\|Z_{k+1}&\|_{n,r_{k+1}}=\|\varphi_{X_k}^{*}(Z_k)+\varphi_{X_k}^{*}(\pi)-\pi\|_{n,r_{k+1}}\leq\\
\nonumber &\leq C_nr^{-c_n}(\|Z_k\|_{n,r_{k}}+\|X_k\|_{n+1,r_k}\|Z_k\|_{0,r_k}+\|X_k\|_{n+1,r_k}\|\pi\|_{n+1,r_k})\leq\\
\nonumber&\leq C_nr^{-c_n}(\|Z_k\|_{n,r_{k}}+\|X_k\|_{n+1,r_k})\leq C_nr^{-c_n}t_k^{s+2}\|Z_k\|_{n,r_k},
\end{align}
where we used Lemma \ref{Lemma_tame_flow_pull_back}, the inductive hypothesis, and inequality (\ref{EQ_X_n_l})
with $m=n+1$ and $l=s+1$. For $n=p$, using also that $s+2+\alpha\leq \frac{3}{2}\alpha-1$, this gives
$(b_{k+1})$:
\begin{align*}
\|Z_{k+1}\|_{p,r_{k+1}}&\leq Cr^{-c}t_k^{s+2+\alpha}\leq Cr^{-c}t_k^{\frac{3}{2}\alpha-1}\leq Cr^{-c}t_0^{-1}t_{k+1}^{\alpha}\leq t_{k+1}^{\alpha}.
\end{align*}

To prove $(a_{k+1})$, we write $Z_{k+1}=V_k+\varphi_{X_k}^{*}(U_{k})$, where
\[V_k:=\varphi_{X_k}^{*}(\pi)-\pi-\varphi_{X_k}^{*}([X_k,\pi]),\ \ U_{k}:=Z_{k}-[\pi,X_k].\]
Using Lemma \ref{Lemma_tame_flow_pull_back} and inequality (\ref{EQ_X_n_n}), we bound the two terms by
\begin{align}\label{EQ_V}
&\|V_k\|_{s,r_{k+1}}\leq Cr^{-c}\|\pi\|_{s+2,r_k}\|X_k\|_{0,r_k}\|X_k\|_{s+2,r_k}\leq  C r^{-c}t_k^{s+4-2\alpha},\\
\label{EQ_U}&\|\varphi_{X_k}^{*}(U_{k})\|_{s,r_{k+1}}\leq Cr^{-c}(\|U_k\|_{s,r_k}+\|U_k\|_{0,r_k}\|X_k\|_{s+1,r_{k}})\leq\\
\nonumber&\phantom{\|\varphi_{X_k}^{*}(U_{k})\|_{s,r_{k+1}}}\leq Cr^{-c}(\|U_k\|_{s,r_k}+t_k^{s+2-\alpha}\|U_k\|_{0,r_k}).
\end{align}
To compute the $C^s$-norm for $U_k$, we rewrite it as
\begin{align*}
U_k&=Z_k-[\pi,X_k]=[\pi,h_1^{r_k}(Z_k)]+h_2^{r_k}([\pi,Z_k])-[\pi,X_k]=\\
&=[\pi,(I-S_k)h_1^{r_k}(Z_k)]-\frac{1}{2}h_2^{r_k}([Z_k,Z_k]).
\end{align*}
By tameness of the Lie bracket, the first term can be bounded by
\begin{align*}
\|[\pi,(I-S_k)& h_1^{r_k}(Z_k)]\|_{s,r_k}\leq C r^{-c}\|(I-S_k)h_1^{r_k}(Z_k)\|_{s+1,r_k}\leq\\
&\leq Cr^{-c} t_k^{2-p+2s} \|h_1^{r_k}(Z_k)\|_{p-s,r_k}\leq C r^{-c}t_k^{2-p+2s} \|Z_k\|_{p,r_k}\leq\\
&\leq Cr^{-c} t_k^{2-p+2s+\alpha}=Cr^{-c} t_k^{-\frac{3}{2}\alpha-1},
\end{align*}
and using also the interpolation inequalities, for the second term we obtain
\begin{align*}
\|\frac{1}{2}h_2^{r_k}([Z_k,Z_k])&\|_{s,r_k}\leq C\|[Z_k,Z_k]\|_{2s,r_k}\leq Cr^{-c}\|Z_k\|_{0,r_k}\|Z_k\|_{2s+1,r_k}\leq\\
&\leq Cr^{-c}t_{k}^{-\alpha}\|Z_k\|_{s,r_k}^{\frac{p-(2s+1)}{p-s}}\|Z_k\|_{p,r_k}^{\frac{s+1}{p-s}}\leq Cr^{-c}t_k^{-\alpha(1+\frac{p-(3s+2)}{p-s})}.
\end{align*}
Since $-\alpha(1+\frac{p-(3s+2)}{p-s})\leq -\frac{3}{2}\alpha-1$, these two inequalities imply that
\begin{equation}\label{EQ_U_s}
\|U_k\|_{s,r_k}\leq Cr^{-c}t_k^{-\frac{3}{2}\alpha-1}.
\end{equation}
Using (\ref{EQ_X_n_n}), we bound the $C^0$-norm of $U_k$ by
\begin{equation}\label{EQ_U_0}
\|U_k\|_{0,r_{k}}\leq \|Z_k\|_{0,r_{k}}+\|[\pi,X_{k}]\|_{0,r_k}\leq t_k^{-\alpha}+Cr^{-c}\|X_{k}\|_{1,r_k}\leq Cr^{-c}t_{k}^{2-\alpha}.
\end{equation}
By (\ref{EQ_V}), (\ref{EQ_U}), (\ref{EQ_U_s}), (\ref{EQ_U_0}), and $s+4-2\alpha=-\frac{3}{2}\alpha-1$, $(a_{k+1})$
follows:
\begin{align*}
\|Z_{k+1}\|_{s,r_{k+1}}&\leq Cr^{-c}(t_k^{s+4-2\alpha}+t_k^{-\frac{3}{2}\alpha-1})\leq\\
&\leq Cr^{-c}t_k^{-\frac{3}{2}\alpha-1}\leq(r^{-c}C/t_0)t_{k}^{-\frac{3}{2}\alpha}\leq t_{k+1}^{-\alpha}.
\end{align*}
This finishes the induction.

Using (\ref{EQ_Z_n}), for every $n\geq 1$, we find a $k_n\geq 0$ such that
\[\|Z_{k+1}\|_{n,r_{k+1}}\leq t_k^{s+3}\|Z_{k}\|_{n,r_{k}}, \ \ \forall\ k\geq k_n.\]
Iterating this we obtain
\[t_k^{s+3}\|Z_{k}\|_{n,r_{k}}\leq (t_kt_{k-1}\ldots t_{k_n})^{s+3}\|Z_{k_n}\|_{n,r_{k_n}}.\]
On the other hand we have that
\[t_kt_{k-1}\ldots t_{k_n}=t_{k_n}^{1+\frac{3}{2}+\ldots+(\frac{3}{2})^{k-k_n}}\leq t_{k_n}^{2(\frac{3}{2})^{k+1-k_n}}=t_k^3.\]
Therefore, we obtain a bound valid for all $k>k_n$
\[\|Z_{k}\|_{n,r_{k}}\leq t_{k}^{2(s+3)}\|Z_{k_n}\|_{n,r_{k_n}}.\]

Consider now $m>s$ and denote by $n:=4m-3s$. Applying the interpolation inequalities, for $k> k_{n}$, we obtain
\begin{align*}
\|Z_k\|_{m,r_k}\leq& C_m r^{-c_m} \|Z_k\|_{s,r_k}^{\frac{n-m}{n-s}}\|Z_k\|_{n,r_k}^{\frac{m-s}{n-s}}= C_mr^{-c_m}\|Z_k\|_{s,r_k}^{\frac{3}{4}}\|Z_k\|_{n,r_k}^{\frac{1}{4}}\leq \\
\leq& C_m r^{-c_m}t_k^{-\alpha\frac{3}{4}+2(s+3)\frac{1}{4}}\|Z_{k_{n}}\|_{n,r_{k_{n}}}^{\frac{1}{4}}=C_m r^{-c_m}t_k^{-(s+6)}\|Z_{k_{n}}\|_{n,r_{k_{n}}}^{\frac{1}{4}}.
\end{align*}
Using also inequality (\ref{EQ_X_n_l}), for $l=s$, we obtain
\begin{align*}
\|X_{k}\|_{m,r_k}\leq C_mr^{-c_m}t_k^{s+1}\|Z_k\|_{m,r_k}\leq  t_k^{-5}\left( C_mr^{-c_m}\|Z_{k_{n}}\|_{n,r_{k_{n}}}^{\frac{1}{4}}\right).
\end{align*}
This shows that the series $\sum_{k\geq 0}\|X_k\|_{m,r_k}$ converges for all $m$. By Lemma \ref{Lemma_size of _the
flow}, also $\sum_{k\geq 0}\mathrm{d}(\varphi_{X_k})_{m,r_{k+1}}$ converges for all $m$ and, moreover, by
(\ref{EQ_X_12}), we have that
\begin{equation*}
\sigma_1:=\sum_{k\geq 1}\mathrm{d}(\varphi_{X_k})_{1,r_{k+1}}\leq Cr^{-c}\sum_{k\geq 1}\|X_k\|_{1,r_k}\leq Cr^{-c}t_0^{-4}\sum_{k\geq 1}\epsilon_k\leq t_0^{-3}.
\end{equation*}
So we may assume that $\sigma_1\leq \theta$ and $\mathrm{d}(\varphi_{X_k})_{1,r_{k+1}}<1$. Then by applying
Lemma \ref{Lemma_convergent_embbedings} we conclude that the sequence $\psi_{k|r}$ converges uniformly in all
$C^n$-norms to a map $\psi:E_r\to E_R$ in $\mathcal{U}_r$ which satisfies
\begin{equation*}
\mathrm{d}(\psi)_{1,r}\leq e^{Cr^{-c}\sigma_1}Cr^{-c}\sigma_1\leq e^{t_0^{-2}}t_0^{-2}\leq Ct_0^{-2}\leq t_0^{-1}.
\end{equation*}
So (\ref{EQ_continuity}) holds, and we can also assume that $\mathrm{d}(\psi)_{1,r}<\theta$, which, by Lemma
\ref{Lemma_embedding}, implies that $\psi$ is a local diffeomorphism. Since $\psi_{k|r}$ converges in the
$C^1$-topology to $\psi$ and $\psi_k^*(\widetilde{\pi})=(d\psi_k)^{-1}(\widetilde{\pi}_{\psi_k})$, it follows that
$\psi_k^*(\widetilde{\pi})_{|r}$ converges in the $C^0$-topology to $\psi^*(\widetilde{\pi})$. On the other hand,
$Z_{k|r}=\psi_k^*(\widetilde{\pi})_{|r}-\pi_{|r}$ converges to $0$ in the $C^0$-norm, hence
$\psi^*(\widetilde{\pi})=\pi_{|r}$. So $\psi$ is a Poisson map and a local diffeomorphism between
\[\psi:(E_r,\pi_{|r})\rmap (E_R,\widetilde{\pi}).\]

For \textbf{Procedure P$_1$}, as noted before the proposition, the first jet of $\varphi_{X_k}$ is that of the identity
along $S$. This clearly holds also for $\psi_k$, and since $\psi_{k|r}$ converges to $\psi$ in the $C^1$-topology, we
have that $\psi$ is also the identity along $S$ up to first order.
\end{proof}

Now we are ready to finish the proof of Theorem \ref{Theorem_FOUR}.

\subsubsection*{Proof of part (a) of Theorem \ref{Theorem_FOUR}}

We have to check that the rigidity property from Definition \ref{Definition_CpC1} holds. Consider
$U:=\mathrm{int}(E_{\rho})$ for some ${\rho}\in (0,1)$, and let $O\subset U$ be an open such that $S\subset
O\subset \overline{O}\subset U$. Let $r<R$ be such that $O\subset E_r\subset E_R\subset U$. With $d$ and $\delta$
from Proposition \ref{Proposition_technical}, we let
\[\mathcal{V}_O:=\{W\in\mathfrak{X}^2(U): \|W_{|R}-\pi_{|R}\|_{p,R}<\delta(r(R-r))^d\}.\]
For $\widetilde{\pi}\in \mathcal{V}_O$, define $\psi_{\widetilde{\pi}}$ to be the restriction to $\overline{O}$ of the
map $\psi$ obtained by applying \textbf{Procedure P$_0$} to $\widetilde{\pi}_{|R}$. Then $\psi$ is a Poisson
diffeomorphism $(O,\pi_{|O})\to (U,\widetilde{\pi})$, and by (\ref{EQ_continuity}), the assignment
$\widetilde{\pi}\mapsto \psi$ has the required continuity property.

\subsubsection*{Proof of part (b) of Theorem \ref{Theorem_FOUR}}

Consider $\widetilde{\pi}$ a Poisson structure on some neighborhood of $S$ with
$j^1_{|S}\widetilde{\pi}=j^1_{|S}\pi$. First we show that $\pi$ and $\widetilde{\pi}$ are formally Poisson
diffeomorphic around $S$. For this we have to check that $H^2(A_S,\mathcal{S}^k(\nu_S^*))=0$. The Lie groupoid
$\mathcal{G}_{|S}\rightrightarrows S$ integrates $A_S$ and is $s$-connected. Since $\nu_S^*\subset A_S$ is an
ideal, by Lemma \ref{Lemma_integrating_ideals}, the action of $A_S$ on $\nu_S^*$ (hence also on
$\mathcal{S}^k(\nu_S^*)$) also integrates to $\mathcal{G}_{|S}$. Since $\mathcal{G}_{|S}$ has compact $s$-fibers
with vanishing $H^2$, the Tame Vanishing Lemma implies that $H^2(A_S, \mathcal{S}^k(\nu_S^*))=0$. So we can
apply Theorem \ref{Theorem_THREE} to conclude that there exists a diffeomorphism $\varphi$ between open
neighborhoods of $S$, which is the identity on $S$ up to first order, such that
\[j^{\infty}_{|S}\varphi^*(\widetilde{\pi})=j^{\infty}_{|S}\pi.\]

Consider $R\in(0,1)$ such that $\varphi^*(\widetilde{\pi})$ is defined on $E_R$. Using the Taylor expansion up to
order $2d+1$ around $S$ for the bivector $\pi-\varphi^*(\widetilde{\pi})$ and its partial derivatives up to order $p$,
we find a constant $M>0$ such that
\[\|\varphi^*(\widetilde{\pi})_{|r}-\pi_{|r}\|_{p,r}\leq M r^{2d+1},\ \ \forall\ 0<r<R.\]
If we take $r<2^{-d}\delta/M $, we obtain that $\|\varphi^*(\widetilde{\pi})_{|r}-\pi_{|r}\|_{p,r}<\delta(r(r-r/2))^d$.
So we can apply Proposition \ref{Proposition_technical}, and \textbf{Procedure P$_1$} produces a Poisson
diffeomorphism
\[\tau:(E_{r/2},\pi_{|r/2})\rmap (E_r,\varphi^*(\widetilde{\pi})_{|r}),\]
which is the identity up to first order along $S$. We obtain (b) with $\psi=\varphi\circ\tau$.

\begin{remark}\rm\label{Remark_SCI}
As mentioned already in section \ref{introduction_rigi}, Conn's proof has been formalized in \cite{Miranda,Monnier}
into an abstract Nash Moser normal form theorem, and it is likely that one could use Theorem 6.8 \cite{Miranda} to
prove partially our rigidity result. Nevertheless, the continuity assertion, which is important in applications (see
chapter \ref{ChDef}), is not a consequence of this result. There are also several technical reasons why we cannot apply
\cite{Miranda}: we need the size of the $C^p$-open to depend polynomially on $r^{-1}$ and $(R-r)^{-1}$, because
we use a formal linearization argument (this dependence is not given in \emph{loc.cit.}); to obtain diffeomorphisms
that fix $S$, we work with vector fields vanishing along $S$ up to first order, and it is unlikely that this Fr\'echet space
admits smoothing operators of degree $0$ (in \emph{loc.cit.} this is the overall assumption); for the inequalities in
Lemma \ref{Lemma_composition} we need special norms (indexed also by ``$\delta$'') for the embeddings (these are
not considered in \emph{loc.cit.}).
\end{remark}

\clearpage \pagestyle{plain}

\chapter{Deformations of the Lie-Poisson sphere of a compact semisimple Lie algebra}\label{ChDef}

\pagestyle{fancy}

\fancyhead[CE]{Chapter \ref{ChDef}} 
\fancyhead[CO]{Deformations of the Lie-Poisson sphere of a compact semisimple Lie algebra}

The linear Poisson structure $(\mathfrak{g}^*,\pi_{\mathfrak{g}})$ corresponding to a compact semisimple Lie
algebra $\mathfrak{g}$ induces a Poisson structure $\pi_{\mathbb{S}}$ on the unit sphere
$\mathbb{S}(\mathfrak{g}^*)\subset \mathfrak{g}^*$. As a surprising application of Theorem \ref{Theorem_FOUR},
in this chapter we compute the moduli space of Poisson structures on $\mathbb{S}(\mathfrak{g}^*)$ around
$\pi_{\mathbb{S}}$ (Theorem \ref{Theorem_FIVE}). This is the first explicit computation of a Poisson moduli space in
dimension greater or equal than three around a degenerate (i.e.\ not symplectic) Poisson bivector. The content of this
chapter is available as a preprint at \cite{MarDef}.

\section{Statement of Theorem \ref{Theorem_FIVE}}

Let $(\mathfrak{g},[\cdot,\cdot])$ be a semisimple Lie algebra of compact type. We consider an inner product on
$\mathfrak{g}^*$ which is not just $\mathfrak{g}$-invariant, but also $Aut(\mathfrak{g})$-invariant (e.g. induced by
the Killing form). The corresponding unit sphere around the origin, denoted by $\mathbb{S}(\mathfrak{g}^*)$, is a
Poisson submanifold, and as such it inherits a Poisson structure
$\pi_{\mathbb{S}}:=\pi_{\mathfrak{g}|\mathbb{S}(\mathfrak{g}^*)}$. We will call the Poisson manifold
\[(\mathbb{S}(\mathfrak{g}^*),\pi_{\mathbb{S}})\]
the \textbf{Lie-Poisson sphere}\index{Lie-Poisson sphere} corresponding to $\mathfrak{g}$. Lie algebra
automorphisms of $\mathfrak{g}$ restrict to Poisson diffeomorphisms of $\pi_{\mathbb{S}}$. The inner
automorphisms preserve the coadjoint orbits and so, they act trivially on the space of Casimir functions\index{Casimir
function}. Therefore $Out(\mathfrak{g})$, the group of outer automorphisms of $\mathfrak{g}$, acts naturally on the
space of Casimir functions, which we denote by
\[\mathfrak{Casim}(\mathbb{S}(\mathfrak{g}^*),\pi_{\mathbb{S}}).\]

The main result of this chapter is the following description of Poisson structures on $\mathbb{S}(\mathfrak{g}^*)$
near $\pi_{\mathbb{S}}$.\index{linear Poisson structure}

\begin{mtheorem}\label{Theorem_FIVE}
For the Lie-Poisson sphere $(\mathbb{S}(\mathfrak{g}^*),\pi_{\mathbb{S}})$, corresponding to a compact
semisimple Lie algebra $\mathfrak{g}$, the following hold:
\begin{enumerate}[(a)]
\item There exists a $C^p$-open $\mathcal{W}\subset \mathfrak{X}^2(\mathbb{S}(\mathfrak{g}^*))$ around $\pi_{\mathbb{S}}$, such that every Poisson structure in $\mathcal{W}$ is isomorphic to one
of the form $f\pi_{\mathbb{S}}$, where $f$ is a positive Casimir, by a diffeomorphism isotopic to the identity.
\item For two positive Casimirs $f$ and $g$, the Poisson manifolds $(\mathbb{S}(\mathfrak{g}^*),f\pi_{\mathbb{S}})$ and $(\mathbb{S}(\mathfrak{g}^*),g\pi_{\mathbb{S}})$ are isomorphic
precisely when $f$ and $g$ are related by an outer automorphism of $\mathfrak{g}$.
\end{enumerate}
\end{mtheorem}

The open $\mathcal{W}$ will be constructed such that it contains all Poisson structures of the form
$f\pi_{\mathbb{S}}$, with $f$ a positive Casimir. Therefore, the map $F\mapsto e^{F}\pi_{\mathbb{S}}$ induces a
bijection between the space
\begin{equation*}
\mathfrak{Casim}(\mathbb{S}(\mathfrak{g}^*),\pi_{\mathbb{S}})/Out(\mathfrak{g})
\end{equation*}
and an open around $\pi_{\mathbb{S}}$ in the Poisson moduli space of $\mathbb{S}(\mathfrak{g}^*)$. Using
classical invariant theory, we show in section \ref{Section_the_space_of_Casimirs} that this space is isomorphic to
\[C^{\infty}(\overline{B})/Out(\mathfrak{g}),\]
where $B\subset \mathbb{R}^{l-1}$ ($l=\mathrm{rank}(\mathfrak{g})$) is a bounded open which is invariant under
a linear action of
$Out(\mathfrak{g})$ on $\mathbb{R}^{l-1}$.\\

Recall from subsection \ref{SSPoissonCoho} that the space of Casimir functions of a Poisson manifold $(M,\pi)$ is the
$0$-th Poisson cohomology group $H^0_{\pi}(M)$.\index{cohomology, Poisson} Also, that the second Poisson
cohomology group $H^2_{\pi}(M)$ has the heuristic interpretation of being the ``tangent space'' to the Poisson
moduli space at $\pi$. As our result suggests, for the Lie-Poisson sphere we have that multiplication with
$[\pi_{\mathbb{S}}]$ induces an isomorphism (see Proposition \ref{Proposition_Poisson_cohomology})
\[\mathfrak{Casim}(\mathbb{S}(\mathfrak{g}^*),\pi_{\mathbb{S}})\cong H^2_{\pi_{\mathbb{S}}}(\mathbb{S}(\mathfrak{g}^*)).\]

There are only few descriptions, in the literature, of opens in the Poisson moduli space of a compact manifold. We recall
below two such results.

For a compact symplectic manifold $(M,\omega)$, every Poisson structure $C^0$-close to $\omega^{-1}$ is
symplectic as well. The Moser argument shows that two symplectic structures in the same cohomology class, and
which are close enough to $\omega$, are symplectomorphic by a diffeomorphism isotopic to the identity. This implies
that the map $\pi\mapsto [\pi^{-1}]\in H^2(M)$ induces a bijection between an open in the space of all Poisson
structures modulo diffeomorphisms isotopic to the identity and an open in $H^2(M)$. Also the heuristic prognosis
holds, since $H^2(M)\cong H^2_{\omega^{-1}}(M)$. In general it is difficult to say more, that is, to determine
whether two symplectic structures, different in cohomology, are symplectomorphic. In Corollary
\ref{Corollary_moduli_on_coadjoint} we achieve this for the maximal coadjoint orbits of a compact semisimple Lie
algebra.

In \cite{Radko}, Radko obtains a description of the moduli space of \emph{topologically stable Poisson structures} on
a compact oriented surface $\Sigma$. These are Poisson structures $\pi\in \mathfrak{X}^2(\Sigma)$ that intersect the
zero section of $\Lambda^2 T\Sigma$ transversally, and so, they form a dense $C^1$-open in the space of all Poisson
structures. The moduli space decomposes as a union of finite dimensional manifolds (of different
dimensions), and its tangent space at $\pi$ is precisely $H^2_{\pi}(\Sigma)$. The fact that $\Sigma$ is two-dimensional facilitates the study, since any bivector on $\Sigma$ is Poisson.\\

The main difficulty when studying deformations of Poisson structures on compact manifolds (in contrast e.g.\ to
complex structures) is that the Poisson complex fails to be elliptic, unless the structure is symplectic. Therefore, in
general, $H^2_{\pi}(M)$ and the Poisson moduli space are infinite dimensional. This is also the case for the
Lie-Poisson sphere of a compact semisimple Lie algebra $\mathfrak{g}$, except for
$\mathfrak{g}=\mathfrak{su}(2)$. The Lie algebra $\mathfrak{su}(2)$ is special also because it is the only one for
which the Lie-Poisson sphere is symplectic (thus the result follows from Moser's theorem). Moreover, it is the only one
for which the Lie-Poisson sphere is an integrable Poisson manifold (in the sense of section
\ref{Section_symplectic_groupoids}).

\section{Proof of part $(a)$ of Theorem \ref{Theorem_FIVE}}

Throughout this chapter, we fix a compact semisimple Lie algebra $(\mathfrak{g},[\cdot,\cdot])$ and an
$Aut(\mathfrak{g})$-invariant inner product on $\mathfrak{g}$ (e.g.\ the negative of the Killing form), hence also on
$\mathfrak{g}^*$. We denote by $G$ a 1-connected Lie group integrating $\mathfrak{g}$. Then $G$ is compact and
$H^2(G)=0$.

The linear Poisson manifold $(\mathfrak{g}^*,\pi_{\mathfrak{g}})$ is integrable by the action Lie groupoid $G\ltimes
\mathfrak{g}^*$ (see subsection \ref{Subsection_examples_of_symplectic_groupoids}). Since the $s$-fibers of
$G\ltimes \mathfrak{g}^*$ are diffeomorphic to $G$, we can apply Theorem \ref{Theorem_FOUR} and conclude that
$\pi_{\mathfrak{g}}$ is $C^p$-$C^1$-rigid around any compact Poisson submanifold.

For $f\in C^{\infty}(\mathbb{S}(\mathfrak{g}^*))$, with $f>0$, consider the following codimension-one sphere $S_f$
in $\mathfrak{g}^*\backslash\{0\}$
\[S_f:=\left\{\frac{1}{f(\xi)}\xi \ |\  \xi\in \mathbb{S}(\mathfrak{g}^*)\right\}.\]
We denote the map parameterizing $S_f$ by
\[\varphi_f: \mathbb{S}(\mathfrak{g}^*)\diffto S_f, \ \ \varphi_f(\xi):=\xi/f(\xi).\]
The spheres of the type $S_f$ form a $C^1$-open in the space of all (unparameterized) spheres. Namely, every sphere
$S\subset \mathfrak{g}^*\backslash\{0\}$ for which the map
\[\mathrm{pr}:S\rmap \mathbb{S}(\mathfrak{g}^*),\ \ \xi\mapsto \frac{1}{|\xi|}\xi\]
is a diffeomorphism, is of the form $S_f$ for some positive function $f$ on $\mathbb{S}(\mathfrak{g}^*)$.

\begin{lemma}\label{Lemma_B}
The sphere $S_f$ is a Poisson submanifold if and only $f$ is a Casimir. In this case, the following map is a Poisson
diffeomorphism
\[\widetilde{\varphi}_f:(\mathbb{S}(\mathfrak{g}^*)\times\mathbb{R}_{+}, tf\pi_{\mathbb{S}}) \diffto (\mathfrak{g}^*\backslash\{0\},\pi_{\mathfrak{g}}), \ (\xi,t)\mapsto \frac{1}{tf(\xi)}\xi.\]
\end{lemma}
\begin{proof}
Compact Poisson submanifolds of $\mathfrak{g}^*$ are the same as $G$-invariant submanifolds, and Casimirs of
$\pi_{\mathbb{S}}$ are the same as $G$-invariant functions on $\mathbb{S}(\mathfrak{g}^*)$. This implies the first
part. For the second part, it suffices to check that $\widetilde{\varphi}_f^*$ preserves the Poisson bracket on
elements in $\mathfrak{g}\subset C^{\infty}(\mathfrak{g}^*)$. Using that Casimirs go inside the bracket, and that
$\mathbb{S}(\mathfrak{g}^*)$ is a Poisson submanifold, it is a straightforward computation:
\begin{align*}
\widetilde{\varphi}_{f}^*(\{X,Y\})=&\frac{1}{tf}\{X,Y\}_{|\mathbb{S}(\mathfrak{g}^*)}=tf\{\frac{1}{tf}X_{|\mathbb{S}(\mathfrak{g}^*)},\frac{1}{tf}Y_{|\mathbb{S}(\mathfrak{g}^*)}\}_{|\mathbb{S}(\mathfrak{g}^*)}=\\
&=tf\{\widetilde{\varphi}_f^*(X),\widetilde{\varphi}_f^*(Y)\}_{|\mathbb{S}(\mathfrak{g}^*)}.\qedhere
\end{align*}
\end{proof}

Now, we are ready to prove the first part of Theorem \ref{Theorem_FIVE}.
\begin{proof}[Proof of part (a) of Theorem \ref{Theorem_FIVE}]
For every positive Casimir $f$ we construct a $C^p$-open $\mathcal{W}_{f}\subset
\mathfrak{X}^2(\mathbb{S}(\mathfrak{g}^*))$ containing $f\pi_{\mathbb{S}}$, such that every Poisson structure in
$\mathcal{W}_f$ is isomorphic to one of the form $g\pi_{\mathbb{S}}$, for some positive Casimir $g$, and moreover,
by a diffeomorphism isotopic to the identity. This will imply that the $C^p$-open
$\mathcal{W}:=\cup_{f}\mathcal{W}_f$ satisfies the conclusion.

For the compact Poisson submanifold $S_f$, consider opens $S_f\subset O\subset U$, as in Definition
\ref{Definition_CpC1} of $C^p$-$C^1$-rigidity, with $0\notin U$. Denote by $\mathcal{U}_f$ the set of functions
$\chi\in C^{\infty}(\overline{O},\mathfrak{g}^*)$ satisfying $0\notin \chi(S_f)$, and for which the map
\[\mathrm{pr}\circ \chi \circ \varphi_{f}:\mathbb{S}(\mathfrak{g}^*)\rmap \mathbb{S}(\mathfrak{g}^*)\]
is a diffeomorphism isotopic to the identity. The first condition is $C^0$-open and the second is $C^1$-open (see
subsection \ref{subsection_CpC1isotopies}). For the inclusion $\textrm{Id}_{\overline{O}}$ of $\overline{O}$ in
$\mathfrak{g}^*$, we have that $\mathrm{pr}\circ \textrm{Id}_{\overline{O}}\circ \varphi_f=\textrm{Id}$, thus
$\mathcal{U}_f$ is $C^1$-neighborhood of $\textrm{Id}_{\overline{O}}$. By continuity of the map $\psi$ from
Definition \ref{Definition_CpC1}, there exists a $C^p$-neighborhood $\mathcal{V}_f\subset \mathfrak{X}^2(U)$ of
$\pi_{\mathfrak{g}|U}$, such that $\psi_{\widetilde{\pi}}\in \mathcal{U}_f$, for every Poisson structure
$\widetilde{\pi}$ in $\mathcal{V}_f$. We define the $C^p$-open $\mathcal{W}_f$ as follows
\[\mathcal{W}_f:=\{W\in\mathfrak{X}^2(\mathbb{S}(\mathfrak{g}^*)) | \widetilde{\varphi}_{f,*}(t W)_{|U}\in \mathcal{V}_f\}.\]
By Lemma \ref{Lemma_B}, we have that $\widetilde{\varphi}_{f,*}(t
f\pi_{\mathbb{S}})_{|U}=\pi_{\mathfrak{g}|U}$, thus $f\pi_{\mathbb{S}}\in \mathcal{W}_f$. Let $\overline{\pi}$ be
a Poisson structure in $\mathcal{W}_f$. Then for $\widetilde{\pi}:=\widetilde{\varphi}_{f,*}(t\overline{\pi})_{|U}\in
\mathcal{V}_f$, we have that $\psi_{\widetilde{\pi}}\in\mathcal{U}_f$ is a Poisson map between
\[\psi_{\widetilde{\pi}}:(O,\widetilde{\pi}_{|O})\rmap (U,\pi_{\mathfrak{g}|U}).\]
By the discussion before Lemma \ref{Lemma_B}, the fact that the map $\mathrm{pr}\circ
\psi_{\widetilde{\pi}}\circ\varphi_f$ is a diffeomorphism, implies that $\psi_{\widetilde{\pi}}(S_f)=S_g$, for some
$g>0$. Clearly, $(\mathbb{S}(\mathfrak{g}^*)\times \{1\},\overline{\pi})$ is a Poisson submanifold of
$(\mathbb{S}(\mathfrak{g}^*)\times \mathbb{R}_{+},t\overline{\pi})$, therefore also
$S_f=\widetilde{\varphi}_f(\mathbb{S}(\mathfrak{g}^*)\times \{1\})$ is a Poisson submanifold of
$(O,\widetilde{\pi}_{|O})$, and since $\psi_{\widetilde{\pi}}$ is a Poisson map, we have that also
$S_g=\psi_{\widetilde{\pi}}(S_f)$ is a Poisson submanifold of $(\mathfrak{g}^*,\pi_{\mathfrak{g}})$. So by Lemma
\ref{Lemma_B} $g$ is a Casimir and
\[\varphi_g:(\mathbb{S}(\mathfrak{g}^*),g\pi_{\mathbb{S}})\diffto (S_g,\pi_{\mathfrak{g}|S_g})\]
is a Poisson diffeomorphism. Therefore also the map
\[\varphi_g^{-1}\circ\psi_{\widetilde{\pi}}\circ\varphi_f:(\mathbb{S}(\mathfrak{g}^*),\overline{\pi})\diffto (\mathbb{S}(\mathfrak{g}^*),g\pi_{\mathbb{S}}),\]
is a Poisson diffeomorphism. This map is isotopic to the identity, because $\varphi_g^{-1}=\mathrm{pr}_{|S_g}$, and
by construction $\mathrm{pr}\circ \psi_{\widetilde{\pi}} \circ \varphi_{f}$ is isotopic to the identity.
\end{proof}

\subsection{The Poisson cohomology of the Lie-Poisson sphere}

Recall from subsection \ref{SSPoissonCoho} that\index{cohomology, Poisson}
\[\mathfrak{Casim}(\mathbb{S}(\mathfrak{g}^*),\pi_{\mathbb{S}})= H^0_{\pi_{\mathbb{S}}}(\mathbb{S}(\mathfrak{g}^*)).\]
This group will be described in section \ref{Section_the_space_of_Casimirs} more explicitly. Recall also that
$H^2_{\pi_{\mathbb{S}}}(\mathbb{S}(\mathfrak{g}^*))$ has the heuristic interpretation of being the ``tangent
space'' to the Poisson moduli space at $\pi_{\mathbb{S}}$. Theorem \ref{Theorem_FIVE} shows that all deformations
of $\pi_{\mathbb{S}}$ are, modulo diffeomorphisms, of the form $\pi^t=e^{tF}\pi_{\mathbb{S}}$, where $F\in
\mathfrak{Casim}(\mathbb{S}(\mathfrak{g}^*),\pi_{\mathbb{S}})$. As expected, this is reflected also at the level of
the Poisson cohomology.
\begin{proposition}\label{Proposition_Poisson_cohomology}
The first Poisson cohomology groups of the Lie-Poisson sphere satisfy:
$H^1_{\pi_{\mathbb{S}}}(\mathbb{S}(\mathfrak{g}^*))=0$ and
$H^0_{\pi_{\mathbb{S}}}(\mathbb{S}(\mathfrak{g}^*))\cong
H^2_{\pi_{\mathbb{S}}}(\mathbb{S}(\mathfrak{g}^*))$ via the map
\[H^0_{\pi_{\mathbb{S}}}(\mathbb{S}(\mathfrak{g}^*))\diffto H^2_{\pi_{\mathbb{S}}}(\mathbb{S}(\mathfrak{g}^*)), \ f\mapsto [f\pi_{\mathbb{S}}].\]
\end{proposition}

\begin{proof}
The invariant open $U:=\mathfrak{g}^*\backslash\{0\}$ is integrable by the Lie groupoid $G\ltimes U$. This groupoid
is proper and has 2-connected $s$-fibers. Therefore, by the results of \cite{Cra} (discussed in section
\ref{Section_COHOMOLOGY}), or by our Theorem \ref{Theorem_tame_vanishing}, we have that
\[H^1_{\pi_{\mathfrak{g}|U}}(U)=0, \ \ H^2_{\pi_{\mathfrak{g}|U}}(U)=0.\]
Using the Poisson diffeomorphism $\widetilde{\varphi}_1$ from Lemma \ref{Lemma_B}, we identify
$(U,\pi_{\mathfrak{g}|U})$ with the Poisson manifold $(\mathbb{S}(\mathfrak{g}^*)\times\mathbb{R}_+,
t\pi_{\mathbb{S}}).$

Let $X\in \mathfrak{X}^1(\mathbb{S}(\mathfrak{g}^*))$ be such that $[\pi_{\mathbb{S}},X]=0$. If we regard $X$
as a vector field on $\mathbb{S}(\mathfrak{g}^*)\times\mathbb{R}_{+}$ it is also a cocycle. Since
$H^1_{\pi_{\mathfrak{g}|U}}(U)=0$, there is a function $f_t\in
C^{\infty}(\mathbb{S}(\mathfrak{g}^*)\times\mathbb{R}_+)$ such that $[t\pi_{\mathbb{S}},f_t]=X$. For $t=1$,
this shows that $X$ is a coboundary. Hence $H^1_{\pi_{\mathbb{S}}}(\mathbb{S}(\mathfrak{g}^*))=0$.

Let $W\in \mathfrak{X}^2(\mathbb{S}(\mathfrak{g}^*))$ be such that $[\pi_{\mathbb{S}},W]=0$. As before, there
exists vector fields $V_t=X_t+f_t\frac{\partial}{\partial t}$ such that $W=[t\pi_{\mathbb{S}},V_t]$. Hence,
\[W=t[\pi_{\mathbb{S}},X_t]+t[\pi_{\mathbb{S}},f_t]\wedge \frac{\partial}{\partial t}-f_t\pi_{\mathbb{S}}.\]
For $t=1$, this shows that $[\pi_{\mathbb{S}},f_1]=0$ (i.e.\ $f_1$ is a Casimir), and so $[W]$ and
$-[f_1\pi_{\mathbb{S}}]$ represent the same class. Therefore, multiplication by $[\pi_{\mathbb{S}}]$ is onto:
\[H^2_{\pi_{\mathbb{S}}}(\mathbb{S}(\mathfrak{g}^*))=H^0_{\pi_{\mathbb{S}}}(\mathbb{S}(\mathfrak{g}^*))[\pi_{\mathbb{S}}].\]

To check injectivity, let $f\in H^0_{\pi_{\mathbb{S}}}(\mathbb{S}(\mathfrak{g}^*))$ be so that
$[\pi_{\mathbb{S}},X]=f\pi_{\mathbb{S}}$, for some $X\in \mathfrak{X}^1(\mathbb{S}(\mathfrak{g}^*))$. Then we
have that
\[[t\pi_{\mathbb{S}},1/tX+f\frac{\partial}{\partial t}]=[\pi_{\mathbb{S}},X]-f\pi_{\mathbb{S}}=0.\]
So the vector field $V_t=1/tX+f\frac{\partial}{\partial t}$ is a cocycle in $H^1_{\pi_{\mathfrak{g}|U}}(U)$. Since
this group is trivial, there exists a smooth function $g_t$ such that $V_t=[\pi_{\mathbb{S}},g_t]$. This implies that
$f=0$ and finishes the proof.
\end{proof}

\section{Some standard Lie theoretical results}\label{Section_Lie_results}
In this section we recall some results on compact semisimple Lie algebras, which will be used in the proof of part $(b)$
of Theorem \ref{Theorem_FIVE}. Most of these can be found in standard textbooks like
\cite{Dixmier,DK,Humphreys,Knapp}.

\subsection{The coadjoint action and its symplectic orbits}

Recall from subsection \ref{SSExamplesPoisson} that the symplectic\index{coadjoint orbit} leaf of the linear Poisson
manifold $(\mathfrak{g}^*,\pi_{\mathfrak{g}})$ through $\xi\in\mathfrak{g}^*$ is the coadjoint orbit through $\xi$;
we will denote it by $(O_{\xi},\omega_{\xi})$. Denote the stabilizer of $\xi$ by $G_{\xi}\subset G$, and its Lie algebra
by $\mathfrak{g}_{\xi}\subset \mathfrak{g}$. Recall also (\ref{EQ_pull_back_coadjoint}) that the pullback of
$\omega_{\xi}$ via the map
\[p:G\rmap O_{\xi}, \ \ g\mapsto Ad_{g^{-1}}^{*}(\xi),\]
is given by
\[p^*(\omega_{\xi})=-d\xi^l,\]
where $\xi^{l}$ is the left invariant extension of $\xi$ to $G$.

The adjoint\index{coadjoint action} and the coadjoint representations are isomorphic, the invariant metric gives an
intertwining map between the two. We restate here some standard results about the adjoint action in terms of the
coadjoint (as a reference see section 3.2 in \cite{DK}). We are interested especially in the set
$\mathfrak{g}^{*\mathrm{reg}}$ of \textbf{regular} elements. An element $\xi\in\mathfrak{g}^*$ is regular if and
only if it satisfies any of the following equivalent conditions:
\begin{itemize}
\item $\mathfrak{g}_{\xi}$ is a maximal abelian subalgebra;
\item the leaf $O_{\xi}$ has maximal dimension among all leaves;
\item $G_{\xi}$ is a maximal torus in $G$.
\end{itemize}
We fix $T\subset G$ a maximal torus and denote by $\mathfrak{t}$ its Lie algebra. We regard $\mathfrak{t}^*$ as a
subspace in $\mathfrak{g}^*$, via the identification
\[\mathfrak{t}^*=\{\xi\in\mathfrak{g}^* | \mathfrak{t}\subset\mathfrak{g}_{\xi}\}\subset\mathfrak{g}^*.\]
Consider $\mathfrak{t}^{*\mathrm{reg}}:=\mathfrak{t}^{*}\cap \mathfrak{g}^{*\mathrm{reg}}$, the regular
part of $\mathfrak{t}^*$. Then $\mathfrak{t}^{*\mathrm{reg}}$ is the union of the open Weyl
chambers\index{Weyl chamber}. Fix $\mathfrak{c}$ such a chamber. The structure of
$\mathfrak{g}^{*\mathrm{reg}}$ is described by the equivariant diffeomorphism (Proposition 3.8.1 \cite{DK})
\begin{equation}\label{EQ_regular_part}
\Psi:G/T\times \mathfrak{c} \diffto \mathfrak{g}^{*\mathrm{reg}}, \ \ \Psi([g],\xi)= Ad_{g^{-1}}^{*}(\xi).
\end{equation}

\subsection{Roots and dual roots}

Let $\Phi\subset i\mathfrak{t}^*$ be the root system corresponding to $\mathfrak{t}$. Denote the group of
automorphisms of $\Phi$ by\index{root}\index{dual root}\index{root system}
\[Aut(\Phi):=\{f\in Gl(i\mathfrak{t}^*) | f(\Phi)=\Phi\}.\]
For $\alpha\in \Phi$ let $\mathfrak{g}_{\alpha}\subset \mathfrak{g}\otimes\mathbb{C}$ be the corresponding root
space. The \textbf{root dual} to $\alpha\in \Phi$ is the element $\alpha^{\vee}\in
[\mathfrak{g}_{-\alpha},\mathfrak{g}_{\alpha}]$ satisfying $\alpha(\alpha^{\vee})=2$. The set of dual roots forms
the \textbf{dual root system} $\Phi^{\vee}\subset i\mathfrak{t}$. For all $\alpha,\beta\in \Phi$, we have that
$\alpha(\beta^{\vee}) \in\mathbb{Z}$, and these numbers are called the \textbf{Cartan integers}\index{Cartan
integers} (Theorem 3.10.2 \cite{DK}). One can also define $\Phi^{\vee}$ in terms of only the abstract root system
$\Phi$: namely, $\alpha^{\vee}$ is the unique element in $i\mathfrak{t}$ for which the map
\begin{equation}\label{EQ_symmetries}
s_{\alpha}:i\mathfrak{t}^*\rmap i\mathfrak{t}^*, \ \ s_{\alpha}(\lambda)=\lambda - \lambda(\alpha^{\vee})\alpha
\end{equation}
sends $\alpha$ to $-\alpha$ and preserves the root system (Lemma 9.1 \cite{Humphreys}). With this description, one
can define the double dual $\Phi^{\vee\vee}$. This turns out to be $\Phi$.
\begin{lemma}\label{Lemma_doubble_dual}
We have that $\alpha^{\vee\vee}=\alpha$, and that the following map is a group isomorphism
\[Aut(\Phi)\diffto Aut(\Phi^{\vee}), \ \ f\mapsto (f^{-1})^*.\]
\end{lemma}
\begin{proof}
By Lemma 9.2 \cite{Humphreys}, we have that $f\circ s_{\beta}\circ f^{-1}=s_{f(\beta)}$, for all $f\in Aut(\Phi)$ and
$\beta\in \Phi$. For $f=s_{\alpha}$ this gives $s_{\alpha} \circ s_{\beta}\circ s_{\alpha}=s_{s_{\alpha}(\beta)}$.
Computing both sides of this equation, we obtain
$s_{\alpha}(\beta)^{\vee}=\beta^{\vee}-\alpha(\beta^{\vee})\alpha^{\vee}$. This shows that the map
\[i\mathfrak{t}\rmap i\mathfrak{t}, \ \xi\mapsto \xi-\alpha(\xi)\alpha^{\vee}\]
maps $\Phi^{\vee}$ to itself, and sends $\alpha^{\vee}$ to $-\alpha^{\vee}$. So, by the canonical definition of the
dual root system, we obtain that $\alpha^{\vee\vee}=\alpha$.

Using the first part, it suffices to show that $(f^{-1})^*\in Aut(\Phi^{\vee})$ for all $f\in Aut(\Phi)$. Since $f$
preserves that Cartan integers (Lemma 9.2 \cite{Humphreys}), for all $\alpha,\beta\in \Phi$, the following holds:
\[\alpha((f^{-1})^*(\beta^{\vee}))=f^{-1}(\alpha)(\beta^{\vee})=\alpha(f(\beta)^{\vee}).\]
Since $\Phi$ spans $i\mathfrak{t}^*$, we obtain $(f^{-1})^*(\beta^{\vee})=f(\beta)^{\vee}$, hence $(f^{-1})^*\in
Aut(\Phi^{\vee})$.
\end{proof}

\begin{remark}\rm
Denote the Killing form of $\mathfrak{g}$ by $(\cdot,\cdot)$ and the induced isomorphism
$\mathfrak{g}\cong\mathfrak{g}^*$ by
\[\kappa:\mathfrak{g}\rmap \mathfrak{g}^*,\ \kappa(X)(Y):=(X,Y).\]
We note that some authors (e.g.\ \cite{Humphreys}) regard the dual roots inside $i\mathfrak{t}^*$, they refer to the
element $\frac{2\alpha}{(\alpha,\alpha)}$ as being the root dual to $\alpha\in\Phi$. These two conventions are related
by the formula (Proposition 8.3 (g) \cite{Humphreys}):
\begin{equation}\label{EQ}
\kappa(\alpha^{\vee})=\frac{2\alpha}{(\alpha,\alpha)}.
\end{equation}
\end{remark}

Denote the kernel of the exponential of $T$ by
\[\Pi:=\ker(\exp_{|\mathfrak{t}})\subset \mathfrak{t}.\]
Then $\Pi$ is a lattice in $\mathfrak{t}$. For each $\alpha\in \Phi$, consider the element
\[X_{\alpha}:=2\pi i \alpha^{\vee}\in \mathfrak{t}.\]
These elements satisfy the following:

\begin{lemma}\label{Lemma_2}
For all $\alpha\in \Phi$, we have that $\mathbb{R}X_{\alpha}\cap \Pi =\mathbb{Z}X_{\alpha}$. Moreover,
\[\mathfrak{t}^{*\mathrm{reg}}=\{\xi\in\mathfrak{t}^{*} | \xi(X_{\alpha})\neq 0\ \forall \alpha\in \Phi  \}.\]
\end{lemma}
\begin{proof}
According to Theorem 3.10.2 (iv) \cite{DK}, we have that $X_{\alpha}\in \Pi$. Let $\alpha_{1}, \ldots ,\alpha_l$ be a
simple system for $\Phi$ such that $\alpha_1=\alpha$. The corresponding fundamental weights
$\lambda_1,\ldots,\lambda_l\in \mathfrak{t}^*\otimes \mathbb{C}$ are defined by the relations
$\lambda_i(\alpha_j^{\vee})=\delta_{i,j}$, for all $1\leq i,j\leq l$. Let $V_{\lambda_1}$ be the irreducible
representation of $\mathfrak{g}\otimes \mathbb{C}$ of highest weight $\lambda_1$. In particular, $V_{\lambda_1}$
contains a vector $v\neq 0$, such that the action of $X\in\mathfrak{t}$ on $v$ is given by $X\cdot v= \lambda_1(X)v$.
Thus
\[tX_{\alpha}\cdot v=\lambda_1(2\pi i t \alpha^{\vee})v=2\pi it v, \ \forall t\in\mathbb{R}.\]
Integrating $V_{\lambda_1}$ to the action of $G$, we have that $\exp(t X_{\alpha})\cdot v= e^{2\pi i t}\cdot v$. So, if
$t X_{\alpha}\in \Pi$, then $e^{2\pi i t}=1$, hence $t\in\mathbb{Z}$. This shows that $\mathbb{R}X_{\alpha}\cap \Pi
=\mathbb{Z}X_{\alpha}$.

Theorem 3.7.1 (ii) \cite{DK} describes $\mathfrak{t}^{\mathrm{reg}}$ as the complement of the hyperplanes
$\ker(\alpha)$ for $\alpha\in\Phi$. Since $\kappa:\mathfrak{g}\to\mathfrak{g}^*$ is $G$-equivariant and symmetric
(i.e. $\kappa^*=\kappa$), and since $\kappa(X_{\alpha})=\frac{4\pi i \alpha}{(\alpha,\alpha)}$ (\ref{EQ}), we have
that $\mathfrak{t}^{*\mathrm{reg}}$ is the complement of the hyperplanes $\{\xi | \xi(X_\alpha)=0\}$, for $\alpha\in
\Phi$.
\end{proof}

The result below plays a crucial role in the proof of Theorem \ref{Theorem_FIVE}.

\begin{corollary}\label{Corollary_1}
If $f:\mathfrak{t}^*\to \mathfrak{t}^*$ is an isomorphism such that it preserves $\mathfrak{t}^{*\mathrm{reg}}$
and $f^*:\mathfrak{t}\to\mathfrak{t}$ induces an isomorphism of $\Pi$, then $f\in Aut(\Phi)$.
\end{corollary}
\begin{proof}
By the previous lemma, the complement of $\mathfrak{t}^{*\mathrm{reg}}$ is the union of the hyperplanes $\{\xi |
\xi(X_{\alpha})=0\}$ for $\alpha\in \Phi$. Thus $f^*$ preserves the set of lines
$\{\mathbb{R}X_{\alpha}|\alpha\in\Phi\}$. So we can write $f^*(X_{\alpha})=\theta(\alpha) X_{\sigma(\alpha)}$.
Since $X_{-\sigma(\alpha)}=-X_{\sigma(\alpha)}$, we may assume that $\theta(\alpha)>0$. On the other hand, $f^*$ is
an isomorphism of $\Pi$, so it induces an isomorphism between $\Pi\cap \mathbb{R}X_{\alpha}$ and $\Pi\cap
\mathbb{R}X_{\sigma(\alpha)}$. By the lemma, $\theta(\alpha)$ is a generator of $\mathbb{Z}$, hence
$\theta(\alpha)=1$. So $f^*(X_{\alpha})=X_{\sigma(\alpha)}$, and also $f^*(\alpha^{\vee})=\sigma(\alpha)^{\vee}$.
This shows that $f^*\in Aut(\Phi^{\vee})$. Hence, by Lemma \ref{Lemma_doubble_dual}, we obtain that $f\in
Aut(\Phi)$.
\end{proof}

\subsection{Automorphisms}

The group $Aut(\mathfrak{g})$ of Lie algebra automorphisms of $\mathfrak{g}$ contains the normal subgroup
$Ad(G)$ of inner automorphisms. Below we recall two descriptions of the group of \textbf{outer
automorphisms}\index{outer automorphism} $Out(\mathfrak{g}):=Aut(\mathfrak{g})/Ad(G)$.

Let $Aut(\mathfrak{g},\mathfrak{t})$ be the subgroup of $Aut(\mathfrak{g})$ consisting of elements which send
$\mathfrak{t}$ to itself. Since every two maximal tori are conjugated, $Aut(\mathfrak{g},\mathfrak{t})$ intersects
every component of $Aut(\mathfrak{g})$; hence \[Out(\mathfrak{g})\cong
Aut(\mathfrak{g},\mathfrak{t})/Ad(N_G(T)),\] where $N_G(T)$ is the normalizer of $T$ in $G$. An element $\sigma\in
Aut(\mathfrak{g},\mathfrak{t})$ induces a symmetry of $\Phi$, and we denote the resulting group homomorphism by
\[\tau: Aut(\mathfrak{g},\mathfrak{t})\rmap Aut(\Phi), \ \sigma\mapsto (\sigma^{-1}_{|\mathfrak{t}})^*.\]

Recall also that the \textbf{Weyl group}\index{Weyl group} of $\Phi$, denoted by $W\subset Aut(\Phi)$, is the group
generated by the symmetries $s_{\alpha}$ (defined in (\ref{EQ_symmetries})), for $\alpha\in \Phi$.

For the following lemma see Theorem 7.8 \cite{Knapp} and section 3.15 \cite{DK}.

\begin{lemma}\label{Lemma_3}
The map $\tau:Aut(\mathfrak{g},\mathfrak{t})\to Aut(\Phi)$ is surjective, and \[\tau^{-1}(W)=Ad(N_G(T))\subset
Aut(\mathfrak{g},\mathfrak{t}).\] Therefore, $\tau$ induces an isomorphism between the groups
\[Out(\mathfrak{g})\cong Aut(\Phi)/W.\]

If $\mathfrak{c}\subset \mathfrak{t}^*$ is an open Weyl chamber, then $Aut(\Phi)=Aut(\Phi,\mathfrak{c})\ltimes W$,
where
\[Aut(\Phi,\mathfrak{c}):=\{f\in Aut(\Phi)| f(i\mathfrak{c})=i\mathfrak{c}\}.\]
Moreover, $Aut(\Phi,\mathfrak{c})$ (hence also $Out(\mathfrak{g})$) is isomorphic to the symmetry group of the
Dynkin diagram of $\Phi$.
\end{lemma}

The last part of the lemma allows us to compute $Out(\mathfrak{g})$ for all semisimple compact Lie algebras. First, it
is enough to consider simple Lie algebras, since, if $\mathfrak{g}$ decomposes into simple components as
$n_1\mathfrak{s}_1\oplus \ldots \oplus n_k \mathfrak{s}_k$, then \[Out(\mathfrak{g})\cong S_{n_1}\ltimes
Out(\mathfrak{s}_1)^{n_1}\times\ldots \times S_{n_k}\ltimes Out(\mathfrak{s}_k)^{n_k}.\] Further, for the simple
Lie algebras, a glimpse at their Dynkin diagrams reveals that the only ones with nontrivial outer automorphism group
are: $A_{n\geq 2}$, $D_{n\geq 5}$, $E_6$ with $Out\cong \mathbb{Z}_2$, and $D_{4}$ with $Out\cong S_{3}$.

\subsection{The maximal coadjoint orbits}

The manifold $G/T$ is called a \textbf{generalized flag manifold}, and it is diffeomorphic to all maximal leaves of the
linear Poisson structure on $\mathfrak{g}^*$. The topology of $G/T$ is well understood \cite{Borel}; here we discuss
some aspects needed for the proof.

Using that $G$ is 2-connected, we see that the first terms in the long exact sequence in homotopy associated to the
principal $T$-bundle $G\to G/T$ are
\[1\rmap \pi_2(G/T)\rmap \pi_1(T)\rmap 1 \rmap \pi_1(G/T)\rmap 1.\]
Hence $G/T$ is simply connected and that $\pi_1(T)\cong \pi_2(G/T)$. On the other hand, we have an isomorphism
between
\[\Pi\diffto \pi_1(T),\ \ X\mapsto \gamma_X,\ \ \textrm{where }\gamma_X(t):=\exp(tX).\]
The induced isomorphism $\Pi\diffto \pi_2(G/T)$ can be given explicitly: $X\mapsto S_{X}$, if and only if there is a disc
$D_{X}$ in $G$ which projects to $S_{X}$ and has as boundary the curve $\gamma_{X}$.

The lemma below describes the dual of this isomorphism.

\begin{lemma}\label{Lemma_1}
\begin{enumerate}[(a)]
\item For $\xi\in \mathfrak{t}^*$, the 2-form $-d\xi^l$ is the pullback of a closed 2-form $\eta_{\xi}$ on $G/T$, which satisfies
\[\xi(X)=-\int_{S_{X}}\eta_{\xi}, \  \ X\in \Pi.\]
\item The assignment $\mathfrak{t}^*\ni\xi\mapsto [\eta_{\xi}]\in H^2(G/T)$ is a linear isomorphism.
\item We have that $\xi\in\mathfrak{t}^{*\mathrm{reg}}$ if and only if $\Lambda^{\mathrm{top}}[\eta_{\xi}]\in H^{\mathrm{top}}(G/T)$ is not zero.
\end{enumerate}
\end{lemma}

\begin{proof}
Since $-d\xi^l$ is the pullback of $\omega_{\xi}$ via the map $G\to G/G_{\xi}\cong O_{\xi}$, and since $T\subset
G_{\xi}$, it follows that $\eta_{\xi}$ is the pullback of $\omega_{\xi}$ via the projection $G/T \to G/G_{\xi}$. The
second claim in (a) follows using Stokes' theorem:
\[-\int_{S_X}\eta_{\xi}=\int_{D_X}d\xi^l=\int_{\gamma_X}\xi^l=\int_{0}^1\xi\circ d l_{\gamma_X(-t)}\left(\frac{d \gamma_X}{dt}(t)\right)dt=\xi(X).\]

Since $G/T$ is simply connected, we can apply the Hurewicz theorem to conclude that
$H_2(G/T,\mathbb{Z})=\pi_2(G/T)\cong \Pi$. In particular, the second Betti number of $G/T$ equals
$\mathrm{dim}(\mathfrak{t})$. So it suffices to show injectivity of the map $\xi\mapsto [\eta_{\xi}]$. For $\xi\in
\mathfrak{t}^*$, with $\xi\neq 0$, we can find $X\in\Pi$ such that $\xi(X)\neq 0$. Then $\int_{S_X}\eta_{\xi}\neq 0$,
thus $\eta_{\xi}$ is not exact. This proves (b).

If $\xi\in \mathfrak{t}^{*\mathrm{reg}}$, then $T=G_{\xi}$ and so $\eta_{\xi}=\omega_{\xi}$, which is symplectic.
Therefore $\Lambda^{\mathrm{top}}\eta_{\xi}$ is a volume form, and so $\Lambda^{\mathrm{top}}[\eta_{\xi}]\neq
0$. On the other hand, if $\xi\notin \mathfrak{t}^{*\mathrm{reg}}$, then $\mathfrak{t}\subsetneq
\mathfrak{g}_{\xi}$, thus $\mathrm{dim}(T)<\mathrm{dim}(G_{\xi})$. Now, $\eta_{\xi}$ is the pullback of
$\omega_{\xi}$ via the map $G/T\to G/G_{\xi}$, whose fibers have positive dimension. Therefore, the rank of
$\eta_{\xi}$ is constant and strictly smaller than $\mathrm{dim}(G/T)$; hence
$\Lambda^{\mathrm{top}}\eta_{\xi}=0$. This finishes the proof of (c).
\end{proof}
An element $\sigma\in Aut(\mathfrak{g},\mathfrak{t})$ integrates to a Lie group isomorphism of $G$, denoted by the
same symbol, which satisfies $\sigma(T)=T$. Therefore it induces a diffeomorphism $\overline{\sigma}$ of $G/T$. This
diffeomorphism satisfies:

\begin{lemma}\label{Lemma_sigma}
We have that $\overline{\sigma}^{*}(\eta_{\xi})=\eta_{\sigma^*(\xi)}$.
\end{lemma}
\begin{proof}
Using that $l_{\sigma(g)^{-1}}\circ \sigma=\sigma \circ l_{g^{-1}}$ for $g\in G$, the result follows from the
computation below, for $X\in T_gG$:
\[\sigma^*(\xi^l)(X)=\xi(dl_{\sigma(g)^{-1}}\circ d\sigma(X))=\xi(d\sigma\circ d l_{g^{-1}}(X))=\sigma^*(\xi)^l(X).\qedhere\]
\end{proof}

Every diffeomorphism of $G/T$ induces an algebra automorphism of the cohomology ring
$H^{\bullet}(G/T;\mathbb{Z})$, and by Theorem 1.2 \cite{Papadima}, the possible outcomes are covered by the maps
$\overline{\sigma}$. For completeness, we include a proof for the action on $H^2(G/T)$; this will be needed later on.

\begin{proposition}\label{Proposition_1}
For every diffeomorphism $\varphi:G/T\diffto G/T$ there exists $\sigma\in  Aut(\mathfrak{g},\mathfrak{t})$ such that
$\overline{\sigma}$ induces the same map on $H^2(G/T)$
\[\varphi^*=\overline{\sigma}^*:H^2(G/T)\rmap H^2(G/T).\]
\end{proposition}
\begin{proof}
By part (b) of Lemma \ref{Lemma_1}, there exists an automorphism $f:\mathfrak{t}^*\to \mathfrak{t}^*$ such that
$\varphi^*([\eta_{\xi}])=[\eta_{f(\xi)}]$. Since $\varphi^*$ preserves the classes $[\omega]\in H^2(G/T)$ with
$\wedge^{\mathrm{top}}[\omega]=0$, Lemma \ref{Lemma_1} (c) implies that $f$ preserves
$\mathfrak{t}^{*\mathrm{reg}}$. Let $X\in \Pi$, and denote $Y\in \Pi$ the element satisfying $\varphi_*(S_X)=S_Y\in
\pi_2(G/T)$. By Lemma \ref{Lemma_1} (a), for all $\xi\in \mathfrak{t}^*$, we have that
\begin{align*}
\xi(f^*(X))=f(\xi)(X)=-\int_{S_X}\eta_{f(\xi)}=-\int_{S_X}\varphi^*(\eta_{\xi})=-\int_{S_Y}\eta_{\xi}=\xi(Y).
\end{align*}
Thus $f^*(X)=Y$. This shows that $\varphi_*(S_X)=S_{f^*(X)}$ for all $X\in \Pi$. Since $\varphi_*$ is an isomorphism
of $H_2(G/T,\mathbb{Z})$, it follows that $f^*$ is an isomorphism of $\Pi$. By Corollary \ref{Corollary_1} we
conclude that $f\in Aut(\Phi)$. By Lemma \ref{Lemma_3}, there exists $\sigma\in Aut(\mathfrak{g},\mathfrak{t})$
such that $\sigma^*_{|\mathfrak{t}^*}=f$, and so Lemma \ref{Lemma_sigma} implies the result.
\end{proof}

The following consequence will not be used in the proof of Theorem \ref{Theorem_FIVE}.

\begin{corollary}\label{Corollary_moduli_on_coadjoint}
The map $\xi\mapsto \eta_{\xi}$ induces a homeomorphism between the space
$\mathfrak{t}^{*\mathrm{reg}}/Aut(\Phi)$ and an open in the moduli space of all symplectic structures on $G/T$.
\end{corollary}
\begin{proof}
First, Proposition \ref{Proposition_1}, Lemma \ref{Lemma_sigma} and Lemma \ref{Lemma_3} imply that for
$\xi_1,\xi_2\in\mathfrak{t}^{*\mathrm{reg}}$, we have that $\eta_{\xi_1}$ and $\eta_{\xi_2}$ are symplectomorphic,
if and only if $\xi_1=f(\xi_2)$ for some $f\in Aut(\Phi)$. Therefore the map $\xi\mapsto \eta_{\xi}$ induces a bijection
\[\Theta:\mathfrak{t}^{*\mathrm{reg}}/Aut(\Phi)\diffto \mathcal{S}/\textrm{Diff}(G/T),\]
where $\mathcal{S}$ denotes the space of all symplectic forms on $G/T$ that are symplectomorphic to one of the type
$\eta_{\xi}$, for $\xi\in \mathfrak{t}^{*\mathrm{reg}}$. The Moser argument implies that $\mathcal{S}$ is
$C^0$-open in the space of all symplectic forms. Next, we show that $\Theta$ is a homeomorphism. Continuity of
$\Theta$ follows from that of the map $\xi\mapsto \eta_{\xi}\in \mathcal{S}$. Hodge theory implies that taking
cohomology is a continuous map $\mathcal{S}\to H^2(G/T)$ and, composing with the isomorphism
$H^2(G/T)\cong\mathfrak{t}^*$, it follows that the induced map $\mathcal{S}\to \mathfrak{t}^{*}$ is continuous.
By Lemma \ref{Lemma_1} (c), the image of this map is $\mathfrak{t}^{*\mathrm{reg}}$. Therefore, the lift of
$\Theta^{-1}$ to a map $\mathcal{S}\to \mathfrak{t}^{*\mathrm{reg}}/Aut(\Phi)$ is continuous. Thus also
$\Theta^{-1}$ is continuous, and so $\Theta$ is a homeomorphism.
\end{proof}

\section{Proof of part $(b)$ of Theorem \ref{Theorem_FIVE}}

In Poisson geometric terms, $\mathfrak{g}^{*\mathrm{reg}}$ is described as the regular part of
$(\mathfrak{g}^*,\pi_{\mathfrak{g}})$, i.e.\ the open consisting of leaves of maximal dimension. The regular part of
$\mathbb{S}(\mathfrak{g}^*)$ is
$\mathbb{S}(\mathfrak{g}^*)^{\mathrm{reg}}=\mathfrak{g}^{*\mathrm{reg}}\cap\mathbb{S}(\mathfrak{g}^*)$.
Let $\mathfrak{c}\subset \mathfrak{t}^*$ be an open Weyl chamber and denote by
$\mathbb{S}(\mathfrak{c}):=\mathfrak{c}\cap \mathbb{S}(\mathfrak{g}^*)$. From (\ref{EQ_regular_part}) it
follows that $\mathbb{S}(\mathfrak{g}^*)^{\mathrm{reg}}$ is described by the diffeomorphism
\[\Psi: G/T\times \mathbb{S}(\mathfrak{c})\diffto \mathbb{S}(\mathfrak{g}^{*})^{\mathrm{reg}}, \ \ \Psi([g],\xi):=Ad_{g^{-1}}^{*}(\xi),\]
and the symplectic leaves correspond to the slices $(G/T\times \{\xi\},\eta_{\xi})$, $\xi\in\mathfrak{c}$.

\begin{proof}[Proof of part (b) of Theorem \ref{Theorem_FIVE}]
Let $\phi:(\mathbb{S}(\mathfrak{g}^*),f\pi_{\mathbb{S}})\rmap (\mathbb{S}(\mathfrak{g}^*),g\pi_{\mathbb{S}})$
be a Poisson diffeomorphism, where $f,g$ are positive Casimirs. Now, the symplectic leaves of $f\pi_{\mathbb{S}}$
and $g\pi_{\mathbb{S}}$ are also the coadjoint orbits $O_{\xi}$, for $\xi\in \mathbb{S}(\mathfrak{g}^*)$, but with
symplectic structures $1/f(\xi)\omega_{\xi}$, respectively $1/g(\xi)\omega_{\xi}$. In particular, they have the same
regular part $\mathbb{S}(\mathfrak{g}^*)^{\mathrm{reg}}$. So, after conjugating with $\Psi$, the Poisson
diffeomorphism on the regular parts takes the form
\[\Psi^{-1}\circ \phi\circ \Psi: (G/T\times \mathbb{S}(\mathfrak{c}), \Psi^*(f\pi_{\mathbb{S}}))\diffto (G/T\times \mathbb{S}(\mathfrak{c}), \Psi^*(g\pi_{\mathbb{S}})),\]
\[(x,\xi)\mapsto (\phi_{\xi}(x),\theta(\xi)), \]
for a diffeomorphism $\theta:\mathbb{S}(\mathfrak{c})\diffto \mathbb{S}(\mathfrak{c})$ and a symplectomorphism
\[\phi_{\xi}:(G/T,\eta_{\xi/f(\xi)})\diffto (G/T,\eta_{\theta(\xi)/g(\theta(\xi))}).\]
By connectivity of $\mathbb{S}(\mathfrak{c})$, the maps $\phi_{\xi}$ for $\xi\in \mathbb{S}(\mathfrak{c})$ are
isotopic to each other. In particular, they induces the same map on $H^2(G/T)$. By Proposition \ref{Proposition_1},
this map is induced also by a diffeomorphism $\overline{\sigma}$, for some $\sigma\in
Aut(\mathfrak{g},\mathfrak{t})$. Lemma \ref{Lemma_sigma} implies the following equality in $H^2(G/T)$ for all
$\xi\in \mathbb{S}(\mathfrak{c})$:
\[[\eta_{\xi/f(\xi)}]=[\phi_{\xi}^*(\eta_{\theta(\xi)/g(\theta(\xi))})]=[\overline{\sigma}^*(\eta_{\theta(\xi)/g(\theta(\xi))})]=[\eta_{\sigma^*(\theta(\xi)/g(\theta(\xi)))}].\]
Using Lemma \ref{Lemma_1} we obtain that $\xi/f(\xi)=\sigma^*(\theta(\xi))/g(\theta(\xi))$. Since $\sigma^*$
preserves the norm, we get that $f(\xi)=g(\theta(\xi))$. This shows that $\xi=\sigma^*(\theta(\xi))$. So $\sigma^*$
preserves $\mathbb{S}(\mathfrak{c})$ and on this space $\theta=(\sigma^{-1})^*$. So $f\circ \sigma^*(\xi)=g(\xi)$
for all $\xi\in \mathbb{S}(\mathfrak{c})$. Since the regular leaves are dense in $\mathbb{S}(\mathfrak{g}^*)$, they
all hit $\mathbb{S}(\mathfrak{c})$, and since both $f\circ \sigma^*$ and $g$ are Casimirs, it follows that $f\circ
\sigma^*=g$.
\end{proof}

\section{The space of Casimirs}\label{Section_the_space_of_Casimirs}

Theorem \ref{Theorem_FIVE} implies that the map that associates to $F\in
\mathfrak{Casim}(\mathbb{S}(\mathfrak{g}^*),\pi_{\mathbb{S}})$ the Poisson structure $e^{F}\pi_{\mathbb{S}}$
on $\mathbb{S}(\mathfrak{g}^*)$ induces a parametrization of an open in the Poisson moduli space of
$\mathbb{S}(\mathfrak{g}^*)$ around $\pi_{\mathbb{S}}$ by the space
\[\mathfrak{Casim}(\mathbb{S}(\mathfrak{g}^*),\pi_{\mathbb{S}})/Out(\mathfrak{g}).\]
In this section we describe this space using classical invariant theory.\index{Casimir function}

Let $P[\mathfrak{g}^*]$ and $P[\mathfrak{t}^*]$ denote the algebras of polynomials on $\mathfrak{g}^*$ and
$\mathfrak{t}^*$ respectively. A classical result (see e.g. Theorem 7.3.5 \cite{Dixmier}) states that the restriction
map $P[\mathfrak{g}^*]\to P[\mathfrak{t}^*]$ induces an isomorphism between the algebras of invariants
\begin{equation}\label{EQ_Invariant_Polynomials}
P[\mathfrak{g}^*]^G\cong P[\mathfrak{t}^*]^W.
\end{equation}
A theorem of Schwarz, extends this result to the smooth setting
\begin{equation}\label{EQ_Invariant_Functions}
C^{\infty}(\mathfrak{g}^*)^G\cong C^{\infty}(\mathfrak{t}^*)^W.
\end{equation}
To explain this, first recall that $P[\mathfrak{g}^*]^G$ is generated by $l:=\mathrm{dim}(\mathfrak{t})$
algebraically independent homogeneous polynomials $p_1,\ldots,p_l$ (Theorem 7.3.8 \cite{Dixmier}). Hence by
(\ref{EQ_Invariant_Polynomials}), $P[\mathfrak{t}^*]^W$ is generated by
$q_1:=p_{1|\mathfrak{t}^*},\ldots,q_l:=p_{l|\mathfrak{t}^*}$. Consider the maps
\[p=(p_1,\ldots,p_l):\mathfrak{g}^*\rmap \mathbb{R}^{l},\  \textrm{and}\ q=(q_1,\ldots,q_l):\mathfrak{t}^*\rmap \mathbb{R}^{l},\]
and denote by $\Delta:=p(\mathfrak{g}^*)$. Since the inclusion $\mathfrak{t}^*\subset \mathfrak{g}^*$ induces a
bijection between the $W$-orbits and the $G$-orbits, it follows that $q(\mathfrak{t}^*)=\Delta$. The theorem of
Schwarz \cite{Schwarz} applied to the action of $G$ on $\mathfrak{g}^*$ and to the action of $W$ on
$\mathfrak{t}^*$, shows that the pullbacks by $p$ and $q$ give isomorphisms
\begin{equation}\label{EQ_Invariant_Functions2}
C^{\infty}(\mathfrak{g}^*)^G\cong C^{\infty}(\Delta), \ \ C^{\infty}(\mathfrak{t}^*)^W\cong C^{\infty}(\Delta).
\end{equation}
In particular, we obtain (\ref{EQ_Invariant_Functions}). Schwarz's result asserts that $p$, respectively $q$, induce
homeomorphisms between the orbit spaces and $\Delta$
\begin{equation}\label{EQ_Leaf_space_homeomorphisms}
\mathfrak{g}^*/G\cong \Delta\cong \mathfrak{t}^*/W.
\end{equation}
We can describe the orbit space also using an open Weyl chamber $\mathfrak{c}\subset \mathfrak{t}^*$.

\begin{lemma}\label{Lemma_2}
The map $q:\overline{\mathfrak{c}}\to \Delta$ is a homeomorphism and restricts to a diffeomorphism between the
interiors $q:\mathfrak{c}\to \mathrm{int}(\Delta)$.
\end{lemma}
\begin{proof}
It is well known that $\overline{\mathfrak{c}}$ intersects each orbit of $W$ exactly once (see e.g.\ \cite{DK}) and so,
by (\ref{EQ_Leaf_space_homeomorphisms}), the map is a bijection. Since $q:t^*\to \mathbb{R}^l$ is proper, it follows
that also $q_{|\overline{c}}$ is proper, and this implies the first part. We are left to check that $q_{|\mathfrak{c}}$
is an immersion. Let $V\in T_\xi\mathfrak{c}$ be a nonzero vector. Consider $\chi$ a smooth, compactly supported
function on $\mathfrak{c}$ satisfying $d\chi_{\xi}(V)\neq 0$. Since the action gives a homeomorphism $W\times
\mathfrak{c}\cong \mathfrak{t}^{*\mathrm{reg}}$, $\chi$ has a unique $W$-invariant extension to
$\mathfrak{t}^*$, which is defined on $w\mathfrak{c}$ by $\widetilde{\chi}=w^*(\chi)$, and extended by zero on
$\mathfrak{t}^*\backslash \mathfrak{t}^{*\mathrm{reg}}$. Then, by (\ref{EQ_Invariant_Functions2}),
$\widetilde{\chi}$ is of the form $\widetilde{\chi}=h\circ q$, for some $h\in C^{\infty}(\Delta)$. Differentiating in the
direction of $V$, we obtain that $d_{\xi}q(V)\neq 0$, and this finishes the proof.
\end{proof}

The polynomials $p_1,\ldots,p_l$ are not unique; a necessary and sufficient condition for a set of homogeneous
polynomials to be such a generating system is that their image in $I/I^2$ forms a basis, where $I\subset
P[\mathfrak{g}^*]^G$ denotes the ideal of polynomials vanishing at $0$. Since $I^2$ is $Out(\mathfrak{g})$
invariant, it is easy to see that we can choose $p_1,\ldots,p_l$ such that $p_1(\xi)=|\xi|^2$ and the linear span of
$p_2,\ldots, p_l$ is $Out(\mathfrak{g})$ invariant. This choice endows $\mathbb{R}^l$ with a linear action of
$Out(\mathfrak{g})$ for which $p$ is $Aut(\mathfrak{g})$ equivariant. Moreover, the action is trivial on the first
component and $\{0\}\times\mathbb{R}^{l-1}$ is an invariant subspace. The isomorphism $Aut(\Phi)/W\cong
Out(\mathfrak{g})$ from Lemma \ref{Lemma_3} shows that also $q$ is equivariant with respect to the actions of
$Aut(\Phi)$ and $Out(\mathfrak{g})$. Thus we have isomorphisms between the spaces
\[C^{\infty}(\mathfrak{g}^*)^G/Aut(\mathfrak{g})\cong C^{\infty}(\mathfrak{t}^*)^W/Aut(\Phi)\cong C^{\infty}(\Delta)/Out(\mathfrak{g}).\]

Notice that every Casimir $f$ on $\mathbb{S}(\mathfrak{g}^*)$ can be extended to a $G$-invariant smooth function on $\mathfrak{g}^*$, 
therefore
\[\mathfrak{Casim}(\mathbb{S}(\mathfrak{g}^*),\pi_{\mathbb{S}})\cong C^{\infty}(\mathfrak{g}^*)^G_{|\mathbb{S}(\mathfrak{g}^*)}.\]
Since $p_1(\xi)=|\xi|^2$, it follows that $p(\mathbb{S}(\mathfrak{g}^*))=\left(\{1\}\times
\mathbb{R}^{l-1}\right)\cap \Delta$. Denote
\[p':=(p_2,\ldots,p_l):\mathfrak{g}^*\rmap \mathbb{R}^{l-1} \ \ \textrm{and} \ \ \Delta':=p'(\mathbb{S}(\mathfrak{g}^*)).\]
We have that $C^{\infty}(\Delta')=C^{\infty}(\Delta)_{|\{1\}\times \Delta'}$. By Lemma \ref{Lemma_2}, we have that
$q':=p'_{|\mathbb{S}(\mathfrak{t}^*)}$ is a homeomorphism between $\mathbb{S}(\overline{\mathfrak{c}})\cong
\Delta'$, which restricts to a diffeomorphism between $\mathbb{S}(\mathfrak{c})\cong \mathrm{int}(\Delta')$. This
shows that $\Delta'=\overline{B}$, where $B$ is a bounded open, diffeomorphic to an open ball.

With these, we have the following description of the Casimirs.

\begin{corollary}
The polynomial map $p':\mathfrak{g}^*\to \mathbb{R}^{l-1}$ is equivariant with respect to the actions of
$Aut(\mathfrak{g})$ and $Out(\mathfrak{g})$, and $q':=p'_{|\mathfrak{t}^*}$ is equivariant with respect to the
actions of $Aut(\Phi)$ and $Out(\mathfrak{g})$. These maps induce isomorphisms between the spaces
\[\mathfrak{Casim}(\mathbb{S}(\mathfrak{g}^*),\pi_{\mathbb{S}})/Out(\mathfrak{g})\cong C^{\infty}(\mathbb{S}(\mathfrak{t}^*))^W/Out(\mathfrak{g})\cong C^{\infty}(\Delta')/Out(\mathfrak{g}),\]
and $Out(\mathfrak{g})$-equivariant homeomorphisms between the spaces
\[\mathbb{S}(\mathfrak{g}^*)/G\cong \mathbb{S}(\mathfrak{t}^*)/W\cong \Delta'.\]
\end{corollary}

\section{The case of $\mathfrak{su}(3)$.}

In this section, we describe our result for the Lie algebra $\mathfrak{g}=\mathfrak{su}(3)$, whose 1-connected Lie
group is $G=\mathbf{SU}(3)$. Recall that
\[\mathfrak{su}(3)=\{A\in M_3(\mathbb{C})| A+A^*=0, \ tr(A)=0\}, \]
\[  \mathbf{SU}(3)=\{U\in M_3(\mathbb{C})| UU^*=I, \ det(U)=1\}.\]
We use the invariant inner product given by the negative of the trace form $(A,B):=-tr(AB)$. Let $\mathfrak{t}$ be the
space of diagonal matrices in $\mathfrak{su}(3)$
\[\mathfrak{t}:=\left\{D(ix_1,ix_2,ix_3):=\left(
                                               \begin{array}{ccc}
                                                 ix_1 & 0 & 0 \\
                                                 0 & ix_2 & 0 \\
                                                 0 & 0 & ix_3 \\
                                               \end{array}
                                             \right) | \ x_j\in\mathbb{R},\ \sum_j x_j=0 \right\}.\]
The corresponding maximal torus is
\[T:=\{D(e^{i\theta_1},e^{i\theta_2},e^{i\theta_3})\ |\ \theta_j\in\mathbb{R},\ \prod_j e^{i\theta_j}=1\}.\]
The Weyl group is $W=S_3$. It acts on $\mathfrak{t}$ as follows
\[\sigma D(ix_1,ix_2,ix_3)=D(ix_{\sigma(1)},ix_{\sigma(2)},ix_{\sigma(3)}), \ \sigma\in S_3.\]
The Dynkin diagram of $\mathfrak{su}(3)$ is $A_2$ (a graph with one edge), so its symmetry group is
$\mathbb{Z}_2$. A generator of $Out(\mathfrak{su}(3))$ is complex conjugation
\[\gamma\in Aut(\mathfrak{su}(3),\mathfrak{t}),\ \gamma(A)=\overline{A}.\]
On $\mathfrak{t}$, $\gamma$ acts by multiplication with $-1$.

Under the identification of $\mathfrak{t}\cong\mathfrak{t}^*$ given by the inner product, the invariant polynomials
$P[\mathfrak{t}]^{S_3}$ are generated by the symmetric polynomials
\[q_1(D(ix_1,ix_2,ix_3))=x_1^2+x_2^2+x_3^2, \ \ q_2(D(ix_1,ix_2,ix_3))=\sqrt{6}(x_1^3+x_2^3+x_3^3).\]
Identifying also $\mathfrak{su}(3)\cong\mathfrak{su}^*(3)$, $q_1$ and $q_2$ are the restriction to $\mathfrak{t}$ of
the invariant polynomials $p_1$, $p_2\in P[\mathfrak{su}^*(3)]^{\mathbf{SU}(3)}$ (which generate
$P[\mathfrak{su}^*(3)]^{\mathbf{SU}(3)}$)
\[ p_1(A)=-tr(A^2), \ \ \ p_2(A)=i\sqrt{6}tr(A^3).\]
Clearly $p_2\circ\gamma=-p_2$. The inner product on $\mathfrak{t}$ is
\[(D(ix_1,ix_2,ix_3),D(ix'_1,ix'_2,ix'_3))=x_1x_1'+x_2x_2'+x_3x_3',\]
and we have that $\mathbb{S}(\mathfrak{t}^*)\cong \mathbb{S}(\mathfrak{t})$ is a circle, isometrically
parameterized by
\[A(\theta):= \frac{cos(\theta)}{\sqrt{2}} D\left(i,-i,0\right)+\frac{sin(\theta)}{\sqrt{6}}D\left(i,i,-2i\right), \  \theta\in [0,2\pi].\]
In polar coordinates on $\mathfrak{t}$, the polynomials $q_1$ and $q_2$ become
\[q_1(rA(\theta))=r^2, \ \ q_2(rA(\theta))=r^3 sin(3\theta).\]
This implies that the space $\Delta$ is given by
\[\Delta=\{(r^2, r^3 sin(3\theta)) | r\geq 0, \ \theta\in [0,2\pi] \}=\{(x,y)\in \mathbb{R}^2 | x^3\geq y^2\}.\]
The map $q:=(q_1,q_2):\mathfrak{t}\to \mathbb{R}^2$, restricted to the open Weyl chamber
\[\mathfrak{c}:=\{rA(\theta)| r>0, \ \theta\in (-\pi/6,\pi/6 )\},\]
is a diffeomorphism onto $\mathrm{int}(\Delta)$. The linear action of $\mathbb{Z}_2=Out(\mathfrak{su}(3))$ on
$\mathbb{R}^2$, for which $q$ is equivariant, is multiplication by $-1$ on the second component. Therefore
$q':=q_2$ is a $\mathbb{Z}_2$-equivariant homeomorphism between
\[q':\mathbb{S}(\overline{\mathfrak{c}})=\{A(\theta) | \theta\in [-\pi/6,\pi/6]\}\diffto \Delta':=[-1,1],\]
which restricts to a diffeomorphism between the interiors.

We conclude that the Poisson moduli space of the 7-dimensional sphere $\mathbb{S}(\mathfrak{su}(3)^*)$ is
parameterized around $\pi_{\mathbb{S}}$ by the space
\[C^{\infty}([-1,1])/\mathbb{Z}_2,\]
where $\mathbb{Z}_2$ acts on $C^{\infty}([-1,1])$ by the involution
\[\gamma(f)(x)=f(-x), \ \ f\in C^{\infty}([-1,1]).\]

\clearpage \pagestyle{plain}

\appendix

\chapter{The Tame Vanishing Lemma, and Rigidity of Foliations, an unpublished paper of Richard S.
Hamilton}\label{ChFoli}

\pagestyle{fancy}

\fancyhead[CE]{Appendix} 
\fancyhead[CO]{The Tame Vanishing Lemma}

In the first part of this appendix we prove the Tame Vanishing Lemma, an existence result for tame homotopy
operators on the complex computing Lie algebroid cohomology with coefficients. We have used this result in chapter
\ref{ChRigidity} for the Poisson complex. In the second part we revisit a theorem of Richard S. Hamilton \cite{Ham3}
on rigidity of foliations, and we show that the Tame Vanishing Lemma implies ``tame infinitesimal rigidity'', which is a
crucial step in the proof of this result.

\section{The Tame Vanishing Lemma}\label{section_homotopy_operators}

\subsection{The weak $C^{\infty}$-topology}\label{subsection_weak_topology}

The compact-open $C^k$-topology, defined in section \ref{introduction_rigi}, can be generated by a family of
semi-norms, and we recall here a construction of such semi-norms, generalizing the construction from section
\ref{section_technical_rigidity}. These semi-norms will be used to express the tameness property of the homotopy
operators.

Let $W\to M$ be a vector bundle. Consider $\mathcal{U}:=\{U_i\}_{i\in I}$ a locally finite open cover of $M$ by
relatively compact domains of coordinate charts $\{\chi_i:U_i\diffto \mathbb{R}^m\}_{i\in I}$ and choose
trivializations for $W_{|U_i}$. Let $\mathcal{O}:=\{O_i\}_{i\in I}$ be a second open cover, with $\overline{O}_i$
compact and $\overline{O}_i\subset U_i$. A section $\sigma\in\Gamma(W)$ can be represented in these charts by a
family of smooth functions $\{\sigma_i:\mathbb{R}^m\to \mathbb{R}^k\}_{i\in I}$, where $k$ is the rank of $W$. For
$U\subset M$, an open set with compact closure, we have that $\overline{U}$ intersects only a finite number of the
coordinate charts $U_i$. Denote the set of such indexes by $I_{U}\subset I$. Define the $n$-th norm of $\sigma$ on
$U$ by
\[\|\sigma\|_{n,\overline{U}}:=\sup \left\{|\frac{\partial^{|\alpha|}\sigma_{i}}{\partial x^{\alpha}}(x)|: |\alpha|\leq n,\   x\in \chi_i(U\cap O_i),\  i\in I_U\right\}.\]

For a fixed $n$, the family of semi-norms $\|\cdot\|_{n,\overline{U}}$, with $U$ a relatively compact open in $M$,
generate the \textbf{compact-open}\index{compact-open} $C^{n}$-\textbf{topology} on $\Gamma(W)$ discussed in
section \ref{introduction_rigi}. The union of all these topologies, for $n\geq 0$, is called the \textbf{weak}
$C^{\infty}$-\textbf{topology} on $\Gamma(W)$. Observe that the semi-norms $\{\|\cdot\|_{n,\overline{U}}\}_{n\geq
0}$ induce norms on $\Gamma(W_{|\overline{U}})$.\index{C$^n$-norms}

\subsection{The statement of the Tame Vanishing Lemma}

We use the notations from subsection \ref{subsection_Lie_algebroid_cohomology} Lie algebroid cohomology with
coefficients. The main result of the appendix is:\index{cohomology, Lie algebroid}
\begin{tvL}\label{Theorem_tame_vanishing}
Let ${\mathcal{G}}\rightrightarrows M$ be a Hausdorff Lie groupoid with Lie algebroid $A$ and let $V$ be a
representation of $\mathcal{G}$. If the $s$-fibers of ${\mathcal{G}}$ are compact and their de Rham cohomology
vanishes in degree $p$, then \[H^p(A,V)=0.\] Moreover, there exist linear homotopy operators
\[\Omega^{p-1}(A,V) \stackrel{h_1}{\longleftarrow}\Omega^{p}(A,V)\stackrel{h_2}{\longleftarrow}\Omega^{p+1}(A,V), \]
\[d_{\nabla}h_1+h_2d_{\nabla}=\mathrm{Id},\]
which satisfy\index{homotopy operators}
\begin{enumerate}[(1)]
\item invariant locality: for every orbit $O$ of $A$, they induce linear maps
\[\Omega^{p-1}(A_{|O},V_{|O}) \stackrel{h_{1,O}}{\longleftarrow}\Omega^{p}(A_{|O},V_{|O})\stackrel{h_{2,O}}{\longleftarrow}\Omega^{p+1}(A_{|O},V_{|O}),\]
such that for all $\omega\in \Omega^{p}(A,V)$, $\eta\in \Omega^{p+1}(A,V)$, we have that
\[h_{1,O}(\omega_{|O})=(h_1\omega)_{|O},\  \ h_{2,O}(\eta_{|O})=(h_2 \eta)_{|O},\]
\item tameness: for every invariant open $U\subset M$, with $\overline{U}$ compact, there are constants
$C_{n,U}>0$, such that
\[\|h_1(\omega)\|_{n,\overline{U}}\leq C_{n,U} \|\omega \|_{n+s,\overline{U}},\ \ \|h_2(\eta)\|_{n,\overline{U}}\leq C_{n,U}  \|\eta \|_{n+s,\overline{U}},\]
for all $\omega\in\Omega^{p}(A_{|\overline{U}},V_{|\overline{U}})$ and
$\eta\in\Omega^{p+1}(A_{|\overline{U}},V_{|\overline{U}})$, where \[s=\lfloor\frac{1}{2}rank(A)\rfloor+1.\]
\end{enumerate}
\end{tvL}
We also note the following consequences of the proof:
\begin{corollary}\label{corollary_unu}
The constants $C_{n,U}$ can be chosen such that they are uniform over all invariant open subsets of $U$. More
precisely: if $V\subset U$ is a second invariant open, then one can choose $C_{n,V}:=C_{n,U}$, assuming that the
norms on $\overline{U}$ and $\overline{V}$ are computed using the same charts and trivializations. 
\end{corollary}
\begin{corollary}\label{corollary_unu_prim}
The homotopy operators preserve the order of vanishing around orbits. More precisely: if $O$ is an orbit of $A$ and
$\omega\in \Omega^{p}(A,V)$ is a form such that $j^k_{|O}\omega=0$, then $j^k_{|O}h_1(\omega)=0$; and
similarly for $h_2$.
\end{corollary}

\subsection{The de Rham complex of a fiber bundle}\label{subsection_de Rham_complex}

To prove the Tame Vanishing Lemma, we first construct tame homotopy operators for the foliated de Rham complex of
a fiber bundle. For this, we use a result on the family of inverses of elliptic operators (Proposition
\ref{Theorem_family_of_operators}), which we prove at the end of the section.

Let $\pi:\mathcal{B}\to M$ be a locally trivial fiber bundle whose fibers $\mathcal{B}_x:=\pi^{-1}(x)$ are
diffeomorphic to a compact, connected manifold $F$ and let $V\to M$ be a vector bundle. The space of vertical vectors
on $\mathcal{B}$ will be denoted by $T^{\pi}\mathcal{B}$ and the space of foliated forms with values in $\pi^*(V)$
by $\Omega^{\bullet}(T^{\pi}\mathcal{B},\pi^*(V))$. An element
$\omega\in\Omega^{\bullet}(T^{\pi}\mathcal{B},\pi^*(V))$ is a smooth family of forms on the fibers of $\pi$ with
values in $V$
\[\omega=\{\omega_{x}\}_{x\in M}, \ \ \ \ \ \omega_x\in \Omega^{\bullet}(\mathcal{B}_x,V_x).\]
The fiberwise exterior derivative induces the differential
\[d\otimes I_V:\Omega^{\bullet}(T^{\pi}\mathcal{B},\pi^*(V))\rmap \Omega^{\bullet+1}(T^{\pi}\mathcal{B},\pi^*(V)), \]
\[d\otimes I_V(\omega)_x:=(d\otimes I_{V_x})(\omega_x),\  x\in M.\]

We construct the homotopy operators using Hodge theory. Let $m$ be a metric on $T^{\pi}{\mathcal{B}}$, or
equivalently a smooth family of Riemannian metrics $\{m_x\}_{x\in M}$ on the fibers of $\pi$. Integration against the
volume density gives an inner product on $\Omega^{\bullet}(\mathcal{B}_x)$
\[(\eta,\theta):=\int_{\mathcal{B}_x}m_x(\eta,\theta)|dVol(m_x)|,\ \eta,\theta\in\Omega^q(\mathcal{B}_x).\]
Let $\delta_x$ denote the formal adjoint of $d$ with respect to this inner product
\[\delta_x:\Omega^{\bullet+1}({\mathcal{B}}_x)\rmap \Omega^{\bullet}({\mathcal{B}}_x),\]
i.e.\ $\delta_x$ is the unique linear first order differential operator satisfying
\[(d\eta,\theta)=(\eta,\delta_x\theta),\ \ \forall\ \eta\in\Omega^{\bullet}(\mathcal{B}_x), \theta\in\Omega^{{\bullet}+1}(\mathcal{B}_x).\]
The Laplace-Beltrami operator associated to $m_x$ will be denoted by\index{Laplace-Beltrami operator}
\[\Delta_x:\Omega^{\bullet}({\mathcal{B}}_x)\rmap \Omega^{\bullet}({\mathcal{B}}_x),\ \ \Delta_x:=d\delta_x+\delta_x d.\]
Both these operators induce linear differential operators on $\Omega^{\bullet}(T^{\pi}\mathcal{B},\pi^{*}(V))$
\[\delta\otimes I_V:\Omega^{\bullet+1}(T^{\pi}\mathcal{B},\pi^{*}(V))\to \Omega^{\bullet}(T^{\pi}\mathcal{B},\pi^{*}(V)),\ \ \ \delta\otimes I_V(\omega)_x:=(\delta_x\otimes I_{V_x})(\omega_x),\]
\[\Delta\otimes I_V:\Omega^{\bullet}(T^{\pi}\mathcal{B},\pi^{*}(V))\to \Omega^{\bullet}(T^{\pi}\mathcal{B},\pi^{*}(V)),\ \ \ \Delta\otimes I_V(\omega)_x:=(\Delta_x\otimes I_{V_x})(\omega_x).\]

By the Hodge theorem (see e.g.\ \cite{Warner}), if the fiber $F$ of $\mathcal{B}$ has vanishing de Rham cohomology
in degree $p$, then $\Delta_x$ is invertible in degree $p$.
\begin{lemma}\label{Lemma_fiberwise_Green}
If $H^p(F)=0$ then the following hold:
\begin{enumerate}[(a)]
\item $\Delta\otimes I_V$ is invertible in degree $p$ and its inverse is given by
\begin{align*}
G\otimes I_V :\Omega^{p}(T^{\pi}{\mathcal{B}},\pi^{*}(V))\rmap \Omega^{p}(T^{\pi}{\mathcal{B}},\pi^{*}(V)),\\
(G\otimes I_V)(\omega)_x:=(\Delta_x^{-1}\otimes I_{V_x})(\omega_x), \  x\in M;
\end{align*}
\item the maps $H_1:=(\delta\otimes I_V)\circ (G\otimes I_V)$ and $H_2:=(G\otimes I_V)\circ (\delta\otimes I_V)$
\[\Omega^{p-1}(T^{\pi}\mathcal{B},\pi^{*}(V))\stackrel{H_1}{\longleftarrow}\Omega^{p}(T^{\pi}\mathcal{B},\pi^{*}(V))\stackrel{H_2}{\longleftarrow}\Omega^{p+1}(T^{\pi}\mathcal{B},\pi^{*}(V))\]
are linear homotopy operators in degree $p$;\index{homotopy operators}
\item $H_1$  and $H_2$ satisfy the following local-tameness property: for every relatively compact open $U\subset M$, there are
constants $C_{n,U}>0$ such that
\[\|H_1(\eta) \|_{n,\mathcal{B}_{|\overline{U}}}\leq C_{n,U} \|\eta \|_{n+s,\mathcal{B}_{|\overline{U}}},\ \forall\ \eta\in\Omega^{p}(T^{\pi}{\mathcal{B}_{|\overline{U}}},\pi^{*}(V_{|\overline{U}})),\]
\[\|H_2(\omega) \|_{n,\mathcal{B}_{|\overline{U}}}\leq C_{n,U} \|\omega \|_{n+s,\mathcal{B}_{|\overline{U}}},\ \forall\ \omega\in\Omega^{p+1}(T^{\pi}{\mathcal{B}_{|\overline{U}}},\pi^{*}(V_{|\overline{U}})).\]
where $s=\lfloor\frac{1}{2}\mathrm{dim}(F)\rfloor+1$.

Moreover, if $V\subset U$, then one can take $C_{n,V} := C_{n,U}$.
\end{enumerate}
\end{lemma}
\begin{proof}
In a trivialization chart the operator $\Delta\otimes I_V$ is given by a smooth family of Laplace-Beltrami operators:
\[\Delta_x:\Omega^p(F)^k\rmap \Omega^p(F)^k,\]
where $k$ is the rank of $V$. These operators are elliptic and invertible, therefore, by Proposition
\ref{Theorem_family_of_operators}, $\Delta_x^{-1}(\omega_x)$ is smooth in $x$, for every smooth family $\omega_x\in
\Omega^p(F)^k$. This shows that $G\otimes I_V$ maps smooth forms to smooth forms. Clearly $G\otimes I_V$ is the
inverse of $\Delta\otimes I_V$, so we have proven (a).

For part (c), let $U\subset M$ be a relatively compact open. Applying part (2) of Proposition
\ref{Theorem_family_of_operators} to a family of coordinate charts which cover $\overline{U}$, we find constants
$D_{n,U}$ such that
\[\|G\otimes I_V(\eta) \|_{n,\mathcal{B}_{|\overline{U}}}\leq D_{n,U} \|\eta \|_{n+s-1,\mathcal{B}_{|\overline{U}}},\ \forall\ \eta\in\Omega^{p}(T^{\pi}{\mathcal{B}_{|\overline{U}}},\pi^{*}(V_{|\overline{U}})).\]
Moreover, the constants may be assumed to be decreasing in $U$. Since $H_1$ and $H_2$ are defined as the
composition of $G\otimes I_V$ with a linear differential operator of degree one, it follows that we can also find
constants $C_{n,U}$ such that the inequalities form (c) are satisfied, and which are also decreasing in $U$.

For part (b), using that $\delta_x^2=0$, we obtain that $\Delta_x$ commutes with $d \delta_x$
\[\Delta_x d\delta_x=(d\delta_x+\delta_x d)d\delta_x=d\delta_x d\delta_x+ \delta_x d^2\delta_x=d\delta_x d\delta_x,\]
\[d\delta_x\Delta_x=d\delta_x(d\delta_x+\delta_x d)=d\delta_x d\delta_x+d\delta_x^2 d=d\delta_x d\delta_x.\]
This implies that $\Delta\otimes I_V$ commutes with $(d\otimes I_V)(\delta\otimes I_V)$, and thus $G\otimes I_V$
commutes with $(d\otimes I_V)( \delta\otimes I_V)$. Using this, we obtain that $H_1$ and $H_2$ are homotopy
operators
\begin{align*}
I=&(G\otimes I_V)(\Delta\otimes I_V)=(G\otimes I_V)((d\otimes I_V)(\delta\otimes I_V)+(\delta\otimes I_V)( d\otimes I_V))=\\
&=(d\otimes I_V)(\delta\otimes I_V)(G\otimes I_V)+(G\otimes I_V)(\delta\otimes I_V)( d\otimes I_V)=\\
&=(d\otimes I_V) H_1+H_2(d\otimes I_V).
\end{align*}
\end{proof}

\subsection{Proof of the Tame Vanishing Lemma}

Let $\mathcal{G}\rightrightarrows M$ be as in the statement. By passing to the connected components of the
identities in the $s$-fibers (see subsection \ref{subsection_some_properties_of_Lie_algebroids}), we may assume that
$\mathcal{G}$ is $s$-connected. Then $s:\mathcal{G}\to M$ is a locally trivial fiber bundle, so we can apply Lemma
\ref{Lemma_fiberwise_Green} to the complex of $s$-foliated forms with coefficients in $s^*(V)$
\[(\Omega^{{\bullet}}(T^s{\mathcal{G}}, s^*(V)),d\otimes I_V).\]
Recall from subsection \ref{subsection_Lie_algebroid_cohomology} that the $\mathcal{G}$ invariant part of this
complex is isomorphic to the complex computing the Lie algebroid cohomology of $A$ with coefficients in $V$.
Moreover, we obtained an explicit isomorphism (\ref{EQ_definition_of_J})
\begin{equation}\label{EQ_isomorphism}
J:(\Omega^{\bullet}(A,V),d_{\nabla})\diffto (\Omega^{\bullet}(T^s{\mathcal{G}},s^*(V))^{\mathcal{G}},d\otimes I_V),
\end{equation}
and a left inverse of $J$, defined in (\ref{EQ_definition_of_P}), was denoted by $P$.

Let $\langle\cdot,\cdot\rangle$ be an inner product on $A$. Using right translations, we extend
$\langle\cdot,\cdot\rangle$ to an invariant metric $m$ on $T^s{\mathcal{G}}$:
\[m(X,Y)_g:=\langle dr_{g^{-1}}X,dr_{g^{-1}}Y\rangle_{t(g)},\  \forall \ X,Y\in T_g^s{\mathcal{G}}.\]
Invariance of $m$ implies that the right translation by an arrow $g:x\to y$ is an isometry between the $s$-fibers
\[r_g:({\mathcal{G}}_y,m_y)\rmap ({\mathcal{G}}_x,m_x).\]
The corresponding operators from subsection \ref{subsection_de Rham_complex} are also invariant.
\begin{lemma}\label{Lemma_delta_inv}
The operators $\delta\otimes I_V$, $\Delta\otimes I_V$, $H_1$ and $H_2$, corresponding to $m$, send invariant forms
to invariant forms.
\end{lemma}
\begin{proof}
Since right translations are isometries and the operators $\delta_z$ are invariant under isometries we have that
$r_g^*\circ \delta_x=\delta_y\circ r_g^*$, for all arrows $g:x\to y$.

For $\eta\in\Omega^{\bullet}(T^s{\mathcal{G}},s^*(V))^{\mathcal{G}}$ we have that
\begin{align*}
(r_g^*\otimes g)&(\delta\otimes I_V(\eta))_{|{\mathcal{G}}_x}=(r_g^*\circ \delta_x\otimes g)(\eta_{|{\mathcal{G}}_x})=(\delta_y\circ r_g^*\otimes g)(\eta_{|{\mathcal{G}}_x})=\\
&=(\delta_y\otimes I_{V_y})(r_g^*\otimes g)(\eta_{|{\mathcal{G}}_x})=(\delta_y\otimes I_{V_y})(\eta_{|{\mathcal{G}}_y})=(\delta\otimes I_V)(\eta)_{|{\mathcal{G}}_y}.
\end{align*}
This shows that $\delta\otimes I_V(\eta)\in\Omega^{\bullet}(T^s{\mathcal{G}},s^*(V))^{\mathcal{G}}$. The other
operators are constructed in
terms of $\delta\otimes I_V$ and $d\otimes I_V$, thus they also preserve $\Omega^{\bullet}(T^s{\mathcal{G}},s^*(V))^{\mathcal{G}}$.
\end{proof}

This lemma and the isomorphism (\ref{EQ_isomorphism}) imply that the maps
\[\Omega^{p-1}(A,V) \stackrel{h_1}{\longleftarrow}\Omega^{p}(A,V)\stackrel{h_2}{\longleftarrow}\Omega^{p+1}(A,V),\]
\[h_1:=P\circ H_{1}\circ J, \ \  h_2:=P\circ H_{2}\circ J,\]
are linear homotopy operators for the Lie algebroid complex in degree $p$.

For part (1) of the Tame Vanishing Lemma, let $\omega\in\Omega^p(A,V)$ and $O\subset M$ an orbit of $A$. Since
$\mathcal{G}$ is $s$-connected we have that $s^{-1}(O)=t^{-1}(O)=\mathcal{G}_{|O}$. Clearly
$J(\omega)_{|s^{-1}(O)}$ depends only on $\omega_{|O}$. By the construction of $H_1$, for all $x\in O$, we have that
\[h_1(\omega)_x=H_1(J(\omega))_{1_x}=(\delta_x\circ\Delta_x^{-1}\otimes I_{V_x})(J(\omega)_{|s^{-1}(x)})_{1_x}.\]
Thus $h_1(\omega)_{|O}$ depends only on $\omega_{|O}$. The same argument applies also to $h_2$.

Before checking part (2), we give a simple lemma:
\begin{lemma} \label{Lemma_pullback_tame}
Consider a vector bundle map $A:F_1\to F_2$ between vector bundles $F_1\to M_1$ and $F_2\to M_2$, covering a map
$f:M_1\to M_2$. If $A$ is fiberwise invertible and $f$ is proper, then the pullback map
\[A^*:\Gamma(F_2)\rmap \Gamma(F_1), \ A(\sigma)_x:=A_{x}^{-1}(\sigma_{f(x)})\]
satisfies the following tameness inequalities: for every open $U\subset M_2$, with $\overline{U}$ compact, there are
constants $C_{n,U}>0$ such that
\[\|A^*(\sigma)\|_{n, \overline{f^{-1}(U)}}\leq C_{n,U}\|\sigma\|_{n, U},\ \ \forall\  \sigma \in \Gamma(F_{2|\overline{U}}).\]
Moreover:
\begin{enumerate}[(a)]
\item if $U'\subset U$ is open, and one uses the same charts when computing the norms, then one can choose $C_{n,U'}:=C_{n,U}$;
\item if $N\subset M_2$ is a submanifold and $\sigma\in\Gamma(F_2)$ satisfies $j^k_{|N}(\sigma)=0$, then its pullback satisfies
$j^k_{|f^{-1}(N)}(A^*(\sigma))=0$.
\end{enumerate}
\end{lemma}
\begin{proof}
Since $A$ is fiberwise invertible, we can assume that $F_1=f^*(F_2)$ and $A^*=f^*$. By choosing a vector bundle
$F'$ such that $F_2\oplus F'$ is trivial, we reduce the problem to the case when $F_2$ is the trivial line bundle. So, we
have to check that $f^*:C^{\infty}(M_2)\to C^{\infty}(M_1)$ has the desired properties. But this is straightforward:
we just cover both $\overline{f^{-1}(U)}$ and $\overline{U}$ by charts, and apply the chain rule. The constants
$C_{n,U}$ are the $C^n$-norm of $f$ over $\overline{f^{-1}(U)}$, and therefore are getting smaller if $U$ gets
smaller. This implies (a). For part (b), just observe that $j^k_{f(x)}(\sigma)=0$ implies $j^k_x(\sigma\circ f)=0$.
\end{proof}

Note that the maps $J$ and $P$ are induced by maps satisfying the conditions of the lemma (see the explicit formulas
(\ref{EQ_definition_of_J}) and (\ref{EQ_definition_of_P})).

Part (2) of the Tame Vanishing Lemma follows by Lemma \ref{Lemma_fiberwise_Green} (c) and by applying Lemma
\ref{Lemma_pullback_tame} to $J$ and $P$. Corollary \ref{corollary_unu} follows from Lemma
\ref{Lemma_pullback_tame} (a) and Lemma \ref{Lemma_fiberwise_Green} (c). To prove Corollary
\ref{corollary_unu_prim}, consider $\omega$ a form with $j^k_{|O}\omega=0$, for $O$ an orbit. Then, by Lemma
\ref{Lemma_pullback_tame} (b), it follows that $J(\omega)$ vanishes up to order $k$ along
$t^{-1}(O)=\mathcal{G}_{|O}$. By construction, we have that $H_1$ is $C^{\infty}(M)$ linear, therefore also
$H_1(J(\omega))$ vanishes up to order $k$ along $\mathcal{G}_{|O}$; and again by Lemma
\ref{Lemma_pullback_tame} (b) $h_1(\omega)=P(H_1(J(\omega)))$ vanishes along $O$ up to order $k$.

\subsection{The inverse of a family of elliptic operators}\label{section_inverse_family}

This subsection is devoted to proving the following result:

\begin{proposition}\label{Theorem_family_of_operators}
Consider a smooth family of linear differential operators between the vector bundles $V$ and $W$ over a compact base
$F$
\[P_{x}:\Gamma(V)\rmap \Gamma(W),\ \ x\in \mathbb{R}^m.\]
If $P_x$ is elliptic of degree $d\geq 1$ and invertible for all $x\in \mathbb{R}^m$, then
\begin{enumerate}[(1)]
\item the family of inverses $\{Q_x:=P_{x}^{-1}\}_{x\in\mathbb{R}^m}$ induces a linear operator
\[Q:\Gamma(p^*(W))\rmap \Gamma(p^*(V)),\ \  \{\omega_{x}\}_{x\in \mathbb{R}^m}\mapsto \{Q_x\omega_x\}_{x\in \mathbb{R}^m},\]
where $p^*(V):=V\times \mathbb{R}^m\to F\times\mathbb{R}^m$ and $p^*(W):=W\times \mathbb{R}^m\to
F\times\mathbb{R}^m$;
\item $Q$ is locally tame, in the sense that for all bounded opens $U\subset\mathbb{R}^m$, there exist constants $C_{n,U}>0$, such that
the following inequalities hold
\[\|Q(\omega)\|_{n,F\times\overline{U}}\leq C_{n,U}\|\omega\|_{n+s-1,F\times \overline{U}}, \ \forall \omega \in \Gamma(p^*(W)_{| F\times\overline{U}}),\]
with $s=\lfloor\frac{1}{2}\mathrm{dim}(F)\rfloor+1$. If $U'\subset U$, then one can take $C_{n,U'}:= C_{n,U}$.
\end{enumerate}
\end{proposition}

Fixing $C^n$-norms $\|\cdot\|_{n}$ on $\Gamma(V)$, we induce semi-norms on $\Gamma(p^*(V))$:
\[\|\omega\|_{n,F\times\overline{U}}:=\sup_{0\leq k+|\alpha|\leq n}\sup_{x\in U}\|\frac{\partial^{|\alpha|}\omega_{x}}{\partial x^{\alpha}}\|_{k},\]
where $\omega\in\Gamma(p^*(V))$ is regarded as a smooth family $\omega=\{\omega_x\in
\Gamma(V)\}_{x\in\mathbb{R}^m}$. Similarly, fixing norms on $\Gamma(W)$, we define also norms on
$\Gamma(p^*(W))$.

Endow $\Gamma(V)$ and $\Gamma(W)$ also with Sobolev norms, denoted by $\{|\cdot|_n\}_{n\geq 0}$. Loosely
speaking, $|\omega|_n$, measures the $L^2$-norm of $\omega$ and its partial derivatives up to order $n$ (for a
precise definition, see e.g. \cite{Gilkey}). Denote by $H_n(\Gamma(V))$ and by $H_n(\Gamma(W))$ the completion of
$\Gamma(V)$, respectively of $\Gamma(W)$, with respect to the Sobolev norm $|\cdot|_n$.

We will use the standard inequalities between the Sobolev norms and the $C^n$-norms that follow from the Sobolev
embedding theorem
\begin{equation}\label{EQ_Sobolev}
\|\omega\|_n\leq C_n |\omega|_{n+s},  \ \ |\omega|_n\leq C_n \|\omega\|_{n},
\end{equation}
for all $\omega\in\Gamma(V)$ (resp. $\Gamma(W)$), where $s=\lfloor\frac{1}{2}\mathrm{dim}(F)\rfloor+1$ and
$C_n>0$.

Since $P_x$ is of order $d$, it induces continuous linear maps between the Sobolev spaces, denoted by
\[[P_x]_n: H_{n+d}(\Gamma(V))\rmap  H_{n}(\Gamma(W)).\]

These maps are invertible.

\begin{lemma}\label{L_invertible_on_Sobolev}
If an elliptic differential operator of degree $d$ \[P:\Gamma(V)\rmap\Gamma(W)\] is invertible, then for every $n\geq 0$
the induced map
\[[P]_n:H_{n+d}(\Gamma(V))\rmap H_{n}(\Gamma(W))\] is also invertible and its inverse is induced by the inverse of $P$.
\end{lemma}

\begin{proof}
Since $P$ is elliptic, it is invertible modulo smoothing operators (see Lemma 1.3.5 in \cite{Gilkey}), i.e.\ there exists a
pseudo-differential operator
\[\Psi:\Gamma(W)\rmap \Gamma(V),\]
of degree $-d$ such that $\Psi P-\textrm{Id}=K_1$ and $P \Psi-\textrm{Id}=K_2$, where $K_1$ and $K_2$ are
smoothing operators. Therefore, $\Psi$ induces continuous maps
\[[\Psi]_n:H_{n}(\Gamma(W))\rmap H_{n+d}(\Gamma(V)),\]
and $K_1$ and $K_2$ induce continuous maps
\[[K_1]_n:H_{n}(\Gamma(V))\rmap \Gamma(V), \ \ \ [K_2]_n:H_{n}(\Gamma(W))\rmap \Gamma(W).\]
Now we show that $[P]_n$ is a bijection:

\underline{injective}: For $\eta\in H_{n+d}(\Gamma(V))$, with $[P]_n\eta=0$, we have that
\[\eta=(\textrm{Id}-[\Psi]_n [P]_n)\eta=-[K_1]_n\eta\in \Gamma(V),\]
hence $[P]_n\eta=P\eta$. By injectivity of $P$, we have that $\eta=0$.

\underline{surjective}: For $\theta\in H_{n}(\Gamma(W))$, we have that
\[([P]_n[\Psi]_n-\textrm{Id})\theta=[K_2]_n\theta\in \Gamma(W).\]
Since $P$ is onto, $[K_2]_n\theta=P\eta$ for some $\eta\in \Gamma(V)$. We obtain that $\theta$ is in the range of
$[P]_n$:
\[\theta=[P]_n([\Psi]_n\theta-\eta).\]

The inverse of a bounded operator between Banach spaces is bounded, therefore $[P]_n^{-1}$ is continuous. Since on
smooth sections $[P]_n^{-1}$ coincides with $P^{-1}$, and since the space of smooth sections is dense in all Sobolev
spaces, it follows that $P^{-1}$ induces a continuous map $H_{n}(\Gamma(W))\to H_{n+d}(\Gamma(V))$, and that
that map is $[P]_n^{-1}$.
\end{proof}

For two Banach spaces $B_1$ and $B_2$ denote by $Lin(B_1,B_2)$ the Banach space of bounded linear maps between
them and by $Iso(B_1,B_2)$ the open subset consisting of invertible maps. The following proves that the family
$[P_x]_n$ is smooth.

\begin{lemma}\label{L_smooth_family_operators}
Let $\{P_x\}_{x\in \mathbb{R}^m}$ be a smooth family of linear differential operators of order $d$ between the
sections of vector bundles $V$ and $W$, both over a compact manifold $F$. Then the map induced by $P$ from
$\mathbb{R}^m$ to the space of bounded linear operators between the Sobolev spaces
\[\mathbb{R}^m\ni x\mapsto [P_x]_n\in Lin(H_{n+d}(\Gamma(V)),H_{n}(\Gamma(W)))\]
is smooth and its derivatives are induced by the derivatives of $P_x$.
\end{lemma}
\begin{proof}
Linear differential operators of degree $d$ from $V$ to $W$ are sections of the vector bundle
$Hom(J^d(V);W)=J^d(V)^*\otimes W$, where $J^d(V)\to F$ is the $d$-th jet bundle of $V$. Therefore, $P$ can be
viewed as a smooth section of the pullback bundle \[p^*(Hom(J^d(V);W)):=Hom(J^d(V);W)\times \mathbb{R}^m \to
F\times \mathbb{R}^m.\] Since $F$ is compact, by choosing a partition of unity on $F$ with supports inside opens on
which $V$ and $W$ trivialize, one can write any section of $p^*(Hom(J^d(V);W))$ as a linear combination of sections
of $Hom(J^d(V);W)$ with coefficients in $C^{\infty}(\mathbb{R}^m\times F)$. Hence, there are constant differential
operators $P_i$ and $f_i\in C^{\infty}(\mathbb{R}^m\times F)$, such that
\[P_x=\sum f_i(x)P_i.\]
So it suffices to prove that for $f\in C^{\infty}(\mathbb{R}^m\times F)$, multiplication with $f(x)$ induces a smooth
map
\[\mathbb{R}^m\ni x\mapsto [f(x)\textrm{Id}]_n\in Lin(H_{n}(\Gamma(W)),H_{n}(\Gamma(W))).\]
First, it is easy to see that for any smooth function $g\in C^{\infty}(\mathbb{R}^m\times F)$ and every compact
$K\subset \mathbb{R}^m$, there are constants $C_n(g,K)$ such that $|g(x)\sigma|_n\leq C_n(g,K)|\sigma|_n$ for all
$x\in K$ and $\sigma\in H_n(\Gamma(W))$; or equivalently that the operator norm satisfies
$|[g(x)\mathrm{Id}]_n|_{op}\leq C_n(g,K)$, for $x\in K$.

Consider $f\in C^{\infty}(\mathbb{R}^m\times F)$, and $\overline{x}\in\mathbb{R}^m$ and $K$ is a closed ball
centered at $\overline{x}$. Using the Taylor expansion of $f$ at $\overline{x}$, write
\[f(x)-f(\overline{x})=\sum_{i=1}^m(x_i-\overline{x}_i)T^i_{\overline{x}}(x),\]
\[f(x)-f(\overline{x})-\sum_{i=1}^m(x_i-\overline{x}_i)\frac{\partial f}{\partial x_i}(\overline{x})=\sum_{1\leq i\leq j\leq m}(x_i-\overline{x}_i)(x_j-\overline{x}_j)T^{i,j}_{\overline{x}}(x),\]
where $T^{i}_{\overline{x}},T^{i,j}_{\overline{x}}\in C^{\infty}(\mathbb{R}^m\times F)$. Thus, for all $x\in K$,
we have that
\begin{align*}
|[f(x)\mathrm{Id}]_n-[f(\overline{x})\mathrm{Id}]_n&|_{op}\leq |x-\overline{x}| \sum_{1\leq i\leq m}C_n(T^{i}_{\overline{x}},K),\\
|[f(x)\mathrm{Id}]_n-[f(\overline{x})\mathrm{Id}]_n-&\sum_{i=1}^m(x_i-\overline{x}_i)[\frac{\partial f}{\partial x_i}(\overline{x})\mathrm{Id}]_n|_{op}\leq\\
&\leq |x-\overline{x}|^2\sum_{1\leq i\leq j\leq m}C_n(T^{i,j}_{\overline{x}},K).
\end{align*}
The first inequality implies that the map $x\mapsto [f(x)\mathrm{Id}]_n$ is $C^0$ and the second that it is $C^1$,
with partial derivatives given by
\[\frac{\partial }{\partial x_i}[f\mathrm{Id}]_n=[\frac{\partial f}{\partial x_i}\mathrm{Id}]_n.\]
The statement follows now inductively.
\end{proof}

\subsubsection*{Proof of Proposition \ref{Theorem_family_of_operators}}

By Lemma \ref{L_invertible_on_Sobolev}, $Q_x=P_x^{-1}$ induces continuous operators
\[[Q_x]_n:H_{n}(\Gamma(W))\rmap H_{n+d}(\Gamma(V)).\]
We claim that the following map is smooth
\[\mathbb{R}^m\ni x\mapsto [Q_x]_n\in Lin(H_{n}(\Gamma(W)),H_{n+d}(\Gamma(V))).\]
This follows by Lemma \ref{L_invertible_on_Sobolev} and Lemma \ref{L_smooth_family_operators}, since we can write
\[[Q_x]_n=[P_x^{-1}]_n=[P_x]_n^{-1}=\iota([P_x]_n),\]
where $\iota$ is the (smooth) inversion map
\[\iota:Iso(H_{n+d}(\Gamma(V)),H_{n}(\Gamma(W)))\rmap Iso(H_{n}(\Gamma(W)),H_{n+d}(\Gamma(V))).\]

Let $\omega=\{\omega_x\}_{x\in \mathbb{R}^m}\in\Gamma(p^*(W))$. By our claim it follows that
\[x\mapsto [Q_x]_n[\omega_x]_n=[Q_x\omega_x]_{n+d}\in H_{n+d}(\Gamma(V))\] is a smooth map. On the other hand, the Sobolev inequalities
(\ref{EQ_Sobolev}) show that the inclusion $\Gamma(V)\to  \Gamma^n(V)$, where $\Gamma^n(V)$ is the space of
sections of $V$ of class $C^n$ (endowed with the norm $\|\cdot\|_{n}$), extends to a continuous map
\[H_{n+s}(\Gamma(V))\rmap \Gamma^n(V).\]
Since also evaluation $ev_p:\Gamma^n(V)\to V_p$ at $p\in F$ is continuous, it follows that the map $x\mapsto
Q_x\omega_x(p)\in V_p$ is smooth. This is enough to conclude smoothness of the family $\{Q_x\omega_x\}_{x\in
\mathbb{R}^m}$, so $Q(\omega)\in\Gamma(p^*(V))$. This finishes the proof of the first part.

For the second part, let $U\subset \mathbb{R}^m$ be an open with $\overline{U}$ compact. Since the map $x\mapsto
[Q_x]_n$ is smooth, it follows that the numbers are finite:
\begin{equation}\label{EQ_finite_numbers}
D_{n,m,U}:=\sup_{x\in U}\sup_{|\alpha|\leq m}|\frac{\partial^{|\alpha|}}{\partial x^\alpha}[Q_x]_n|_{op},
\end{equation}
where $|\cdot|_{op}$ denotes the operator norm. Let $\omega=\{\omega_x\}_{x\in \overline{U}}$ be an element of
$\Gamma(p^*(W)_{|F\times \overline{U}})$. By Lemma \ref{L_smooth_family_operators}, also the map $x\mapsto
[\omega_x]_{n}\in H_{n}(\Gamma(W))$ is smooth and that for all multi-indices $\gamma$
\[\frac{\partial^{|\gamma|}}{\partial x^\gamma}[\omega_x]_{n}=[\frac{\partial^{|\gamma|}}{\partial x^\gamma}\omega_x]_{n}.\]
Let $k$ and $\alpha$ be such that $|\alpha|+k\leq n$. Using (\ref{EQ_Sobolev}), (\ref{EQ_finite_numbers}) we obtain
\begin{align*}
\|\frac{\partial^{|\alpha|}}{\partial x^\alpha}& (Q_x\omega_x)\|_{k}\leq \|\frac{\partial^{|\alpha|}}{\partial x^\alpha} (Q_x\omega_x)\|_{k+d-1}\leq
 C_{k+d-1} |\frac{\partial^{|\alpha|}}{\partial x^\alpha} (Q_x \omega_x)|_{k+s+d-1}\leq\\
&\leq C_{k+d-1} \sum_{\beta+\gamma=\alpha}\binom{\alpha}{\beta\  \gamma} |\frac{\partial^{|\beta|}}{\partial x^\beta}Q_x\frac{\partial^{|\gamma|}}{\partial x^\gamma}\omega_x|_{k+s+d-1}\leq\\
&\leq C_{k+d-1} \sum_{\beta+\gamma=\alpha}\binom{\alpha}{\beta\  \gamma} D_{k+s-1,|\beta|,U}|\frac{\partial^{|\gamma|}}{\partial x^\gamma}\omega_x|_{k+s-1}\leq\\
&\leq C_{k+d-1}C_{k+s-1} \sum_{\beta+\gamma=\alpha}\binom{\alpha}{\beta\  \gamma} D_{k+s-1,|\beta|,U}\|\frac{\partial^{|\gamma|}}{\partial x^\gamma}\omega_x\|_{k+s-1}\leq\\
&\leq C_{n,U}\|\omega\|_{n+s-1,F\times\overline{U}}.
\end{align*}
This proves the second part:
\[\|Q(\omega)\|_{n,F\times\overline{U}}\leq C_{n,U}\|\omega\|_{n+s-1,F\times\overline{U}}.\]
The constants $D_{n,m,U}$ are clearly decreasing in $U$, hence for $U'\subset U$ we also have that $C_{n,U'}\leq
C_{n,U}$. This finishes the proof of Proposition \ref{Theorem_family_of_operators}.

\clearpage \pagestyle{plain}

\pagestyle{fancy} \fancyhead[CE]{Appendix}

\fancyhead[CO]{Rigidity of Foliations, an unpublished paper of Richard S. Hamilton}

\section{Rigidity of foliations, an unpublished paper of Richard S. Hamilton}\label{Section_Richard}

In this section we revisit an unpublished paper of Hamilton \cite{Ham3} on rigidity of foliations. The novelty in our
approach compared to Hamilton's lies in the use of Lie groupoids and Lie algebroids, and the only point where we
deviate from his proof is in the construction of the tame homotopy operators for the deformation complex of the
foliation. This complex computes the Lie algebroid cohomology of the foliation with coefficients in the Bott
representation, and therefore, the Tame Vanishing Lemma provides such operators. Besides this, we briefly rewrite the
content of \cite{Ham3}.

\subsection{The statement of Hamilton's rigidity theorem}

Let $M$ be a compact connected manifold and let $\mathcal{F}$ be a foliation on $M$. The foliation is called
\textbf{Hausdorff}\index{Hausdorff foliation}\index{foliation} if the quotient topology on the space of leaves
$M/\mathcal{F}$ is Hausdorff. This condition has some strong consequences, it implies that $M$ has a generic
compact leaf $L$ such that all the leaves in an saturated dense open of $M$ are diffeomorphic to $L$, and such that
$L$ is a finite covering space of any other leaf (see the proof of Lemma \ref{Lemma_hausdorff}).

We recall now the main result from \cite{Ham3} (see also \cite{Epstein} for a similar result on $C^k$-foliations).

\begin{theorem}[Global Reeb-Thurston Stability Theorem]\label{Theorem_foliations}
Let $M$ be a compact manifold and let $\mathcal{F}$ be a Hausdorff foliation on $M$. If the generic leaf $L$ satisfies
$H^1(L)=0$, then the foliation $\mathcal{F}$ is $C^{\infty}$-rigid.
\end{theorem}

The rigidity property is similar to the rigidity of Poisson structures presented in Definition \ref{Definition_CpC1}; we
will state this conclusion in more detail in Theorem \ref{Theorem_rigi_crit_folii}.

\subsubsection*{Principal $\mathbb{S}^1$-bundles}

To motivate a bit the condition $H^1(L)=0$, let us look at the foliation $\mathcal{F}$ given by the fibers of a compact
principal $\mathbb{S}^1$-bundle $M\to N$. This foliation is Hausdorff, since its leaf space is $N$. We claim that,
unless $N$ is a point, $\mathcal{F}$ is not rigid. For this, let
\begin{itemize}
\item $\frac{\partial}{\partial \theta}$ be the infinitesimal generator of the $\mathbb{S}^1$-action,
\item $X$ be a vector field on $N$ with at least one noncompact flow line,
\item $\widetilde{X}$ be a lift of $X$ to $M$.
\end{itemize}
Denote by $\mathcal{F}_{\epsilon}$ the foliation given by the flow lines of $\frac{\partial}{\partial \theta}+\epsilon
\widetilde{X}$. Then $\mathcal{F}_{0}=\mathcal{F}$, but for $\epsilon\neq 0$, $\mathcal{F}_{\epsilon}$ is not
diffeomorphic to $\mathcal{F}$ since it has at least one noncompact leaf. So $\mathcal{F}$ is not rigid.

\subsection{The Inverse Function Theorem of Nash and Moser}\label{subsection_inv_funct_teorema}

The proof of Theorem \ref{Theorem_foliations} uses the Nash-Moser Inverse Function Theorem developed by Hamilton
\cite{Ham}. In this subsection we explain this result.

\subsubsection*{Tame Fr\'echet manifolds}

First, we recall some terminology developed in \cite{Ham}.\\

\noindent $\bullet$ A Fr\'echet space $F$ endowed with an increasing family of norms $\{\|\cdot\|_n\}_{n\geq 0}$
generating its
topology is called a \textbf{graded Fr\'echet space}.\index{graded Fr\'echet space}\\
\noindent $\bullet$ A family $\{S_t:F\to F\}_{t>1}$ of linear operators are called \textbf{smoothing
operators}\index{smoothing operators} of degree $d$, if there exist constants $C_{n,m}>0$ such that for all $n,m\geq
0$ and $\sigma\in F$ the following inequalities hold:
\[\|S_t\sigma\|_{n+m}\leq t^{m+d}C_{n,m}\|\sigma\|_{n},\ \ \|S_t\sigma-\sigma\|_{n}\leq t^{d-m}C_{n,m}\|\sigma\|_{n+m}.\]
\noindent $\bullet$ A graded Fr\'echet space that admits smoothing operators is called a \textbf{tame Fr\'echet
space}\index{tame Fr\'echet space}.

Let $({{F}},\{\|\cdot\|_n\}_{n\geq 0})$ and $({{E}},\{\|\cdot\|_n\}_{n\geq 0})$ be two graded Fr\'echet
spaces.\\
\noindent $\bullet$ A liner map
\[A:{F}\rmap {E},\]
is called \textbf{tame linear}\index{tame linear} of degree $r$ and base $b$ if it satisfies
\[\|A\sigma\|_n\leq C_n \|\sigma\|_{n+r},\ \ \forall\ \sigma\in F,\]
for all $n\geq b$, with constants $C_n>0$ depending only on $n$.\\
\noindent $\bullet$ A continuous map $P:U\to {E}$, where $U\subset {F}$ is an open set, is called \textbf{tame} of
degree $r$ and base $b$, if each point in $U$ has an open neighborhood $V\subset U$ on which the following estimates
hold:
\[\|P(\sigma)\|_n\leq C_n(1+ \|\sigma\|_{n+r}),\ \ \forall\  \sigma\in V,\]
for all $n\geq b$ and with constants $C_n>0$ depending only on $V$ and $n$.\\
\noindent $\bullet$ $P$ is called \textbf{smooth tame}\index{smooth tame}, if $P$ is smooth and all its higher
derivatives are tame. We denote the differential of $P$ by
\[DP:U\times F\rmap E, \ \ D_{\sigma}Pv:=\lim_{t\to 0}\frac{1}{t}\left(P(\sigma+tv)-P(\sigma)\right).\]
The higher derivatives of $P$ are the symmetric maps
\[D^lP:U\times F\times \ldots \times F\rmap E, \]
\[ D^l_{\sigma}P\{v_1,\ldots,v_l\}:=\frac{\partial^l}{\partial t_1\ldots\partial t_l}P(\sigma+t_1v_1+\ldots+t_lv_l)_{|t_1=\ldots=t_l=0}.\]
\noindent $\bullet$ A \textbf{tame Fr\'echet manifold}\index{tame Fr\'echet manifold} is a Hausdorff topological space
endowed with an atlas with
values in graded Fr\'echet spaces, for which the transition functions are smooth tame maps.\\
\noindent $\bullet$ A \textbf{smooth tame map} between two tame Fr\'echet manifolds is a continuous map which is
smooth tame in every coordinate chart.\\
\noindent $\bullet$ A \textbf{tame Fr\'echet vector bundle}\index{tame Fr\'echet vector bundle} consists of a smooth
tame map between tame Fr\'echet manifolds
\[P:\mathcal{V}\rmap \mathcal{M},\]
such that the fibers of $P$ are endowed with the structure of tame Fr\'echet spaces and there are (smooth tame) local
charts on $\mathcal{V}$ of the form
\[U_{\mathcal{V}}\cong F\times U_{\mathcal{M}},\]
where $U_{\mathcal{M}}$ is a chart on $\mathcal{M}$, $F$ is a tame Fr\'echet vector space,
$P_{|U_{\mathcal{V}}}$ is the second projection, and the tame Fr\'echet space structure on the fibers of $P$ is that
of $F$.

\subsubsection*{Examples of tame manifolds}

The reference for the examples below is again \cite{Ham}.\\

The main example of a tame Fr\'echet space is the space of section of a vector bundle $V\to M$,
\[(\Gamma(V),\{\|\cdot\|_n\}_{n\geq 0}),\]
where $M$ is a compact manifold and the norms $\|\cdot\|_n$ are $C^n$-norms (see subsection
\ref{subsection_weak_topology}).

If $W$ is a second vector bundle over $M$, and
\[P:\Gamma(V)\rmap \Gamma(W)\]
is a linear differential operator of degree $d$, then $P$ is a tame linear map of degree $d$ and base 0. More generally,
let $P$ be a nonlinear differential operator of degree $d$, i.e.\ $P$ is of the form
\[P:\Gamma(V)\rmap \Gamma(W), \ \ P({\sigma})=p_*(j^d({\sigma})), \ \ \forall \ {\sigma}\in\Gamma(V),\]
where $p$ is a smooth bundle map from the $d$-th jet bundle
\[p:J^d(V)\rmap W.\]
Then $P$ is a smooth tame map of degree $d$ and base 0.\\

The main example of a tame Fr\'echet manifold is the space of section $\Gamma(B)$ of a fiber bundle $b:B\to M$ over a
compact manifold $M$. The topology on $\Gamma(B)$ is the weak $C^{\infty}$-topology described in subsection
\ref{subsection_weak_topology}, which, by compactness of $M$, coincides with the strong $C^{\infty}$-topology.

The tangent space at ${\sigma}\in\Gamma(B)$ is
\[T_{{\sigma}}\Gamma(B)=\Gamma({\sigma}^*(T^{b}B)),\]
where $T^bB\subset TB$ is the kernel of $db$. The tangent bundle of $\Gamma(B)$ is
\[T\Gamma(B)=\Gamma(T^bB\to M)\to \Gamma(B).\]
The space of vector fields on $\Gamma(B)$ is $\mathfrak{X}(\Gamma(B))=\Gamma(T^bB)$. 

To construct a chart around ${\sigma}\in\Gamma(B)$, let
\[\Psi:{\sigma}^*(T^bB)\hookrightarrow B\]
be a tubular neighborhood of ${\sigma}(M)$ in $B$ along the fibers of $b$ (i.e.\ $\Psi$ is a bundle map). Then a chart
around ${\sigma}$ is given by
\[\Psi_*:\Gamma({\sigma}^*(T^bB))\rmap \Gamma(B).\]

Let $b_1:B_1\to M$, $b_2:B_2\to M$ be two fiber bundles. A bundle map $p:B_1\to B_2$ induces a smooth tame map
between the Fr\'echet manifolds $p_*:\Gamma(B_1)\to\Gamma(B_2)$. The differential of $p_*$ is simply
\[D_{\sigma}(p_*)v=dp\circ v\in \Gamma((p\circ {\sigma})^*(T^{b_2}B_2)),\]
where ${\sigma}\in\Gamma(B_1)$ and $v\in\Gamma({\sigma}^*(T^{b_1}B_1))=T_{{\sigma}}\Gamma(B_1)$. More
generally, a bundle map $p:J^d(B_1)\to B_2$ induces a nonlinear differential operator
\[P:\Gamma(B_1)\rmap \Gamma(B_2), \ \ P({\sigma})=p_*(j^d({\sigma})).\]
Then $P$ is a smooth tame map.

The space of functions $C^{\infty}(M,N)$, for $M$ compact, is a tame Fr\'echet manifold. This can be easily seen by
identifying
\[C^{\infty}(M,N)=\Gamma(M\times N\to M).\]

In particular, the group of diffeomorphisms of a compact manifold $M$
\[\textrm{Diff}(M)\subset C^{\infty}(M,M)\]
is a $C^1$-open. In fact, it is a \textbf{tame Lie group}\index{tame Lie group}, in the sense that all structure maps
are tame and smooth. The Lie algebra of this group is
\[(\mathfrak{X}(M),[\cdot,\cdot],\{\|\cdot\|_n\}_{n\geq 0}),\]
and the exponential map is the map that sends a vector field to its time 1 flow. As opposed to finite dimensional Lie
groups, the exponential is not a local diffeomorphism (its range is not a neighborhood of the identity \cite{Ham}). This
is related to the failure of the standard inverse function theorem in the case of tame Fr\'echet manifolds. Nevertheless,
a version of it exists and will be explained in the following subsection.

\subsubsection*{The Inverse Function Theorem of Nash and Moser}

Hamilton has several versions of the inverse function theorem for tame Fr\'echet manifolds \cite{Ham,Ham2,Ham3}.
Here we state the version needed for Theorem \ref{Theorem_foliations}, which is more of an implicit function theorem.

Consider $\mathcal{V}\to \mathcal{M}$ a tame Fr\'echet vector bundle over a tame Fr\'echet manifold and
$Q:\mathcal{M}\to \mathcal{V}$ a section of $\mathcal{V}$. We are interested in the problem of locally
parameterizing the zeros of $Q$. Identifying $\mathcal{M}$ with the zero section
\[\mathcal{M}\hookrightarrow \mathcal{V},\]
we have a canonical isomorphism
\[T_{\sigma}\mathcal{V}=T_\sigma\mathcal{M}\oplus \mathcal{V}_{\sigma}, \ \ \sigma\in\mathcal{M}.\]
If $\sigma$ is a zero of $Q$, then the differential of $Q$ takes the following form
\[D_{\sigma}Q:T_{\sigma}\mathcal{M}\rmap T_{\sigma}\mathcal{M}\oplus \mathcal{V}_{\sigma}, \ \ D_{\sigma}Q v=(v,\delta_{\sigma}Q v).\]
We call the tame linear map
\[\delta_{\sigma}Q:T_{\sigma}\mathcal{M}\rmap \mathcal{V}_{\sigma}\]
the vertical derivative of $Q$ at ${\sigma}$.

Consider a second tame Fr\'echet manifold $\mathcal{P}$ and a tame smooth map
\[P:\mathcal{P}\rmap \mathcal{M},\]
such that $Q\circ P=0$. We say that $P$ \textbf{parameterizes} the zeros of $Q$ around $\sigma_{0}=P(\tau_0)$, if
$\sigma_0$ has an open neighborhood $U$, such that for every $\sigma\in U$ with $Q(\sigma)=0$, there exists
$\tau\in\mathcal{P}$ such that $P(\tau)=\sigma$. The infinitesimal version of this condition is exactness at
$\tau=\tau_0$ of tame linear the complex:
\begin{equation}\label{EQ_linear_complex}
T_{\tau}\mathcal{P}\stackrel{D_{\tau}P}{\rmap}T_{P(\tau)}\mathcal{M}\stackrel{\delta_{P(\tau)}Q}{\rmap}\mathcal{V}_{P(\tau)},\ \ \ \tau\in\mathcal{P}.
\end{equation}

We state now the Nash-Moser Exactness Theorem from \cite{Ham2}.

\begin{theorem}
Assume that the linear complex (\ref{EQ_linear_complex}) is tame exact, i.e.\ there exist smooth tame vector bundle
maps
\[T\mathcal{P}\stackrel{VP}{\longleftarrow}P^*(T\mathcal{M})\stackrel{VQ}{\longleftarrow}P^*(\mathcal{V}),\]
such that
\[D_{\tau}P\circ V_{\tau}Pv+V_{\tau}Q\circ  \delta_{P(\tau)}Qv=v,\]
for all $\tau\in \mathcal{P}$ and $v\in T_{P(\tau)}\mathcal{M}$. Then $P$ parameterizes locally the zeros of $Q$.
More precisely, for every $\tau_0\in \mathcal{P}$, there exists an open neighborhood $U\subset \mathcal{M}$ of
$\sigma_0:=P(\tau_0)$, and a smooth tame map $S:U\to \mathcal{P}$, such that $S(\sigma_0)=\tau_0$ and for every
$\sigma\in U$ that satisfies $Q(\sigma)=0$, we have that
\[\sigma=P(S(\sigma)).\]
\end{theorem}

In order to apply this result, one needs to find the operators $V_{\tau}P$ and $V_{\tau}Q$, for every $\tau\in
\mathcal{P}$. Potentially, this means solving infinitely many equations. Nonetheless, for our application to foliations
(and in general in rigidity problems coming from geometry) this problem simplifies. To explain this, consider an
equation
\[Q(\sigma)=0\]
that is equivariant with respect to the action of a group $\mathcal{G}$, and fix a solution $\sigma_0$. The
corresponding rigidity problem is to determine whether the orbit $\mathcal{G}\sigma_0$ is a neighborhood of
$\sigma_0$ in $Q^{-1}(0)$. If this happens, we call $\sigma_0$ rigid.\\

We put this abstract rigidity problem in the setting of the theorem. Consider the following objects:
\begin{itemize}
\item $\mathcal{V}\to \mathcal{M}$ a tame Fr\'echet vector bundle,
\item $\mathcal{G}$ a tame Fr\'echet Lie group,
\item a tame smooth action of $\mathcal{G}$ by vector bundle maps on $\mathcal{V}\to \mathcal{M}$,
\item $Q:\mathcal{M}\to \mathcal{V}$ a tame smooth $\mathcal{G}$-equivariant section,
\item $\sigma_0\in\mathcal{M}$ a zero of $Q$.
\end{itemize}

Denote the smooth tame map parameterizing the orbit of $\sigma_0$ by
\[P:\mathcal{G}\rmap \mathcal{M}, \ \ P(g):=g\sigma_0.\]
Rigidity of $\sigma_0$ is equivalent to $Q^{-1}(0)$ being parameterized by $P$ around $\sigma_0$.

Consider the linear complex at the unit $e\in \mathcal{G}$
\begin{equation}\label{EQ_linear_complex1}
T_e\mathcal{G}\stackrel{D_eP}{\rmap}T_{\sigma_0}\mathcal{M}\stackrel{\delta_{\sigma_0}Q}{\rmap}\mathcal{V}_{\sigma_0},
\end{equation}
and let us assume that it is tame exact, i.e.\ there are tame linear maps
\begin{equation}\label{EQ_tame_split}
T_e\mathcal{G}\stackrel{V_eP}{\longleftarrow}T_{\sigma_0}\mathcal{M}\stackrel{V_eQ}{\longleftarrow}\mathcal{V}_{\sigma_0},
\end{equation}
such that
\begin{equation}\label{EQ_tame_split2}
D_{e}P\circ V_{e}Pv+V_{e}Q\circ  \delta_{\sigma_0}Qv=v,\ \ \forall \ v\in T_{\sigma_0}\mathcal{M}.
\end{equation}
Let $l_g$ denote the left translation by $g$ on $\mathcal{G}$ and $\mu_g$ denote the action of $g$ on
$\mathcal{M}$ and $\mathcal{V}$. Since $P(g)=l_g(\sigma_0)$ and since $Q$ is equivariant, we have that
\[D_gP=D_{\sigma_0}\mu_g\circ D_eP\circ D_el_{g^{-1}}, \ \delta_{g\sigma_{0}}Q=D_{\sigma_0}\mu_g\circ \delta_{\sigma_0}Q\circ D_{\sigma_0}\mu_{g^{-1}}.\]
Consider the smooth tame vector bundle maps
\[T\mathcal{G}\stackrel{VP}{\longleftarrow}P^*(T\mathcal{M})\stackrel{VQ}{\longleftarrow}P^*(\mathcal{V}),\]
\[V_gP:=D_{e}l_g\circ V_eP\circ D_{\sigma_0}\mu_{g^{-1}}, \ V_gQ:=D_{\sigma_0}\mu_g\circ V_eQ\circ D_{\sigma_0}\mu_{g^{-1}}.\]
The above properties imply that these maps satisfy
\[D_{g}P\circ V_{g}Pv+V_{g}Q\circ  \delta_{g\sigma_0}Qv=v,\]
for all $g\in\mathcal{G}$ and $v\in T_{g\sigma_0}\mathcal{M}$. So, as in \cite{Ham3}, we obtain:
\begin{corollary}\label{corollary_inverse_with_groups}
If there are tame linear maps as in (\ref{EQ_tame_split}), satisfying (\ref{EQ_tame_split2}), then $\sigma_0$ is rigid.
More precisely, there exists an open neighborhood $U\subset \mathcal{M}$ of $\sigma_0$ and a smooth tame map
$S:U\to \mathcal{G}$, such that $S(\sigma_0)=e$ and such that, for $\sigma\in U$, the following implication holds
\[Q(\sigma)=0 \ \  \Rightarrow \ \ \sigma=S(\sigma)\sigma_0.\]
\end{corollary}

\subsection{Proof of Theorem \ref{Theorem_foliations}}

\subsubsection*{The Grassmannian}

We start with some facts from linear algebra. Let $V$ be a vector space, and denote the $k$-th\index{Grassmannian}
Grassmannian of $V$ by $\textrm{Gr}(k,V)$. A point in $\textrm{Gr}(k,V)$ is a $k$-dimensional subspaces $W$ of
$V$. We claim that the tangent space of $\textrm{Gr}(k,V)$ at $W$ can be canonically identified with
\begin{equation}\label{EQ_grass}
T_W(\textrm{Gr}(k,V))=W^*\otimes V/W, \ \ W\in \textrm{Gr}(k,V).
\end{equation}
To see this, note that $\textrm{Gr}(k,V)$ is a homogenous $\textrm{Gl}(V)$-space and that the stabilizer of $W$ has
Lie algebra
\[\mathfrak{gl}(W,V)=\{A\in\mathfrak{gl}(V) | A(W)\subset W\}.\]
Since $\mathfrak{gl}(V)/\mathfrak{gl}(W,V)\cong W^*\otimes V/W$, (\ref{EQ_grass}) follows from the short exact
sequence induced by the action
\[0\rmap \mathfrak{gl}(W,V)\rmap \mathfrak{gl}(V) \rmap T_W(\textrm{Gr}(k,V))\rmap 0.\]

\subsubsection*{Rigidity of foliations}

For a compact manifold $M$, denote the $k$-th Grassmannian bundle by
\[\Gr\rmap M, \ \ \Gr_x:=\textrm{Gr}(k,T_xM), \ \ x\in M.\]
A section of $\Gr$ is a called a \textbf{distribution} of rank $k$, i.e.\ a subbundle of $TM$ of rank $k$. By the
discussion in subsection \ref{subsection_inv_funct_teorema}, the space of distributions is a tame Fr\'echet manifold,
which we denote by
\[\mathcal{M}:=\Gamma(\Gr).\]
Denote the normal bundle of $B\in \mathcal{M}$ by
\[\nu_B:=TM/B.\]
The description of the tangent bundle from subsection \ref{subsection_inv_funct_teorema} and formula
(\ref{EQ_grass}), allow us to identify
\[T_{B}\mathcal{M}\cong\Omega^1(B,\nu_B):=\Gamma(B^*\otimes \nu_B).\]
Explicitly, let $B_{\epsilon}$ be a family of distributions, with $B_0=B$ and consider a family $a_{\epsilon}\in
\Gamma(Gl(TM))$ of automorphisms of $TM$, such that
\[a_0=\textrm{Id}, \ \ a_{\epsilon}(B)=B_{\epsilon}.\]
Then the 1-form $\eta\in\Omega^1(B,\nu_B)$ corresponding to the tangent vector
\[\frac{d}{d\epsilon}_{|\epsilon=0}B_{\epsilon}\in T_{B}\mathcal{M}\]
is given by
\[\eta(X):=\left( \dot{a}_{0}(X) \ \mathrm{mod}\ B\right)  \in\nu_B, \ \ X\in B.\]

By the Frobenius theorem, a distribution is the tangent bundle of a foliation if and only if it is involutive. To put this
condition in our setting, consider the tame Fr\'echet vector bundle over $\mathcal{M}$
\[\mathcal{V}\rmap \mathcal{M}, \ \ \mathcal{V}_B:=\Omega^2(B,\nu_B).\]
This vector bundle has a canonical section which measures integrability
\[Q:\mathcal{M}\rmap \mathcal{V},\]
\[Q_B(X,Y):=\left([X,Y]\ \textrm{mod} \ B\right) \in \Gamma(\nu_B), \ \ X,Y\in\Gamma(B).\]
That $Q_B$ is $C^{\infty}(M)$-linear follows from the Leibniz rule. The zeros of $Q$ is the space of foliations on $M$.

The tame Fr\'echet Lie group \[\mathcal{G}:=\textrm{Diff}(M)\] has a smooth tame action on
$\mathcal{V}\to\mathcal{M}$. For $\varphi\in\mathcal{G}$, $B\in\mathcal{M}$ and $\omega\in\mathcal{V}_B$, the
action is defined as follows
\[\varphi_*(B)_{\varphi(x)}:=d_x\varphi(B_{x}),\ \varphi_*(\omega)(X,Y):=\overline{d\varphi}\circ\omega(d\varphi^{-1}(X),d\varphi^{-1}(Y)),\]
for $X,Y\in\Gamma(\varphi_*(B))$, where $\overline{d\varphi}$ denotes the isomorphism induced between
\[\overline{d_x\varphi}:\nu_{B_x}\diffto \nu_{\varphi_*(B)_{\varphi(x)}}.\]
The section $Q$ is equivariant with respect to the action.

Let $\mathcal{F}$ be a foliation on $M$ and denote its tangent bundle by
\[B_{\mathcal{F}}=T\mathcal{F}\in \mathcal{M}.\]
As before, we consider the map
\[P:\mathcal{G}\rmap \mathcal{M}, \ \ P(\varphi):=\varphi_*(B_{\mathcal{F}}).\]

We compute now the linear complex.

\begin{lemma}
The linear complex (\ref{EQ_linear_complex1}) is isomorphic to:
\begin{equation}\label{EQ_some_sequence}
\mathfrak{X}(M)\xrightarrow{X\mapsto d_{\mathcal{F}}([X])}\Omega^1(B_{\mathcal{F}},\nu_{\mathcal{F}})\xrightarrow{\eta\mapsto d_{\mathcal{F}}(\eta)} \Omega^2(B_{\mathcal{F}},\nu_{\mathcal{F}}).
\end{equation}
\end{lemma}

\begin{proof}
First, note that we have a canonical identification
\[\mathfrak{X}(M)\diffto T_{e}\mathcal{G},\ \ \ X\mapsto \frac{d}{d\epsilon}_{|\epsilon=0}\varphi_X^{\epsilon},\]
where $\varphi_X^{\epsilon}$ denotes the flow of $X$.

Fix $X\in\mathfrak{X}(M)$, and consider $a_{\epsilon}\in \Gamma(Gl(TM))$ such that
\[a_0=\textrm{Id}, \ \  a_{\epsilon}(B_{\mathcal{F}})=\varphi_{X,*}^{\epsilon}(B_{\mathcal{F}}).\]
For $Y\in\Gamma(B_{\mathcal{F}})$, let $Y_{\epsilon}\in\Gamma(B_{\mathcal{F}})$ denote the smooth family
defined by
\[\varphi_{X,*}^{\epsilon}(Y)=a_{\epsilon}(Y_{\epsilon}).\]
Clearly $Y_0=Y$. Taking the derivative at $\epsilon=0$, we obtain that
\[-[X,Y]=\dot{a}_0(Y)+\dot{Y}_0.\]
Since $\dot{Y}_0\in\Gamma(B_{\mathcal{F}})$, we obtain the first map of the complex:
\begin{align*}
D_eP(X)(Y)&=\left(\dot{a}_0(Y)\ \textrm{mod}\ B_{\mathcal{F}}\right)=\left([Y,X]\ \textrm{mod}\ B_{\mathcal{F}}\right)=\\
&=\nabla_{Y}[X]=d_{\mathcal{F}}([X])(Y),
\end{align*}
where $[X]:=\left(X\ \textrm{mod}\ B_{\mathcal{F}}\right)\in\Gamma(\nu_{\mathcal{F}})$ and we used the
differential $d_{\mathcal{F}}$ corresponding to the Bott connection\index{Bott connection} $\nabla$ on
$\nu_{\mathcal{F}}$ (see subsection \ref{subsection_model_sympl_foli}).

We compute now $\delta_{B_{\mathcal{F}}}Q$. Consider $\eta\in T_{B_{\mathcal{F}}}\mathcal{M}$ and let
$B_{\epsilon}$ be a smooth path in $\mathcal{M}$ with $B_0=B_{\mathcal{F}}$ and $\dot{B}_0=\eta$. To compute
$\delta_{B_{\mathcal{F}}}Q(\eta)$, we must trivialize $\mathcal{V}$ along $B_{\epsilon}$. For this, choose a smooth
family $a_{\epsilon}\in \Gamma(Gl(TM))$, such that
\[a_0=\textrm{Id}, \ \ a_{\epsilon}(B)=B_{\epsilon}.\]
A trivialization is given by the map
\[A_{\epsilon}:\mathcal{V}_{B_{\epsilon}}\diffto \mathcal{V}_{B_{\mathcal{F}}}, \ A_{\epsilon}(\omega)(X,Y):=(\overline{a}_{\epsilon})^{-1}\circ \omega(a_{\epsilon}(X),a_{\epsilon}(X)),\]
for $X,Y\in \Gamma(B_{\mathcal{F}})$, where $\overline{a}_{\epsilon}$ denotes the induced isomorphism
\[\overline{a}_{\epsilon}:\nu_{\mathcal{F}}\diffto \nu_{B_{\epsilon}}.\]
For $X,Y\in\Gamma(B_{\mathcal{F}})$, we compute
\begin{align*}
\delta_{B_{\mathcal{F}}}Q(\eta)(X,Y)&=\frac{d}{d\epsilon}_{|\epsilon=0}\left((\overline{a}_{\epsilon})^{-1}\circ Q_{B_{\epsilon}}(a_{\epsilon}(X),a_{\epsilon}(X))\right)=\\
&=\frac{d}{d\epsilon}_{|\epsilon=0}\left(a_{\epsilon}^{-1}\circ [a_{\epsilon}(X),a_{\epsilon}(X)]\right) \textrm{mod} \ B_{\mathcal{F}}=\\
&=\left(-\dot{a}_{0}([X,Y])+[X,\dot{a}_{0}(Y)]+[\dot{a}_{0}(X),Y]\right) \ \textrm{mod} \ B_{\mathcal{F}}=\\
&=-\eta([X,Y])+\nabla_X(\eta(Y))+\nabla_Y(\eta(X))=d_{\mathcal{F}}\eta(X,Y).
\end{align*}
This identifies also the second map in the complex.
\end{proof}
As in \cite{Ham3}, we obtain the following result.

\begin{theorem}\label{Theorem_rigi_crit_folii}
Let $M$ be a compact manifold with a foliation $\mathcal{F}$ of rank $k$. Assume that
$H^1(\mathcal{F},\nu_{\mathcal{F}})$ vanishes tamely, in the sense that there are tame linear maps
\[\Omega^0(T\mathcal{F},\nu_{\mathcal{F}})\stackrel{h_1}{\longleftarrow}\Omega^1(T\mathcal{F},\nu_{\mathcal{F}})\stackrel{h_2}{\longleftarrow}\Omega^2(T\mathcal{F},\nu_{\mathcal{F}}),\]
satisfying
\[d_{\mathcal{F}}\circ h_1(\eta)+h_2\circ d_{\mathcal{F}}(\eta)=\eta, \  \ \forall \ \eta\in \Omega^1(T\mathcal{F},\nu_{\mathcal{F}}).\]
Then $\mathcal{F}$ is rigid. More precisely, there is a smooth tame map
\[S:U\subset \Gamma(\Gr)\rmap \textrm{Diff}(M),\]
where $U$ is an open neighborhood of $T\mathcal{F}$ in the space of all rank $k$ distributions on $M$, such that
$S(T\mathcal{F})=\mathrm{Id}_M$, and for every foliation $\widetilde{\mathcal{F}}$, with
$T\widetilde{\mathcal{F}}\in U$, we have that $\varphi:=S(T\widetilde{\mathcal{F}})$ is an isomorphism of foliated
manifolds
\[\varphi:(M,\mathcal{F})\diffto (M,\widetilde{\mathcal{F}}).\]
\end{theorem}
\begin{proof}
Using a metric on $M$, we obtain a tame linear splitting \[\iota:\Gamma(\nu_{\mathcal{F}})\hookrightarrow
\mathfrak{X}(M)\] of the exact sequence
\[0\rmap \Gamma(B_{\mathcal{F}})\rmap \mathfrak{X}(M)\rmap \Gamma(\nu_{\mathcal{F}})\rmap 0.\]
The tame linear homotopy operators for the complex (\ref{EQ_some_sequence}) are given by:
\[\mathfrak{X}(M)\stackrel{\iota\circ h_1}{\longleftarrow}\Omega^1(T\mathcal{F},\nu_{\mathcal{F}})\stackrel{h_2}{\longleftarrow}\Omega^2(T\mathcal{F},\nu_{\mathcal{F}}).\]
Therefore, Corollary (\ref{corollary_inverse_with_groups}) concludes the proof.
\end{proof}

\subsubsection*{Hausdorff foliations}

In order to put the Hausdorffness\index{Hausdorff foliation} condition on $M/\mathcal{F}$ in the framework of Lie
algebroids/groupoids, we prove a well-known result (yet difficult to extract from the literature).
\begin{lemma}\label{Lemma_hausdorff}
Let $\mathcal{F}$ be a foliation on a compact connected manifold $M$. The condition that $\mathcal{F}$ is Hausdorff
is equivalent to the holonomy groupoid
\[\mathrm{Hol}(\mathcal{F})\rightrightarrows M,\]
being Hausdorff and proper. In this case, the generic leaf of $\mathcal{F}$ is diffeomorphic to any $s$-fiber of
$\mathrm{Hol}(\mathcal{F})$.
\end{lemma}
\begin{proof}
Assume first that $\textrm{Hol}(\mathcal{F})$ is Hausdorff and proper.\index{holonomy}\index{holonomy
groupoid}\index{holonomy group} Since $M$ is compact, this implies compactness of $\textrm{Hol}(\mathcal{F})$.
In particular, its $s$-fibers are compact and its isotropy groups are finite. Since the $s$-fibers are the holonomy covers
of the leaves, it follows that the leaves are compact, and since the isotropy groups are the holonomy groups, it follows
that the Local Reeb Stability Theorem \ref{Reeb_Theorem}\index{Reeb stability} applies. We conclude that every leaf
$S$ has a saturated open neighborhood $U$ of the form
\begin{equation}\label{EQ_reeb_model}
(U,\mathcal{F}_{|U})\cong (\widetilde{S}\times_{\textrm{Hol}_x}\nu_x,\mathcal{F}_N),
\end{equation}
where $\widetilde{S}$ is the holonomy cover of $S$ and $\nu_x$ is endowed with the linear holonomy action of
$\textrm{Hol}_x$. The leaves of $\mathcal{F}_N$ are of the form
\[S_v=\widetilde{S}\times_{\textrm{Hol}_x}\textrm{Hol}_xv\cong \widetilde{S}/\textrm{H}_v,\ \ v\in\nu_x,\]
where $\textrm{H}_v\subset \textrm{Hol}_x$ is the stabilizer of $v$. Since $\textrm{Hol}_x$ is finite and it acts
faithfully on $\nu_x$, the set of points $v$, with $\textrm{H}_v=\{e\}$ is open and dense in $\nu_x$. This follows by
presenting its complement as a finite union of proper vector subspaces
\[\bigcup_{g\in \textrm{Hol}_x\backslash\{e\}}\{v\in \nu_x | (g-\textrm{Id}_{\nu_x})v=0\}.\]
This implies the fact that the generic leaf of $\mathcal{F}_{|U}$ is diffeomorphic to $\widetilde{S}$, the $s$-fiber of
$\textrm{Hol}(\mathcal{F})_{|U}$. Since $M$ is connected, this also provides a proof for the fact that $M$ has a
generic leaf which is a finite cover for any other leaf.

Let $B_r$ be the open ball in $\nu_x$ of radius $r$ centered at the origin with respect to a $\textrm{Hol}_x$-invariant
inner product on $\nu_{x}$ and let $U_r\subset U$ be the open corresponding to
$\widetilde{S}\times_{\textrm{Hol}_x} B_r$. If $S'$ is a second leaf, denote by $U'_r$ the corresponding opens
around $S'$. Since $S$ and $S'$ are compact, there exist $r,s>0$ small enough, such that $\overline{U}_r\cap
\overline{U}'_s=\emptyset$. Hence $U_r/\mathcal{F}$ and $U'_s/\mathcal{F}$ are opens in $M/\mathcal{F}$
separating the images of $S$ and $S'$. This shows that $M/\mathcal{F}$ is Hausdorff.

Conversely, assume that $M/\mathcal{F}$ is Hausdorff. In particular, the points in $M/\mathcal{F}$ are closed,
therefore, the leaves of $\mathcal{F}$ are closed. By compactness of $M$, the leaves are also compact. Now,
Hausdorffness and compact leaves imply that all holonomy groups are finite, for this see Theorems 4.1 and 4.2 in
\cite{Epstein2}. Hence, around every leaf $S$, the local model (\ref{EQ_reeb_model}) holds on some open $U$. Using
the description of the holonomy groupoid of such foliations from subsection \ref{trivial holonomy, non-Hausdorff
holonomy group} and that $\textrm{Hol}_x$ acts linearly and faithfully on $\nu_x$, we see that the holonomy groupoid
of $U$ is isomorphic to
\[\textrm{Hol}(\mathcal{F}_{|U})\cong(\widetilde{S}\times\nu_x\times\nu_x)/{\textrm{Hol}_x}.\]
In particular it is Hausdorff and proper. Since $U$ is invariant, we have that
\[\textrm{Hol}(\mathcal{F}_{|U})=\textrm{Hol}(\mathcal{F})_{|U}.\]
This implies that $\textrm{Hol}(\mathcal{F})$ is proper: a sequence $g_n$ in $\textrm{Hol}(\mathcal{F})$, such that
$(s(g_n),t(g_n))$ converges to $(x,y)\in S\times S$, will belong to $\textrm{Hol}(\mathcal{F})_{|U}$, for $n$ big
enough, thus it contains a convergent subsequence. Hausdorffness also follows: if $g\neq g'$ are arrows over the leaves
$S$ and $S'$ respectively, then, if $S=S'$ then $g$ and $g'$ can be separated in $\textrm{Hol}(\mathcal{F})_{|U}$,
and if $S\neq S'$, by Hausdorffness of $M/\mathcal{F}$, we can find opens neighborhoods $S\subset V$ and $S'\subset
V'$ such that $V\cap V'=\emptyset$, and then $\textrm{Hol}(\mathcal{F})_{|V}$ and
$\textrm{Hol}(\mathcal{F})_{|V'}$ separate $g$ and $g'$.
\end{proof}

\subsubsection*{Proof of Theorem \ref{Theorem_foliations}}

Since $\mathcal{F}$ is Hausdorff and $M$ is compact, the lemma above implies that $\textrm{Hol}(\mathcal{F})$ is
Hausdorff, proper and that its $s$-fibers are compact and have vanishing $H^1$. Now, the parallel transport of the
Bott connection (see section \ref{section_local_reeb_stability_for_non-compact_leaves}) defines an action of
$\textrm{Hol}(\mathcal{F})$ on the normal bundle $\nu_{\mathcal{F}}$, and this action integrates the Bott
representation of $T\mathcal{F}$. Thus, we may apply the Tame Vanishing Lemma to construct tame homotopy
operators for the complex $(\Omega^{\bullet}(T\mathcal{F},\nu_{\mathcal{F}}),d_{\mathcal{F}})$ in degree one.
Theorem \ref{Theorem_rigi_crit_folii} concludes the proof.

\clearpage \pagestyle{plain}

\bibliographystyle{amsplain}
\def\lllll{}

\clearpage \pagestyle{plain}

\printindex

\clearpage \pagestyle{plain}

\chapter*{Samenvatting}
\addcontentsline{toc}{chapter}{Samenvatting} \pagestyle{fancy} 
\fancyhead[CE]{Samenvatting} 
\fancyhead[CO]{Samenvatting}

Dit proefschrift presenteert een aantal nieuwe resultaten in het vakgebied Poisson-meetkunde. Het gaat hoofdzakelijk
om een normaalvormstelling (`normal form theorem'), een stelling over lokale rigiditeit en een expliciete beschrijving
van de Poisson-moduli-ruimtes rond Lie-Poissonbollen. Daarnaast bewijzen we ook een standaardvormstelling voor
symplectische foliaties, een stelling over formele equivalentie rond Poissondeelvari\"eteiten en een resultaat over het
getemd verdwijnen van Lie-algebro\"ide-cohomologie. We geven ook in detail de bewijzen van enkele bekende
resultaten: de existentie van symplectische realisaties (met een origineel bewijs), Conn's lineariserings-stelling (met
enkele vereenvoudigingen) en een resultaat van Hamilton over de rigiditeit van foliaties (welke een applicatie is van het
getemd verdwijnen van de cohomologie).

Poissonstructuren leven op differentieerbare vari\"eteiten. Een vari\"eteit van dimensie $n$ is een ruimte die rond elk
punt lijkt op een $n$-dimensionale kopie van de Euclidische ruimte $\mathbb R^n$. Zo'n ruimte is differentieerbaar als
er calculus op gedaan kan worden (vergelijkbaar met calculus in meerdere variabelen op $\mathbb R^n$). Lijnen en
cirkels zijn \'e\'en-dimensionale vari\"eteiten. Twee-dimensionale vari\"eteiten worden ook wel oppervlakken genoemd.
Voorbeelden hiervan zijn het vlak, de bolschil (het oppervlak van een bol) en de torus (het oppervlak van een donut).

Symplectische vari\"eteiten zijn de klassieke voorbeelden van Poisson-vari\"eteiten. Een symplectische structuur op
een vari\"eteit is een gesloten niet-gedegenereerde twee-vorm. Symplectische meetkunde heeft zijn oorsprong in het
Hamilton-formalisme van de klassieke mechanica, waarbij de faseruimtes van bepaalde klasieke systemen de structuur
aannemen van een symplectische vari\"eteit. Poisson-meetkunde stelt ons in staat om op een differentieerbare manier
meerdere van zulke systemen samen te voegen. De configuratieruimte van een object met massa $m$ dat beweegt in
een gravitationeel veld is bijvoorbeeld een symplectische vari\"eteit. Door deze ruimtes samen te voegen voor alle
waarden van $m$ ontstaat er een Poissonvari\"eteit.

Meetkundig gezien heeft een Poissonvari\"eteit een canonieke ontbinding in symplectische\break vari\"eteiten van
verschillende dimensies, welke symplectische bladeren worden genoemd. Voor de drie- dimensionale Euclidische
ruimte $\mathbb R^3$ bestaat er bijvoorbeeld een Poissonstructuur wiens bladeren bestaan uit de cocentrische bollen
rond de oorsprong (die twee-dimensionaal zijn) en de oorsprong zelf (die nul-dimensionaal is).

Een van de hoofdresultaten van dit proefschrift (Stelling 2) is een standaardvormstelling rond symplecische bladeren.
Het generaliseert Conn's stelling vanuit vaste punten (nul-dimensionale bladeren). In het bijzonder geeft de stelling
voor bladeren die aan onze hypothese voldoen een expliciete beschrijving van alle nabijgelegen bladeren.

We bewijzen ook een rigiditeitsresultaat (Stelling 4) voor integreerbare Poissonvari\"eteiten. Een ``rigide''
Poissonvari\"eteit is er een die in essentie niet vervormd kan worden. Dit betekent dat elke nabijgelegen
Poissonstructuur gelijk is op een isomorfisme na. Dit resultaat heeft als gevolg een versterking van Stelling 2.

Lie-theorie geeft interessante voorbeelden van Poissonvari\"eteiten. In het bijzonder kan er vanuit een compacte
semi-simpele Lie algebra een zogenaamde Lie-Poissonbol geconstrueerd worden. Stelling 4 blijkt hier een verrassende
toepassing te hebben: in Stelling 5 beschrijven we de ruimte van gladde vervormingen van de Lie-Poissonstructuur op
isomorfisme na. We bewijzen dat de Lie-Poissonbol in feite rigide is: Elke nabijgelegen Poissonstructuur is isomorf aan
de Lie-Poissonstructuur op een herschaling van de symplectische vormen op de bladeren na. Dit resultaat is de eerste
berekening van een Poisson-moduli-ruimte in een dimensie groter dan twee rond een Poissonstructuur die niet
symplectisch is.

\clearpage \pagestyle{plain}

\chapter*{Acknowledgements}
\addcontentsline{toc}{chapter}{Acknowledgements} \pagestyle{fancy} 
\fancyhead[CE]{Acknowledgements} 
\fancyhead[CO]{Acknowledgements}

It would not have been possible to write this thesis without the help and support of the kind people around me. I would
like to mention some of them.

First of all, I want to thank my advisor, Marius Crainic. Dear Marius, thank you for introducing me to this beautiful
field of mathematics, for all the nice problems you gave me to work on, for your constant help and support, for your
kind advises regarding my mathematical career, and foremost, for your friendship. I hope that we will solve many more
problems together and that we will continue our endless discussions about politics. Many thanks also to Dacina,
Andreea-Carola and Ana-Nicola for the very nice time we had at your house.

I would like to thank the members of the reading committee, professors Marco Gualtieri, Victor Guillemin, Rui Loja
Fernandes, Alan Weinstein and Nguyen Tien Zung, for taking the time to read this thesis and for their useful
comments.

The Mathematics Department at Utrecht University has been a very pleasant and stimulating workplace, and there are
several people who contributed to this. First, I would like to thank Prof. Ieke Moerdijk for organizing so many fruitful
reading seminars, which created a stimulating work environment that persisted also after he left Utrecht. I am grateful
also to Prof. Erik van den Ban for answering all my questions about Lie groups. To my dear colleagues, thank you for
your friendship and support: Mar\'ia, Ivan, Roy, Andor, Camilo, M\'arti, Jo\~ao, Pedro, Sergey, Boris, Dmitry, Florian,
Benoit, Floris, Jaap, Bas, Dave, Ori, Gil, Andr\'e, KaYin and Ittay. I am grateful also to Helga, Jean, Ria and Wilke for
taking care of everything and making my life very easy.

I am indebted to my teachers at Babe\cb{s}-Bolyai University in Cluj for all I learned from them, especially Andrei
M\u{a}rcu\cb{s}, Paul Blaga and Valeriu Anisiu.

Many thanks also to my mathematical friends from high school and my time in Cluj: Anda, Adriana, Dan, Ralph, Iuli,
Vali, Ovidiu and Raul for sharing my interests and for the nice discoveries we made together.

My visits to Deventer made me feel more at home in the Netherlands, thank you Alina, Rogier and Sam for this.

Dana, you made Utrecht a sunny place. Thanks for constantly distracting me from work, without you this thesis would
have been twice as long.

Last but not the least, I want to thank my dear parents, to whom I dedicate this thesis: Dragi p\u{a}rin\cb{t}i,
mul\cb{t}umesc pentru c\u{a} \^{i}ntotdeuna m-a\cb{t}i l\u{a}sat s\u{a} fac orice am vrut, \cb{s}i mai mult, pentru
c\u{a} \^{i}ntotdeauna m-a\cb{t}i sprijinit in orice am vrut s\u{a} fac. Ca un mic semn de
recuno\cb{s}tin\cb{t}\u{a}, v\u{a} dedic aceast\u{a} c\u{a}rticic\u{a}.

\clearpage \pagestyle{plain}
\chapter*{Curriculum Vitae}
\addcontentsline{toc}{chapter}{Curriculum Vitae} \pagestyle{fancy} 
\fancyhead[CE]{Curriculum Vitae} 
\fancyhead[CO]{Curriculum Vitae}

Ioan Tiberiu M\u{a}rcu\cb{t} was born in Sibiu, Romania on the $24^{\textrm{th}}$ of March 1984. In the period
1999-2003 he studied at the German-language high school Samuel von Brukenthal Gymnasium in Sibiu.

After obtaining his high school degree (Baccalaureate), in October 2003 he moved to Cluj Napoca where he studied
Mathematics at Babe\cb{s}-Bolyai University. During 2005-2006 he participated at master-level courses organized by
\cb{S}coala Normal\u{a} Superioar\u{a} Bucure\cb{s}ti, a scientific institution supported by the Mathematical
Institute of the Romanian Academy. In June 2007 he obtained his Bachelor's degree in Mathematics under the
supervision of Prof. Andrei M\u{a}rcu\cb{s} at Babe\cb{s}-Bolyai University.

In September 2007 he joined the master class program ``Quantum Groups, Affine Lie Algebras and their Applications'',
organized at Utrecht University by the Mathematical Research Institute. He wrote his research project under the
supervision of Prof. Jasper Stokman.

Between September 2008 and September 2012 he was a Ph.D. student at Utrecht University under the guidance of
Prof. Marius Crainic. Currently, he is working as a postdoctoral researcher at Utrecht University.

\end{document}